\documentclass[reqno,english,11pt]{amsart}
\usepackage{amsfonts,amsmath,latexsym,verbatim,amscd,mathrsfs,color,array}
\usepackage{amssymb,amsthm,graphicx }
\usepackage{amsfonts}
\usepackage[colorlinks=true]{hyperref}
\usepackage[normalem]{ulem}

\hypersetup{
	colorlinks=true,
	linkcolor=red
}

\allowdisplaybreaks

\usepackage[hmargin=2.3cm, vmargin=2.3cm]{geometry}
\usepackage{bbm}
\usepackage{pdfsync}
\usepackage{epstopdf}
\usepackage{cite}
\usepackage{graphicx}
\usepackage{multirow}
\usepackage[font=bf,aboveskip=15pt]{caption}
\usepackage[toc,page]{appendix}
\usepackage[colorlinks=true]{hyperref}
\usepackage{bookmark}
\DeclareFontShape{OT1}{cmr}{bx}{sc}{<-> cmbcsc10}{}
\usepackage[square,numbers]{natbib}

\usepackage{esint}%for fint
\usepackage{float}
\usepackage{ulem}
\usepackage{syntonly}
\usepackage{mathtools}
\usepackage{bm}
\usepackage{amsfonts,amsmath,latexsym,verbatim,amscd,mathrsfs,color,array}
\usepackage[colorlinks=true]{hyperref}

\usepackage{pdfsync}
\usepackage{epstopdf}
\usepackage{upgreek}

\providecommand{\Set}[1]{\left\{#1\right\}}

\newcommand{\1}{\mathbf{1}}

\newcommand{\tr}{\mathsf{tr}}

\newcommand{\fbf}{{\mathbf{f}}}
\newcommand{\gbf}{{\mathbf{g}}}

\newcommand{\TT}{{\mathcal T}  }
\newcommand{\LLL}{{\mathcal L}  }

\newcommand{\pp}{ {\partial} }
\newcommand{\B}{ \mathcal{B} }

\newcommand{\RR}{{{\mathbb R}}}

\newcommand{\R} {\mathbb R}

\newcommand{\DD}{{\mathcal D}}

\newcommand{\be}{\begin{equation}}
\newcommand{\ee}{\end{equation}}

\newcommand{\la}{\lambda}

\newcommand{\J}{ {\mbox{{\tiny{$[j]$}}} } }
\newcommand{\M}{ {\mbox{{\tiny{$[m]$}}}} }
\newcommand{\K}{ {\mbox{{\tiny{$[k]$}}}} }

\newcommand{\N}{ {\mbox{{\tiny{$[n]$}}}} }

\newtheorem{lemma}{Lemma}[section]
\newtheorem{prop}{Proposition}[section]
\newtheorem{theorem}{Theorem}
\newtheorem{corollary}{Corollary}[section]
\newtheorem{remark}{Remark}[section]
\newcommand{\bremark}{\begin{remark} \em}
	\newcommand{\eremark}{\end{remark} }

\numberwithin{equation}{section}

\begin{document}
\title[Finite-time blow-up for LLG]{Finite-time singularity formations for the Landau-Lifshitz-Gilbert equation in dimension two}

\author[J. Wei]{Juncheng Wei}
\address{\noindent
Department of Mathematics,
Chinese University of Hong Kong,
Shatin, NT, Hong Kong}
\email{jcwei@math.ubc.ca}

\author[Q. Zhang]{Qidi Zhang}
\address{\noindent
Department of Mathematics, The University of Hong Kong, Hong Kong, China}
\email{qdz@amss.ac.cn}

\author[Y. Zhou]{Yifu Zhou}
\address{\noindent
School of Mathematics and Statistics, Wuhan University, Wuhan 430072, China}
\email{yifuzhou@whu.edu.cn}

\begin{abstract}
We construct {\it non-equivariant} blow-up solutions to the Landau-Lifshitz-Gilbert equation (LLG)  from $\RR^2$ into $S^2$
\begin{equation*}
	\begin{cases}
		u_t=  a(\Delta u+|\nabla u|^2u)
		-b u\wedge \Delta u
	 &\mbox{ \ in \ } \RR^2\times(0,T),
		\\
		u(\cdot,0) = u_0\in S^2  &\mbox{ \ in \ } \RR^2,
	\end{cases}
\end{equation*}
where $a^2+b^2=1,~a > 0,~ b\in\RR$. Given any prescribed, distinct $N$ points in $\R^2$ and small  $T>0$, we prove that there exists a smooth initial data such that the gradient of the solution blows up precisely at these points at finite time $t=T$, taking around each point the profile of sharply scaled degree 1 harmonic map with the blow-up speed
\begin{equation*}
\| \nabla u(\cdot,t)\|_{L^\infty(\mathbb{R}^2 ) } \sim |\ln T|^{-1} (T-t)^{-1} |\ln(T-t)|^2 \mbox{ \ for \ } t \in (0,T).
\end{equation*}

While blow-ups for Harmonic Map Flow (HMF, $a=1$)
have been constructed by
D\'avila, del Pino, and Wei \cite{17HMF}, substantial difficulties arise in the gluing construction due to the coupling between HMF and Schr\"odinger Map Flow (SMF) in LLG, and such
coupling produces both dissipative ($a>0$) and dispersive ($b \not =0$) features. A direct consequence of the presence
of dispersion is the {\em lack of maximum principle} for suitable quantities, which makes the analysis more delicate
even at the linearized level. The dispersion cannot be treated perturbatively, even in the dissipation-dominating
case $a/|b| \gg 1$, and one has to include this as part of the leading order. To overcome these difficulties, we utilize two key technical ingredients. First, for the resolution of the inner problem, we employ the  {\em distorted Fourier transform}, as developed by Krieger, Miao,  Schlag, and Tataru \cite{Krieger09Duke, KMS20WM}. Second, the linear theory for the outer problem is achieved by means of the sub-Gaussian estimate for the fundamental solution of the parabolic system in non-divergence form with coefficients of Dini mean oscillation in space ($\mathsf{DMO_x}$), which was proved by Dong, Kim, and Lee \cite{dong22-non-divergence}.

\end{abstract}

\maketitle

\tableofcontents

\section{Introduction and main results}

\subsection{Introduction}

Let $\mathcal{M}$ be an $m$-dimensional Riemannian manifold of metric $g$ and $S^2$ be the $2$-sphere embedded in $\RR^3$. The Landau-Lifshitz-Gilbert  equation (LLG)  on $\mathcal{M}$ is given by
\begin{equation}\label{LLG-eq-M}
	\begin{cases}
		u_t=  -a
		u\wedge (u\wedge \Delta_{\mathcal{M}} u
		)
		-b u\wedge \Delta_{\mathcal{M}} u
		 & \mbox{ in }~\mathcal M\times(0,T),
		\\
		u(\cdot,0) = u_0\in S^2  & \mbox{ in }~\mathcal M,
	\end{cases}
\end{equation}
where $a^2+b^2=1$,  $a\ge 0$, $b\in\RR$,
$\Delta_{\mathcal{M}}  = |g|^{-1/2}
\pp_{x_{\beta}}(g^{\alpha\beta} \sqrt{|g|} \pp_{x_{\alpha} } )$ is the Laplace-Beltrami operator, and $u=[u_1,u_2,u_3]^{\tr}$ is a $3$-vector with normalized length which is a mapping $u(x,t): \mathcal{M} \times(0,T) \rightarrow S^2$.
First formulated by Landau and Lifshitz \cite{LLG-LL} in 1935, LLG  \eqref{LLG-eq-M} is an important system modeling the effects of a magnetic field on ferromagnetic materials in micromagnetics, and it describes the evolution of spin fields in continuum ferromagnetism; Gilbert proposed the famous Gilbert damping later in  \cite{LLG-G}.  LLG
\eqref{LLG-eq-M} can be viewed as a bridge between the harmonic map flow (HMF) when $a=1$, $b=0$ and the Schr\"odinger map flow (SMF) when $a=0$, $b=-1$.

In the context of HMF, Struwe \cite{Struwe1985CMH} proved the existence and uniqueness of weak solution with at most finitely many singular points when $\mathcal M$ is a Riemann surface. Freire \cite{freire1995uniqueness} and Lin-Wang \cite{lin2010uniqueness} proved that Struwe's solution is unique in the class of weak solutions with decreasing Dirichlet energy. See also Freire \cite{Freire95} for further generalizations and Struwe \cite{StruweHD}, Chen-Struwe  \cite{89ChenStruwe} for higher dimensional cases. Chang, Ding and Ye \cite{Chang92} first proved the existence of finite-time blow-up solutions for HMF from disk into $S^2$. See also Coron-Ghidaglia \cite{CoronHMF}, Chen-Ding \cite{ChenDing1990}, Ding-Tian \cite{DT95CAG}, Qing \cite{Qing95CAG}, Wang \cite{Wang96HJM}, Qing-Tian \cite{QT97CPAM}, Lin-Wang \cite{LW98CVPDE}, Topping \cite{Topping04Annals} and the references therein for profound bubbling analysis and blow-up examples in related contexts.
Recent advancements in bubbling decompositions have been achieved by Jendrej-Lawrie \cite{JL2023CVPDE}, Jendrej-Lawrie-Schlag \cite{JLS2023Pi}. We refer to the monograph by Lin and Wang \cite{LinWangbook08} for comprehensive results on bubbling phenomena, regularity theory for harmonic maps and their heat flows.

In \cite{vandenBerg03}, via formal analysis, van den Berg, Hulshof, and King predicted the existence of blow-up solutions for the two-dimensional HMF into $S^2$ with quantized rates
\begin{equation}\label{vandenBerg-rates}
	\lambda_k(t) \sim
	(T-t)^k  |\ln(T-t)|^{-\frac{2k}{2k-1}},\quad k\in \mathbb{N}^+.
\end{equation}
There is a class of solutions taking the following special form
\begin{equation}\label{equi-intro}
u(x,t)=u(re^{i\theta},t)=\big(\cos(n\theta)\sin v(r,t),~\sin(n\theta)\sin v(r,t),~\cos v(r,t) \big),
\end{equation}
called $n$-equivariant solution with $n\in \mathbb Z$. While van der Hout \cite{van_der_Hout2003} excluded finite-time bubble trees in the $1$-equivariant class, finite-time blow-ups do exist in such case. For the case $\mathcal{M}=\RR^2$ and the target manifold is a revolution surface, using the profile of degree 1 harmonic map $Q_1$, Rapha\"el and Schweyer \cite{Raphael13, Raphael14} constructed finite-time blow-up solutions with rates \eqref{vandenBerg-rates} for all $k\ge 1$ in the $1$-equivariant class, where the initial data can be taken arbitrarily close to $Q_1$ in the energy-critical topology. For the case that $\mathcal{M}$ is a general bounded domain in $\RR^2$, D\'avila, del Pino and Wei \cite{17HMF} considered the general case without symmetry and constructed {\it non-equivariant} solutions which blow up at finitely many points with the type II rate \eqref{vandenBerg-rates} for $k=1$, and they further investigated the stability of blow-ups and reverse bubbling phenomena.  The construction in \cite{17HMF} can be generalized to the case  $\mathcal{M}=\RR^2$.

On the other hand, for SMF with $\mathcal{M}=\RR^2$, Merle, Rapha\"el and Rodnianski \cite{13MerleRaphaelRodnianski} constructed
the finite-time blow-up solution with the rate \eqref{vandenBerg-rates} for $k=1$ in the $1$-equivariant class. Analogous to the results of Krieger, Schlag, and Tataru \cite{Krieger08} for wave maps, Perelman \cite{Perelman14} constructed finite-time blow-up solutions with continuous rates, i.e., Krieger-Schlag-Tataru type.  The global well-posedness results in various critical spaces and space dimensions, and the dynamics of SMF near ground state have been studied widely in the works by Bejenaru, Ionescu, Kenig, and Tataru \cite{14-BejenaruTataru-book,SMF1,SMF2,SMF3,SMF4} and the references therein.

\medskip

For LLG, in the case $\mathcal{M}=\RR^3$, $a>0$, Alouges and Soyeur \cite{92Alouges}  proved the existence of weak solutions for \eqref{LLG-eq-M} and constructed infinitely many weak solutions. The existence of the weak solution to LLG has been established by Guo and Hong \cite{93BolingMinchunCVPDE} when $\mathcal{M}$ is a closed Riemannian manifold with $m\geq3$, while for the case that $\mathcal{M}$ is a closed Riemann surface, the weak solution was shown to be unique and regular except for at most finitely many points \cite{93BolingMinchunCVPDE}.  When $\mathcal{M}=\RR^2$ and the target manifold is a smooth closed surface embedded in $\RR^3$, approximation by discretization was used by Ko
 \cite{05KoJoy} to construct a solution of LLG that is smooth away from a two-dimensional locally finite Hausdorff measure.

In general, one cannot expect good partial regularity results for weak solutions in the higher dimensional case $m\geq 3$ without further regularity or energy minimizing assumptions. In fact, Rivi\`ere \cite{Tristan95} constructed weakly harmonic maps from the ball $B^3\subset \RR^3$ into $S^2$ for which the singular set is the entire closed ball $\overline {B^3}$, and this result can be generalized to higher dimensions. Note that harmonic maps also solve LLG. In a similar spirit to Chen-Struwe \cite{89ChenStruwe} for higher dimensional HMF, Melcher \cite{Melcher05} proved that for $\mathcal{M}=\R^m$ with $m=3$ there exists a global weak solution to LLG whose singular set has finite 3-dimensional parabolic Hausdorff measure. Later, this result was generalized to $m\le 4$ by Wang  \cite{Changyou06}. With the additional stability assumption for the weak solution, for $m\le 4$, Moser \cite{moser02partial} proved a better estimate for the singular set. The partial regularity of LLG \eqref{LLG-eq-M} for $m\ge 5$ still remains open.

For $\mathcal{M}=\RR^m$, the global existence, uniqueness, and decay properties for the solution of \eqref{LLG-eq-M} were established by Melcher \cite{C-Melcher12} for $m\ge 3$ with initial data $u_0$ close to a fixed point in $S^2$ in the $L^m$ norm. Lin, Lai, and Wang \cite{15-Changyou} generalized the result to Morrey space and $m\ge 2$. For $u_0$ away from a fixed point in $S^2$ with BMO semi-norm sufficiently small, Guti\'errez and de Laire \cite{19-Susana-Andre} proved the global existence, uniqueness, and regularity results for LLG.
We refer to a recent survey \cite{LLGsurvey} by de Laire for current developments on LLG.

The study of the dynamics for LLG with initial data close to harmonic maps is of special significance and can provide hints on the mechanism of singularity formation.
A series of works by Gustafson-Kang-Tsai \cite{Gustafson07,Gustafson08}, Guan-Gustafson-Tsai \cite{Guan-Gustafson-Tsai2009}, Gustafson-Nakanishi-Tsai \cite{Gustafson10} are devoted to the behavior of the solutions to LLG with $\mathcal{M}=\RR^2$ and with initial data $u_0$ close to the harmonic map in the $n$-equivariant class. They found, among other things,
that there is no finite-time blow-up for LLG and HMF with $u_0$ close to $n$-equivariant harmonic maps for $n\ge 3$ and $n\geq 2$, respectively. In sharp contrast to LLG and HMF, blow-ups do happen in the higher equivariant class for wave maps; see Rodnianski-Sterbenz \cite{Rodnianski-Sterbenz2010} and Rapha\"el-Rodnianski \cite{RR2012IHES}. Recently, interesting investigations were further extended to the near-soliton dynamics of the 2-equivariant SMF by Bejenaru-Pillai-Tataru \cite{d2Smap}, and to a complete classification of global dynamics, of the soliton resolution type, for equivariant-HMF with $n\geq3$ and energy-critical semilinear heat equations by Kim-Merle \cite{KimMerle}.

The singularity formation for LLG is an important and challenging topic. For the case that $\mathcal{M}$ is a compact manifold with or without boundary in dimensions $m=3,~4$, Ding and Wang \cite{07WangDing} obtained the existence of a smooth finite-time blow-up solution for LLG, and they stressed the importance of finite time singularity when $m=2$ in view of the seminal work of Chang-Ding-Ye \cite{Chang92} on HMF. However,  {\it neither} Bochner's formula {\it nor} Struwe's parabolic energy monotonicity formula is available in LLG, while these play a crucial role in the singularity analysis for HMF.
For $\mathcal{M}\subset \RR^2$, as an analogue of Qing \cite{Qing95CAG}
for HMF, Harpes \cite{04Harpes} gave descriptions of solutions to LLG \eqref{LLG-eq-M} near the singular points, but no example of finite-time singularity for LLG in $\mathbb{R}^2$ was given. For the energy critical case that $\mathcal M$ is a disk in $\R^2$, in an interesting paper \cite{VW13EJAM}, van den Berg and Williams predicted the existence of finite-time blow-up by formal asymptotic analysis supported with numerical simulations. For $\mathcal{M}=\RR^2$, Xu and Zhao \cite{xu2020blowup} rigorously constructed a finite-time blow-up solution to \eqref{LLG-eq-M} in a special $1$-equivariant class as in \eqref{equi-intro}.

\medskip

\subsection{Main results}

In this paper, we consider the case with target manifold $S^2$, $\mathcal{M}=\RR^2$, and positive damping parameter $a>0$. \eqref{LLG-eq-M} can then be written as
\begin{equation}\label{LLG-eq}
	\begin{cases}
		u_t=  a(\Delta u+|\nabla u|^2u)
		-b u\wedge \Delta u
		 ~&\mbox{ in }~\mathcal \R^2\times(0,T),
		\\
		u(\cdot,0) = u_0\in S^2  ~&\mbox{ in }~\mathcal \R^2.
	\end{cases}
\end{equation}
The Dirichlet energy
$
E[u]=\frac12\int_{\R^2} |\nabla u|^2
$
is non-increasing along smooth solutions to \eqref{LLG-eq} with sufficient decay as
$
\frac{d}{dt}E[u]=-a\int_{\R^2} |u\wedge\Delta u|^2.
$
In this sense, the parameter $a$ in the case $a>0$ can be regarded as a damping that produces dissipation in the energy.

We are interested in the general {\it non-radially symmetric setting} to \eqref{LLG-eq}, where the solution blows up in finite time taking the profile of {\it multiple bubbles}, and thus the solution is {\bf non-equivariant}. We remark that very little is known about the singularity formation beyond the equivariant class. The construction of non-equivariant solutions produces essential difficulties and a tremendous amount of careful analysis, as already observed by D\'avila-del Pino-Wei in HMF \cite{17HMF} and by Krieger-Miao-Schlag in wave maps \cite{KMS20WM}. The general case requires the control of {\it all the modes/angular momenta}, including the equivariant mode $0$, as well as the complicated interactions among bubbles.

Our construction is based on the following degree $1$ profile
\begin{equation}\label{W-def}
	W(y):=
	\frac{1}{|y|^2+1}
	\begin{bmatrix}
		2y_1
		\\ 2y_2
		\\
		|y|^2 -1
	\end{bmatrix},
	\quad
	y=(y_1,y_2)\in\RR^2.
\end{equation}
Clearly, $Q_{\gamma}W\big( \lambda^{-1} (x-\xi) \big)$ solves the stationary equation of \eqref{LLG-eq} for any $\xi\in\R^2$, $\la>0$, and any $\gamma$-rotation matrix around $z$-axis
\begin{equation}\label{Q-gamma-def}
	Q_{\gamma}:=
	\begin{bmatrix}
		\cos\gamma & -\sin\gamma & 0 \\[0.3em]
		\sin \gamma & \cos\gamma  & 0 \\[0.3em]
		0 & 0 & 1
	\end{bmatrix}.
\end{equation}
Denote $U_{\infty}=[0,0,1]^{\tr}$. Obviously, $W(\infty) = U_{\infty}$. Our main result is stated as follows.

\begin{theorem}\label{thm}
	
Assume $a^2+b^2=1$, $a>0$, $b\in \mathbb{R}$ in \eqref{LLG-eq}. Given $N\in \mathbb Z_+$ and arbitrary $N$ distinct points $q^{\J}\in \R^2$, $j=1,2,\dots, N$, for $T>0$ sufficiently small, there exists a smooth initial data $u_0$ such that the gradient of the solution $u$ to \eqref{LLG-eq} blows up at these $N$ points at finite time $t=T$ simultaneously. More precisely, the solution $u$ takes the sharply scaled degree $1$ profile around each point $q^{\J}$
\begin{equation*}
u(x,t)= -(N-1) U_{\infty} +  \sum\limits_{j=1}^{	N} Q_{\gamma_j(t)}W\bigg(\frac{x-\xi^{\J}(t)}{\lambda_j(t)}\bigg)+\Phi_{\rm per}(x,t)
\end{equation*}
with
\begin{align*}
&~\lambda_j(t) = \kappa_j^*  \lambda_*(t) (1+O(|\ln T|^{-\frac{1}{2}}) ),
\quad
\lambda_*(t)=\frac{|\ln T|(T-t)}{|\ln(T-t)|^2},\\
&~
\xi^{\J}(t) = q^{\J} + O( (T-t)^{1+\epsilon_0}),
\quad
\gamma_j(t) = \gamma_j^* + O(|\ln T|^{-\frac{1}{2}}),
\end{align*}
where $\kappa_j^*>0$ is a constant independent of $a,b$,
$\gamma_j^*\in [-\pi/2, \pi/2]$ is a constant depending on $a,b$, the constant
$\epsilon_0>0$ is sufficiently small, and the perturbation term $\Phi_{\rm per}$ satisfies
\begin{equation*}
\| \Phi_{\rm per} \|_{L^{\infty}(\R^2\times (0,T))} \ll 1,\quad
\| \nabla \Phi_{\rm per} (\cdot,t) \|_{L^{\infty}(\R^2) } \lesssim \lambda_*^{\epsilon_0 -1}(t).
\end{equation*}
\end{theorem}

\medskip

The solution constructed in Theorem \ref{thm} exhibits rather precise asymptotic behavior. Based on the analysis, a strong convergence and a weak-$*$ convergence of the Radon measure are shown.

\begin{corollary}\label{cor-intro}
The solution in Theorem \ref{thm} satisfies
\begin{equation*}%\label{strong-converge}
u(x,t)-u_*(x)-\sum_{j=1}^N Q_{\gamma_j(t)}\bigg[W\bigg(\frac{x-\xi^{\J}(t)}{\lambda_j(t)}\bigg)-U_{\infty}\bigg]\to 0 ~\mbox{ as }~ t\to T
\end{equation*}
in $H^{1}_{\rm loc}(\mathbb{R}^2)\cap L^{\infty}(\mathbb{R}^2)$ for some $
u_*(x)\in H^1_{\rm loc}(\R^2)\cap L^\infty(\R^2) $. Moreover,
$$|\nabla u(\cdot,t)|^2\,dx
\mathrel{\stackrel {w^{*}}{\rightarrow }}
|\nabla u_*|^2\,dx +8\pi \sum_{j=1}^N \delta_{q^{\J}} ~\mbox{ as }~ t\to T$$
as weak-$*$ convergence of the Radon measure.
\end{corollary}

\begin{remark}
\noindent
\begin{enumerate}
\item The damping term $a >0$ plays a crucial role in the construction, both in the near-singularity and remote regions. In the current framework, it seems to be difficult to obtain uniform estimates in the limit $a \downarrow 0$.

\medskip

\item For $j=1,2, \dots, N$, $$ |\nabla u(q^{\J},t)|\sim \frac{|\ln(T-t)|^2}{ \kappa_j^* |\ln T|(T-t)}\gg (T-t)^{-1/2},$$ exhibiting a type II blow-up pattern at each blow-up point.

\medskip

\item The stability of the non-equivariant blow-up remains an important open question. We conjecture that the blow-up solution in Theorem \ref{thm} persists for initial data staying within a manifold with higher codimension. The number of unstable directions might depend on rotation parameters $\gamma_j(t)$ ($j=1, 2, \dots,N$) and the freedom needed when adjusting the vanishing property of the outer solution.  In this regard, deriving various Lipschitz-dependences is a challenging problem as the outer problem is a quasilinear parabolic system.

\medskip

\item Due to the parabolic gluing method employed, the construction works as well for the case of smooth, bounded domain $\Omega\subset\R^2$ with Dirichlet or Neumann boundary conditions, and the main difference in the construction reflects in the fundamental solution in the sense of Dong-Kim-Lee \cite{dong22-non-divergence} with corresponding boundary conditions.
\end{enumerate}
\end{remark}

\medskip

\subsection{Strategy and novelties in the construction}

The proof of Theorem \ref{thm} is a gluing construction extending the {\it parabolic gluing method} to quasilinear system with dispersion. The parabolic gluing method was first established by Cort\'azar-del Pino-Musso \cite{Green16JEMS} and  D\'avila-del Pino-Wei \cite{17HMF} to investigate the singularity formation for parabolic PDEs.
The elliptic version, called {\it inner-outer gluing method}, was developed earlier by del Pino-Kowalczyk-Wei  \cite{DKW07CPAM, DKW11Annals} for the higher dimensional concentration of nonlinear Schr\"odinger equations and the counterexample to the De Giorgi's conjecture in large dimensions. The gluing method turns out to be rather versatile and has been generalized to various evolution equations later. For recent developments in gluing method, we refer to  D\'avila-del Pino-Musso-Wei \cite{GluingEuler,leapfrogging,EulerFilament}, del Pino-Musso-Wei \cite{TowerNLH,type25D,17type2}, Sire-Wei-Zheng \cite{17halfHMF,HMFwFB}, D\'avila-del Pino-Dolbeault-Musso-Wei \cite{gluingKS} on fluid equations, geometric flows and those stemming from mathematical biology and physics.

\medskip

Our study of the singularity formation for LLG is motivated by the endpoint case ($a=1$) for HMF \cite{17HMF}. However, substantial difficulties arise due to the coupling between HMF and SMF in LLG \eqref{LLG-eq}, and such coupling produces both dissipative ($a>0$) and dispersive ($b\neq 0$) features. A direct consequence of the presence of dispersion is the {\it lack of maximum principle} for suitable quantities, which makes the analysis more delicate even at the linearized level. The dispersion cannot be treated perturbatively even in the dissipation-dominating case $a/|b|\gg 1$, and one has to include this as part of the leading order. In our inner-outer gluing construction, new linear theories for both inner and outer problems need to be developed, taking into account the dissipation and dispersion concurrently. Based on these linear theories, weighted spaces that capture the precise asymptotics of solution in near-singularity zones and remote regions are devised carefully.

\medskip

\subsubsection{Distorted Fourier transform and re-gluing process in the inner problems}

The new linear theory for the inner problems is developed by analyzing each Fourier mode, which is the Fourier expansion of the complex form on each tangent plane of the bubble on $S^2$. Due to the {\it absence of maximum principle}, several steps combining energy methods, solving the elliptic equations, and Duhamel formulas, are employed to get rough upper bounds for each mode. More refined bounds at different Fourier modes are obtained by different methods.

\medskip

$\bullet$ {\it Mode $k$, $|k| \ge 2$}. One of the main challenges is the {\it convergence in $k$} when summing over all modes, {\it while maintaining a sufficiently fast decay}. By employing the rotation form of the right-hand side and applying a careful scaling argument, in conjunction with the regularity theory in the $\mathsf{DMO}_x$ space, we manage to extract the negative power of $|k|$. Combining these with the re-gluing process, we derive an upper bound that depends explicitly on $k$. This bound is sufficient to ensure the convergence of the summation over all modes. See Subsection \ref{hmode-subsec} for further details. In order to refine the bounds and get better {\it pointwise decay estimates}, we perform another gluing procedure, called {\it re-gluing process}, at all the modes except mode $-1$. The re-gluing process was first used in the analysis of linearization of HMF at mode $0$ in \cite{17HMF}, and here we generalize this technique to all modes except mode $-1$. The re-gluing process aims to improve the time decay rate in the apriori estimates and provides more flexibility in choosing parameters to devise the topologies to solve the gluing system.

\medskip

$\bullet$ {\it Mode $0$ and mode $1$.} In contrast to the mode $k$, $|k|\geq 2$, the elliptic operators for mode $0$ and mode $1$
admit {\it bounded} kernels function with {\it decay} (cf. \eqref{scalr-Z}), for which orthogonality conditions are required to recover the decay information of the right-hand side.
These orthogonality conditions and
the use of the re-gluing
lead to perturbation terms $c_{*0}(\tau)$, $c_{*1}(\tau)$, which make the reduced equations, especially for mode $1$, and non-orthogonal inner problems \eqref{inner-eq-2} more complicated. See Proposition \ref{Re-m0-prop}, Proposition \ref{qd24July12-8-prop}, and \eqref{orth1-eqr}.

\medskip

$\bullet$ {\it Mode $-1$.} The use of the above method does give a solution, but this solution deteriorates in the innermost region and is not sufficient for the gluing to be implemented. The reason is that by \eqref{qd240729-1} and \eqref{cal-L-ope}, mode $-1$ can be roughly viewed as a heat equation in $\mathbb{R}^{2}$ near spatial infinity, and the estimates obtained are worse than any other mode as one cannot gain spatial decay by the Duhamel's formula.
Instead, motivated by the groundbreaking work of Krieger, Miao, and Schlag \cite{KMS20WM} on the stability of blow-up for wave maps {\it beyond} the equivariant class, we utilize the powerful and versatile techniques of the {\it distorted Fourier transform} for the dealing of mode $-1$.

\medskip

The {\it distorted Fourier transform} has been successfully developed and applied in various problems. The general framework and theories on the spectral analysis of the half-line Schr\"odinger operator with strongly singular potentials have been developed by Gesztesy and Zinchenko \cite{GZ06Mana}. Schlag \cite{Schlag-2007} established the Littlewood-Paley theory for resonant Schr\"odinger operators.
Of significant importance are its applications in the singularity formation, dispersive estimates and asymptotic stability; see fundamental works by Krieger-Schlag \cite{KS06JAMS, Krieger-Schlag-JEMS2009} for constructing stable and stable blow-up manifolds for Schr\"odinger equations, and Krieger-Schlag-Tataru \cite{Krieger08, Krieger09Duke, KST09AIM}, Krieger-Schlag \cite{Krieger14JMPA}, Donninger-Huang-Krieger-Schlag \cite{Donninger-Huang-Krieger-Schlag2014} for the blow-ups in critical wave equations, wave maps and hyperbolic Yang-Mills equation. Schlag-Soffer-Staubach \cite{Schlag-Soffer-Staubach2010-I, Schlag-Soffer-Staubach2010-II} proved dispersive estimates for Schr\"odinger and wave evolutions on Riemannian manifolds with conical ends; we refer to a good survey \cite{SchlagSurvey} by Schlag in this regard. Donninger-Schlag-Soffer \cite{Donninger-Schlag-Soffer2011, Donninger-Schlag-Soffer2012} investigated the stability and decay estimates in general relativity.
Krieger-Nakanishi-Schlag \cite{Krieger-Nakanishi-Schlag2012} classified the global dynamics of Klein-Gordon equations with energy above that of the ground state slightly.
For the application in the dispersive decay and scattering theory of Schr\"odinger equations, wave equations, and wave maps,
we refer to Goldberg-Schlag \cite{Goldberg-Schlag2004}, Costin-Schlag-Staubach-Tanveer \cite{Costin-etc2008}, Donninger-Schlag \cite{Donninger-Schlag2010}, Costin-Donninger-Schlag-Tanveer \cite{Costin-Donninger-Schlag-Tanveer2012}, Lawrie-Schlag \cite{Lawrie-Schlag2013} and the references therein. Recently, there are growing interests in asymptotic stability of solitons/kinks and blow-ups in many PDEs via distorted Fourier transform. See, for instance, Krieger-Miao \cite{krieger2020stability} and Krieger-Miao-Schlag \cite{KMS20WM} for wave maps in $2+1$ dimensions, Germain-Pusateri \cite{Germain2022}, L\"uhrmann-Schlag \cite{LuhrmannDuke, Luhrmann2024}, and Lindblad-L\"{u}hrmann-Schlag-Soffer \cite{Lindblad-Luhrmann-Schlag-Soffer2023} for Klein-Gordon equations, Bejenaru-Pillai-Tataru \cite{d2Smap} for SMF, Palacios-Pusateri \cite{PusateriGL} for Ginzburg-Landau evolutions, and Chen-L\"{u}hrmann \cite{chen2024asymptotic} for the sine-Gordon equation.

\medskip

Using the distorted Fourier transform, we develop linear theory at mode $-1$ with or without orthogonality conditions. The version with orthogonality removes the logarithmic loss compared to the one without orthogonality.  See Section \ref{sec-linearinner-1} for more details. In this paper, for mode $-1$, we only use the one without orthogonality since the introduction of two new modulation parameters corresponding to rotations will further complicate the interactions, and we control the logarithmic loss by H\"older continuity and the well-designed vanishing property of the outer solution. We note that the linear theory developed is in the general range including the purely dissipative case $a=1$, $b=0$, and this seems to be the first application of the distorted Fourier transform in the {\it parabolic} setting.

\medskip

\subsubsection{Regularity estimates in the $\mathsf{DMO_x}$-class for the outer problem}

The outer problem \eqref{outer-eq} turns out to be a {\it quasilinear parabolic system} in {\it non-divergence form}. Different from the outer problem in HMF, the one in LLG is a coupled system and thus cannot be solved componentwisely. The leading coefficients \eqref{B-matrix} of the outer problem \eqref{outer-eq} include the blow-up profile. So one cannot expect good H\"older continuity for \eqref{B-matrix} and has to work in a weaker class. On the other hand, estimates for higher-order derivatives are needed to control error terms. These suggest that the regularity class must be chosen rather carefully and precisely, roughly weaker than $C^\alpha$ but stronger than $C^0$.  The linear theory for the outer problem is achieved by means of the {\it sub-Gaussian estimate} for the fundamental solution of the parabolic system in non-divergence form with coefficients of {\it Dini mean oscillation} in space ($\mathsf{DMO_x}$), which was proved by Dong, Kim, and Lee \cite{dong22-non-divergence}. We introduce Dini mean absolute oscillation in space ($\mathsf{|DMO|_x}$), which is a subspace of $\mathsf{DMO_x}$. Under some weak assumptions, the functions in $\mathsf{|DMO|_x}$ are closed under arithmetic (see Lemma \ref{DMO-|DMO|-lem}). This property makes it more convenient to verify that the leading coefficients \eqref{B-matrix} of the outer problem belong to $\mathsf{|DMO|_x}$, and we note that the {\it type II speed} as in Theorem \ref{thm} ($\la_j(t)\lesssim (T-t)^{\frac{1}{2} +\epsilon}$ with a constant $0<\epsilon \ll 1$) plays a rather important role here.

\medskip

The estimates of second-order derivatives are necessary to control the dispersive part, i.e., error terms produced by $bu\wedge \Delta u$ in the equation, and we need rather precise weighted estimates for the gluing. In fact, the weights are eventually chosen very carefully, reflecting in finding a solution in the system for constants measuring the weights. See the end of Subsection \ref{sol-glu+orth-sec}. To get the {\it quantitative estimates} of second-order derivatives of the inner solutions, we first analyze the representation form of the outer solution via sub-Gaussian and then adopt the regularity theory with $\mathsf{DMO_x}$ coefficients developed by Dong, Escauriaza, and Kim in \cite{dong21-C0-para}.

\medskip

\subsubsection{Improvement of slow decay and tricks used in the interacting error terms}

Another aspect of the construction is the dealing with slow decaying errors, usually present in lower dimensional problems. The improvement of these slow decaying errors involves finding good global corrections (non-local in the corresponding modulation parameters), which in turn make the dynamics for the parameters in the corresponding mode non-local. In the context of LLG, the mode with slow decaying error that we shall deal with is mode $0$, which corresponds to the invariance of scaling and rotation around the $z$-axis. To capture the precise blow-up dynamics, the global correction at mode $0$ should be rather explicit. However, due to the aforementioned structure of the outer problem, one cannot improve the error by solving the linearized system directly and has to extract part of the parabolic system instead, i.e., the {\it approximate} parabolic system. It turns out that the combination of the new errors produced by the global corrections and the remainder in the parabolic system together make the non-local equations for the scaling parameter $\lambda_j$ and rotational parameter $\gamma_j$ a well-structured complex system. See Section \ref{sec-orthogonal}.

\medskip

The construction of multiple bubbles involves carefully analyzing complicated and lengthy interactions. The unit-length property of the map $|u| = 1$ with multiple bubbles also produces delicate interactions. See \eqref{u-def} and \eqref{A-def}. Fortunately, we find a subtle {\it cancellation} in the estimate of an error term $\Delta_x U_* -2\left(U_* \cdot \nabla_x U_*\right)\cdot \nabla_x U_*$, which is essential for finding well-designed topologies to complete the construction. See Remark \ref{qd240727-1-rem}.

On the other hand, we adopt a trick that we call {\it $U_*$-operation} (see \eqref{U-direc-trick}), which can adjust errors in the $U_*$-direction for the multi-bubble case and can thus simplify analysis. This idea first appeared in D\'avila-del Pino-Wei \cite{17HMF} in the case of a single bubble for HMF, and we modify this in the context of LLG. See also Krieger-Miao-Schlag \cite{KMS20WM} for a similar argument for the wave map of a single bubble.

\medskip

\subsection{Comments on other related problems and techniques}

Well-posedness and singularity formation are also central topics in dispersive and hyperbolic PDEs. We refer to the books \cite{ShatahStruwe} by Shatah and Struwe and  \cite{Tao} by Tao in the hyperbolic and dispersive set-ups.
There are numerous profound studies in wave equations and general hyperbolic equations.
For the regularity theory of wave maps, Klainerman-Machedon investigated in \cite{Klainerman-Machedon1993, Klainerman-Machedon1995-Duke} low-regularity solutions and proved in \cite{Klainerman-Machedon1997} well-posedness for initial data with optimal regularity; see also Klainerman-Selberg \cite{Klainerman-Selberg1997}. Tataru proved the global existence and scattering in Besov spaces for wave maps in $n+1$ dimensions with $n\geq 4$ in \cite{Tataru1998} and with
$n=2,~3$ in \cite{Tataru2001}. In \cite{Tao2001-waveI, Tao2001-waveII}, Tao achieved the global regularity for wave maps in the critical Sobolev space.
Finite-time blow-up results of wave maps were established by Rodnianski-Sterbenz \cite{Rodnianski-Sterbenz2010},
Rapha\"{e}l-Rodnianski \cite{RR2012IHES}. For Krieger-Schlag-Tataru type blow-ups,
we refer to Perelman \cite{Perelman14},
Krieger-Schmid \cite{krieger2024finite-I, krieger2024finite-II},
Bahouri-Marachli-Perelman \cite{Bahouri-Marachli-Perelman2019}; see also Donninger-Krieger \cite{Krieger13MA} for infinite-time versions of Krieger-Schlag-Tataru type, and
Pillai \cite{Pillai2023} for more general global solutions.

\medskip

Concerning the classification results of wave maps and energy-critical wave equations,  Duyckaerts-Jia-Kenig-Merle \cite{Duyckaerts-Jia-Kenig-Merle2018} studied the small blow-up solutions via the channel of energy-type inequalities developed earlier in \cite{DJKM2017GAFA} for critical wave equations. Jendrej-Lawrie \cite{Jendrej-Lawrie2018} classified the two-bubble dynamics by the Kenig-Merle type concentration-compactness techniques together with modulation method. For the soliton resolution, we refer to Duyckaerts-Kenig-Martel-Merle \cite{Duyckaerts-Kenig-Martel-Merle2022}, Duyckaerts-Kenig-Merle \cite{Duyckaerts-Kenig-Merle2023}, Jendrej-Lawrie \cite{Jendrej-Lawrie2022MRL, Jendrej-Lawrie2023Ann.PDE}, Collot-Duyckaerts-Kenig-Merle \cite{Collot-Duyckaerts-Kenig-Merle2024} and their references; see also Krieger-Nakanishi-Schlag \cite{Krieger-Nakanishi-Schlag2014},
Krieger-Wong \cite{Krieger-Wong2014} for threshold dynamics.

\medskip

On the other hand, powerful modulation techniques have been widely developed by Collot, Merle, Rapha\"el, Rodnianski, Szeftel and collaborators in \cite{MR2005Annals,RR2012IHES,MMR2014Acta,CMR2020JAMS,MRRS2022Inventiones,MRRS2022Annals1,MRRS2022Annals2} and their references to study singularity formations for various dispersive, hyperbolic, parabolic equations and fluid dynamics.

\bigskip

\subsection{Main steps of the construction}

Due to the complexities and technicalities in the construction, in this subsection we sketch a roadmap of the major steps and present detailed illustrations of the ideas mentioned above.

\medskip

\noindent
$\bullet$ {\bf Multi-bubble ansatz.}
The construction begins with a careful choice of first approximation. Since the target is $S^2$, one has to choose some profile for {\it multiple bubbles}, which is relatively reasonable to analyze. In Subsection \ref{sec-first-appr}, we take the first approximation as
\begin{equation*}
U_*= -(N-1) U_{\infty} +  \sum\limits_{j=1}^{	N} U^{\J}(x,t),
\mbox{ \ where \ }
U^{\J}(x,t):=Q_{\gamma_j}W\bigg(\frac{x-\xi^{\J}}{\la_j}\bigg).
\end{equation*}
Notice that $|U_*| = 1 +o(1)$ at any space-times as those bubbles are essentially separated, assuming \eqref{U*-norm}. Denote the error function as
\begin{equation*}
	S[\fbf] :=  - \partial_t \fbf + a(\Delta_x \fbf + |\nabla_x \fbf|^2 \fbf)-b \fbf\wedge\Delta_x \fbf
	\mbox{ \ for \ } \fbf=[f_1,f_2,f_3]^{\tr}\in \R^3.
\end{equation*}
The error $S[U_*]$ contains slowly decaying terms $\sum_{j=1}^N\mathcal E_0^{\J} $ (see \eqref{Ej-0}),
which correspond to the errors corresponding to the invariance of scaling and rotation around $z$-axis (both belong to Fourier mode $0$ in complex notation). Here, the slow decay is in the sense that the spatial decay is not fast enough to apply the inner linear theory developed later on.

\medskip

\noindent
$\bullet$ {\bf Global corrections by approximate parabolic systems.}
In Subsection \ref{global-cor-sec}, to improve the spatial decay of the errors at the remote region, we add well-designed global corrections around each bubble. Since the operator
\begin{equation*}
-\partial_t+(a-bU^{\J}\wedge)\Delta_x
\end{equation*}
depends on the blow-up profile $U^{\J}$ as well as the parameters $\la_j$, $\gamma_j$, and $\xi^{\J}$, one cannot expect an explicit representation formula. However, the explicit representation of global corrections is crucial for capturing the blow-up dynamics. Instead, we consider an {\it approximate} parabolic operator
\begin{equation*}
-\partial_t+(a-bU_{\infty}\wedge)\Delta_x
\end{equation*}
and add the global corrections $\Phi_0^{*\J}$ around the blow-up point $q^{\J}$ with
\begin{equation*}
-\partial_t \Phi_0^{*\J}+(a-bU_{\infty}\wedge)\Delta_x \Phi_0^{*\J} +\mathcal E_0^{\J}\approx 0.
\end{equation*}
As mentioned earlier, the difference, involving $U^{\J}-U_\infty$, also serves as one of the leading parts in the reduced problems through orthogonality conditions.

\medskip

We regularize the global corrections
with an extra factor $r_j^{3} (r_j^{3}+ \la_j^{3})^{-1} $ to avoid the non-smoothness in spatial variables of errors caused by $\sin(\theta_j)$ and $\cos(\theta_j)$ in mode $0$ (see \eqref{qd240728-1}). Here $\la_j y^{\J}=x-\xi^{\J}= r_j e^{i\theta_{j}}$. More specifically, since the terms like $\langle y^{\J} \rangle^{c} \sin(\theta_j)$ with $c\in \mathbb{R}$ are not in $\mathsf{DMO_x}(B_0(1))$ in terms of the spatial variable $y^{\J}$, we need to multiply the power of $|y^{\J}|$ to avoid low regularity in $y^{\J}$. This is important for deriving the second-order derivative estimate for the inner problems, and the remainder of the terms produced by this discrepancy need to be analyzed as well.

\medskip

Subsection \ref{upp-global-cor-sec} contains the upper bound of global corrections. In Subsection \ref{nonlocal-err-set}, we compute the new errors with corrections given by those created by $\Phi_0^{*\J}$ and the remainder $b(U_{\infty}-U^{\J})\wedge \Delta_x \Phi^{*{\J}}_0$. The accurate form of errors is rather important in analyzing the reduced equations of mode $0$ and mode $1$.

\medskip

 \noindent
 $\bullet$ {\bf Formulation of the inner-outer gluing system.} In Subsection \ref{system-newerror-sec}, we then perturb around $U_*$ and look for solution to LLG in the form
\begin{equation*}
	u=(1+A)U_*+\Phi-(\Phi\cdot U_*)U_*
\end{equation*}
with some perturbation terms $\Phi$ and $A$, where $\Phi$ is taken as
\begin{equation*}
	\Phi(x,t) := \sum_{j=1}^{N}\left( \eta_R^{\J}(x,t) Q_{\gamma_j}\Phi_{\rm in}^{\J}(y^{\J},t)
	+
	\eta_{d_q}^{\J}(x,t) \Phi^{*{\J}}_0(|x-\xi^{\J}|,t)\right)+\Phi_{\rm out}(x,t).
\end{equation*}
Here $\Phi_{\rm in}^{\J} \cdot W^{\J} \equiv 0$; $\eta_R^{\J}$ and $\eta_{d_q}^{\J}$, defined in \eqref{qd24Apr12-3}, are suitable cut-off functions near $q^{\J}$; $\Phi_{\rm in}^{\J}$ and $\Phi_{\rm out}$ will be solved in the inner-outer gluing system; $A$ is a real-valued function depending on $\Phi$ to ensure $|u|\equiv 1$ (see \eqref{A-def}). Note that part of the interactions between bubbles get encoded in the scalar function $A$.

\medskip

By elaborated calculations for $S[u]$ in Subsections \ref{system-newerror-sec} and \ref{N-simplify-sec} with the application of $U_*$-operation, for $S[u]=0$, it suffices to solve the inner-outer gluing system \eqref{outer-eq}-\eqref{inner-eq} in Subsection \ref{gluing-sys-sec}. Perturbation terms $\sum_{m=1}^{N} \sum_{n=1}^3 c_{mn}  \vartheta_{mn}(x) $ are added in the initial data to achieve the vanishing property for the outer problem \eqref{outer-eq} at the blow-up points. This is important in several estimates needed in the gluing procedure, and might be viewed as extra modulation parameters related to codimensional stability; see also the role of the rotational parameters in Krieger-Miao-Schlag \cite{KMS20WM}.

\medskip

For the full system above, finding blow-up of LLG at multiple points now gets reduced to finding well-behaved inner and outer solutions such that the gluing procedure can be implemented. In other words, we need to devise appropriate weighted topologies in which the gluing system becomes weakly coupled and thus can be solved by the fixed-point argument.

\medskip

Subsection \ref{sec-topologies} includes the weighted topologies for the inner and outer problems. In Subsection \ref{ortho-non-in-sec}, we decompose the inner problem \eqref{inner-eq} into orthogonal and non-orthogonal parts \eqref{inner-eq-1} and \eqref{inner-eq-2}. The principle of the allocation of the right-hand side of \eqref{inner-eq} is to make the reduced equations \eqref{ortho-eq} more convenient to handle, while those terms in the non-orthogonal part \eqref{inner-eq-2} carry faster time decay with suitably chosen parameters.

\medskip

\noindent
$\bullet$ {\bf Reduced equations.} In Section \ref{sec-orthogonal}, we reformulate the reduced equations \eqref{ortho-eq} into \eqref{orth0-eq-r} and \eqref{orth1-eqr}, and then present the linear theorem for non-local reduced equations. These reduced equations determine the blow-up dynamics. The non-local feature of reduced equations in mode 0 (see \eqref{orth0-eq-r}) gets inherited from the global corrections $\Phi^{*{\J}}_0$ as the global corrections are essentially for mode $0$.

\medskip

Here, the complex system involving both $\la_j$ and $\gamma_j$ might be a rather sophisticated form due to the dissipation-dispersion interaction. However, it turns out that the contribution of both $\Phi^{*{\J}}_0$ and the remainder $b(U_{\infty}-U^{\J})\wedge \Delta_x \Phi^{*{\J}}_0$ in the reduced equations at mode $0$ results in the following well-structured non-local problem
\begin{equation*}
	\int_{-T}^{t-\la^2_j(t)}\frac{\dot p_j(s)}{t-s}ds\sim a-ib
	\mbox{ \ with \ } p_j(t) :=\la_j(t)e^{i\gamma_j(t)}.
\end{equation*}
This system was first found and handled by D\'{a}vila, del Pino, and Wei in \cite[Propositions 6.5 and 6.6]{17HMF} for HMF ($b\equiv 0$). Surprisingly, this comes with a similar form in LLG with the presence of dispersion $(b\neq 0)$.

\medskip

\noindent
$\bullet$ {\bf $\mathsf{DMO_x}$ and the role of type II blow-up in the outer problem.} In Section \ref{sec-linearouter}, we develop the linear theory for the outer problem.
The outer problem \eqref{outer-eq} is a {\it quasilinear parabolic system}. In Subsection \ref{DMO-sec}, we give basic concepts of $\mathsf{DMO_x}$ and $\mathsf{|DMO|_x}$ spaces, and regularity results with $\mathsf{DMO_x}$ coefficients. Subsection \ref{fund-out-sec} gives the estimates of the fundamental solution for a parabolic system with $\mathsf{DMO_x}$ coefficients.

\medskip

In Subsection \ref{sec-DMO}, we show that the outer system satisfies the Legendre-Hadamard ellipticity using the assumption $a>0$ and prove that the leading coefficients of the outer problem \eqref{outer-eq} belong to $(\mathsf{|DMO|_x} \cap L^{\infty})(\mathbb{R}^{2}\times(0, T) )$ under suitable choice of topologies and parameters. Here, the fact that the scaling parameter is of {\it type II}, or in other words $$\lambda_j(t) \lesssim (T-t)^{\frac{1}{2} +\epsilon}$$
with a constant $0<\epsilon \ll 1$, is crucial to ensure $U_{*}\in \mathsf{|DMO|_x}(\mathbb{R}^{2}\times(0,T))$.

\medskip

To obtain the estimates of the outer problem, in Appendix \ref{convo-finite-tim-sec}, we give general convolution estimates in finite time. Then, the topology of the outer problem is derived in Appendix \ref{out-top-Deri-sec}. Due to the complicated interaction of different bubbles, lots of efforts are devoted to the estimates of the right-hand side of the outer problem in Appendix \ref{G-est-sec}, where a delicate {\it cancellation} for $$\Delta_x U_* -2\left(U_* \cdot \nabla_x U_*\right)\cdot \nabla_x U_*$$ is essential to find suitable parameters to close the fixed-point argument. See Remark \ref{qd240727-1-rem}.

\medskip

\noindent
$\bullet$ {\bf Linear theory for the inner problems.} In Section \ref{sec-linearinner}, we develop the linear theory for the inner problems.
We project the linear problem of the inner
problem to the tangent plane of $W(y)$ to transform the parabolic system into a complex-valued parabolic equation. Then, we expand the equation into Fourier modes and analyze each mode $k$ $(\in \mathbb{Z})$.  The linearized operator at mode $k$ has the form \eqref{inner-Fourier-form}. For all modes $k\in\mathbb Z\backslash\{-1\}$, good inner solutions are found by the following strategy.

\medskip

\begin{itemize}
\item[Step 1:] We first use energy methods to get a rough pointwise upper bound for the inner solutions;

\smallskip

\item[Step 2:] Next, we solve the corresponding elliptic equations and use Duhamel's formula and orthogonality conditions especially for mode $0$ and mode $1$, to refine the pointwise bounds and gain decay estimates;

\smallskip

\item[Step 3:] Finally, we further perform a re-gluing procedure to obtain better estimates in the innermost region.
\end{itemize}

\medskip

For mode $k$, $|k|\ge 2$, techniques are developed to specify the dependence on $k$ rather explicitly in the estimates for the convergence of summation of all the modes. See Subsection \ref{hmode-subsec}.

\medskip

The approach that we use for mode $0$ and mode $1$ is different from mode $k$, $|k| \ge 2$. The motivation is behind the decay of the corresponding bounded kernel functions (see \eqref{scalr-Z}). Some information on the spatial decay gets lost when constructing properly behaved inner solutions of modes $0$ and $1$. However, with the adjustment of modulation parameters  $\lambda_j$, $\gamma_j$ (mode $0$), $\xi^{\J}$ (mode $1$),  leading to the reduced equations \eqref{ortho-eq}, the spatial decay (in the intermediate gluing region) of solutions is recovered and is sufficient for the gluing construction after the re-gluing procedure. The re-gluing however produces tails $c_{*0}^{\J}(\tau_j(t ))$ in the non-orthogonal part \eqref{inner-eq-2}  and $c_{*1}^{\J}( \tau_j(t) )$ in the reduced equations \eqref{orth1-eqr} for mode $1$.
For more flexibility in the choice of parameters to handle $c_{*0}^{\J}(\tau_j(t ))$ and $c_{*1}^{\J}( \tau_j(t) )$, we clarify clearly the requirements on parameters in Proposition \ref{Re-m0-prop} and Proposition \ref{qd24July12-8-prop}.

\medskip

For mode $-1$, we use distorted Fourier transform to develop two versions of linear theory, with or without orthogonality conditions imposed on the right-hand side. We first derive the representation formula via distorted Fourier transform and then take advantage of the spectral properties to obtain precise weighted pointwise estimates for the inner solution. These rely on the estimates for the associated generalized eigenfunction and density of the spectral measure. See Section \ref{sec-linearinner-1}.

\medskip

\noindent
$\bullet$ {\bf Completing the proof of Theorem \ref{thm} and Corollary \ref{cor-intro}.} Finally, we solve the gluing system and the reduced equations in Subsection \ref{sol-glu+orth-sec} by the Schauder fixed-point argument. Here, the leading term of $p_j$ is given by Proposition \ref{keyprop}, which depends on a given function $Z_*(x)$ (see \eqref{qd240729-3}). The summation of all the modes yields the pointwise estimate of inner problems. Then, we go back to the original parabolic system of inner problems to deduce second-order estimates in a precise manner, and this is done by the regularity theory with $\mathsf{DMO_x}$ coefficients and a scaling argument. Convergence results in Corollary \ref{cor-intro} are derived in Subsection \ref{converge-sec}.

\medskip

The rest of this paper is devoted to the proofs of Theorem \ref{thm} and Corollary \ref{cor-intro}.

\bigskip

\section{Notations and preliminaries}\label{sec-nota}

In this section, we list some notations and preliminaries that we shall use repeatedly throughout this paper. For convenience, the index for terminologies and symbols is given in Appendix \ref{Index}.

\begin{itemize}

\item Denote $\mathbb{R}^{d+1} =\mathbb{R}^d \times
\mathbb{R}$, where $\mathbb{R}^d$ and $\mathbb{R}$ are domains of spatial and time variables respectively.
	
\item We assume $c_1\lesssim c_2$ if there exists a constant $C>0$ such that $c_1\le C c_2$. Denote $c_1\sim c_2$ if $c_1\lesssim c_2 \lesssim c_1$. Denote $f=O(g)$ if $|f|\lesssim g$. All constants stated in the paper are independent of $T$. For $x\in \R^d$, denote $\langle x \rangle := \sqrt{|x|^2+1}$. For $c>0$, $c\ll 1$ ($c\gg 1$) denotes $c$ sufficiently small (large).

\item For any $c \in \R$, we use the notation $c-$ to denote a constant less than $c$ that can be chosen arbitrarily close to $c$. We denote $c_{+}=\max\left\{ c, 0\right\}$.

\item Write the indicator
function $\1_{\Omega}(x)$ of a set $\Omega$ as
$\1_{\Omega}(x)
=1$ if $x\in \Omega$ and $\1_{\Omega}(x)
=0$ if $x\notin \Omega$.
We will use $\1_{\Omega}$ to
denote $\1_{\Omega}(x)$ if there is no ambiguity.

\item Set $\eta(x)$ as a smooth cut-off function satisfying $0\le \eta(x) \le 1$,
$\eta(x) = 1$ if $|x|\le 1$ and $\eta(x) = 0$ if $|x|\ge 2$.

\item Given $a>0$, $b\in \mathbb{R}$ satisfying $a^2+b^2=1$, denote $\Gamma_d^{\natural}$ as the fundamental solution of $\pp_{t} u = (a-ib) \Delta u$ in $\mathbb{R}^d$,
and $\Gamma_d^{\natural}$ is given by
\begin{equation}\label{Gammad-def}
	\Gamma_d^{\natural}(x,t) =
(a-ib)^{-\frac{d}{2}}	(4\pi t)^{-\frac{d}{2}} e^{-\frac{|x|^2}{4(a-ib) t }}.
\end{equation}
Obviously, $ |\Gamma_d^{\natural}(x,t) | \le (4\pi t)^{-\frac{d}{2}} e^{-\frac{a |x|^2}{4 t }}$.

\item  Given a fundamental solution $\Gamma(x,y,t,s)$ for a parabolic system in $\R^d$ and some admissible functions $f(x)$, $h(x,t)$, denote
\begin{equation*}
(\Gamma * f)(x,t,t_0) :=
\int_{\RR^d}
\Gamma(x,y,t,t_0) f(y) d y ,
\quad
(\Gamma ** h)(x,t, t_0) :=
\int_{t_0}^t\int_{\RR^d}
\Gamma(x,y,t,s) h(y,s) d y d s.
\end{equation*}
We usually omit the initial time $t_0$ if there is no ambiguity from the context.
\item For any vector $\vec{a}=[a_{1}, a_{2},a_{3}]^{\tr} \in\RR^3$, where ``$[\cdots]^{\tr}$'' means the transpose of a matrix, and we identify $[a_{1}, a_{2},a_{3}]^{\tr}:= [a_{1} +ia_{2},a_{3}]^{\tr}$.
For $\vec{b}=[b_{1}, b_{2},b_{3}]^{\tr}\in\RR^3$, it is easy to see that $\vec{a} \cdot \vec{b}= \mathrm{Re}\left[ (a_1+ia_2) (b_1-ib_2)\right] + a_3 b_3$.

\item For any matrix $A = (a_{ij})_{n\times m}$, denote $|A| = \big(\sum\limits_{i=1}^n \sum\limits_{j=1}^m |a_{ij}|^2 \big)^{1/2}$.

\item Given functions $f(x,t)$ and $x=x(t)$, denote $\pp_{t} f(x(t),t)=(\pp_{t} f)(x(t),t) $ and $\pp_{t}(f(x(t),t))= (\pp_{t}f)(x(t),t) + \dot{x}(t) \cdot  (\nabla_x f)(x(t),t)  $.

\item Denote
\begin{equation*}
\vec{w} \cdot \nabla \vec{v} :=
\begin{bmatrix}
\vec{w}  \cdot \nabla v_1, \
\vec{w}  \cdot \nabla v_2, \
\vec{w}  \cdot \nabla v_3
\end{bmatrix}^{\tr}
\mbox{ \ for \ }
\vec{v}=[v_1,v_2,v_3]^{\tr} \in C^1(\R^2,\R^3), ~\vec{w} \in \R^2;
\end{equation*}
\begin{equation}\label{dot2}
\begin{bmatrix}
a_{1}
\\
a_{2}
\\
a_{3}
\end{bmatrix}
\cdot
\begin{bmatrix}
b_{11} & b_{12}
\\
b_{21} & b_{22}
\\
b_{31} & b_{32}
\end{bmatrix}
:=
\begin{bmatrix}
\sum\limits_{k=1}^3 a_{k}	b_{k1}
,
 \sum\limits_{k=1}^3 a_{k} b_{k2}
\end{bmatrix}
;
\quad
\vec{v}\wedge :=
\begin{bmatrix}
0 & -v_3 & v_2
\\
v_3 & 0 & -v_1
\\
-v_2 & v_1 & 0
\end{bmatrix}.
\end{equation}
\end{itemize}

\bigskip

We consider the Landau-Lifshitz-Gilbert equation given in \eqref{LLG-eq}. The steady-state equation of \eqref{LLG-eq} is the harmonic map equation.
$W(y)$ given in \eqref{W-def} is the least energy harmonic map, which solves the harmonic map equation. Since we shall consider the case of multiple bubbles, subscript ``${}_{j}$'' or superscript ``$\J$''  will be used to distinguish different bubbles and their associated tangent planes. In the (rescaled) polar coordinates around $\xi^{\J}=(\xi^{\J}_1,\xi^{\J}_2)\in\R^2$, denote
\begin{equation}\label{polar-coor}
y^{\J}=\frac{x-\xi^{\J}}{\la_j} = \rho_{j} e^{i\theta_{j}},\quad x=\xi^{\J}+\la_{j}\rho_{j} e^{i\theta_{j}},\quad \rho_j=|y^{\J}|,
\quad
r_j=|x-\xi^{\J}| =\lambda_j \rho_j,
\quad \theta_j= \arctan\bigg(\frac{x_2-\xi_2^{\J}}{x_1-\xi_1^{\J} }\bigg),
\end{equation}
where we used the natural complex form $y^{\J}_1 + i y^{\J}_2$ for $y^{\J}$ and the similar form for others. Denote
\begin{equation}\label{def-Wjjj}
	W^{\J}:=W(y^{\J})=\begin{bmatrix}
		\cos\theta_{j} \sin w(\rho_{j})\\
		\sin\theta_{j} \sin w(\rho_{j})\\
		\cos w(\rho_{j})
	\end{bmatrix}:=\begin{bmatrix}
		e^{i\theta_{j}} \sin w(\rho_{j})\\
		\cos w(\rho_{j})
	\end{bmatrix}
\mbox{ \ with  \ }
w(\rho_j) :=\pi-2 \arctan (\rho_j),
\end{equation}
for $j=1,2,\dots,N$, and we have
\begin{equation}\label{nablaW}
	w_{\rho_j}= \frac{-2}{\rho_j^2+1},\quad\sin w(\rho_j)=-\rho_j w_{\rho_j}=\frac{2\rho_j}{\rho_j^2+1},\quad\cos w(\rho_j)=\frac{\rho_j^2-1}{\rho_j^2+1},
	\quad
	|\nabla_{y^{\J}} W(y^{\J})|^2 =2w_{\rho_j}^2
	=
	\frac{8}{(\rho_j^2+1)^2}.
\end{equation}

We denote the Frenet basis associated to $W^{\J}$ as
\begin{equation}\label{def-E1E2}
	E_{1}^{\J}=\begin{bmatrix}\cos\theta_j\cos w(\rho_j)\\\sin\theta_j \cos w(\rho_j) \\ -\sin w(\rho_j)
	\end{bmatrix}
	:=\begin{bmatrix}e^{i\theta_j}\cos w(\rho_j)\\ -\sin w(\rho_j)
	\end{bmatrix},
\quad
E_2^{\J}=\begin{bmatrix} -\sin\theta_j\\  \cos\theta_j\\ 0 \end{bmatrix}
	:=\begin{bmatrix} ie^{i\theta_j}\\ 0 \end{bmatrix}.
\end{equation}
So
\begin{equation}\label{wedge-formula}
	W^{\J}\wedge E_1^{\J}=E_2^{\J},\quad W^{\J}\wedge E_2^{\J}=-E_1^{\J},\quad E_1^{\J}\wedge E_2^{\J}=W^{\J}.
\end{equation}

It is direct to check that in the polar coordinates \eqref{polar-coor}
\begin{equation}\label{Frenet-deri}
	\begin{aligned}
		&\partial_{\rho_j} W^{\J}=w_{\rho_j} E_1^{\J}, \quad  \partial_{\rho_j \rho_j} W^{\J}=w_{\rho_j \rho_j } E_1^{\J} -w_{\rho_j}^2 W^{\J},
		\quad
		  \partial_{\theta_j} W^{\J}=\sin w(\rho_j) E_2^{\J},
		\\
		&
		\partial_{\theta_j\theta_j} W^{\J}=-\sin w(\rho_j)\big(\sin w(\rho_j) W^{\J}+\cos w(\rho_j) E_1^{\J} \big),
		\\
		&\partial_{\rho_j} E_1^{\J}=-w_{\rho_j} W^{\J},
		\quad
		  \partial_{\rho_j\rho_j} E_1^{\J}=-w_{\rho_j\rho_j} W^{\J} -w_{\rho_j}^2 E_1^{\J},
		  \quad
		    \partial_{\theta_j} E_1^{\J}=\cos w(\rho_j) E_2^{\J},
		\\
		&
		\partial_{\theta_j\theta_j} E_1^{\J} =-\cos w(\rho_j) \big(\sin w(\rho_j) W^{\J}+\cos w(\rho_j) E_1^{\J} \big),
		\\
		&\partial_{\rho_j} E_2^{\J}=\partial_{\rho_j\rho_j} E_2^{\J}=0,
		\quad
		  \partial_{\theta_j} E_2^{\J}=-\sin w(\rho_j) W^{\J}-\cos w(\rho_j) E_1^{\J},
		  \quad
		   \partial_{\theta_j\theta_j} E_2^{\J}=-E_2^{\J}.
	\end{aligned}
\end{equation}

The linearization of the harmonic map equation around $W^{\J}$ is the elliptic operator
\begin{equation*}%\label{def-linearization}
	L_{W^{\J}}[\phi]:=\Delta_{y^{\J}} \phi + |\nabla_{y^{\J}} W^{\J}|^2 \phi +2\big(\nabla_{y^{\J}} W^{\J}\cdot\nabla_{y^{\J}} \phi\big) W^{\J}.
\end{equation*}
Denote  the $s$-rotation matrices around $z$-axis, $x$-axis, $y$-axis respectively as
\begin{equation*}%\label{def-Qz}
	Q_{s} =
	\begin{bmatrix}
		\cos s & -\sin s & 0 \\
		\sin s & \cos s  & 0 \\
		0 & 0 & 1
	\end{bmatrix},
	\quad
	Q_{s}^{x} :=
	\begin{bmatrix}
		1  & 0 & 0\\
		0  & \cos s & -\sin s\\
		0 & \sin s & \cos s
	\end{bmatrix}
	,
	\quad
	Q_{s}^{y} :=
	\begin{bmatrix}
		\cos s & 0 & \sin s\\
		0  & 1 & 0\\
		- \sin s & 0 & \cos s
	\end{bmatrix}.
\end{equation*}
Due to the invariance of group action for the harmonic map equation, the corresponding kernels of $L_{W^{\J}}[\cdot]=0$ are given by
\begin{align}
%\begin{equation}
%	\begin{aligned}
		&Z_{0,1}^{\J}(y^{\J}) = -\pp_s \left( W(s^{-1} y^{\J})\right)\big|_{s=1} =\rho_j w_{\rho_j}(\rho_j) E_1^{\J}(y^{\J}),
		\nonumber
		\\
		&Z_{0,2}^{\J}(y^{\J}) = -\pp_{s} \left( Q_{s} W(y^{\J})\right)\big|_{s=0} =\rho_j w_{\rho_j}(\rho_j) E_2^{\J}(y^{\J}),
		\nonumber
		\\
		&Z_{1,1}^{\J}(y^{\J})= \pp_{y_1^{\J}} W(y^{\J})
		=w_{\rho_j}(\rho_j)[\cos \theta_j E_1^{\J}(y^{\J})+\sin\theta_j E_2^{\J}(y^{\J})],
		\nonumber
		\\
		&Z_{1,2}^{\J}(y^{\J})
		= \pp_{y_2^{\J}} W(y^{\J})
		=w_{\rho_j}(\rho_j)[\sin \theta_j E_1^{\J}(y^{\J})-\cos\theta_j E_2^{\J}(y^{\J})],
		\nonumber
		\\
		&Z_{-1,1}^{\J}(y^{\J})
		= - Z_{1,1}^{\J}(y^{\J}) - 2 \pp_{s} \left( Q_{s}^{y} W(y^{\J})\right)\big|_{s=0}
		=\rho_j^2 w_{\rho_j}(\rho_j)[\cos \theta_j E_1^{\J}(y^{\J})-\sin \theta_j E_2^{\J}(y^{\J})],
		\nonumber
		\\
		&Z_{-1,2}^{\J}(y^{\J})
		 = - Z_{1,2}^{\J}(y^{\J}) + 2 \pp_{s} \left( Q_{s}^{x} W(y^{\J})\right)\big|_{s=0}
		=\rho_j^2 w_{\rho_j}(\rho_j)[\sin \theta_j E_1^{\J}(y^{\J})+\cos \theta_j E_2^{\J}(y^{\J})].
		\label{def-kernels}
%	\end{aligned}
%\end{equation}
\end{align}
Set
\begin{equation}\label{def-U1U2}
	U^{\J}(x,t):=Q_{\gamma_j}W\bigg(\frac{x-\xi^{\J}}{\la_j}\bigg).
\end{equation}

For $\fbf=[f_1,f_2,f_3]^{\tr}\in \R^3$,
\begin{equation}\label{rota-vector}
	Q_{\gamma_j}  \fbf = \Big[ {\rm{Re}} \left( e^{i\gamma_j} (f_1+if_2)\right), ~{\rm{Im}} \left( e^{i\gamma_j} (f_1+if_2)\right),~ f_3\Big]^{\tr}
	:=
	\begin{bmatrix}
 e^{i\gamma_j} (f_1+if_2)
 \\
  f_3
	\end{bmatrix}.
\end{equation}

For $\fbf, \gbf\in \R^3$, we have  $(Q_{\gamma_j }\fbf)\wedge (Q_{\gamma_j } \gbf) = Q_{\gamma_j} (\fbf\wedge \gbf)$. Combining \eqref{wedge-formula}, we have
\begin{equation}\label{wedge+rota}
	U^{\J}\wedge (Q_{\gamma_j}E_1^{\J} )=Q_{\gamma_j}E_2^{\J},\quad U^{\J}\wedge ( Q_{\gamma_j} E_2^{\J})=-Q_{\gamma_j}E_1^{\J},\quad (Q_{\gamma_j} E_1^{\J})\wedge (Q_{\gamma_j} E_2^{\J})= U^{\J}.
\end{equation}

To deal with linearization near concentration zones, it will be convenient to use complex notations as all the analysis will be done on the associated tangent plane.

For any $\mathbf{f}\in \R^3$ satisfying $\mathbf{f} \cdot U^{\J}=0$, we define the equivalent complex form of $\mathbf{f}$ as
\begin{equation}\label{qd24May12-1}
\mathbf{f}_{\mathcal{C}_j }:=
\mathbf{f} \cdot (Q_{\gamma_{j}}E_1^{\J})  + i \mathbf{f} \cdot (Q_{\gamma_{j}}E_2^{\J}).
\end{equation}
For any complex-valued function $f$, we define
\begin{equation}\label{C-cal-inverse}
	f_{\mathcal{C}_j^{-1}} := \left({\rm{Re}}  f  \right)  Q_{\gamma_j} E_1^{\J}
	+
	\left({\rm{Im}}  f  \right)  Q_{\gamma_j} E_2^{\J} .
\end{equation}
By \eqref{wedge+rota},
\begin{equation}\label{U-wedge-C-cal}
	U^{\J}\wedge  f_{\mathcal{C}_j^{-1}}
	= \left({\rm{Re}}  f  \right)  Q_{\gamma_j} E_2^{\J}
	-
	\left({\rm{Im}}  f  \right)  Q_{\gamma_j} E_1^{\J}
	=
	(if)_{\mathcal{C}_j^{-1}} .
\end{equation}
Similarly,
for any $\mathbf{g}\in \R^3$ satisfying $\mathbf{g} \cdot W^{\J}=0$, the equivalent complex form of $\mathbf{g}$ is defined as
\begin{equation}\label{Cbb-def}
\mathbf{g}_{\mathbb{C}_j }:=\mathbf{g} \cdot E_1^{\J} + i \mathbf{g} \cdot E_2^{\J}.
\end{equation}
For any complex-valued function $g$, we define
\begin{equation}\label{Cbb-inverse}
	g_{\mathbb{C}_j^{-1}} := \left({\rm{Re}}  g  \right)   E_1^{\J}
	+
	\left({\rm{Im}}  g  \right)   E_2^{\J} .
\end{equation}
For any $\gbf(y^{\J},\tau)\in \R^3$ satisfying $\gbf \cdot W^{\J}=0$, we denote the mode $k$ component for $\gbf$ as
\begin{equation}\label{modek-def}
	\gbf_{\mathbb{C}_j, k} (\rho_j,\tau) :=
	(2\pi)^{-1}
	\int_{0}^{2\pi} \gbf_{\mathbb{C}_j}(\rho_j e^{is},\tau) e^{-iks} ds.
\end{equation}

 For any $\mathbf{f},\mathbf{g}\in\R^3$, we define
\begin{equation}\label{def-proj}
	\Pi_{\gbf^{\perp}} \fbf:=\fbf-(\fbf\cdot \gbf)\gbf.
\end{equation}
In particular, when $|\gbf|=1$, $\Pi_{\gbf^{\perp}}$ is the usual orthogonal projection on $\gbf^{\perp}$.

Notice for any $\mathbf{f}=\left[f_1,f_2,f_3\right]^{\tr}\in \R^3$, by \eqref{rota-vector}, we have
\begin{equation}\label{proj-cal-C}
	\begin{aligned}
		&
		\left(\Pi_{U^{\J \perp}} \mathbf{f} \right)_{\mathcal{C}_j }
		=
		\bigg(1
		-
		\frac{2}{\rho_j^2+1}
		{\rm{Re}}
		\bigg) \left[ \left(f_1+if_2\right)
		e^{-i\left(\theta_j+\gamma_j\right)}\right]
		-\frac{2\rho_j}{\rho_j^2+1} f_3 ,
		\\
		&
		\mathbf{f} \cdot U^{\J}
		=
		\frac{2\rho_j}{\rho_j^2+1}
		{\rm{Re}} \left[
		\left(f_1+if_2 \right)e^{-i\left(\theta_j+\gamma_j\right)}
		\right]
		+  \frac{\rho_j^2-1}{\rho_j^2+1} f_3.
	\end{aligned}
\end{equation}

The linearization of the harmonic map equation around $U^{\J}$ is given by
\begin{equation*}
	L_{U^{\J}}[\phi]:=\Delta_x \phi + |\nabla_x  U^{\J}|^2 \phi +2(\nabla_x  U^{\J}\cdot\nabla_x \phi) U^{\J}.
\end{equation*}
It is clear that
\begin{equation*}
	L_{U^{\J}}[ Q_{\gamma_j} \fbf(y^{\J}) ]=\la_j^{-2} Q_{\gamma_j} L_{W^{\J}}[\fbf(y^{\J})],
	\mbox{ \ where \ }  y^{\J}=\frac{x-\xi^{\J}}{\la_j}.
\end{equation*}

We now give several useful formulas with proofs similar to those of \cite[Section 3]{17HMF}. For any function $\mathbf{f}:\RR^2\rightarrow \RR^3$,
we set
\begin{equation*}%\label{def-tildeL}
	\tilde L_{U^{\J} }[\mathbf{f}]:=|\nabla_{x} U^{\J}|^2\Pi_{U^{\J\perp}}\mathbf{f}-2\nabla_{x}(\mathbf{f}\cdot U^{\J}) \cdot \nabla_{x} U^{\J},
	\mbox{ \ where \ }
	\nabla_x(\mathbf{f}\cdot U^{\J}) \cdot \nabla_x U^{\J}=\sum_{k=1}^2\partial_{x_k} (\mathbf{f}\cdot U^{\J}) \partial_{x_k} U^{\J}.
\end{equation*}
Obviously, $ U^{\J}\cdot\tilde{L}_{U^{\J}}[\fbf] =0$. Similarly, we set
\begin{equation*}
	\tilde{L}_{ W^{\J} } [\mathbf{f}] := |\nabla_{y^{\J}} W^{\J} |^2 \Pi_{ W^{\J \perp} } \mathbf{f} -2\nabla_{y^{\J}} (\mathbf{f}\cdot W^{\J}) \cdot \nabla_{y^{\J}}  W^{\J},
\end{equation*}
and then $W^{\J} \cdot \tilde{L}_{ W^{\J} } [\mathbf{f}] =0$. It is straightforward to get
\begin{equation*}
	\tilde{L}_{ U^{\J} } [ Q_{\gamma_j} \mathbf{f}(y^{\J}) ]
	=
	\lambda_j^{-2} Q_{\gamma_j}
	\tilde{L}_{ W^{\J} } [\mathbf{f}(y^{\J}) ],
	\quad
	L_{U^{\J}}[\Pi_{U^{\J \perp}}\mathbf{f}]=\Pi_{U^{\J \perp}}\Delta_x \mathbf{f} +\tilde L_{U^{\J}}[\mathbf{f}].
\end{equation*}

For $\mathbf{f} = [f_1,f_2,f_3]^{\tr}$, to analyze in different modes hereafter, we deduce that
\begin{align*}
%\begin{equation*}
%	\begin{aligned}
		&
		\tilde{L}_{ U^{\J} } [ Q_{\gamma_j} \fbf ]
				=
		\lambda_j^{-1}
		\Big\{
		\rho_j w_{\rho_j}^2(\rho_j)
		(
		\pp_{x_{1}}
		f_1 +  \pp_{x_{2}} f_2 )
		- e^{i\theta_j} w_{\rho_j} (\rho_j)
		\cos w(\rho_j) (  \pp_{x_{1}} f_3
		-i
		\pp_{x_{2}} f_3
		)
		\\
		&
		- e^{ - i\theta_j}  w_{\rho_j}(\rho_j)
		\cos w(\rho_j)  ( \pp_{x_{1}} f_3
		+
		i  \pp_{x_{2}} f_3
		)
		+  e^{2i\theta_j} \frac{ 1 }{2} \rho_j w_{\rho_j}^2 (\rho_j)
		\left[
		( \pp_{x_{1}} f_1
		- \pp_{x_{2}} f_2 )
		-i
		( \pp_{x_{2}} f_1  + \pp_{x_{1}} f_2 )
		\right]
		\\
		&
		+  e^{ - 2i\theta_j}  \frac{ 1 }{2} \rho_j w_{\rho_j}^2(\rho_j)
		\left[
		( \pp_{x_{1}} f_1
		- \pp_{x_{2}} f_2 )
		+
		i
		( \pp_{x_{2}} f_1  + \pp_{x_{1}} f_2 )
		\right]
		\Big\}   Q_{\gamma_j} E_{1}^{\J}
		\\
		&
		+    \lambda_j^{-1}
		\Big\{
		- \rho_j w_{\rho_j}^2(\rho_j)
		( \pp_{x_{2}} f_1 - \pp_{x_{1}} f_2  )
		+
		e^{i\theta_j}  w_{\rho_j}(\rho_j) \cos w(\rho_j) (
		\pp_{x_{2}} f_3 + i  \pp_{x_{1}} f_3
		)
		\\
		&
		+ e^{ - i\theta_j}   w_{\rho_j}(\rho_j) \cos w(\rho_j) (
		\pp_{x_{2}} f_3
		-
		i \pp_{x_{1}} f_3
		)
		- e^{2i\theta_j} \frac{1}{2} \rho_j w_{\rho_j}^2(\rho_j)
		\left[
		(
		\pp_{x_{2}} f_1
		+ \pp_{x_{1}} f_2
		)
		+
		i
		( \pp_{x_{1}} f_1
		-   \pp_{x_{2}} f_2
		)
		\right]
		\\
		&
		+ e^{ - 2i\theta_j}  \frac{1}{2} \rho_j w_{\rho_j}^2(\rho_j)
		\left[
		-
		(
		\pp_{x_{2}} f_1
		+ \pp_{x_{1}} f_2
		)
		+
		i
		( \pp_{x_{1}} f_1
		-   \pp_{x_{2}} f_2
		)
		\right]
		\Big\}  Q_{\gamma_j}  E_{2}^{\J}.
%	\end{aligned}
%\end{equation*}
\end{align*}
The corresponding complex form is given by
\begin{equation*}%\label{tilde-L-complex}
	\begin{aligned}
		&	(\tilde{L}_{ U^{\J} } [ Q_{\gamma_j} \fbf  ]  ) _{\mathcal{C}_j }
		=
		\lambda_j^{-1}
		\big\{
		\rho_j w_{\rho_j}^2(\rho_j)
		\left[
		(
		\pp_{x_{1}}
		f_1 +  \pp_{x_{2}} f_2 )
		-  i
		( \pp_{x_{2}} f_1 - \pp_{x_{1}} f_2  )
		\right]
		\\
		&
		+
		e^{i\theta_j}  2 w_{\rho_j}(\rho_j) \cos w(\rho_j)
		( -  \pp_{x_{1}} f_3  + i
		\pp_{x_{2}} f_3
		)
		+  e^{2i\theta_j}  \rho_j w_{\rho_j}^2(\rho_j)
		\left[
		( \pp_{x_{1}} f_1
		- \pp_{x_{2}} f_2 )
		-i
		( \pp_{x_{2}} f_1  + \pp_{x_{1}} f_2 )
		\right]
		\big\}.
	\end{aligned}
\end{equation*}
In particular,
\begin{equation}\label{tilde-L-complex}
	( \tilde{L}_{ U^{\J} } [ \fbf  ]  )_{\mathcal{C}_j}
	=
	( \tilde{L}_{ U^{\J} } [ \fbf  ] )_{\mathcal{C}_{j0} }
	+
	e^{i\theta_j}  ( \tilde{L}_{ U^{\J} } [ \fbf  ]  )_{\mathcal{C}_{j1} }
	+  e^{2i\theta_j} ( \tilde{L}_{ U^{\J} } [ \fbf  ]  )_{\mathcal{C}_{j2} },
\end{equation}
where we denote
\begin{align}
%\begin{equation}
%	\begin{aligned}
		( \tilde{L}_{ U^{\J} } [ \fbf  ] )_{\mathcal{C}_{j0} }
		:= \ &
		\lambda_j^{-1}
		\rho_j w_{\rho_j}^2(\rho_j)
		\left[
		\pp_{x_{1}}
		( Q_{-\gamma_j} \fbf)_1 +  \pp_{x_{2}} ( Q_{-\gamma_j}  \fbf)_2
		-  i
		\left( \pp_{x_{2}} ( Q_{-\gamma_j} \fbf)_1 - \pp_{x_{1}} ( Q_{-\gamma_j} \fbf)_2  \right)
		\right]
		\nonumber
		\\
		= \ &
		\lambda_j^{-1}
		\rho_j w_{\rho_j}^2(\rho_j)  e^{-i\gamma_j }
		\left[
		\pp_{x_{1}} f_1 +
		\pp_{x_{2}} f_2  + i \left(\pp_{x_{1}} f_2
		-
		\pp_{x_{2}}  f_1  \right)
		\right],
		\nonumber
		\\
		( \tilde{L}_{ U^{\J} } [ \fbf  ]  )_{\mathcal{C}_{j1} }
		:= \ &
		2 \lambda_j^{-1} w_{\rho_j}(\rho_j) \cos w(\rho_j)
		\left( -  \pp_{x_{1}} ( Q_{-\gamma_j}  \fbf)_3  + i
		\pp_{x_{2}} ( Q_{-\gamma_j}  \fbf)_3
		\right)
		\nonumber
		\\
		= \ &
		2 \lambda_j^{-1} w_{\rho_j}(\rho_j) \cos w(\rho_j)
		\left( -  \pp_{x_{1}} f_3  + i
		\pp_{x_{2}} f_3
		\right)
		,
		\nonumber
		\\
		( \tilde{L}_{ U^{\J} } [ \fbf  ]  )_{\mathcal{C}_{j2}}
		:= \ &
		\lambda_j^{-1} \rho_j w_{\rho_j}^2(\rho_j)
		\left[  \pp_{x_{1}} ( Q_{-\gamma_j} \fbf)_1
		- \pp_{x_{2}} ( Q_{-\gamma_j} \fbf)_2
		-i
		\left( \pp_{x_{2}} ( Q_{-\gamma_j} \fbf)_1  + \pp_{x_{1}} ( Q_{-\gamma_j}  \fbf)_2 \right)
		\right]
		\nonumber
		\\
		= \ &
		\lambda_j^{-1} \rho_j w_{\rho_j}^2(\rho_j)
		e^{i\gamma_j}\left[
		\pp_{x_1} f_1 - \pp_{x_2} f_2
		-i\left(\pp_{x_1} f_2 + \pp_{x_2} f_1  \right)
		\right],
		\label{tildeL-component}
%	\end{aligned}
%\end{equation}
\end{align}
since by \eqref{rota-vector}, we have
\begin{equation*}
	\begin{aligned}
		&
		\pp_{x_{1}}
		( Q_{-\gamma_j} \fbf )_1 +  \pp_{x_{2}} ( Q_{-\gamma_j}  \fbf )_2
		-  i
		\left( \pp_{x_{2}} ( Q_{-\gamma_j} \fbf )_1 - \pp_{x_{1}} ( Q_{-\gamma_j} \fbf )_2  \right)
		\\
		= \ &
		\pp_{x_{1}}
		{\rm{Re}} \left[
		e^{-i\gamma_j } \left( f_1 + if_2 \right)
		\right]
		+  \pp_{x_{2}}
		{\rm{Im}} \left[
		e^{-i\gamma_j } \left( f_1 + if_2 \right)
		\right]
		-  i \pp_{x_{2}} {\rm{Re}} \left[
		e^{-i\gamma_j } \left( f_1 + if_2 \right)
		\right]
		+
		i \pp_{x_{1}}
		{\rm{Im}} \left[
		e^{-i\gamma_j } \left( f_1 + if_2 \right)
		\right]
		\\
		= \ &
		e^{-i\gamma_j }
		\left[
		\pp_{x_{1}}
		\left( f_1 + if_2 \right)
		-  i \pp_{x_{2}}
		\left( f_1 + if_2 \right)
		\right]
		=
		e^{-i\gamma_j }
		\left[
		\pp_{x_{1}} f_1 +
		\pp_{x_{2}} f_2  + i \left(\pp_{x_{1}} f_2
		-
		\pp_{x_{2}}  f_1  \right)
		\right],
	\end{aligned}
\end{equation*}
\begin{equation*}
	\begin{aligned}
		&
		\pp_{x_{1}} ( Q_{-\gamma_j} \fbf)_1
		- \pp_{x_{2}} ( Q_{-\gamma_j} \fbf)_2
		-i
		\left( \pp_{x_{2}} ( Q_{-\gamma_j} \fbf)_1  + \pp_{x_{1}} ( Q_{-\gamma_j}  \fbf)_2 \right)
		\\
		= \ &
		\pp_{x_{1}}
		{\rm{Re}} \left[
		e^{-i\gamma_j } \left( f_1 + if_2 \right)
		\right]
		-  \pp_{x_{2}}
		{\rm{Im}} \left[
		e^{-i\gamma_j } \left( f_1 + if_2 \right)
		\right]
		-  i \pp_{x_{2}} {\rm{Re}} \left[
		e^{-i\gamma_j } \left( f_1 + if_2 \right)
		\right]
		-
		i \pp_{x_{1}}
		{\rm{Im}} \left[
		e^{-i\gamma_j } \left( f_1 + if_2 \right)
		\right]
		\\
		= \ &
		e^{i\gamma_j}\left[
		\pp_{x_1} f_1 - \pp_{x_2} f_2
		-i\left(\pp_{x_1} f_2 + \pp_{x_2} f_1  \right)
		\right].
	\end{aligned}
\end{equation*}

By \eqref{wedge+rota}, one has
\begin{align*}
%\begin{equation*}
%	\begin{aligned}
	 &
		Q_{-\gamma_j}
		\big[
		\big( a-bU^{\J}  \wedge \big)
		\tilde{L}_{U^{\J}}[\fbf]
		\big]
		=
		Q_{-\gamma_j}
		\Big[
		\big( a-bU^{\J}  \wedge \big)
		\Big\{
		{\rm{Re}}\big[ \big( \tilde{L}_{ U^{\J} } [\fbf]  \big)_{\mathcal{C}_j } \big]
		Q_{\gamma_j} E_1^{\J}
		+
		{\rm{Im}}\big[ \big( \tilde{L}_{ U^{\J} } [\fbf]  \big)_{\mathcal{C}_j } \big]
		Q_{\gamma_j} E_2^{\J}
		\Big\}
		\Big]
		\\
		= \ &
		{\rm{Re}}\big[ \big( \tilde{L}_{ U^{\J} } [\fbf]  \big)_{\mathcal{C}_j } \big]
		\big(
		a  E_1^{\J}
		-
		b  E_2^{\J}
		\big)
		+
		{\rm{Im}}\big[ \big( \tilde{L}_{ U^{\J} } [\fbf]  \big)_{\mathcal{C}_j } \big]
		\big(
		a  E_2^{\J}
		+
		b  E_1^{\J}
		\big),
%	\end{aligned}
%\end{equation*}
\end{align*}
and thus
\begin{equation}\label{outer into inner}
		\big\{ Q_{-\gamma_j}
		\big[
		( a-bU^{\J}  \wedge )
		\tilde L_{U^{\J} }[\fbf ]
		\big] \big\}_{\mathbb{C}_j}
		=
		(a-ib) \big( \tilde{L}_{ U^{\J} } [\fbf]  \big)_{\mathcal{C}_j }.
\end{equation}

\section{Approximation and improvement}

\subsection{First approximation}\label{sec-first-appr}

Given an integer $N\ge 1$ and arbitrary $N$ different points $q^{\J}\in \mathbb{R}^2$, $j=1,2,\dots,N$, denote
\begin{equation}\label{dq-pj-def}
	d_q := \min\limits_{k\ne m} |q^{\K} -q^{\M}|/9, \quad p_j(t) := \lambda_j(t) e^{i\gamma_j(t)}.
\end{equation}
Throughout this paper, we make the following ansatzes that for $j=1,2,\dots, N$,
\begin{equation}\label{lam-ansatz}
	\begin{aligned}
		&
		C_{\lambda}^{-1} \lambda_*(t) \le |p_j(t)| = \lambda_j(t) \le C_{\lambda} \lambda_*(t),
		\quad  \lambda_*(t) := \frac{|\ln T|(T-t)}{\ln^2(T-t)},
		\quad
		| \dot{\gamma}_j(t) |\le C_{\gamma} (T-t)^{-1},
		\\
		&
		C_{\lambda}^{-1} \frac{|\ln T| }{\ln^2(T-t)} \le
		|\dot{p}_{j}(t) | \le C_{\lambda} \frac{|\ln T| }{\ln^2(T-t)},
		\quad
		| \dot{\xi}^{\J}(t) | \le C_{\xi} \lambda_*^{\epsilon_{\xi}}(t),
		\quad \xi^{\J}(T) = q^{\J}
	\end{aligned}
\end{equation}
with some constants $C_{\lambda}\ge 1$, $C_{\xi}>0$, $C_{\gamma}>0$, and a small $\epsilon_{\xi}>0$ to be determined later.

We will construct blow-up solutions which blow up simultaneously at these prescribed points $q^{\J}$. We take the first approximation as
\begin{equation}\label{def-U*}
	U_*(x,t)
	:= -(N-1) U_{\infty} +  \sum\limits_{j=1}^{	N}U^{\J}(x,t),
\end{equation}
where $U^{\J}$ are given in \eqref{def-U1U2} and $U_{\infty}=[0,0,1]^{\tr}$.
For $t\in [0,T)$ with $T\ll 1$, we have
\begin{equation}\label{U*-norm}
\min
\limits_{k \ne m} |\xi^{\K}(t)-\xi^{\M}(t)| >8d_q>0;
\quad	|U_*|=1+O\Big(\sum\limits_{j=1}^{N}\lambda_j \Big);
\quad
|U_*-U^{\K}| \lesssim \sum_{j=1, j\ne k}^N \langle y^{\J} \rangle^{-1}.
\end{equation}

Given a function $\fbf=[f_1,f_2,f_3]^{\tr}\in \R^3$, denote the error function as
\begin{equation}\label{S-error fun}
S[\fbf] :=  - \partial_t \fbf + a(\Delta_x \fbf + |\nabla_x \fbf|^2 \fbf)-b \fbf\wedge\Delta_x \fbf.
\end{equation}
The error of the first approximate solution is
\begin{equation*}
S[ U_* ]=-\sum\limits_{j=1}^{N}\partial_t U^{\J}
	+ a(\Delta_x U_* + |\nabla_x U_*|^2 U_*)-bU_*\wedge\Delta_x U_*.
\end{equation*}
Notice
\begin{equation}\label{ppt-UJ}
	- \partial_t U^{\J} =  \mathcal E_0^{\J} + \mathcal E_1^{\J}, \quad
	\mathcal E_0^{\J}:=-\dot\la_j \partial_{\la_j} U^{\J}-\dot\gamma_j \partial_{\gamma_j}U^{\J},
	\quad
	\mathcal E_1^{\J}:=- \dot\xi^{\J}_1  \partial_{\xi^{\J}_1} U^{\J}-\dot\xi_2^{\J} \partial_{\xi^{\J}_2}U^{\J},
\end{equation}
where
\begin{equation*}
\left\{
\begin{aligned}
& \partial_{\la_j} U^{\J}(x)=-\la_j^{-1} Q_{\gamma_j} Z^{\J}_{0,1}(y^{\J}),
		\quad
		\partial_{\gamma_j} U^{\J}(x)= -Q_{\gamma_j} Z_{0,2}^{\J}(y^{\J}),
\\
& \partial_{\xi_1^{\J}} U^{\J}(x) = -\la_j^{-1}  Q_{\gamma_j} Z^{\J}_{1,1}(y^{\J}),
		\quad
		\partial_{\xi_2^{\J}} U^{\J}(x) = -\la_j^{-1}  Q_{\gamma_j} Z_{1,2}^{\J}(y^{\J})
\end{aligned}
\right.
\end{equation*}
with $Z_{m,n}^{\J}$ given in \eqref{def-kernels}. It is straightforward to compute
\begin{equation}\label{Ej-0}
	\begin{aligned}
		&
		\mathcal E_0^{\J}= Q_{\gamma_j}\left(\la_j^{-1}\dot\la_j Z_{0,1}^{\J}(y^{\J})+\dot\gamma_j Z_{0,2}^{\J}(y^{\J}) \right)
		=\rho_j w_{\rho_j} Q_{\gamma_j}  \left(\la_j^{-1}\dot\la_jE_1^{\J} +\dot\gamma_j E_2^{\J}  \right)
		\\
		= \ &
		\frac{-2\rho_j}{\rho_j^2+1}
		\begin{bmatrix}
			\left(\la_j^{-1}\dot\la_j  \cos w(\rho_j) + 	i \dot\gamma_j  \right)e^{i(\theta_j+\gamma_j)} \\ -\la_j^{-1}\dot\la_j  \sin w(\rho_j)
		\end{bmatrix} ,
	\end{aligned}
\end{equation}
\begin{equation*}%\label{Ej-cal-C}
	(\mathcal E_{0}^{\J} )_{\mathcal{C}_j}
	= -2\rho_j (\rho_j^2+1)^{-1} (\la_j^{-1}\dot\la_j + i\dot\gamma_j ),
\end{equation*}
\begin{equation}\label{Ej-1}
	\begin{aligned}
		\mathcal E_{1}^{\J}
		= \ &
		\frac{-2 \la_j^{-1} }{\rho_j^2+1} {\rm{Re}}\left[  \left(\dot \xi_1^{\J}  -i \dot\xi_2^{\J}  \right) e^{i\theta_j}  \right] Q_{\gamma_j} E_{1}^{\J}
		-
		\frac{2 \la_j^{-1} }{\rho_j^2+1}
		{\rm{Im}}\left[  \left(\dot \xi_1^{\J}  -i \dot\xi_2^{\J}  \right) e^{i\theta_j}  \right] Q_{\gamma_j} E_2^{\J}
		,
	\end{aligned}
\end{equation}
\begin{equation*}%\label{Ej-1-cal-C}
	( \mathcal E_{1}^{\J} )_{\mathcal{C}_j }
	=   -2\la_j^{-1} (\dot \xi_1^{\J}  -i \dot\xi_2^{\J}  ) (\rho_j^2+1)^{-1} e^{i\theta_j}.
\end{equation*}

Combining \eqref{Ej-0} and \eqref{Ej-1}, we have
\begin{equation}\label{ppt-U-est}
	|\pp_{t} U^{\J}|\lesssim
	(\la_j^{-1}|\dot\la_j| + |\dot\gamma_j | ) \langle \rho_j \rangle^{-1}
	+
	\la_j^{-1}|\dot{\xi}^{\J}| \langle \rho_j \rangle^{-2}.
\end{equation}

Notice that $S[U_*]$ contains errors $\mathcal E_0^{\J}$ with slow decay in space, which will break down the gluing process without improvement. We shall introduce global corrections to improve the spatial decay of the errors.

\subsection{Global corrections by parabolic systems}\label{global-cor-sec}

In this Section, we will transfer slow decay terms by parabolic systems.
Around each bubble, the slow decaying term in \eqref{Ej-0} is given by
\begin{equation*}
	\mathcal E_0^{\J}
	\approx
	-\frac{2 }{z_j}
	\begin{bmatrix}
		\dot p_j(t) e^{i\theta_j}
		\\
		0
	\end{bmatrix},
\end{equation*}
where
\begin{equation}\label{zj-def}
z_j :=(\la_j^2(t)+r_j^2)^{1/2} = (\lambda_j^2(t) + |x-\xi^{\J}(t)|^2)^{1/2},
\quad r_j=|x^{\J}|, \quad  x^{\J} :=x-\xi^{\J}(t).
\end{equation}

We aim to find global corrections $\Phi^{*\J}_0(r_j,t)$
to make
\begin{equation*}
	-\pp_t (\Phi_0^{*\J})  +
	\left(a  -b U_{\infty}\wedge \right) \Delta_x \Phi_0^{*\J}
	-\frac{2 }{z_j}
	\begin{bmatrix}
		\dot p_j(t) e^{i\theta_j}
		\\
		0
	\end{bmatrix}
	\approx 0
\end{equation*}
with the form
\begin{equation}\label{def-globalcorrection-J}
	\Phi^{*{\J}}_0(r_j,t) :=
	\frac{r_j^{\mu} }{r_j^{\mu}+ \la_j^{\mu} }
	\begin{bmatrix}
		\Phi^{\J}_0( \sqrt{r_j^2+\la_j^2} , t ) e^{i\theta_j}
		\\
		0
	\end{bmatrix}
	=
	\begin{bmatrix}
		\frac{\rho_j^{\mu}}{\rho_j^{\mu}+ 1 }	\Phi^{\J}_0(z_j , t ) e^{i\theta_j}
		\\
		0
	\end{bmatrix},
\end{equation}
where $\mu\ge 2$ is a constant to be determined later. The term $
\frac{r_j^\mu}{r_j^\mu+ \lambda_j^\mu }  $ is used to avoid terms that are singular at the origin when calculating new errors.

We now present calculations that improve approximately the slow decaying error.
\begin{equation*}
	\Delta_x 	
	\begin{bmatrix}
		\Phi^{\J}_0 e^{i\theta_j}
		\\
		0
	\end{bmatrix}
	\approx
	\begin{bmatrix}
		\left(\pp_{z_j z_j}\Phi^{\J}_0+ z_j^{-1} \pp_{z_j}\Phi^{\J}_0 - z_j^{-2} \Phi^{\J}_0 \right) e^{i\theta_j}
		\\
		0
	\end{bmatrix}
	.
\end{equation*}
Since for any $v_1,v_2\in\RR$,
\begin{equation}\label{wedge-basic}
	\left(a-bU_{\infty}\wedge\right)
	\begin{bmatrix}
		v_1 \\
		v_2 \\
		0
	\end{bmatrix}
	=
	\begin{bmatrix}
		{\rm{Re}} \left[ (a-ib)(v_1+iv_2)\right] \\
		{\rm{Im}} \left[ (a-ib)(v_1+iv_2)\right] \\
		0
	\end{bmatrix}
:=
\begin{bmatrix}
 (a-ib)(v_1+iv_2) \\
	0
\end{bmatrix}
,
\end{equation}
then
\begin{equation*}
	\begin{aligned}
		&
		-\pp_t (\Phi_0^{*\J})  +
		\left( a -b U_{\infty}\wedge \right) \Delta_x \Phi_0^{*\J}
		-\frac{2 }{z_j}
		\begin{bmatrix}
			\dot p_j(t) e^{i\theta_j}
			\\
			0
		\end{bmatrix}
		\\
		\approx \ &
		\begin{bmatrix}
			-\pp_t \Phi^{\J}_0 e^{i\theta_j} +	(a-ib) 	\left(\pp_{z_j z_j }\Phi^{\J}_0 + z_j^{-1} \pp_{z_j}\Phi^{\J}_0 - z_j^{-2} \Phi^{\J}_0\right) e^{i\theta_j}
			\\
			0
		\end{bmatrix}
		-\frac{2 }{z_j}
		\begin{bmatrix}
			\dot p_j(t) e^{i\theta_j}
			\\
			0
		\end{bmatrix} .
	\end{aligned}
\end{equation*}
For this reason, we choose $\Phi^{\J}_0(z_j,t)$ to solve
\begin{equation}\label{eqn-Phi0}
	(a+ib)\partial_t \Phi^{\J}_0 = \partial_{z_jz_j} \Phi^{\J}_0 +\frac1{z_j} \partial_{z_j} \Phi^{\J}_0
	-\frac1{z_j^2}\Phi^{\J}_0 -
	\frac{2(a+ib)\dot p_j(t) }{z_j}.
\end{equation}

The analysis of \eqref{eqn-Phi0} is the same as \cite[(4.7)]{17HMF}. We consider a more general equation
\begin{equation}\label{feqn}
	(a+ib)\partial_t f = \partial_{zz} f +\frac1{z} \partial_{z} f
	-\frac1{z^2}f +
	\frac{g(t)}{z}.
\end{equation}
First, we look for the self-similar profile to
\begin{equation*}
(a+ib)\partial_t f_1=
\partial_{zz} f_1 +\frac1{z} \partial_{z} f_1 -\frac1{z^2} f_1+\frac1{z}
\mbox{ \ with \ }
f_1(z,t)=t^{1/2}f_2(\frac{z}{t^{1/2}}).
\end{equation*}
Then $f_2$ satisfies
$$
f_2''(\xi)
+
\left(\frac{1}{\xi}+\frac{a+ib}{2}\xi\right)f_2'(\xi)
-
\left(\frac{1}{\xi^2}+\frac{a+ib}{2}\right)f_2(\xi)
+\frac1{\xi}=0,\quad \xi=\frac{z}{t^{1/2}}.
$$
Observing that $\xi$ is a homogeneous solution yields a solution
\begin{equation*}
	f_2(\xi)= \xi\int_{\xi}^{\infty}\frac{e^{-\frac{a+ib}{4}\eta^2}}{\eta^3} d\eta \int_0^{\eta} s e^{\frac{a+ib}{4}s^2} ds
	= \frac{2\xi}{a+ib} \int_{\xi}^{\infty}\frac{1-e^{-\frac{a+ib}{4}\eta^2}}{\eta^3} d\eta,
\end{equation*}
and $ |f_2(\xi)|\lesssim
		\xi \langle \ln \xi \rangle \1_{\{ 0\le \xi\le 1\}}
		+
		\xi^{-1} \1_{\{ \xi> 1\}}$.
It follows that $\lim_{t\downarrow 0} f_1(z,t) =0$ uniformly for all $z>0$. Then for $g(t)\in C^1([-T,T])$, by Duhamel's formula, one has a solution to \eqref{feqn} for $t\in (-T,T)$,
\begin{equation}\label{f-Duhamel}
		f(z,t)= \int_{-T}^t \dot g(s)f_1(z,t-s) ds
		+
		g(-T)f_1(z,t+T)
		= \int_{-T}^t g(s)\partial_t f_1(z,t-s) ds
		=
		\int_{-T}^t g(s) \frac{1-e^{-\frac{a+ib}{4}\frac{z^2}{t-s}}}{(a+ib) z} ds.
\end{equation}
If $g\in L^\infty([-T,T])$, then \eqref{f-Duhamel} solves \eqref{feqn} in weak sense.
Thus for \eqref{eqn-Phi0}, we have a solution
\begin{equation}\label{def-Phi0}
	\Phi^{\J}_0(z_j,t)=-z_j\int_{-T}^t \frac{\dot p_j(s) }{t-s}K_0(\frac{z_j^2}{t-s})ds,
\quad
	K_0(\zeta_j):=2\frac{1-e^{-\frac{a+ib}{4} \zeta_j } }{\zeta_j},
\end{equation}
where
\begin{equation}\label{zetaj}
	\zeta_j :=\frac{z_j^2}{t-s} =
	\iota_j (\rho_j^2+1),\quad \iota_j:= \frac{\lambda_j^2(t)}{t-s} .
\end{equation}
Since $a>0$, it is straightforward to verify
\begin{align}
%\begin{equation}
%	\begin{aligned}
	&	K_0(\zeta_j)
		=
		\Big(\frac{a+ib}{2}+ O\left(\zeta_j\right) \Big)\1_{\{ \zeta_j\le 1\}}
		+
		O(\zeta_j^{-1}) \1_{\{ \zeta_j > 1\}}
		,
		\quad
		\zeta_j K_{0\zeta_j}(\zeta_j)
		=
		O\left(\zeta_j\right) \1_{\{ \zeta_j\le 1\}}
		+
		O(\zeta_j^{-1} ) \1_{\{ \zeta_j > 1\}} ,
		\nonumber
		\\
	&	\zeta_j^2 K_{0\zeta_j \zeta_j}(\zeta_j)
		=
		O\left(\zeta_j^2\right) \1_{\{ \zeta_j\le 1\}}
		+
		O(\zeta_j^{-1} ) \1_{\{ \zeta_j > 1\}}.
		\label{K-est}
%	\end{aligned}
%\end{equation}
\end{align}
It is easy to get
\begin{equation}\label{Phi-J-deri}
\pp_{z_j} \Phi^{\J}_0
=
-\int_{-T}^t \frac{\dot p_j(s) }{t-s}
\big(
K_0(\zeta_j)
+ 2\zeta_j K_{0\zeta_j}(\zeta_j)
\big)
ds,
\quad
\pp_{z_jz_j} \Phi^{\J}_0
		=   -z_j^{-1} \int_{-T}^t \frac{\dot p_j(s) }{t-s}
		\big(
		6\zeta_j K_{0\zeta_j}( \zeta_j )
		+ 4\zeta_j^2 K_{0\zeta_j \zeta_j}( \zeta_j )
		\big)
		ds.
\end{equation}

\subsection{The upper bounds of global correction terms}\label{upp-global-cor-sec}

By \eqref{K-est} and  $|\dot{p}_j|\lesssim |\dot{\lambda}_{*}|$ in \eqref{lam-ansatz}, then
\begin{equation}\label{z2z2}
	|	\Phi^{\J}_0 |
	+
	z_j |\pp_{z_j} \Phi^{\J}_0 |
	+
	z_j^2 | \pp_{z_jz_j} \Phi^{\J}_0 |
		\lesssim z_j\int_{-T}^t \frac{ |\dot{\lambda}_{*}(s) | }{t-s}
		\big(  \1_{\{ \zeta_j\le 1\}}
		+
		\zeta_j^{-1}  \1_{\{ \zeta_j > 1\}} \big)\Big|_{\zeta_j = z_j^2 (t-s)^{-1}}
		ds.
\end{equation}

Claim: for $0\le t <T$,
\begin{align}
%\begin{equation}
%	\begin{aligned}
		&
		\int_{-T}^t \frac{ |\dot{\lambda}_{*}(s)| }{t-s}
		\big( \1_{\{ \zeta_j\le 1\}}
		+
		\zeta_j^{-1}  \1_{\{ \zeta_j > 1\}} \big)\Big|_{\zeta_j = z_j^2 (t-s)^{-1}}
		ds
        \nonumber
		\\
		\lesssim \ &
		\begin{cases}
			T |\ln T|^{-1} z_j^{-2} ,
			&
			z_j^2 \ge t+T \\
			\begin{cases}
				|\ln T|^{-1} \langle \ln(\frac{T}{z_j^2}) \rangle ,
				&
				t\le \frac{T}{2},
				\\
				\frac{|\ln T|}{|\ln (2T)|}
				-\frac{|\ln T|}{|\ln (2(T-t))|}
				+ |\dot{\lambda}_{*}(t) |
				\langle \ln(\frac{T-t}{z_j^2}) \rangle  ,
				&
				t > \frac{T}{2}, z_j^2 < T-t,
				\\
				\frac{|\ln T|}{|\ln(2T)|}-\frac{|\ln T|}{|\ln(T-t+z_j^2)|}+ |\ln T| (\ln z_j )^{-2},
				&
				t > \frac{T}{2}, z_j^2 \ge T-t,
			\end{cases}
			&
			z_j^2 < t +T
		\end{cases}
        \nonumber
		\\
		\lesssim \ &
		\1_{\{ z_j^2< t+T \}}  +  T |\ln T|^{-1} z_j^{-2}  \1_{\{ z_j^2\ge t+T \}}.
        \label{nonlocal-est}
%	\end{aligned}
%\end{equation}
\end{align}
By \eqref{z2z2} and \eqref{nonlocal-est}, we have
\begin{equation}\label{Phi-0-j-upp}
	|	\Phi^{\J}_0 |
	+
	z_j |\pp_{z_j} \Phi^{\J}_0 |
	+
	z_j^2 | \pp_{z_jz_j} \Phi^{\J}_0 |
	\lesssim
	z_j \1_{\{ z_j^2< t+T \}}  +  T |\ln T|^{-1} z_j^{-1}  \1_{\{ z_j^2\ge t+T \}}.
\end{equation}
Using \eqref{def-globalcorrection-J}, \eqref{Znabla Phi*-upp}, \eqref{ZDelta-Phi*}, and \eqref{Phi-0-j-upp}, we have
\begin{equation}\label{Phi*-0-j-upp}
	|\Phi^{*{\J}}_0  |
	+
	z_j |\nabla_x \Phi^{*{\J}}_0 |
	+
	z_j^2 | \Delta_{x} \Phi_0^{*\J} |
	\lesssim
	z_j \1_{\{ z_j^2< t+T \}}  +  T |\ln T|^{-1} z_j^{-1}  \1_{\{ z_j^2\ge t+T \}}.
\end{equation}

\begin{proof}[Proof of Claim \eqref{nonlocal-est}]

For $0\le t <T$, denote $
		g(z_j,t):= \int_{-T}^t \frac{ |\dot{\lambda}_{*}(s)| }{t-s}
		\big( \1_{\{ \zeta_j\le 1\}}
		+
		\zeta_j^{-1}  \1_{\{ \zeta_j > 1\}} \big)\big|_{\zeta_j = z_j^2 (t-s)^{-1}}
		ds $.
	
$\bullet$ For $z_j^2 \ge t+T$, $g(z_j,t)
		=
		z_j^{-2}
		\int_{-T}^t  |\dot{\lambda}_{*}(s)|
		ds
		\sim
		z_j^{-2}
		|\ln T | \int_{T-t}^{2T}  |\ln s_1|^{-2}
		ds_1
		\sim
		T |\ln T|^{-1}  z_j^{-2} $.
	
$\bullet$ 	For $z_j^2 < t+T $, $g(z_j,t)
		=
		\int_{-T}^{t- z_j^2 } \frac{ |\dot{\lambda}_{*}(s)| }{t-s}
		ds
		+
		z_j^{-2}
		\int_{t- z_j^2 }^t |\dot{\lambda}_{*}(s)|
		ds $.
	If $0\le t\le \frac{T}{2}$,
	\begin{equation*}
		g(z_j,t)
		\sim |\ln T|^{-1} \langle \ln(\frac{t+T}{z_j^2}) \rangle \sim |\ln T|^{-1} \langle \ln(\frac{T}{z_j^2}) \rangle.
	\end{equation*}

If $t> \frac{T}{2}$ and $z_j^2 < T-t$,
	\begin{equation*}
		g(z_j,t)
		\sim
		\frac{|\ln T|}{|\ln (2T)|}
		-\frac{|\ln T|}{|\ln (2(T-t))|}
		+ |\dot{\lambda}_{*}(t) |
		\langle \ln(\frac{T-t}{z_j^2}) \rangle
	\end{equation*}
	since
	\begin{align*}
%	\begin{equation*}
%		\begin{aligned}
		&	\int_{-T}^{t- z_j^2 } \frac{ |\dot{\lambda}_{*}(s)| }{t-s}
			ds
			=
			\bigg(
			\int_{-T}^{t- (T-t) }  + \int_{t- (T-t) } ^{t- z_j^2 }
			\bigg)
			\frac{ |\dot{\lambda}_{*}(s)| }{t-s}
			ds
			\sim
			\int_{-T}^{t- (T-t) }  \frac{ |\dot{\lambda}_{*}(s)| }{T-s}
			ds + |\dot{\lambda}_{*}(t) | \int_{t- (T-t) } ^{t- z_j^2 } \frac{ 1 }{t-s}
			ds
			\\
			\sim \ &
			\int_{-T}^{t- (T-t) }  \frac{ |\ln T| }{(T-s)|\ln (T-s)|^2}
			ds + |\dot{\lambda}_{*}(t) |
			\ln(\frac{T-t}{z_j^2})
			=
			\frac{|\ln T|}{|\ln (2T)|}
			-\frac{|\ln T|}{|\ln (2(T-t))|}
			+ |\dot{\lambda}_{*}(t) |
			\ln(\frac{T-t}{z_j^2}),
%		\end{aligned}
%	\end{equation*}
	\end{align*}
	\begin{equation*}
		z_j^{-2}
		\int_{t- z_j^2 }^t |\dot{\lambda}_{*}(s)|
		ds \sim |\dot{\lambda}_{*}(t)| .
	\end{equation*}
	
	If $t> \frac{T}{2}$ and $ T-t \le z_j^2 < t+T$, then we have
	\begin{equation*}
		g(z_j,t) \lesssim \frac{|\ln T|}{|\ln(2T)|}-\frac{|\ln T|}{|\ln(T-t+z_j^2)|} + |\ln T| (\ln z_j )^{-2}
	\end{equation*}
	since
	\begin{equation*}
		\int_{-T}^{t- z_j^2 } \frac{ |\dot{\lambda}_{*}(s)| }{t-s}
		ds
		\sim
		\int_{-T}^{t- z_j^2 } \frac{ |\dot{\lambda}_{*}(s)| }{T-s}
		ds
		\sim
		\int_{-T}^{t- z_j^2 } \frac{ |\ln T| }{(T-s)|\ln(T-s)|^2}
		ds
		=
		\frac{|\ln T|}{|\ln(2T)|}
		-\frac{|\ln T|}{|\ln(T-t+z_j^2)|},
	\end{equation*}
	\begin{equation*}
		\begin{aligned}
			&
			z_j^{-2}
			\int_{t- z_j^2 }^t |\dot{\lambda}_{*}(s)|
			ds
			\sim
			z_j^{-2} |\ln T|
			\int_{t- z_j^2 }^t (\ln(T-s))^{-2} ds
			=
			z_j^{-2} |\ln T|
			\int_{T-t }^{T-t+z_j^2} (\ln v )^{-2} dv
			\\
			\lesssim \ &
			z_j^{-2} |\ln T|
			(T-t+z_j^2) (\ln(T-t+z_j^2))^{-2}
			\sim
			|\ln T| (\ln z_j )^{-2}.
		\end{aligned}
	\end{equation*}
	In sum, we get the first part of \eqref{nonlocal-est}.
In particular, for $z_j^2 < t+T, t\le \frac{T}{2}$, we have $z_j^2 \gtrsim
		\left[ \frac{|\ln T|(T-t)}{|\ln(T-t)|^2} \right]^2
		\sim T^2 |\ln T|^{-2} $.
	Thus $ |\ln T|^{-1} \langle \ln(\frac{T}{z_j^2}) \rangle
		\lesssim 1 $.
	For $z_j^2< t+T, t > \frac{T}{2}, z_j^2 < T-t$,
we have	$|\dot{\lambda}_{*}(t) |
		\langle \ln(\frac{T-t}{z_j^2})\rangle
		\lesssim
		\frac{|\ln T|}{|\ln (T-t)|^2}
		\langle\ln(\frac{T-t}{ \lambda_*^2(t)}) \rangle
		\lesssim 1 $.
	Thus, we have the second part of \eqref{nonlocal-est}.
\end{proof}

\subsection{New errors produced by the global corrections}\label{nonlocal-err-set}

In this subsection, we will calculate the new errors produced by the introduction of $\Phi_0^{*\J}$ defined in \eqref{def-globalcorrection-J}, that is,
\begin{align}
%\begin{equation*}
%	\begin{aligned}
		\mathcal{S}^{\J}:= \ &
		-\pp_t (\Phi_0^{*\J} )
		+
		(a-bU^{\J}  \wedge)
		\big[
		\Delta_x \Phi_0^{*\J} +
		|\nabla_x U^{\J} |^2 \Phi^{*{\J}}_0
		- 2 \nabla_x \big(
		U^{\J} \cdot
		\Phi^{*{\J}}_0  \big) \cdot \nabla_x U^{\J}
		\big]
		- \partial_t U^{\J}
		\nonumber
		\\
		= \ &
		-\pp_t (\Phi_0^{*\J} )
		+
		\left( a-b U_{\infty}  \wedge \right)
		\Delta_x \Phi_0^{*\J}
		- \partial_t U^{\J}
		-b\left(U^{\J} -U_{\infty} \right)  \wedge
		\Delta_x \Phi_0^{*\J}
		\nonumber
		\\
		&
		+
		a |\nabla_x U^{\J} |^2 \Phi^{*{\J}}_0 +
		b|\nabla_x U^{\J} |^2
		\Phi^{*{\J}}_0 \wedge U^{\J}
		+
		(a-bU^{\J}  \wedge)
		\big[
		- 2 \nabla_x \big(
		U^{\J} \cdot
		\Phi^{*{\J}}_0  \big) \cdot  \nabla_x U^{\J}
		\big].
		\label{Sj-def}
%	\end{aligned}
%\end{equation*}
\end{align}
Both precise versions in different modes and rough upper bounds will be deduced, which will be used for solving reduced equations and the estimates for the forthcoming gluing system.
By \eqref{zj-def},
\begin{equation*}
		 \pp_{t}  \Big( \frac{r_j^\mu \Phi^{\J}_0(z_j , t )}{r_j^\mu + \lambda_j^\mu }   \Big)
		=
		\frac{r_j^\mu}{r_j^\mu + \lambda_j^\mu } \Big[	\pp_{t} \Phi^{\J}_0 +
		\frac{  \dot{\lambda}_j \lambda_j - \dot{\xi}^{\J}\cdot (x-\xi^{\J}) }{ \sqrt{r_j^2 +\lambda_j^2 } } \pp_{z_j} \Phi^{\J}_0 \Big] -
		\frac{\mu\lambda_j^\mu r_j^{\mu-2}\dot{\xi}^{\J}\cdot(x-\xi^{\J})+\mu\dot{\lambda}_j \lambda_j^{\mu-1} r_j^\mu}{ (r_j^\mu+ \lambda_j^\mu)^2 }
		\Phi^{\J}_0,
\end{equation*}
\begin{equation*}
		\pp_{r_j} \Big( \frac{r_j^\mu \Phi^{\J}_0(z_j, t )}{r_j^\mu+ \lambda_j^\mu }  \Big)
		=
		\frac{r_j^{\mu+1}}{(r_j^\mu+ \lambda_j^\mu) (r_j^2+\lambda_j^2)^{\frac{1}{2}} } \pp_{z_j} \Phi^{\J}_0
		+
		\frac{\mu\lambda_j^\mu r_j^{\mu-1}}{(r_j^\mu+\lambda_j^\mu)^2}
		\Phi^{\J}_0,
\end{equation*}
\begin{equation*}
	\begin{aligned}
			\pp_{r_jr_j} \Big( \frac{r_j^\mu \Phi^{\J}_0(z_j , t ) }{r_j^\mu+ \lambda_j^\mu }  \Big)
		= \ &
		\frac{r_j^{\mu+2} }{ (r_j^\mu+ \lambda_j^\mu)( r_j^2+\lambda_j^2  ) }
		\pp_{z_j z_j} \Phi^{\J}_0
		+
		\big[
		\frac{ 2\mu\lambda_j^\mu r_j^{\mu}}{(r_j^\mu+\lambda_j^\mu)^2 (r_j^2+\lambda_j^2)^{\frac{1}{2}}}
		+
		\frac{ \lambda_j^2 r_j^\mu}{ (r_j^\mu+ \lambda_j^\mu) (r_j^2+\lambda_j^2)^{\frac{3}{2}}  }
		\big]
		\pp_{z_j} \Phi^{\J}_0
		\\
		&
		+
		\frac{  (\mu-1)\mu \lambda_j^{2\mu} r_j^{\mu-2}
			-(\mu+1)\mu\lambda_j^\mu r_j^{2\mu-2}  }{(r_j^\mu+\lambda_j^\mu)^3 }
		\Phi^{\J}_0.
	\end{aligned}
\end{equation*}
Then by \eqref{def-globalcorrection-J}, we have
\begin{align}
%\begin{equation}
%	\begin{aligned}
		\pp_{r_j}\Phi^{*{\J}}_0
		= \ &
		\Big[
		\big[
		\frac{\rho_j^{\mu+1}}{(\rho_j^\mu+ 1) (\rho_j^2+1)^{\frac{1}{2}} } \pp_{z_j} \Phi^{\J}_0
		+
		\frac{\mu\lambda_j^{-1} \rho_j^{\mu-1}}{(\rho_j^\mu+1)^2}
		\Phi^{\J}_0
		\big]
		e^{i\theta_j} ,
		0 \Big]^{\tr},
		\nonumber
		\\
		\pp_{\theta_j}\Phi^{*{\J}}_0 = \ &
		\Big[
		\frac{r_j^\mu}{r_j^\mu+ \la_j^\mu }	\Phi^{\J}_0  ie^{i\theta_j},
		0 \Big]^{\tr}
		=
		\Big[
		\frac{\rho_j^\mu}{\rho_j^\mu+ 1 }	\Phi^{\J}_0  ie^{i\theta_j} ,
		0 \Big]^{\tr}.
		\label{Phi*-deri}
%	\end{aligned}
%\end{equation}
\end{align}
It follows that
\begin{equation*}%\label{nabla Phi*-upp}
	\begin{aligned}
		&
		|\nabla_x \Phi^{*{\J}}_0 |^2
		=
		|\nabla_{x^{\J}} \Phi^{*{\J}}_0 |^2
		=
		|\pp_{r_j} \Phi^{*{\J}}_0 |^2 + r_j^{-2} |\pp_{\theta_j} \Phi^{*{\J}}_0 |^2
		\\
		= \ &
		\bigg|
		\frac{\rho_j^{\mu+1}}{(\rho_j^\mu+ 1) (\rho_j^2+1)^{\frac{1}{2}} } \pp_{z_j} \Phi^{\J}_0
		+
		\frac{\mu\lambda_j^{-1} \rho_j^{\mu-1}}{(\rho_j^\mu+1)^2}
		\Phi^{\J}_0
		\bigg|^2
		+
		\lambda_j^{-2} \rho_j^{-2}
		\bigg|
		\frac{\rho_j^\mu}{\rho_j^\mu + 1 }	\Phi^{\J}_0
		\bigg|^2.
	\end{aligned}
\end{equation*}
Since $\mu\ge 2$, it follows that
\begin{equation}\label{Znabla Phi*-upp}
	|\nabla_x \Phi^{*{\J}}_0 |
	\lesssim |\pp_{z_j} \Phi^{\J}_0 |
	+ z_{j}^{-1} |\Phi^{\J}_0|.
\end{equation}
$\bullet$ By \eqref{polar-coor},
\begin{align}
%\begin{equation}
%	\begin{aligned}
		-\pp_t (\Phi_0^{*\J} )
				= \ &
		\Big[
		\Big\{
		\frac{ - \rho_j^\mu}{\rho_j^\mu + 1 } \pp_{t} \Phi^{\J}_0
		+
		\frac{   \rho_j^\mu  ( \dot{\xi}^{\J}\cdot y^{\J} - \dot{\lambda}_j ) }{ ( \rho_j^\mu + 1 ) ( \rho_j^2 + 1 )^{\frac{1}{2}} } \pp_{z_j} \Phi^{\J}_0
		+
		\Big[
		\frac{\mu\lambda_j^{-1} \rho_j^{\mu-2}\dot{\xi}^{\J}\cdot y^{\J} +\mu\dot{\lambda}_j \lambda_j^{-1} \rho_j^\mu}{ (\rho_j^\mu+ 1)^2 }
		\nonumber
		\\
		&
		+
		\frac{ i \lambda_j^{-1} \rho_j^{\mu-2} (\dot{\xi}_2^{\J} y_1^{\J} - \dot{\xi}_1^{\J} y_2^{\J} ) }{ \rho_j^\mu + 1 }
		\Big] \Phi^{\J}_0
		\Big\}
		e^{i\theta_j} , 0\Big]^{\tr}.
		\label{ppt-Phi*}
%	\end{aligned}
%\end{equation}
\end{align}

$\bullet$
\begin{small}
\begin{align}
%\begin{equation}
%	\begin{aligned}
		&
		\Delta_{x} \Phi_0^{*\J}
		=
		\Big[ \big(\pp_{r_j r_j} + \frac{1}{r_j} \pp_{r_j} + \frac{1}{r_j^2} \pp_{\theta_j\theta_j}  \big) \big( \frac{r_j^\mu}{r_j^\mu+ \la_j^\mu }	\Phi^{\J}_0  e^{i\theta_j} \big), 0 \Big]^{\tr}
		\nonumber
		\\
		= \ &
		\Big[  \Big\{
		\frac{ \rho_j^{\mu+2} }{ (\rho_j^\mu+ 1)( \rho_j^2+1 ) }
		\pp_{z_j z_j} \Phi^{\J}_0
		+
		\big[
		\frac{ 2\mu\lambda_j^{-1} \rho_j^{\mu}}{(\rho_j^\mu+1)^2 (\rho_j^2+1)^{\frac{1}{2}}}
		+
		\frac{ \lambda_j^{-1} \rho_j^\mu}{ (\rho_j^\mu+ 1) (\rho_j^2+1)^{\frac{3}{2}}  }
		+
		\frac{ \lambda_j^{-1} \rho_j^{\mu}}{(\rho_j^\mu+ 1) (\rho_j^2+1)^{\frac{1}{2}} }
		\big]
		\pp_{z_j} \Phi^{\J}_0
		\nonumber
		\\
		&
		+
		\big[
		\mu^2 \lambda_j^{-2}
		\frac{ \rho_j^{\mu-2} - \rho_j^{2\mu-2}  }{(\rho_j^\mu+1 )^3 }
		-
		\frac{ \lambda_j^{-2}\rho_j^{\mu-2} }{\rho_j^\mu+ 1 }
		\big]
		\Phi^{\J}_0
		\Big\} e^{i\theta_j}
		, 0 \Big]^{\tr}.
		\label{Delta-Phi*}
%	\end{aligned}
%\end{equation}
\end{align}
\end{small}
Since $\mu\ge 2$, we have
\begin{equation}\label{ZDelta-Phi*}
	| \Delta_{x} \Phi_0^{*\J} |
	\lesssim
	|\pp_{z_j z_j} \Phi^{\J}_0 |
	+ z_j^{-1} | \pp_{z_j } \Phi^{\J}_0 |
	+
	z_j^{-2} | \Phi^{\J}_0 |.
\end{equation}
$\bullet$
By \eqref{ppt-UJ}, \eqref{ppt-Phi*}, \eqref{Delta-Phi*}, and \eqref{wedge-basic}, we have
\begin{small}
\begin{align}
%\begin{equation}
%	\begin{aligned}
		&
		-\pp_t (\Phi_0^{*\J} )
		+
		\left( a-b U_{\infty}  \wedge \right)
		\Delta_x \Phi_0^{*\J}
		- \partial_t U^{\J}
		\nonumber
		\\
		= \ &
		\Big[
		\Big\{
		\frac{  \dot{\xi}^{\J}\cdot y^{\J} \rho_j^\mu  }{ ( \rho_j^\mu + 1 ) ( \rho_j^2 + 1 )^{\frac{1}{2}} } \pp_{z_j} \Phi^{\J}_0
		+
		\big[
		\frac{\mu\lambda_j^{-1} \dot{\xi}^{\J}\cdot y^{\J} \rho_j^{\mu-2} }{ (\rho_j^\mu+ 1)^2 }
		+
		\frac{ i \lambda_j^{-1}  (\dot{\xi}_2^{\J} y_1^{\J} - \dot{\xi}_1^{\J} y_2^{\J} ) \rho_j^{\mu-2} }{ \rho_j^\mu + 1 }
		\big] \Phi^{\J}_0
		\Big\}
		e^{i\theta_j} , 0\Big]^{\tr}
		\nonumber
		\\
		& +
		\Big[
		\big[
		\frac{ - \dot{\lambda}_j \rho_j^\mu }{ ( \rho_j^\mu + 1 ) ( \rho_j^2 + 1 )^{\frac{1}{2}} } \pp_{z_j} \Phi^{\J}_0
		+
		\frac{ \mu \lambda_j^{-1} \dot{\lambda}_j \rho_j^\mu}{ (\rho_j^\mu+ 1)^2 }   \Phi^{\J}_0
		\big]
		e^{i\theta_j} , 0\Big]^{\tr}
		\nonumber
		\\
		& +
		\Big[  (a-ib)\Big\{
		\frac{ - \rho_j^{\mu} }{ (\rho_j^\mu+ 1)( \rho_j^2+1 ) }
		\pp_{z_j z_j} \Phi^{\J}_0
		+
		\big[
		\frac{ 2\mu\lambda_j^{-1} \rho_j^{\mu}}{(\rho_j^\mu+1)^2 (\rho_j^2+1)^{\frac{1}{2}}}
		+
		\frac{ \lambda_j^{-1} \rho_j^\mu}{ (\rho_j^\mu+ 1) (\rho_j^2+1)^{\frac{3}{2}}  }
		\big]
		\pp_{z_j} \Phi^{\J}_0
		\nonumber
		\\
		&
		+
		\lambda_j^{-2}\rho_j^{\mu-2}
		\big[
		\mu^2
		\frac{ 1 - \rho_j^{\mu}  }{(\rho_j^\mu+1 )^3 }
		-
		\frac{ 1}{ (\rho_j^\mu+ 1)(\rho_j^2+1) }
		\big]
		\Phi^{\J}_0
		\Big\} e^{i\theta_j}
		, 0 \Big]^{\tr}
		+
		\mathcal{E}_0^{\J} +
		\Big[ \frac{2\lambda_j^{-1} \dot{p}_j \rho_j^{\mu}}{ (\rho_j^\mu+1) (\rho_j^2+1)^{\frac{1}{2}}} e^{i\theta_j} , 0 \Big]^{\tr}
		+ \mathcal{E}_1^{\J},
		\label{noloc-1}
%	\end{aligned}
%\end{equation}
\end{align}
\end{small}
where we have used \eqref{eqn-Phi0}, and
\begin{equation}\label{xi-dot-y}
	\begin{aligned}
		&
		\dot{\xi}^{\J}\cdot y^{\J}
		=
		2^{-1}
		\rho_j
		\big[
		(\dot{\xi}_1^{\J} - i\dot{\xi}_2^{\J} )e^{i\theta_j}
		+
		(\dot{\xi}_1^{\J} + i\dot{\xi}_2^{\J} )e^{-i\theta_j}
		\big],
		\\
		&
		\dot{\xi}_2^{\J} y_1^{\J} -
		\dot{\xi}_1^{\J} y_2^{\J}
		=
		2^{-1} \rho_j
		\big[
		(\dot{\xi}_2^{\J} + i\dot{\xi}_1^{\J} )e^{i\theta_j}
		+
		(\dot{\xi}_2^{\J} - i\dot{\xi}_1^{\J} )e^{-i\theta_j}
		\big].
	\end{aligned}
\end{equation}
By \eqref{Ej-0},
\begin{small}
\begin{align*}
%\begin{equation*}
%	\begin{aligned}
		&
		\mathcal{E}_0^{\J} +
		\Big[ \frac{2\lambda_j^{-1} \dot{p}_j \rho_j^{\mu}}{ (\rho_j^\mu+1) (\rho_j^2+1)^{\frac{1}{2}}} e^{i\theta_j} , 0 \Big]^{\tr}
		\\
		= \ &
		\frac{-2\rho_j}{\rho_j^2+1}
		\Big[
		\big[\la_j^{-1}\dot\la_j \big(1- \frac{2 }{\rho_j^2+1} \big) + 	i \dot\gamma_j  \big] e^{i(\theta_j+\gamma_j)}
		,
		-\la_j^{-1}\dot\la_j  \frac{2\rho_j}{\rho_j^2+1}
		\Big]^{\tr}
		+
		\Big[ \frac{2 (\lambda_j^{-1} \dot{\lambda}_j +i\dot{\gamma}_j)  \rho_j^{\mu}}{ (\rho_j^\mu+1) (\rho_j^2+1)^{\frac{1}{2}}} e^{i (\theta_j + \gamma_j) } , 0 \Big]^{\tr}
		\\
		= \ &
		\Big[
		\Big\{
		(\la_j^{-1}\dot\la_j + i \dot\gamma_j )
		\frac{2\rho_j[ \rho_j^{\mu-1} - \rho_j  - (\rho_j^2+1)^{\frac{1}{2}} ] }{ [ \rho_j + (\rho_j^2+1)^{\frac{1}{2}} ] (\rho_j^\mu + 1)(\rho_j^2+1)}
		+
		\frac{4\la_j^{-1}\dot\la_j  \rho_j }{(\rho_j^2+1)^2}
		\Big\}
		e^{i(\theta_j+\gamma_j)}
		,
		\frac{ 4\la_j^{-1}\dot\la_j \rho_j^2}{ (\rho_j^2+1)^2 }
		\Big]^{\tr}.
%	\end{aligned}
%\end{equation*}
\end{align*}
\end{small}
Then by \eqref{proj-cal-C},
\begin{align}
%\begin{equation}
%	\begin{aligned}
		&
		\bigg(\Pi_{U^{\J \perp}} \bigg( \mathcal{E}_0^{\J} +
		\Big[ \frac{2\lambda_j^{-1} \dot{p}_j \rho_j^{\mu}}{ (\rho_j^\mu+1) (\rho_j^2+1)^{\frac{1}{2}}} e^{i\theta_j} , 0 \Big]^{\tr} \bigg) \bigg)_{\mathcal{C}_j }
		\nonumber
		\\
		= \ &
		\bigg(1
		-
		\frac{2}{\rho_j^2+1}
		{\rm{Re}}
		\bigg) \bigg[
		(\la_j^{-1}\dot\la_j + i \dot\gamma_j )
		\frac{2\rho_j[ \rho_j^{\mu-1} - \rho_j  - (\rho_j^2+1)^{\frac{1}{2}} ] }{ [ \rho_j + (\rho_j^2+1)^{\frac{1}{2}} ] (\rho_j^\mu + 1)(\rho_j^2+1)}
		+
		\frac{4\la_j^{-1}\dot\la_j  \rho_j }{(\rho_j^2+1)^2}
		\bigg]
		-
		\frac{ 8\la_j^{-1}\dot\la_j \rho_j^3}{ (\rho_j^2+1)^3 }
		\nonumber
		\\
		= \ &
		(\la_j^{-1}\dot\la_j + i \dot\gamma_j )
		\frac{2\rho_j[ \rho_j^{\mu-1} - \rho_j  - (\rho_j^2+1)^{\frac{1}{2}} ] }{ [ \rho_j + (\rho_j^2+1)^{\frac{1}{2}} ] (\rho_j^\mu + 1)(\rho_j^2+1)}
		-
		\la_j^{-1}\dot\la_j
		\frac{ 4\rho_j^{\mu} [  \rho_j^2  + \rho_j(\rho_j^2+1)^{\frac{1}{2}} + 1 ]  }{ [ \rho_j + (\rho_j^2+1)^{\frac{1}{2}} ] (\rho_j^\mu + 1)(\rho_j^2+1)^2 },
		\label{nonloc-2-E}
%	\end{aligned}
%\end{equation}
\end{align}
\begin{equation}\label{nonloc-2-U}
	\bigg(  \mathcal{E}_0^{\J} +
	\Big[ \frac{2\lambda_j^{-1} \dot{p}_j \rho_j^{\mu}}{ (\rho_j^\mu+1) (\rho_j^2+1)^{\frac{1}{2}}} e^{i\theta_j} , 0 \Big]^{\tr}   \bigg) \cdot U^{\J}
	=
	\frac{4\lambda_j^{-1} \dot{\lambda}_j \rho_j^{\mu+1} [\rho_j^2 + \rho_j(\rho_j^2+1)^{\frac{1}{2}} +1 ]}{[\rho_j+ (\rho_j^2+1)^{\frac{1}{2}}] (\rho_j^\mu + 1)(\rho_j^2+1)^2}.
\end{equation}
$\bullet$
\begin{equation*}%\label{U*-U-wdege}
	\begin{aligned}
		&
		-b (U^{\J}-U_{\infty} )\wedge \Delta_x \Phi_0^{*\J}
		=
		\frac{-2b }{|y^{\J}|^2+1}
		\Big[ y^{\J}_1 \cos \gamma_j  - y_2^{\J} \sin\gamma_j , y_1^{\J} \sin \gamma_j  + y_2^{\J} \cos\gamma_j  , -1 \Big]^{\tr} \wedge \Delta_x \Phi_0^{*\J}
		\\
		= \ &
		\frac{-2b}{\rho_j^2+1}
		\Big[ (\Delta_x \Phi_0^{*\J} )_{2} , -(\Delta_x \Phi_0^{*\J} )_1  , (y^{\J}_1 \cos \gamma_j  - y_2^{\J} \sin\gamma_j ) (\Delta_x \Phi_0^{*\J} )_{2} - (y_1^{\J} \sin \gamma_j  + y_2^{\J} \cos\gamma_j)(\Delta_x \Phi_0^{*\J} )_1 \Big]^{\tr}
		\\
		= \ &
		\frac{-2b}{\rho_j^2+1}
		\Big[ (\Delta_x \Phi_0^{*\J} )_{2} , -(\Delta_x \Phi_0^{*\J} )_1  ,
		\rho_j
		{\rm{Re}} \left[\left(  (\Delta_x \Phi_0^{*\J} )_2 - i(\Delta_x \Phi_0^{*\J} )_1\right) e^{-i(\theta_j+\gamma_j)} \right]  \Big]^{\tr},
	\end{aligned}
\end{equation*}
where for the last equality, we use the following formula. For any $a_1, a_2\in \R$,
\begin{align*}
%\begin{equation}
%	\begin{aligned}
		&
		\big( y^{\J}_1 \cos \gamma_j  - y_2^{\J} \sin\gamma_j \big) a_1 - \big(y_1^{\J} \sin \gamma_j  + y_2^{\J} \cos\gamma_j \big) a_2
		=
		\rho_j \left(a_1\cos(\theta_j+\gamma_j) - a_2 \sin (\theta_j+\gamma_j) \right)
		\nonumber
		\\
		= \ &
		\rho_j
		{\rm{Re}} \big[(a_1-ia_2) e^{-i(\theta_j+\gamma_j)} \big]
		=
		\rho_j
		{\rm{Im}} \big[(a_2+ia_1) e^{-i(\theta_j+\gamma_j)} \big].
		%\label{basic-wedge}
%	\end{aligned}
%\end{equation}
\end{align*}
Then by \eqref{proj-cal-C}, we have
\begin{small}
\begin{align}
%\begin{equation}
%	\begin{aligned}
		&
		\left(\Pi_{U^{\J \perp}}
		\left[-b (U^{\J}-U_{\infty} )\wedge \Delta_x \Phi_0^{*\J}  \right] \right)_{\mathcal{C}_j}
		=
		\frac{-2b}{\rho_j^2+1}
		\bigg\{
		\bigg(1
		-
		\frac{2}{\rho_j^2+1}
		{\rm{Re}}
		\bigg) \left[ \left( (\Delta_x \Phi_0^{*\J} )_{2}  -i (\Delta_x \Phi_0^{*\J} )_1  \right)e^{-i(\theta_j+\gamma_j) } \right]
	\nonumber	
        \\
		&
		-\frac{2\rho_j}{\rho_j^2+1} \rho_j
		{\rm{Re}} \left[\left( (\Delta_x \Phi_0^{*\J} )_2 - i (\Delta_x \Phi_0^{*\J} )_1 \right) e^{-i(\theta_j+\gamma_j)} \right] \bigg\}
		=
		\frac{2i b  }{\rho_j^2+1}
		\overline{
			\big( (\Delta_x \Phi_0^{*\J} )_1
			+ i (\Delta_x \Phi_0^{*\J} )_{2} \big)e^{-i(\theta_j+\gamma_j) } }
	\nonumber	
        \\
		= \ &
		\bigg\{
		\frac{ 2i b \rho_j^{\mu+2} }{ (\rho_j^\mu+ 1)( \rho_j^2+1 )^2 }
		\pp_{z_j z_j} \overline{ \Phi^{\J}_0 }
		+
		\frac{2i b \lambda_j^{-1} \rho_j^{\mu} }{\rho_j^2+1}
		\Big[
		\frac{ 2\mu }{(\rho_j^\mu+1)^2 (\rho_j^2+1)^{\frac{1}{2}}}
		+
		\frac{ \rho_j^2 +2 }{ (\rho_j^\mu+ 1) (\rho_j^2+1)^{\frac{3}{2}}  }
		\Big]
		\pp_{z_j} \overline{\Phi^{\J}_0 }
	\nonumber	
        \\
		&
		-
		\frac{2i b \lambda_j^{-2}\rho_j^{\mu-2}  }{\rho_j^2+1}
		\bigg[
		\mu^2
		\frac{\rho_j^{\mu} -1 }{(\rho_j^\mu+1 )^3 }
		+
		\frac{1}{\rho_j^\mu+ 1 }
		\bigg]
		\overline{\Phi^{\J}_0 }
		\bigg\} e^{i\gamma_j},
        \label{nonlocal-3-E}
%	\end{aligned}
%\end{equation}
\end{align}
\end{small}
where we used \eqref{Delta-Phi*} for the last equality.
\begin{small}
\begin{equation}\label{nonloc-3-U}
	\begin{aligned}
		&
		\left[-b (U^{\J}-U_{\infty} )\wedge \Delta_x \Phi_0^{*\J} \right] \cdot U^{\J}
		=
		\frac{-2b}{\rho_j^2+1}
		\bigg\{
		\frac{2\rho_j}{\rho_j^2+1}
		{\rm{Re}} \left[
		\left(
		(\Delta_x \Phi_0^{*\J} )_{2}
		-
		i (\Delta_x \Phi_0^{*\J} )_1
		\right)
		e^{-i(\theta_j+\gamma_j)}
		\right]
		\\
		&
		+  \frac{\rho_j^2-1}{\rho_j^2+1}
		\rho_j
		{\rm{Re}} \left[\left(  (\Delta_x \Phi_0^{*\J} )_2 - i (\Delta_x \Phi_0^{*\J} )_1 \right) e^{-i(\theta_j+\gamma_j)} \right]
		\bigg\}
		=
		\frac{-2b \rho_j}{\rho_j^2+1}
		{\rm{Im}} \left[
		\left(
		(\Delta_x \Phi_0^{*\J} )_1
		+
		i
		(\Delta_x \Phi_0^{*\J} )_{2}
		\right)
		e^{-i(\theta_j+\gamma_j)}
		\right].
	\end{aligned}
\end{equation}
\end{small}

$\bullet$ By \eqref{nablaW} and  \eqref{def-globalcorrection-J},
\begin{equation*}
	a |\nabla_x  U^{\J}|^2 \Phi_0^{*\J}
	=
	a \lambda_j^{-2} |\nabla_{y^{\J}}  U^{\J}|^2 \Phi_0^{*\J}
	=
	\Big[
	\frac{8a \lambda_j^{-2} \rho_j^\mu}{(\rho_j^2+ 1)^2(\rho_j^\mu+1) }	\Phi^{\J}_0 e^{i\theta_j},0
	\Big]^{\tr}
	.
\end{equation*}
Then by \eqref{proj-cal-C},
\begin{equation}\label{nonloc-4}
	\begin{aligned}
		&
		\left(\Pi_{U^{\J \perp}} \left(	a |\nabla_x  U^{\J}|^2 \Phi_0^{*\J}\right) \right)_{\mathcal{C}_j }
		=
		\frac{8a \lambda_j^{-2} \rho_j^\mu}{(\rho_j^2+ 1)^2(\rho_j^\mu+1) }
		\bigg(1
		-
		\frac{2}{\rho_j^2+1}
		{\rm{Re}}
		\bigg) \left( 	\Phi^{\J}_0  e^{-i \gamma_j } \right)
		,
		\\
		&
		\left(	a |\nabla_x  U^{\J}|^2 \Phi_0^{*\J}\right) \cdot U^{\J}
		=
		\frac{16 a \lambda_j^{-2} \rho_j^{1+\mu}}{(\rho_j^2+ 1)^3(\rho_j^\mu+1) }	
		{\rm{Re}} \left( \Phi^{\J}_0  e^{-i \gamma_j } \right).
	\end{aligned}
\end{equation}
$\bullet$
Notice
\begin{equation}\label{nonloc-5-U}
	\big(b |\nabla_x U^{\J} |^2 \Phi_0^{*\J} \wedge  U^{\J} \big) \cdot U^{\J} =0.
\end{equation}
By \eqref{U-wedge-C-cal} and \eqref{nonloc-4}, we have
\begin{small}
\begin{equation}\label{nonloc-5}
		\left( 	b |\nabla_x U^{\J} |^2 \Phi_0^{*\J} \wedge  U^{\J} \right)_{\mathcal{C}_j }
=
b
\left( 	\left[\Pi_{U^{\J \perp}} \left(	|\nabla_x  U^{\J}|^2 \Phi_0^{*\J}\right) \right]  \wedge  U^{\J} \right)_{\mathcal{C}_j }
		=
		\frac{-8i b \lambda_j^{-2} \rho_j^\mu }{(\rho_j^2+1)^2 (\rho_j^\mu+1) }
		\Big(
		1-\frac{2 }{\rho_j^2+1}
		{\rm{Re}}
		\Big)
		\big( 	\Phi^{\J}_0
		e^{-i \gamma_j } \big).
\end{equation}
\end{small}
$\bullet$ Notice
$
	\pp_{r_j} U^{\J}
	=\lambda_j^{-1}  w_{\rho_j} Q_{\gamma_j} E_1^{\J}$,
$
	\pp_{\theta_j} U^{\J} = \sin w(\rho_j) Q_{\gamma_j} E_2^{\J} .
$
By \eqref{def-globalcorrection-J} and \eqref{Phi*-deri}, we have
\begin{align*}
%\begin{equation}
%	\begin{aligned}
		&
		-2\nabla_x(\Phi_0^{*\J}\cdot U^{\J}) \cdot \nabla_x U^{\J}
		=
		-2
		\pp_{r_j}(\Phi_0^{*\J}\cdot U^{\J}) \pp_{r_j} U^{\J}
		-2
		r_j^{-2}
		\pp_{\theta_j}(\Phi_0^{*\J}\cdot U^{\J}) \pp_{\theta_j} U^{\J}
		\nonumber
		\\
		= \ &
		-2
		\bigg\{ \Big[
		\big[
		\frac{\rho_j^{\mu+1}}{(\rho_j^\mu+ 1) (\rho_j^2+1)^{\frac{1}{2}} } \pp_{z_j} \Phi^{\J}_0
		+
		\frac{\mu\lambda_j^{-1} \rho_j^{\mu-1}}{(\rho_j^\mu+1)^2}
		\Phi^{\J}_0
		\big]
		e^{i\theta_j} ,
		0 \Big]^{\tr} \cdot U^{\J}
		\nonumber
		\\
		&
		+
		\Big[\frac{\rho_j^\mu}{\rho_j^\mu+ 1 }	\Phi^{\J}_0 e^{i\theta_j},0\Big]^{\tr}\cdot \lambda_j^{-1}  \frac{-2}{\rho_j^2+1}  Q_{\gamma_j} E_1^{\J} \bigg\} \lambda_j^{-1}  \frac{-2}{\rho_j^2+1} Q_{\gamma_j} E_1^{\J}
		\nonumber
		\\
		&
		-
		2r_j^{-2}
		\bigg\{ \Big[
		\frac{\rho_j^\mu}{\rho_j^\mu+ 1 }	\Phi^{\J}_0  ie^{i\theta_j} ,
		0 \Big]^{\tr}\cdot U^{\J}
		+
		\Big[\frac{\rho_j^\mu}{\rho_j^\mu+ 1 }	\Phi^{\J}_0 e^{i\theta_j},0\Big]^{\tr}
		\cdot \frac{2\rho_j}{\rho_j^2+1} Q_{\gamma_j} E_2^{\J} \bigg\} \frac{2\rho_j}{\rho_j^2+1} Q_{\gamma_j} E_2^{\J}
		\nonumber
		\\
		= \ &
		\bigg[
		\mathrm{Re}
		\bigg\{
		\bigg[
		\frac{\rho_j^{\mu+1}}{(\rho_j^\mu+ 1) (\rho_j^2+1)^{\frac{1}{2}} } \pp_{z_j} \Phi^{\J}_0
		+
		\frac{\mu\lambda_j^{-1} \rho_j^{\mu-1}}{(\rho_j^\mu+1)^2}
		\Phi^{\J}_0
		\bigg]
		\sin w(\rho_j)
		e^{-i\gamma_j}
		\bigg\}
		\nonumber
		\\
		&
		+
		\frac{-2 \lambda_j^{-1}\rho_j^\mu}{(\rho_j^\mu+ 1)(\rho_j^2+1) } \cos w(\rho_j) \mathrm{Re}\left(	\Phi^{\J}_0 e^{-i\gamma_j} \right)
		\bigg]   \frac{4 \lambda_j^{-1} }{\rho_j^2+1} Q_{\gamma_j} E_1^{\J}
		\nonumber
		\\
		&
		-
		2r_j^{-2}
		\bigg\{
		\frac{\rho_j^\mu}{\rho_j^\mu+ 1 } \sin w(\rho_j)	\mathrm{Re} \left(\Phi^{\J}_0  ie^{-i\gamma_j} \right)
		-
		\frac{2\rho_j^{\mu+1}}{(\rho_j^\mu+ 1)(\rho_j^2+1) }	\mathrm{Re} \left(\Phi^{\J}_0 ie^{-i\gamma_j} \right) \bigg\} \frac{2\rho_j}{\rho_j^2+1} Q_{\gamma_j} E_2^{\J}
		\nonumber
		\\
		= \ &
		\bigg\{
		\frac{8\lambda_j^{-1}\rho_j^{\mu+2}}{(\rho_j^\mu+ 1) (\rho_j^2+1)^{\frac{5}{2}} }
		\mathrm{Re} \left(\pp_{z_j} \Phi^{\J}_0 e^{-i\gamma_j} \right)
		+
		\bigg[
		\frac{8\mu\lambda_j^{-2} \rho_j^{\mu}}{(\rho_j^\mu+1)^2(\rho_j^2+1)^2}
		-
		\frac{8 \lambda_j^{-2}\rho_j^\mu (\rho_j^2-1)}{(\rho_j^\mu+ 1)(\rho_j^2+1)^3 }
		\bigg]
		\mathrm{Re} \left(\Phi^{\J}_0 e^{-i\gamma_j} \right)
		\bigg\} Q_{\gamma_j} E_1^{\J}.
		%\label{nab-nab-Phi*}
%	\end{aligned}
%\end{equation}
\end{align*}
Then, it is easy to see
\begin{equation}\label{nonloc-6-U}
	\big\{
	( a-bU^{\J}  \wedge )
	\big[
	- 2 \nabla_x \big(
	U^{\J} \cdot
	\Phi^{*{\J}}_0  \big)\cdot \nabla_x U^{\J}
	\big] \big\}
	\cdot U^{\J} =0.
\end{equation}
By \eqref{wedge+rota},
\begin{equation}\label{nonloc-6}
	\begin{aligned}
		&
		\left( a-bU^{\J}  \wedge \right)
		\left[
		- 2 \nabla_x \left(
		U^{\J} \cdot
		\Phi^{*{\J}}_0  \right) \cdot \nabla_x U^{\J}
		\right]
		=
		\bigg\{
		\frac{8\lambda_j^{-1}\rho_j^{\mu+2}}{(\rho_j^\mu+ 1) (\rho_j^2+1)^{\frac{5}{2}} }
		\mathrm{Re} \left(\pp_{z_j} \Phi^{\J}_0 e^{-i\gamma_j} \right)
		\\
		&
		\quad+
		\bigg[
		\frac{8\mu\lambda_j^{-2} \rho_j^{\mu}}{(\rho_j^\mu+1)^2(\rho_j^2+1)^2}
		-
		\frac{8 \lambda_j^{-2}\rho_j^\mu (\rho_j^2-1)}{(\rho_j^\mu+ 1)(\rho_j^2+1)^3 }
		\bigg]
		\mathrm{Re} \left(\Phi^{\J}_0 e^{-i\gamma_j} \right)
		\bigg\}
		\left(aQ_{\gamma_j} E_1^{\J}
		-b Q_{\gamma_j} E_2^{\J}
		\right).
	\end{aligned}
\end{equation}

In sum, by \eqref{proj-cal-C},  \eqref{noloc-1}, \eqref{nonloc-2-U}, \eqref{Ej-1}, \eqref{nonloc-3-U}, \eqref{nonloc-4}, \eqref{nonloc-5-U}, \eqref{nonloc-6-U}, one has
\begin{small}
	\begin{align*}
%\begin{equation*}%\label{Sj-U-dire}
%	\begin{aligned}
	&	\mathcal{S}^{\J} \cdot U^{\J}
		=
		\frac{2\rho_j}{\rho_j^2+1}
		\mathrm{Re} \bigg\{
		\bigg[
		\frac{  \dot{\xi}^{\J}\cdot y^{\J} \rho_j^\mu  }{ ( \rho_j^\mu + 1 ) ( \rho_j^2 + 1 )^{\frac{1}{2}} } \pp_{z_j} \Phi^{\J}_0
		+
		\Big[
		\frac{\mu\lambda_j^{-1} \dot{\xi}^{\J}\cdot y^{\J} \rho_j^{\mu-2} }{ (\rho_j^\mu+ 1)^2 }
		+
		\frac{ i \lambda_j^{-1}  (\dot{\xi}_2^{\J} y_1^{\J} - \dot{\xi}_1^{\J} y_2^{\J} ) \rho_j^{\mu-2} }{ \rho_j^\mu + 1 }
		\Big] \Phi^{\J}_0
		\\
		&~
		+
		\frac{ - \dot{\lambda}_j \rho_j^\mu }{ ( \rho_j^\mu + 1 ) ( \rho_j^2 + 1 )^{\frac{1}{2}} } \pp_{z_j} \Phi^{\J}_0
		+
		\frac{ \mu \lambda_j^{-1} \dot{\lambda}_j \rho_j^\mu}{ (\rho_j^\mu+ 1)^2 }   \Phi^{\J}_0
		 +
		(a-ib)\Big\{
		\frac{ - \rho_j^{\mu} }{ (\rho_j^\mu+ 1)( \rho_j^2+1 ) }
		\pp_{z_j z_j} \Phi^{\J}_0
		\\
		&~
		+
		\Big[
		\frac{ 2\mu\lambda_j^{-1} \rho_j^{\mu}}{(\rho_j^\mu+1)^2 (\rho_j^2+1)^{\frac{1}{2}}}
		+
		\frac{ \lambda_j^{-1} \rho_j^\mu}{ (\rho_j^\mu+ 1) (\rho_j^2+1)^{\frac{3}{2}}  }
		\Big]
		\pp_{z_j} \Phi^{\J}_0
		+
		\lambda_j^{-2}\rho_j^{\mu-2}
		\Big[
		\mu^2
		\frac{ 1 - \rho_j^{\mu}  }{(\rho_j^\mu+1 )^3 }
		-
		\frac{ 1}{ (\rho_j^\mu+ 1)(\rho_j^2+1) }
		\Big]
		\Phi^{\J}_0
		\Big\}
		\bigg]
		e^{-i\gamma_j}  \bigg\}
		\\
		& ~
		+
		\frac{4\lambda_j^{-1} \dot{\lambda}_j \rho_j^{\mu+1} [\rho_j^2 + \rho_j(\rho_j^2+1)^{\frac{1}{2}} +1 ]}{[\rho_j+ (\rho_j^2+1)^{\frac{1}{2}}] (\rho_j^\mu + 1)(\rho_j^2+1)^2}
		-
		\frac{2b \rho_j}{\rho_j^2+1}
		{\rm{Im}} \left[
		\left(
		(\Delta_x \Phi_0^{*\J} )_1
		+
		i
		(\Delta_x \Phi_0^{*\J} )_{2}
		\right)
		e^{-i(\theta_j+\gamma_j)}
		\right]
		\\
		& ~+
		\frac{16 a \lambda_j^{-2} \rho_j^{1+\mu}}{(\rho_j^2+ 1)^3(\rho_j^\mu+1) }	
		{\rm{Re}} ( 	\Phi^{\J}_0  e^{-i \gamma_j } ).
%	\end{aligned}
%\end{equation*}
	\end{align*}
\end{small}
By $\mu\ge 2$, \eqref{lam-ansatz}, \eqref{Phi-0-j-upp}, and \eqref{Phi*-0-j-upp}, we have
\begin{equation}\label{Sj-U-dire-est}
	|\mathcal{S}^{\J} \cdot U^{\J} |
	\lesssim
	| \dot{\xi}^{\J}| \langle  \rho_j \rangle^{-1}
	+ |\lambda_j|^{-1}  \langle  \rho_j \rangle^{-2}.
\end{equation}

By \eqref{proj-cal-C}, \eqref{xi-dot-y} and \eqref{noloc-1}, \eqref{nonloc-2-E}, \eqref{Ej-1}, \eqref{nonlocal-3-E},  \eqref{nonloc-4}, \eqref{nonloc-5}, \eqref{nonloc-6}, we have
\begin{small}
\begin{align*}
%\begin{equation*}
%	\begin{aligned}
		&~
		\big(
		\Pi_{U^{\J \perp} } \mathcal{S}^{\J} \big)_{\mathcal{C}_j}
		=
		\big[
		(\dot{\xi}_1^{\J} - i\dot{\xi}_2^{\J} )e^{i\theta_j}
		+
		(\dot{\xi}_1^{\J} + i\dot{\xi}_2^{\J} )e^{-i\theta_j}
		\big]
		\\
		& \times
		\bigg[
		\frac{   \rho_j^{\mu+1}  }{ 2 ( \rho_j^\mu + 1 ) ( \rho_j^2 + 1 )^{\frac{1}{2}} } \bigg( 1-\frac{2}{\rho_j^2+1}\mathrm{Re} \bigg)\left(\pp_{z_j} \Phi^{\J}_0  e^{-i\gamma_j} \right)
		+
		\frac{\mu\lambda_j^{-1} \rho_j^{\mu-1} }{ 2(\rho_j^\mu+ 1)^2 }
		\bigg( 1-\frac{2}{\rho_j^2+1}\mathrm{Re} \bigg) \left(\Phi^{\J}_0  e^{-i\gamma_j} \right)
		\bigg]
		\\
		& +
		\left[
		(\dot{\xi}_2^{\J} + i\dot{\xi}_1^{\J} )e^{i\theta_j}
		+
		(\dot{\xi}_2^{\J} - i\dot{\xi}_1^{\J} )e^{-i\theta_j}
		\right]
		\frac{  \lambda_j^{-1}   \rho_j^{\mu-1} }{ 2(\rho_j^\mu + 1) }
		\bigg( i+\frac{2}{\rho_j^2+1}\mathrm{Im} \bigg) \left( \Phi^{\J}_0  e^{-i\gamma_j} \right)
		\\
		& +
		\bigg( 1-\frac{2}{\rho_j^2+1}\mathrm{Re} \bigg)
		\bigg[
		\bigg\{
		\frac{ - \dot{\lambda}_j \rho_j^\mu }{ ( \rho_j^\mu + 1 ) ( \rho_j^2 + 1 )^{\frac{1}{2}} } \pp_{z_j} \Phi^{\J}_0
		+
		\frac{ \mu \lambda_j^{-1} \dot{\lambda}_j \rho_j^\mu}{ (\rho_j^\mu+ 1)^2 }   \Phi^{\J}_0
		\\
		& +
		(a-ib)\Big\{
		\frac{ - \rho_j^{\mu} }{ (\rho_j^\mu+ 1)( \rho_j^2+1 ) }
		\pp_{z_j z_j} \Phi^{\J}_0
		+
		\bigg[
		\frac{ 2\mu\lambda_j^{-1} \rho_j^{\mu}}{(\rho_j^\mu+1)^2 (\rho_j^2+1)^{\frac{1}{2}}}
		+
		\frac{ \lambda_j^{-1} \rho_j^\mu}{ (\rho_j^\mu+ 1) (\rho_j^2+1)^{\frac{3}{2}}  }
		\bigg]
		\pp_{z_j} \Phi^{\J}_0
		\\
		&
		+
		\lambda_j^{-2}\rho_j^{\mu-2}
		\bigg[
		\mu^2
		\frac{ 1 - \rho_j^{\mu}  }{(\rho_j^\mu+1 )^3 }
		-
		\frac{ 1}{ (\rho_j^\mu+ 1)(\rho_j^2+1) }
		\bigg]
		\Phi^{\J}_0 \Big\} \bigg\} e^{-i\gamma_j} \bigg]
		\\
		&
		+
		(\la_j^{-1}\dot\la_j + i \dot\gamma_j )
		\frac{2\rho_j[ \rho_j^{\mu-1} - \rho_j  - (\rho_j^2+1)^{\frac{1}{2}} ] }{ [ \rho_j + (\rho_j^2+1)^{\frac{1}{2}} ] (\rho_j^\mu + 1)(\rho_j^2+1)}
		-
		\la_j^{-1}\dot\la_j
		\frac{ 4\rho_j^{\mu} [  \rho_j^2  + \rho_j(\rho_j^2+1)^{\frac{1}{2}} + 1 ]  }{ [ \rho_j + (\rho_j^2+1)^{\frac{1}{2}} ] (\rho_j^\mu + 1)(\rho_j^2+1)^2 }
		\\
		&
		-
		\frac{2\la_j^{-1} (\dot \xi_1^{\J}  -i \dot\xi_2^{\J}  )  }{\rho_j^2+1} e^{i\theta_j}
	 +
		\bigg\{
		\frac{ 2i b \rho_j^{\mu+2} }{ (\rho_j^\mu+ 1)( \rho_j^2+1 )^2 }
		\pp_{z_j z_j} \overline{ \Phi^{\J}_0 }
		+
		\frac{2i b \lambda_j^{-1} \rho_j^{\mu} }{\rho_j^2+1}
		\Big[
		\frac{ 2\mu }{(\rho_j^\mu+1)^2 (\rho_j^2+1)^{\frac{1}{2}}}
		+
		\frac{ \rho_j^2 +2 }{ (\rho_j^\mu+ 1) (\rho_j^2+1)^{\frac{3}{2}}  }
		\Big]
		\pp_{z_j} \overline{\Phi^{\J}_0 }
		\\
		&
		-
		\frac{2i b \lambda_j^{-2}\rho_j^{\mu-2}  }{\rho_j^2+1}
		\bigg[
		\mu^2
		\frac{\rho_j^{\mu} -1 }{(\rho_j^\mu+1 )^3 }
		+
		\frac{1}{\rho_j^\mu+ 1 }
		\bigg]
		\overline{\Phi^{\J}_0 }
		\bigg\} e^{i\gamma_j}
	 +
		\frac{8(a-ib) \lambda_j^{-2} \rho_j^\mu}{(\rho_j^2+ 1)^2(\rho_j^\mu+1) }
		\bigg(1
		-
		\frac{2}{\rho_j^2+1}
		{\rm{Re}}
		\bigg) \left( 	\Phi^{\J}_0  e^{-i \gamma_j } \right)
		\\
		&
		+ (a-ib)
		\bigg\{
		\frac{8\lambda_j^{-1}\rho_j^{\mu+2}}{(\rho_j^\mu+ 1) (\rho_j^2+1)^{\frac{5}{2}} }
		\mathrm{Re} \left(\pp_{z_j} \Phi^{\J}_0 e^{-i\gamma_j} \right)
		+
		\bigg[
		\frac{8\mu\lambda_j^{-2} \rho_j^{\mu}}{(\rho_j^\mu+1)^2(\rho_j^2+1)^2}
		-
		\frac{8 \lambda_j^{-2}\rho_j^\mu (\rho_j^2-1)}{(\rho_j^\mu+ 1)(\rho_j^2+1)^3 }
		\bigg]
		\mathrm{Re} \left(\Phi^{\J}_0 e^{-i\gamma_j} \right)
		\bigg\}
		\\
%	\end{aligned}
%\end{equation*}
%\begin{equation*}
%	\begin{aligned}
		= \ &
		\bigg( 1-\frac{2}{\rho_j^2+1}\mathrm{Re} \bigg)
		\bigg[
		\bigg\{
		\frac{ - \dot{\lambda}_j \rho_j^\mu }{ ( \rho_j^\mu + 1 ) ( \rho_j^2 + 1 )^{\frac{1}{2}} } \pp_{z_j} \Phi^{\J}_0
		+
		\frac{ \mu \lambda_j^{-1} \dot{\lambda}_j \rho_j^\mu}{ (\rho_j^\mu+ 1)^2 }   \Phi^{\J}_0
		\\
		& +
		(a-ib)\Big\{
		\frac{ - \rho_j^{\mu} }{ (\rho_j^\mu+ 1)( \rho_j^2+1 ) }
		\pp_{z_j z_j} \Phi^{\J}_0
		+
		\bigg[
		\frac{ 2\mu\lambda_j^{-1} \rho_j^{\mu}}{(\rho_j^\mu+1)^2 (\rho_j^2+1)^{\frac{1}{2}}}
		+
		\frac{ \lambda_j^{-1} \rho_j^\mu}{ (\rho_j^\mu+ 1) (\rho_j^2+1)^{\frac{3}{2}}  }
		\bigg]
		\pp_{z_j} \Phi^{\J}_0
		\\
		&
		+
		\lambda_j^{-2}\rho_j^{\mu-2}
		\bigg[
		\mu^2
		\frac{ 1 - \rho_j^{\mu}  }{(\rho_j^\mu+1 )^3 }
		-
		\frac{ 1}{ (\rho_j^\mu+ 1)(\rho_j^2+1) }
		\bigg]
		\Phi^{\J}_0 \Big\} \bigg\} e^{-i\gamma_j} \bigg]
		\\
		&
		+
		(\la_j^{-1}\dot\la_j + i \dot\gamma_j )
		\frac{2\rho_j[ \rho_j^{\mu-1} - \rho_j  - (\rho_j^2+1)^{\frac{1}{2}} ] }{ [ \rho_j + (\rho_j^2+1)^{\frac{1}{2}} ] (\rho_j^\mu + 1)(\rho_j^2+1)}
		-
		\la_j^{-1}\dot\la_j
		\frac{ 4\rho_j^{\mu} [  \rho_j^2  + \rho_j(\rho_j^2+1)^{\frac{1}{2}} + 1 ]  }{ [ \rho_j + (\rho_j^2+1)^{\frac{1}{2}} ] (\rho_j^\mu + 1)(\rho_j^2+1)^2 }
		\\
		& +
		\bigg\{
		\frac{ 2i b \rho_j^{\mu+2} }{ (\rho_j^\mu+ 1)( \rho_j^2+1 )^2 }
		\pp_{z_j z_j} \overline{ \Phi^{\J}_0 }
		+
		\frac{2i b \lambda_j^{-1} \rho_j^{\mu} }{\rho_j^2+1}
		\Big[
		\frac{ 2\mu }{(\rho_j^\mu+1)^2 (\rho_j^2+1)^{\frac{1}{2}}}
		+
		\frac{ \rho_j^2 +2 }{ (\rho_j^\mu+ 1) (\rho_j^2+1)^{\frac{3}{2}}  }
		\Big]
		\pp_{z_j} \overline{\Phi^{\J}_0 }
		\\
		&
		-
		\frac{2i b \lambda_j^{-2}\rho_j^{\mu-2}  }{\rho_j^2+1}
		\bigg[
		\mu^2
		\frac{\rho_j^{\mu} -1 }{(\rho_j^\mu+1 )^3 }
		+
		\frac{1}{\rho_j^\mu+ 1 }
		\bigg]
		\overline{\Phi^{\J}_0 }
		\bigg\} e^{i\gamma_j}
		 +
		\frac{8(a-ib) \lambda_j^{-2} \rho_j^\mu}{(\rho_j^2+ 1)^2(\rho_j^\mu+1) }
		\bigg(1
		-
		\frac{2}{\rho_j^2+1}
		{\rm{Re}}
		\bigg) \left( 	\Phi^{\J}_0  e^{-i \gamma_j } \right)
		\\
		&
		+ (a-ib)
		\bigg\{
		\frac{8\lambda_j^{-1}\rho_j^{\mu+2}}{(\rho_j^\mu+ 1) (\rho_j^2+1)^{\frac{5}{2}} }
		\mathrm{Re} \left(\pp_{z_j} \Phi^{\J}_0 e^{-i\gamma_j} \right)
		+
		\bigg[
		\frac{8\mu\lambda_j^{-2} \rho_j^{\mu}}{(\rho_j^\mu+1)^2(\rho_j^2+1)^2}
		-
		\frac{8 \lambda_j^{-2}\rho_j^\mu (\rho_j^2-1)}{(\rho_j^\mu+ 1)(\rho_j^2+1)^3 }
		\bigg]
		\mathrm{Re} \left(\Phi^{\J}_0 e^{-i\gamma_j} \right)
		\bigg\}
		\\
		&
		+ e^{i\theta_j}
		\bigg\{
		\frac{ - 2\la_j^{-1} (\dot \xi_1^{\J}  -i \dot\xi_2^{\J}  )  }{\rho_j^2+1}
		+
		(\dot{\xi}_1^{\J} - i\dot{\xi}_2^{\J} )
		\bigg[
		\frac{   \rho_j^{\mu+1}  }{ 2 ( \rho_j^\mu + 1 ) ( \rho_j^2 + 1 )^{\frac{1}{2}} } \bigg( 1-\frac{2}{\rho_j^2+1}\mathrm{Re} \bigg)\left(\pp_{z_j} \Phi^{\J}_0  e^{-i\gamma_j} \right)
		\\
		&
		+
		\frac{\mu\lambda_j^{-1} \rho_j^{\mu-1} }{ 2(\rho_j^\mu+ 1)^2 }
		\bigg( 1-\frac{2}{\rho_j^2+1}\mathrm{Re} \bigg) \left(\Phi^{\J}_0  e^{-i\gamma_j} \right)
		\bigg]
		+
		(\dot{\xi}_2^{\J} + i\dot{\xi}_1^{\J} )
		\frac{  \lambda_j^{-1}   \rho_j^{\mu-1} }{ 2(\rho_j^\mu + 1) }
		\bigg( i+\frac{2}{\rho_j^2+1}\mathrm{Im} \bigg) \left( \Phi^{\J}_0  e^{-i\gamma_j} \right)
		\bigg\}
		\\
		& +
		e^{-i\theta_j}
		\bigg\{
		(\dot{\xi}_1^{\J} + i\dot{\xi}_2^{\J} )
		\bigg[
		\frac{   \rho_j^{\mu+1}  }{ 2 ( \rho_j^\mu + 1 ) ( \rho_j^2 + 1 )^{\frac{1}{2}} } \bigg( 1-\frac{2}{\rho_j^2+1}\mathrm{Re} \bigg)\left(\pp_{z_j} \Phi^{\J}_0 e^{-i\gamma_j}  \right)
		\\
		&
		+
		\frac{\mu\lambda_j^{-1} \rho_j^{\mu-1} }{ 2(\rho_j^\mu+ 1)^2 }
		\bigg( 1-\frac{2}{\rho_j^2+1}\mathrm{Re} \bigg) \left(\Phi^{\J}_0  e^{-i\gamma_j} \right)
		\bigg]
		+
		(\dot{\xi}_2^{\J} - i\dot{\xi}_1^{\J} )
		\frac{  \lambda_j^{-1}   \rho_j^{\mu-1} }{ 2(\rho_j^\mu + 1) }
		\bigg( i+\frac{2}{\rho_j^2+1}\mathrm{Im} \bigg) \left( \Phi^{\J}_0  e^{-i\gamma_j} \right)
		\bigg\}.
%	\end{aligned}
%\end{equation*}
\end{align*}
\end{small}
Using \eqref{def-Phi0} and \eqref{Phi-J-deri}, we have
\begin{small}
\begin{align*}
%\begin{equation*}
%	\begin{aligned}
		&
		\left(
		\Pi_{U^{\J \perp} } \mathcal{S}^{\J} \right)_{\mathcal{C}_j}
		=
		\bigg( 1-\frac{2}{\rho_j^2+1}\mathrm{Re} \bigg)
		\bigg\{
		\frac{ \dot{\lambda}_j \rho_j^\mu }{ ( \rho_j^\mu + 1 ) ( \rho_j^2 + 1 )^{\frac{1}{2}} }
		\int_{-T}^t \frac{\dot p_j(s) e^{-i\gamma_j(t)} }{t-s}
		\left(
		K_0(\zeta_j)
		+ 2\zeta_j K_{0\zeta_j}(\zeta_j)
		\right)
		ds
		\\
		&
		-
		\frac{ \mu \dot{\lambda}_j \rho_j^\mu (\rho_j^2+1)^{\frac{1}{2}} }{ (\rho_j^\mu+ 1)^{2} }   \int_{-T}^t \frac{\dot p_j(s) e^{-i\gamma_j(t)} }{t-s}K_0(\zeta_j)ds
		\\
		& +
		(a-ib)\Big\{
		\frac{  \lambda_j^{-1} \rho_j^{\mu} }{ (\rho_j^\mu+ 1)( \rho_j^2+1 )^{\frac{3}{2}} }
		\int_{-T}^t \frac{\dot p_j(s) e^{-i\gamma_j(t)} }{t-s}
		\left(
		6\zeta_j K_{0\zeta_j}( \zeta_j )
		+ 4\zeta_j^2 K_{0\zeta_j \zeta_j}( \zeta_j )
		\right)
		ds
		\\
		&
		-
		\bigg[
		\frac{ 2\mu\lambda_j^{-1} \rho_j^{\mu}}{(\rho_j^\mu+1)^2 (\rho_j^2+1)^{\frac{1}{2}} }
		+
		\frac{ \lambda_j^{-1} \rho_j^\mu}{ (\rho_j^\mu+ 1) (\rho_j^2+1)^{\frac{3}{2}}  }
		\bigg]
		\int_{-T}^t \frac{\dot p_j(s) e^{-i\gamma_j(t)} }{t-s}
		\left(
		K_0(\zeta_j)
		+ 2\zeta_j K_{0\zeta_j}(\zeta_j)
		\right)
		ds
		\\
		&
		-
		\lambda_j^{-1}\rho_j^{\mu-2}
		\bigg[
		\mu^2
		\frac{ (1 - \rho_j^{\mu}) (\rho_j^2+1)^{\frac{1}{2}}  }{(\rho_j^\mu+1 )^{3} }
		-
		\frac{ 1}{ (\rho_j^\mu+ 1)(\rho_j^2+1)^{\frac{1}{2}} }
		\bigg] \int_{-T}^t \frac{\dot p_j(s) e^{-i\gamma_j(t)} }{t-s}K_0(\zeta_j)ds \Big\} \bigg\}
		\\
		&
		+
		(\la_j^{-1}\dot\la_j + i \dot\gamma_j )
		\frac{2\rho_j[ \rho_j^{\mu-1} - \rho_j  - (\rho_j^2+1)^{\frac{1}{2}} ] }{ [ \rho_j + (\rho_j^2+1)^{\frac{1}{2}} ] (\rho_j^\mu + 1)(\rho_j^2+1)}
		-
		\la_j^{-1}\dot\la_j
		\frac{ 4\rho_j^{\mu} [  \rho_j^2  + \rho_j(\rho_j^2+1)^{\frac{1}{2}} + 1 ]  }{ [ \rho_j + (\rho_j^2+1)^{\frac{1}{2}} ] (\rho_j^\mu + 1)(\rho_j^2+1)^2 }
		\\
		& +
		\frac{ - 2i b  \lambda_j^{-1}  \rho_j^{\mu+2} }{ (\rho_j^\mu+ 1)( \rho_j^2+1 )^{\frac{5}{2}} }
		\int_{-T}^t \frac{ \overline{\dot{p}_j}(s) e^{i\gamma_j(t)}   }{t-s}
		\left(
		6\zeta_j \overline{K_{0\zeta_j}}( \zeta_j )
		+ 4\zeta_j^2 \overline{K_{0\zeta_j \zeta_j}}( \zeta_j )
		\right)
		ds
		\\
		&
		-
		\frac{2i b \lambda_j^{-1} \rho_j^{\mu} }{\rho_j^2+1}
		\Big[
		\frac{ 2\mu }{(\rho_j^\mu+1)^2 (\rho_j^2+1)^{\frac{1}{2}}}
		+
		\frac{ \rho_j^2 +2 }{ (\rho_j^\mu+ 1) (\rho_j^2+1)^{\frac{3}{2}}  }
		\Big]
		\int_{-T}^t \frac{\overline{\dot{p}_j}(s) e^{i\gamma_j(t)}  }{t-s}
		\left(
		\overline{K_0}(\zeta_j)
		+ 2\zeta_j \overline{K_{0\zeta_j}}(\zeta_j)
		\right)
		ds
		\\
		&
		+
		\frac{2i b \lambda_j^{-1}\rho_j^{\mu-2}  }{ (\rho_j^2+1)^{\frac{1}{2}} }
		\bigg[
		\mu^2
		\frac{\rho_j^{\mu} -1 }{(\rho_j^\mu+1 )^3 }
		+
		\frac{1}{\rho_j^\mu+ 1 }
		\bigg]
		\int_{-T}^t \frac{ \overline{\dot{p}_j}(s) e^{i\gamma_j(t)} }{t-s}\overline{K_0}(\zeta_j)ds
		\\
		& -
		\frac{8(a-ib) \lambda_j^{-1} \rho_j^\mu}{(\rho_j^2+ 1)^{\frac{3}{2}}(\rho_j^\mu+1) }
		\bigg(1
		-
		\frac{2}{\rho_j^2+1}
		{\rm{Re}}
		\bigg) \bigg( 	\int_{-T}^t \frac{\dot p_j(s) e^{-i \gamma_j(t) } }{t-s}K_0(\zeta_j)ds   \bigg)
		\\
		&
		+ (a-ib)
		\bigg\{
		\frac{-8\lambda_j^{-1}\rho_j^{\mu+2}}{(\rho_j^\mu+ 1) (\rho_j^2+1)^{\frac{5}{2}} }
		\mathrm{Re} \bigg( \int_{-T}^t \frac{\dot p_j(s) e^{-i\gamma_j(t)} }{t-s}
		\left(
		K_0(\zeta_j)
		+ 2\zeta_j K_{0\zeta_j}(\zeta_j)
		\right)
		ds  \bigg)
		\\
		&
		-
		\bigg[
		\frac{8\mu\lambda_j^{-1} \rho_j^{\mu}}{(\rho_j^\mu+1)^2(\rho_j^2+1)^{\frac{3}{2}}}
		-
		\frac{8 \lambda_j^{-1}\rho_j^\mu (\rho_j^2-1)}{(\rho_j^\mu+ 1)(\rho_j^2+1)^{\frac{5}{2}} }
		\bigg]
		\mathrm{Re} \bigg( \int_{-T}^t \frac{\dot p_j(s) e^{-i\gamma_j(t)} }{t-s}K_0(\zeta_j)ds \bigg)
		\bigg\}
		\\
%	\end{aligned}
%\end{equation*}
%\begin{equation*}
%	\begin{aligned}
		&
		+ e^{i\theta_j}
		\bigg\{  - (\dot{\xi}_1^{\J} - i\dot{\xi}_2^{\J} )
		\bigg[
		\frac{ 2\la_j^{-1} }{\rho_j^2+1}
		+
		\frac{   \rho_j^{\mu+1}  }{ 2 ( \rho_j^\mu + 1 ) ( \rho_j^2 + 1 )^{\frac{1}{2}} } \bigg( 1-\frac{2}{\rho_j^2+1}\mathrm{Re} \bigg)\bigg( \int_{-T}^t \frac{\dot p_j(s) e^{-i\gamma_j(t)}  }{t-s}
		\left(
		K_0(\zeta_j)
		+ 2\zeta_j K_{0\zeta_j}(\zeta_j)
		\right)
		ds  \bigg)
		\\
		&
		+
		\frac{\mu \rho_j^{\mu-1} (\rho_j^2+1)^{\frac{1}{2}} }{ 2(\rho_j^\mu+ 1)^2 }
		\bigg( 1-\frac{2}{\rho_j^2+1}\mathrm{Re} \bigg) \bigg( \int_{-T}^t \frac{\dot p_j(s) e^{-i\gamma_j(t)}  }{t-s}K_0(\zeta_j)ds \bigg)
		\bigg]
		\\
		&
		-
		(\dot{\xi}_2^{\J} + i\dot{\xi}_1^{\J} )
		\frac{  \rho_j^{\mu-1} (\rho_j^2+1)^{\frac{1}{2}} }{ 2(\rho_j^\mu + 1) }
		\bigg( i+\frac{2}{\rho_j^2+1}\mathrm{Im} \bigg) \bigg( \int_{-T}^t \frac{\dot p_j(s) e^{-i\gamma_j(t)} }{t-s}K_0(\zeta_j)ds  \bigg)
		\bigg\}
		\\
		& +
		e^{-i\theta_j}
		\bigg\{
		(\dot{\xi}_1^{\J} + i\dot{\xi}_2^{\J} )
		\bigg[
		\frac{  -\rho_j^{\mu+1}  }{ 2 ( \rho_j^\mu + 1 ) ( \rho_j^2 + 1 )^{\frac{1}{2}} } \bigg( 1-\frac{2}{\rho_j^2+1}\mathrm{Re} \bigg)\bigg( \int_{-T}^t \frac{\dot p_j(s) e^{-i\gamma_j(t) }  }{t-s}
		\left(
		K_0(\zeta_j)
		+ 2\zeta_j K_{0\zeta_j}(\zeta_j)
		\right)
		ds  \bigg)
		\\
		&
		-
		\frac{\mu \rho_j^{\mu-1} (\rho_j^2+1)^{\frac{1}{2}} }{ 2(\rho_j^\mu+ 1)^2 }
		\bigg( 1-\frac{2}{\rho_j^2+1}\mathrm{Re} \bigg) \bigg( \int_{-T}^t \frac{\dot p_j(s) e^{-i\gamma_j(t)} }{t-s}K_0(\zeta_j)ds  \bigg)
		\bigg]
		\\
		&
		-
		(\dot{\xi}_2^{\J} - i\dot{\xi}_1^{\J} )
		\frac{ \rho_j^{\mu-1}(\rho_j^2+1)^{\frac{1}{2}} }{ 2(\rho_j^\mu + 1) }
		\bigg( i+\frac{2}{\rho_j^2+1}\mathrm{Im} \bigg) \bigg( \int_{-T}^t \frac{\dot p_j(s) e^{-i\gamma_j(t)}  }{t-s}K_0(\zeta_j)ds  \bigg)
		\bigg\}
		\\
		:= \ &
		M_0^{\J}(\rho_j,t)
		+
		\tilde{M}_0^{\J}(\rho_j,t)
		+ e^{i\theta_j}
		\big( M_1^{\J}(\rho_j,t)
		+
		\tilde{M}_1^{\J}(\rho_j,t)
		\big)
		+ e^{-i\theta_j} M_{-1}^{\J}(\rho_j,t),
%	\end{aligned}
%\end{equation*}
\end{align*}
\end{small}
where $M_0^{\J}(\rho_j,t), \tilde{M}_0^{\J}(\rho_j,t),  M_1^{\J}(\rho_j,t),
\tilde{M}_1^{\J}(\rho_j,t), M_{-1}^{\J}(\rho_j,t) $ are given as follows.
\begin{small}
\begin{align*}
%\begin{equation}
%	\begin{aligned}
		&
		M_0^{\J}(\rho_j,t)
		:=
		\lambda_j^{-1}
		\bigg( 1-\frac{2}{\rho_j^2+1}\mathrm{Re} \bigg)
		\bigg\{
		(a-ib)\bigg[
		\nonumber
		\\
		&
		\int_{-T}^t \frac{\dot p_j(s) e^{-i\gamma_j(t)} }{t-s}
		\bigg\{
		\bigg[
		\frac{\rho_j^{\mu-2}}{ (\rho_j^\mu+ 1) (\rho_j^2+1)^{\frac{3}{2}}  }
		-
		\frac{ 2\mu  \rho_j^{\mu}}{(\rho_j^\mu+1)^2 (\rho_j^2+1)^{\frac{1}{2}} }
		-
		\mu^2
		\frac{ \rho_j^{\mu-2}  (1 - \rho_j^{\mu}) (\rho_j^2+1)^{\frac{1}{2}}  }{(\rho_j^\mu+1 )^{3} }
		\bigg]K_0(\zeta_j)
		\nonumber
		\\
		& +
		\bigg[
		\frac{ 4 \rho_j^{\mu} }{ (\rho_j^\mu+ 1)( \rho_j^2+1 )^{\frac{3}{2}} }
		-
		\frac{ 4\mu  \rho_j^{\mu}}{(\rho_j^\mu+1)^2 (\rho_j^2+1)^{\frac{1}{2}} }
		\bigg]
		\zeta_j K_{0\zeta_j}( \zeta_j )
		+
		\frac{ 4 \rho_j^{\mu} }{ (\rho_j^\mu+ 1)( \rho_j^2+1 )^{\frac{3}{2}} }
		\zeta_j^2 K_{0\zeta_j \zeta_j}( \zeta_j )
		\bigg\}
		ds   \bigg] \bigg\}
		\nonumber
		\\
		&
		-
		i b  \lambda_j^{-1}
		\int_{-T}^t \frac{ \overline{\dot{p}_j}(s) e^{i\gamma_j(t)} }{t-s}
		\bigg\{
		\bigg[
		\frac{ 2 \rho_j^{\mu} [2\mu(\rho_j^2+1) + (\rho_j^2+2)(\rho_j^\mu+1) ]}{ (\rho_j^\mu+ 1)^2 (\rho_j^2+1)^{\frac{5}{2}}  }
		-
		\frac{2 \rho_j^{\mu-2} [ \mu^2   (\rho_j^{\mu} -1) + (\rho_j^\mu +1 )^2 ] }{(\rho_j^\mu+1 )^3 (\rho_j^2+1)^{\frac{1}{2}} }
		\bigg]
		\overline{K_0}(\zeta_j)
		\nonumber
		\\
		&
		+
		\bigg[
		\frac{ 8\mu \rho_j^{\mu}  }{(\rho_j^\mu+1)^2 (\rho_j^2+1)^{\frac{3}{2}}}
		+
		\frac{ 16\rho_j^{\mu+2} +8\rho_j^\mu }{ (\rho_j^\mu+ 1) (\rho_j^2+1)^{\frac{5}{2}}  }
		\bigg]
		\zeta_j \overline{K_{0\zeta_j}}(\zeta_j)
		+
		\frac{ 8 \rho_j^{\mu+2} }{ (\rho_j^\mu+ 1)( \rho_j^2+1 )^{\frac{5}{2}} }
		\zeta_j^2 \overline{K_{0\zeta_j \zeta_j}}( \zeta_j )
		\bigg\} ds
		\nonumber
		\\
		& -
		(a-ib) \lambda_j^{-1}
		\bigg(1
		-
		\frac{2}{\rho_j^2+1}
		{\rm{Re}}
		\bigg) \bigg( 	\int_{-T}^t \frac{\dot p_j(s) e^{-i \gamma_j(t) } }{t-s}
		\frac{8  \rho_j^\mu  }{(\rho_j^2+ 1)^{\frac{3}{2}}(\rho_j^\mu+1) }
		K_0(\zeta_j)
		ds   \bigg)
		\nonumber
		\\
		&
		- (a-ib)  \lambda_j^{-1}
		\mathrm{Re}
		\bigg[
		\int_{-T}^t \frac{\dot p_j(s) e^{-i\gamma_j(t)} }{t-s}
		\bigg\{
		\bigg[
		\frac{8\mu  \rho_j^{\mu}}{(\rho_j^\mu+1)^2(\rho_j^2+1)^{\frac{3}{2}}}
		+
		\frac{ 8\rho_j^\mu }{(\rho_j^\mu+ 1)(\rho_j^2+1)^{\frac{5}{2}} }
		\bigg]
		K_0(\zeta_j)
		\nonumber
		\\
		&
		+
		\frac{16 \rho_j^{\mu+2}}{(\rho_j^\mu+ 1) (\rho_j^2+1)^{\frac{5}{2}} }\zeta_j K_{0\zeta_j}(\zeta_j)
		\bigg\}
		ds  \bigg]
		\nonumber
		\\
		&
		+
		(\la_j^{-1}\dot\la_j + i \dot\gamma_j )
		\frac{2\rho_j[ \rho_j^{\mu-1} - \rho_j  - (\rho_j^2+1)^{\frac{1}{2}} ] }{ [ \rho_j + (\rho_j^2+1)^{\frac{1}{2}} ] (\rho_j^\mu + 1)(\rho_j^2+1)}
		-
		\la_j^{-1}\dot\la_j
		\frac{ 4\rho_j^{\mu} [  \rho_j^2  + \rho_j(\rho_j^2+1)^{\frac{1}{2}} + 1 ]  }{ [ \rho_j + (\rho_j^2+1)^{\frac{1}{2}} ] (\rho_j^\mu + 1)(\rho_j^2+1)^2 }, %\label{M0-def}
%	\end{aligned}
%\end{equation}
\end{align*}
\end{small}
\begin{small}
\begin{align}
%\begin{equation}
%	\begin{aligned}
		\tilde{M}_0^{\J}(\rho_j,t) := \ &
		\dot{\lambda}_j  \bigg( 1-\frac{2}{\rho_j^2+1}\mathrm{Re} \bigg)
		\bigg[
		\int_{-T}^t \frac{\dot p_j(s) e^{-i\gamma_j(t)} }{t-s}
		\bigg\{
		\bigg[
		\frac{ \rho_j^\mu }{ ( \rho_j^\mu + 1 ) ( \rho_j^2 + 1 )^{\frac{1}{2}} }
		-
		\frac{ \mu  \rho_j^\mu (\rho_j^2+1)^{\frac{1}{2}} }{ (\rho_j^\mu+ 1)^{2} }
		\bigg]
		K_0(\zeta_j)
		\nonumber
		\\
		&
		+
		\frac{ 2 \rho_j^\mu }{ ( \rho_j^\mu + 1 ) ( \rho_j^2 + 1 )^{\frac{1}{2}} }
		\zeta_j K_{0\zeta_j}(\zeta_j)
		\bigg\}
		ds \bigg],
		\label{til-M0-def}
%	\end{aligned}
%\end{equation}
\end{align}
\end{small}
\begin{equation}\label{M1-def}
	M_1^{\J}(\rho_j,t)
	:=  - 2\la_j^{-1} (\dot \xi_1^{\J}  -i \dot\xi_2^{\J}  )
	(\rho_j^2+1)^{-1},
\end{equation}
\begin{small}
\begin{equation}\label{tilde-M1-def}
	\begin{aligned}
		&
		\tilde{M}_1^{\J}(\rho_j,t)
		: =
		-
		(\dot{\xi}_1^{\J} - i\dot{\xi}_2^{\J} )
		\bigg( 1-\frac{2}{\rho_j^2+1}\mathrm{Re} \bigg)
		\bigg[
		\int_{-T}^t \frac{\dot p_j(s) e^{-i\gamma_j(t)}  }{t-s}
		\bigg\{
		\bigg[
		\frac{   \rho_j^{\mu+1}  }{ 2 ( \rho_j^\mu + 1 ) ( \rho_j^2 + 1 )^{\frac{1}{2}} }
		+
		\frac{\mu \rho_j^{\mu-1} (\rho_j^2+1)^{\frac{1}{2}} }{ 2(\rho_j^\mu+ 1)^2 }
		\bigg]
		K_0(\zeta_j)
		\\
		&
		+
		\frac{   \rho_j^{\mu+1}  }{  ( \rho_j^\mu + 1 ) ( \rho_j^2 + 1 )^{\frac{1}{2}} } \zeta_j K_{0\zeta_j}(\zeta_j)
		\bigg\}
		ds
		\bigg]
		-
		(\dot{\xi}_2^{\J} + i\dot{\xi}_1^{\J} )
		\frac{  \rho_j^{\mu-1} (\rho_j^2+1)^{\frac{1}{2}} }{ 2(\rho_j^\mu + 1) }
		\bigg( i+\frac{2}{\rho_j^2+1}\mathrm{Im} \bigg) \bigg( \int_{-T}^t \frac{\dot p_j(s) e^{-i\gamma_j(t)} }{t-s}K_0(\zeta_j)ds  \bigg),
	\end{aligned}
\end{equation}
\end{small}
\begin{small}
\begin{equation}\label{M(-1)-def}
	\begin{aligned}
		&
		M_{-1}^{\J}(\rho_j,t)
		:=
		- (\dot{\xi}_1^{\J} + i\dot{\xi}_2^{\J} )
		\bigg( 1-\frac{2}{\rho_j^2+1}\mathrm{Re} \bigg)
		\bigg[
		\int_{-T}^t \frac{\dot p_j(s) e^{-i\gamma_j(t) }  }{t-s}
		\bigg\{
		\bigg[
		\frac{  \rho_j^{\mu+1}  }{ 2 ( \rho_j^\mu + 1 ) ( \rho_j^2 + 1 )^{\frac{1}{2}} }
		+
		\frac{\mu \rho_j^{\mu-1} (\rho_j^2+1)^{\frac{1}{2}} }{ 2(\rho_j^\mu+ 1)^2 }
		\bigg]
		K_0(\zeta_j)
		\\
		&
		+
		\frac{  \rho_j^{\mu+1}  }{  ( \rho_j^\mu + 1 ) ( \rho_j^2 + 1 )^{\frac{1}{2}} }
		\zeta_j K_{0\zeta_j}(\zeta_j)
		\bigg\}
		ds
		\bigg]
		-
		(\dot{\xi}_2^{\J} - i\dot{\xi}_1^{\J} )
		\frac{ \rho_j^{\mu-1}(\rho_j^2+1)^{\frac{1}{2}} }{ 2(\rho_j^\mu + 1) }
		\bigg( i+\frac{2}{\rho_j^2+1}\mathrm{Im} \bigg) \bigg( \int_{-T}^t \frac{\dot p_j(s) e^{-i\gamma_j(t)}  }{t-s}K_0(\zeta_j)ds  \bigg).
	\end{aligned}
\end{equation}
\end{small}

To avoid non-smoothness due to terms like $e^{i\theta_j}$, we need to take $\mu$ not so small to reserve some vanishing of $\rho_j$ as $\rho_j\rightarrow 0$ in the new error. From now on, we take $\mu=3$. Then
\begin{small}
\begin{align} \notag
		M_0^{\J}
		= &~
		\lambda_j^{-1}
		\bigg( 1-\frac{2}{\rho_j^2+1}\mathrm{Re} \bigg)
		\bigg[
		(a-ib)\bigg\{
		\int_{-T}^t \frac{\dot p_j(s) e^{-i\gamma_j(t)} }{t-s}
		\bigg[
		\frac{\rho_j (3\rho_j^7+ \rho_j^6 +12\rho_j^5 -15\rho_j^4+11\rho_j^3-24\rho_j^2-8)}{(\rho_j^2+1)^{\frac{3}{2}} (\rho_j^3+1)^3 }
		K_0(\zeta_j)
		\\ \notag
		& ~+
		\frac{4\rho_j^3(\rho_j^3-3\rho_j^2-2)}{(\rho_j^2+1)^{\frac{3}{2}} (\rho_j^3+1)^2 }
		\zeta_j K_{0\zeta_j}( \zeta_j )
		+
		\frac{ 4 \rho_j^{3} }{ (\rho_j^3+ 1)( \rho_j^2+1 )^{\frac{3}{2}} }
		\zeta_j^2 K_{0\zeta_j \zeta_j}( \zeta_j )
		\bigg]
		ds   \bigg\} \bigg]
		\\ \notag
		&~
		-
		i b  \lambda_j^{-1}
		\int_{-T}^t \frac{ \overline{\dot{p}_j}(s) e^{i\gamma_j(t)} }{t-s}
		\bigg[
		-\frac{2\rho_j ( 3\rho_j^7 + \rho_j^6 +12\rho_j^5 -15 \rho_j^4 +11\rho_j^3-24\rho_j^2-8)}{(\rho_j^2+1)^{\frac{5}{2}} (\rho_j^3+1)^3}
		\overline{K_0}(\zeta_j)
		\\ \notag
		&~
		+
		\frac{8\rho_j^3(2\rho_j^5 +\rho_j^3+5\rho_j^2+4)}{(\rho_j^2+1)^{\frac{5}{2}} (\rho_j^3+1)^2 }
		\zeta_j \overline{K_{0\zeta_j}}(\zeta_j)
		+
		\frac{ 8 \rho_j^{5} }{ (\rho_j^3+ 1)( \rho_j^2+1 )^{\frac{5}{2}} }
		\zeta_j^2 \overline{K_{0\zeta_j \zeta_j}}( \zeta_j )
		\bigg] ds
		\\ \notag
		&~ -
		(a-ib) \lambda_j^{-1}
		\bigg(1
		-
		\frac{2}{\rho_j^2+1}
		{\rm{Re}}
		\bigg) \bigg( 	\int_{-T}^t \frac{\dot p_j(s) e^{-i \gamma_j(t) } }{t-s}
		\frac{8  \rho_j^3  }{(\rho_j^2+ 1)^{\frac{3}{2}}(\rho_j^3+1) }
		K_0(\zeta_j)
		ds   \bigg)
		\\ \notag
		&~
		- (a-ib)  \lambda_j^{-1}
		\mathrm{Re}
		\bigg\{
		\int_{-T}^t \frac{\dot p_j(s) e^{-i\gamma_j(t)} }{t-s}
		\bigg[
		\frac{8\rho_j^3(\rho_j^3+3\rho_j^2+4)}{(\rho_j^2+1)^{\frac{5}{2}} (\rho_j^3+1)^2 }
		K_0(\zeta_j)
		+
		\frac{16 \rho_j^{5}}{(\rho_j^3+ 1) (\rho_j^2+1)^{\frac{5}{2}} }\zeta_j K_{0\zeta_j}(\zeta_j)
		\bigg]
		ds  \bigg\}
		\\
		&~
		+
		p_j^{-1} \dot{p}_j
		\frac{2\rho_j[ \rho_j^{2} - \rho_j  - (\rho_j^2+1)^{\frac{1}{2}} ] }{ [ \rho_j + (\rho_j^2+1)^{\frac{1}{2}} ] (\rho_j^3 + 1)(\rho_j^2+1)}
		-
		\la_j^{-1}\dot\la_j
		\frac{ 4\rho_j^{3} [  \rho_j^2  + \rho_j(\rho_j^2+1)^{\frac{1}{2}} + 1 ]  }{ [ \rho_j + (\rho_j^2+1)^{\frac{1}{2}} ] (\rho_j^3 + 1)(\rho_j^2+1)^2 }.\label{M0-def-mu3}
	\end{align}
\end{small}
 By \eqref{K-est}, we have
 \begin{small}
\begin{align}
%\begin{equation}
%	\begin{aligned}
		&
		M_0^{\J}
		=
		\lambda_j^{-1}
		\bigg( 1-\frac{2}{\rho_j^2+1}\mathrm{Re} \bigg)
		\bigg\{
		(a-ib)\bigg[
		\int_{-T}^t \frac{\dot p_j(s) e^{-i\gamma_j(t)} }{t-s}
		\bigg\{
		\frac{\rho_j (3\rho_j^7+ \rho_j^6 +12\rho_j^5 -15\rho_j^4+11\rho_j^3-24\rho_j^2-8)}{(\rho_j^2+1)^{\frac{3}{2}} (\rho_j^3+1)^3 }
		\nonumber
		\\
		& \times
		\bigg[ \bigg(\frac{a+ib}{2}+ O\left(\zeta_j\right) \bigg)\1_{\{ \zeta_j\le 1\}}
		+
		O(\zeta_j^{-1}) \1_{\{ \zeta_j > 1\}} \bigg]
	 +
		\frac{4\rho_j^3(\rho_j^3-3\rho_j^2-2)}{(\rho_j^2+1)^{\frac{3}{2}} (\rho_j^3+1)^2 }
		\left( O\left(\zeta_j\right) \1_{\{ \zeta_j\le 1\}}
		+
		O(\zeta_j^{-1} ) \1_{\{ \zeta_j > 1\}} \right)
		\nonumber
		\\
		&
		+
		\frac{ 4 \rho_j^{3} }{ (\rho_j^3+ 1)( \rho_j^2+1 )^{\frac{3}{2}} }
		\left(
		O\left(\zeta_j^2\right) \1_{\{ \zeta_j\le 1\}}
		+
		O(\zeta_j^{-1}) \1_{\{ \zeta_j > 1\}}
		\right)
		\bigg\}
		ds   \bigg] \bigg\}
		\nonumber
		\\
		&
		-
		i b  \lambda_j^{-1}
		\int_{-T}^t \frac{ \overline{\dot{p}_j}(s) e^{i\gamma_j(t)} }{t-s}
		\bigg\{
		-\frac{2\rho_j ( 3\rho_j^7 + \rho_j^6 +12\rho_j^5 -15 \rho_j^4 +11\rho_j^3-24\rho_j^2-8)}{(\rho_j^2+1)^{\frac{5}{2}} (\rho_j^3+1)^3}
		\nonumber
		\\
		&\times
		\left[
		\left(\frac{a-ib}{2}+ O\left(\zeta_j\right) \right)\1_{\{ \zeta_j\le 1\}}
		+
		O(\zeta_j^{-1}) \1_{\{ \zeta_j > 1\}}
		\right]
		+
		\frac{8\rho_j^3(2\rho_j^5 +\rho_j^3+5\rho_j^2+4)}{(\rho_j^2+1)^{\frac{5}{2}} (\rho_j^3+1)^2 }
		\left( O\left(\zeta_j\right) \1_{\{ \zeta_j\le 1\}}
		+
		O(\zeta_j^{-1}) \1_{\{ \zeta_j > 1\}} \right)
		\nonumber
		\\
		& +
		\frac{ 8 \rho_j^{5} }{ (\rho_j^3+ 1)( \rho_j^2+1 )^{\frac{5}{2}} }
		\left(
		O\left(\zeta_j^2\right) \1_{\{ \zeta_j\le 1\}}
		+
		O(\zeta_j^{-1} ) \1_{\{ \zeta_j > 1\}}
		\right)
		\bigg\} ds
		\nonumber
		\\
		& -
		(a-ib) \lambda_j^{-1}
		\bigg(1
		-
		\frac{2}{\rho_j^2+1}
		{\rm{Re}}
		\bigg)
		\bigg\{ 	\int_{-T}^t \frac{\dot p_j(s) e^{-i \gamma_j(t) } }{t-s}
		\frac{8  \rho_j^3  }{(\rho_j^2+ 1)^{\frac{3}{2}}(\rho_j^3+1) }
		\nonumber
		\\
		& \times
		\left[
		\left(\frac{a+ib}{2}+ O\left(\zeta_j\right) \right)\1_{\{ \zeta_j\le 1\}}
		+
		O\left(\zeta_j^{-1} \right) \1_{\{ \zeta_j > 1\}}
		\right]
		ds   \bigg\}
		\nonumber
		\\
		&
		- (a-ib)  \lambda_j^{-1}
		\mathrm{Re}
		\bigg[
		\int_{-T}^t \frac{\dot p_j(s) e^{-i\gamma_j(t)} }{t-s}
		\bigg\{
		\frac{8\rho_j^3(\rho_j^3+3\rho_j^2+4)}{(\rho_j^2+1)^{\frac{5}{2}} (\rho_j^3+1)^2 }
		\left[
		\left(\frac{a+ib}{2}+ O\left(\zeta_j\right) \right)\1_{\{ \zeta_j\le 1\}}
		+
		O(\zeta_j^{-1}) \1_{\{ \zeta_j > 1\}}
		\right]
		\nonumber
		\\
		&
		+
		\frac{16 \rho_j^{5}}{(\rho_j^3+ 1) (\rho_j^2+1)^{\frac{5}{2}} }
		\left( O\left(\zeta_j\right) \1_{\{ \zeta_j\le 1\}}
		+
		O\left(\zeta_j^{-1} \right) \1_{\{ \zeta_j > 1\}} \right)
		\bigg\}
		ds \bigg]
		\nonumber
		\\
		&
		+
		p_j^{-1} \dot{p}_j
		\frac{2\rho_j[ \rho_j^{2} - \rho_j  - (\rho_j^2+1)^{\frac{1}{2}} ] }{ [ \rho_j + (\rho_j^2+1)^{\frac{1}{2}} ] (\rho_j^3 + 1)(\rho_j^2+1)}
		-
		\la_j^{-1}\dot\la_j
		\frac{ 4\rho_j^{3} [  \rho_j^2  + \rho_j(\rho_j^2+1)^{\frac{1}{2}} + 1 ]  }{ [ \rho_j + (\rho_j^2+1)^{\frac{1}{2}} ] (\rho_j^3 + 1)(\rho_j^2+1)^2 }
\nonumber
\\
		= \ &
		\lambda_j^{-1}
		\bigg( 1-\frac{2}{\rho_j^2+1}\mathrm{Re} \bigg)
		\bigg[
		\int_{-T}^t \frac{\dot p_j(s) e^{-i\gamma_j(t)} }{t-s}
		\bigg\{
		\bigg[
		\frac{\rho_j (3\rho_j^7+ \rho_j^6 +12\rho_j^5 -15\rho_j^4+11\rho_j^3-24\rho_j^2-8)}{2(\rho_j^2+1)^{\frac{3}{2}} (\rho_j^3+1)^3 }
		+ O\left(\zeta_j \langle \rho_j\rangle^{-3} \right) \bigg]
		\1_{\{ \zeta_j\le 1\}}
		\nonumber
		\\
		&
		+
		O\left(\zeta_j^{-1} \langle \rho_j\rangle^{-3} \right) \1_{\{ \zeta_j > 1\}}
		\bigg\}
		ds   \bigg]
		+
		b(ia+b) \lambda_j^{-1}
		\int_{-T}^t \frac{ \overline{\dot{p}_j}(s) e^{i\gamma_j(t)} }{t-s}
		\bigg\{
		\nonumber
		\\
		&
		\bigg[
		\frac{\rho_j ( 3\rho_j^7 + \rho_j^6 +12\rho_j^5 -15 \rho_j^4 +11\rho_j^3-24\rho_j^2-8)}{(\rho_j^2+1)^{\frac{5}{2}} (\rho_j^3+1)^3}
		+ O\left(\zeta_j \langle \rho_j\rangle^{-3} \right) \bigg]
		\1_{\{ \zeta_j\le 1\}}
		+
		O\left(\zeta_j^{-1}
		\langle \rho_j\rangle^{-3} \right) \1_{\{ \zeta_j > 1\}}
		\bigg\} ds
		\nonumber
		\\
		& -
		(a-ib) \lambda_j^{-1}
		\bigg(1
		-
		\frac{2}{\rho_j^2+1}
		{\rm{Re}}
		\bigg)
		\bigg[	\int_{-T}^t \frac{\dot p_j(s) e^{-i \gamma_j(t) } }{t-s}
		\nonumber
		\\
		&\times
		\bigg\{
		\bigg[ \frac{4  (a+ib) \rho_j^3  }{(\rho_j^2+ 1)^{\frac{3}{2}}(\rho_j^3+1) } + O\left(\zeta_j \langle \rho_j\rangle^{-3}\right) \bigg]\1_{\{ \zeta_j\le 1\}}
		+
		O\left(\zeta_j^{-1} \langle \rho_j\rangle^{-3} \right) \1_{\{ \zeta_j > 1\}}
		\bigg\}
		ds   \bigg]
		\nonumber
		\\
		&
		- (a-ib)  \lambda_j^{-1}
		\mathrm{Re}
		\bigg[
		\int_{-T}^t \frac{\dot p_j(s) e^{-i\gamma_j(t)} }{t-s}
		\bigg\{
		\bigg[ (a+ib)
		\frac{4\rho_j^3(\rho_j^3+3\rho_j^2+4)}{(\rho_j^2+1)^{\frac{5}{2}} (\rho_j^3+1)^2 }
		+ O\left(\zeta_j \langle \rho_j\rangle^{-3} \right) \bigg]
		\1_{\{ \zeta_j\le 1\}}
		+
		O(\zeta_j^{-1} \langle \rho_j\rangle^{-3} ) \1_{\{ \zeta_j > 1\}}
		\bigg\}
		ds \bigg]
		\nonumber
		\\
		&
		+
		p_j^{-1} \dot{p}_j
		\frac{2\rho_j[ \rho_j^{2} - \rho_j  - (\rho_j^2+1)^{\frac{1}{2}} ] }{ [ \rho_j + (\rho_j^2+1)^{\frac{1}{2}} ] (\rho_j^3 + 1)(\rho_j^2+1)}
		-
		\la_j^{-1}\dot\la_j
		\frac{ 4\rho_j^{3} [  \rho_j^2  + \rho_j(\rho_j^2+1)^{\frac{1}{2}} + 1 ]  }{ [ \rho_j + (\rho_j^2+1)^{\frac{1}{2}} ] (\rho_j^3 + 1)(\rho_j^2+1)^2 }.
		\label{M0-ortho-form}
%	\end{aligned}
%\end{equation}
\end{align}
 \end{small}
Then  by \eqref{lam-ansatz}, \eqref{K-est}, and \eqref{nonlocal-est}, we obtain
\begin{align}
%\begin{equation}
%	\begin{aligned}
		&
		|M_0^{\J} |
		\lesssim
		\lambda_*^{-1} \langle \rho_j\rangle^{-3}
		\int_{-T}^t \frac{|\dot{\lambda}_*(s)| }{t-s}
		\left(
		\1_{\{ \zeta_j\le 1\}}
		+
		\zeta_j^{-1}
		\1_{\{ \zeta_j > 1\}}
		\right)
		ds
		+
		\lambda_*^{-1} |\dot{\lambda}_*| \langle \rho_j\rangle^{-3}
		\lesssim
		\lambda_*^{-1} \langle \rho_j\rangle^{-3},
		\nonumber
		\\
		&
		| \tilde{M}_0^{\J} | \lesssim
		|\dot{\lambda}_*|
		\langle \rho_j\rangle^{-1}
		\int_{-T}^t \frac{ |\dot{\lambda}_*(s)|  }{t-s}
		\left(
		\1_{\{ \zeta_j\le 1\}}
		+
		\zeta_j^{-1} \1_{\{ \zeta_j > 1\}}
		\right)
		ds
		\lesssim
		|\dot{\lambda}_*|
		\langle \rho_j\rangle^{-1},
		\quad
		|M_1^{\J}|
		\lesssim  \la_*^{-1} |\dot{\xi}^{\J}| \langle \rho_j\rangle^{-2},
			\nonumber
		\\
		&
		| \tilde{M}_1^{\J} |
		\lesssim
		|\dot{\xi}^{\J}|
		\int_{-T}^t \frac{ |\dot{\lambda}_*(s)|  }{t-s}
		\left(
		\1_{\{ \zeta_j\le 1\}}
		+
		\zeta_j^{-1} \1_{\{ \zeta_j > 1\}}
		\right)
		ds
		\lesssim
		|\dot{\xi}^{\J}|,
		\quad
		| M_{-1}^{\J} |
		\lesssim
		|\dot{\xi}^{\J}|.  \label{M-est}
%	\end{aligned}
%\end{equation}
\end{align}
By \eqref{M-est}, we have
\begin{equation}\label{Sj-U-tangent-est}
	\big| \big(
	\Pi_{U^{\J \perp} } \mathcal{S}^{\J} \big)_{\mathcal{C}_j } \big|
	\lesssim
	\lambda_*^{-1} \langle \rho_j\rangle^{-3}
	+
	|\dot{\lambda}_*|
	\langle \rho_j\rangle^{-1}
	+
	|\dot{\xi}^{\J}|
	\big( \lambda_*^{-1}  \langle \rho_j \rangle^{-2}
	+
	1 \big).
\end{equation}
Integrating \eqref{Sj-U-dire-est} and \eqref{Sj-U-tangent-est}, we have
\begin{equation}\label{Sj-est}
	|\mathcal{S}^{\J} | \lesssim
	\lambda_*^{-1} \langle \rho_j\rangle^{-2}
	+
	|\dot{\lambda}_*|
	\langle \rho_j\rangle^{-1}
	+
	|\dot{\xi}^{\J}|.
\end{equation}

\section{Gluing system}

In this section, we will derive the inner-outer gluing system and present the corresponding topologies with carefully designed weights such that solutions with desired asymptotics can be found.

\subsection{Error analysis}\label{system-newerror-sec}

We look for the solution $u$ of the form
\begin{equation}\label{u-def}
	\begin{aligned}
	&	u= (1+A)U_*+\Phi-(\Phi\cdot U_*)U_*,\\
	&	\Phi(x,t) := \sum_{j=1}^{N}\left( \eta_R^{\J}(x,t) Q_{\gamma_j}\Phi_{\rm in}^{\J}(y^{\J},t)
		+
		\eta_{d_q}^{\J}(x,t) \Phi^{*{\J}}_0(|x-\xi^{\J}(t)|,t)\right)+\Phi_{\rm out}(x,t) ,
\\
& \Phi_{\rm in}^{\J}(y^{\J},t) \cdot W^{\J}=0
\mbox{ \ for all \ }  t\in(0,T), \quad j=1,2,\dots,N,
	\end{aligned}
\end{equation}
where
\begin{equation}\label{qd24Apr12-3}  \eta_R^{\J}(x,t)=\eta\Big(\frac{x-\xi^{\J}(t)}{\la_*(t)R(t)}\Big),
		\quad
		\eta_{d_q}^{\J}(x,t)=\eta\Big( \frac{x-\xi^{\J}(t) }{d_q} \Big),
\end{equation}
$\eta$ is a smooth cut-off function satisfying $0\le \eta(x) \le 1$,
$\eta(x) = 1$ if $|x|\le 1$ and $\eta(x) = 0$ if $|x|\ge 2$; $A$ is a real-valued function to to make $|u|=1$; $\Phi_{\rm in}^{\J}$ and $\Phi_{\rm out}$ will be solved in the inner-outer gluing system, where  $\Phi_{\rm in}^{\J}$ solves the inner problem near each bubble $U^{\J}$, while $\Phi_{\rm out}$ handles the region away from the concentration zones; $\Phi^{*{\J}}_0$ is defined in \eqref{def-globalcorrection-J}.
Throughout this paper, we make the ansatz
\begin{equation}\label{RPhi-ansatz}
		R(t) = \lambda_*^{-\beta}(t), \quad
		|\Phi|\ll 1,
\end{equation}
where $0<\beta<1$ will be chosen later.
 Notice that
\begin{equation*}
	\eta_{d_q}^{\J}\equiv 1 \mbox{ \ in \ } |x-\xi^{\J}(t)|\le 2\lambda_*(t) R(t).
\end{equation*}
The scalar function $A$ will be chosen in \eqref{u-def} to make $|u|=1$. Indeed,
\begin{align*}
%\begin{equation*}
%	\begin{aligned}
		&
		|u|^2=1 \Leftrightarrow
		(1+A)^2|U_*|^2 + 2 (1+A)
		(\Phi\cdot U_*)(1-|U_*|^2 )
		+ |\Phi-(\Phi\cdot U_*)U_*|^2=1
		\\
		&	\Leftrightarrow
		\bigg[ 1+A+
		\frac{ (\Phi\cdot U_*)(1-|U_*|^2)  }{|U_*|^2} \bigg]^2
		=\frac{1 -|\Phi-(\Phi\cdot U_*)U_*|^2 }{|U_*|^2}
		+
		\bigg[\frac{ (\Phi\cdot U_*)(1-|U_*|^2) }{|U_*|^2} \bigg]^2.
%	\end{aligned}
%\end{equation*}
\end{align*}
We take
\begin{equation}\label{A-def}
	A
	=
	\bigg\{
	1+
	\frac{1-|U_*|^2 -|\Phi-(\Phi\cdot U_*)U_*|^2 }{|U_*|^2}
	+
	\bigg[\frac{ (\Phi\cdot U_*)(1-|U_*|^2)  }{|U_*|^2} \bigg]^2
	\bigg\}^{1/2}
	-
	1-
	\frac{(\Phi\cdot U_*)(1-|U_*|^2 ) }{|U_*|^2}.
\end{equation}
By \eqref{U*-norm}, \eqref{lam-ansatz}, and \eqref{RPhi-ansatz}, we have
\begin{equation}\label{A-est}
	A	=
	\big(
	1+
	O(\lambda_*+|\Phi|^2)
	+
	O(\lambda_*^2 |\Phi|^2)
	\big)^{1/2}
	-
	1+O(\lambda_*|\Phi|)
	=
	O(\lambda_*+\lambda_*|\Phi|+|\Phi|^2)
	=
	O(\lambda_*+|\Phi|^2) .
\end{equation}

One important insight is that we only need to solve
\begin{equation}\label{U-direc-trick}
S\left[u\right]=\Xi(x,t) U_*
\end{equation}
for some scalar function $\Xi$.
Indeed, since $|u|=1$ is kept for all $t\in (0,T)$,  and as the perturbation, $u-U_*$  is uniformly small,  then
\begin{equation*}
	(U_*\cdot u) \Xi= S[u]\cdot u=-\frac{1}2 \pp_{t} (|u|^2)+\frac{a}2 \Delta |u|^2=0.
\end{equation*}
If $U_*\cdot u\geq \delta_0>0$, then $\Xi \equiv 0$. \eqref{U-direc-trick} provides us the flexibility to adjust the error terms in $U_{*}$ direction and we call this {\it $U_{*}$-operation} mentioned earlier.
We compute
\begin{small}
\begin{align}
%\begin{equation}
%\begin{aligned}
		-\partial_t \Phi
		&=  -\pp_t \Phi_{\rm out}
		+\sum_{j=1}^N
		\Big\{ - \pp_t \big(\eta_{d_q}^{\J} \Phi_0^{*\J} \big) +  \eta_R^{\J}Q_{\gamma_j}\big[ -\pp_t \Phi_{\rm in}^{\J}
		+
		\big(\la_j^{-1}\dot\la_j y^{\J}+\la_j^{-1}\dot\xi^{\J} \big)
		\cdot\nabla_{y^{\J}} \Phi_{\rm in}^{\J}-\dot\gamma_j J\Phi_{\rm in}^{\J}\big]
		-  Q_{\gamma_j}\Phi_{\rm in}^{\J} \pp_t \eta_R^{\J} \Big\},
		\nonumber
\\
\Delta_x \Phi &
	= \Delta_x \Phi_{\rm out}+\sum_{j=1}^N \Delta_x \big(\eta_{d_q}^{\J} \Phi_0^{*\J} \big)
	+\sum_{j=1}^N \eta_R^{\J}Q_{\gamma_j}  \Delta_x \Phi_{\rm in}^{\J}
	+\sum_{j=1}^N Q_{\gamma_j} \big( \Phi_{\rm in}^{\J} \Delta_x \eta_R^{\J}
	+2 \nabla_x \eta_R^{\J}\cdot \nabla_x \Phi_{\rm in}^{\J} \big),
	\label{Phi-Lap-ppt}
%\end{aligned}
%\end{equation}
\end{align}
\end{small}
where we used
$\pp_{t}( Q_{\gamma_j} )  =  \dot{\gamma}_j J Q_{\gamma_j} = \dot{\gamma}_j  Q_{\gamma_j} J $ with $
	J:=\begin{bmatrix}
		0 & -1 & 0 \\
		1 & 0 & 0 \\
		0 & 0 & 0
	\end{bmatrix}$.
Notice that
\begin{align*}
%\begin{equation*}
%	\begin{aligned}
		&
		U_* \Delta_x A +(1+A)\Delta_x U_*+2\nabla_x A \cdot \nabla_x U_*+\Delta_x \left[ \Phi- (\Phi\cdot U_*)U_* \right]
		\\
		&
		+ \left|\nabla_x \left[(1+A)U_*+\Phi-(\Phi\cdot U_*)U_*\right] \right|^2 \left[(1+A)U_*+\Phi-(\Phi\cdot U_*)U_*\right]
		\\
		= \ &
		\Delta_x \left[ \Phi- (\Phi\cdot U_*)U_* \right]
		+ \left|\nabla_x U_*\right|^2 \left[ \Phi-(\Phi\cdot U_*)U_*\right]
	\\
		&
		+ \left|\nabla_x \left[(1+A)U_*\right] \right|^2 \left[ \Phi-(\Phi\cdot U_*)U_*\right] - \left|\nabla_x U_*\right|^2 \left[ \Phi-(\Phi\cdot U_*)U_*\right]
		\\
		&
		+
		\left\{
		2
		\nabla_x \left[(1+A)U_*\right]
		\cdot
		\nabla_x \left[\Phi-(\Phi\cdot U_*)U_*\right]
		+
		\left|\nabla_x \left[ \Phi-(\Phi\cdot U_*)U_*\right] \right|^2
		\right\}
		\left[ \Phi-(\Phi\cdot U_*)U_*\right]
		\\
		&
		+
		2\nabla_x A \cdot \nabla_x U_*
		+
		(1+A) \Delta_x U_*
		\\
		&
		+
		U_* \left[ \Delta_x A
		+ \left|\nabla_x \left[(1+A)U_*+\Phi-(\Phi\cdot U_*)U_*\right] \right|^2 (1+A )
		\right]
		\\
		= \ &
		\Delta_x \Phi- 2\nabla_x (\Phi\cdot U_*)\cdot \nabla_x U_*
		+ \left|\nabla_x U_*\right|^2 \Phi
		\\
		&
		+ \left|\nabla_x \left[(1+A)U_*\right] \right|^2 \left[ \Phi-(\Phi\cdot U_*)U_*\right] - \left|\nabla_x U_*\right|^2 \left[ \Phi-(\Phi\cdot U_*)U_*\right]
		\\
		&
		+
		\left\{
		2
		\nabla_x \left[(1+A)U_*\right]
		\cdot
		\nabla_x \left[\Phi-(\Phi\cdot U_*)U_*\right]
		+
		\left|\nabla_x \left[ \Phi-(\Phi\cdot U_*)U_*\right] \right|^2
		\right\}
		\left[ \Phi-(\Phi\cdot U_*)U_*\right]
		\\
		&
		+
		2\nabla_x A \cdot \nabla_x U_*
		+
		\left[1+A -  (\Phi\cdot U_*) \right] \Delta_x U_*
		\\
		&
		+
		U_* \left\{ \Delta_x A
		+ \left|\nabla_x \left[(1+A)U_*+\Phi-(\Phi\cdot U_*)U_*\right] \right|^2 (1+A )
		-
		\left|\nabla_x U_*\right|^2 (\Phi\cdot U_*)
		-
		\Delta_x (\Phi\cdot U_*)
		\right\},
%	\end{aligned}
%\end{equation*}
\end{align*}
and
\begin{align*}
		&
		\left[(1+A)U_*+\Phi-(\Phi\cdot U_*)U_*\right]\wedge \Delta_x \left[(1+A)U_*+\Phi-(\Phi\cdot U_*)U_*\right]
		\\
		= \ &
		\left[ \Phi-(\Phi\cdot U_*)U_*\right]\wedge \Delta_x \left[(1+A)U_* \right]
		+\left[(1+A)U_* \right]\wedge \Delta_x \left[ \Phi-(\Phi\cdot U_*)U_*\right]
		\\
		&
		+
		\left[ \Phi-(\Phi\cdot U_*)U_*\right]\wedge \Delta_x \left[ \Phi-(\Phi\cdot U_*)U_*\right]
		 +
		(1+A)U_* \wedge \Delta_x \left[(1+A)U_* \right]
		\\
		= \ &
		\left[ \Phi-(\Phi\cdot U_*)U_*\right]\wedge \Delta_x U_*
		+
		U_* \wedge \Delta_x \left[ \Phi-(\Phi\cdot U_*)U_*\right]
		\\
		&
		+
		\left[ \Phi-(\Phi\cdot U_*)U_*\right]\wedge \Delta_x \left( A U_* \right)
		+
		A U_* \wedge \Delta_x \left[ \Phi-(\Phi\cdot U_*)U_*\right]
		\\
		&
		+
		\left[ \Phi-(\Phi\cdot U_*)U_*\right]\wedge \Delta_x \left[ \Phi-(\Phi\cdot U_*)U_*\right]
		 +
		(1+A)U_* \wedge \Delta_x \left[(1+A)U_* \right]
		\\
		= \ &
		\Phi \wedge \Delta_x U_*
		+
		U_* \wedge  \left[ \Delta_x \Phi
		-
		2\nabla_x (\Phi\cdot U_*) \cdot  \nabla_x  U_*
		\right]
		\\
		&
		+
		\left[ \Phi-(\Phi\cdot U_*)U_*\right]\wedge \Delta_x \left( A U_* \right)
		+
		A U_* \wedge \Delta_x \left[ \Phi-(\Phi\cdot U_*)U_*\right]
		\\
		&
		+
		\left[ \Phi-(\Phi\cdot U_*)U_*\right]\wedge \Delta_x \left[ \Phi-(\Phi\cdot U_*)U_*\right]
		\\
		& +
		(1+A)U_* \wedge \Delta_x \left[(1+A)U_* \right]
		-
		2(\Phi\cdot U_*)U_*\wedge \Delta_x U_*.
	\end{align*}
By the above identities, we arrange terms in the error as
\begin{align*}
%\begin{equation}
%	\begin{aligned}
		& S[u]
		= -U_*\partial_t A-(1+A)\partial_t U_*-\partial_t\Phi+(\Phi\cdot U_*)\partial_t U_*+ U_* \partial_t(\Phi\cdot U_*)
		\nonumber
		\\
		&  +a\Big\{
		U_* \Delta_x A +(1+A)\Delta_x U_*+2\nabla_x A\cdot\nabla_x U_*+\Delta_x [ \Phi-(\Phi\cdot U_*)U_*]
		\nonumber
		\\
		&
		+ \left|\nabla_x \left[(1+A)U_*+\Phi-(\Phi\cdot U_*)U_*\right] \right|^2 \left[(1+A)U_*+\Phi-(\Phi\cdot U_*)U_*\right] \Big\}
		\nonumber
		\\
		&
		-b\left[(1+A)U_*+\Phi-(\Phi\cdot U_*)U_*\right]\wedge \Delta_x \left[(1+A)U_*+\Phi-(\Phi\cdot U_*)U_*\right]
		\nonumber
		\\
		= \ &  -(1+A)\partial_t U_*-\partial_t\Phi+(\Phi\cdot U_*)\partial_t U_*+ U_* \left[\partial_t(\Phi\cdot U_*)
		- \partial_t A \right]
		\nonumber
		\\
		&  +a\bigg[
		\Delta_x \Phi- 2\nabla_x (\Phi\cdot U_*)\cdot \nabla_x U_*
		+ \left|\nabla_x U_*\right|^2 \Phi
		\nonumber
		\\
		&
		+ \left|\nabla_x \left[(1+A)U_*\right] \right|^2 \left[ \Phi-(\Phi\cdot U_*)U_*\right] - \left|\nabla_x U_*\right|^2 \left[ \Phi-(\Phi\cdot U_*)U_*\right]
		\nonumber
		\\
		&
		+
		\left\{
		2
		\nabla_x \left[(1+A)U_*\right]
		\cdot
		\nabla_x \left[\Phi-(\Phi\cdot U_*)U_*\right]
		+
		\left|\nabla_x \left[ \Phi-(\Phi\cdot U_*)U_*\right] \right|^2
		\right\}
		\left[ \Phi-(\Phi\cdot U_*)U_*\right]
		\nonumber
		\\
		&
		+
		2\nabla_x A \cdot \nabla_x U_*
		+
		\left[1+A -  (\Phi\cdot U_*) \right] \Delta_x U_*
		\nonumber
		\\
		&
		+
		U_* \left\{ \Delta_x A
		+ \left|\nabla_x \left[(1+A)U_*+\Phi-(\Phi\cdot U_*)U_*\right] \right|^2 (1+A )
		-
		\left|\nabla_x U_*\right|^2 (\Phi\cdot U_*)
		-
		\Delta_x (\Phi\cdot U_*)
		\right\} \bigg]
		\nonumber
		\\
		&
		-b\Big\{
		\Phi \wedge \Delta_x U_*
		+
		U_* \wedge  \left[ \Delta_x \Phi
		-
		2\nabla_x (\Phi\cdot U_*) \cdot \nabla_x  U_*
		\right]
		\nonumber
		\\
		&
		+
		\left[ \Phi-(\Phi\cdot U_*)U_*\right]\wedge \Delta_x \left( A U_* \right)
		+
		A U_* \wedge \Delta_x \left[ \Phi-(\Phi\cdot U_*)U_*\right]
		\nonumber
		\\
		&
		+
		\left[ \Phi-(\Phi\cdot U_*)U_*\right]\wedge \Delta_x \left[ \Phi-(\Phi\cdot U_*)U_*\right]
		\nonumber
		\\
		& +
		(1+A)U_* \wedge \Delta_x \left[(1+A)U_* \right]
		-
		2(\Phi\cdot U_*)U_*\wedge \Delta_x U_* \Big\}
		\nonumber
		\\
		= \ &   -\partial_t\Phi +a\left[
		\Delta_x \Phi- 2\nabla_x (\Phi\cdot U_*) \cdot \nabla_x U_*
		+ \left|\nabla_x U_*\right|^2 \Phi \right]
		-b\left\{
		\Phi \wedge \Delta_x U_*
		+
		U_* \wedge  \left[ \Delta_x \Phi
		-
		2\nabla_x (\Phi\cdot U_*) \cdot \nabla_x  U_*
		\right] \right\}
		\nonumber
		\\
		&  - \partial_t U_*  +
		\left[(\Phi\cdot U_*) - A\right] \partial_t U_* + \mathcal N[\Phi]
		+ \Xi[\Phi] U_*,
%	\end{aligned}
%\end{equation}
\end{align*}
where
\begin{small}
\begin{align}
%\begin{equation}
%	\begin{aligned}
		&
		\mathcal N[\Phi]
		:=
		a\Big[
		\left\{
		\left|\nabla_x \left[(1+A)U_*\right] \right|^2
		- \left|\nabla_x U_*\right|^2
		+
		2
		\nabla_x \left[(1+A)U_*\right]
		\cdot
		\nabla_x \Pi_{U_*^{\perp}} \Phi
		+
		\big|\nabla_x \Pi_{U_*^{\perp}} \Phi  \big|^2
		\right\}
		\Pi_{U_*^{\perp}} \Phi
		\nonumber
		\\
		&
		+
		2\nabla_x A \cdot \nabla_x U_*
		+
		\left[1+A -  (\Phi\cdot U_*) \right] \Delta_x U_*
		\Big]
		-b\Big\{
		\Pi_{U_*^{\perp}} \Phi \wedge \Delta_x \left( A U_* \right)
		+
		A U_* \wedge \Delta_x \Pi_{U_*^{\perp}} \Phi
		+
		\Pi_{U_*^{\perp}} \Phi \wedge \Delta_x \Pi_{U_*^{\perp}} \Phi
		\nonumber
		\\
		& +
		(1+A)U_* \wedge \Delta_x \left[(1+A)U_* \right]
		-
		2(\Phi\cdot U_*)U_*\wedge \Delta_x U_* \Big\},
		\label{def-N}
%	\end{aligned}
%\end{equation}
\end{align}
\end{small}
\begin{small}
\begin{equation*}
	\Xi[\Phi]:=
	\partial_t\left[(\Phi\cdot U_*)
	-  A \right]
	+
	a \left\{ \Delta_x A
	+ \left|\nabla_x \left[(1+A)U_*+\Phi-(\Phi\cdot U_*)U_*\right] \right|^2 (1+A )
	-
	\left|\nabla_x U_*\right|^2 (\Phi\cdot U_*)
	-
	\Delta_x (\Phi\cdot U_*)
	\right\}.
\end{equation*}
\end{small}
Then
\begin{align*}
%\begin{equation*}
%	\begin{aligned}
		 S[u] = &  -\partial_t\Phi +
		\left(a -bU_* \wedge \right) \left[
		\Delta_x \Phi- 2\nabla_x (\Phi\cdot U_*) \cdot \nabla_x U_*
		\right]
		+
		a \Phi \sum\limits_{j=1}^N |\nabla_x U^{\J} |^2
		+b \Phi \wedge \sum\limits_{j=1}^N |\nabla_x U^{\J} |^2  U^{\J}
		\\
		&  - \partial_t U_*
		+
		a \Phi \sum\limits_{j,k=1,j\ne k}^N \nabla_x U^{\J}\cdot \nabla_x U^{\K}
		+ \left[ (\Phi\cdot U_*) -A \right]\partial_t U_* + \mathcal N[\Phi]
		+ \Xi[\Phi] U_*
		\\
		= \ &
		-\partial_t\Phi +
		\left(a -bU_* \wedge \right) \bigg[
		\Delta_x \Phi- 2 \sum\limits_{j=1}^{N}\nabla_x \left(
		\Phi\cdot U^{\J} \right)  \cdot  \nabla_x U^{\J}
		\bigg]
		+
		\sum\limits_{j=1}^N |\nabla_x U^{\J} |^2
		\left( a-bU^{\J}  \wedge \right) \Phi
		\\
		& - \partial_t U_*
		+
		\left(a -bU_* \wedge \right) \bigg\{
		- 2 \sum\limits_{j=1}^{N} \nabla_x \left[
		\Phi\cdot \left(  U_*- U^{\J}  \right)
		\right] \cdot \nabla_x U^{\J}
		\bigg\}
		\\
		&
		+
		a \Phi \sum\limits_{j,k=1,j\ne k}^N \nabla_x U^{\J}\cdot \nabla_x U^{\K}
		+ \left[ (\Phi\cdot U_*)-A\right]\partial_t U_* + \mathcal N[\Phi]
		+ \Xi[\Phi] U_*.
%	\end{aligned}
%\end{equation*}
\end{align*}
Using \eqref{u-def} and \eqref{Phi-Lap-ppt}, we have
\begin{small}
\begin{align*}
%\begin{equation*}
%	\begin{aligned}
	S[u] = &
		-\pp_t \Phi_{\rm out}
		-
		\sum_{j=1}^N \pp_t (\eta_{d_q}^{\J} \Phi_0^{*\J} )
		+\sum_{j=1}^N \eta_R^{\J}Q_{\gamma_j}\left[ -\pp_t \Phi_{\rm in}^{\J}
		+
		\left(\la_j^{-1}\dot\la_j y^{\J}+\la_j^{-1}\dot\xi^{\J} \right)
		\cdot\nabla_{y^{\J}} \Phi_{\rm in}^{\J}-\dot\gamma_j J\Phi_{\rm in}^{\J}\right]
		-\sum_{j=1}^N   Q_{\gamma_j}\Phi_{\rm in}^{\J} \pp_t \eta_R^{\J}
		\\
		&
		+
		\left(a -bU_* \wedge \right) \bigg\{
		\Delta_x \Phi_{\rm out}+\sum_{j=1}^N \Delta_x (\eta_{d_q}^{\J} \Phi_0^{*\J} )
		+\sum_{j=1}^N \eta_R^{\J}Q_{\gamma_j}  \Delta_x \Phi_{\rm in}^{\J}
		+\sum_{j=1}^N Q_{\gamma_j}\left( \Phi_{\rm in}^{\J} \Delta_x \eta_R^{\J}
		+2 \nabla_x \eta_R^{\J}\cdot \nabla_x \Phi_{\rm in}^{\J} \right)
		\\
		&
		- 2 \sum\limits_{j=1}^{N}\nabla_x \left( U^{\J} \cdot \Phi_{\rm out}  \right) \cdot \nabla_x U^{\J}
		- 2 \sum\limits_{j=1}^{N}\nabla_x \left[
		U^{\J} \cdot
		\left( \eta_R^{\J}Q_{\gamma_j}\Phi_{\rm in}^{\J} + \eta_{d_q}^{\J} \Phi^{*{\J}}_0 \right) \right] \cdot \nabla_x U^{\J}
		\\
		&
		- 2 \sum\limits_{j=1}^{N}\nabla_x \bigg[
		U^{\J} \cdot \sum_{k=1, k\ne j}^{N} \left( \eta_R^{\K}Q_{\gamma_k}\Phi_{\rm in}^{\K} +
		\eta_{d_q}^{\K} \Phi^{*{\K}}_0 \right) \bigg] \cdot  \nabla_x U^{\J}
		\bigg\}
		\\
		&
		+
		\sum\limits_{j=1}^N |\nabla_x U^{\J} |^2
		\left( a-bU^{\J}  \wedge \right)
		\Phi_{\rm {out} }
		+
		\sum\limits_{j=1}^N |\nabla_x U^{\J} |^2
		\left( a-bU^{\J}  \wedge \right)
		\left( \eta_R^{\J}Q_{\gamma_j}\Phi_{\rm in}^{\J}
		+
		\eta_{d_q}^{\J} \Phi^{*{\J}}_0 \right)
		\\
		&
		+
		\sum\limits_{j=1}^N |\nabla_x U^{\J} |^2
		\left( a-bU^{\J}  \wedge \right)
		\sum_{k=1, k\ne j}^N \left( \eta_R^{\K}Q_{\gamma_k}\Phi_{\rm in}^{\K} +
		\eta_{d_q}^{\K} \Phi^{*{\K}}_0 \right)
		\\
		&
		- \partial_t U_*
		+
		\left(a -bU_* \wedge \right) \bigg\{
		- 2 \sum\limits_{j=1}^{N} \nabla_x \left[
		\Phi\cdot \left(  U_*- U^{\J}  \right)
		\right] \cdot  \nabla_x U^{\J}
		\bigg\}
		\\
		&
		+
		a \Phi \sum\limits_{j,k=1,j\ne k}^N \nabla_x U^{\J}\cdot \nabla_x U^{\K}
		+ \left[ (\Phi\cdot U_*) -A\right]\partial_t U_* + \mathcal N[\Phi]
		+ \Xi[\Phi] U_*
		\\
		= \ &
		-\pp_t \Phi_{\rm out}
		+
		\left(a -bU_* \wedge \right) \Delta_x \Phi_{\rm out}
		+
		\sum\limits_{j=1}^N
		\left(1-\eta_R^{\J}\right)
		\left( a-bU^{\J}  \wedge \right)
		\left[	|\nabla_x U^{\J} |^2 \Phi_{\rm {out} }
		- 2 \nabla_x \left( U^{\J} \cdot \Phi_{\rm out}  \right) \cdot \nabla_x U^{\J}
		\right]
		\\
		&
		+
		\sum\limits_{j=1}^N
		\left(1-\eta_R^{\J}\right)
		\Big\{
		-\pp_t (\eta_{d_q}^{\J} \Phi_0^{*\J} )
		+
		\left( a-bU^{\J}  \wedge \right)
		\Big[
		\Delta_x (\eta_{d_q}^{\J}\Phi_0^{*\J} ) +
		|\nabla_x U^{\J} |^2 \eta_{d_q}^{\J} \Phi^{*{\J}}_0
		\\
		&
		- 2 \nabla_x \left(
		U^{\J} \cdot
		\eta_{d_q}^{\J} \Phi^{*{\J}}_0  \right) \cdot \nabla_x U^{\J}
		\Big]
		- \partial_t U^{\J}
		\Big\}
		\\
		&
		+
		\sum_{j=1}^{N}
		\eta_R^{\J}Q_{\gamma_j}
		\bigg\{
		-
		\pp_t \Phi_{\rm in}^{\J}
		+
		\lambda_j^{-2}\left( 	a -b W^{\J}  \wedge \right)
		\Big[
		\Delta_{y^{\J}} \Phi_{\rm in}^{\J}
		+
		|\nabla_{y^{\J}} W^{\J} |^2 \Phi_{\rm in}^{\J}
		- 2 \nabla_{y^{\J}} 	
		\left( 	W^{\J} \cdot   \Phi_{\rm in}^{\J} \right) \cdot \nabla_{y^{\J}} W^{\J}
		\\
		&
		+
		2  \left( 	\nabla_{y^{\J}} W^{\J} \cdot    \nabla_{y^{\J}} \Phi_{\rm in}^{\J} \right) W^{\J}	
		\Big]
		\\
		&
		+
		Q_{-\gamma_j}
		\left\{
		\left( a-bU^{\J}  \wedge \right)
		\left[	|\nabla_x U^{\J} |^2 \Pi_{U^{\J \perp}} \Phi_{\rm {out} }
		- 2 \nabla_x \left( U^{\J} \cdot \Phi_{\rm out}  \right) \cdot \nabla_x U^{\J}
		\right]
		\right\}
		\\
		&
		+
		Q_{-\gamma_j}
		\Pi_{U^{\J \perp}}
		\Big\{
		-\pp_t (\Phi_0^{*\J} )
		+
		\left( a-bU^{\J}  \wedge \right)
		\left[
		\Delta_x \Phi_0^{*\J} +
		|\nabla_x U^{\J} |^2 \Phi^{*{\J}}_0
		- 2 \nabla_x \left(
		U^{\J} \cdot
		\Phi^{*{\J}}_0  \right) \cdot \nabla_x U^{\J}
		\right]
		- \partial_t U^{\J}
		\Big\}
		\bigg\}
		\\
		&
		+\sum_{j=1}^{N} \eta_R^{\J}Q_{\gamma_j}\left[
		\left(\la_j^{-1}\dot\la_j y^{\J}+\la_j^{-1}\dot\xi^{\J} \right)
		\cdot\nabla_{y^{\J}} \Phi_{\rm in}^{\J}-\dot\gamma_j J\Phi_{\rm in}^{\J}\right]\\
		&
		+
		\sum_{j=1}^{N}
		Q_{\gamma_j}
		\left\{ - \Phi_{\rm in}^{\J} \pp_t \eta_R^{\J}
		+
		\left( 	a -b W^{\J}  \wedge \right)
		\left[ \Phi_{\rm in}^{\J} \Delta_x \eta_R^{\J}
		+2 \nabla_x \eta_R^{\J} \cdot \nabla_x \Phi_{\rm in}^{\J}
		-
		\left(W^{\J} \cdot   \Phi_{\rm in}^{\J} \right) \left(
		2  \nabla_x 	\eta_R^{\J} \cdot \nabla_x W^{\J}
		\right)
		\right] \right\}
		\\
		&
		-
		\sum_{j=1}^{N} b\left( U_*-U^{\J}
		\right) \wedge  \bigg\{
		\Delta_x (\eta_{d_q}^{\J}\Phi_0^{*\J} )
		+ \eta_R^{\J}Q_{\gamma_j}  \Delta_x \Phi_{\rm in}^{\J}
		+ Q_{\gamma_j}\left( \Phi_{\rm in}^{\J} \Delta_x \eta_R^{\J}
		+2 \nabla_x \eta_R^{\J} \cdot \nabla_x \Phi_{\rm in}^{\J} \right)
		\\
		&
		- 2 \nabla_x \left( U^{\J} \cdot \Phi_{\rm out}  \right) \cdot \nabla_x U^{\J}
		- 2 \nabla_x \left[
		U^{\J} \cdot
		\left( \eta_R^{\J}Q_{\gamma_j}\Phi_{\rm in}^{\J} + \eta_{d_q}^{\J} \Phi^{*{\J}}_0 \right) \right] \cdot \nabla_x U^{\J}
		\bigg\}
		\\
		&
		+
		\left(a -bU_* \wedge \right) \bigg\{
		- 2 \sum\limits_{j=1}^{N} \nabla_x \left[
		\Phi\cdot \left(  U_*- U^{\J}  \right)
		\right]\cdot \nabla_x U^{\J}
		\bigg\}
		\\
		&
		+
		\left(a -bU_* \wedge \right) \bigg\{
		- 2 \sum\limits_{j=1}^{N}\nabla_x \bigg[
		U^{\J} \cdot \sum_{k=1,k\ne j}^N \left( \eta_R^{\K}Q_{\gamma_k}\Phi_{\rm in}^{\K} + \eta_{d_q}^{\K} \Phi^{*{\K}}_0 \right) \bigg] \cdot \nabla_x U^{\J}
		\bigg\}
		\\
		&
		+
		\sum\limits_{j=1}^{N} |\nabla_x U^{\J} |^2
		\left( a-bU^{\J}  \wedge \right)
		\sum_{k=1,k\ne j}^N \left( \eta_R^{\K}Q_{\gamma_k}\Phi_{\rm in}^{\K} + \eta_{d_q}^{\K} \Phi^{*{\K}}_0 \right)
		\\
		&
		+
		a \Phi \sum\limits_{j,k=1,j\ne k}^N \nabla_x U^{\J}\cdot \nabla_x U^{\K}
		+
		\left[(\Phi\cdot U_*) -A \right] \partial_t U_* + \mathcal N[\Phi]
		+ \Xi[\Phi] U_*
		\\
		&+
		\sum_{j=1}^{N}
		\eta_R^{\J} \left(U^{\J} - U_{*} + U_{*}	 \right)	
		\bigg\{
		-
		2a   \left( 	\nabla_x W^{\J} \cdot    \nabla_x \Phi_{\rm in}^{\J} \right)
		+
		a
		|\nabla_x U^{\J} |^2
		\left(U^{\J} \cdot \Phi_{\rm {out} } \right)
		\\
		& +
		\Big\{
		-\pp_t (\Phi_0^{*\J} )
		+
		\left( a-bU^{\J}  \wedge \right)
		\left[
		\Delta_x \Phi_0^{*\J} +
		|\nabla_x U^{\J} |^2 \Phi^{*{\J}}_0
		- 2 \nabla_x \left(
		U^{\J} \cdot
		\Phi^{*{\J}}_0  \right) \cdot \nabla_x U^{\J}
		\right]
		- \partial_t U^{\J}
		\Big\} \cdot U^{\J}
		\bigg\}.
%	\end{aligned}
%\end{equation*}
\end{align*}
\end{small}

\subsection{Simplification of the nonlinear terms $\mathcal N[\Phi]$}\label{N-simplify-sec}

In this subsection, we will single out the second-order derivatives of $\Phi$ in $\mathcal N[\Phi]$ in \eqref{def-N} and extract terms involving $\Phi$ and its derivatives in $\mathcal N[\Phi]$. The purpose of this step is to obtain a convenient form for the inner-outer gluing system and estimates in the construction.
\begin{equation}\label{Z1Z4}
	\begin{aligned}
		&~\Pi_{U_*^{\perp}} \Phi \wedge \Delta_x \left( A U_* \right)
		+
		A U_* \wedge \Delta_x \Pi_{U_*^{\perp}} \Phi
		+
		\Pi_{U_*^{\perp}} \Phi \wedge \Delta_x \Pi_{U_*^{\perp}} \Phi\\
		=&~ (\Phi\wedge U_*) \Delta_x A +\Pi_{U_*^{\perp}} \Phi \wedge \left[A\Delta_x U_*+2\nabla_x A \cdot \nabla_x U_*\right]\\
		&~+AU_*\wedge \Delta_x \Phi-A U_*\wedge [(\Phi\cdot U_*)\Delta_x U_*+2\nabla_x (\Phi\cdot U_*) \cdot \nabla_x U_*]\\
		&~+\Pi_{U_*^{\perp}} \Phi \wedge \Delta_x \Phi -(\Phi\wedge U_*)\Delta_x (\Phi\cdot U_*)-\Pi_{U_*^{\perp}} \Phi \wedge [(\Phi\cdot U_*)\Delta_x U_*+2\nabla_x (\Phi\cdot U_*)\cdot \nabla_x U_*]\\
		=&~(\Phi\wedge U_*) \Delta_x A +AU_*\wedge \Delta_x \Phi+\Pi_{U_*^{\perp}} \Phi \wedge \Delta_x \Phi -(\Phi\wedge U_*)\Delta_x (\Phi\cdot U_*) \\
		&~-(\Pi_{U_*^{\perp}} \Phi +AU_*)\wedge [2\nabla_x (\Phi\cdot U_*)\cdot\nabla_x U_*]+[A-(\Phi\cdot U_*)]\Phi\wedge \Delta_x U_*\\
		&~+\Pi_{U_*^{\perp}} \Phi \wedge (2\nabla_x A \cdot \nabla_x U_*)+[(\Phi\cdot U_*)^2-2A(\Phi\cdot U_*)] U_*\wedge \Delta_x U_*.
	\end{aligned}
\end{equation}

Next, we give explicit formulas for $\nabla_x A$ and $\Delta_x A$.
Due to the choice of  \eqref{u-def}, $|u|=1$ is equivalent to
\begin{equation}\label{|u|=1}
	(1+A)^2 |U_*|^2 +2(1+A)\big( U_*  \cdot \Pi_{U_*^{\perp}}\Phi\big)+|\Pi_{U_*^{\perp}}\Phi|^2=1.
\end{equation}

Acting $\nabla_x$ on both sides of \eqref{|u|=1}, we get
$$
2(1+A)|U_*|^2\nabla_x A+(1+A)^2\nabla_x (|U_*|^2)+ \nabla_x(|\Pi_{U_*^{\perp}}\Phi|^2)+2(1+A) \nabla_x(  U_* \cdot \Pi_{U_*^{\perp}}\Phi)+2( U_* \cdot  \Pi_{U_*^{\perp}}\Phi)\nabla_x A=0.
$$
So
\begin{equation}\label{nab-A-def}
	\nabla_x A =  -\frac{(1+A)^2\nabla_x (|U_*|^2)+ \nabla_x(|\Pi_{U_*^{\perp}}\Phi|^2)+2(1+A)\nabla_x( U_*  \cdot \Pi_{U_*^{\perp}}\Phi)}{2(1+A)|U_*|^2+2( U_* \cdot
		\Pi_{U_*^{\perp}}\Phi ) }.
\end{equation}

 Acting $\Delta_x$ on both sides of \eqref{|u|=1}, we have
\begin{equation*}
	\begin{aligned}
		&(1+A)^2\Delta_x(|U_*|^2)+|U_*|^2\big[2(1+A)\Delta_x A+2|\nabla_x A|^2\big]+4(1+A)\nabla_x (|U_*|^2)\cdot \nabla_x A
		\\
		&+2(1+A)\Delta_x \big(  U_* \cdot \Pi_{U_*^{\perp}}\Phi\big)+2\big( U_* \cdot  \Pi_{U_*^{\perp}}\Phi \big)\Delta_x A+4\nabla_x \big(U_* \cdot  \Pi_{U_*^{\perp}}\Phi \big)\cdot \nabla_x A+\Delta_x\big(|\Pi_{U_*^{\perp}}\Phi|^2\big)=0.
	\end{aligned}
\end{equation*}
Thus, we have
\begin{equation}\label{Delta-A}
	\begin{aligned}
		& \Delta_x A
		= -2^{-1}\big[(1+A)|U_*|^2+\big( U_*  \cdot \Pi_{U_*^{\perp}}\Phi \big)\big]^{-1} \Big[\Delta_x \big(|\Pi_{U_*^{\perp}}\Phi|^2\big)+2(1+A)\Delta_x \big(  U_* \cdot \Pi_{U_*^{\perp}}\Phi\big)
		\\
		& \  +4\nabla_x \big(  U_* \cdot \Pi_{U_*^{\perp}}\Phi\big)\cdot \nabla_x A +2|U_*|^2|\nabla_x A|^2+4(1+A)\nabla_x (|U_*|^2)\cdot \nabla_x A+(1+A)^2\Delta_x(|U_*|^2)\Big] .
	\end{aligned}
\end{equation}
Notice that
\begin{equation*}
	\begin{aligned}
		\Delta_x(|\Pi_{U_*^{\perp}}\Phi|^2)
		=&~2\Phi\cdot \Delta_x \Phi+2(|U_*|^2-2)\left[(\Phi\cdot U_*)\Delta_x(\Phi\cdot U_*)+|\nabla_x (\Phi\cdot U_*)|^2\right]\\
		&~+2\nabla_x(|U_*|^2)\cdot \nabla_x[(\Phi\cdot U_*)^2]+2|\nabla_x \Phi|^2+(\Phi\cdot U_*)^2 \Delta_x (|U_*|^2)
	\end{aligned}
\end{equation*}
and
\begin{equation*}
		\Delta_x (  U_* \cdot \Pi_{U_*^{\perp}}\Phi)
		= (1-|U_*|^2)\Delta_x (\Phi\cdot U_*)
		-(\Phi\cdot U_*)\Delta_x(|U_*|^2)-2\nabla_x(|U_*|^2)\cdot \nabla_x(\Phi\cdot U_*).
\end{equation*}
Then \eqref{Delta-A} can be rephrased into
\begin{small}
\begin{align}
%\begin{equation}
%	\begin{aligned}
		 \Delta_x A
		= \ & -2^{-1}\left[(1+A)|U_*|^2+\left( U_*  \cdot \Pi_{U_*^{\perp}}\Phi \right)\right]^{-1} \Big[
		2\Phi\cdot \Delta_x \Phi+\left[2(|U_*|^2-2)(\Phi\cdot U_*)+2(1+A)(1-|U_*|^2)\right] \Delta_x (\Phi\cdot U_*)
		\nonumber
		\\
		&+2(|U_*|^2-2)|\nabla_x (\Phi\cdot U_*)|^2+2|\nabla_x \Phi|^2+4[(\Phi\cdot U_*)-(1+A)]\nabla_x (|U_*|^2)\cdot \nabla_x (\Phi\cdot U_*)
		+2|U_*|^2|\nabla_x A|^2
		\nonumber
		\\
		&
		+4\nabla_x (  U_* \cdot \Pi_{U_*^{\perp}}\Phi)\cdot \nabla_x A+4(1+A)\nabla_x (|U_*|^2)\cdot \nabla_x A
		+[(\Phi\cdot U_*)-(1+A)]^2\Delta_x (|U_*|^2)\Big].
		\label{DeltaA-1}
%	\end{aligned}
%\end{equation}
\end{align}
\end{small}
By \eqref{DeltaA-1}, part of the terms in \eqref{Z1Z4} can be rewritten as
\begin{align}
\notag
		&
		(\Phi\wedge U_*) \Delta_x A +AU_*\wedge \Delta_x \Phi+\Pi_{U_*^{\perp}} \Phi \wedge \Delta_x \Phi -(\Phi\wedge U_*)\Delta_x (\Phi\cdot U_*)
		\\ \notag
		= \ &
		-2^{-1} (\Phi\wedge U_*)
		\left[(1+A)|U_*|^2+\left( U_*  \cdot \Pi_{U_*^{\perp}}\Phi \right)\right]^{-1}
		\\ \notag
		&
		\times
		\left\{ 2\Phi\cdot \Delta_x \Phi
		+\left[2(|U_*|^2-2)(\Phi\cdot U_*)+2(1+A)(1-|U_*|^2)\right]
		\left( U_*\cdot\Delta_x \Phi  \right)
		\right\}
		\\ \notag
		&
		+AU_*\wedge \Delta_x \Phi+\Pi_{U_*^{\perp}} \Phi \wedge \Delta_x \Phi -(\Phi\wedge U_*)
		\left( U_*\cdot\Delta_x \Phi \right)
		\\ \notag
		&
		-2^{-1} (\Phi\wedge U_*)
		\left[(1+A)|U_*|^2+\left( U_*  \cdot \Pi_{U_*^{\perp}}\Phi \right)\right]^{-1}
		\\ \notag
		& \times
		\Big\{  \left[2(|U_*|^2-2)(\Phi\cdot U_*)+2(1+A)(1-|U_*|^2)\right]
		\left( 2\nabla_x \Phi\cdot\nabla_x U_*+\Phi\cdot \Delta_x U_*  \right)
		\\ \notag
		& +2(|U_*|^2-2)|\nabla_x (\Phi\cdot U_*)|^2+2|\nabla_x \Phi|^2+4[(\Phi\cdot U_*)-(1+A)]\nabla_x (|U_*|^2)\cdot \nabla_x (\Phi\cdot U_*)\\  \notag
		& +2|U_*|^2|\nabla_x A|^2+4\nabla_x (  U_* \cdot \Pi_{U_*^{\perp}}\Phi)\cdot \nabla_x A+4(1+A)\nabla_x (|U_*|^2)\cdot \nabla_x A
		 +[(\Phi\cdot U_*)-(1+A)]^2\Delta_x (|U_*|^2) \Big\}
		\\ \notag
		&
		-(\Phi\wedge U_*)
		\left(  2\nabla_x \Phi\cdot\nabla_x U_*+\Phi\cdot \Delta_x U_*  \right)
		\\ \notag
		= \ &
		-2^{-1} (\Phi\wedge U_*)
		\left[(1+A)|U_*|^2+\left( U_*  \cdot \Pi_{U_*^{\perp}}\Phi \right)\right]^{-1}
		\left[ 2\Phi\cdot \Delta_x \Phi
		+2 \left( 1+A-\Phi\cdot U_* \right)
		\left( U_*\cdot\Delta_x \Phi  \right)
		\right]
		\\ \notag
		&
		+AU_*\wedge \Delta_x \Phi+\Pi_{U_*^{\perp}} \Phi \wedge \Delta_x \Phi
		\\ \notag
		&
		-2^{-1} (\Phi\wedge U_*)
		\left[(1+A)|U_*|^2+\left( U_*  \cdot \Pi_{U_*^{\perp}}\Phi \right)\right]^{-1}
		\Big\{  2 \left( 1+A-\Phi\cdot U_* \right)
		\left( 2\nabla_x \Phi\cdot\nabla_x U_*+\Phi\cdot \Delta_x U_*  \right)
		\\ \notag
		& +2(|U_*|^2-2)|\nabla_x (\Phi\cdot U_*)|^2+2|\nabla_x \Phi|^2+4[(\Phi\cdot U_*)-(1+A)]\nabla_x (|U_*|^2)\cdot \nabla_x (\Phi\cdot U_*)
		+2|U_*|^2|\nabla_x A|^2
		\\
		&
		+4\nabla_x (  U_* \cdot \Pi_{U_*^{\perp}}\Phi)\cdot \nabla_x A+4(1+A)\nabla_x (|U_*|^2)\cdot \nabla_x A
	 +[(\Phi\cdot U_*)-(1+A)]^2\Delta_x (|U_*|^2) \Big\}. \label{Z1Z5}
	\end{align}

Combining \eqref{def-N}, \eqref{Z1Z4}, and \eqref{Z1Z5}, we get
\begin{align*}
%\begin{equation}
%	\begin{aligned}
		&
		\mathcal N[\Phi]
		=
		b\Big\{
		(\Phi\wedge U_*)
		\left[(1+A)|U_*|^2+\left( U_*  \cdot \Pi_{U_*^{\perp}}\Phi \right)\right]^{-1}
		\left[ \Phi\cdot \Delta_x \Phi
		+ \left( 1+A-\Phi\cdot U_* \right)
		\left( U_*\cdot\Delta_x \Phi  \right)
		\right]
		\nonumber
		\\
		&
		- AU_*\wedge \Delta_x \Phi
		-
		\left( \Pi_{U_*^{\perp}} \Phi \right) \wedge \Delta_x \Phi  \Big\}
		\nonumber
		\\
		&
		+
		a\bigg[
		\left\{
		\left|\nabla_x \left[(1+A)U_*\right] \right|^2
		- \left|\nabla_x U_*\right|^2
		+
		2
		\nabla_x \left[(1+A)U_*\right]
		\cdot
		\nabla_x \left(\Pi_{U_*^{\perp}} \Phi \right)
		+
		\left|\nabla_x \left(\Pi_{U_*^{\perp}} \Phi  \right) \right|^2
		\right\}
		\Pi_{U_*^{\perp}} \Phi
		\nonumber
		\\
		&
		+
		2\nabla_x A \cdot \nabla_x U_*
		+
		\left(1+A -  \Phi\cdot U_* \right) \Delta_x U_*
		\bigg]
		\nonumber
		\\
		&
		-b\bigg[
		-2^{-1} (\Phi\wedge U_*)
		\left[(1+A)|U_*|^2+\left( U_*  \cdot \Pi_{U_*^{\perp}}\Phi \right)\right]^{-1}
		\Big\{  2 \left( 1+A-\Phi\cdot U_* \right)
		\left( 2\nabla_x \Phi\cdot\nabla_x U_*+\Phi\cdot \Delta_x U_*  \right)
		\nonumber
		\\
		& +2(|U_*|^2-2)|\nabla_x (\Phi\cdot U_*)|^2+2|\nabla_x \Phi|^2+4[(\Phi\cdot U_*)-(1+A)]\nabla_x (|U_*|^2)\cdot \nabla_x (\Phi\cdot U_*)
		\nonumber
		\\
		& +2|U_*|^2|\nabla_x A|^2+4\nabla_x (  U_* \cdot \Pi_{U_*^{\perp}}\Phi)\cdot \nabla_x A+4(1+A)\nabla_x (|U_*|^2)\cdot \nabla_x A
		 +[(\Phi\cdot U_*)-(1+A)]^2\Delta_x (|U_*|^2) \Big\}
		 \nonumber
		 \\
		& -(\Pi_{U_*^{\perp}} \Phi +AU_*)\wedge [2\nabla_x (\Phi\cdot U_*) \cdot \nabla_x U_*]+[A-(\Phi\cdot U_*)]\Phi\wedge \Delta_x U_*
		\nonumber
		\\
		& +\Pi_{U_*^{\perp}} \Phi \wedge (2\nabla_x A \cdot \nabla_x U_*)+[(\Phi\cdot U_*)^2-2A(\Phi\cdot U_*)] U_*\wedge \Delta_x U_*
		\nonumber
		\\
		& +
		(1+A)U_* \wedge  \left[(1+A) \Delta_x U_* + 2\nabla_x A \cdot \nabla_x U_* \right]
		-
		2(\Phi\cdot U_*)U_*\wedge \Delta_x U_* \bigg].
%	\end{aligned}
%\end{equation}
\end{align*}
Since
\begin{equation*}
	\begin{aligned}
		&
		2
		\nabla_x \left[(1+A)U_*\right]
		\cdot
		\nabla_x \left( \Pi_{U_*^{\perp}} \Phi \right)
		=
		2 \sum\limits_{k=1}^{2}
		\Big\{
		\left[  \left(\pp_{x_k} A  \right) U_{*}\cdot
		\pp_{x_k} \Phi   + (1+A) \pp_{x_k} U_{*} \cdot
		\pp_{x_k} \Phi  \right]
		\\
		& \quad
		-
		\pp_{x_k}\left( U_{*} \cdot  \Phi \right) \left[ |U_{*}|^2 \pp_{x_k} A + (1+A)
		U_{*} \cdot  \pp_{x_k} U_{*} \right]
		- \left(U_* \cdot \Phi \right) \left[ \left(\pp_{x_k} A \right) U_{*} \cdot   \pp_{x_k} U_*   + (1+A) \left|\pp_{x_k} U_{*}\right|^2 \right]
		\Big\},
	\end{aligned}
\end{equation*}
then $\mathcal N[\Phi] $ can be expanded as
\begin{small}
\begin{align}
%\begin{equation}
%	\begin{aligned}
		&
		\mathcal N[\Phi]
		=
		b\Big\{
		(\Phi\wedge U_*)
		\left[(1+A)|U_*|^2+\left( U_*  \cdot \Pi_{U_*^{\perp}}\Phi \right)\right]^{-1}
		\big[ \Phi\cdot \Delta_x \Phi
		+ \left( 1+A-\Phi\cdot U_* \right)
		\left( U_*\cdot\Delta_x \Phi  \right)
		\big]
		\nonumber
		\\
		&
		- AU_*\wedge \Delta_x \Phi
		-
		\left( \Pi_{U_*^{\perp}} \Phi \right) \wedge \Delta_x \Phi  \Big\}
		+
		a\bigg\{
		\bigg[
		|\nabla_x A|^2 |U_*|^2
		+
		2(1+A)\nabla_x A \cdot \left( U_* \cdot \nabla_x U_*\right)
		+
		A(2+A) \left|\nabla_x U_*\right|^2
		\nonumber
		\\
		&
		+
		2 \sum\limits_{k=1}^{2}
		\Big\{
		\left[  \left(\pp_{x_k} A  \right) U_{*}\cdot
		\pp_{x_k} \Phi   + A \pp_{x_k} U_{*} \cdot
		\pp_{x_k} \Phi  \right]
		-
		\pp_{x_k}\left( U_{*} \cdot  \Phi \right) \left[ |U_{*}|^2 \pp_{x_k} A + (1+A)
		U_{*} \cdot  \pp_{x_k} U_{*} \right]
		\nonumber
		\\
		&
		- \left(U_* \cdot \Phi \right) \left[ \left(\pp_{x_k} A \right) U_{*} \cdot   \pp_{x_k} U_*   + (1+A) \left|\pp_{x_k} U_{*}\right|^2 \right]
		\Big\}
		+
		\sum\limits_{k=1}^2
		\left| \pp_{x_k} \Phi
		-
		U_* \pp_{x_k}\left(
		\Phi\cdot  U_* \right)
		-
		(\Phi\cdot U_*) \pp_{x_k} U_*
		\right|^2
		\bigg]
		\Pi_{U_*^{\perp}} \Phi
		\nonumber
		\\
		&
		+
		2 \left(\nabla_x A + U_* \cdot \nabla_x U_* + \Phi\cdot \nabla_x \Phi \right) \cdot \nabla_x U_*
		+ \Delta_x U_* -2 \left(U_* \cdot \nabla_x U_*  \right)\cdot \nabla_x U_*
		+
		\left(A -  \Phi\cdot U_* \right) \Delta_x U_*
		\bigg\}
		\nonumber
		\\
		& +
		2a\left[ \left(\nabla_x U_* \cdot \nabla_x \Phi \right)  \Phi -  \left(  \Phi\cdot \nabla_x \Phi \right)
		\cdot \nabla_x U_*  \right]
		-
		2a  \left(\nabla_x U_* \cdot \nabla_x \Phi \right)
		(U_*\cdot \Phi ) U_*
		\nonumber
		\\
		&
		-
		2b U_* \wedge
		\left[
		\left( \nabla_x U_* \cdot \nabla_x \Phi \right) \Phi - \left(   \Phi\cdot \nabla_x \Phi \right) \cdot \nabla_x U_* \right]
		\nonumber
		\\
		&
		+
		b   (\Phi\wedge U_*)
		\left[(1+A)|U_*|^2+\left( U_*  \cdot \Pi_{U_*^{\perp}}\Phi \right)\right]^{-1}
		\left( 1+A-\Phi\cdot U_* \right)
		\left( 2\nabla_x \Phi\cdot\nabla_x U_*  \right)
		-
		b   (\Phi\wedge U_*) \left( 2\nabla_x \Phi\cdot\nabla_x U_*  \right)
		\nonumber
		\\
		&
		-b\bigg[
		-2^{-1} (\Phi\wedge U_*)
		\left[(1+A)|U_*|^2+\left( U_*  \cdot \Pi_{U_*^{\perp}}\Phi \right)\right]^{-1}
		\bigg\{  2 \left( 1+A-\Phi\cdot U_* \right)
		\left(  \Phi\cdot \Delta_x U_*  \right)
		\nonumber
		\\
		& +2(|U_*|^2-2)| \nabla_x \left(\Phi\cdot  U_* \right)  |^2+2|\nabla_x \Phi|^2
		+8[(\Phi\cdot U_*)-(1+A)] (U_* \cdot \nabla_x U_*)\cdot  \nabla_x\left(\Phi\cdot  U_* \right)
		\nonumber
		\\
		& +2|U_*|^2|\nabla_x A|^2+4
		\left[
		-2(\Phi\cdot U_*) U_*\cdot \nabla_x U_*
		+
		(1-|U_*|^2)
		\nabla_x
		\left(
		\Phi\cdot  U_*
		\right)
		\right]
		\cdot \nabla_x A
		\nonumber
		\\
		&
		+8(1+A)\left(U_* \cdot  \nabla_x  U_*\right)\cdot \nabla_x A
		+ 2\left[(\Phi\cdot U_*)-(1+A)\right]^2  \left(|\nabla_x U_*|^2 + U_* \cdot \Delta_x U_*\right) \bigg\}
		\nonumber
		\\
		& -(\Pi_{U_*^{\perp}} \Phi +AU_*)\wedge \left[2 \nabla_x \left(\Phi\cdot  U_* \right)\cdot  \nabla_x U_* \right]+[A-(\Phi\cdot U_*)]\Phi\wedge \Delta_x U_*
		\nonumber
		\\
		& +\Pi_{U_*^{\perp}} \Phi \wedge (2\nabla_x A \cdot \nabla_x U_*)+
		\left[(\Phi\cdot U_*)^2-2A(\Phi\cdot U_*) -
		2(\Phi\cdot U_*)\right] U_*\wedge \Delta_x U_*
		\nonumber
		\\
		& +
		(1+A)U_* \wedge  \left[ A \Delta_x U_* + 2\left( \nabla_x A + U_* \cdot \nabla_x U_* + \Phi\cdot \nabla_x \Phi \right)\cdot \nabla_x U_* + \Delta_x U_* -2\left( U_* \cdot \nabla_x U_*  \right) \cdot\nabla_x U_* \right] \bigg]
		\nonumber
		\\
		& +2b
		A U_* \wedge
		\left[
		\left(   \Phi\cdot \nabla_x \Phi \right) \cdot \nabla_x U_* \right],
		\label{N-simplify-form}
%	\end{aligned}
%\end{equation}
\end{align}
\end{small}
where $ U_* \cdot \nabla_x U_*$, $\Phi\cdot \nabla_x \Phi$ are defined in \eqref{dot2}.

\subsection{Inner-outer gluing system}\label{gluing-sys-sec}

By $U_*$-operation  \eqref{U-direc-trick}, we can adjust the terms in the $U_*$ direction flexibly. By the expansion form of $S[u]$ at the end of Subsection \ref{system-newerror-sec} and  \eqref{N-simplify-form}, a sufficient condition for $S[u]=0$ is that $(\Phi_{\rm in}^{\J},\Phi_{\rm out})$ solve the following inner-outer gluing system
\begin{equation}\label{outer-eq}
	\begin{cases}
		\pp_t \Phi_{\rm {out}} =
		\mathbf{B}_{\Phi,U_{*}} \Delta_x \Phi_{\rm out}  + \mathcal{G}
		&\mbox{ in }~\R^2\times (0,T) ,
		\\
		\Phi_{\rm out}(x,0)=  Z_{*}(x) + \sum\limits_{m=1}^{N} \sum\limits_{n=1}^3 c_{mn}  \vartheta_{mn}(x)
		& \mbox{ in }~\R^2;
	\end{cases}
\end{equation}
\begin{equation}\label{inner-eq}
	\begin{aligned}
		\lambda_j^{2} \pp_t \Phi_{\rm in}^{\J}  = \ &
		\left( 	a -b W^{\J}  \wedge \right)
		\Big[
		\Delta_{y^{\J}} \Phi_{\rm in}^{\J}
		+
		|\nabla_{y^{\J}} W^{\J} |^2 \Phi_{\rm in}^{\J}
		- 2 \nabla_{y^{\J}} 	
		\left( 	W^{\J} \cdot   \Phi_{\rm in}^{\J} \right) \cdot \nabla_{y^{\J}} W^{\J}
		\\
		&
		+
		2  \left( 	\nabla_{y^{\J}} W^{\J} \cdot    \nabla_{y^{\J}} \Phi_{\rm in}^{\J} \right) W^{\J}	
		\Big]  + \mathcal{H}^{\J}
		\quad
		\mbox{ in }~\mathsf{D}_{2 C_{\lambda} R},
	\end{aligned}
\end{equation}
where $C_{\lambda}$ is given in \eqref{lam-ansatz},
\begin{equation}\label{Hj-def}
\mathsf{D}_{2 C_{\lambda} R}:=\left\{(y ,t) \ | \  |y |< 2 C_{\lambda} R(t), \  t\in (0,T)  \right\} ;
\quad
	\mathcal{H}^{\J} :=   \mathcal{H}_1^{\J} + \mathcal{H}_{ \rm{in} }^{\J},
\end{equation}
\begin{equation}\label{Hj-1-def}
\begin{aligned}
	 \mathcal{H}_1^{\J} := \ &
		\lambda_j^{2} Q_{-\gamma_j}
	\Big[
		\left( a-bU^{\J}  \wedge \right)
\tilde{L}_{U^{\J} }[\Phi_{\rm{out}}]
		+
		\left(  M_0^{\J} + e^{i\theta_j} M_{1}^{\J} \right)_{\mathcal{C}_j^{-1}}
	\Big]
\\
= \ &
\left( a-bW^{\J}  \wedge \right)
\tilde{L}_{W^{\J} }[Q_{-\gamma_j}\Phi_{\rm{out}}]
+
\lambda_j^{2}
\left(  M_0^{\J} + e^{i\theta_j} M_{1}^{\J} \right)_{\mathbb{C}_j^{-1}},
\end{aligned}
\end{equation}
\begin{small}
\begin{align}
\notag
		\mathcal{H}_{ \rm{in} }^{\J}
		:= \ &
		\lambda_j^{2} Q_{-\gamma_j}
		\bigg[
		2\left(
		a -
		b U^{\J} \wedge \right)
		\Big\{ \left[\nabla_x U^{\J} \cdot \nabla_x \left( \eta_R^{\J}  Q_{\gamma_j}\Phi_{\rm in}^{\J}  \right) \right]  \left( Q_{\gamma_j}\Phi_{\rm in}^{\J}  \right)
		-  \left[  \left( Q_{\gamma_j}\Phi_{\rm in}^{\J}  \right) \cdot \nabla_x \left( \eta_R^{\J}  Q_{\gamma_j}\Phi_{\rm in}^{\J}  \right) \right] \cdot \nabla_x U^{\J}  \Big\}
		\bigg]
		\\ \label{def-HPhi2}
		= \ &
		2\left(
		a -
		b W^{\J} \wedge \right)
		\left\{ \left[\nabla_{y^{\J}} W^{\J} \cdot \nabla_{y^{\J}} \left( \eta_R^{\J} \Phi_{\rm in}^{\J}  \right) \right]    \Phi_{\rm in}^{\J}
		-  \left[  \Phi_{\rm in}^{\J}  \cdot \nabla_{y^{\J}} \left( \eta_R^{\J} \Phi_{\rm in}^{\J}  \right) \right]\cdot  \nabla_{y^{\J}} W^{\J}  \right\};
	\end{align}
\end{small}
\begin{small}
\begin{align}
%\begin{equation}
%	\begin{aligned}
	&	\mathcal{G} : =
		\sum\limits_{j=1}^N
		\left(1-\eta_R^{\J}\right)
		\left( a-bU^{\J}  \wedge \right)
		\left[	|\nabla_x U^{\J} |^2 \Phi_{\rm {out} }
		- 2 \nabla_x \left( U^{\J} \cdot \Phi_{\rm out}  \right)\cdot \nabla_x U^{\J}
		\right]
		\nonumber
		\\
		&
		+
		\sum\limits_{j=1}^N
		\left(1-\eta_R^{\J}\right)
		\Big\{
		-\pp_t (\eta_{d_q}^{\J} \Phi_0^{*\J} )
		+
		\left( a-bU^{\J}  \wedge \right)
		\Big[
		\Delta_x (\eta_{d_q}^{\J}\Phi_0^{*\J} ) +
		|\nabla_x U^{\J} |^2 \eta_{d_q}^{\J} \Phi^{*{\J}}_0
		- 2 \nabla_x \left(
		U^{\J} \cdot
		\eta_{d_q}^{\J} \Phi^{*{\J}}_0  \right)\cdot \nabla_x U^{\J}
		\Big]
		\nonumber
		\\
		& - \partial_t U^{\J}
		\Big\}
		+ \sum\limits_{j=1}^N\eta_R^{\J}
		\left(
		\tilde{M}_0^{\J}
		+
		e^{i\theta_j} \tilde{M}_1^{\J} + e^{-i\theta_j} M_{-1}^{\J} \right)_{\mathcal{C}_j^{-1}}
		+\sum_{j=1}^{N} \eta_R^{\J}Q_{\gamma_j}\left[
		\left(\la_j^{-1}\dot\la_j y^{\J}+\la_j^{-1}\dot\xi^{\J} \right)
		\cdot\nabla_{y^{\J}} \Phi_{\rm in}^{\J}-\dot\gamma_j J\Phi_{\rm in}^{\J}\right]
		\nonumber
		\\
		&
		+
		\sum_{j=1}^{N}
		Q_{\gamma_j}
		\left\{ - \Phi_{\rm in}^{\J} \pp_t \eta_R^{\J}
		+
		\left( 	a -b W^{\J}  \wedge \right)
		\left[ \Phi_{\rm in}^{\J} \Delta_x \eta_R^{\J}
		+2 \nabla_x \eta_R^{\J} \cdot \nabla_x \Phi_{\rm in}^{\J}
		-
		\left(W^{\J} \cdot   \Phi_{\rm in}^{\J} \right) \left(
		2  \nabla_x 	\eta_R^{\J} \cdot \nabla_x W^{\J}
		\right)
		\right] \right\}
		\nonumber
		\\
		&
		-
		\sum_{j=1}^{N} b\left( U_*-U^{\J}
		\right) \wedge  \bigg\{
		\Delta_x (\eta_{d_q}^{\J}\Phi_0^{*\J} )
		+ \eta_R^{\J}Q_{\gamma_j}  \Delta_x \Phi_{\rm in}^{\J}
		+ Q_{\gamma_j}\left( \Phi_{\rm in}^{\J} \Delta_x \eta_R^{\J}
		+2 \nabla_x \eta_R^{\J} \cdot \nabla_x \Phi_{\rm in}^{\J} \right)
		\nonumber
		\\
		&
		- 2 \nabla_x \left( U^{\J} \cdot \Phi_{\rm out}  \right)\cdot \nabla_x U^{\J}
		- 2 \nabla_x \left[
		U^{\J} \cdot
		\left( \eta_R^{\J}Q_{\gamma_j}\Phi_{\rm in}^{\J} + \eta_{d_q}^{\J} \Phi^{*{\J}}_0 \right) \right] \cdot \nabla_x U^{\J}
		\bigg\}
		\nonumber
		\\
		&
		+
		\left(a -bU_* \wedge \right) \bigg\{
		- 2 \sum\limits_{j=1}^{N} \nabla_x \left[
		\Phi\cdot \left(  U_*- U^{\J}  \right)
		\right] \cdot \nabla_x U^{\J}
		\bigg\}
		\nonumber
		\\
		&
		+
		\left(a -bU_* \wedge \right) \bigg\{
		- 2 \sum\limits_{j=1}^{N}\nabla_x \bigg[
		U^{\J} \cdot \sum_{k=1,k\ne j}^N \left( \eta_R^{\K}Q_{\gamma_k}\Phi_{\rm in}^{\K} + \eta_{d_q}^{\K} \Phi^{*{\K}}_0 \right) \bigg] \cdot \nabla_x U^{\J}
		\bigg\}
		\nonumber
		\\
		&
		+
		\sum\limits_{j=1}^{N} |\nabla_x U^{\J} |^2
		\left( a-bU^{\J}  \wedge \right)
		\sum_{k=1,k\ne j}^N \left( \eta_R^{\K}Q_{\gamma_k}\Phi_{\rm in}^{\K} + \eta_{d_q}^{\K} \Phi^{*{\K}}_0 \right)
		+
		a \Phi \sum\limits_{j,k=1,j\ne k}^N \nabla_x U^{\J}\cdot \nabla_x U^{\K}
		+\left[ (\Phi\cdot U_*)-A \right]\partial_t U_*
		\nonumber
		\\
		&+
		\sum_{j=1}^{N}
		\eta_R^{\J} \left(U^{\J} - U_{*}  \right)	
		\Big[
		-
		2a   \left( 	\nabla_x W^{\J} \cdot    \nabla_x \Phi_{\rm in}^{\J} \right)
		+
		a
		|\nabla_x U^{\J} |^2
		\left(U^{\J} \cdot \Phi_{\rm {out} } \right)
		\nonumber
		\\
		& +
		\left\{
		-\pp_t (\Phi_0^{*\J} )
		+
		\left( a-bU^{\J}  \wedge \right)
		\left[
		\Delta_x \Phi_0^{*\J} +
		|\nabla_x U^{\J} |^2 \Phi^{*{\J}}_0
		- 2 \nabla_x \left(
		U^{\J} \cdot
		\Phi^{*{\J}}_0  \right) \cdot \nabla_x U^{\J}
		\right]
		- \partial_t U^{\J}
		\right\} \cdot U^{\J}
		\Big]
		\nonumber
		\\
%	\end{aligned}
%\end{equation}
%\begin{equation}
%	\begin{aligned}
		& +
		b\Big\{
		(\Phi\wedge U_*)
		\left[(1+A)|U_*|^2+\left( U_*  \cdot \Pi_{U_*^{\perp}}\Phi \right)\right]^{-1}
		\left[ \Phi
		+ \left( 1+A-\Phi\cdot U_* \right) U_* \right] \cdot \Delta_x \left(\Phi-
		\Phi_{\rm {out}} \right)
		\nonumber
		\\
		&
		- 	\left(
		AU_* +  \Pi_{U_*^{\perp}} \Phi \right) \wedge \Delta_x  \left(\Phi-
		\Phi_{\rm {out}} \right) \Big\}
		+
		a\bigg\{
		\bigg[
		|\nabla_x A|^2 |U_*|^2
		+
		2(1+A)\nabla_x A \cdot \left( U_* \cdot \nabla_x U_*\right)
		+
		A(2+A) \left|\nabla_x U_*\right|^2
		\nonumber
		\\
		&
		+
		2 \sum\limits_{k=1}^{2}
		\Big\{
		\left[  \left(\pp_{x_k} A  \right) U_{*}\cdot
		\pp_{x_k} \Phi   + A \pp_{x_k} U_{*} \cdot
		\pp_{x_k} \Phi  \right]
		-
		\pp_{x_k}\left( U_{*} \cdot  \Phi \right) \left[ |U_{*}|^2 \pp_{x_k} A + (1+A)
		U_{*} \cdot  \pp_{x_k} U_{*} \right]
		\nonumber
		\\
		&
		- \left(U_* \cdot \Phi \right) \left[ \left(\pp_{x_k} A \right) U_{*} \cdot   \pp_{x_k} U_*   + (1+A) \left|\pp_{x_k} U_{*}\right|^2 \right]
		\Big\}
		+
		\sum\limits_{k=1}^2
		\left| \pp_{x_k} \Phi
		-
		U_* \pp_{x_k}\left(
		\Phi\cdot  U_* \right)
		-
		(\Phi\cdot U_*) \pp_{x_k} U_*
		\right|^2
		\bigg]
		\Pi_{U_*^{\perp}} \Phi
		\nonumber
		\\
		&
		+
		2 \left(\nabla_x A + U_* \cdot \nabla_x U_* + \Phi\cdot \nabla_x \Phi \right)\cdot \nabla_x U_*
		+ \Delta_x U_* -2 \left(U_* \cdot \nabla_x U_*  \right)\cdot \nabla_x U_*
		+
		\left(A -  \Phi\cdot U_* \right) \Delta_x U_*
		\bigg\}
		\nonumber
		\\
		& +
		2
		\left(
		a -
		b U_* \wedge \right)  \left[ \left(\nabla_x U_* \cdot \nabla_x \Phi \right)  \Phi -  \left(  \Phi\cdot \nabla_x \Phi \right)\cdot \nabla_x U_*  \right]
		-
		2a  \left(\nabla_x U_* \cdot \nabla_x \Phi \right)
		(U_*\cdot \Phi ) U_*
		\nonumber
		\\
		&
		-
		\sum\limits_{j=1}^N
		2\left(
		a -
		b U^{\J} \wedge \right)
		\left\{ \left[\nabla_x U^{\J} \cdot \nabla_x \left( \eta_R^{\J}  Q_{\gamma_j}\Phi_{\rm in}^{\J}  \right) \right]  \left( \eta_R^{\J}  Q_{\gamma_j}\Phi_{\rm in}^{\J}  \right)
		-  \left[  \left( \eta_R^{\J}  Q_{\gamma_j}\Phi_{\rm in}^{\J}  \right) \cdot \nabla_x \left( \eta_R^{\J}  Q_{\gamma_j}\Phi_{\rm in}^{\J}  \right) \right] \cdot \nabla_x U^{\J}  \right\}
		\nonumber
		\\
		&
		+
		b   (\Phi\wedge U_*)
		\left[(1+A)|U_*|^2+\left( U_*  \cdot \Pi_{U_*^{\perp}}\Phi \right)\right]^{-1}
		\left( 1+A-\Phi\cdot U_* \right)
		\left( 2\nabla_x \Phi\cdot\nabla_x U_*  \right)
		-
		b   (\Phi\wedge U_*) \left( 2\nabla_x \Phi\cdot\nabla_x U_*  \right)
		\nonumber
		\\
		&
		-b\bigg[
		-2^{-1} (\Phi\wedge U_*)
		\left[(1+A)|U_*|^2+\left( U_*  \cdot \Pi_{U_*^{\perp}}\Phi \right)\right]^{-1}
		\bigg\{  2 \left( 1+A-\Phi\cdot U_* \right)
		\left(  \Phi\cdot \Delta_x U_*  \right)
		\nonumber
		\\
		& +2(|U_*|^2-2)| \nabla_x \left(\Phi\cdot  U_* \right)  |^2+2|\nabla_x \Phi|^2
		+8[(\Phi\cdot U_*)-(1+A)] (U_* \cdot \nabla_x U_*)\cdot  \nabla_x\left(\Phi\cdot  U_* \right)
		\nonumber
		\\
		& +2|U_*|^2|\nabla_x A|^2+4
		\left[
		-2(\Phi\cdot U_*) U_*\cdot \nabla_x U_*
		+
		(1-|U_*|^2)
		\nabla_x
		\left(
		\Phi\cdot  U_*
		\right)
		\right]
		\cdot \nabla_x A
		\nonumber
		\\
		&
		+8(1+A)\left(U_* \cdot  \nabla_x  U_*\right)\cdot \nabla_x A
		+ 2\left[(\Phi\cdot U_*)-(1+A)\right]^2  \left(|\nabla_x U_*|^2 + U_* \cdot \Delta_x U_*\right) \bigg\}
		\nonumber
		\\
		& -(\Pi_{U_*^{\perp}} \Phi +AU_*)\wedge \left[2 \nabla_x \left(\Phi\cdot  U_* \right) \cdot \nabla_x U_* \right]+[A-(\Phi\cdot U_*)]\Phi\wedge \Delta_x U_*
		\nonumber
		\\
		& +\Pi_{U_*^{\perp}} \Phi \wedge (2\nabla_x A \cdot \nabla_x U_*)+
		\left[(\Phi\cdot U_*)^2-2A(\Phi\cdot U_*) -
		2(\Phi\cdot U_*)\right] U_*\wedge \Delta_x U_*
		\nonumber
		\\
		& +
		(1+A)U_* \wedge  \left[ A \Delta_x U_* + 2\left( \nabla_x A + U_* \cdot \nabla_x U_* + \Phi\cdot \nabla_x \Phi \right) \cdot \nabla_x U_* + \Delta_x U_* -2\left( U_* \cdot \nabla_x U_*  \right) \cdot  \nabla_x U_* \right] \bigg]
		\nonumber
		\\
		& +2b
		A U_* \wedge
		\left[
		\left(   \Phi\cdot \nabla_x \Phi \right)\cdot \nabla_x U_* \right]
		+
		\Xi_{\mathcal{G}}(x,t) U_*,
		\label{G-def}
%	\end{aligned}
%\end{equation}
\end{align}
\end{small}
where  $\Xi_{\mathcal{G}}(x,t)$ is some scalar function from the aforementioned $U_*$-operation; $M_0^{\J}, \tilde{M}_0^{\J},  M_1^{\J},
\tilde{M}_1^{\J}, M_{-1}^{\J} $ are given in \eqref{M0-def-mu3}, \eqref{til-M0-def}, \eqref{M1-def}, \eqref{tilde-M1-def}, \eqref{M(-1)-def}, respectively, with $\mu=3$;
\begin{equation}\label{B-matrix}
	\mathbf{B}_{\Phi,U_{*}} := a \mathbf{I}_3 -bU_* \wedge  + \tilde{\mathbf{B}}_{\Phi,U_{*}} ,
\end{equation}
$\mathbf{I}_3$ is the $3\times 3$ identity matrix,
\begin{equation}\label{til-B-matrix}
	\tilde{ \mathbf{B} }_{\Phi, U_{*}} : = b
	\left[(1+A)|U_*|^2+\left( U_*  \cdot \Pi_{U_*^{\perp}}\Phi \right)\right]^{-1}
	\begin{bmatrix}
		(\Phi\wedge U_*)_1 \left[ \Phi
		+ \left( 1+A-\Phi\cdot U_* \right)
		U_*
		\right]^{\tr}
		\\
		(\Phi\wedge U_*)_2
		\left[ \Phi
		+ \left( 1+A-\Phi\cdot U_* \right)
		U_*
		\right]^{\tr}
		\\
		(\Phi\wedge U_*)_3
		\left[ \Phi
		+ \left( 1+A-\Phi\cdot U_* \right)
		U_*
		\right]^{\tr}
	\end{bmatrix}
	- b \left(AU_* + \Pi_{U_*^{\perp}} \Phi \right)\wedge,
\end{equation}
\begin{equation}\label{Z*-def}
	Z_{*}(x) \in C_0^{\infty}(\RR^2),
	\quad
	\mathsf{supp}Z_{*} \subset B_{C_q},
	 \quad  \| Z_{*} \|_{C^3(\R^2)}\ll 1,  \quad
	\left[
	\pp_{x_{1}} Z_{*1} +
	\pp_{x_{2}} Z_{*2}  + i \left(\pp_{x_{1}} Z_{*2}
	-
	\pp_{x_{2}}  Z_{*1} \right)
	\right](q^{\J})	\ne 0,
\end{equation}
$j=1,2,\dots,N$, where $C_{q} = 9 \max\limits_{ j=1,2,\dots,N } |q^{\J}|$;
\begin{equation}\label{vartheta-def}
\begin{aligned}
&
	\vartheta_{mn}\in C_0^{\infty}(\RR^2), \quad
	\mathsf{supp} \vartheta_{mn} \subset B_{C_q},
	\quad
	\| \vartheta_{mn} \|_{C^{3}(\RR^2)} \le 2,
	\\
	&
	\vartheta_{mn}(q^{\K}) =\delta_{mk} \mathbf{e}_{n},
	\quad
	\nabla \vartheta_{mn}(q^{\K})=0
	 \mbox{ \ for \ } m,k =1,2,\dots,N, \ n=1,2,3,
	\end{aligned}
\end{equation}
\begin{equation*}
	\mathbf{e}_{1} = [1,0,0]^{\tr}, \quad
	\mathbf{e}_{2} = [0,1,0]^{\tr}, \quad
	\mathbf{e}_{3} = [0,0,1]^{\tr} .
\end{equation*}
$c_{mn}$ will be chosen to make $\Phi_{\rm{out}}(q^{\K},T) = 0$ for $k=1,2,\dots,N$.

\subsection{Weighted topologies for the inner and outer problems}\label{sec-topologies}

The topologies for the inner and outer problems are listed in this part.
Recall \eqref{lam-ansatz} and the form of \eqref{inner-eq}. It is natural to introduce new time variables
\begin{equation}\label{tau-j-def}
	\tau_j = \tau_j(t) :=\int_{0}^{t} \lambda_j^{-2}(s) ds + C_{\tau} T \lambda_*^{-2}(0),\quad
	\tau_j(0) =\tau_0:= C_{\tau} T \lambda_*^{-2}(0)
\end{equation}
with a constant $C_\tau>0$ sufficiently large. It follows that $\lambda_j^{2} \pp_t \Phi_{\rm in}^{\J}= \pp_{\tau} \Phi_{\rm in}^{\J}$,
\begin{equation}\label{tau-t}
\begin{aligned}
&
\tau_j(t)\sim |\ln T|^{-2}(T-t)^{-1} |\ln(T-t)|^4,
\quad
	\ln (\tau_j(t)) \sim |\ln(T-t)|,
	\\
&	
\lambda_*(t(\tau_j)) \sim |\ln T|^{-1} \tau_j^{-1} (\ln \tau_j)^2,
\quad
T-t(\tau_j)\sim |\ln T|^{-2} \tau_j^{-1} (\ln \tau_j)^4,
\quad
\pp_{\tau_j} t(\tau_j) = \lambda_j^2 (t(\tau_j)),
\\
&
\pp_{\tau_j} \lambda_*(t(\tau_j))
= (\pp_{t} \lambda_*)(t(\tau_j)) \pp_{\tau_j} t(\tau_j)
\sim -|\ln T|^{-1} \tau_j^{-2} (\ln \tau_j)^2.
\end{aligned}	
\end{equation}

$\bullet$ We endow solutions of the inner problems with the following norms.
\begin{align}
%\begin{equation}
%	\begin{aligned}
		&
		\| \Phi_{\rm{in} }^{\J} \|_{{\rm in},\nu-\delta_0,l } :=
		\sup\limits_{(y,\tau_j) \in \mathcal{D}_{2 C_{\lambda} R}} \Big[
		\big(
		\lambda_*^{\nu-\delta_0}(t(\tau_j) ) \langle y \rangle^{-l}
		\big)^{-1}
		\Big( \big|\Phi_{\rm{in} }^{\J}(y,\tau_j) \big|
		+ \langle y\rangle  \big|D \Phi_{\rm{in} }^{\J}(y,\tau_j) \big|
		+ \langle y\rangle^2  \big|D^2 \Phi_{\rm{in} }^{\J}(y,\tau_j ) \big|
		\Big) \Big],
		\nonumber
		\\
		&
		[\Phi_{\rm{in} }^{\J}]_{{\rm in},\nu-\delta_0,l,\varsigma_{\rm{in} }}
		:=
		\sup\limits_{(y,\tau_j) \in \mathcal{D}_{2 C_{\lambda} R}, \max\{\tau_0, \tau_j-R^2(t(\tau_j)) \} \le s_1 < s_2 \le \tau_j } \bigg\{
		\big[ \big( \lambda_*^{\nu-\delta_0} R^{2-\varsigma_{\rm{in} }} \big)(t(\tau_j)) \big]^{-1}
		\frac{|\Phi_{\rm{in} }^{\J}(y,s_1) - \Phi_{\rm{in} }^{\J}(y,s_2) |}{|s_1 -s_2|^{\varsigma_{\rm{in} }/2}}
		\nonumber
		\\
		&
		\qquad +
		\big[ \big( \lambda_*^{\nu-\delta_0} R^{1-\varsigma_{\rm{in} }} \big)(t(\tau_j)) \big]^{-1}
		\frac{|D \Phi_{\rm{in} }^{\J}(y,s_1) - D \Phi_{\rm{in} }^{\J}(y,s_2) |}{|s_1 -s_2|^{\varsigma_{\rm{in} }/2}}
		\bigg\},
		\nonumber
		\\
		&
		\| \Phi_{\rm{in} }^{\J} \|_{{\rm in},\nu-\delta_0,l,\varsigma_{\rm{in} }} :=
		\| \Phi_{\rm{in} }^{\J} \|_{{\rm in},\nu-\delta_0,l }
		+
		\big[ \Phi_{\rm{in} }^{\J} \big]_{{\rm in},\nu-\delta_0,l,\varsigma_{\rm{in} }},
		\label{inn-topo}
%	\end{aligned}
%\end{equation}
\end{align}
where $\mathcal{D}_{2R}  :=\left\{ (y,\tau_j) \ | \  \tau_j> \tau_0, \ |y| < 2 C_{\lambda} R(t(\tau_j)) \right\}$,
\begin{equation}\label{inn-top0-para}
	l>0,
	\quad
	0<\varsigma_{\rm{in} }<1,
	\quad
	0<\delta_0<\nu <1.
\end{equation}
Set $R_0(t)=\lambda_*^{-\delta_0/6}(t)$, which will be used in the inner problems and reduced equations.

The inner problems will be solved in the following space
\begin{equation}\label{inner-space}
	B_{\rm{in}}^{\J} : = \left\{ \fbf \ | \
	\|  \fbf  \|_{{\rm in},\nu-\delta_0,l, \varsigma_{\rm{in} } } \le \Lambda_{\rm{in}},\quad \fbf\cdot W^{\J}=0   \right\}, \quad j=1,2,\dots, N
\end{equation}
for a constant  $\Lambda_{\rm{in}}\ge 1$ to be determined later.

\medskip

\noindent $\bullet$ For the outer problem, we use the following weights  to control the right-hand side of the outer problem
\begin{equation}\label{rho-weights}
	\varrho_1^{\J}:=  \lambda_*^{\Theta} (\lambda_*R)^{-1} \1_{\{ |x-q^{\J}|\le 3\lambda_*R \} },
		\quad
		\varrho_2^{\J}:=  T^{-\sigma_0} \frac{\lambda_*^{1-\sigma_0}}{|x-q^{\J}|^2}
		\1_{ \{  \lambda_*R/2 \le |x-q^{\J}|\le d_q \} },
		\quad
		\varrho_3:=  T^{-\sigma_0},
\end{equation}
where $d_q$ is given in \eqref{dq-pj-def},
\begin{equation}\label{out-topo-para}
	 \Theta+\beta-1<0, \quad
	0<\Theta<1, \quad 0<\sigma_0<1.
\end{equation}
For a function $f(x,t)$, we define the $L^\infty$-weighted norm
\begin{equation}\label{G-topo-def}
	\|f\|_{**} : =   \sup_{\R^2 \times (0,T)}
	\bigg[\sum_{j=1}^N \left(\varrho_1^{\J}+\varrho_2^{\J}\right)+\varrho_3\bigg]^{-1} |f(x,t)|.
\end{equation}
Also, we define the $L^\infty$-weighted norm for $\Phi_{\rm out}$:
\begin{align}
%\begin{equation}
%	\begin{aligned}
		&
		\| \Phi_{\rm {out} }\|_{\sharp, \Theta,\alpha}
		:=
		\left( |\ln T| \lambda^{\Theta+1}_*(0) R(0)
		+
		\| Z_* \|_{C^3(\R^2)} \right)^{-1} \|\Phi_{\rm {out}}\|_{L^\infty(\R^2\times (0,T))}
		\nonumber
		\\
		&
		\quad +
		\left(\lambda^{\Theta}_*(0)
		+
		\| Z_* \|_{C^3(\R^2)}  \right)^{-1}\|\nabla_x \Phi_{\rm {out} }\|_{L^\infty(\R^2\times (0,T))}
		\nonumber
		\\
		&
		\quad +
		\sup_{\R^2\times (0,T)}
		\left[ |\ln(T-t)|
		\lambda^{\Theta+1}_*(t) R(t)
		+
		(T-t)  \| Z_* \|_{C^3(\R^2)}
		\right]^{-1}  |\Phi_{\rm out}(x,t)-\Phi_{\rm out}(x,T)|
		\nonumber
		\\
		&
		\quad + \sup_{\R^2\times (0,T)}
		\big[ \lambda^{\Theta}_*(t)
		+
		(T-t)^{\frac{\alpha}{2}}  \| Z_* \|_{C^3(\R^2)}
		\big]^{-1}
		\left|\nabla_x \Phi_{\rm out}(x,t)-\nabla_x \Phi_{\rm out}(x,T) \right|
		\nonumber
		\\
		&
		\quad + \sup_{x, x_*\in \mathbb{R}^2, 0<t<t_* < T, \ t_* -t < (T-t)/4 }
		\left[
		\lambda^{\Theta}_*(t)
		(\lambda_*(t) R(t))^{-\alpha}
		+
		\| Z_* \|_{C^3(\R^2)}
		\right]^{-1} \frac {|\nabla_x \Phi_{\rm out}(x,t) -\nabla_x \Phi_{\rm out}(x_{*},t_{*}) |}{ \big( |x-x_{*}| + \sqrt{|t-t_{*}|} \big)^{\alpha} }
\nonumber
\\
&
\quad + \sup_{ x\in \mathbb{R}^2, 0<t<t_*< T, \  t_* - t<(T-t)/4 }
\left[
T^{A_{\rm{o,h}} } (1 + \| Z_* \|_{C^3(\R^2)} )
\right]^{-1} \frac {| \Phi_{\rm out}(x,t) - \Phi_{\rm out}(x,t_{*}) |}{ (t_{*} -t)^{\alpha /2 } }
\label{out-topo}
%	\end{aligned}
%\end{equation}
\end{align}
under assumptions \eqref{para-rho-con} for the parameters.
The outer problem will be solved in
\begin{equation}\label{out-space}
	B_{\rm{out}} : = \left\{ \fbf \ | \  \|  \fbf  \|_{\sharp, \Theta,\alpha}\le \Lambda_{\rm{o}}, \  \fbf(q^{\J},T)=0 \mbox{ \ for \ } j=1,2,\dots,N \right\},
\end{equation}
where $\Lambda_{\rm{o}} \ge 1$ will be determined later.

We take $T$, $\| Z_* \|_{C^3(\R^2)} \ll 1$ depending on $\Lambda_{\rm{o}}$ such that $\| |\Phi_{\rm {out} }| + |\nabla \Phi_{\rm {out} }| \|_{L^{\infty}(\mathbb{R}^2\times (0,T))}
	\ll 1 $. Since $ \Phi_{\rm{in} }^{\J} \in B_{\rm{in}}^{\J}$ with \eqref{inn-top0-para} and $\Phi^{*{\J}}_0$ satisfies \eqref{Phi*-0-j-upp}, for $\Phi$ given in \eqref{u-def}, $|\Phi| \ll a$ holds.

\subsection{Strategy for solving the inner problems}\label{ortho-non-in-sec}

In order to find inner solutions with sufficient space-time decay, we need to impose orthogonality conditions for  $\mathcal{H}^{\J}$ given in \eqref{Hj-def}.
Due to the non-local feature at mode $0$, we will only solve the non-local problem at the leading order and leave the remainder to another piece of an inner problem without the orthogonality condition at mode $0$.

By \eqref{outer into inner} and \eqref{tildeL-component},
we reformulate
\begin{equation}\label{Qdecompo}
	\left(Q_{-\gamma_j}
	\left[
	\left( a-bU^{\J}  \wedge \right)
	\tilde{L}_{U^{\J} }[\Phi_{\rm{out}}]  	\right] \right)_{\mathbb{C}_j}  (x,t)
	=     \tilde{L}_{j}^{\#}[\Phi_{\rm {out} }](y^{\J},t)
	+ \tilde{l}_{j}^{\#}[\Phi_{\rm {out} }](x,t),
\end{equation}
where the leading term $ \tilde{L}_{j}^{\#}[\Phi_{\rm {out} }]$ is given by
\begin{small}
	\begin{align}
		%\begin{equation}
		%	\begin{aligned}
			& \tilde{L}_{j}^{\#}[\Phi_{\rm {out} }](y^{\J},t)
			:=   \tilde{L}_{j,0}^{\#}[\Phi_{\rm {out} }](\rho_j,t)
			+
			e^{i\theta_j} \tilde{L}_{j,1}^{\#}[\Phi_{\rm {out} }](\rho_j,t) + e^{2i\theta_j} \tilde{L}_{j,2}^{\#}[\Phi_{\rm {out} }](\rho_j,t),
			\nonumber
			\\
			& \tilde{L}_{j,0}^{\#}[\Phi_{\rm {out} }](\rho_j,t) :=
			(a-ib) \lambda_j^{-1}
			\rho_j w_{\rho_j}^2(\rho_j)  e^{-i\gamma_j }
			\left[
			\pp_{x_{1}} \left(\Phi_{\rm {out} }\right)_1 +
			\pp_{x_{2}} \left(\Phi_{\rm {out} }\right)_2  + i \left(\pp_{x_{1}} \left(\Phi_{\rm {out} }\right)_2
			-
			\pp_{x_{2}}  \left(\Phi_{\rm {out} }\right)_1  \right)
			\right](q^{\J},t),
			\nonumber
			\\
			& \tilde{L}_{j,1}^{\#}[\Phi_{\rm {out} }](\rho_j,t) := (a-ib) 2
			\lambda_j^{-1} w_{\rho_j}(\rho_j) \cos w(\rho_j)
			\left[ -  \pp_{x_{1}} \left(\Phi_{\rm {out} }\right)_3  + i
			\pp_{x_{2}} \left(\Phi_{\rm {out} }\right)_3
			\right](q^{\J},t),
			\nonumber
			\\
			& \tilde{L}_{j,2}^{\#}[\Phi_{\rm {out} }](\rho_j,t) :=
			(a-ib) \lambda_j^{-1} \rho_j w_{\rho_j}^2(\rho_j)
			e^{i\gamma_j}
			\left[
			\pp_{x_1} \left(\Phi_{\rm {out} }\right)_1 - \pp_{x_2} \left(\Phi_{\rm {out} }\right)_2
			-i\left(\pp_{x_1} \left(\Phi_{\rm {out} }\right)_2 + \pp_{x_2} \left(\Phi_{\rm {out} }\right)_1  \right)
			\right](q^{\J},t),
			\label{Llarge}
			%	\end{aligned}
		%\end{equation}
	\end{align}
\end{small}
and the smaller term $\tilde{l}_{j}^{\#}[\Phi_{\rm {out} }]$ is given by
\begin{small}
	\begin{align}
%	\begin{equation}
%		\begin{aligned}
			&
			\tilde{l}_{j}^{\#}[\Phi_{\rm {out} }](x,t) :=
			(a-ib)
			\Big[
			\lambda_j^{-1}
			\rho_j w_{\rho_j}^2(\rho_j)  e^{-i\gamma_j }
			\Big\{
			\left[
			\pp_{x_{1}} \left(\Phi_{\rm {out} }\right)_1 +
			\pp_{x_{2}} \left(\Phi_{\rm {out} }\right)_2  + i \left(\pp_{x_{1}} \left(\Phi_{\rm {out} }\right)_2
			-
			\pp_{x_{2}}  \left(\Phi_{\rm {out} }\right)_1  \right)
			\right](x,t)
			\nonumber
			\\
			&
			-
			\left[
			\pp_{x_{1}} \left(\Phi_{\rm {out} }\right)_1 +
			\pp_{x_{2}} \left(\Phi_{\rm {out} }\right)_2  + i \left(\pp_{x_{1}} \left(\Phi_{\rm {out} }\right)_2
			-
			\pp_{x_{2}}  \left(\Phi_{\rm {out} }\right)_1  \right)
			\right](q^{\J},t) \Big\}
			\nonumber
			\\
			& +
			e^{i\theta_j} 2
			\lambda_j^{-1} w_{\rho_j}(\rho_j) \cos w(\rho_j)
			\left\{
			\left[ -  \pp_{x_{1}} \left(\Phi_{\rm {out} }\right)_3  + i
			\pp_{x_{2}} \left(\Phi_{\rm {out} }\right)_3
			\right](x,t)
			-
			\left[ -  \pp_{x_{1}} \left(\Phi_{\rm {out} }\right)_3  + i
			\pp_{x_{2}} \left(\Phi_{\rm {out} }\right)_3
			\right](q^{\J},t)
			\right\}
			\nonumber
			\\
			& + e^{2i\theta_j}
			\lambda_j^{-1} \rho_j w_{\rho_j}^2(\rho_j)
			e^{i\gamma_j}\Big\{ \left[
			\pp_{x_1} \left(\Phi_{\rm {out} }\right)_1 - \pp_{x_2} \left(\Phi_{\rm {out} }\right)_2
			-i\left(\pp_{x_1} \left(\Phi_{\rm {out} }\right)_2 + \pp_{x_2} \left(\Phi_{\rm {out} }\right)_1  \right)
			\right](x,t)
			\nonumber
			\\
			&
			-
			\left[
			\pp_{x_1} \left(\Phi_{\rm {out} }\right)_1 - \pp_{x_2} \left(\Phi_{\rm {out} }\right)_2
			-i\left(\pp_{x_1} \left(\Phi_{\rm {out} }\right)_2 + \pp_{x_2} \left(\Phi_{\rm {out} }\right)_1  \right)
			\right](q^{\J},t)
			\Big\} \Big].
			\label{Lsmall}
%		\end{aligned}
%	\end{equation}
	\end{align}
\end{small}
By \eqref{out-topo},
$| \dot{\xi}^{\J}| \le C_{\xi} \lambda_*^{\epsilon_{\xi}} $ in
\eqref{lam-ansatz}, we have
\begin{equation}\label{out-to-in-1}
	\begin{aligned}
		&
		|\tilde{l}_{j}^{\#}[\Phi_{\rm {out} }](x,t)|	\lesssim
		\lambda_j^{-1} \langle \rho_j \rangle^{-2} \| \Phi_{\rm {out} }\|_{\sharp, \Theta,\alpha}
		\left( \lambda^{\Theta}_* (\lambda_* R)^{-\alpha}
		+  \| Z_* \|_{C^3(\R^2)}
		\right)\left| x-\xi^{\J} + \xi^{\J}-q^{\J}\right|^{\alpha}
		\\
		\lesssim \ &
		\lambda_*^{\alpha-1}
		\left( \lambda^{\Theta}_* (\lambda_* R)^{-\alpha}
		+  \| Z_* \|_{C^3(\R^2)}
		\right)
		\langle \rho_j \rangle^{\alpha-2} \| \Phi_{\rm {out} }\|_{\sharp, \Theta,\alpha}.
	\end{aligned}
\end{equation}

Similar to \eqref{modek-def},
using the polar coordinates \eqref{polar-coor}, we define the mode $k$ component of $\tilde{l}_{j}^{\#}[\Phi_{\rm {out} }]$ by
\begin{equation}\label{l-jk0-def}
	\tilde{l}_{j,k}^{\#}[\Phi_{\rm {out} }](\rho_j,t) :=
	(2\pi)^{-1}
	\int_{0}^{2\pi} \tilde{l}_{j}^{\#}[\Phi_{\rm {out} }]( \lambda_j \rho_j e^{is} + \xi^{\J},t) e^{-ik s}ds.
\end{equation}
In view of Proposition \ref{keyprop}, we will put $\tilde{l}_{j,0}^{\#}[\Phi_{\rm {out} }](\rho_j,t)$ and the mode $0$ component $( \mathcal{H}_{ \rm{in} }^{\J})_{\mathbb{C}_j,0}$
into non-orthogonal inner problems instead of orthogonal inner problems. Here, the orthogonal (resp. non-orthogonal) inner problem denotes the inner problem with (resp. without) orthogonality conditions at corresponding modes imposed.
 More precisely, for $j= 1,2, \dots,N$,
we consider the following two parts.
\\
\textbf{Orthogonal inner problems:}
\begin{align}
	\notag
		 \lambda_j^{2} \pp_t \Phi_{\rm in}^{\mbox{{\tiny{$[j1]$}}}}  = &~
		( 	a -b W^{\J}  \wedge )
		\Big[
		\Delta_{y^{\J}} \Phi_{\rm in}^{\mbox{{\tiny{$[j1]$}}}}
		+
		|\nabla_{y^{\J}} W^{\J} |^2 \Phi_{\rm in}^{\mbox{{\tiny{$[j1]$}}}}
		- 2 \nabla_{y^{\J}} 	
		\Big( 	W^{\J} \cdot   \Phi_{\rm in}^{\mbox{{\tiny{$[j1]$}}}}  \Big) \cdot \nabla_{y^{\J}} W^{\J}
		\\ \notag
		&~
		+
		2  \Big( 	\nabla_{y^{\J}} W^{\J} \cdot    \nabla_{y^{\J}} \Phi_{\rm in}^{\mbox{{\tiny{$[j1]$}}}}  \Big) W^{\J}	
		\Big]  + \mathcal{H}^{\J}_1
		-
		\lambda_j^{2} \Big( \tilde{l}_{j,0}^{\#}[\Phi_{\rm {out} }]  \Big)_{\mathbb{C}_j^{-1}}
		+
		\mathcal{H}_{ \rm{in} }^{\J}
		-
		\Big( ( \mathcal{H}_{ \rm{in} }^{\J} )_{\mathbb{C}_j,0}
		\Big)_{ \mathbb{C}_j^{-1} }
		\\ \label{inner-eq-1}
		&~
		+  \Big(
		\sum\limits_{k=0}^1
		e^{ik \theta_j}
		c_k^{\J}(\tau_j(t ) ) \eta(|y^{\J}|) \mathcal{Z}_{k,1}(|y^{\J}|)
		\Big)_{\mathbb{C}_j^{-1}}
		\quad
		\mbox{ in }~ \mathsf{D}_{2 C_{\lambda} R},
	\end{align}
\textbf{Non-orthogonal inner problems:}
\begin{align}
	\notag
		 \lambda_j^{2} \pp_t \Phi_{\rm in}^{\mbox{{\tiny{$[j2]$}}}}  = &~
		( a -b W^{\J}  \wedge )
		\Big[
		\Delta_{y^{\J}} \Phi_{\rm in}^{\mbox{{\tiny{$[j2]$}}}}
		+
		|\nabla_{y^{\J}} W^{\J} |^2 \Phi_{\rm in}^{\mbox{{\tiny{$[j2]$}}}}
		- 2 \nabla_{y^{\J}} 	
		\Big( 	W^{\J} \cdot   \Phi_{\rm in}^{\mbox{{\tiny{$[j2]$}}}}  \Big) \cdot \nabla_{y^{\J}} W^{\J}
		\\ \notag
		&~
		+
		2  \Big( \nabla_{y^{\J}} W^{\J} \cdot    \nabla_{y^{\J}} \Phi_{\rm in}^{\mbox{{\tiny{$[j2]$}}}}  \Big) W^{\J}	
		\Big]
		+
		\lambda_j^{2} \Big( \tilde{l}_{j,0}^{\#}[\Phi_{\rm {out} }] \Big)_{\mathbb{C}_j^{-1}}
		+
		\Big( ( \mathcal{H}_{ \rm{in} }^{\J}  )_{\mathbb{C}_j,0}
		\Big)_{ \mathbb{C}_j^{-1} }  +
		\mathbf{R}_0\left[\Phi_{\rm{out}} , \lambda_j, \gamma_j \right]
		\\ \label{inner-eq-2}
		&~ +
		\Big(\int_{0}^2 \eta(r) \mathcal{Z}_{0,1}^2(r) r dr \Big)^{-1}
		\Big(
		c_{*0}^{\J}(\tau_j(t ))   \eta(|y^{\J}|) \mathcal{Z}_{0,1}(|y^{\J}|)
		\Big)_{\mathbb{C}_j^{-1}}
		\quad
		\mbox{ in }~ \mathsf{D}_{2 C_{\lambda} R},
	\end{align}
where
\begin{align}
	\notag
		&
		\mathbf{R}_0 [\Phi_{\rm{out}}, \lambda_j, \gamma_j ](y,t):=
		-\Big(\int_{0}^2 \eta(r) \mathcal{Z}_{0,1}^2(r) r dr \Big)^{-1}
		\lambda_j
		\eta( |y^{\J}| ) \mathcal{Z}_{0,1}( |y^{\J}| )
		\big( e^{-i\gamma_j(t)}
		\mathcal{R}_0[ {\rm DC}_j [\Phi_{\rm{out}}] ] (t)
		\big)_{\mathbb{C}_{j}^{-1} },
		\\ \label{R0-def}
		&
		{\rm DC}_j[\mathbf{f}] =
		{\rm DC}_j[\mathbf{f}](t) :=(a-ib)
		\left[
		\pp_{x_{1}} f_1 +
		\pp_{x_{2}} f_2  + i \left(\pp_{x_{1}} f_2
		-
		\pp_{x_{2}}  f_1 \right)
		\right](q^{\J},t)
	\end{align}
for $\mathbf{f}= (f_1,f_2,f_3) \in L^{\infty} \big( (0,T); C^1(\mathbb{R}^2) \big)$;
the operator $\mathcal{R}_0$ will be given in Proposition \ref{keyprop}, and the reason for the choice of $\mathbf{R}_0$ will be shown in \eqref{orth0-eq-r}; under suitable assumptions on parameters, using Propositions \ref{Re-m0-prop} and \ref{qd24July12-8-prop}  with $R_*=R_1=\infty$, we will take
\begin{align}
	%\begin{equation}
	%	\begin{aligned}
		&
		c_{0}^{\J}(\tau_j(t))=  c_{0} \big[ [ ( \mathcal{H}^{\J}_1  )_{\mathbb{C}_j , 0 }
		-
		\lambda_j^{2}   \tilde{l}_{j,0}^{\#}
		]_{\mathbb{C}_j^{-1}} \big](\tau_j(t))
		\nonumber
		\\
		& \qquad \qquad =
		-\Big(\int_{0}^2 \eta(r) \mathcal{Z}_{0,1}^2(r) r dr\Big)^{-1}
		\Big\{  \int_{0}^{\infty}
		\big[ ( \mathcal{H}^{\J}_1 )_{\mathbb{C}_j , 0 } -
		\lambda_j^{2}   \tilde{l}_{j,0}^{\#}   \big] (\rho_j,t)  \mathcal{Z}_{0,1}(\rho_j) \rho_j d \rho_j
		+ c_{*0}^{\J}(\tau_j(t)) \Big\}
		,
		\nonumber
		\\
		&
		c_{*0}^{\J}(\tau_j(t))=c_{*0}\big[ [ ( \mathcal{H}^{\J}_1 )_{\mathbb{C}_j , 0 }
		-
		\lambda_j^{2}   \tilde{l}_{j,0}^{\#}
		]_{\mathbb{C}_j^{-1}} \big](\tau_j(t)),
		\nonumber
		\\
		&
		c_{1}^{\J}(\tau_j(t))
		=
		c_{1}\big[ [  ( \mathcal{H}^{\J}_1
		+
		\mathcal{H}_{ \rm{in} }^{\J}   )_{\mathbb{C}_j,1} e^{i\theta_j}  ]_{\mathbb{C}_j^{-1}} \big](\tau_j(t))
		\nonumber
		\\
		&
		\qquad \qquad
		=
		-\Big(\int_{0}^2 \eta(r) \mathcal{Z}_{1,1}^2(r) r dr\Big)^{-1}
		\Big\{ \int_{0}^{\infty}
		\big( \mathcal{H}^{\J}_1
		+
		\mathcal{H}_{ \rm{in} }^{\J}  \big)_{\mathbb{C}_j,1}(\rho_j,t)  \mathcal{Z}_{1,1}(\rho_j) \rho_j d \rho_j
		+ c_{*1}^{\J}(\tau_j(t)) \Big\},
		\nonumber
		\\
		&
		c_{*1}^{\J}(\tau_j(t))= c_{*1}\big[ [  ( \mathcal{H}^{\J}_1
		+
		\mathcal{H}_{ \rm{in} }^{\J}   )_{\mathbb{C}_j,1} e^{i\theta_j}  ]_{\mathbb{C}_j^{-1}} \big](\tau_j(t)).
		\label{qd24July14-1}
		%	\end{aligned}
	%\end{equation}
\end{align}

\section{Reduced equations}\label{sec-orthogonal}

\subsection{Reformulation of reduced equations}

We will consider the reduced equations
\begin{align}
\notag
		& c_0^{\J}(\tau_j(t)) + \Big(\int_{0}^2 \eta(r) \mathcal{Z}_{0,1}^2(r) r dr \Big)^{-1}
		c_{*0}^{\J}(\tau_j(t )) - \Big(\int_{0}^2 \eta(r) \mathcal{Z}_{0,1}^2(r) rdr \Big)^{-1} \lambda_j   e^{-i\gamma_j(t)} \mathcal{R}_0[ {\rm DC}_j [\Phi_{\rm{out}}] ](t)  =0,
		\\ \label{ortho-eq}
		&
		c_1^{\J}(\tau_j(t))=0 \mbox{ \ for \ } t\in (0,T).
	\end{align}
If \eqref{ortho-eq} is true,
then $\eqref{inner-eq-1}$ and $\eqref{inner-eq-2}$ will give a solution $\Phi_{\rm in}^{\J} = \Phi_{\rm in}^{\mbox{{\tiny{$[j1]$}}}} + \Phi_{\rm in}^{\mbox{{\tiny{$[j2]$}}}}$ for \eqref{inner-eq}.
In the following lemma, we write \eqref{ortho-eq} in a form that is more convenient to handle.

\begin{lemma}
	The reduced problem \eqref{ortho-eq} is equivalent to
	\begin{equation}\label{orth0-eq-r}
		\mathcal{B}_0[p_j](t) =
		{\rm DC}_j[\Phi_{\rm{out}}](t)  +
		\mathcal{R}_0[ {\rm DC}_j [\Phi_{\rm{out}}] ](t),
	\end{equation}
	\begin{equation}\label{orth1-eqr}
		\begin{aligned}
			\dot{\xi}_1^{\J} - i\dot{\xi}_2^{\J}    = \ &
			\lambda_j
			\int_{0}^{\infty}
			\left( Q_{-\gamma_j}
			\left[
			\left( a-bU^{\J}  \wedge \right)
			\tilde{L}_{U^{\J} }[\Phi_{\rm{out}}]  	\right] \right)_{\mathbb{C}_j,1} (\rho_j,t)  \mathcal{Z}_{1,1}(\rho_j) \rho_j d\rho_j
			\\
			& +
			\lambda_j^{-1}
			\int_{0}^{\infty}
			(
			\mathcal{H}_{ \rm{in} }^{\J} )_{\mathbb{C}_j,1} (\rho_j,t)  \mathcal{Z}_{1,1}(\rho_j) \rho_j d\rho_j
			+ \lambda_j^{-1} c_{*1}^{\J}( \tau_j(t) ),
		\end{aligned}
	\end{equation}
	where
	\begin{equation}\label{qd240802-3}
		\begin{aligned}
			&
			\mathcal{B}_0[p_j](t) := - \bigg\{ \int_{-T}^t \frac{\dot p_j(s) }{t-s}
			\left[
			\left(
			-1+O\left( \iota_j \langle \ln \iota_j \rangle \right) \right)\1_{\{ \iota_j\le 1\}}
			+
			O (\iota_j^{-1}) \1_{\{ \iota_j > 1\}}
			\right]
			ds
			\\
			& \quad +
			(a-ib) e^{i\gamma_j(t)}
			{\rm{Re}}
			\bigg[
			\int_{-T}^t \frac{\dot p_j(s) e^{-i \gamma_j(t) } }{t-s}
			\left(
			O( \iota_j \langle \ln \iota_j \rangle )
			\1_{\{\iota_j \le 1 \} }
			+
			O (\iota_j^{-1})
			\1_{\{\iota_j > 1 \} }
			\right)
			ds   \bigg]
			+  C_{p1} \dot{p}_j
			+ C_{p2} e^{i\gamma_j(t)} \dot\la_j
			\bigg\}.
		\end{aligned}
	\end{equation}
	
\end{lemma}

\begin{proof}
	
	\textbf{Terms in $c_0^{\J}(\tau_j(t) )$.}
	There exists the correspondence
	between kernels \eqref{def-kernels} and \eqref{scalr-Z} through \eqref{ZZ0} and \eqref{ZZ1}.
	See \eqref{scalr-Z}, $\mathcal{Z}_{0,1}(\rho_j) \rho_j = \rho_j^2 (\rho_j^2+ 1)^{-1}$.
	By \eqref{Hj-1-def} and \eqref{Qdecompo}, we have
	\begin{equation*}%\label{mode-orth-int}
		\int_{0}^{\infty}
		\big[
		( \mathcal{H}^{\J}_1 )_{\mathbb{C}_j , 0 } -
		\lambda_j^{2}   \tilde{l}_{j,0}^{\#} \big]  (\rho_j,t)  \mathcal{Z}_{0,1}(\rho_j) \rho_j d \rho_j
		=
		\lambda_j^2  \int_{0}^{\infty}
		\left(
		\tilde{L}_{j,0}^{\#} +  M_0^{\J}    \right)(\rho_j,t)  \mathcal{Z}_{0,1}(\rho_j) \rho_j d\rho_j.
	\end{equation*}
	Using \eqref{qd24July14-1}, the first equation in \eqref{ortho-eq} can be written as
	\begin{equation}\label{qd240803-1}
		\mathcal{R}_0[ {\rm DC}_j [\Phi_{\rm{out}}] ] (t)  +
		\lambda_j e^{i\gamma_j(t)} \int_{0}^{\infty}
		\left(
		\tilde{L}_{j,0}^{\#} +  M_0^{\J}    \right)(\rho_j,t)  \mathcal{Z}_{0,1}(\rho_j) \rho_j d\rho_j = 0.
	\end{equation}
	Here,
	\begin{equation}\label{qd24July12-10}
		\int_{0}^{\infty} \tilde{L}_{j,0}^{\#}(\rho_j,t)  \mathcal{Z}_{0,1}(\rho_j) \rho_j d\rho_j
		=
		\lambda_j^{-1}
		e^{-i\gamma_j(t)}   {\rm DC}_j [\Phi_{\rm{out}}],	
	\end{equation}
	where we used $\int_{0}^{\infty}  \rho_j w_{\rho_j}^2(\rho_j) \mathcal{Z}_{0,1}(\rho_j) \rho_j d\rho_j = \int_{0}^{\infty}  \frac{4\rho_j^3}{(\rho_j^2+1)^3} d\rho_j = 1$.
	
	Recall \eqref{zetaj},
	$\zeta_j =
	\iota_j (\rho_j^2+1), \iota_j= \lambda_j^2(t) (t-s)^{-1} $. By \eqref{M0-ortho-form} and $\mathcal{Z}_{0,1}(\rho_j) \rho_j = \rho_j^2 (\rho_j^2+ 1)^{-1}$, then
	\begin{small}
		\begin{align*}
			%\begin{small}
			%\begin{equation*}
			%	\begin{aligned}
				&
				\int_{0}^{\infty} M_0(\rho_j,t) \mathcal{Z}_{0,1}(\rho_j) \rho_j d\rho_j
				=
				\lambda_j^{-1}
				\int_{-T}^t \frac{\dot p_j(s) e^{-i\gamma_j(t)} }{t-s}
				\int_{0}^{\infty}
				\\
				&
				\bigg\{
				\bigg[
				\frac{\rho_j^3 (3\rho_j^7+ \rho_j^6 +12\rho_j^5 -15\rho_j^4+11\rho_j^3-24\rho_j^2-8)}{2(\rho_j^2+1)^{\frac{5}{2}} (\rho_j^3+1)^3 }
				+ O\left( \iota_j  \langle \rho_j\rangle^{-1} \right) \bigg]
				\1_{\{ \iota_j (\rho_j^2+1) \le 1\}}
				+
				O\left( \iota_j ^{-1} \langle \rho_j\rangle^{-5} \right) \1_{\{ \iota_j (\rho_j^2+1) > 1\}}
				\bigg\} d\rho_j
				ds
				\\
				& -
				\lambda_j^{-1}
				\mathrm{Re}
				\bigg[
				\int_{-T}^t \frac{\dot p_j(s) e^{-i\gamma_j(t)} }{t-s}
				\int_{0}^{\infty}
				\bigg\{
				\bigg[
				\frac{\rho_j^3 (3\rho_j^7+ \rho_j^6 +12\rho_j^5 -15\rho_j^4+11\rho_j^3-24\rho_j^2-8)}{(\rho_j^2+1)^{\frac{7}{2}} (\rho_j^3+1)^3 }
				+ O\left( \iota_j  \langle \rho_j\rangle^{-3} \right) \bigg]
				\1_{\{ \iota_j (\rho_j^2+1) \le 1\}}
				\\
				&
				+
				O\left( \iota_j ^{-1} \langle \rho_j\rangle^{-7} \right) \1_{\{ \iota_j (\rho_j^2+1) > 1\}}
				\bigg\} d\rho_j
				ds   \bigg]
				+
				b(ia+b) \lambda_j^{-1}  \mathrm{Re} \bigg[
				\\
				&
				\int_{-T}^t \frac{ \dot{p}_j(s) e^{-i\gamma_j(t)} }{t-s}
				\int_{0}^{\infty}
				\bigg\{
				\bigg[
				\frac{\rho_j^3 ( 3\rho_j^7 + \rho_j^6 +12\rho_j^5 -15 \rho_j^4 +11\rho_j^3-24\rho_j^2-8)}{(\rho_j^2+1)^{\frac{7}{2}} (\rho_j^3+1)^3}
				+ O\left(\iota_j \langle \rho_j\rangle^{-1} \right) \bigg]
				\1_{\{ \iota_j (\rho_j^2+1)\le 1\}}
				\\
				&
				+
				O\left( \iota_j ^{-1} \langle \rho_j\rangle^{-5} \right) \1_{\{ \iota_j (\rho_j^2+1) > 1\}}
				\bigg\} d\rho_j ds  \bigg]
				+
				b(ib-a) \lambda_j^{-1}
				\mathrm{Re} \bigg[
				\\
				&
				\int_{-T}^t \frac{ \dot{p}_j(s) e^{-i\gamma_j(t)} }{t-s}
				\int_{0}^{\infty}
				\bigg\{
				\bigg[
				\frac{i\rho_j^3 ( 3\rho_j^7 + \rho_j^6 +12\rho_j^5 -15 \rho_j^4 +11\rho_j^3-24\rho_j^2-8)}{(\rho_j^2+1)^{\frac{7}{2}} (\rho_j^3+1)^3}
				+ O\left(\iota_j \langle \rho_j\rangle^{-1} \right) \bigg]
				\1_{\{ \iota_j (\rho_j^2+1)\le 1\}}
				\\
				&
				+
				O\left( \iota_j ^{-1} \langle \rho_j\rangle^{-5} \right) \1_{\{ \iota_j (\rho_j^2+1) > 1\}}
				\bigg\} d\rho_j ds   \bigg]
				-
				(a-ib) \lambda_j^{-1}
				\int_{-T}^t \frac{\dot p_j(s) e^{-i \gamma_j(t) } }{t-s}  \int_{0}^{\infty}
				\bigg\{
				\\
				&
				\bigg[ \frac{4  (a+ib) \rho_j^5  }{(\rho_j^2+ 1)^{\frac{5}{2}}(\rho_j^3+1) } + O\left(\iota_j  \langle \rho_j\rangle^{-1}\right) \bigg]\1_{\{ \iota_j (\rho_j^2+1)\le 1\}}
				+
				O\left(\iota_j^{-1}  \langle \rho_j\rangle^{-5} \right) \1_{\{ \iota_j (\rho_j^2+1) > 1\}}
				\bigg\}
				d\rho_j
				ds
				\\
				& +
				(a-ib) \lambda_j^{-1}
				{\rm{Re}}
				\bigg[ 	\int_{-T}^t \frac{\dot p_j(s) e^{-i \gamma_j(t) } }{t-s}
				\int_{0}^{\infty}
				\bigg\{
				\\
				&
				\bigg[ \frac{8  (a+ib) \rho_j^5  }{(\rho_j^2+ 1)^{\frac{7}{2}}(\rho_j^3+1) } + O\left(\iota_j  \langle \rho_j\rangle^{-3}\right) \bigg]\1_{\{ \iota_j (\rho_j^2+1)\le 1\}}
				+
				O\left(\iota_j^{-1}  \langle \rho_j\rangle^{-7} \right) \1_{\{ \iota_j (\rho_j^2+1) > 1\}}
				\bigg\}
				d\rho_j
				ds   \bigg]
				\\
				&
				- (a-ib)  \lambda_j^{-1}
				\mathrm{Re}
				\bigg[
				\int_{-T}^t \frac{\dot p_j(s) e^{-i\gamma_j(t)} }{t-s}
				\int_{0}^{\infty}  \bigg\{
				\bigg[ (a+ib)
				\frac{4\rho_j^5(\rho_j^3+3\rho_j^2+4)}{(\rho_j^2+1)^{\frac{7}{2}} (\rho_j^3+1)^2 }
				+ O\left(\iota_j  \langle \rho_j\rangle^{-1} \right) \bigg]
				\1_{\{ \iota_j (\rho_j^2+1)\le 1\}}
				\\
				&
				+
				O\left( \iota_j^{-1} \langle \rho_j\rangle^{-5} \right)   \1_{\{ \iota_j (\rho_j^2+1) > 1\}}
				\bigg\} d\rho_j
				ds \bigg]
				+  C_{p1} p_j^{-1} \dot{p}_j
				+ C_{p2}  \la_j^{-1}\dot\la_j
				\\
				%	\end{aligned}
			%\end{equation*}
			%\end{small}
			%\begin{equation*}
			%	\begin{aligned}
				= \ &
				\lambda_j^{-1}
				\int_{-T}^t \frac{\dot p_j(s) e^{-i\gamma_j(t)} }{t-s}
				\int_{0}^{\infty}
				\left\{
				\left[  f_1(\rho_j)
				+ O\left( \iota_j  \langle \rho_j\rangle^{-1} \right) \right]
				\1_{\{ \iota_j (\rho_j^2+1) \le 1\}}
				+
				O\left( \iota_j ^{-1} \langle \rho_j\rangle^{-5} \right) \1_{\{ \iota_j (\rho_j^2+1) > 1\}}
				\right\} d\rho_j
				ds
				\\
				&
				+
				(a-ib) \lambda_j^{-1}
				{\rm{Re}}
				\bigg[ 	\int_{-T}^t \frac{\dot p_j(s) e^{-i \gamma_j(t) } }{t-s}
				\int_{0}^{\infty}
				\bigg\{
				\left[ (a+ib) f_2(\rho_j) + O\left(\iota_j  \langle \rho_j\rangle^{-1}\right) \right]\1_{\{ \iota_j (\rho_j^2+1)\le 1\}}
				\\
				&
				+
				O\left(\iota_j^{-1}  \langle \rho_j\rangle^{-5} \right) \1_{\{ \iota_j (\rho_j^2+1) > 1\}}
				\bigg\}
				d\rho_j
				ds   \bigg]
				+  C_{p1} p_j^{-1} \dot{p}_j
				+ C_{p2} \la_j^{-1}\dot\la_j,
				%	\end{aligned}
			%\end{equation*}
		\end{align*}
	\end{small}
	where, for brevity, we denote
	\begin{align*}
		%\begin{equation*}
		%	\begin{aligned}
			&
			C_{p1}:=
			\int_{0}^{\infty}
			\frac{2\rho_j^3[ \rho_j^{2} - \rho_j  - (\rho_j^2+1)^{\frac{1}{2}} ] }{ [ \rho_j + (\rho_j^2+1)^{\frac{1}{2}} ] (\rho_j^3 + 1)(\rho_j^2+1)^2}
			d\rho_j
			= -0.123584,
			\\
			&
			C_{p2}:=
			-
			\int_{0}^{\infty}
			\frac{ 4\rho_j^{5} [  \rho_j^2  + \rho_j(\rho_j^2+1)^{\frac{1}{2}} + 1 ]  }{ [ \rho_j + (\rho_j^2+1)^{\frac{1}{2}} ] (\rho_j^3 + 1)(\rho_j^2+1)^3 }d\rho_j = -0.823455,
			\\
			&
			f_1(\rho_j) :=
			\frac{-8 \rho_j^{11} + 3 \rho_j^{10} + \rho_j^9 - 4 \rho_j^8 - 15 \rho_j^7 + 11 \rho_j^6 - 32 \rho_j^5 - 8 \rho_j^3
			}{2(\rho_j^2+1)^{\frac{5}{2}} (\rho_j^3+1)^3 } ,
			\\
			&
			f_2(\rho_j) :=
			\frac{
				4\rho_j^{11} -15\rho_j^{10} - \rho_j^9 -16\rho_j^8 + 3\rho_j^7 -11\rho_j^6 +16\rho_j^5 + 8\rho_j^3		
			}{(\rho_j^2+ 1)^{\frac{7}{2}}(\rho_j^3+1)^3 }.
			%\end{aligned}
			%\end{equation*}
		\end{align*}
		Notice $\int_{0}^{\infty}  f_1(\rho_j)  d\rho_j = -1 $,
		$\int_{0}^{\infty}   f_2(\rho_j)  d\rho_j =0 $. Then
		\begin{align*}
			%\begin{equation*}
			%	\begin{aligned}
				&
				\int_{0}^{\infty}
				\left\{
				\left[  f_1(\rho_j)
				+ O\left( \iota_j  \langle \rho_j\rangle^{-1} \right) \right]
				\1_{\{ \iota_j (\rho_j^2+1) \le 1\}}
				+
				O\left( \iota_j ^{-1} \langle \rho_j\rangle^{-5} \right) \1_{\{ \iota_j (\rho_j^2+1) > 1\}}
				\right\} d\rho_j
				\\
				= \ &
				\left(
				-1+O( \iota_j \langle \ln \iota_j \rangle ) \right) \1_{\{ \iota_j\le 1\}}
				+
				O (\iota_j^{-1}) \1_{\{ \iota_j > 1\}},
				\\
				&
				\int_{0}^{\infty}
				\left\{
				\left[ (a+ib) f_2(\rho_j) + O\left(\iota_j  \langle \rho_j\rangle^{-1}\right) \right]\1_{\{ \iota_j (\rho_j^2+1)\le 1\}}
				+
				O\left(\iota_j^{-1}  \langle \rho_j\rangle^{-5} \right) \1_{\{ \iota_j (\rho_j^2+1) > 1\}}
				\right\}
				d\rho_j
				\\
				= \ &
				O(\iota_j \langle \ln \iota_j \rangle )
				\1_{\{\iota_j \le 1 \} }
				+
				O (\iota_j^{-1})
				\1_{\{\iota_j > 1 \} }
				%	\end{aligned}
			%\end{equation*}
		\end{align*}
		since
		for $\iota_j>1$, $
		\int_{0}^{\infty}
		O\big( \iota_j ^{-1} \langle \rho_j\rangle^{-5} \big)   d\rho_j  =
		O\big( \iota_j ^{-1}  \big) $, and
		for $0< \iota_j\le 1$,
$$
				\int_{0}^{(\iota_j^{-1} -1)^{\frac{1}{2}} }
				\left[  f_1(\rho_j)
				+ O\left( \iota_j  \langle \rho_j\rangle^{-1} \right) \right]   d\rho_j   +
				\int_{(\iota_j^{-1} -1)^{\frac{1}{2}}}^{\infty}
				O\left( \iota_j ^{-1} \langle \rho_j\rangle^{-5} \right)  d\rho_j
				=
				-1 + O(\iota_j \langle \ln \iota_j\rangle ),
$$
$$
				\int_{0}^{(\iota_j^{-1} -1)^{\frac{1}{2}} }
				\left[ (a+ib) f_2(\rho_j) + O\left(\iota_j  \langle \rho_j\rangle^{-1}\right) \right]
				d\rho_j
				+ \int_{(\iota_j^{-1} -1)^{\frac{1}{2}}}^{\infty}
				O\left(\iota_j^{-1}  \langle \rho_j\rangle^{-5} \right)
				d\rho_j
				=
				O(\iota_j \langle \ln \iota_j \rangle ).
$$
		Thus, we arrive at
		\begin{equation}\label{M0-ortho}
			\begin{aligned}
				&
				\int_{0}^{\infty} M_0(\rho_j,t) \mathcal{Z}_{0,1}(\rho_j) \rho_j d\rho_j
				=
				\lambda_j^{-1}
				\int_{-T}^t \frac{\dot p_j(s) e^{-i\gamma_j(t)} }{t-s}
				\left[
				\left(
				-1+O\left( \iota_j \langle \ln \iota_j \rangle \right) \right)\1_{\{ \iota_j\le 1\}}
				+
				O (\iota_j^{-1}) \1_{\{ \iota_j > 1\}}
				\right]
				ds
				\\
				& +
				(a-ib) \lambda_j^{-1}
				{\rm{Re}}
				\bigg[
				\int_{-T}^t \frac{\dot p_j(s) e^{-i \gamma_j(t) } }{t-s}
				\left(
				O( \iota_j \langle \ln \iota_j \rangle )
				\1_{\{\iota_j \le 1 \} }
				+
				O (\iota_j^{-1})
				\1_{\{\iota_j > 1 \} }
				\right)
				ds   \bigg]
				+  C_{p1} p_j^{-1} \dot{p}_j
				+ C_{p2} \la_j^{-1}\dot\la_j.
			\end{aligned}
		\end{equation}
		
		By plugging \eqref{qd24July12-10} and \eqref{M0-ortho} into
		\eqref{qd240803-1}, we can re-write the first equation in \eqref{ortho-eq} as \eqref{orth0-eq-r}.

		\textbf{Terms in $c_1^{\J}(\tau_j(t))$.}
		See \eqref{scalr-Z}, $\mathcal{Z}_{1,1}(\rho_j)   = \frac{1}{\rho_j^2+1}$. By \eqref{Hj-1-def}, one has
		\begin{equation}\label{m1-integ}
			\begin{aligned}
				&
				\int_{0}^{\infty}
				\big( \mathcal{H}^{\J}_1
				+
				\mathcal{H}_{ \rm{in} }^{\J}  \big)_{\mathbb{C}_j,1}(\rho_j,t)  \mathcal{Z}_{1,1}(\rho_j) \rho_j d \rho_j
				\\
				= \ &
				\int_{0}^{\infty}
				\Big\{
				\lambda_j^{2}
				\Big( Q_{-\gamma_j}
				\Big[
				(a-bU^{\J}  \wedge)
				\tilde{L}_{U^{\J} }[\Phi_{\rm{out}}]  	\Big] \Big)_{\mathbb{C}_j,1}
				+
				\lambda_j^{2}  M_{1}^{\J}
				+
				(
				\mathcal{H}_{ \rm{in} }^{\J} )_{\mathbb{C}_j,1} \Big\}(\rho_j,t)  \mathcal{Z}_{1,1}(\rho_j) \rho_j d\rho_j,
			\end{aligned}
		\end{equation}
		where we recall $M_1^{\J}$ given in \eqref{M1-def} and get
		\begin{equation}\label{M1-orth}
			\int_0^{\infty} M_1^{\J}(\rho_j,t)
			\mathcal{Z}_{1,1}(\rho_j)  \rho_j d\rho_j
			=
			-
			(\dot{\xi}_1^{\J} - i\dot{\xi}_2^{\J} )
			\la_j^{-1}
			\int_0^{\infty} \frac{2\rho_j}{(\rho_j^2+1)^2}
			d\rho_j
			=
			-
			(\dot{\xi}_1^{\J} - i\dot{\xi}_2^{\J} )
			\la_j^{-1}.
		\end{equation}
		
		By \eqref{m1-integ} and \eqref{M1-orth}, the second equation of \eqref{ortho-eq} can be rewritten as \eqref{orth1-eqr}.
	\end{proof}

\subsection{Linear theory for the non-local reduced equations}

To introduce the space for the parameter function $p_j(t)$,
we recall that the non-local operator $\mathcal B_0$ given in \eqref{qd240802-3} for mode 0 is of the approximate form
\begin{align*}
	\mathcal B_0[p ]
	=   \int_{-T} ^{t-\lambda_*^2}    \frac{\dot p(s)}{t-s}ds\, + O ( |\dot{p}(t)| ).
\end{align*}
For  $\Theta\in (0,1)$, $\varpi\in \R$ and a continuous function $g:[-T,T]\to \mathbb C$, we define the norm
\begin{equation*}
	\|g\|_{\Theta,\varpi} = \sup_{t\in [-T,T]} \, (T-t)^{-\Theta} |\ln(T-t)|^{\varpi} |g(t)| ,
\end{equation*}
and for $\alpha \in (0,1)$, $\tilde{m}$, $\varpi \in \R$, we define the semi-norm
\begin{equation*}
	[ g]_{\frac{\alpha}2,\tilde{m},\varpi} = \sup_{-T\le s <t \le T, \ t-s \le (T-t)/4} (T-t)^{-\tilde{m}}  |\ln(T-t)|^{\varpi} \frac{|g(t)-g(s)|}{(t-s)^{\alpha/2 }}.
\end{equation*}

The following proposition proved in \cite[Proposition 6.5, Proposition 6.6]{17HMF} gives an approximate inverse of the non-local operator $\mathcal B_0$ with a small remainder $\mathcal{R}_0$.
\begin{prop}\label{keyprop}
	Let $\alpha_0 ,  \frac{\alpha}2 \in (0,\frac{1}{2})$, $\varpi\in \R$, $C_1 >1$. There exists $\flat>0$ such that if $\Theta\in(0,\flat)$ and $\tilde{m} \leq\Theta - \frac{\alpha}2$, then for $h(t) :[0,T]\to \mathbb C$ satisfying
	\begin{equation}\label{hypA00}
	C_1^{-1} \le | h(T) | \le C_1,
	\quad
	 T^\Theta |\ln T|^{1+\sigma -\varpi} \| h(\cdot) - h(T) \|_{\Theta,\varpi-1}
			+ [h]_{\frac{\alpha}2,\tilde{m},\varpi-1}
			\le C_1
	\end{equation}
	for some $\sigma \in (0,1)$ and $T >0$ small enough, there exist two linear operators $\mathcal P $ and $\mathcal{R}_0$ so that $p = \mathcal{P}[h]: [-T,T]\to \mathbb C$ satisfies
	\begin{equation*}
		\mathcal{B}_0[p](t)
		= h(t) + \mathcal{R}_0[h](t) , \quad t \in [0,T]
	\end{equation*}
	with
	\begin{equation*}%\label{calR-est}
		 |\mathcal{R}_0[h](t) | \leq  C
		\Bigl( T^{\sigma_1}
		+ T^\Theta  \frac{\ln |\ln T|}{|\ln T|}  \| h(\cdot) - h(T) \|_{\Theta,\varpi-1}
		+ [h]_{\frac{\alpha}2,\tilde{m},\varpi-1} \Bigr)
		\frac{(T-t)^{\tilde{m}+\frac{(1+\alpha_0 ) \alpha}2}}{  |\ln(T-t)|^{\varpi}}
	\end{equation*}
for some $\sigma_1>0$. Moreover,
\begin{equation*}
	\mathcal{P}[h] = p_{0,\kappa}[h] + \mathcal{P}_1[h] + \mathcal P_2[h],
\end{equation*}
where $p_{0,\kappa}[h] $ is the leading term of $\mathcal{P}[h]$ with
\begin{align}
%\begin{equation}
%	\begin{aligned}
&
p_{0,\kappa}[h](t) =
\kappa |\ln T|
\int_t^T |\ln(T-s)|^{-2}
ds
=
\kappa \left( 1+ O(|\ln T|^{-1})\right) |\ln T|
(T-t)|\ln(T-t)|^{-2},
\nonumber
\\
	&
\kappa =\kappa[h] = h(T) \left( 1+ O(|\ln T|^{-1})\right),
\quad
|\pp_{t}\mathcal{P}_1[h](t) | \le C \frac{|\ln T|^{1-\sigma} \left( \ln (|\ln T|)\right)^2}{|\ln(T-t)|^{3-\sigma}},
\nonumber
\\
&
|\pp_{t}^2\mathcal{P}_1[h](t)| \le C\frac{|\ln T|}{|\ln(T-t)|^3 (T-t)},
\quad
\|\pp_{t}\mathcal{P}_2[h] \|_{\Theta,\varpi}
\le C
\Big(
T^{\frac{1}{2}+\sigma-\Theta}
+
\| h(\cdot) -h(T) \|_{\Theta,\varpi-1}
\Big),
\nonumber
\\
&
[\pp_{t}\mathcal{P}_2[h] ]_{\frac{\alpha}{2},\tilde{m},\varpi}
\le C
\Big(
|\ln T|^{\varpi-3} T^{\flat -\tilde{m}-\frac{\alpha}{2}}
+
T^{\Theta} |\ln T|^{-1} \ln|\ln T|
\| h(\cdot) -h(T) \|_{\Theta,\varpi-1}
+
[h]_{\frac{\alpha}{2},\tilde{m},\varpi-1}
\Big).
\label{p-est}
%	\end{aligned}
%\end{equation}
\end{align}

\end{prop}

We now impose constraints on the parameters such that we can apply Proposition \ref{keyprop} to provide a linear mapping to $p_j$ for \eqref{orth0-eq-r} with $h = {\rm DC}_j [\Phi_{\rm{out}}](t)$.
The vanishing and H\"older properties in \eqref{hypA00} are exactly the ones inherited from the weighted topology \eqref{out-topo} for the outer problem, namely
 \begin{equation*}
\begin{aligned}
	&
|{\rm DC}_j[\Phi_{\rm{out}}](t)- {\rm DC}_j [\Phi_{\rm{out}}](T)|\lesssim \la_*^{\Theta}(t)+(T-t)^{\frac{\alpha}{2}}\|Z_*\|_{C^3(\R^2)},
\\
&
\frac{|{\rm DC}_j [\Phi_{\rm{out}}](t)-
	{\rm DC}_j[\Phi_{\rm{out}}](s)|}{|t-s|^{\alpha/2}}\lesssim \lambda_*(t)^{\Theta-\alpha(1-\beta)}
	+
	\| Z_* \|_{C^3(\R^2)} \mbox{ \ for \ } |t-s|<\frac{T-t}{4}.
\end{aligned}
\end{equation*}
In order for both $\| {\rm DC}_j[\Phi_{\rm{out}}](\cdot) - {\rm DC}_j[\Phi_{\rm{out}}](T) \|_{\Theta,\varpi-1}$, $[{\rm DC}_j[\Phi_{\rm{out}}]]_{\frac{\alpha}2,\tilde{m},\varpi-1}$ to be finite, we need
\begin{equation}\label{qd24July05-1}
\varpi-1-2\Theta<0,
\quad
\Theta< \alpha/2,
\quad
		\tilde{m} < \min\left\{\Theta-\alpha(1-\beta), 0 \right\}.
\end{equation}

We put the remainder $\mathcal R_0[{\rm DC}_j[\Phi_{\rm{out}}]]$ in the non-orthogonal inner problem \eqref{inner-eq-2}. For the gluing to work, suitable parameters will be chosen such that $\mathcal R_0[{\rm DC}_j[\Phi_{\rm{out}}]]$ has fast time decay.

\section{Linear theory for the outer problem}\label{sec-linearouter}

\subsection{$\mathsf{DMO_x}$, $\mathsf{|DMO|_x}$ spaces, and regularity results}\label{DMO-sec}

Given a vector-valued function $\mathbf{f}$ defined in $Q :=  \Omega \times (t_0,t_1) \subset \mathbb{R}^{d+1}$,
for $X = (x,t) \in Q$, $ (B_r(x) \cap \Omega) \times (t-r^2, t) \subset Q$, we define
\begin{align*}
%\begin{equation*}
%\begin{aligned}
&
	\omega_{\mathbf{f}, Q}^{\mathsf x}(r, X):= \fint_{ (B_r(x)  \cap \Omega) \times (t-r^2, t) }  \Big|\mathbf{f}(y,s)- \fint_{B_r(x) \cap \Omega} \mathbf{f}(z,s) dz \Big| dy ds,
\\
&
\omega_{\mathbf{f}}^{\mathsf x}(r, Q):=\sup\Set{\omega_{\mathbf{f}, Q}^{\mathsf x}(r, X) \ | \ X \in Q } \quad \text{and}\quad  \omega_{\mathbf f}^{\mathsf x}(r):=\omega_{\mathbf f}^{\mathsf x}(r, \R^{d+1}).
%\end{aligned}
%\end{equation*}
\end{align*}
We say that $\mathbf f$ is of {\textbf{Dini mean oscillation in $\mathbf{x}$} over $Q$ and write $\mathbf{f} \in \mathsf{DMO_x}(Q)$ if $\omega_{\mathbf f}^{\mathsf x}(r, Q)$ satisfies the Dini condition
$\int_0^1 r^{-1} \omega_{\mathbf f}^{\mathsf x}(r,Q) dr <+\infty$.
Denote the $\mathsf{DMO_x}(Q)$ {\textbf{semi-norm}} as $[ {\mathbf f} ]_{\mathsf{DMO_x}(Q)} := \int_0^1 r^{-1} \omega_{\mathbf f}^{\mathsf x}(r,Q) dr$.
Similarly, for $X = (x,t) \in Q$, $ (B_r(x) \cap \Omega) \times (t-r^2, t) \subset Q$, we define
\begin{align*}
&
	|\omega|_{\mathbf{f}, Q}^{\mathsf x}(r, X):= \fint_{ (B_r(x) \cap \Omega) \times  (t-r^2, t)}  \fint_{ B_r(x) \cap \Omega } \left|\mathbf{f}(y,s)-  \mathbf{f}(z,s) \right| dz dy ds,
\\
&
	|\omega|_{\mathbf f}^{\mathsf x}(r, Q):=\sup\Set{ |\omega|_{\mathbf{f}, Q}^{\mathsf x}(r, X) \ | \ X \in Q} \quad \text{and}\quad  |\omega|_{\mathbf f}^{\mathsf x}(r):=|\omega|_{\mathbf f}^{\mathsf x}(r, \R^{d+1}).
\end{align*}
We say that $\mathbf f$ is of \textbf{Dini mean absolute oscillation in $\mathbf{x}$} over $Q$ and write $\mathbf{f} \in \mathsf{|DMO|_x}(Q)$ if $|\omega|_{\mathbf f}^{\mathsf x}(r, Q)$ satisfies the Dini condition $	\int_0^1 r^{-1} |\omega|_{\mathbf f}^{\mathsf x}(r,Q) dr <+\infty$.
Denote the $\mathsf{|DMO|_x}(Q)$ {\textbf{semi-norm}} as $[ {\mathbf f} ]_{\mathsf{|DMO|_x}(Q)} := \int_0^1 r^{-1} |\omega|_{\mathbf f}^{\mathsf x}(r,Q) dr$, and
\begin{equation*}
\| f\|_{(\mathsf{|DMO|_x} \cap L^{\infty})(Q)} := [f]_{ \mathsf{|DMO|_x} (Q)} + \| f\|_{L^{\infty}(Q)}.
\end{equation*}

If $|B_r(x) \cap \Omega| \ge C |B_r(x)|$ with a constant $C\in (0,1)$ for all $x\in \Omega$, it follows that  $|\omega|_{\mathbf{f}, Q}^{\mathsf x}(r, X) \lesssim |\omega|_{\mathbf{f}, \mathbb{R}^{d+1}}^{\mathsf x}(r, X)$ and thus $\mathsf{|DMO|_x}(\mathbb{R}^{d+1}) \subset \mathsf{|DMO|_x}(Q)$.

We present some basic properties about $\mathsf{DMO_x}(Q)$ and $\mathsf{|DMO|_x}(Q)$ in the following lemma.
\begin{lemma}\label{DMO-|DMO|-lem}

	\begin{enumerate}
		\item For $f\in \mathsf{|DMO|_x}(Q)$, then $[f]_{\mathsf{DMO_x}(Q)} \le [f]_{\mathsf{|DMO|_x}(Q)}$, $[ |f| ]_{\mathsf{|DMO|_x}(Q)} \le [f]_{\mathsf{|DMO|_x}(Q)}$.
		\item If $|f(x,t) - f(y,t)| \le C |x-y|^{\alpha_f}$ with constants $C>0$, $0<\alpha_f \le 1$ for all $(x,t),(y,t) \in Q$, then $[f]_{\mathsf{|DMO|_x}(Q)} \le \alpha_f^{-1} C$.
		\item
		For $f,g\in \mathsf{DMO_x}(Q) (\mathsf{|DMO|_x}(Q))$, $c\in \R$, then $[f +g]_{\mathsf{DMO_x}(Q)} \le [f]_{\mathsf{DMO_x}(Q)} + [g]_{\mathsf{DMO_x}(Q)}$, $[c f]_{\mathsf{DMO_x}(Q)} = |c| [ f]_{\mathsf{DMO_x}(Q)}$
		$([f +g]_{\mathsf{|DMO|_x}(Q)} \le [f]_{\mathsf{|DMO|_x}(Q)} + [g]_{\mathsf{|DMO|_x}(Q)}, [c f]_{\mathsf{|DMO|_x}(Q)} = |c| [ f]_{\mathsf{|DMO|_x}(Q)})$.
		\item
		For $f\in \mathsf{|DMO|_x}(Q)$ satisfying $|f|\ge C_1>0$ uniformly in $Q$, then $[\frac{1}{f}]_{\mathsf{|DMO|_x}(Q)}
		\le C_1^{-2} [f]_{\mathsf{|DMO|_x}(Q)} $,
		$[|f|^{\theta}]_{\mathsf{|DMO|_x}(Q)}
		\le \theta C_1^{\theta-1} [f]_{\mathsf{|DMO|_x}(Q)}$ with $0<\theta<1$.
		\item
		For $f,g\in \mathsf{|DMO|_x}(Q) \cap L^{\infty}(Q)$, then $[fg]_{\mathsf{|DMO|_x}(Q)} \le [f]_{\mathsf{|DMO|_x}(Q)} \|g\|_{L^{\infty}(Q) } + [g]_{\mathsf{|DMO|_x}(Q)} \|f\|_{L^{\infty}(Q) }$,
		$[|f|^{\theta}]_{\mathsf{|DMO|_x}(Q)} \le \theta \| f\|_{L^{\infty}(Q)}^{\theta-1} [f]_{\mathsf{|DMO|_x}(Q)} $ with $\theta> 1$.
	\end{enumerate}
	
\end{lemma}
\begin{proof}
	The proof is straightforward by the definition.
\end{proof}
Compared with $\mathsf{DMO_x}(Q)$, $\mathsf{|DMO|_x}(Q)$ has the advantage that the functions in $\mathsf{|DMO|_x}(Q)$ are closed under arithmetic under some weak assumptions.

Denote $Q_r^{-}(X) = Q_r^{-}(x,t) :=  B_r(x) \times (t-r^2, t) \subset \mathbb{R}^{d+1}$.
\begin{prop}\label{qd240725-6-prop}
	
	Consider the second-order parabolic system
	\begin{equation*}
		\mathbf{u}_t - \sum_{\alpha, \beta=1}^d \mathbf{A}^{\alpha \beta} D_{\alpha \beta} \mathbf{u} +
		\sum_{\alpha=1}^d \mathbf{B}^{\alpha} D_{\alpha} \mathbf{u}
		+ \mathbf{C} \mathbf{u} =\mathbf{g}  \mbox{ \ in \ } Q_2^{-}(0),
	\end{equation*}
	where $\mathbf{u} =(u_1, u_2,\dots, u_m)$, $\mathbf{A}^{\alpha \beta} = (A^{\alpha \beta}_{ij}(x,t) )_{i,j=1}^m$, $\mathbf{B}^{\alpha} = ( \mathbf{B}^{\alpha}_{ij}(x,t) )_{i,j=1}^m$,
	$\mathbf{C} = ( \mathbf{C}_{ij}(x,t) )_{i,j=1}^m$, $\mathbf{g} =(g_1, g_2,\dots, g_m)$,
	$\mathbf{A}^{\alpha \beta}, \mathbf{B}^{\alpha}, \mathbf{C} \in {\mathsf{DMO_x}}(Q_2^{-}(0)) \cap L^{\infty}(Q_2^{-}(0))$, and $\mathbf{A}^{\alpha \beta}$ satisfies the
	Legendre-Hadamard ellipticity
	\begin{equation}\label{LH-ellip}
		A^{\alpha \beta}_{ij}(x,t) \xi_{\alpha} \xi_{\beta} \vartheta^i \vartheta^j \ge c_1 |\xi|^2 |\vartheta|^2
	\end{equation}
	with a constant $c_1>0$
	for all $(x,t) \in Q_2^{-}(0)$, $\xi \in \mathbb{R}^d, \vartheta \in \mathbb{R}^m$. Let $\mathbf{u} \in W_{2}^{1,2}(Q_2^{-}(0))$ be the strong solution and
	$\mathbf{u} \in L^\infty(Q_2^{-}(0))$, $\mathbf{g} \in {\mathsf {DMO_x} }(Q_2^{-}(0)) \cap L^{\infty}(Q_2^{-}(0))$. Then
	\begin{equation*}
		\| D^2 \mathbf{u}\|_{L^{\infty}(Q_1^{-}(0))} + \| \partial_{t} \mathbf{u} \|_{L^{\infty}(Q_1^{-}(0))}
		\lesssim
		\| \mathbf{u} \|_{L^{\infty}(Q_2^{-}(0)) } + 	\| \mathbf{g} \|_{L^{\infty}(Q_2^{-}(0)) } +
		[ {\mathbf{g}} ]_{\mathsf{DMO_x}(Q_2^{-}(0))}.
	\end{equation*}
	
\end{prop}

\begin{proof}
	It is a direct application of \cite[Lemma 4.13]{dong21-C0-para} and the $L^p$ estimates.
\end{proof}

\begin{prop}\label{scaling-prop-0722}
	
Consider the second-order parabolic system
\begin{equation*}
	\mathbf{u}_t - \sum_{\alpha, \beta=1}^d \mathbf{A}^{\alpha \beta} D_{\alpha \beta} \mathbf{u} +
	\sum_{\alpha=1}^d \mathbf{B}^{\alpha} D_{\alpha} \mathbf{u}
	+ \mathbf{C} \mathbf{u} =\mathbf{g}  \mbox{ \ in \ } \mathbb{R}^{d+1}.
\end{equation*}
Given $(x_*, t_*) \in \mathbb{R}^{d+1}$, $\rho_* >0$,
denote $\tilde{\mathbf{u}}(z,s) := \mathbf{u}(x_*+ \rho_*  z , t_* + \rho_*^2 s) $,
$ \tilde{\mathbf{g}}(z,s):= \rho_*^2 \mathbf{g}(x_*+ \rho_*  z , t_* + \rho_*^2 s) $, $\tilde{\mathbf{A} }^{\alpha \beta}(z,s) := \mathbf{A}^{\alpha \beta}(x_* +\rho_* z,t_* + \rho_*^2 s)$, $\tilde{\mathbf{B}}^{\alpha}(z,s):=\rho_* \mathbf{B}^{\alpha}(x_* +\rho_* z,t_* + \rho_*^2 s)$, $\tilde{\mathbf{C}}(z,s) := \rho_*^2\mathbf{C}(x_* +\rho_* z,t_* + \rho_*^2 s)$. Suppose that $\tilde{\mathbf{u}}$, $\tilde{\mathbf{g}}$, $\tilde{\mathbf{A} }^{\alpha \beta}$, $\tilde{\mathbf{B}}^{\alpha}$, $\tilde{\mathbf{C}}$ satisfy the assumption in Proposition \ref{qd240725-6-prop}, and
\begin{equation*}%\label{qd240722-7}
	[f]_{ {\mathsf{DMO_x}}(Q_2^{-}(0)) } + \| f \|_{ L^{\infty}(Q_2^{-}(0)) } \le C,
	\quad f=\tilde{\mathbf{A} }^{\alpha \beta},
	\tilde{\mathbf{B}}^{\alpha}, \tilde{\mathbf{C}}
\end{equation*}
with a constant $C$ independent of $x_*, t_*$ and $\rho_*$, then
\begin{equation*}
	\begin{aligned}
		&
		\| (D_{x}^2 \mathbf{u})(x_*+ \rho_*  z , t_* + \rho_*^2 s) \|_{L^{\infty}(Q_1^{-}(0) )} + \| (\partial_t \mathbf{u})(x_*+ \rho_*  z , t_* + \rho_*^2 s) \|_{L^{\infty}(Q_1^{-}(0) )}
		\\
		\lesssim \ &
		\rho_*^{-2}
		\| \mathbf{u}(x_*+ \rho_*  z , t_* + \rho_*^2 s) \|_{L^{\infty}(Q_2^{-}(0) ) } + 	\| \mathbf{g}(x_*+ \rho_*  z , t_* + \rho_*^2 s) \|_{L^{\infty}(Q_2^{-}(0) ) } +  [ \mathbf{g}(x_*+ \rho_*  z , t_* + \rho_*^2 s) ]_{\mathsf{DMO_x}(Q_2^{-}(0))}
	\end{aligned}
\end{equation*}
with ``$\lesssim$'' independent of $x_*, t_*$ and $\rho_*$.
	
\end{prop}

\begin{proof}
	
Since $\partial_s \tilde{\mathbf{u}} - \sum_{\alpha, \beta=1}^d \tilde{\mathbf{A} }^{\alpha \beta} \partial_{z_{\alpha} z_{\beta}} \tilde{\mathbf{u}} +
		\sum_{\alpha=1}^d
		\tilde{\mathbf{B}}^{\alpha} \partial_{z_{\alpha}} \tilde{\mathbf{u}}
		+ \tilde{\mathbf{C}} \tilde{\mathbf{u}} = \tilde{\mathbf{g}}$,
the conclusion is a direct application of Proposition \ref{qd240725-6-prop}.
\end{proof}

\subsection{Fundamental solution for the outer problem}\label{fund-out-sec}

\begin{prop}\label{funda-prop}
	
Consider the second-order parabolic system
\begin{equation}\label{para-sys}
	\mathbf{u}_t  = \sum_{\alpha, \beta=1}^d \mathbf{A}^{\alpha \beta} D_{\alpha \beta} \mathbf{u}  \mbox{ \ in \ } \mathbb{R}^{d+1},
\end{equation}
where $\mathbf{u} =(u_1, u_2,\dots, u_m)$, $\mathbf{A}^{\alpha \beta} = (A^{\alpha \beta}_{ij}(x,t) )_{i,j=1}^m$, $[\mathbf{A}^{\alpha \beta}]_{ \mathsf{DMO_x}(\R^{d+1}) } + \| \mathbf{A}^{\alpha \beta} \|_{L^{\infty}(\R^{d+1}) } \le \Lambda $ for a constant $\Lambda>0$ and Legendre-Hadamard ellipticity \eqref{LH-ellip} for all $(x,t) \in \R^{d+1}$. Then the parabolic system \eqref{para-sys} has a fundamental solution $\Gamma(x,t,y,s)$ satisfying
\begin{equation*}
\Big(\pp_{t} - \sum_{\alpha, \beta=1}^d \mathbf{A}^{\alpha \beta} D_{\alpha \beta} \Big)\Gamma(\cdot,\cdot,y,s)=0
\mbox{ \ in \ } \R^d \times (s,\infty),
\quad
\lim\limits_{t\rightarrow s^{+} } \Gamma(\cdot,t,y,s) =\delta_y(\cdot) \mbox{ \ on \ } \R^d.
\end{equation*}
Moreover, for any $\delta \in(0,1)$, there exists  a universal constant $c>0$ such that  for $ 0<t-s \le 1$,
\begin{equation}\label{Ga-est}	
\begin{aligned}
&
		(t-s)
		\left( |\pp_{t} \Gamma(x,t,y,s)|
		+
		|D_x^2 \Gamma(x,t,y,s)|
		\right)
		+
		(t-s)^{\frac 12} |D_x \Gamma(x,t,y,s)|
		+
		|\Gamma(x,t,y,s)|
\\
		\le \ &  C (d,c_1,\Lambda,\omega_{\mathbf A}^{\mathsf x},\delta) (t-s)^{-\frac d2}
		e^{ -c  \big(\frac{|x-y|}{\sqrt{t-s}} \big)^{2-\delta} };
\end{aligned}
\end{equation}
for $s < t_1<t_2\le s+1$, $x_1,x_2 \in \R^d$, $\alpha \in (0,1)$,
\begin{equation}\label{Ga-Holder}
	\frac{ |  \Gamma(  x_1, t_1,y,s ) -  \Gamma(x_2, t_2, y,s )  | }{  \big(| x_1 - x_2| + \sqrt{| t_1- t_2|} \big)^{\alpha} }
	\le
 C (\alpha,d,c_1,\Lambda,\omega_{\mathbf A}^{\mathsf x},\delta)
	(t_2-s)^{-\frac{\alpha}{2} }
	\Big[
	(t_1-s)^{-\frac {d}2}
	e^{ -c  \big(\frac{|x_1-y|}{\sqrt{t_1-s}} \big)^{2-\delta} }
	+
	(t_2-s)^{-\frac {d}2}
	e^{ -c  \big(\frac{|x_2-y|}{\sqrt{t_2-s}} \big)^{2-\delta} }
	\Big],
\end{equation}
\begin{equation}\label{nabla-Ga-Holder}
\begin{aligned}
&
		\frac{ | ( D_x \Gamma) (  x_1, t_1,y,s ) - ( D_x \Gamma) (x_2, t_2, y,s )  | }{  \big(| x_1 - x_2| + \sqrt{| t_1- t_2|} \big)^{\alpha} }
		\\
		\le \ &  C (\alpha,d,c_1,\Lambda,\omega_{\mathbf A}^{\mathsf x},\delta)
		(t_2-s)^{-\frac{\alpha}{2} }
		\Big[
		(t_1-s)^{-\frac {d+1}2}
		e^{ -c  \big(\frac{|x_1-y|}{\sqrt{t_1-s}} \big)^{2-\delta} }
		+
		(t_2-s)^{-\frac {d+1}2}
		e^{ -c  \big(\frac{|x_2-y|}{\sqrt{t_2-s}} \big)^{2-\delta} }
		\Big].
\end{aligned}
\end{equation}
Furthermore, for $f\in C^3_0(\R^d)$ satisfying $\mathsf{supp} f \subset B_{C_1}$ with a constant $C_1>0$, denote
$\left(\Gamma*f\right)(x,t):=\int_{\R^d} \Gamma(x,t,y,s) f(y) dy$. Then for $ 0<t-s \le 1$, $s < t_1<t_2\le s+1$, $x_1,x_2 \in \R^d$,
\begin{equation}\label{Cauchy-est-o}
\begin{aligned}
&
	\left|\Gamma*f \right|
	+
\left|D_x \left(\Gamma*f\right)  \right| + \left|D^2_x \left(\Gamma*f\right)  \right|
+ \left|\pp_{t} \left(\Gamma*f\right)  \right|
\lesssim C (C_1,d,c_1,\Lambda,\omega_{\mathbf A}^{\mathsf x},\delta) \| f\|_{C^3 (\R^d)},
\\
&
\frac{ \left| D_x \left(\Gamma* \mathbf{f} \right)  (  x_1, t_1 ) -  D_x \left(\Gamma* \mathbf{f} \right) (x_2, t_2  )  \right| }{  \big(| x_1 - x_2| + \sqrt{| t_1- t_2|} \big)^{\alpha} }
\lesssim
C (\alpha,C_1,d,c_1,\Lambda,\omega_{\mathbf A}^{\mathsf x},\delta) \|\mathbf{f}\|_{C^3(\R^d)}.
\end{aligned}
\end{equation}

\end{prop}

\begin{proof}
The existence of the fundamental solution is a generalization of \cite[Theorems 1.1, 1.3]{dong22-non-divergence} to the parabolic system \eqref{para-sys}.
Indeed,
$W_p^{2,1}$ estimates for parabolic systems are given in \cite{dong11-higherLp}, which can be used to generalize  \cite[Lemma 2.2]{dong22-non-divergence} to parabolic systems. The results \cite[Lemma 2.3]{dong22-non-divergence} and \cite[Theorem 3.3]{dong21-C0-para} can also be generalized to parabolic systems.

 Given $(x_*,t_*) \in \mathbb{R}^{d+1}$, $\rho_* \le 1$, $[\mathbf{A}^{\alpha \beta}(x_* +\rho_* z, t_* +\rho_*^2 \tau) ]_{ \mathsf{DMO_x}(\R^{d+1}) } \le \Lambda$. By a generalized version of \cite[Theorem 1.3]{dong22-non-divergence}, \cite[Theorem 3.2]{dong21-C0-para} for parabolic systems, and the scaling argument similar to Proposition \ref{scaling-prop-0722}, the validity of \eqref{Ga-est} follows,
where
	$C (d,c_1,\Lambda,\omega_{\mathbf A}^{\mathsf x},\delta)$, $c>0$ will vary from line to line. The constant ``$C$'' in the proof depends on $d,c_1,\Lambda,\omega_{\mathbf A}^{\mathsf x},\delta$, and for simplicity, we will not stress this dependence by writing these explicitly.
	
	For any $(x_*,t_*),(y,s)\in \R^{d+1}$, $s<t_* \le s+1$, set $\rho_* = (t_*-s)^{\frac 12}$. Consider $\Gamma(x_* + \rho_{*} z, t_* + \rho^2_{*}\tau,y,s ) $ as a function of $z,\tau$. For $p>d+2$, set $\alpha_1=1-\frac{d+2}{p}$. Then
	\begin{align*}
			&
			\rho_{*}^{-1-\alpha_1} \| \Gamma(x_* + \rho_{*} z, t_* + \rho^2_{*}\tau,y,s ) \|_{L^p(   B(0,\frac 1{2}) \times (-\frac 14,0)   )}
			\\
			\lesssim \ &
			(t_*-s)^{-\frac {1+\alpha_1}2}
			\|
			(t_* + \rho_{*}^2 \tau -s)^{-\frac d2}
			e^{ -c  \big(\frac{|x_* + \rho_{*} z -y|}{\sqrt{t_* + \rho_{*}^2 \tau -s}} \big)^{2-\delta} }
			\|_{L^p(   B(0,\frac 1{2}) \times (-\frac 14,0)   )}
			\\
			\lesssim \ &
			\begin{cases}
				(t_*-s)^{-\frac {d+1+\alpha_1}2}
				&
				\mbox{ \ if \ } |x_*-y|\le (t_*-s)^{\frac 12}
				\\
				(t_*-s)^{-\frac {d+1+\alpha_1}2}
				e^{ -c  (\frac{|x_* -y|}{\sqrt{t_* -s}} )^{2-\delta} }
				&
				\mbox{ \ if \ } |x_*-y| > (t_*-s)^{\frac 12}
			\end{cases}
			\sim
			(t_*-s)^{-\frac {d+1+\alpha_1}2}
			e^{ -c  (\frac{|x_* -y|}{\sqrt{t_* -s}} )^{2-\delta} },
		\end{align*}
	where we used
	\begin{equation*}
		\frac 34 (t_* -s) \le t_* + \rho_{*}^2 \tau -s \le t_* -s,
\quad
		\frac{|x_* + \rho_{*} z -y|}{\sqrt{t_* + \rho_{*}^2\tau -s}}
		\begin{cases}
			\lesssim  1 &
			\mbox{ \ if \ } |x_*-y|\le (t_*-s)^{\frac 12}
			\\
			\sim \frac{|x_* -y|}{\sqrt{t_* -s}}
			&
			\mbox{ \ if \ } |x_*-y| > (t_*-s)^{\frac 12} .
		\end{cases}
	\end{equation*}
Similarly,
	\begin{equation*}
		\begin{aligned}
			&
			\rho_{*}^{-\alpha_1}   \|  (D_x \Gamma)(x_* + \rho_{*} z, t_* + \rho_{*}^2 \tau,y,s )  \|_{L^p(   B(0,\frac 1{2}) \times (-\frac 14,0)   )}
			+
			\rho^{1-\alpha_1}_{*} \|
			(D_x^2\Gamma)(x_* + \rho_{*} z, t_* + \rho^2_{*}\tau,y,s )
			\|_{L^p(   B(0,\frac 1{2}) \times (-\frac 14,0)   )}
			\\
			&
			\quad+
			\rho^{1-\alpha_1}_{*} \|
			(\pp_{t} \Gamma)(x_* + \rho_{*} z, t_* + \rho^2_{*}\tau,y,s )
			\|_{L^p(   B(0,\frac 1{2}) \times (-\frac 14,0)   )}
			\lesssim
			(t_*-s)^{-\frac {d+1+\alpha_1}2}
			e^{-c  (\frac{|x_* -y|}{\sqrt{t_* -s}} )^{2-\delta} }.
		\end{aligned}
	\end{equation*}
By the Sobolev embedding theorem (see \cite[Lemma 2.1]{dong21-C0-para}  for instance),
	\begin{equation*}
		\sup\limits_{ x_1,x_2\in B(x_*,\frac { (t_*-s)^{\frac 12} }{2}), t_1,t_2\in (t_* -\frac { t_*-s }4, t_*)}
		\frac{ | ( D_x \Gamma) (  x_1, t_1,y,s ) - ( D_x \Gamma) (x_2, t_2, y,s )  | }{ (| x_1 - x_2| + \sqrt{| t_1- t_2|} )^{\alpha_1}}
		\lesssim C(\alpha_1)
		(t_*-s)^{-\frac {d+1+\alpha_1}2}
		e^{ -c  (\frac{|x_* -y|}{\sqrt{t_* -s}} )^{2-\delta} } .
	\end{equation*}
	
	For $(x_1,t_1)\notin B(x_*,\frac { (t_*-s)^{\frac 12} }{2})\times (t_* -\frac { t_*-s }4, t_*)$, by \eqref{Ga-est}, we have
	\begin{equation*}
			\frac{ | ( D_x \Gamma) (  x_1, t_1,y,s ) - ( D_x \Gamma) (x_*, t_*, y,s )  | }{  (| x_1 - x_*| + \sqrt{| t_1- t_*|} )^{\alpha_1} }
			\lesssim
			(t_*-s)^{-\frac{\alpha_1}{2} }
			\Big\{
			(t_1-s)^{-\frac {d+1}2}
			e^{ -c  (\frac{|x_1-y|}{\sqrt{t_1-s}} )^{2-\delta} }
			+
			(t_*-s)^{-\frac {d+1}2}
			e^{ -c  (\frac{|x_*-y|}{\sqrt{t_*-s}} )^{2-\delta} }
			\Big\}  .
	\end{equation*}
	Similarly, we have
	\begin{equation*}
		\frac{ |\Gamma (  x_1, t_1,y,s ) - \Gamma (x_*, t_*, y,s )  | }{  (| x_1 - x_*| + \sqrt{| t_1- t_*|} )^{\alpha_2} }
		\lesssim C(\alpha_2)
		(t_*-s)^{-\frac {d+\alpha_2}2}
		e^{ -c  (\frac{|x_* -y|}{\sqrt{t_* -s}} )^{2-\delta} }
	\end{equation*}
	for $(x_1,t_1)\in B(x_*,\frac { (t_*-s)^{\frac 12} }{2})\times (t_* -\frac { t_*-s }4, t_*)$, where
$ \alpha_2 =
		\begin{cases}
			2-\frac{d+2}{p}
			&
			\mbox{ \ if \ } p<d+2
			\\
			1-\epsilon \text{ for any } \epsilon\in (0,1),
			&
			\mbox{ \ if \ } p\ge d+2.
		\end{cases}$
	
	For $(x_1,t_1)\notin B(x_*,\frac { (t_*-s)^{\frac 12} }{2})\times (t_* -\frac { t_*-s }4, t_*)$, by \eqref{Ga-est}, we get
	\begin{equation*}
		\frac{ |  \Gamma (  x_1, t_1,y,s ) - \Gamma(x_*, t_*, y,s )  | }{  (| x_1 - x_*| + \sqrt{| t_1- t_*|} )^{\alpha_2} }
		\lesssim
		(t_*-s)^{-\frac{\alpha_2}{2} }
		\Big\{
		(t_1-s)^{-\frac {d}2}
		e^{ -c  (\frac{|x_1-y|}{\sqrt{t_1-s}} )^{2-\delta} }
		+
		(t_*-s)^{-\frac {d}2}
		e^{ -c  (\frac{|x_*-y|}{\sqrt{t_*-s}} )^{2-\delta} }
		\Big\}.
	\end{equation*}

By \eqref{Ga-est}, \eqref{basic cal in x}, one has $\left|\Gamma*f \right|\lesssim \| f\|_{L^\infty}$.
By \cite[Theorem 3.2]{dong21-C0-para}, and $W^{1,2}_p$ estimates \cite[Theorem 2]{dong11-higherLp} (where we used $\mathsf{supp} f \subset B_{C_1}$), we conclude the validity of \eqref{Cauchy-est-o}.
\end{proof}

\subsection{Properties of the leading coefficients for the outer problem}\label{sec-DMO}

\begin{prop}\label{Bprop}
	Suppose that $T\ll 1$, $\lambda_j$ given in \eqref{lam-ansatz}, $|\Phi|\ll a$, $\| |\Phi_{\rm {out} }| + |\nabla \Phi_{\rm {out} }| \|_{L^{\infty}(\mathbb{R}^2\times (0,T))}
	\le 1 $,
	$\| \Phi_{\rm{in} }^{\J} \|_{{\rm in},\nu-\delta_0,l } \lesssim 1$, parameter assumption \eqref{inn-top0-para} and
	\begin{equation}\label{Phi-|DMO|-req1}
		\nu -\delta_0>1/2
	\end{equation}
	holds,
	then $\mathbf{B}_{\Phi,U_{*}}$  defined in
	\eqref{B-matrix} satisfies the
	Legendre-Hadamard ellipticity \eqref{LH-ellip} with a constant $c_1$
	and $\|\mathbf{B}_{\Phi,U_{*}} \|_{ (\mathsf{|DMO|_x} \cap L^{\infty} ) (\mathbb{R}^2 \times (0,T)) }$ $ \le C$, where $c_1, C$ are positive constants independent of $T$.
\end{prop}

\begin{proof}
	
	In this proof, we assume $\mathbb{R}_T^2 = \mathbb{R}^2 \times (0,T)$ and all ``$\lesssim$'' are independent of $T$.
	Set $ \mathbf{A}^{11} =  \mathbf{A}^{22} =
	a\mathbf{I}_3 -b U_{*} \wedge  $,
	and $\mathbf{A}^{\alpha \beta}=0$ for all $(\alpha,\beta)\ne (1,1), (2,2)$.
	For $\fbf=(f_1,f_2) \in \R^2$, $\gbf \in \R^3$,
	\begin{equation*}
		\sum_{\alpha, \beta=1}^{2}
		\gbf^{\tr}  \mathbf{A}^{\alpha\beta} f_{\alpha} f_{\beta} \gbf
		=
		\gbf^{\tr}
		[
		|\fbf|^2  (a\mathbf{I}_3 -b U_{*} \wedge )
		\gbf
		]
		=
		| \fbf |^2 \left(
		a|\gbf|^2 - b \gbf^{\tr}
		[  U_{*} \wedge
		\gbf
		] \right)
		=
		a |\fbf|^2 |\gbf|^2.
	\end{equation*}
	By \eqref{til-B-matrix}, \eqref{A-est}, then $|\tilde{ \mathbf{B} }_{\Phi, U_{*}}| = O(\lambda_* + |\Phi|)$. Taking $T \ll 1$ and $|\Phi| \ll a$, one has that $\mathbf{B}_{\Phi,U_{*}}$ satisfies \eqref{LH-ellip} with a constant $c_1>0$ independent of $T$. Next we will prove $	\| \mathbf{B}_{\Phi,U_{*}} \|_{(\mathsf{|DMO|_x} \cap L^{\infty} )(\mathbb{R}_T^2)} \lesssim 1$.
	Recall $U^{\J}$ given in \eqref{def-U1U2}. Obviously, $|U^{\J}| \equiv 1$.
	Since
	\begin{equation*}
		\frac{2y^{\J}}{|y^{\J}|^2+1}
		=
		\frac{2\lambda_j(t)(x-\xi^{\J}(t))}{ |x-\xi^{\J}(t)|^2+ \lambda_j^2(t) }, \quad
		\frac{|y^{\J}|^2-1}{|y^{\J}|^2+1}
		=
		1-
		\frac{ 2\lambda_j^2(t) }{ |x-\xi^{\J}(t)|^2+ \lambda_j^2(t) },
	\end{equation*}
	in order to get  $[U^{\J}]_{\mathsf{|DMO|_x}(\R^{2}_T)} \lesssim 1$, it suffices to prove
	\begin{equation}\label{U_*-|DMO|_x}
		[f]_{\mathsf{|DMO|_x}(\R^{2}_T)} \lesssim 1 \mbox{ \ for \ } f=	\frac{ \lambda_j^2(t) }{ |x-\xi^{\J}(t)|^2+ \lambda_j^2(t) }, \quad 	\frac{ \lambda_j(t)(x-\xi^{\J}(t))}{ |x-\xi^{\J}(t)|^2+ \lambda_j^2(t) }.
	\end{equation}
	\textbf{Proof of \eqref{U_*-|DMO|_x}.}
	\begin{align*}
		%	\begin{equation*}
			%		\begin{aligned}
				&
				\bigg|
				\frac{ \lambda_j^2(s) }{ |w-\xi^{\J}(s)|^2+ \lambda_j^2(s) }
				-
				\frac{ \lambda_j^2(s) }{ |z-\xi^{\J}(s)|^2+ \lambda_j^2(s) }
				\bigg|
				=
				  \frac{ \lambda_j^2(s) \left| | z-\xi^{\J}(s) |^2 - | w-\xi^{\J}(s) |^2  \right| }{ \big( | w-\xi^{\J}(s) |^2 + \lambda_j^2(s) \big) \big( | z-\xi^{\J}(s) |^2 + \lambda_j^2(s) \big)  }  	
				\\
				\le \ &
				|w-z| \lambda_j^2(s)  \frac{  \left| z-\xi^{\J}(s) \right| + \left| w-\xi^{\J}(s) \right| }{ \big( | w-\xi^{\J}(s) |^2 + \lambda_j^2(s) \big) \big( | z-\xi^{\J}(s) |^2 + \lambda_j^2(s) \big)  }  	
				\\
				\lesssim \ &
				|w-z| \lambda_j(s)
				\big[
				\big( \left| w-\xi^{\J}(s) \right|^2 + \lambda_j^2(s) \big)^{-1}
				+
				\big( \left| z-\xi^{\J}(s) \right|^2 + \lambda_j^2(s) \big)^{-1}	
				\big].
				%		\end{aligned}
			%	\end{equation*}
	\end{align*}
	Then for any $Q_r^{-}(X) \subset \R^{2}_T$,
	\begin{align}
		%	\begin{equation}
			%		\begin{aligned}
				&
				\fint_{ Q_r^{-}(X)}  \fint_{ B_r(x) }
				\bigg| 	 \frac{\lambda_j^2(s) }{ \left| w-\xi^{\J}(s) \right|^2 + \lambda_j^2(s) }
				-
				\frac{ \lambda_j^2(s) }{ \left| z-\xi^{\J}(s) \right|^2 + \lambda_j^2(s) } \bigg| dz  dwds
				\nonumber
				\\
				\lesssim \ &
				r
				\fint_{ Q_r^{-}(X)}   \fint_{ B_r(x) }   \lambda_j(s)
				\Big[
				\big( \left| w-\xi^{\J}(s) \right|^2 + \lambda_j^2(s) \big)^{-1}
				+
				\big( \left| z-\xi^{\J}(s) \right|^2 + \lambda_j^2(s) \big)^{-1}	
				\Big]   dz dw ds
				\nonumber
				\\
				\sim \ &
				r
				\fint_{t-r^2}^t   \fint_{ B_r(x) }   \lambda_j(s)
				\big( \left| z-\xi^{\J}(s) \right|^2 + \lambda_j^2(s) \big)^{-1}	
				dz ds
				\le
				r
				\fint_{t-r^2}^t   \fint_{ B_r(0) }   \lambda_j(s)
				\big( \left| z \right|^2 + \lambda_j^2(s) \big)^{-1}	
				dz ds
				\nonumber
				\\
				\sim \ &
				r^{-3}
				\int_{t-r^2}^t \lambda_j(s)   \int_{0 }^r
				\left( v^2 + \lambda_j^2(s) \right)^{-1}	
				v dv ds
				\sim
				r^{-3}
				\int_{t-r^2}^t \lambda_j(s)
				\ln \big(1+ \lambda_j^{-2}(s) r^2 \big) ds.
				\label{Z1Z2}
				%		\end{aligned}
			%	\end{equation}
	\end{align}
	It suffices to prove that the following integral is bounded.
	\begin{align*}
			&
			\int_0^1
			r^{-4}
			\int_{t-r^2}^t
			\lambda_j(s)
			\ln \left(1+ \lambda_j^{-2}(s) r^2 \right)
			ds dr
			=
			\int_{0}^t
			\int_0^1
			r^{-4}
			\lambda_j(s)
			\ln \left(1+  \lambda_j^{-2}(s) r^2 \right)
			\1_{ \{ r\ge (t-s)^{\frac 12} \}}
			dr ds
			\\
			= \ &
			\Big( \int_{[t-(T-t)]_+}^t + \int_{0}^{[t-(T-t)]_+} \Big)
			\lambda_j^{-2}(s)
			\int_{\frac{(t-s)^{\frac 12}}{\lambda_j(s)}}^{\frac{1}{\lambda_j(s)}}
			z^{-4} \ln(1+z^2) dz ds .
		\end{align*}
	Recall $\lambda_j$ in \eqref{lam-ansatz}.	For the first part, since $T-t\le T-s\le 2(T-t)$,
	there exist constants  $c_1, c_2>0$ such that
	\begin{small}
		\begin{equation*}
			\begin{aligned}
				&
				\int_{[t-(T-t)]_+}^t
				\lambda_j^{-2}(s)
				\int_{\frac{(t-s)^{\frac 12}}{\lambda_j(s)}}^{\frac{1}{\lambda_j(s)}}
				z^{-4} \ln(1+z^2) dz ds
				\lesssim
				\Big( \int_{[t-(T-t)]_+}^{[t-\lambda_j^2(t)]_+} +  \int_{[t-\lambda_j^2(t)]_+}^t \Big)
				\lambda_j^{-2}(t)
				\int_{c_2 \frac{(t-s)^{\frac 12}}{\lambda_j(t)}}^{\frac{c_1}{\lambda_j(t)}}
				z^{-4} \ln(1+z^2) dz ds
				\\
				\lesssim \ &
				\int_{[t-(T-t)]_+}^{[t-\lambda_j^2(t)]_+}
				\lambda_j^{-2}(t)
				\Big[\frac{(t-s)^{\frac 12}}{\lambda_j(t)} \Big]^{-3} \ln\Big(1+\Big[\frac{(t-s)^{\frac 12}}{\lambda_j(t)} \Big]^2\Big)  ds + 1
				\lesssim
				\int_{1}^{\frac{(T-t)^{\frac 12}}{\lambda_j(t)}}
				y^{-2} \ln(1+y^2) dy + 1
				\lesssim
				1.
			\end{aligned}
		\end{equation*}
	\end{small}
	For the second part, since $(T-s)/2 \le t-s \le T-s$,
	\begin{align*}
			&
			\int_{0}^{[t-(T-t)]_+}
			\lambda_j^{-2}(s)
			\int_{\frac{(t-s)^{\frac 12}}{\lambda_j(s)}}^{\frac{1}{\lambda_j(s)}}
			z^{-4} \ln(1+z^2) dz ds
			\le
			\int_{0}^{[t-(T-t)]_+}
			\lambda_j^{-2}(s)
			\int_{\frac{(T-s)^{\frac 12}}{\sqrt{2} \lambda_j(s)}}^{\frac{1}{\lambda_j(s)}}
			z^{-4} \ln(1+z^2) dz ds
			\\
			\lesssim  \ &
			\int_{0}^{[t-(T-t)]_+}
			\lambda_j(s)
			(T-s)^{-\frac 32} \ln\Big(1+\Big[ \frac{(T-s)^{\frac 12}}{\sqrt{2} \lambda_j(s)} \Big]^2 \Big) ds
			\lesssim
			1,
		\end{align*}
where the type II speed $\lambda_j(t) \lesssim (T-t)^{\frac{1}{2} +\epsilon}$ with a constant $0<\epsilon\ll 1$ is essential for the last step. Next for $i=1,2$,
	\begin{align*}
		%	\begin{equation*}
			%		\begin{aligned}
				&
				\bigg|
				\frac{\lambda_j(s)\big(w_i-\xi_i^{\J}(s) \big)}{ \left|w-\xi^{\J}(s)\right|^2 + \lambda_j^2(s) }
				-   \frac{\lambda_j(s)\big(z_i -\xi_i^{\J}(s)\big)}{ \left|z-\xi^{\J}(s)\right|^2 + \lambda_j^2(s) }  \bigg|
				\\
				= \ &
				\lambda_j(s)
				\bigg|
				\frac{ \big(w_i-\xi_i^{\J}(s) \big) \big( \left|z-\xi^{\J}(s)\right|^2 + \lambda_j^2(s)  \big) -
					\big(z_i -\xi_i^{\J}(s)\big)\big( \left|w-\xi^{\J}(s)\right|^2 + \lambda_j^2(s) \big)}{
					\big( \left|w-\xi^{\J}(s)\right|^2 + \lambda_j^2(s) \big)\big( \left|z-\xi^{\J}(s)\right|^2 + \lambda_j^2(s)  \big)  }
				\bigg|
				\\
				\le \ &
				\left| w-z \right|  \lambda_j(s)
				\frac{ \left|z-\xi^{\J}(s)\right|^2 + \lambda_j^2(s)
					+
					\left|z -\xi^{\J}(s)\right|\left( \left|w-\xi^{\J}(s)\right| + \left|z-\xi^{\J}(s)\right| \right)}{
					\big( \left|w-\xi^{\J}(s)\right|^2 + \lambda_j^2(s) \big)\big( \left|z-\xi^{\J}(s)\right|^2 + \lambda_j^2(s)  \big)  }
				\\
				\lesssim \ &
				\left| w-z \right| \lambda_j(s)
				\big[
				\big( \left|w-\xi^{\J}(s)\right|^2 + \lambda_j^2(s) \big)^{-1}
				+
				\big( \left|z-\xi^{\J}(s)\right|^2 + \lambda_j^2(s)  \big)^{-1}
				\big].
				%		\end{aligned}
			%	\end{equation*}
	\end{align*}
	We conclude that $[\frac{ \lambda_j(t)(x-\xi^{\J}(t))}{ |x-\xi^{\J}(t)|^2+ \lambda_j^2(t) } ]_{\mathsf{|DMO|_x}(\R^{2}_T)} \lesssim 1 $ by the same reasoning as \eqref{Z1Z2}. Thus, we deduce \eqref{U_*-|DMO|_x}. It follows that $\| U_{*} \|_{(\mathsf{|DMO|_x} \cap L^{\infty}) (\R^{2}_T)} \lesssim 1 $.

	By Lemma \ref{DMO-|DMO|-lem} (5), the $\| \cdot\|_{(\mathsf{|DMO|_x} \cap L^{\infty}) (\R^{2}_T) }$-norm of
	the multiplicity of finitely many terms of the components of $U_{*}$ is finite.
	By Lemma \ref{DMO-|DMO|-lem} $(2)$ and  \eqref{Phi*-0-j-upp}, then $\| \sum\limits_{j=1}^{N} \Phi_{0}^{*\J} \|_{(\mathsf{|DMO|_x} \cap L^{\infty}) (\R^{2}_T)} \lesssim 1$; similarly, we have $\| g\|_{(\mathsf{|DMO|_x} \cap L^{\infty}) (\R^{2}_T)} \lesssim 1$ for  $g=\Phi_{\rm{out}},~ \eta_{d_q}^{\J} $ since $\| |\Phi_{\rm {out} }| + |\nabla \Phi_{\rm {out} }| \|_{L^{\infty}(\mathbb{R}^2_T)}
	\le 1 $.

	By $\| \Phi_{\rm{in} }^{\J} \|_{{\rm in},\nu-\delta_0,l } \lesssim 1$ and parameter assumption \eqref{inn-top0-para}, we have
	\begin{small}
		\begin{align*}
			%\begin{equation*}
			%	\begin{aligned}
				&
				\fint_{Q_r^{-}(X) }  \fint_{B_r(x)} \Big|
				\eta\Big(\frac{x-\xi^{\J}(s)}{\la_*(s)R(s)}\Big)
				Q_{\gamma_j(s)}\Phi_{\rm in}^{\J}\Big( \frac{x-\xi^{\J}(s)}{\lambda_j(s) } ,s\Big)
				-
				\eta\Big(\frac{z-\xi^{\J}(s)}{\la_*(s)R(s)}\Big) Q_{\gamma_j(s)}\Phi_{\rm in}^{\J}\Big( \frac{z-\xi^{\J}(s)}{\lambda_j(s) } ,s\Big)
				\Big| dx dz  ds
				\\
				\lesssim \ &
				r^{-2}\int_{t-r^2}^t r \lambda_*^{\nu-\delta_0-1}(s) ds
				\sim
				r^{-1} |\ln T|^{\nu-\delta_0-1}
				\int_{T-t}^{T-t+r^2} \frac{ z^{\nu-\delta_0-1}} {|\ln z|^{2\nu-2\delta_0-2} } dz
				\lesssim
				|\ln T|^{\nu-\delta_0-1}
				\frac{ r^{2\nu-2\delta_0-1}} {|\ln r|^{2\nu-2\delta_0-2} },
				%	\end{aligned}
			%\end{equation*}
		\end{align*}
	\end{small}
which is a Dini function under the assumption \eqref{Phi-|DMO|-req1}. And the corresponding $[\cdot]_{ \mathsf{|DMO|_x}(\R^{2}_T) } \lesssim 1$ since $\nu-\delta_0 <1$ by \eqref{inn-top0-para}.
	
	In sum, for $\Phi$ given in \eqref{u-def}, by Lemma \ref{DMO-|DMO|-lem},  under the assumptions \eqref{inn-top0-para} and \eqref{Phi-|DMO|-req1}, we have
	\begin{equation}\label{Phi in |DMO|x}
		\| \Phi \|_{(\mathsf{|DMO|_x} \cap L^{\infty}) (\R^{2}_T)} \lesssim 1.
	\end{equation}
	For $A$ given in \eqref{A-def}, by  \eqref{U*-norm}, \eqref{Phi in |DMO|x}, $|\Phi|\ll 1$ and Lemma \ref{DMO-|DMO|-lem}, we have $\| A \|_{(\mathsf{|DMO|_x} \cap L^{\infty}) (\R^{2}_T)} \lesssim 1
	$. By \eqref{A-est}, $\lambda_*$ in \eqref{lam-ansatz} and $|\Phi|\ll 1$, then $|A|\ll 1$.  By the similar argument, $\| \mathbf{B}_{\Phi,U_{*}} \|_{(\mathsf{|DMO|_x} \cap L^{\infty}) (\R^{2}_T)} \lesssim 1
	$.
\end{proof}

\section{Completion of the construction}

\subsection{Proof of Theorem \ref{thm}}\label{sol-glu+orth-sec}

{\textbf{Step 1.}}
$Z_*(x)$ is the leading part of the initial value of the outer problem \eqref{outer-eq}. For $Z_*(x)$ given in \eqref{Z*-def}, ${\rm DC}_j [Z_*]$ in \eqref{R0-def} is independent of $t$ and satisfies \eqref{hypA00}. By Proposition \ref{keyprop}, as the leading term of $p_{j}$, $p_{j0} = \mathcal{P}[{\rm DC}_j[Z_*]]$ satisfies
\begin{equation}\label{qd240729-3}
	\begin{aligned}
		&
		\mathcal{B}_0[p_{j0}](t) =
		{\rm DC}_j[Z_*]  +
		\mathcal{R}_0\left[ {\rm DC}_j [Z_*] \right], \quad t \in [0, T],
		\\
		&
		p_{j0} =  {\rm DC}_j [Z_*] \left( 1+ O(|\ln T|^{-1})\right) \lambda_*,
		\quad
		\dot{p}_{j0} =  - {\rm DC}_j [Z_*] \left( 1+ O(|\ln T|^{-1})\right)
		(T-t)^{-1} \lambda_*.
	\end{aligned}
\end{equation}
Set $p_j = p_{j0} + p_{j1}$, where $p_{j1}$ is the next order term of $p_j$. Denote
\begin{equation*}
	\begin{aligned}
		&
		{\mathbf{p}}_{\cdot 0} := ( p_{10}, p_{20}, \dots, p_{N 0} ),
		\quad
		{\mathbf{p}}_{\cdot 1} := ( p_{11}, p_{21}, \dots, p_{N 1} ),
		\quad
		\dot{\bm{\xi}}^{[\cdot]} := (\dot{\xi}^{[1]}, \dot{\xi}^{[2]}, \dots, \dot{\xi}^{[N]}),
		\\
		&	\mathbf{\Phi_{\rm{in} }^{[\cdot]} } = \big( \Phi_{\rm{in} }^{[1]}, \Phi_{\rm{in} }^{[2]} , \dots, \Phi_{\rm{in} }^{[N]} \big),
		\quad
		\mathcal{B}_{\rm{in}}^{[\cdot]} =
		\big(
		B_{\rm{in}}^{[1]},  B_{\rm{in}}^{[2]}, \dots, B_{\rm{in}}^{[N]}
		\big).
	\end{aligned}
\end{equation*}
Recall the ansatz \eqref{lam-ansatz}. We will solve ${\mathbf{p}}_{\cdot 1}$, $\dot{\bm{\xi}}^{[\cdot]}$ in the following spaces respectively,
\begin{equation*}
	\begin{aligned}
		&
		B_{ {\mathbf{p}}_{\cdot 1} } :=
		\big\{ ( p_{11}, p_{21}, \dots, p_{N 1} ) \ \big| \
		|p_{j1}(t)| + (T-t) |\dot{p}_{j1}(t)|  \le  |\ln T|^{-\frac{1}{2}}
		\lambda_*(t),
		\ t\in [0,T), \ j=1,2,\dots, N
		\big\},
		\\
		&
		B_{\dot{\bm{\xi}}^{[\cdot]}} :=
		\big\{ (\dot{\xi}^{[1]}, \dot{\xi}^{[2]}, \dots, \dot{\xi}^{[N]}) \ \big| \
		|\dot{\xi}^{\J}(t)| \le C_{\xi} \lambda_*^{\epsilon_{\xi} }(t) , \ t\in [0,T), \ j=1,2,\dots, N
		\big\}.
	\end{aligned}
\end{equation*}
Set $\xi^{\J}(t) = q^{\J} + \int_t^T \dot{\xi}^{\J}(s) ds$. Given ${\mathbf{p}}_{\cdot 1} \in B_{ {\mathbf{p}}_{\cdot 1} }, \dot{\bm{\xi}}^{[\cdot]} \in B_{\dot{\bm{\xi}}^{[\cdot]}}$, we have
\begin{equation*}
	p_j
	= {\rm DC}_j [Z_*] \lambda_* \big(1+ O( |\ln T|^{-\frac{1}{2}} ) \big),
	\quad
|p_j| \sim \lambda_*,
\quad
|\dot{p}_{j}| \sim (T-t)^{-1} \lambda_*.
\end{equation*}
Recalling $p_j = \lambda_j e^{i\gamma_j}$, $\lambda_j = |p_j|$, $\gamma_j = \arctan({\rm{Im}} (p_j)  / {\rm{Re}} (p_j)  )$, it follows that
\begin{equation*}
	|\pp_{t} (p_j^{c_3}|p_j|^{c_4})| \lesssim
	\lambda_*^{c_3+c_4} (T-t)^{-1} \mbox{ \ for \ } c_3,c_4\in \mathbb{R},
	\quad
	|\dot{\gamma}_j |= |p_j|^{-2}
	| {\rm{Im} } (\dot{p}_j) {\rm{Re} } (p_j) - {\rm{Im} } (p_j) {\rm{Re} } (\dot{p}_j) | \lesssim (T-t)^{-1}.
\end{equation*}
Hence, the ansatz \eqref{lam-ansatz} holds.
Direct calculation concludes the properties of $\lambda_j, \xi^{\J}, \gamma_j$ in Theorem \ref{thm}.

{\textbf{Step 2.}}
Given $\big( \mathbf{\Phi_{\rm{in} }^{[\cdot]} }, \Phi_{\rm{out}}, {\mathbf{p}}_{\cdot 1},  \dot{\bm{\xi}}^{[\cdot]} \big) \in
\mathcal{B}_{\rm{in}}^{[\cdot]} \times B_{\rm{out}} \times B_{ {\mathbf{p}}_{\cdot 1} } \times B_{\dot{\bm{\xi}}^{[\cdot]}}$, recalling $c_{0}^{\J}$, $c_{1}^{\J}$ given in \eqref{qd24July14-1}, we will give a solution
\begin{equation}\label{qd24Sep28-1}
	\big( \Phi_{\rm in}^{\mbox{{\tiny{$[j1]$}}}} , c_{*0}^{\J} , c_{*1}^{\J} \big) = \big( \Phi_{\rm in}^{\mbox{{\tiny{$[j1]$}}}} , c_{*0}^{\J} , c_{*1}^{\J} \big)\big[\mathbf{\Phi_{\rm{in} }^{[\cdot]} }, \Phi_{\rm{out}}, {\mathbf{p}}_{\cdot 1},  \dot{\bm{\xi}}^{[\cdot]} \big]
\end{equation}
to \eqref{inner-eq-1}. We always assume $\rho_j \le 2C_{\lambda} R$ in \textbf{Step 2}. For $\mathcal{H}_{ \rm{in} }^{\J} $ given in \eqref{def-HPhi2}, by \eqref{inn-topo},
\begin{equation}\label{Hin-upp}
	|\mathcal{H}_{ \rm{in} }^{\J} |
	\lesssim
	\| \Phi_{\rm{in} }^{\J} \|_{{\rm in},\nu-\delta_0,l}^2 \lambda_*^{2\nu-2\delta_0} \langle \rho_j\rangle^{-2l-3}.
\end{equation}
For $\mathcal{H}_1^{\J}$ given in \eqref{Hj-1-def}, by \eqref{Qdecompo}, we have $(\mathcal{H}_1^{\J} )_{\mathbb{C}_j} =
\lambda_j^{2}
\big( \tilde{L}_{j}^{\#}[\Phi_{\rm {out} }](y^{\J},t)
+ \tilde{l}_{j}^{\#}[\Phi_{\rm {out} }](x,t) + M_0^{\J} + e^{i\theta_j} M_{1}^{\J} \big) $. Recalling the right-hand side of \eqref{inner-eq-1}, for brevity, we	denote
\begin{equation*}
	F^{\J}:=\mathcal{H}^{\J}_1
	-
	\lambda_j^{2} \big( \tilde{l}_{j,0}^{\#}[\Phi_{\rm {out} }]  \big)_{\mathbb{C}_j^{-1}}
	+
	\mathcal{H}_{ \rm{in} }^{\J}
	-
	\big( ( \mathcal{H}_{ \rm{in} }^{\J} )_{\mathbb{C}_j,0}
	\big)_{ \mathbb{C}_j^{-1} },
	\quad
	F_k^{\J} := \big(e^{ik\theta } (F^{\J})_{\mathbb{C}_j,k} \big)_{\mathbb{C}_j^{-1} }.
\end{equation*}

{\textbf{Mode $0$}.} By \eqref{Llarge},
$ (F^{\J})_{\mathbb{C}_j, 0} = \lambda_j^{2}
\big( \tilde{L}_{j,0}^{\#}[\Phi_{\rm {out} }](\rho_j,t)  + M_0^{\J}  \big)$. Combining \eqref{nablaW}, \eqref{out-topo} and \eqref{M-est}, we have
$| (F^{\J})_{\mathbb{C}_j, 0} |\lesssim
\lambda_* \langle \rho_j\rangle^{-3}$.
Obviously, $F_0^{\J}  = \big( ( \mathcal{H}^{\J}_1  )_{\mathbb{C}_j , 0 }
-
\lambda_j^{2}   \tilde{l}_{j,0}^{\#}[\Phi_{\rm {out} }]
\big)_{\mathbb{C}_j^{-1}}$.
By \eqref{tau-t}, $\lambda_*(t(\tau_j)) \sim |\ln T|^{-1} \tau_j^{-1} (\ln \tau_j)^2$.
Applying Proposition \ref{Re-m0-prop} (with $R_0 =\lambda_*^{-\delta_0/6} $, $R_1 =R_* =\infty$) gives a mapping $( \TT_0^{2 C_{\lambda} R}[F_{0}^{\J}], c_{0}^{\J}[F_{0}^{\J}] )$, where $c_{0}^{\J}[F_{0}^{\J}]$ is given in \eqref{qd24July14-1},
\begin{equation*}%\label{qd24July13-5}
	\TT_0^{2 C_{\lambda} R}[F_{0}^{\J}] \cdot W^{\J} =0,
	\quad
	|\TT_0^{2 C_{\lambda} R}[F_{0}^{\J}] |
	\lesssim R_0^{5-\ell_0 } \ln R_0 \lambda_*
	\langle \rho_j \rangle^{2-\ell_0 }
	\lesssim \lambda_*^{\nu-\delta_0 +\epsilon} \langle \rho_j \rangle^{-l},
	\quad
	|c_{*0}^{\J}[F_{0}^{\J}] | \lesssim  R_0^{1 -\ell_0}\ln R_0  \lambda_*
\end{equation*}
provided $0<\frac{\delta_0}{6} <\beta <\frac{1}{2}$, $\ell_0 \in (1,3)$,
$-\frac{\delta_0}{6}(5-\ell_0) +1 >\nu-\delta_0$,
$2-\ell_0 \le -l$, where $\epsilon>0$ is a sufficiently small constant varying from line to line.

{\textbf{Mode $1$}.}
By \eqref{tilde-L-complex} and \eqref{tildeL-component}, one has
\begin{equation}\label{qd24July15-1}
	\big| \lambda_j^{2} Q_{-\gamma_j}
	(a-bU^{\J}  \wedge)
	\tilde{L}_{U^{\J} }[\Phi_{\rm{out}}] \big| \lesssim \lambda_* \langle \rho_j \rangle^{-2}.
\end{equation}
Combining \eqref{M-est},
$|  (\mathcal{H}_1^{\J} )_{\mathbb{C}_j,1} |
\lesssim
\lambda_* \langle \rho_j\rangle^{-2}$.
Integrating \eqref{Hin-upp}, we have $|F_1^{\J}| \lesssim \lambda_* \langle \rho_j\rangle^{-2}  $ provided $\nu-\delta_0\ge 1/2$. Suppose $0<\frac{\delta_0}{6} <\beta <\frac{1}{2}$, $\delta_0<3/2$, applying Proposition \ref{qd24July12-8-prop} (with $R_0 =\lambda_*^{-\delta_0/6} $, $R_1 =R_* =\infty$) gives a mapping
$( \TT_1^{2C_{\lambda} R}[F_1^{\J}], c_{1}^{\J}[F_1^{\J}] )$, where $c_{1}^{\J}[F_1^{\J}]$ is given in \eqref{qd24July14-1},
\begin{equation*}%\label{qd24July13-4}
	\TT_1^{2 C_{\lambda} R}[F_1^{\J}] \cdot W^{\J} =0,
	\quad
	|\TT_1^{2 C_{\lambda} R}[F_1^{\J}]|
	\lesssim R_0^{5} \lambda_* \lesssim \lambda_*^{\nu-\delta_0 +\epsilon} \langle \rho_j \rangle^{-l},
	\quad
	| c_{*1}^{\J}[F_{1}^{\J}] |\lesssim R_0^{-1} \lambda_*
\end{equation*}
provided $
\nu+\beta l -1<\delta_0 /6$ since $\rho_j \le 2C_{\lambda} R$.

{\textbf{Mode $-1$}.} By \eqref{Hj-1-def} and \eqref{Qdecompo}, we have $(\mathcal{H}_1^{\J} )_{\mathbb{C}_j,-1} = \lambda_j^2 \tilde{l}_{j,-1}^{\#}[\Phi_{\rm {out} }]$.
By \eqref{out-to-in-1}, $\rho_j \le 2C_{\lambda} R$, $\beta<1/2$, then
$ |  (\mathcal{H}_1^{\J} )_{\mathbb{C}_j,-1}   |
\lesssim
\lambda_* \langle \rho_j \rangle^{-2}
\big(  \lambda^{\Theta}_*
+
\lambda_*^{ \alpha/2 }
\big)$.
Using \eqref{Hin-upp} with $\nu-\delta_0>1/2$, we have $|F_{-1}^{\J}| \lesssim \lambda_*^{1+\epsilon_1} \langle \rho_j \rangle^{-2-\epsilon_1}$ with a sufficiently small constant $\epsilon_1>0$. Proposition \ref{qd24July13-3-prop} gives a mapping $\mathcal{T}_{-1}[F_{-1}^{\J}]$ satisfying
\begin{equation*}
	\mathcal{T}_{-1}[F_{-1}^{\J}] \cdot W^{\J} =0,
	\quad
	| \mathcal{T}_{-1}[F_{-1}^{\J}] | \lesssim \tau_j^{-1} \lesssim \lambda_*^{\nu-\delta_0 +\epsilon} \langle \rho_j \rangle^{-l}
\end{equation*}
provided $\nu+\beta l -\delta_0-1<0$.

{\textbf{Mode $k$, $|k|\ge 2$}.}
By \eqref{Hj-1-def}, \eqref{Qdecompo}, \eqref{Llarge}, \eqref{out-to-in-1}, for $|k|\ge 2$, $\rho_j \le 2C_{\lambda} R$, $\beta<1/2$, we have
\begin{equation*}
	| (\mathcal{H}_1^{\J} )_{\mathbb{C}_j,k} |
	\lesssim
	\lambda_* \langle \rho_j\rangle^{-3}  \left| \nabla_x \Phi_{\rm{out}}(q^{\J},T)\right|
	+
	\lambda_* \langle \rho_j \rangle^{-2}
	\big(  \lambda^{\Theta}_*
	+
	\lambda_*^{ \alpha/2 }
	\big)
	\lesssim
	\lambda_*^{\nu-\delta_0+\epsilon} \langle \rho_j\rangle^{-2-l}
\end{equation*}
provided
$\nu-\delta_0<1$, $0<l<1$,  $\nu-\delta_0 +\beta l -1<\min\{ \Theta, \alpha/2 \}$.
Combining \eqref{Hin-upp} with $\nu>\delta_0$, we get $|F_k^{\J}| \lesssim \lambda_*^{\nu-\delta_0+\epsilon} \langle \rho_j\rangle^{-2-l} $. For $\beta\in (0,1/2)$, Proposition \ref{modek-prop} gives a mapping
$\TT_k^{2 C_{\lambda} R}[F_k^{\J}]$ satisfying
\begin{equation*}
	\TT_k^{2 C_{\lambda} R}[F_k^{\J}] \cdot W^{\J} =0,
	\quad
	|\TT_k^{2 C_{\lambda} R}[F_k^{\J}] |
	\lesssim |k|^{-1-(0.05)^2} \lambda_*^{\nu-\delta_0+\epsilon}
	\langle \rho_j \rangle^{-l}.
\end{equation*}

As a summary of all the modes above, under the parameter restrictions
\begin{equation}\label{qd240725-2}
	\begin{aligned}
		& 0<\frac{\delta_0}{6} <\beta <\frac{1}{2},
		\quad
		\ell_0 \in (1,3),
		\quad
		-\frac{\delta_0}{6}(5-\ell_0) +1 >\nu-\delta_0 > 1/2,
		\quad
		2-\ell_0 \le -l,
		\\
		&
		\delta_0<3/2,
		\quad
		\nu+\beta l -1<\delta_0 /6,
		\quad
		0<l<1,
	\end{aligned}
\end{equation}
we set  $\Phi_{\rm in}^{\mbox{{\tiny{$[j1]$}}}} = \mathcal{T}_{-1}[F_{-1}^{\J}]
+
\sum_{k \in \mathbb{Z} , k\ne -1} \TT_k^{2 C_{\lambda} R}[F_k^{\J}]$. We have found \eqref{qd24Sep28-1} solving \eqref{inner-eq-1} with the estimates
\begin{equation}\label{qd240724-5}
	\Phi_{\rm in}^{\mbox{{\tiny{$[j1]$}}}}  \cdot W^{\J} =0,
	\quad	
	|\Phi_{\rm in}^{\mbox{{\tiny{$[j1]$}}}} | \lesssim \lambda_*^{\nu-\delta_0 +\epsilon} \langle \rho_j \rangle^{-l},
	\quad
	|c_{*0}^{\J}| \lesssim  R_0^{1 -\ell_0}\ln R_0  \lambda_*,
	\quad
	| c_{*1}^{\J}|\lesssim R_0^{-1} \lambda_*.
\end{equation}

{\textbf{Step 3.}} Given $\mathbf{\Phi_{\rm{in} }^{[\cdot]} } \in \mathcal{B}_{\rm{in}}^{[\cdot]}$, we will solve $\big( \Phi_{\rm{out}}, {\mathbf{p}}_{\cdot 1}, \dot{\bm{\xi}}^{[\cdot]} \big) = \big( \Phi_{\rm{out}}, {\mathbf{p}}_{\cdot 1}, \dot{\bm{\xi}}^{[\cdot]} \big)[\mathbf{\Phi_{\rm{in} }^{[\cdot]} }]$ in $B_{\rm{out}} \times B_{ {\mathbf{p}}_{\cdot 1} } \times B_{\dot{\bm{\xi}}^{[\cdot]}}$.

Recall the outer problem \eqref{outer-eq}. Under the assumption in Proposition \ref{Bprop} and using Proposition \ref{funda-prop}, there exists a fundamental solution $\Gamma_{\Phi, U_{*}} (x,t,y,s)$ for
\begin{equation*}
	\pp_{t} \fbf = \mathbf{B}_{\Phi,U_{*}}  \Delta_x \fbf \mbox{ \ in \ } \R^2\times (0,T).
\end{equation*}

We will choose $c_{mn}$ such that $\Phi_{\rm{out}}(q^{\K},T) = 0$ for $k=1,2,\dots,N$, that is,
\begin{equation}\label{adjust-0-eq}
	\left(\Gamma_{\Phi, U_{*}} ** \mathcal{G} \right)(q^{\K},T) +
	\left(\Gamma_{\Phi, U_{*}} * Z_{*} \right) (q^{\K},T)
	+
	\sum\limits_{m=1}^{N} \sum\limits_{n=1}^3 c_{mn}
	\left( \Gamma_{\Phi, U_{*}} *  \vartheta_{mn} \right) (q^{\K},T) =0.
\end{equation}
By Propositions \ref{Bprop} and \ref{funda-prop}, for $T, \| Z_*\|_{C^3} \ll 1$ depending on $\Lambda_{\rm{o}}$, $\Gamma_{\Phi, U_{*}}$ satisfies the estimates in Proposition \ref{funda-prop} , which is independent of $\Lambda_{\rm{o}}$. For $\mathbf{f}= Z_{*}$ or $\vartheta_{mn}$ (see \eqref{Z*-def}, \eqref{vartheta-def}), we have
\begin{equation}\label{out-cau-deri}
	\left|\Gamma_{\Phi, U_{*}} * \mathbf{f} \right| +
	\left|D_x \left( \Gamma_{\Phi, U_{*}} * \mathbf{f} \right) \right| + \left|D_x^2 \left( \Gamma_{\Phi, U_{*}} * \mathbf{f}
	\right) \right|
	+ \left|\pp_{t} \left( \Gamma_{\Phi, U_{*}} * \mathbf{f} \right) \right|  \lesssim \|\mathbf{f}\|_{C^3(\R^2)} \mbox{ \ in \ } \R^2\times (0,T) ,
\end{equation}
\begin{equation}\label{out-cau-nab-Hol}
	\frac{ \left| D_x \left(\Gamma_{\Phi, U_{*}} * \mathbf{f} \right)  (  x, t ) -  D_x \left(\Gamma_{\Phi, U_{*}} * \mathbf{f} \right) (x_*, t_*  )  \right| }{  \big(| x - x_*| + \sqrt{| t- t_*|} \big)^{\alpha} }
	\lesssim \|\mathbf{f}\|_{C^3(\R^2)} \mbox{ \ for \ } 0<\alpha <1, \ \  (x,t),~(x_*,t_*)\in \R^2\times (0,T),
\end{equation}
where both ``$\lesssim$'' are independent of $\Lambda_{\rm{o}}$. By \eqref{out-cau-deri}, for $m,k =1,2,\dots,N$, $n=1,2,3$, we have
\begin{equation*}
	| ( \Gamma_{\Phi, U_{*}} *  \vartheta_{mn} ) (q^{\K},T) - \vartheta_{mn}(q^{\K}) |
	=
	| ( \Gamma_{\Phi, U_{*}} *  \vartheta_{mn} ) (q^{\K},T) - \delta_{mk} \mathbf{e}_{n} |
	\lesssim T.
\end{equation*}
Thus we can find unique $c_{mn}
= c_{mn1} + c_{mn2}$ for $m =1,2,\dots,N$, $n=1,2,3$ solving \eqref{adjust-0-eq}, where $ c_{mn1}=c_{mn1}[\Phi, U_{*}, Z_{*}]$ and $c_{mn2}= c_{mn2}[\Phi, U_{*},\mathcal{G} ]$ satisfy
\begin{equation}\label{qd24July07-1}
	\begin{aligned}
		&
		(\Gamma_{\Phi, U_{*}} * Z_{*} ) (q^{\K},T)
		+
		\sum\limits_{m=1}^{N} \sum\limits_{n=1}^3 c_{mn1}
		( \Gamma_{\Phi, U_{*}} *  \vartheta_{mn}) (q^{\K},T) =0,
		\\
		&
		(\Gamma_{\Phi, U_{*}} ** \mathcal{G})(q^{\K},T)
		+
		\sum\limits_{m=1}^{N} \sum\limits_{n=1}^3 c_{mn2}
		( \Gamma_{\Phi, U_{*}} *  \vartheta_{mn}) (q^{\K},T) =0,
		\mbox{ \ for \ } k=1,2,\dots,N,
	\end{aligned}
\end{equation}
and thus
\begin{equation}\label{cmn-est}
	|c_{mn1} | \lesssim
	\sum\limits_{k=1}^N  |
	(\Gamma_{\Phi, U_{*}} * Z_{*} ) (q^{\K},T) |
	\lesssim \| Z_{*} \|_{C^3(\R^2)},
	\quad
	|c_{mn2} | \lesssim
	\sum\limits_{k=1}^N  | (\Gamma_{\Phi, U_{*}} ** \mathcal{G} )(q^{\K},T) |,
\end{equation}
where the estimate of $c_{mn1} $ is independent of $\Lambda_{\rm{o}}$.
To find a solution to the outer problem \eqref{outer-eq}, it suffices to solve the following fixed-point problem:
\begin{equation*}
	\mathcal{T}_{\rm{o}} [  \Phi_{\rm{out}}, {\mathbf{p}}_{\cdot 1}, \dot{\bm{\xi}}^{[\cdot]} ] = \Phi_{\rm{out}}^{\mbox{{\tiny{$(1)$}}}} + \Phi_{\rm{out}}^{\mbox{{\tiny{$(2)$}}}},
\end{equation*}
where we denote
\begin{align*}
	%\begin{equation*}
	%\begin{aligned}
	&
	\Phi_{\rm{out}}^{\mbox{{\tiny{$(1)$}}}} :=
	\Gamma_{\Phi, U_{*}} * Z_{*}
	+
	\sum\limits_{m=1}^{N} \sum\limits_{n=1}^3
	c_{mn1}
	(\Gamma_{\Phi, U_{*}} *  \vartheta_{mn} ),
	\
	c_{mn1} = c_{mn1}[\Phi, U_{*},Z_* ],
	\\
	&
	\Phi_{\rm{out}}^{\mbox{{\tiny{$(2)$}}}}:=
	\Gamma_{\Phi, U_{*}} ** \mathcal{G}
	+
	\sum\limits_{m=1}^{N} \sum\limits_{n=1}^3
	c_{mn2}
	(\Gamma_{\Phi, U_{*}} *  \vartheta_{mn} ),
	\
	\mathcal{G} = \mathcal{G}[  \Phi_{\rm{out}}, {\mathbf{p}}_{\cdot 1},  \dot{\bm{\xi}}^{[\cdot]} ], \
	c_{mn2} = c_{mn2}\big[\Phi, U_{*},\mathcal{G}[  \Phi_{\rm{out}}, {\mathbf{p}}_{\cdot 1}, \dot{\bm{\xi}}^{[\cdot]} ] \big]
	%\end{aligned}
	%\end{equation*}
\end{align*}
with $\Phi$ given in \eqref{u-def}. Applying \eqref{out-cau-deri}, \eqref{out-cau-nab-Hol} to $ \Gamma_{\Phi, U_{*}} *  Z_{*} $, $\Gamma_{\Phi, U_{*}} *  \vartheta_{mn}$, and using $1-\frac{\alpha}{2}-A_{\rm{o,h}} >0$ for $A_{\rm{o,h}}$ given in \eqref{para-rho-con}, $|c_{mn1} |
\lesssim \| Z_{*} \|_{C^3(\R^2)}$ in \eqref{cmn-est}, we can take the constant $\Lambda_{\rm{o}} \ge 1$ sufficiently large such that $ \| \Phi_{\rm{out}}^{ \mbox{{\tiny{$(1)$}}}} \|_{\sharp, \Theta,\alpha} \le \Lambda_{\rm{o}} /9$. By Lemma \ref{G-est-lem} and Proposition \ref{convolu-prop}, we have $\| \Gamma_{\Phi, U_{*}} ** \mathcal{G} \|_{\sharp, \Theta,\alpha} \lesssim T^{\epsilon}$ and $|c_{mn2} | \lesssim |\ln T| \lambda_*^{\Theta+1}(0) R(0) T^{\epsilon}$, which implies $\| \Phi_{\rm{out}}^{ \mbox{{\tiny{$(2)$}}}} \|_{\sharp, \Theta,\alpha}
\lesssim T^{\epsilon}$. Taking $T\ll 1$, we have $\| \mathcal{T}_{\rm{o}} [  \Phi_{\rm{out}}, {\mathbf{p}}_{\cdot 1},  \dot{\bm{\xi}}^{[\cdot]}  ] \|_{\sharp, \Theta,\alpha}
\le \Lambda_{\rm{o}}$. Due to the choices of $c_{mn1}, c_{mn2}$ in \eqref{qd24July07-1}, $ \mathcal{T}_{\rm{o}} [  \Phi_{\rm{out}}, {\mathbf{p}}_{\cdot 1},  \dot{\bm{\xi}}^{[\cdot]}  ](q^{\J},T) =0$ for $j=1,2,\dots, N$ automatically. Therefore, we have
\begin{equation}\label{qd24July07-2}
	\mathcal{T}_{\rm{o}} [  \Phi_{\rm{out}}, {\mathbf{p}}_{\cdot 1},  \dot{\bm{\xi}}^{[\cdot]}  ]   \in  B_{\rm{out}}.
\end{equation}

Also, $\vartheta_{mn}$ given in \eqref{vartheta-def} satisfy $\nabla \vartheta_{mn}(q^{\J})=0$ for $j=1,2,\dots,N$. Then $\nabla \mathcal{T}_{\rm{o}} [  \Phi_{\rm{out}}, {\mathbf{p}}_{\cdot 1},\dot{\bm{\xi}}^{[\cdot]}  ](q^{\J},0)=  \nabla Z_{*}(q^{\J})$. Recall ${\rm DC}_j$ defined in \eqref{R0-def}. Combining \eqref{qd24July07-2} and $\| \cdot \|_{\sharp, \Theta,\alpha}$-norm defined in \eqref{out-topo}, then
\begin{equation}\label{qd24July07-3}
	{\rm DC}_j\big[\mathcal{T}_{\rm{o}} [  \Phi_{\rm{out}}, {\mathbf{p}}_{\cdot 1}, \dot{\bm{\xi}}^{[\cdot]} ] \big](t) = {\rm DC}_j[Z_*]  + O(T^{\epsilon}),
\end{equation}
which meets the assumption \eqref{hypA00}
under the parameter assumptions \eqref{qd24July05-1}. Then
by Proposition \ref{keyprop},
$\tilde{p}_j := \mathcal{P} \big[{\rm DC}_j [\mathcal{T}_{\rm{o}} [  \Phi_{\rm{out}}, {\mathbf{p}}_{\cdot 1},  \dot{\bm{\xi}}^{[\cdot]} ] ] \big]$ satisfies
\begin{equation}\label{qd24July07-4}
	\begin{aligned}
		&
		\mathcal{B}_0[\tilde{p}_j](t)
		= {\rm DC}_j [\mathcal{T}_{\rm{o}} [  \Phi_{\rm{out}}, {\mathbf{p}}_{\cdot 1},  \dot{\bm{\xi}}^{[\cdot]}] ](t) + \mathcal{R}_0\big[{\rm DC}_j [\mathcal{T}_{\rm{o}} [  \Phi_{\rm{out}}, {\mathbf{p}}_{\cdot 1}, \dot{\bm{\xi}}^{[\cdot]}  ] ] \big](t) , \quad t \in [0,T],
		\\
		&
		\tilde{p}_j =  {\rm DC}_j\big[\mathcal{T}_{\rm{o}} [  \Phi_{\rm{out}}, {\mathbf{p}}_{\cdot 1}, \dot{\bm{\xi}}^{[\cdot]}  ] \big](T) \big( 1+ O(|\ln T|^{-1}) \big) \lambda_*,
		\\
		&
		\dot{\tilde{p}}_j =  - {\rm DC}_j\big[\mathcal{T}_{\rm{o}} [  \Phi_{\rm{out}}, {\mathbf{p}}_{\cdot 1}, \dot{\bm{\xi}}^{[\cdot]}  ] \big](T) \big( 1+ O(|\ln T|^{-1}) \big)
		(T-t)^{-1} \lambda_*.
	\end{aligned}
\end{equation}
We define a mapping
\begin{equation*}
	\mathcal{T}_{ p_{j 1} } [  \Phi_{\rm{out}}, {\mathbf{p}}_{\cdot 1}, \dot{\bm{\xi}}^{[\cdot]}  ] :=
	\tilde{p}_j - p_{j0},
	\quad
	\mathcal{T}_{{\mathbf{p}}_{\cdot 1}}
	:=
	( \mathcal{T}_{ p_{11} }, \mathcal{T}_{ p_{21} }, \dots, \mathcal{T}_{ p_{N1} } ).
\end{equation*}
By \eqref{qd240729-3}, \eqref{qd24July07-3}, and \eqref{qd24July07-4}, we have
\begin{equation}\label{qd24july12-9}
	\big| \mathcal{T}_{{\mathbf{p}}_{j 1} } [  \Phi_{\rm{out}}, {\mathbf{p}}_{\cdot 1}, \dot{\bm{\xi}}^{[\cdot]}  ] \big|
	+
	(T-t)
	\big| \partial_t \mathcal{T}_{{\mathbf{p}}_{j 1} } [  \Phi_{\rm{out}}, {\mathbf{p}}_{\cdot 1}, \dot{\bm{\xi}}^{[\cdot]}  ] \big| \lesssim
	|\ln T|^{-1} \lambda_*.
\end{equation}

Orthogonality equation \eqref{orth1-eqr} gives a mapping of $ \dot{\xi}^{\J}$ from the right-hand side of \eqref{orth1-eqr} to $ \dot{\xi}^{\J}$, which is denoted by $\mathcal{T}_{ \dot{\xi}^{\J} } [  \Phi_{\rm{out}}, {\mathbf{p}}_{\cdot 1}, \dot{\bm{\xi}}^{[\cdot]}  ]$. Write $\mathcal{T}_{\dot{\bm{\xi}}^{[\cdot]}} := ( \mathcal{T}_{ \dot{\xi}^{[1]} }, \mathcal{T}_{ \dot{\xi}^{[2]} },\dots, \mathcal{T}_{ \dot{\xi}^{[N]} })$.

By \eqref{Qdecompo}, \eqref{Llarge}, $\int_0^{\infty}
w_{\rho_j}(\rho_j) \cos w(\rho_j)
\mathcal{Z}_{1,1}(\rho_j) \rho_j d\rho_j = \int_0^{\infty}
\frac{-2\rho_j (\rho_j^2-1)}{ (\rho_j^2+1)^3 } d\rho_j =0$,  and
\eqref{out-to-in-1}, one has
\begin{equation*}%\label{out(q,T)1-or}
	\begin{aligned}
		&
		\Big| \lambda_j	\int_{0}^{\infty}
		\Big( Q_{-\gamma_j}
		\Big[
		( a-bU^{\J}  \wedge )
		\tilde{L}_{U^{\J} }[\Phi_{\rm{out}}]  	\Big] \Big)_{\mathbb{C}_j,1}
		(\rho_j,t)  \mathcal{Z}_{1,1}(\rho_j) \rho_j d \rho_j  \Big|
		\\
		=  \ &
		\Big| \lambda_j		\int_{0}^{\infty}
		\tilde{l}_{j,1}^{\#}[\Phi_{\rm {out} }](\rho_j,t)  \mathcal{Z}_{1,1}(\rho_j) \rho_j d \rho_j \Big|
		\lesssim
		\lambda^{\Theta +\alpha \beta}_*
		+
		\lambda^{\alpha}_*.
	\end{aligned}
\end{equation*}
By \eqref{Hin-upp},
$ \big|\lambda_j^{-1}
\int_{0}^{\infty}
(
\mathcal{H}_{ \rm{in} }^{\J} )_{\mathbb{C}_j,1} (\rho_j,t)  \mathcal{Z}_{1,1}(\rho_j) \rho_j d\rho_j  \big|
\lesssim \| \Phi_{\rm{in} }^{\J} \|_{{\rm in},\nu-\delta_0,l}^2 \lambda_*^{2\nu-2\delta_0-1} $.
By \eqref{qd240724-5}, $|\lambda_j^{-1} c_{*1}^{\J}( \tau_j(t) ) | \lesssim R_0^{-1} = \lambda_*^{\delta_0 /6}$.
In order for $\mathcal{T}_{\dot{\bm{\xi}}^{[\cdot]}}(B_{\dot{\bm{\xi}}^{[\cdot]}})\subset B_{\dot{\bm{\xi}}^{[\cdot]}}$, we take
\begin{equation}\label{xi-est}
	\nu -\delta_0 > 1/2, \quad
	0< 2\epsilon_{\xi} < \min\left\{ \Theta +\alpha \beta, \alpha,  2\nu-2\delta_0-1, \delta_0/6 \right\}.
\end{equation}

By \eqref{qd24July07-2}, \eqref{qd24july12-9}, \eqref{xi-est}, it follows that
$( \mathcal{T}_{\rm{o}}, \mathcal{T}_{{\mathbf{p}}_{\cdot 1}}, \mathcal{T}_{\dot{\bm{\xi}}^{[\cdot]} } )$
maps $B_{\rm{out}} \times B_{ {\mathbf{p}}_{\cdot 1} } \times B_{\dot{\bm{\xi}}^{[\cdot]} }$ to itself.

By Proposition \ref{funda-prop}, since the right-hand side of the outer problem is in the weighted-$L^{\infty}$ space, one can obtain more regularity for $\mathcal{T}_{\rm{o}}$ compared with the norm $\| \cdot\|_{\sharp, \Theta,\alpha}$ defined in \eqref{out-topo} if the weight in $\| \cdot\|_{\sharp, \Theta,\alpha}$ is relaxed, yielding compactness for $\mathcal{T}_{\rm{o}}$.

By \eqref{p-est} in Proposition \ref{keyprop}, $ \partial_t \mathcal{T}_{{\mathbf{p}}_{j 1} } [  \Phi_{\rm{out}}, {\mathbf{p}}_{\cdot 1},  \dot{\bm{\xi}}^{[\cdot]}  ]$ has H\"older continuity.

Recall the terms in the right-hand side of \eqref{orth1-eqr}. By \eqref{qd240725-1}, \eqref{m1-z3} in the proof of Proposition \ref{qd24July12-8-prop}, $c_{*1}^{\J}( \tau_j(t) )$ has quantitative H\"older continuity from the time H\"older continuity of the re-gluing outer problem $\psi_{o,1}$.
Together with the time continuity in the norms of the inner and outer problems in \eqref{inn-topo} and \eqref{out-topo}, it follows that $\mathcal{T}_{\dot{\bm{\xi}}^{[\cdot]}}$ is H\"older continuous, and thus the Schauder fixed-point theorem gives a fixed point for $( \mathcal{T}_{\rm{o}}, \mathcal{T}_{{\mathbf{p}}_{\cdot 1}}, \mathcal{T}_{\dot{\bm{\xi}}^{[\cdot]}} )$
in $B_{\rm{out}} \times B_{ {\mathbf{p}}_{\cdot 1} } \times B_{\dot{\bm{\xi}}^{[\cdot]}}$. Namely, we find a solution $\big( \Phi_{\rm{out}}[\mathbf{\Phi_{\rm{in} }^{[\cdot]} }], {\mathbf{p}}_{\cdot 0} + {\mathbf{p}}_{\cdot 1}[\mathbf{\Phi_{\rm{in} }^{[\cdot]} }],  \dot{\bm{\xi}}^{[\cdot]}[\mathbf{\Phi_{\rm{in} }^{[\cdot]} }] \big)$ of \eqref{outer-eq} and \eqref{ortho-eq}.

{\textbf{Step 4.}}
Denote the right-hand side of the non-orthogonal inner problem \eqref{inner-eq-2} as
\begin{small}
	\begin{equation*}
		F_{\rm{no}} :=
		\lambda_j^{2} \Big( \tilde{l}_{j,0}^{\#}[\Phi_{\rm {out} }] \Big)_{\mathbb{C}_j^{-1}}
		+
		\Big( ( \mathcal{H}_{ \rm{in} }^{\J}  )_{\mathbb{C}_j,0}
		\Big)_{ \mathbb{C}_j^{-1} }  +
		\mathbf{R}_0 [\Phi_{\rm{out}} , \lambda_j, \gamma_j ]
		+
		\Big(\int_{0}^2 \eta(r) \mathcal{Z}_{0,1}^2(r) r dr \Big)^{-1}
		\Big(
		c_{*0}^{\J}(\tau_j(t ))   \eta(|y^{\J}|) \mathcal{Z}_{0,1}(|y^{\J}|)
		\Big)_{\mathbb{C}_j^{-1}}.
	\end{equation*}
\end{small}
By \eqref{out-to-in-1},
$ \big|\lambda_j^{2} \big( \tilde{l}_{j,0}^{\#}[\Phi_{\rm {out} }] \big)_{\mathbb{C}_j^{-1}} \big|
\lesssim
\big( \lambda_*^{1+\Theta+\alpha\beta} + \lambda_*^{1+\alpha} \big)
\langle \rho_j \rangle^{\alpha-2} $.
By \eqref{Hin-upp}, $ \big| \big( ( \mathcal{H}_{ \rm{in} }^{\J}  )_{\mathbb{C}_j,0}
\big)_{ \mathbb{C}_j^{-1} }  \big|
\lesssim
\lambda_*^{2\nu-2\delta_0} \langle \rho_j\rangle^{-2l-3} $.
Under the parameter assumption \eqref{qd24July05-1}, by Proposition \ref{keyprop}, $\mathbf{R}_0 [\Phi_{\rm{out}} , \lambda_j, \gamma_j ]$ given in \eqref{R0-def} satisfies
$\big| \mathbf{R}_0 [\Phi_{\rm{out}} , \lambda_j, \gamma_j ] \big| \lesssim \lambda_* (T-t)^{\tilde{m}+\frac{(1+\alpha_0 ) \alpha}2}   |\ln(T-t)|^{- \varpi} \eta(|y^{\J}|) $.
By \eqref{qd240724-5},
$ \big| \big(
c_{*0}^{\J}(\tau_j(t ))   \eta(|y^{\J}|) \mathcal{Z}_{0,1}(|y^{\J}|)
\big)_{\mathbb{C}_j^{-1}} \big|
\lesssim
R_0^{1 -\ell_0}\ln R_0  \lambda_* \eta(|y^{\J}|)$.
Since $F_{\rm{no}}$ is in mode $0$, under the restrictions of parameters
\begin{equation}\label{nonortho-para}
	\begin{aligned}
		&
		\Theta +\alpha \beta<1, \quad
		1+\Theta+\alpha \beta -2\beta >\nu-\delta_0,
		\quad
		1+\alpha-2\beta >\nu-\delta_0,
		\\
		&
		2\beta <\nu-\delta_0 <1,
		\quad
		\tilde{m} + (1+\alpha_0) \alpha/2 <1,
		\quad
		1+ \tilde{m} + (1+\alpha_0) \alpha/2  -2\beta >\nu-\delta_0,
		\\
		&
		\delta_0 (\ell_0 -1) <6,
		\quad
		\delta_0 (\ell_0 -1)/6 +1-2\beta >\nu-\delta_0,
	\end{aligned}
\end{equation}
Proposition \ref{m0-nonortho} gives a mapping $\Phi_{\rm in}^{\mbox{{\tiny{$[j2]$}}}} = \TT_{00}^{2 C_{\lambda} R}[ F_{\rm{no}}  ]$ for \eqref{inner-eq-2} satisfying
\begin{equation}\label{qd240724-6}
	| \Phi_{\rm in}^{\mbox{{\tiny{$[j2]$}}}} | \lesssim \lambda_*^{\nu-\delta_0 +\epsilon} \langle \rho_j \rangle^{-1},
	\quad
	\Phi_{\rm in}^{\mbox{{\tiny{$[j2]$}}}} \cdot W^{\J} =0.
\end{equation}

We sum up \eqref{inner-eq-1} and \eqref{inner-eq-2} together and set
\begin{equation*}
	\mathcal{T}_{\rm{in}}^{\J}[ \mathbf{\Phi_{\rm{in} }^{[\cdot]} } ] := \Phi_{\rm in}^{\mbox{{\tiny{$[j1]$}}}}  + \Phi_{\rm in}^{\mbox{{\tiny{$[j2]$}}}},
	\quad
	\mathbf{T}_{\rm{in}}^{[\cdot]}[ \mathbf{\Phi_{\rm{in} }^{[\cdot]} } ] := ( \mathcal{T}_{\rm{in}}^{[1]}[ \mathbf{\Phi_{\rm{in} }^{[\cdot]} } ], \mathcal{T}_{\rm{in}}^{[2]}[ \mathbf{\Phi_{\rm{in} }^{[\cdot]} } ],\dots, \mathcal{T}_{\rm{in}}^{[N]}[ \mathbf{\Phi_{\rm{in} }^{[\cdot]} } ] ).
\end{equation*}
$\mathcal{T}_{\rm{in}}^{\J}[ \mathbf{\Phi_{\rm{in} }^{[\cdot]} } ]$ is the inverse mapping of \eqref{inner-eq} since \eqref{ortho-eq} holds. By \eqref{qd240724-5} and \eqref{qd240724-6}, it holds that
\begin{equation}\label{qd240724-3}
	|\mathcal{T}_{\rm{in}}^{\J}[ \mathbf{\Phi_{\rm{in} }^{[\cdot]} } ] | \lesssim \lambda_*^{\nu-\delta_0 +\epsilon} \langle \rho_j \rangle^{-l},
	\quad
	\mathcal{T}_{\rm{in}}^{\J}[ \mathbf{\Phi_{\rm{in} }^{[\cdot]} } ] \cdot W^{\J} =0.
\end{equation}

{\textbf{Step 5.}}  Recall $\mathcal{H}^{\J}$ given in \eqref{Hj-def}.
By \eqref{M-est}, \eqref{qd24July15-1}, and the estimate of $\mathcal{H}_{ \rm{in} }^{\J}$ in \eqref{Hin-upp}, we have
\begin{equation}\label{qd240724-4}
	| \mathcal{H}^{\J} | \lesssim \lambda_* \langle \rho_j \rangle^{-2} + \lambda_*^{2\nu-2\delta_0} \langle \rho_j\rangle^{-2l-3}
	\lesssim \lambda_*^{\nu-\delta_0 +\epsilon} \langle \rho_j \rangle^{-2-l}
	\mbox{ \ for \ } |y^{\J}| \le 2 C_{\lambda} R,
\end{equation}
where for the last step, we require
\begin{equation}\label{H1-est-para}
	\nu +\beta l -\delta_0 -1<0, \quad \delta_0 <\nu.
\end{equation}

Claim:
Given $y_*, \tau_*$ satisfying $|y_*|\le 2C_{\lambda} R(t(\tau_*)) \ll \tau_*^{1/2}$ and $\tilde{\rho} = |y_*|/9$, it holds that
\begin{equation}\label{Hj-|DMO|x}
	\big[ \mathcal{H}^{\J} ( y^{\J}, \tau_j ) \big|_{(y^{\J}, \tau_j)= (y_*+ \tilde{\rho} z, \tau_* + \tilde{\rho}^2 s ) } \big]_{ \mathsf{|DMO|_x}(Q_2^{-}(0)) }
	\lesssim
	\lambda_*^{\nu - \delta_0 +\epsilon}(t(\tau_j))|_{\tau_j = \tau_*} \langle y_* \rangle^{-2-l }
\end{equation}
provided
\begin{equation}\label{qd240725-4}
	\delta_0<\nu, \quad 0<l<1,
	\quad
	\nu-\delta_0<\min\{  \Theta +1 -\beta, 1-\beta l \},
	\quad
	0<\beta<1/2.
\end{equation}

\begin{proof}[Proof of \eqref{Hj-|DMO|x}]
	
	In this proof, for brevity, we denote $(\tilde{z}, \tilde{s}) = (y_*+ \tilde{\rho} z , \tau_* + \tilde{\rho}^2 s)$ with variables $(z,s) \in Q_2^{-}(0)$ and abuse  $\lambda_*(\tau_*)$ to denote $\lambda_*(t(\tau_j))|_{\tau_j = \tau_*}$. Obviously, $|\tilde{z}| \sim |y_*|$, $\tilde{s} \sim \tau_*$.
	
	In $y^{\J}$ variable, $\eta_R^{\J} = \eta\big( y^{\J} \lambda_j / (\lambda_* R) \big)$. By Lemma \ref{DMO-|DMO|-lem} $(2)$ and \eqref{inn-topo} for $\Phi_{\rm in}^{\J}$, one has
	\begin{equation*}
		\begin{aligned}
			&
			\langle y_* \rangle\big(	[\nabla_{y^{\J}} \eta_R^{\J} |_{(y^{\J}, \tau_j)= (\tilde{z}, \tilde{s}) } ]_{\mathsf{|DMO|_x}(Q_2^{-}(0))}
			+
			\|\nabla_{y^{\J}} \eta_R^{\J} |_{(y^{\J}, \tau_j)= (\tilde{z}, \tilde{s}) } \|_{L^\infty( Q_2^{-}(0) )  } \big)
			+
			[\eta_R^{\J} |_{(y^{\J}, \tau_j)= (\tilde{z}, \tilde{s}) }  ]_{\mathsf{|DMO|_x}(Q_2^{-}(0))}  \lesssim 1,
			\\
			&
			\langle y_* \rangle
			\big(
			[\nabla_{y^{\J}} \Phi_{\rm in}^{\J} |_{(y^{\J}, \tau_j)= (\tilde{z}, \tilde{s}) }]_{\mathsf{|DMO|_x}(Q_2^{-}(0))}
			+
			\|\nabla_{y^{\J}} \Phi_{\rm in}^{\J} |_{(y^{\J}, \tau_j)= (\tilde{z}, \tilde{s}) } \|_{L^\infty( Q_2^{-}(0) )  } \big)
			\\
			&
			+
			[\Phi_{\rm in}^{\J} |_{(y^{\J}, \tau_j)= (\tilde{z}, \tilde{s}) } ]_{\mathsf{|DMO|_x}(Q_2^{-}(0)) }
			+ \|\Phi_{\rm in}^{\J} |_{(y^{\J}, \tau_j)= (\tilde{z}, \tilde{s}) } \|_{L^\infty( Q_2^{-}(0) )  } \lesssim \lambda_*^{\nu-\delta_0}(\tau_*)  \langle y_* \rangle^{-l} \| \Phi_{\rm{in} }^{\J} \|_{{\rm in},\nu-\delta_0,l}.
		\end{aligned}
	\end{equation*}
	Similar to \eqref{Hin-upp}, by Lemma \ref{DMO-|DMO|-lem}, provided  $ \delta_0<\nu $, we have
	\begin{equation*}
		[\mathcal{H}_{ \rm{in} }^{\J} |_{(y^{\J}, \tau_j)= (\tilde{z}, \tilde{s}) } ]_{ \mathsf{|DMO|_x}(Q_2^{-}(0)) }
		\lesssim
		\lambda_*^{2\nu-2\delta_0}(\tau_*)  \langle y_* \rangle^{-3-2l} \| \Phi_{\rm{in} }^{\J} \|_{{\rm in},\nu-\delta_0,l }^2
		\lesssim
		\lambda_*^{\nu - \delta_0 +\epsilon}(\tau_*) \langle y_* \rangle^{-2-l }.
	\end{equation*}

	$\bullet$ $\mathsf{|DMO|_x}$ estimates about the coupling terms from the outer problem in $\mathcal{H}_1^{\J}$
	\begin{small}
		\begin{equation}\label{ZPhiout-Hold}
			( a-bW^{\J}  \wedge )
			\tilde{L}_{W^{\J} }[Q_{-\gamma_j}\Phi_{\rm{out}}]
			=
			( a-bW^{\J}  \wedge )
			\big(
			|\nabla_{y^{\J}} W^{\J} |^2 \Pi_{ W^{\J \perp} } (Q_{-\gamma_j}\Phi_{\rm{out}}) -2\nabla_{y^{\J}} ((Q_{-\gamma_j}\Phi_{\rm{out}})\cdot W^{\J}) \cdot \nabla_{y^{\J}}  W^{\J}
			\big).
		\end{equation}
	\end{small}
	By \eqref{out-upp} and \eqref{tau-t}, for $\lambda_j =\lambda_j(t(\tau_j))$, $\xi_j =\xi_j(t(\tau_j))$, we have
	\begin{equation}\label{out-upp-H1}
		\|\Phi_{\rm{out}} (\lambda_j y^{\J} +\xi^{\J}, t(\tau_j) ) \big|_{(y^{\J}, \tau_j)= (\tilde{z}, \tilde{s}) } \|_{L^\infty( Q_2^{-}(0) )  }
		\lesssim
		\ln \tau_*
		\lambda_*^{\Theta+1-\beta}(\tau_*)
		+
		|\ln T|^{-2} \tau_*^{-1} (\ln \tau_*)^4
		+
		\lambda_*(\tau_*) |y_*|.
	\end{equation}
	By \eqref{out-topo} and Lemma \ref{DMO-|DMO|-lem} $(2)$, we get
	\begin{equation}\label{qd240724-1}
		\begin{aligned}
			&
			[\Phi_{\rm{out}} (\lambda_j y^{\J} +\xi^{\J}, t(\tau_j) ) \big|_{(y^{\J}, \tau_j)= (\tilde{z}, \tilde{s}) } ]_{\mathsf{|DMO|_x}(Q_2^{-}(0))} \lesssim \lambda_*(\tau_*) \tilde{\rho} \sim \lambda_*(\tau_*) |y_*|,
			\\
			&
			[ (\nabla_{x} \Phi_{\rm{out}})(\lambda_j y^{\J} +\xi^{\J}, t(\tau_j) ) \big|_{(y^{\J}, \tau_j)= (\tilde{z}, \tilde{s}) } ]_{\mathsf{|DMO|_x}(Q_2^{-}(0))}
			\lesssim
			(\lambda_*^{\Theta +\alpha\beta}(\tau_*) + \lambda_*^{\alpha}(\tau_*) ) |y_*|^{\alpha}.
		\end{aligned}
	\end{equation}

	We will give $\mathsf{|DMO|_x}$ estimates of two typical terms in \eqref{ZPhiout-Hold},  and the remaining two terms can be handled similarly. By \eqref{out-upp-H1} and \eqref{qd240724-1}, we estimate
	\begin{equation*}
		\begin{aligned}
			&
			\Big[
			(a-bW^{\J}  \wedge)
			|\nabla_{y^{\J}} W^{\J} |^2
			\Big( W^{\J}  \cdot  \big(Q_{-\gamma_j}\Phi_{\rm{out}}(\lambda_j y^{\J} +\xi^{\J}, t(\tau_j) ) \big)  \Big)  W^{\J} \Big|_{ (y^{\J}, \tau_j)= (\tilde{z}, \tilde{s}) } \Big]_{\mathsf{|DMO|_x}(Q_2^{-}(0))}
			\\
			\lesssim \ &
			\langle y_* \rangle^{-4}
			\big(
			\|\Phi_{\rm{out}} (\lambda_j y^{\J} +\xi^{\J}, t(\tau_j) ) \big|_{ (y^{\J}, \tau_j)= (\tilde{z}, \tilde{s}) } \|_{L^\infty( Q_2^{-}(0) )  }
			+   [\Phi_{\rm{out}} (\lambda_j y^{\J} +\xi^{\J}, t(\tau_j) ) \big|_{ (y^{\J}, \tau_j)= (\tilde{z}, \tilde{s}) } ]_{\mathsf{|DMO|_x}(Q_2^{-}(0))}
			\big)
			\\
			\lesssim \ &
			\langle y_* \rangle^{-4}
			\big(
			\ln \tau_*
			\lambda_*^{\Theta+1-\beta}(\tau_*)
			+
			|\ln T|^{-2} \tau_*^{-1} (\ln \tau_*)^4
			+
			\lambda_*(\tau_*) |y_*| \big)
			\lesssim
			\lambda_*^{\nu -\delta_0 +\epsilon}(\tau_*) \langle y_* \rangle^{-2-l}
		\end{aligned}
	\end{equation*}
	provided $0<l<1$, $\nu-\delta_0<\min\{  \Theta +1 -\beta, 1 \}$.
	By \eqref{out-topo} and \eqref{qd240724-1}, one has
	\begin{equation*}
		\begin{aligned}
			&
			\Big[	( a-bW^{\J}  \wedge )
			\Big\{ \Big[ W^{\J} \cdot \Big(Q_{-\gamma_j} \nabla_{y^{\J}} \big( \Phi_{\rm{out}}(\lambda_j y^{\J} +\xi^{\J}, t(\tau_j) ) \big) \Big) \Big] \cdot \nabla_{y^{\J}}  W^{\J}
			\Big\}  \Big|_{ (y^{\J}, \tau_j)= (\tilde{z}, \tilde{s}) } \Big]_{\mathsf{|DMO|_x}(Q_2^{-}(0))}
			\\
			\lesssim \ & \lambda_*(\tau_*) \langle y_* \rangle^{-2}
			\big( 1  +
			[ (\nabla_{x} \Phi_{\rm{out}})(\lambda_j y^{\J} +\xi^{\J}, t(\tau_j) ) \big|_{ (y^{\J}, \tau_j)= (\tilde{z}, \tilde{s}) } ]_{\mathsf{|DMO|_x}(Q_2^{-}(0))}
			\big)
			\\
			\lesssim \ & \lambda_*(\tau_*) \langle y_* \rangle^{-2}
			\big( 1  + \lambda_*^{\Theta +\alpha\beta}(\tau_*)   |y_*|^{\alpha}
			\big)
			\lesssim
			\lambda_*^{\nu -\delta_0 +\epsilon}(\tau_*) \langle y_* \rangle^{-2-l}
		\end{aligned}
	\end{equation*}
	provided  $0<\beta<1/2$, $\nu -\delta_0 < 1-\beta l$.

	$\bullet$ $\mathsf{|DMO|_x}$ estimates about
	$ \lambda_j^{2}
	\big(  M_0^{\J} + e^{i\theta_j} M_{1}^{\J} \big)_{\mathbb{C}_j^{-1}}$,
	where $M_0^{\J}$, $M_{1}^{\J}$ are defined in \eqref{M0-def-mu3} and \eqref{M1-def},  respectively. Back to the vector form, for $k=0,1$, by \eqref{def-E1E2} and \eqref{nablaW}, we have
	\begin{align*}
			& \big( e^{ik  \theta_j} M_{k }^{\J} \big)_{\mathbb{C}_j^{-1}}
			=
			\Big[\cos(k \theta_j){\rm Re}M_k^{\J}-\sin(k \theta_j){\rm Im}M_k^{\J})\Big] E^{\J}_1+\Big[\cos(k \theta_j) {\rm Im}M_k^{\J}+\sin(k \theta_j){\rm Re}M_k^{\J})\Big]  E^{\J}_2
			\\
			= \ &
			{\rm Re}M_k^{\J} \Big[\cos(k \theta_j)E^{\J}_1+\sin(k \theta_j) E^{\J}_2\Big]+{\rm Im}M_k^{\J} \Big[\cos(k \theta_j)E^{\J}_2-\sin(k \theta_j) E^{\J}_1\Big]
			\\
			= \ &
			{\rm Re}M_k^{\J} \begin{bmatrix}
				-\cos((k-1) \theta_j)+\frac{2\rho_j^2}{\rho_j^2+1}\cos(k\theta_j)\cos\theta_j\\
				\sin((k-1)\theta_j)+\frac{2\rho_j^2}{\rho_j^2+1}\cos(k\theta_j)\sin\theta_j\\
				-\cos(k \theta_j) \frac{2\rho_j}{\rho_j^2+1}
			\end{bmatrix}
			+{\rm Im}M_k^{\J} \begin{bmatrix}
				\sin((k-1) \theta_j)-\frac{2\rho_j^2}{\rho_j^2+1}\sin(k\theta_j)\cos\theta_j\\
				\cos((k-1)\theta_j)-\frac{2\rho_j^2}{\rho_j^2+1}\sin(k\theta_j)\sin\theta_j\\
				\sin(k \theta_j)  \frac{2\rho_j}{\rho_j^2+1}
			\end{bmatrix}.
		\end{align*}
	
	For mode $1$,
	\begin{equation*}
		\lambda_j^{2}
		\big( e^{i\theta_j} M_{1}^{\J} \big)_{\mathbb{C}_j^{-1}}
		=
		\frac{  2\la_j  }{\rho_j^2+1}
		\Big(
		- \dot{\xi}_1^{\J}
		\begin{bmatrix}
			- 1+\frac{2\rho_j^2}{\rho_j^2+1} \cos^2\theta_j\\
			\frac{2\rho_j^2}{\rho_j^2+1}\cos \theta_j \sin\theta_j\\
			-\cos \theta_j \frac{2\rho_j}{\rho_j^2+1}
		\end{bmatrix}
		+
		\dot{\xi}_2^{\J}
		\begin{bmatrix}
			-\frac{2\rho_j^2}{\rho_j^2+1}\sin \theta_j \cos\theta_j\\
			1-\frac{2\rho_j^2}{\rho_j^2+1} \sin^2\theta_j\\
			\sin \theta_j  \frac{2\rho_j}{\rho_j^2+1}
		\end{bmatrix}
		\Big).
	\end{equation*}
	Then
	\begin{equation*}
		\big[ \big( \lambda_j^{2}
		\big( e^{i\theta_j} M_{1}^{\J} \big)_{\mathbb{C}_j^{-1}}  \big)(\tilde{z}, \tilde{s}) \big]_{\mathsf{|DMO|_x}(Q_2^{-}(0))}
		\lesssim
		\lambda_*^{1+ \epsilon_{\xi} } \langle y_* \rangle^{-2} \lesssim \lambda_*^{\nu-\delta_0+\epsilon} \langle y_* \rangle^{-2-l},
	\end{equation*}
	where for the last ``$\lesssim$'', we require $\nu-\delta_0 <1+\epsilon_{\xi} -\beta l$.
	
	For mode $0$,
	\begin{equation}\label{qd240728-1}
		\lambda_j^{2} (M_0^{\J})_{\mathbb{C}_j^{-1}}
		=
		\lambda_j^{2}
		\Big(
		{\rm Re}M_0^{\J}  \begin{bmatrix}
			-\cos \theta_j +\frac{2\rho_j^2}{\rho_j^2+1}\cos\theta_j\\
			-\sin \theta_j +\frac{2\rho_j^2}{\rho_j^2+1} \sin\theta_j\\
			-  \frac{2\rho_j}{\rho_j^2+1}
		\end{bmatrix}
		+{\rm Im}M_0^{\J}
		\begin{bmatrix}
			-	\sin \theta_j \\
			\cos \theta_j \\
			0
		\end{bmatrix}
		\Big).
	\end{equation}
The vanishing in $M_0^{\J}$ given in \eqref{M0-def-mu3} as $\rho_j\to 0^+$ makes the error \eqref{qd240728-1} involving $\cos \theta_j$, $\sin \theta_j$ smooth at the origin. By \eqref{M0-def-mu3} and Lemma \ref{DMO-|DMO|-lem},
	\begin{equation*}
		\big[ \big( \lambda_j^{2} (M_0^{\J})_{\mathbb{C}_j^{-1}}   \big)(\tilde{z}, \tilde{s}) \big]_{\mathsf{|DMO|_x}(Q_2^{-}(0))}
		\lesssim
		\lambda_* \langle y_* \rangle^{-3}
		\big( m_{01}^{\J}
		+
		m_{02}^{\J} \big)
		+
		|\dot{\lambda}_*| \lambda_* \langle y_* \rangle^{-3},
	\end{equation*}
	where $m_{01}^{\J}(y_*,\tau_*)$, $m_{02}^{\J}(y_*,\tau_*)$ are defined as
	\begin{small}
		\begin{align*}
				& m_{01}^{\J} = m_{01}^{\J}(y_*,\tau_*):= \Big\|
				\int_{-T}^{t(\tau_j)} \frac{\dot p_j(s) 	K_0(\zeta_j) }{t(\tau_j)-s}
				ds \Big|_{  (y^{\J}, \tau_j)= (\tilde{z}, \tilde{s}) } \Big\|_{L^\infty (Q_2^{-}(0))}
				\\
				&
				+
				\Big\|
				\int_{-T}^{t(\tau_j)} \frac{\dot p_j(s) \zeta_j K_{0\zeta_j}( \zeta_j ) }{t(\tau_j)-s}
				ds \Big|_{  (y^{\J}, \tau_j)= (\tilde{z}, \tilde{s}) } \Big\|_{L^\infty (Q_2^{-}(0))}
				+
				\Big\|
				\int_{-T}^{t(\tau_j)} \frac{\dot p_j(s) \zeta_j^2 K_{0\zeta_j \zeta_j}( \zeta_j ) }{t(\tau_j)-s}
				ds \Big|_{  (y^{\J}, \tau_j)= (\tilde{z}, \tilde{s}) } \Big\|_{L^\infty (Q_2^{-}(0))},
				\\
				&
				m_{02}^{\J} =
				m_{02}^{\J}(y_*,\tau_*):=
				\Big[
				\int_{-T}^{t(\tau_j)} \frac{\dot p_j(s) K_0(\zeta_j)  }{t(\tau_j)-s}
				ds \Big|_{  (y^{\J}, \tau_j)= (\tilde{z}, \tilde{s}) } \Big]_{\mathsf{|DMO|_x}(Q_2^{-}(0))}
				\\
				&
				+
				\Big[
				\int_{-T}^{t(\tau_j)} \frac{\dot p_j(s) \zeta_j K_{0\zeta_j}( \zeta_j )  }{t(\tau_j)-s}
				ds \Big|_{  (y^{\J}, \tau_j)= (\tilde{z}, \tilde{s}) } \Big]_{ \mathsf{|DMO|_x}(Q_2^{-}(0)) }
				+
				\Big[
				\int_{-T}^{t(\tau_j)} \frac{\dot p_j(s) \zeta_j^2 K_{0\zeta_j \zeta_j}( \zeta_j )  }{t(\tau_j)-s}
				ds \Big|_{  (y^{\J}, \tau_j)= (\tilde{z}, \tilde{s}) } \Big]_{ \mathsf{|DMO|_x}(Q_2^{-}(0)) },
			\end{align*}
	\end{small}
	with $\zeta_j= \zeta_j(y^{\J},\tau_j,s)=\frac{\lambda_j^2(t(\tau_j)) (1+|y^{\J}|^2)}{t(\tau_j)-s}$.
	By \eqref{K-est} and \eqref{nonlocal-est}, we have $m_{01}^{\J} \lesssim 1 $.
	To estimate $m_{02}^{\J}$, we consider the following more general form, which recovers all the terms in $m_{02}^{\J}$. Set
	\begin{equation*}
		g_0(y^{\J},\tau_j) :=
		\int_{-T}^{t(\tau_j)} \frac{\dot{p}_j(s)  \mathcal{K}(\zeta_j)  }{t(\tau_j) -s} ds,
		\quad
		\mathcal{K}(\zeta_j) :=
		c_0 \zeta_j^{-1} ( 1-e^{- d_0 \zeta_j } )
		+
		\sum\limits_{k=1}^n c_i \zeta_j^{k-1} e^{- d_k \zeta_j },
	\end{equation*}
	where $c_k, d_k$ are complex constants and $\mathrm{Re}(d_k)>0$ for $k=0,1,\dots,n$. It is easy to see that
	\begin{equation}\label{cal-K-est}
		\left| \mathcal{K}(\zeta_j) \right|
		+
		\left| \zeta_j \pp_{\zeta_j}\mathcal{K}(\zeta_j) \right|
		+
		\left| \zeta_j^2 \pp_{\zeta_j \zeta_j}\mathcal{K}(\zeta_j) \right| \lesssim  \1_{\{ \zeta_j\le 1\}}
		+
		\zeta_j^{-1}  \1_{\{ \zeta_j > 1\}} .
	\end{equation}
	Since
	\begin{equation*}
		\nabla_{y^{\J}} g_0 =
		\int_{-T}^{t(\tau_j)} \frac{\dot{p}_j(s)}{t(\tau_j)-s} \pp_{\zeta_j} \mathcal{K}(\zeta_j)  \frac{2\lambda_j^2(t(\tau_j)) y^{\J}}{t(\tau_j)-s} ds
		=
		\int_{-T}^{t(\tau_j)} \frac{\dot{p}_j(s)}{t(\tau_j)-s} \zeta_j \pp_{\zeta_j} \mathcal{K}(\zeta_j)  \frac{2y^{\J}}{1+|y^{\J}|^2} ds ,
	\end{equation*}
	by \eqref{cal-K-est} and \eqref{nonlocal-est}, we have $|\nabla_{y^{\J}} g_0| \lesssim 1$. By Lemma \ref{DMO-|DMO|-lem} $(2)$, then
	$[ g_0(\tilde{z}, \tilde{s})]_{\mathsf{|DMO|_x}(Q_2^{-}(0))}
	\lesssim 1$ and $m_{02}^{\J} \lesssim 1$. Thus,
	\begin{equation*}
		\big[ \big( \lambda_j^{2} (M_0^{\J})_{\mathbb{C}_j^{-1}}   \big)( \tilde{z}, \tilde{s} ) \big]_{\mathsf{|DMO|_x}(Q_2^{-}(0))}
		\lesssim
		\lambda_* \langle y_* \rangle^{-3}.
	\end{equation*}
	We complete the proof of \eqref{Hj-|DMO|x}.
\end{proof}

By \eqref{qd240724-3}, \eqref{qd240724-4}, \eqref{Hj-|DMO|x},
Proposition \ref{scaling-prop-0722}, and the scaling argument, due to the small quantity $\lambda_*^{\epsilon}$, we have $\mathbf{T}_{\rm{in}}^{[\cdot]}[ \mathbf{\Phi_{\rm{in} }^{[\cdot]} } ] \in \mathcal{B}_{\rm{in}}^{[\cdot]}$. Since $p_j,~ \xi^{\J}$ satisfy \eqref{lam-ansatz}, $\dot{\xi}^{\J}$ is H\"older continuous, and $\| \Phi_{\rm{in} }^{\J} \|_{{\rm in},\nu-\delta_0,l,\varsigma_{\rm{in} }}$, $\| \Phi_{\rm {out} }\|_{\sharp, \Theta,\alpha}$ $\lesssim 1$, then $\mathcal{H}^{\J}$ is H\"older continuous. Applying the Schauder estimate to \eqref{inner-eq} and by changing time variable from $t$ to $\tau_j$, we have $\mathbf{T}_{\rm{in}}^{[\cdot]}[ \mathbf{\Phi_{\rm{in} }^{[\cdot]} } ] \in C^{2+c,(2+c)/2}(\mathcal{D}_{2 C_{\lambda} R})$ with a small constant $c\in (0,1)$. By the Schauder fixed-point theorem, we can find a fixed point of $\mathbf{T}_{\rm{in}}^{[\cdot]}[ \mathbf{\Phi_{\rm{in} }^{[\cdot]} } ]$ in $ \mathcal{B}_{\rm{in}}^{[\cdot]}$.

{\textbf{Step 6.}}
Combining restrictions  \eqref{out-topo-para}, \eqref{Phi-|DMO|-req1}, \eqref{G-para}, \eqref{para-rho-con} for the outer problem, parameter assumptions in Proposition \ref{keyprop}, \eqref{qd24July05-1}, \eqref{xi-est} for reduced equations, and \eqref{inn-top0-para}, \eqref{qd240725-2}, \eqref{nonortho-para}, \eqref{H1-est-para}, \eqref{qd240725-4} for the inner problems,  we need to solve the following system of inequalities of
 parameters:
\begin{align*}
		&  \nu -\delta_0>1/2,
		\quad
		0<\Theta<\beta,
		\quad
		\Theta  +\beta + \delta_0 	-\nu <0,
		\quad 3\beta< 1+\Theta , \quad
		\beta(l+1)-1+\nu-\delta_0 -\Theta>0 ,
		\\
		&
		\Theta+	2\beta -1<0, \quad 0<\delta_0  < \nu<1 , \quad 	2\beta+\delta_0-\nu<0,
		\quad
		0<\alpha<1, \ \Theta < \alpha/2,
		\\
		&
		0< \sigma_0<\beta, \ \beta-\sigma_0 -\frac{\alpha}{2} <0, \
		1-\sigma_0-(1+\alpha)(1-\beta)<0,
		\
		\Theta +2\sigma_0-\beta <0,
		\\
		&
		0<A_{\rm{o,h}} <\min\{ \Theta+ (1-\beta)(1-\alpha), 1-2\sigma_0-\alpha(1-\beta), 1-\sigma_0-\frac{\alpha}{2} \},
		\\
		&
		\alpha_0 \in (0, 1/2),
		\quad
		\Theta\in(0,\flat),
		\quad	
		\varpi-1-2\Theta<0,
		\quad
		\tilde{m} < \Theta-\alpha(1-\beta),
		\\
		&
		0< 2\epsilon_{\xi} < \min\{ \Theta +\alpha \beta, \alpha,  2\nu-2\delta_0-1, \delta_0/6 \},
		\quad
		0<\varsigma_{\rm{in} }<1,
		\\
		&
		\delta_0 <6\beta, \quad
		\ell_0 \in (1,3),
		\quad
		2-\ell_0 \le -l,
		\quad
		\nu+\beta l -1<\delta_0 /6,
		\quad
		0<l<1,
		\\
		&
		1+\Theta+\alpha \beta -2\beta >\nu-\delta_0,
		\quad
		1+\alpha-2\beta >\nu-\delta_0,
		\\
		&
		2\beta <\nu-\delta_0,
		\quad
		\tilde{m} + (1+\alpha_0) \alpha/2 <1,
		\quad
		1+ \tilde{m} + (1+\alpha_0) \alpha/2  -2\beta >\nu-\delta_0,
		\\
		&
		\delta_0 (\ell_0 -1)/6 +1-2\beta >\nu-\delta_0,
	\end{align*}
where the constant
$\flat>0$ is given in Proposition \ref{keyprop}.
With the assistance of Mathematica, sound choices satisfying all the restrictions are given below, and the proof of Theorem \ref{thm} is completed.
\begin{align*}
%\begin{equation*}
%	\begin{aligned}
		&
		0 <\Theta<\min\Big\{ \frac{15}{2479}, \flat \Big\}, \quad \beta = \frac{15 + \Theta}{60}, \quad \sigma_0 = \frac{1}{50},
		\quad
		\alpha = \frac{9}{10},
		\quad
		\ell_0 = l+2,
		\\
		&
		3 \Theta < \delta_0 < \frac{1}{40} (-27 + 108 \beta + 120 \Theta),
		\quad
		1 - 2 \beta + \delta_{0} + \Theta < \nu < \frac{1}{3} (3 - 6 \beta + 4 \delta_0),
		\\
		&
		\frac{10}{9} \Big(-\frac{11}{10} + \frac{11 \beta}{5} - 2 \delta_0 + 2 \nu - 2 \Theta \Big) < \alpha_0 < \frac{1}{2},
		\quad
		\frac{1}{2} \Big(-\frac{29}{10} - \frac{9 \alpha_0}{10} + 4 \beta - 2 \delta_0 + 2 \nu \Big) <
		\tilde{m} < -\frac{9}{10} + \frac{9 \beta}{10} + \Theta,
		\\
		&
		\max\Big\{0, -1 + \frac{1-\nu+\delta_0+\Theta}{\beta} ,
		-1+ \frac{6}{\delta_0}(\nu-\delta_0 +2\beta-1) \Big\}
		< l <
		\min\Big\{1, \frac{1}{\beta} \big( \frac{\delta_0}{6} + 1-\nu\big) \Big\},
		\\
		&
		0<A_{\rm{o,h}} <\min\{ \Theta+ (1-\beta)(1-\alpha), 1-2\sigma_0-\alpha(1-\beta), 1-\sigma_0-\frac{\alpha}{2} \},
		\quad
		\varpi<1+2\Theta,
		\\
		& 0< 2\epsilon_{\xi} < \min\{ \Theta +\alpha \beta, \alpha,  2\nu-2\delta_0-1, \delta_0/6 \},
		\quad
		0<\varsigma_{\rm{in} }<1.
%	\end{aligned}
%\end{equation*}
\end{align*}

\subsection{Proof of Corollary \ref{cor-intro}}\label{converge-sec}

In this subsection, based on the weighted spaces for the solution constructed, we shall show the convergence results in Corollary \ref{cor-intro}.
Throughout this subsection, we adopt the usual Sobolev norm for the mapping between Euclidean spaces. $M>0$ is an arbitrary constant. Given a function $f(x,t)$, we use $f(t)$ to denote $f(x,t)$ for simplicity. A function $f(t) \rightarrow 0$ means that $f(t)$ converges to $0$ under some norms as $t\to T$.

Recall \eqref{u-def}. The solution in Theorem \ref{thm} has the form
\begin{align}\label{Phi-per-def}
	u(x,t)=U_*+\Phi_{\rm per},\quad
	\Phi_{\rm per}:= AU_*+\Phi-(\Phi\cdot U_*)U_*,
\end{align}
where $U_*$ is the multi-bubble profile defined in \eqref{def-U*}, and $A$ is given in \eqref{A-def}. Since
\begin{equation*}
	|\Phi|\ll 1,
	\quad
	\| |U_*(\cdot,t) |-1 \|_{L^\infty(\mathbb{R}^2)} \lesssim  \lambda_*,
	\quad
	|A| \lesssim \lambda_* +|\Phi|^2
\end{equation*}
by \eqref{U*-norm}, \eqref{A-est}, we obtain
\begin{equation*}%\label{Phiper-Linfty}
	\| \Phi_{\rm per} \|_{L^{\infty}(\R^2\times (0,T))} \ll 1 .
\end{equation*}
It is straightforward to get
\begin{equation*}
	\nabla_x \Phi_{\rm per} = U_* \nabla_x A
	+
	A \nabla_x U_*
	+ \nabla_x \Phi-(\Phi\cdot U_*) \nabla_x U_*
	- U_* \nabla_x(\Phi\cdot U_*),
\end{equation*}
where we denote $U_*\nabla_x f = (U_* \pp_{x_1} f , U_* \pp_{x_2} f)$ for a scalar function $f$. Applying \eqref{nab-A-rough} to $U_* \nabla_x A$;
\eqref{A-est}, \eqref{nabU*-est}, \eqref{toget-est-1} to $A \nabla_x U_* -(\Phi\cdot U_*) \nabla_x U_*$;
\eqref{nabla-Phi-upp}, \eqref{nab-Phi-cdot-U*} to $\nabla_x \Phi
- U_* \nabla_x(\Phi\cdot U_*)$, we have
\begin{equation*}
	\begin{aligned}
		|\nabla_x \Phi_{\rm per}|
		\lesssim \ &
		\sum_{j=1}^{N}
		\Big\{ 	\1_{\{ |x-q^{\J}|\le 3\lambda_* R \}}
		\Big[ 1 +
		\big(\lambda_*^{\nu-\delta_0-1}
		+
		|\ln(T-t)| \lambda_*^{\Theta} R
		\big)
		\langle \rho_j \rangle^{-l-1}    \Big]
		\\
		&
		+
		\1_{\{  3\lambda_* R < |x-q^{\J}| <3d_q \}}
		\Big\}
		+
		\1_{\{ \cap_{j=1}^N \{|x-q^{\J}| \ge  3d_q \} \}}.
	\end{aligned}
\end{equation*}
Here, $
\big( \lambda_*^{\nu-\delta_0-1} + |\ln(T-t)| \lambda_*^{\Theta} R  \big)
\| \1_{\{ |x-q^{\J}|\le 3\lambda_* R \}}
\langle \rho_j \rangle^{-l -1} \|_{L^2(\mathbb{R}^2)}
\lesssim  \lambda_*^{\nu-\delta_0} + |\ln(T-t)| \lambda_*^{\Theta+1} R $. Thus,
\begin{equation}\label{Phiper-nabla}
	\Phi_{\rm per}\in L^{\infty}\big( (0,T);H^1_{\rm loc}(\R^2)\big),
	\quad
	\lambda_*(t)\| \nabla_x \Phi_{\rm per} (t) \|_{L^{\infty}(\R^2) } \lesssim
	\lambda_*^{\epsilon}(t)
\end{equation}
with a small constant  $\epsilon>0$. We define
\begin{equation}\label{u*-def-1018}
	u_*(x) :=\Phi_{\rm per}(x,T)+W(\infty).
\end{equation}
By definition, we have
\begin{equation*}
	u(x,t)-u_*(x)-\sum_{j=1}^N Q_{\gamma_j}\Big[W\Big(\frac{x-\xi^{\J}}{\la_j}\Big)-W(\infty)\Big]=\Phi_{\rm per}(x,t)-\Phi_{\rm per}(x,T).
\end{equation*}

Claim: For $\Phi$ given in \eqref{u-def},
\begin{equation}\label{Phi-T}
	\Phi(x,T) = \sum_{j=1}^{N}
	\eta_{d_q}^{\J}(x,T) \Phi^{*{\J}}_0(|x-q^{\J}|,T)
	+\Phi_{\rm out}(x,T) ,
\end{equation}
\begin{equation}\label{conv-Phi}
	\| \Phi(t)-\Phi(T) \|_{ L^{\infty}(\R^2) } \rightarrow 0  ,
	\quad
	\|\nabla_x\Phi(t) -\nabla_x (\Phi(T) ) \|_{L^2(B_M)} \to 0 .
\end{equation}
Combining \eqref{out-upp}, \eqref{out-nabla-upp}, and \eqref{Phi*-0-j-upp}, we have
\begin{equation}\label{Phi-T-est}
	\| \Phi(T)\|_{L^\infty(\mathbb{R}^2)}
	\lesssim T^{\epsilon} + \| Z_* \|_{C^3(\R^2)},
	\quad
	\| \nabla_x \left(\Phi(T) \right) \|_{L^\infty(\mathbb{R}^2)}
	\lesssim 1.
\end{equation}

\begin{proof}[Proof of \eqref{Phi-T} and \eqref{conv-Phi}]
	
	For $\Phi_{\rm out}$ solved in   $B_{\rm{out}}$ defined in \eqref{out-space}, we have
	\begin{equation}\label{conv-outer}
		\begin{aligned}
			&\|\Phi_{\rm out}(t)-\Phi_{\rm out}(T)\|_{L^\infty(\mathbb{R}^2)}\lesssim |\ln(T-t)|\la_*^{\Theta+1} R +(T-t) ,
			\\
			&
			\|\nabla_x\Phi_{\rm out}(t)-\nabla_x\Phi_{\rm out}(T)\|_{L^\infty(\mathbb{R}^2)}\lesssim \la_*^{\Theta} +(T-t)^{\alpha/2} .
		\end{aligned}
	\end{equation}
	
	For $\Phi^{\J}_{\rm in}$  solved in  $B_{\rm{in}}^{\J}$ defined in  \eqref{inner-space}, since $0<l<1$, we have
	\begin{align}
		%	\begin{equation}
			%		\begin{aligned}
				&
				\|\Phi^{\J}_{\rm in}( t)\|_{L^\infty(\mathbb{R}^2)}
				\lesssim
				\lambda_*^{\nu-\delta_0} ,
				\nonumber
				\\
				&
				\int_{B_M}|\Phi^{\J}_{\rm in}(y^{\J},t)|^2 dx
				\lesssim
				\lambda_*^{2\nu-2\delta_0}
				\int_{B_M} \langle y^{\J} \rangle^{-2l} dx
				\lesssim
				\lambda_*^{2\nu-2\delta_0 +2}
				\int_{|y^{\J}| \le C(M) \lambda_*^{-1} } \langle y^{\J} \rangle^{-2l} d y^{\J}
				\lesssim
				\lambda_*^{2\nu-2\delta_0 +2l} ,
				\nonumber
				\\
				&
				\int_{B_M}
				\Big|\nabla_x
				\Big( \eta_R^{\J}(x,t) \Phi_{\rm in}^{\J}(y^{\J},t) \Big) \Big|^2 dx
				\lesssim
				\lambda_*^{2\nu-2\delta_0-2}
				\int_{B_M}
				\langle y^{\J} \rangle^{-2l-2}  dx
				\lesssim
				\lambda_*^{2\nu-2\delta_0}.
				\label{conv-inner}
				%		\end{aligned}
			%	\end{equation}
	\end{align}
	In particular, $\Phi^{\J}_{\rm in}( T)=0$, which implies  \eqref{Phi-T}.
	It is easy to get
	\begin{equation}\label{eta-dq-converg}
		\|\eta_{d_q}^{\J}(t) - \eta_{d_q}^{\J}(T) \|_{L^\infty(\mathbb{R}^2)} \rightarrow 0, \quad
		\| \nabla_x \eta_{d_q}^{\J}(t) - \nabla_x  \eta_{d_q}^{\J}(T) \|_{L^\infty(\mathbb{R}^2)} \rightarrow 0.
	\end{equation}
	
	Next, we consider $\Phi^{*{\J}}_0$ defined in \eqref{def-globalcorrection-J} with $\mu=3$. Recalling \eqref{zetaj}, we denote
	\begin{equation*}
		\mathfrak{A}:=\frac{a+ib}{4}, \quad \zeta_j(t):=\frac{|x-\xi^{\J}(t)|^2 +\lambda_j(t)^2}{t-s}, \quad \zeta_j(T) =\frac{|x-q^{\J}|^2}{T-s}.
	\end{equation*}
	By \eqref{lam-ansatz}, we have  $|\xi^{\J} -q^{\J}| \ll \lambda_*$.
	Recall \eqref{def-globalcorrection-J}, \eqref{def-Phi0}, \eqref{polar-coor} and $\mu=3$,
	\begin{equation}\label{Phi0*-formula}
		\Phi^{*{\J}}_0(|x-\xi^{\J} |,t)
		=
		\frac{-2|x-\xi^{\J}|^2 \big[ x_1-\xi^{\J}_1 + i \big( x_2-\xi^{\J}_2 \big) \big] }{\big( |x-\xi^{\J}|^3+ \la_j^3 \big) \big( |x-\xi^{\J}|^2+\lambda_j^2\big)^{1/2}}
		\Big[ \int_{-T}^t \dot{p}_j(s) \big( 1-e^{- \mathfrak{A} \zeta_j(t)} \big) ds,
		0 \Big]^{\tr},
	\end{equation}
	\begin{equation}\label{qd24Oct13-1}
		\begin{aligned}
			&
			\big| \Phi^{*{\J}}_0(|x-\xi^{\J} |,t)  - \Phi^{*{\J}}_0(|x-q^{\J} |,T) \big|
			\\
			\lesssim  \ &
			( |x-\xi^{\J}|+\lambda_j )^{-1}
			\Big|
			\int_{-T}^t \dot{p}_j(s) \big( 1-e^{- \mathfrak{A} \zeta_j(t)} \big) ds
			-
			\int_{-T}^T \dot{p}_j(s) \big( 1-e^{- \mathfrak{A} \zeta_j(T)} \big) ds \Big|
			\\
			&
			+
			\bigg|
			\frac{|x-\xi^{\J}|^2 \big[ x_1-\xi^{\J}_1 + i \big( x_2-\xi^{\J}_2 \big) \big] }{\big( |x-\xi^{\J}|^3+ \la_j^3 \big) \big( |x-\xi^{\J}|^2+\lambda_j^2\big)^{1/2}}
			-
			\frac{|x-q^{\J}|^2 \big[ x_1-q^{\J}_1 + i \big( x_2-q^{\J}_2 \big) \big] }{  |x-q^{\J}|^3    |x-q^{\J}| }
			\bigg|
			\Big| \int_{-T}^T \dot{p}_j(s) \big( 1-e^{- \mathfrak{A} \zeta_j(T)}  \big) ds  \Big|.
		\end{aligned}
	\end{equation}
	We will estimate the above term by term.	Denote $\tilde{r}_j :=| x-q^{\J} |$. Then $ |x-\xi^{\J}|+\lambda_j  \sim  \tilde{r}_j +\lambda_* $, and
	\begin{equation}\label{0520-1}
		\begin{aligned}
			&
			\bigg|
			\frac{|x-\xi^{\J}|^2 \big[ x_1-\xi^{\J}_1 + i \big( x_2-\xi^{\J}_2 \big) \big] }{\big( |x-\xi^{\J}|^3+ \la_j^3 \big) \big( |x-\xi^{\J}|^2+\lambda_j^2\big)^{1/2}}
			-
			\frac{|x-q^{\J}|^2 \big[ x_1-q^{\J}_1 + i \big( x_2-q^{\J}_2 \big) \big] }{  |x-q^{\J}|^3    |x-q^{\J}| }
			\bigg|
			\\
			\lesssim \ &
			\int_0^1
			\frac{| x- \theta \xi^{\J}-(1-\theta)q^{\J}|^2 (
				| \xi^{\J}-q^{\J}| +\lambda_j )  }{[ |x-\theta \xi^{\J}-(1-\theta)q^{\J} |+ \theta \la_j ]^4 }
			d\theta
			\lesssim
			(| \xi^{\J}-q^{\J} | +\lambda_* )
			\int_0^1 ( \tilde{r}_j + \theta \lambda_* )^{-2}
			d\theta.
		\end{aligned}
	\end{equation}

	By \eqref{nonlocal-est}, we have
	\begin{equation*}
		\Big| \int_{-T}^T \dot{p}_j(s) \big( 1- e^{ - \mathfrak{A} \zeta_j(T) }  \big) ds  \Big|
		=
		\tilde{r}_j^2
		\Big| \int_{-T}^T \frac{\dot{p}_j(s) }{T-s} \big( 1- e^{ - \mathfrak{A} \zeta_j(T) }  \big)   \zeta_j(T)^{-1} ds  \Big|
		\lesssim
		\tilde{r}_j^2.
	\end{equation*}
	\begin{equation*}
		\begin{aligned}
			&
			\int_{-T}^t \dot{p}_j(s) \big( 1-e^{- \mathfrak{A} \zeta_j(t) } \big) ds
			-
			\int_{-T}^T \dot{p}_j(s) \big( 1-e^{ - \mathfrak{A}  \zeta_j(T) } \big)   ds
			\\
			= \ &
			\Big[\int_{-T}^{ (t-\la_*^2-\tilde{r}_j^2 )_{+}} +\int_{ (t-\la_*^2-\tilde{r}_j^2 )_{+}}^{(t-\tilde{r}_j^2 )_{+}}+\int_{ (t-\tilde{r}_j^2 )_{+}}^t\Big] \dot{p}_j(s) \big( e^{ - \mathfrak{A}  \zeta_j(T) }-e^{- \mathfrak{A} \zeta_j(t) } \big) ds
			-
			\int_t^T \dot{p}_j(s) \big( 1-e^{ - \mathfrak{A}  \zeta_j(T) } \big)   ds,
		\end{aligned}
	\end{equation*}
	where for the first part,
	\begin{align*}
		%	\begin{equation*}
			%		\begin{aligned}
				& \Big|\int_{-T}^{ (t-\la_*^2-\tilde{r}_j^2 )_{+}} \dot {p}_j(s)
				\big(e^{-\mathfrak{A}\zeta_j(T)}-e^{-\mathfrak{A}\zeta_j(t)} \big) ds\Big|
				\lesssim  \int_{-T}^{(t-\la_*^2-\tilde{r}_j^2 )_{+}} |\dot {\lambda}_*(s) |
				| \zeta_j(T) - \zeta_j(t) |
				ds \\
				\lesssim \ & \int_{-T}^{(t-\la_*^2-\tilde{r}_j^2 )_{+}} |\dot {\lambda}_*(s)|
				\Big[
				\tilde{r}_j^2 \Big( \frac{1}{t-s} - \frac{1}{T-s} \Big)
				+
				\frac{\tilde{r}_j | \xi^{\J}(t) - q^{\J} | + \lambda_*(t)^2 }{t-s}
				\Big]
				ds \\
				\lesssim \ &
				\tilde{r}_j^2
				\ln \Big(
				\frac{(t+T) \big[ T - \big(t-\la_*^2-\tilde{r}_j^2\big)_{+} \big] }{2T \big[t- \big(t-\la_*^2-\tilde{r}_j^2\big)_{+} \big] }
				\Big)
				+
				( \tilde{r}_j | \xi^{\J} - q^{\J} | + \lambda_*^2 )
				\ln \Big( \frac{t+T}{t- \big(t-\la_*^2-\tilde{r}_j^2\big)_{+}}\Big).
				%		\end{aligned}
			%	\end{equation*}
	\end{align*}
	For the second part,
	\begin{equation*}
		\Big| \int_{(t-\la_*^2-\tilde{r}_j^2)_{+}}^{(t-\tilde{r}_j^2)_{+}} \dot {p}_j(s)
		\big(e^{-\mathfrak{A}\zeta_j(T)}-e^{-\mathfrak{A}\zeta_j(t)}
		\big)ds\Big|
		\lesssim
		\lambda_*^2.
	\end{equation*}
	For the third part,
	if $\tilde{r}_j \leq \sqrt{T-t}$, we have
	\begin{equation*}
		\Big| \int_{ (t-\tilde{r}_j^2 )_{+}}^t \dot {p}_j(s)
		\big(e^{-\mathfrak{A}\zeta_j(T)}-e^{-\mathfrak{A}\zeta_j(t)}
		\big)ds \Big|
		\lesssim
		\tilde{r}_j^2;
	\end{equation*}
	if $\tilde{r}_j > \sqrt{T-t} $, $t>T/2$, by similar calculations as in the first part, we have
	\begin{align*}
			&
			\Big|
			\Big(
			\int_{  (t-\tilde{r}_j^2 )_{+}}^{t-(T-t)}
			+
			\int_{t-(T-t)}^t
			\Big) \dot {p}_j(s)
			\big(e^{-\mathfrak{A}\zeta_j(T)}-e^{-\mathfrak{A}\zeta_j(t)}
			\big)ds
			\Big|
			\\
			\lesssim \ &
			\int_{  (t-\tilde{r}_j^2 )_{+}}^{t-(T-t)}
			|\dot{\lambda}_*(s) |
			\Big[
			\tilde{r}_j^2 \Big( \frac{1}{t-s} - \frac{1}{T-s} \Big)
			+
			\frac{\tilde{r}_j | \xi^{\J}(t) - q^{\J}| + \lambda_*(t)^2 }{t-s}
			\Big]
			ds
			+
			\int_{t-(T-t)}^t | \dot{\lambda}_*(s) | ds
			\\
			\lesssim \ &
			\tilde{r}_j^2  \frac{|\ln T|}{\ln^2(T-t)}
			+
			( \tilde{r}_j | \xi^{\J} - q^{\J} | + \lambda_*^2 )
			\ln\Big( \frac{t-  (t-\tilde{r}_j^2 )_{+}}{T-t}\Big)
			+
			\lambda_*,
		\end{align*}
	where we used $ \frac{1}{t-s} - \frac{1}{T-s}  \sim (T-t) (T-s)^{-2}$ for $s\le t-(T-t)$.
	
	For the last part,
	by \eqref{lam-ansatz},
	%\begin{align*}
	\begin{equation*}
		\begin{aligned}
			\Big|\int_t^T \dot p(s)\big(1-e^{-\mathfrak{A}\zeta_j(T)} \big) ds\Big|
			\lesssim \ &
			\Big[
			\Big(
			\tilde{r}_j^2 \int_t^{T-\tilde{r}_j ^2}\frac{| \dot {\lambda}_*(s) | }{T-s}ds+
			\int_{T-\tilde{r}_j ^2}^T |\dot{\lambda}_*(s) |ds \Big) \1_{\{ \tilde{r}_j  \le \sqrt{T-t}  \}}
			+
			\lambda_* \1_{\{ \tilde{r}_j  > \sqrt{T-t} \}}
			\Big]
			\\
			\lesssim \ &
			\tilde{r}_j ^2\1_{\{ \tilde{r}_j  \le  \sqrt{T-t}\}}
			+
			\lambda_*\1_{\{ \tilde{r}_j  > \sqrt{T-t}\}}.
		\end{aligned}
	\end{equation*}
	%\end{align*}

	In sum,	collecting the above estimates, we conclude that for $t>T/2$,
	\begin{align*}
			&
			\big| \Phi^{*{\J}}_0(|x-\xi^{\J} |,t)  - \Phi^{*{\J}}_0(|x-q^{\J} |,T) \big|
			\\
			\lesssim \ &
			\tilde{r}_j
			\ln\big( 2T (t+T)^{-1} \big)
			+
			\tilde{r}_j
			\ln \Big(
			\frac{ T - (t-\la_*^2-\tilde{r}_j^2 )_{+} }{t- (t-\la_*^2-\tilde{r}_j^2 )_{+} }
			\Big)
			+
			\lambda_* \ln \Big( \frac{t+T}{t- (t-\la_*^2-\tilde{r}_j^2 )_{+}}\Big)
			\\
			& + \lambda_*
			+
			\1_{\{ \tilde{r}_j \leq \sqrt{T-t} \}}
			\tilde{r}_j
			+
			\1_{\{ \tilde{r}_j > \sqrt{T-t} \}}
			\Big[
			\tilde{r}_j \frac{|\ln T|}{\ln^2(T-t)}
			+
			\lambda_* \ln\Big( \frac{t-  (t-\tilde{r}_j^2 )_{+}}{T-t}\Big)
			+
			\lambda_* (T-t)^{-1/2}
			\Big].
		\end{align*}
	Here, if $t-\la_*^2-\tilde{r}_j^2 \ge 0$,
	\begin{equation*}
		\begin{aligned}
			&
			\tilde{r}_j
			\ln \Big(
			\frac{ T - (t-\la_*^2-\tilde{r}_j^2 )_{+} }{t- (t-\la_*^2-\tilde{r}_j^2 )_{+} }
			\Big)
			=
			\tilde{r}_j
			\ln \Big(
			1+
			\frac{T - t}{\la_*^2+\tilde{r}_j^2 }
			\Big)
			\\
			\lesssim \ &
			\tilde{r}_j |\ln(T-t)| \1_{\{\tilde{r}_j < \sqrt{T-t}\}}
			+
			\tilde{r}_j^{-1} (T - t) \1_{\{\tilde{r}_j \ge \sqrt{T-t}\}}
			\lesssim
			\sqrt{T-t} |\ln(T-t)|,
		\end{aligned}
	\end{equation*}
	since $\frac{T - t}{2(\sqrt{T-t})^2} \le \frac{T - t}{\la_*^2+\tilde{r}_j^2 } \le \frac{T - t}{\la_*^2}$ if $ \tilde{r}_j < \sqrt{T-t}$. Thus,
	\begin{equation}\label{qd24Oct13-3}
		\tilde{r}_j
		\ln \Big(
		\frac{ T - (t-\la_*^2-\tilde{r}_j^2 )_{+} }{t- (t-\la_*^2-\tilde{r}_j^2 )_{+} }
		\Big)
		\lesssim
		\1_{\{ \tilde{r}_j^2 \le t-\la_*^2 \}}
		\sqrt{T-t} |\ln(T-t)|
		+
		\1_{\{ \tilde{r}_j^2 > t-\la_*^2 \}} \tilde{r}_j
		\ln(T/t).
	\end{equation}
	Similarly,
	\begin{equation}\label{qd24Oct13-4}
		\begin{aligned}
			& \lambda_* \ln \Big( \frac{t+T}{t- (t-\la_*^2-\tilde{r}_j^2 )_{+}}\Big)
			\lesssim
			\1_{\{ \tilde{r}_j^2 \le t-\la_*^2 \}}
			\lambda_* |\ln(T-t)| + \1_{\{ \tilde{r}_j^2 > t-\la_*^2 \}} \lambda_* \ln \Big( \frac{t+T}{t} \Big),
			\\
			&
			\1_{\{ \tilde{r}_j > \sqrt{T-t} \}}
			\lambda_* \ln\Big( \frac{t-  (t-\tilde{r}_j^2 )_{+}}{T-t}\Big)
			=
			\1_{\{ \sqrt{T-t}< \tilde{r}_j \le \sqrt{t} \}}
			\lambda_* \ln\Big( \frac{\tilde{r}_j^2}{T-t}\Big)
			+
			\1_{\{ \tilde{r}_j > \max\{\sqrt{T-t},\sqrt{t} \} \} }
			\lambda_* \ln\Big( \frac{t}{T-t}\Big).
		\end{aligned}
	\end{equation}
	It follows that	for all $\tilde{r}_j^2  \le M$, we have
	\begin{equation}\label{conv-nonlocal}
		\big| \Phi^{*{\J}}_0(|x-\xi^{\J} |,t)  - \Phi^{*{\J}}_0(|x-q^{\J} |,T) \big| \rightarrow 0
		\mbox{ \ uniformly as \ }  t\rightarrow T,
	\end{equation}

	From \eqref{conv-outer}, \eqref{conv-inner}, \eqref{eta-dq-converg}, and \eqref{conv-nonlocal}, we obtain the first part of \eqref{conv-Phi}.

	Using \eqref{Phi0*-formula}, we have
	\begin{equation*}
		\pp_{x_1}
		\Phi^{*{\J}}_0(|x-\xi^{\J} |,t)
		=
		-2 f_1(t)
		\Big[ \int_{-T}^t \dot{p}_j(s) \big( 1-e^{- \mathfrak{A} \zeta_j(t)} \big) ds,
		0 \Big]^{\tr}
		+
		f_2(t)
		\Big[ \int_{-T}^t \dot{p}_j(s)  \zeta_j(t) e^{- \mathfrak{A} \zeta_j(t)}   ds,
		0 \Big]^{\tr},
	\end{equation*}
	where
	\begin{align*}
			f_1(t) :=  &
			\frac{2(x_1-\xi^{\J}_1) \big[ x_1-\xi^{\J}_1 + i \big( x_2-\xi^{\J}_2 \big) \big] }{\big( |x-\xi^{\J}|^3+ \la_j^3 \big) \big( |x-\xi^{\J}|^2+\lambda_j^2\big)^{1/2}}
			+
			\frac{|x-\xi^{\J}|^2  }{\big( |x-\xi^{\J}|^3+ \la_j^3 \big) \big( |x-\xi^{\J}|^2+\lambda_j^2\big)^{1/2}}
			\\
			& +
			\frac{|x-\xi^{\J}|^2 \big[ x_1-\xi^{\J}_1 + i \big( x_2-\xi^{\J}_2 \big) \big] }{ \big( |x-\xi^{\J}|^2+\lambda_j^2\big)^{1/2}}
			(-1) \big( |x-\xi^{\J}|^3+ \la_j^3 \big)^{-2}
			3|x-\xi^{\J}| (x_1-\xi_1^{\J})
			\\
			& +
			\frac{|x-\xi^{\J}|^2 \big[ x_1-\xi^{\J}_1 + i \big( x_2-\xi^{\J}_2 \big) \big] }{  |x-\xi^{\J}|^3+ \la_j^3   }
			(-1)
			\big( |x-\xi^{\J}|^2+\lambda_j^2\big)^{-3/2} (x_1 -\xi_1^{\J} ),
	\\
		f_2(t) := & \frac{-2|x-\xi^{\J}|^2 \big[ x_1-\xi^{\J}_1 + i \big( x_2-\xi^{\J}_2 \big) \big] }{\big( |x-\xi^{\J}|^3+ \la_j^3 \big) \big( |x-\xi^{\J}|^2+\lambda_j^2\big)^{1/2}}
		2\mathfrak{A}
		\frac{x_1-\xi^{\J}_1}{|x-\xi^{\J}|^2+\lambda_j^2}.
	\end{align*}
	It is easy to get $ | f_1(t)| + | f_2(t)| \lesssim ( \tilde{r}_j + \lambda_*)^{-2}$.
	Similar to \eqref{0520-1}, we have
	\begin{equation*}
		| f_1(t) -f_1(T)| + | f_2(t) -f_2(T)|
		\lesssim
		\lambda_*
		\int_0^1
		( \tilde{r}_j + \theta \lambda_* )^{-3}
		d\theta.
	\end{equation*}
	$\int_{-T}^t \dot{p}_j(s) ( 1-e^{- \mathfrak{A} \zeta_j(t)} ) ds$ has been dealt with in the estimate of \eqref{qd24Oct13-1}.
	$\int_{-T}^t \dot{p}_j(s)  \zeta_j(t) e^{- \mathfrak{A} \zeta_j(t)}   ds$ can be handled similarly. In sum, for $t>T/2$,
	\begin{align*}
			&
			| \pp_{x_1} \Phi^{*{\J}}_0(|x-\xi^{\J} |,t)  - \pp_{x_1} \Phi^{*{\J}}_0(|x-q^{\J} |,T) | \lesssim \ln( 1+ \tilde{r}_j^{-1} \lambda_*)
			\\
			& \quad+
			(\tilde{r}_j+\lambda_*)^{-1}
			\Big\{ \tilde{r}_j \ln\big( 2T (t+T)^{-1}  \big)+
			\tilde{r}_j
			\ln \Big(
			\frac{T - (t-\la_*^2-\tilde{r}_j^2)_{+}   }{ t- (t-\la_*^2-\tilde{r}_j^2)_{+}  }
			\Big)
			+
			\lambda_* \ln \Big( \frac{t+T}{t- (t-\la_*^2-\tilde{r}_j^2)_{+}}\Big)
			\\
			& \quad+ \lambda_*
			+
			\1_{\{ \tilde{r}_j \leq \sqrt{T-t} \}}
			\tilde{r}_j
			+
			\1_{\{ \tilde{r}_j > \sqrt{T-t} \}}
			\Big[
			\tilde{r}_j \frac{|\ln T|}{\ln^2(T-t)}
			+
			\lambda_* \ln\Big( \frac{t-  (t-\tilde{r}_j^2 )_{+}}{T-t}\Big)
			+
			\lambda_*  (T-t)^{-1/2}
			\Big]  \Big\}.
		\end{align*}
	And $\pp_{x_2} \Phi^{*{\J}}_0(|x-\xi^{\J} |,t)$ can be dealt with similarly. Using \eqref{qd24Oct13-3}, \eqref{qd24Oct13-4}, and $
	\int_{\tilde{r}_j \le M} ( \tilde{r}_j+\lambda_*)^{-2} dx
	\lesssim
	\ln(M+2) + |\ln \lambda_*|$, $(\tilde{r}_j+\lambda_*)^{-1} \tilde{r}_j \le 1$,  we get
	\begin{equation}\label{qd24Oct14-1}
		\int_{B_M}\Big|
		\nabla_x \Phi^{*{\J}}_0(|x-\xi^{\J} |,t)  -
		\nabla_x \Phi^{*{\J}}_0(|x-q^{\J} |,T)
		\Big|^2 dx\to 0.
	\end{equation}
	
	Applying \eqref{conv-outer} to $\Phi_{\rm out}$, \eqref{conv-inner} to $\eta_R^{\J} Q_{\gamma_j}\Phi_{\rm in}^{\J}$, and \eqref{eta-dq-converg}, \eqref{Phi*-0-j-upp}, \eqref{conv-nonlocal}, \eqref{qd24Oct14-1} to $\eta_{d_q}^{\J} \Phi^{*{\J}}_0$, we conclude the second part of \eqref{conv-Phi}.
\end{proof}

Recall $U_*$ given in \eqref{def-U*}. The pointwise limit as $t$ goes to $T$ is given by
\begin{equation*}
	U_*(T) =
	\begin{cases}
		[0,0,1]^{\tr}=U_{\infty}, &\mbox{ \ if \ } x\not\in \{q^{\J} \ | \  j=1,2,\dots,N \}
		\\
		[0,0,-1 ]^{\tr},
		&\mbox{ \ if \ } x \in \{q^{\J} \ | \  j=1,2,\dots,N \}.
	\end{cases}
\end{equation*}
$U_*(t)$ does not converge in $L_{\rm{loc}}^\infty$ since $U_*(T)$ is not continuous. Instead,
\begin{equation}\label{0521-1}
	\| U_*(\cdot,t) -U_\infty \|_{L^\infty (\cap_{j=1}^N \{ |x-q^{\J}| \ge  3\lambda_*(t) R(t) \} )} \to 0.
\end{equation}

By \eqref{Phi-upp}, one has
\begin{equation}\label{qd24Oct18-4}
	\begin{aligned}
		&
		\Phi(q^{\J},T) = 0, \quad j=1,\dots,N, \mbox{ \ and then \ }
		((\Phi\cdot U_*)U_*)(T) = (\Phi(T)\cdot U_\infty)U_\infty,
		\\
		&
		\| (\Phi\cdot U_*)U_* - ( (\Phi\cdot U_*)U_* )(T)\|_{L^\infty(\mathbb{R}^2) }
		\le
		\| (\Phi\cdot (U_* - U_{\infty} ) )U_*   \|_{L^\infty(\mathbb{R}^2) }
		\\
		&\quad
		+
		\| ( \Phi\cdot U_{\infty} ) ( U_*-U_\infty ) \|_{L^\infty(\mathbb{R}^2) }
		+
		\| [ ( \Phi-\Phi(T) ) \cdot U_{\infty} ] U_{\infty}
		\|_{L^\infty(\mathbb{R}^2) }
		\to 0,
	\end{aligned}
\end{equation}
where we used \eqref{Phi-upp}, \eqref{0521-1}, \eqref{conv-Phi} in the last step.
\begin{align}
%\begin{equation}
%	\begin{aligned}
		&
		\| \pp_{x_i} [(\Phi\cdot U_*)U_* ] - \pp_{x_i}
		[ ((\Phi\cdot U_*)U_* )(T) ]
		\|_{L^2(B_M) }
		\nonumber
		\\
		\le \ &
		\| ( \Phi\cdot U_* ) \pp_{x_i} U_*
		\|_{L^2(B_M) }
		+
		\|
		( \Phi\cdot \pp_{x_i}U_* ) U_*
		\|_{L^2(B_M) }
		+
		\| [ \pp_{x_i} \Phi \cdot ( U_* -U_\infty ) ]U_*
		\|_{L^2(B_M) }
		\nonumber
		\\
		&
		+
		\|
		( \pp_{x_i} \Phi \cdot U_\infty ) ( U_*-U_\infty )
		\|_{L^2(B_M) }
		+
		\|  [ ( \pp_{x_i} \Phi - \pp_{x_i} \Phi(T) )\cdot U_\infty ] U_\infty
		\|_{L^2(B_M) } \to 0,
		\label{qd24Oct18-5}
%	\end{aligned}
%\end{equation}
\end{align}
where for the last step, we used $\| \langle \rho_j\rangle^{-1} \|_{L^2(B_M)} \le C(M) \lambda_* |\ln \lambda_*|^{1/2}$ and applied \eqref{nabU*-est}, \eqref{Phi-upp} to $\|  (\Phi\cdot U_*) \pp_{x_i} U_*
\|_{L^2(B_M) }
+
\|
( \Phi\cdot \pp_{x_i}U_* ) U_*
\|_{L^2(B_M) }$; \eqref{nabla-Phi-upp}, \eqref{0521-1} to $\| [ \pp_{x_i} \Phi \cdot ( U_* -U_\infty ) ]U_*
\|_{L^2(B_M) }
+
\|
( \pp_{x_i} \Phi \cdot U_\infty ) ( U_*-U_\infty )
\|_{L^2(B_M) }$; \eqref{conv-Phi} to $\|
[ ( \pp_{x_i} \Phi - \pp_{x_i} \Phi(T) )\cdot U_\infty ] U_\infty
\|_{L^2(B_M) }$.

\medskip

Recall $A$ defined in \eqref{A-def} and $\Phi(q^{\J},T) = 0$. It is straightforward to get
\begin{align}
%\begin{equation}
%	\begin{aligned}
		&
		A(T)
		=  [
		1  -|\Phi(T)|^2 + (\Phi(T)\cdot U_{\infty})^2 ]^{1/2}
		-
		1,
		\quad
		A(q^{\J},T)=0, \ j=1,2,\dots, N, \quad
		(A U_*)(T) = A(T) U_{\infty} ,
		\nonumber
		\\
		&
		\nabla_x  (A(T) ) =  -\frac{ \nabla_x\big[ \big| \big(\Pi_{U_*^{\perp}}\Phi\big)(T)\big|^2\big] }{2(1+A(T)) } ,
		\quad
		\big| \big(\Pi_{U_*^{\perp}}\Phi \big)(T)\big|^2
		=
		| \Phi(T)|^2 -
		( \Phi(T) \cdot U_{\infty} )^2.
		\label{AT-prop}
%	\end{aligned}
%\end{equation}
\end{align}

Note that
\begin{equation*}
	|\Pi_{U_*^{\perp}}\Phi|^2
	=
	|\Phi-(\Phi\cdot U_*)U_*|^2 =
	|\Phi|^2 -(\Phi\cdot U_*)^2 +
	(\Phi\cdot U_*)^2 ( |U_*|^2-1).
\end{equation*}

By \eqref{Phi-upp}, \eqref{conv-Phi}, and \eqref{0521-1}, we have
\begin{equation*}
	\begin{aligned}
		&
		\| (\Phi\cdot U_*)(t) -  (\Phi\cdot U_*)(T) \|_{L^\infty (\mathbb{R}^2)}
		\le
		\| \Phi(t) \cdot \left( U_*(t) - U_*(T)\right) \|_{L^\infty (\cap_{j=1}^N \{ |x-q^{\J}| > 3\lambda_*(t) R(t) \} )}
		\\
		&\quad
		+
		\| \Phi(t) \cdot \left( U_*(t) - U_*(T)\right) \|_{L^\infty (\cup_{j=1}^N \{ |x-q^{\J}|\le 3\lambda_*(t) R(t) \} )}
		+
		\| \left( \Phi(t) -\Phi(T)\right)  \cdot U_*(T) \|_{L^\infty (\mathbb{R}^2)}
		\to 0.
	\end{aligned}
\end{equation*}
We can handle the convergence of the other terms involved in $A$  directly. Then
\begin{equation}\label{eqn-574574574}
	\| A(t) -A(T) \|_{L^\infty(\mathbb{R}^2)} \to 0.
\end{equation}

Similarly, using \eqref{A-est}, \eqref{Phi-upp}, \eqref{0521-1}, and \eqref{eqn-574574574}, we have
\begin{equation}\label{qd24Oct18-6}
	\begin{aligned}
		&
		\| (A U_*)(t) -  (A U_*)(T) \|_{L^\infty (\mathbb{R}^2)}
		\le
		\| A(t) ( U_*(t) - U_*(T) ) \|_{L^\infty (\cap_{j=1}^N \{ |x-q^{\J}| > 3\lambda_*(t) R(t) \} )}
		\\
		& \quad
		+
		\| A(t) ( U_*(t) - U_*(T) ) \|_{L^\infty (\cup_{j=1}^N \{ |x-q^{\J}|\le 3\lambda_*(t) R(t) \} )}
		+
		\| ( A(t) -A(T) )    U_*(T) \|_{L^\infty (\mathbb{R}^2)}
		\to 0.
	\end{aligned}
\end{equation}

\medskip

Using $(A U_*)(T) = A(T) U_{\infty}$, we write
\begin{align*}
	\pp_{x_i} ( (A U_*)(t)- (A U_*)(T) )
	=
	(U_*(t)-U_{\infty} )  \pp_{x_i}  A(t)
	+
	A(t) \pp_{x_i} U_*(t)
	+
	U_\infty \pp_{x_i}  (A(t)-A(T) )
	,
	\quad
	i=1,2
\end{align*}
and consider the following three terms
\begin{equation*}
	\begin{aligned}
		&
		I_1:=\int_{B_M} |\nabla_x A(x,t) |^2|U_*(x,t)-U_\infty|^2 dx,
		\quad
		I_2:=\int_{B_M} |A(x,t)|^2 |\nabla_x U_*(x,t) |^2 dx,
		\\
		&
		I_3:=\int_{B_M} | \nabla_x (A(x,t)-A(x,T) ) | ^2 dx.
	\end{aligned}
\end{equation*}
Using \eqref{nab-A-rough}, \eqref{0521-1}, and $\| \langle \rho_j\rangle^{-1} \|_{L^2(B_M)} \le C(M) \lambda_* |\ln \lambda_*|^{1/2}$, we have $I_1 \rightarrow 0$.
By \eqref{A-est}, \eqref{Phi-upp}, \eqref{nabU*-est}, and $\| \langle \rho_j\rangle^{-1} \|_{L^2(B_M)} \le C(M) \lambda_* |\ln \lambda_*|^{1/2}$, we have $I_2 \rightarrow 0$.
For $I_3$, using $\nabla_x (A(T) )$ in \eqref{AT-prop} and $\nabla_x A$ given in \eqref{nab-A-def}, we have
\begin{align*}
		& |\nabla_x A-\nabla_x ( A(T) ) |
		\lesssim  \Big| \Big[(1+A)^2\nabla_x (|U_*|^2)+ \nabla_x\big(\big|\Pi_{U_*^{\perp}}\Phi\big|^2\big)+2(1+A)\nabla_x( U_*  \cdot \Pi_{U_*^{\perp}}\Phi)\Big]
		( 1+A(T) )
		\\
		&- \nabla_x \Big[ \big| \big(\Pi_{U_*^{\perp}}\Phi\big)(T) \big|^2 \Big]   \Big[ (1+A)|U_*|^2+  U_* \cdot
		\Pi_{U_*^{\perp}}\Phi  \Big] \Big|
		\\
		= \ &
		\Big| (1+A)^2\nabla_x (|U_*|^2)
		( 1+A(T) )
		+
		2(1+A)\nabla_x( U_*  \cdot \Pi_{U_*^{\perp}}\Phi)
		( 1+A(T) )
		\\
		&
		+
		\Big[ \nabla_x \Big(\big|\Pi_{U_*^{\perp}}\Phi\big|^2\Big)
		-
		\nabla_x\Big( \big| \big(\Pi_{U_*^{\perp}}\Phi\big)(T) \big|^2 \Big) \Big] (1+A(T))
		+
		\nabla_x\Big( \big| \big(\Pi_{U_*^{\perp}}\Phi\big)(T) \big|^2 \Big) (A(T)-A)
		\\
		& +
		\nabla_x\Big( \big| \big(\Pi_{U_*^{\perp}}\Phi\big)(T) \big|^2 \Big) (1+A) (1-|U_*|^2 )
		-
		\nabla_x\Big( \big| \big(\Pi_{U_*^{\perp}}\Phi\big)(T) \big|^2 \Big)  ( 1-|U_*|^2 ) ( \Phi\cdot U_* ) \Big|.
	\end{align*}

\noindent $\bullet$ By \eqref{U*cdot-nabU*-split}, we have
$ \nabla_x (|U_*|^2)  \to 0$ in $L^2_{\rm loc}$.

\noindent $\bullet$ By \eqref{U*-norm}, \eqref{nab-Phi-cdot-U*}, \eqref{Phi-cdot-U*-upp}, and \eqref{U*cdot-nabU*-split},
\begin{equation*}
	| \nabla_x( U_*  \cdot \Pi_{U_*^{\perp}}\Phi) |
	=
	| ( 1-|U_*|^2 ) \nabla_x ( \Phi\cdot U_* )
	-2 ( \Phi\cdot U_* ) U_*\cdot \nabla_x U_* |
	\lesssim  \lambda_*^{1/2}.
\end{equation*}

\noindent $\bullet$  Note that
\begin{equation*}
	\nabla_x (|\Pi_{U_*^{\perp}}\Phi|^2 )
	=
	2\Phi\cdot \nabla_x \Phi +
	( \Phi\cdot U_* )^2  \nabla_x ( |U_*|^2 )
	+
	2( | U_*|^2 -2)
	( \Phi\cdot U_* )
	\nabla_x ( \Phi\cdot U_* ).
\end{equation*}
From \eqref{AT-prop},  we have
\begin{equation*}
	\begin{aligned}
		& \big|\pp_{x_i}\big( \big|\Pi_{U_*^{\perp}}\Phi\big|^2\big)-
		\pp_{x_i}\big( \big|\big(\Pi_{U_*^{\perp}}\Phi\big)(T) \big|^2 \big) \big|
		\le
		2 |\Phi\cdot \pp_{x_i} \Phi-
		\Phi(T) \cdot \pp_{x_i} (\Phi(T) ) |
		+
		\big| ( \Phi\cdot U_* )^2  \pp_{x_i} \big( | U_*|^2 \big)  \big|
		\\
		&\quad
		+
		2 \big| \big( | U_*|^2 -1 \big)
		( \Phi\cdot U_* ) \pp_{x_i} ( \Phi\cdot U_* )    \big|
		+
		2 \big| - ( \Phi\cdot U_* ) \pp_{x_i} ( \Phi\cdot U_* )
		+
		( \Phi(T) \cdot U_{\infty} ) [ \pp_{x_i} (\Phi(T) )\cdot U_{\infty} ]   \big|.
	\end{aligned}
\end{equation*}
By \eqref{conv-Phi} and \eqref{Phi-T-est},
\begin{equation*}
	\|\Phi\cdot \pp_{x_i} \Phi-
	\Phi(T) \cdot \pp_{x_i} (\Phi(T) )   \|_{L^2(B_M)} \rightarrow 0.
\end{equation*}
By \eqref{Phi-upp} and \eqref{U*cdot-nabU*-split},
\begin{equation*}
	\| ( \Phi\cdot U_* )^2  \pp_{x_i} ( |U_*|^2 )  \|_{L^2(B_M)}
	\lesssim
	\| | \Phi |^2
	| \pp_{x_i} ( |U_*|^2) |  \|_{L^2(B_M)} \to 0.
\end{equation*}
By \eqref{U*-norm}, \eqref{Phi-cdot-U*-upp}, and \eqref{nab-Phi-cdot-U*},
\begin{equation*}
	\| (|U_*|^2 -1)
	(\Phi\cdot U_*) \pp_{x_i} ( \Phi\cdot U_*) \|_{L^\infty(\mathbb{R}^2)}
	\lesssim
	\lambda_*
	\|
	( \Phi\cdot U_* ) \pp_{x_i} ( \Phi\cdot U_* ) \|_{L^\infty(\mathbb{R}^2)}
	\to 0.
\end{equation*}
Finally,
\begin{align*}
		&
		\| -
		( \Phi\cdot U_* ) \pp_{x_i} ( \Phi\cdot U_*)
		+
		( \Phi(T) \cdot U_{\infty}) [ \pp_{x_i} (\Phi(T) )\cdot U_{\infty} ]   \|_{L^2(B_M)}
		\\
		\le \ &
		\|
		[ ( \Phi(T) -\Phi)\cdot U_* ] \pp_{x_i} ( \Phi\cdot U_* )
		\|_{L^2(B_M)}
		+
		\|
		[ \Phi(T) \cdot (U_{\infty}-U_* ) ] \pp_{x_i} ( \Phi\cdot U_*)
		\|_{L^2(B_M)}
		\\
		& +
		\|
		( \Phi(T) \cdot U_{\infty} ) ( \Phi\cdot \pp_{x_i} U_* )
		\|_{L^2(B_M)}
		+
		\| ( \Phi(T) \cdot U_{\infty}) [ \pp_{x_i} \Phi\cdot  (U_{\infty}-U_* ) ]
		\|_{L^2(B_M)}
		\\
		& +
		\| ( \Phi(T) \cdot U_{\infty}) [ ( \pp_{x_i} ( \Phi(T) ) - \pp_{x_i} \Phi )\cdot   U_{\infty} ]
		\|_{L^2(B_M)}
		\to 0,
	\end{align*}
where we used \eqref{conv-Phi}, \eqref{nab-Phi-cdot-U*}; \eqref{Phi-T-est}, \eqref{0521-1}; \eqref{Phi-upp}, \eqref{nabU*-est}; \eqref{nabla-Phi-upp}, \eqref{0521-1}; \eqref{conv-Phi} in order for the last step. In sum,
\begin{equation*}
	\|\pp_{x_i} ( |\Pi_{U_*^{\perp}}\Phi |^2 )-
	\pp_{x_i} (  | (\Pi_{U_*^{\perp}}\Phi )(T)  |^2 ) \|_{L^2(B_M)} \to 0.
\end{equation*}

\noindent $\bullet$ By \eqref{AT-prop} and \eqref{Phi-T-est},
\begin{equation}\label{qd24Oct18-3}
	\| \nabla_x ( | (\Pi_{U_*^{\perp}}\Phi )(T) |^2 ) \|_{L^{\infty}(\mathbb{R}^2)} \lesssim 1.
\end{equation}
Combining \eqref{eqn-574574574} and \eqref{U*-norm}, we have
\begin{equation*}
	\| \nabla_x\big( \big| \big(\Pi_{U_*^{\perp}}\Phi\big)(T) \big|^2 \big) (A(T)-A)
	\|_{L^\infty(\mathbb{R}^2)} +	\| 	\nabla_x\big( \big| \big(\Pi_{U_*^{\perp}}\Phi\big)(T) \big|^2 \big) (1+A- \Phi\cdot U_*) \left(1-|U_*|^2 \right)
	\|_{L^\infty(\mathbb{R}^2)}
	\to 0.
\end{equation*}
As a result, we conclude $I_3\to 0$. In sum,
\begin{equation}\label{qd24Oct18-7}
	\| \nabla_x ( (A U_*)(t)- (A U_*)(T) ) \|_{L^2(B_M)} \to 0.
\end{equation}

Consequently, applying \eqref{u*-def-1018}, \eqref{AT-prop}, \eqref{qd24Oct18-4}, \eqref{Phi-T-est}, \eqref{qd24Oct18-3} to $u_*$, and \eqref{conv-Phi}, \eqref{qd24Oct18-4}, \eqref{qd24Oct18-5}, \eqref{qd24Oct18-6}, \eqref{qd24Oct18-7} to $\Phi_{\rm per}(t)-\Phi_{\rm per}(T)$, we attain
\begin{equation}\label{qd240727-1}
	\begin{aligned}
		&
		u_*(x) = A(x,T) U_{\infty} + \Phi(x,T) - (\Phi(x,T) \cdot U_{\infty}) U_{\infty} + U_{\infty}  \in H^1_{\rm loc}(\R^2)\cap L^\infty(\R^2),
		\\
		&
		\Phi_{\rm per}(t)-\Phi_{\rm per}(T) \rightarrow 0
		\mbox{ \ in \ }
		H^{1}_{\rm loc}(\R^2){\cap L^{\infty}(\R^2)}.
	\end{aligned}
\end{equation}

Next, we will prove weak-$*$ convergence. Obviously, $ |\nabla_x u|^2 = |\nabla_x U_*|^2 + 2 \nabla_x U_* \cdot \nabla_x \Phi_{\rm per} + |\nabla_x \Phi_{\rm per} |^2$,
$|\nabla_x u_*|^2 = |\nabla_x \Phi_{\rm per}(x,T) |^2$, $ |\nabla_x U^{\J}|^2 = 8\lambda_j^{-2} \big( \lambda_j^{-2}|x- \xi^{\J}|^2+1\big)^{-2}$.
Given a function $f\in L^{\infty}(\mathbb{R}^2)$ continuous at $q^{\J}$, by dominated convergence theorem and $ 8\int_{\mathbb{R}^2} (|z|^2+1)^{-2} dz=8\pi$, then
\begin{equation*}
	\int_{\mathbb{R}^2} |\nabla_x U^{\J}(x,t)|^2 f(x) dx
	=
	8\int_{\mathbb{R}^2} (|z|^2+1)^{-2}
	f(\lambda_j z +\xi^{\J}) dz \to 8\pi f(q^{\J}).
\end{equation*}

Given a constant $C_1>0$, by \eqref{Phiper-nabla},
\begin{equation*}
	\int_{\mathbb{R}^2} \lambda_j^{-1} \big( \lambda_j^{-2}|x- \xi^{\J}|^2+1\big)^{-1} |\nabla_x \Phi_{\rm per}(x,t)| \1_{\{ |x|\le C_1 \}}dx
	\lesssim \lambda_*^{\epsilon} |\ln \lambda_*|.
\end{equation*}

For $j\ne k$, by splitting $\mathbb{R}^2$ into three parts $\{ x \ | \ |x-q^{\J}|\le d_q \}$, $\{ x \ | \ |x-q^{\K}|\le d_q \}$, and $\{ x \ | \ \min\{|x-q^{\J}|, |x-q^{\K}| \}> d_q \}$ when estimating, we have
\begin{equation*}
	\int_{\mathbb{R}^2} \lambda_j^{-1} \big( \lambda_j^{-2}|x- \xi^{\J}|^2+1\big)^{-1}
	\lambda_k^{-1} \big( \lambda_k^{-2}|x- \xi^{\K}|^2+1\big)^{-1}
	dx
	\lesssim \lambda_*^2 |\ln \lambda_*|.
\end{equation*}
Together with \eqref{qd240727-1}, given a function $f\in L^{\infty}(\mathbb{R}^2)$ with compact support and continuous at $q^{\J}$, $j=1,2,\dots, N$, we have
\begin{equation}
	\lim_{t\to T} \int_{\mathbb{R}^2} f(x) |\nabla u(x,t)|^2  dx =
	\int_{\mathbb{R}^2} f(x) |\nabla u_*(x)|^2  dx + \sum_{j=1}^{N} 8\pi f(q^{\J}),
\end{equation}
which is the weak-$*$ convergence of the Radon measure. We complete the proof of Corollary \ref{cor-intro}.

\section{Linear theory for the inner problems}\label{sec-linearinner}
In this section, we will establish the linear theory for the inner problem \eqref{inner-eq} in different modes. Since this section is rather independent of the other parts,  we abuse the notation a bit and use $R$ for more general cases.
Recall \eqref{tau-j-def}, in the time variable $\tau_j$, \eqref{inner-eq} is the usual parabolic system.  Since the inner problems for  $j=1,2,\dots, N $ all have the same structure,  we omit the subscripts or superscripts ``${}_{j}$'', ``$\J$'' in this section for brevity, and all spatial derivatives are about $y$. Consider
\begin{equation}\label{lin-inner}
	\left\{
	\begin{aligned}
		&
		\pp_{\tau} \Psi
		=
		\left( 	a -b W  \wedge \right)
		\left(
		L_{\rm{in}} \Psi
		\right) + H  && \mbox{ \ in \ } \mathcal{D}_R,
		\\
		&
		\Psi(y,\tau)\cdot W(y) = H(y,\tau) \cdot W(y)=0 && \mbox{ \ in \ } \mathcal{D}_R,
	\end{aligned}
	\right.
\end{equation}
where
\begin{equation}\label{L-in-def}
	L_{\rm{in}} \mathbf{f}:=
	\Delta \mathbf{f}
	+
	|\nabla W |^2  \mathbf{f}
	- 2 \nabla 	
	\left( 	W \cdot  \mathbf{f} \right) \cdot \nabla W
	+
	2  \left( 	\nabla W \cdot    \nabla \mathbf{f} \right) W,
\end{equation}
\begin{equation*}
	\mathcal{D}_R := \left\{
	(y,\tau) \ | \
	\tau \in (\tau_0,\infty),
	y\in B_{R}
	\right\} ,
	\quad
	B_{R} := \left\{ y \in \R^2\ | \ |y| < R(\tau) \right\},
	\quad \tau_0\ge 2.
\end{equation*}

We call a function $f(\tau)$ defined in $(\tau_0,\infty)$ of {\it algebraic power type} if $C^{-1} f_*(\tau) \le f(\tau) \le C f_*(\tau)$, where $C\ge 1$ is a constant, and
$f_*(\tau) = c_0 \tau^{c_1} \langle \ln \tau\rangle^{c_2} \langle \ln \langle \ln \tau\rangle \rangle^{c_3} \cdots$ with finite multiplicity, $c_0>0$ possibly depends on $\tau_0$, $c_i\in \mathbb{R}$, $i=1,2, \dots$, and then we define
\begin{equation}\label{P1-def}
	{\mathbf{P}}_1[f] := c_1.
\end{equation}
Denote ${\mathbf{AP}}$ as the set of algebraic power type functions. Obviously, for any $f_1, f_2 \in {\mathbf{AP}}$, $c\in \mathbb{R}$, we have $f_1 f_2, f_1/f_2, f_1^c \in {\mathbf{AP}}$, $\mathbf{P}_1[f_1 f_2] = \mathbf{P}_1[f_1] + \mathbf{P}_1[f_2]$, $\mathbf{P}_1[1/f_1] = -\mathbf{P}_1[f_1]$, and  $\mathbf{P}_1[f_1^c]=c \mathbf{P}_1[f_1]$.
For $\tau \ge \tau_0 \ge 9$ and $f\in {\mathbf{AP}}$, $C_{f}^{-1} f(\tau) \le f(s) \le C_{f} f(\tau)$  for all  $\tau\le s\le 2\tau$
with a constant $C_{f}>1$. If we assume, in addition, $\ln f\in {\mathbf{AP}}$, then $\mathbf{P}_1[ \ln f ] =0$.

Throughout this section, unless otherwise stated, we always assume that constants, $O(\cdot)$, $\lesssim$, $\sim$ are independent of $\tau_0$, $k\in \mathbb{Z}$ for $|k|\ge 2$ (used for mode $k$, $|k|\ge 2$), and
\begin{equation}\label{nu-assump}
	\begin{aligned}
		&
		\tau_0\ge 9,
		\quad
		v(\tau), R(\tau), R_0(\tau)\in {\mathbf{AP}},
		\quad
		0\le {\mathbf{P}}_1[R_0] \le {\mathbf{P}}_1[R] < 1/2,
		\quad
		2 R_0(\tau) \le R(\tau),
		\\
		&
		\inf_{s\ge \tau_0} R_0(s) \gg 1,
		\quad
		R^2 \big( \ln R + (\ln \tau)^m \big)\le C_1 \tau,
		\quad
		R_0 \in C^1(\tau_0,\infty),
		\quad
		|R_0'| =O( R_0^{-1} )
	\end{aligned}
\end{equation}
for any $m\ge 0$ with a constant $C_1>0$ depending on $m$. Obviously, if $R_0' = O(\tau^{-1} R_0)$, then $|R_0'| =O( R_0^{-1} )$.

Suppose that $\Psi_{\mathbb C}(y,\tau)$, $H_{\mathbb{C}}(y,\tau)$ defined by \eqref{Cbb-def} have the following Fourier expansion respectively,
\begin{equation}\label{complex-Phi-H}
	\begin{aligned}
		\Psi_{ \mathbb C}(y,\tau)= \ & \sum_{k\in \mathbb Z} \psi_{k} (\rho,\tau) e^{ik\theta},
		\quad
		\psi_{k} (\rho,\tau) := \Psi_{\mathbb{C}, k} (\rho,\tau) =
		(2\pi)^{-1}
		\int_{0}^{2\pi} \Psi_{ \mathbb C}(\rho e^{is},\tau) e^{-iks} ds,
		\\
		H_{\mathbb{C}}(y,\tau)= \ & \sum_{k\in \mathbb Z} h_{k} (\rho,\tau) e^{ik\theta},
		\quad
		h_{k} (\rho,\tau) :=
		H_{\mathbb{C}, k} (\rho,\tau) =
		(2\pi)^{-1}
		\int_{0}^{2\pi} H_{\mathbb{C}}(\rho e^{is},\tau) e^{-iks} ds,
	\end{aligned}
\end{equation}
where $\Psi_{\mathbb{C}, k}$ and $H_{\mathbb{C}, k}$ are defined in \eqref{modek-def} and
\begin{equation*}
	y=\rho e^{i\theta},\quad \rho=|y|,\quad \theta = \arctan( y_2 \slash y_1).
\end{equation*}
Using \eqref{Cbb-inverse}, we denote
\begin{equation*}%\label{mk-vec-form}
	\Psi_{ k}(y,\tau) := \big(  \psi_{k} (\rho,\tau) e^{ik\theta}  \big)_{\mathbb{C}^{-1}},
	\quad
	H_k(y,\tau) := \big( h_{k} (\rho,\tau) e^{ik\theta} \big)_{\mathbb{C}^{-1}} \mbox{ \ for \ } k\in \mathbb{Z}.
\end{equation*}
It is easy to see that
\begin{equation}\label{equi-norm-1}
	| \Psi_{ k} | = | \psi_{k} |,\quad
	|H_k| = |h_k| .
\end{equation}

For $\ell \in \RR$ and $v(\tau) >0$ and vectorial complex-valued function $\fbf$, we introduce the weighted topology
\begin{equation}\label{qd24Dec07-1}
	\| \fbf \|_{v,\ell}^{\mathcal{R}} := \sup\limits_{(y,\tau)\in \DD_{\mathcal{R}}} (v(\tau) \langle y \rangle^{-\ell})^{-1}  |\fbf(y,\tau)|
\end{equation}
with a scalar function $\mathcal{R}= \mathcal{R}(\tau)$. By \eqref{equi-norm-1} and \eqref{complex-Phi-H}, we have
\begin{equation}\label{equi-norm-2}
	\| \psi_k \|_{v,\ell}^{\mathcal{R}} = \| \Psi_k \|_{v,\ell}^{\mathcal{R}},\quad
	\| h_k \|_{v,\ell}^{\mathcal{R}} = \| H_k \|_{v,\ell}^{\mathcal{R}}
	\lesssim \| H \|_{v,\ell}^{\mathcal{R}},
	\quad
	|\Psi(y,\tau)| \lesssim
	\sum_{k\in \mathbb Z} |\Psi_{k} (y,\tau)|.
\end{equation}
For the convergence in \eqref{equi-norm-2} when summing up, we have to make the dependence on $k$ very clear in the estimates of mode $k$, $|k|\ge 2$.

\subsection{Complex-valued form of the inner linear equation}

The following lemma bridges the inner problem in the parabolic system form and the complex-scalar form.
\begin{lemma}\label{complex-lem}
	For $L_{\rm{in}} $ defined in \eqref{L-in-def} and $\Psi \cdot W =0$,	we have
	\begin{equation}\label{complex-1}
		\left[
		\left( 	a -b W  \wedge \right)
		\left(
		L_{\rm{in}} \Psi
		\right)  \right] \cdot W  =0,
	\end{equation}
	\begin{equation}\label{complex-2}
		\left[
		\left( 	a -b W  \wedge \right)
		\left(
		L_{\rm{in}} \Psi
		\right)  \right]_{\mathbb{C} }
		=
		(a-ib) \mathcal{L}_{\rm{in}} \Psi_{ \mathbb C},
	\end{equation}
	where
	\begin{equation*}%\label{cal-L-def}
		\mathcal{L}_{\rm{in}} \Psi_{ \mathbb C} : =
		\Big[\pp_{\rho \rho}+\frac1{\rho}\pp_{\rho}+\frac1{\rho^2}\pp_{\theta\theta}-\frac1{\rho^2}+i\frac{2\cos w(\rho)}{\rho^2}\pp_{\theta}+\frac8{(\rho^2+1)^2}\Big] \Psi_{ \mathbb C}.
	\end{equation*}
	Then \eqref{lin-inner} is equivalent to the complex-valued equation
	\begin{equation}\label{Complex-in-eq}
		\pp_{\tau} \Psi_{ \mathbb C}
		=
		(a-ib)\mathcal{L}_{\rm{in}} \Psi_{ \mathbb C}
		+
		H_{\mathbb{C}} \mbox{ \ in \ } \mathcal{D}_R.
	\end{equation}
	Under the Fourier expansion \eqref{complex-Phi-H}, then
	\begin{equation}\label{com-3}
		\mathcal{L}_{\rm{in}}  \big[  e^{ik\theta}\psi_k \big]
		=
		e^{ik\theta} \mathcal{L}_k  \psi_k,
	\end{equation}
	where
	\begin{equation}\label{cal-L-ope}
		\mathcal{L}_k  f :=
		\pp_{\rho \rho }f + \frac{\pp_{\rho} f}{\rho}
		+ V_k(\rho) f, \quad V_k(\rho):= -\frac{(k+1)^2 \rho^4 + (2k^2-6)\rho^2 +(k-1)^2}{ (\rho^2+1)^2} \frac{1}{\rho^2}.
	\end{equation}
	It follows that
	\begin{equation*}
		\pp_{\tau} \Psi_{k}
		=
		\left( 	a -b W  \wedge \right)
		\left(
		L_{\rm{in}} \Psi_{k}
		\right) + H_{k}
	\end{equation*}
	is equivalent to
	\begin{equation}\label{inner-Fourier-form}
		\pp_{\tau} \psi_{k}= (a-ib) \mathcal L_k \psi_{k}  +    h_k .
	\end{equation}

\end{lemma}

\begin{proof}
	Set
	$$
	\Psi(y,\tau)=\varphi_{ 1}(y,\tau) E_1(y)+\varphi_{2}(y,\tau) E_2(y), \mbox{ \ that is, \ } \Psi_{ \mathbb{C}} = \varphi_{ 1} + i \varphi_{ 2} .
	$$
	By \eqref{Frenet-deri}, one has
	\begin{align*}
		%	\begin{equation*}
			%		\begin{aligned}
				&
				\Delta( \varphi_{1} E_1 )
				=
				(\Delta \varphi_{1} ) E_1
				+ 2\Big( \pp_{\rho}  \varphi_{1} \pp_{\rho} E_1
				+ \frac{1}{\rho^2}
				\pp_{\theta}  \varphi_{1} \pp_{\theta} E_1
				\Big) +
				\varphi_{1}
				\Big(\pp_{\rho\rho } + \frac{\pp_{\rho}}{\rho } + \frac{\pp_{\theta \theta}}{\rho^2}\Big) E_1
				\\
				= \ &
				(\Delta \varphi_{1} ) E_1
				- 2 \pp_{\rho}  \varphi_{1} w_{\rho} W
				+ \frac{2}{\rho^2}
				\pp_{\theta}  \varphi_{1} \cos w E_2
				\\
				&
				+
				\varphi_{1}
				\Big[-w_{\rho\rho} W -w_{\rho}^2 E_1 - \frac{1}{\rho }  w_{\rho} W
				- \frac{ 1 }{\rho^2} \cos w (\sin w W +\cos w E_1) \Big]
				\\
				= \ &
				\Big[
				\Delta \varphi_{1}
				-
				\varphi_{1}
				\Big(w_{\rho}^2  + \frac{ \cos^2 w}{\rho^2}   \Big)
				\Big] E_1
				+ \frac{2 \cos w }{\rho^2}
				\pp_{\theta} \varphi_{1}  E_2
				+
				\Big[ - 2 w_{\rho} \pp_{\rho} \varphi_{1}
				-
				\varphi_{1}
				\Big( w_{\rho\rho}  + \frac{ w_{\rho}  }{\rho }
				+   \frac{  \sin w \cos w }{\rho^2}  \Big)
				\Big] W
				\\
				= \ &
				\Big(
				\pp_{\rho\rho} \varphi_{1}
				+
				\frac 1{\rho} \pp_{\rho} \varphi_{1}
				+
				\frac{1}{\rho^2} \pp_{\theta\theta} \varphi_{1}
				-
				\frac{1}{\rho^2}\varphi_{1}
				\Big) E_1
				+ \frac{2 \cos w }{ \rho^2}
				\pp_{\theta} \varphi_{1}  E_2
				+
				\Big( - 2  w_{\rho} \pp_{\rho}  \varphi_{1}
				-
				\frac{ 2 \sin w \cos w }{\rho^2}  \varphi_{1}
				\Big) W,
				%		\end{aligned}
			%	\end{equation*}
	\end{align*}
	where we used $w_{\rho\rho}+\frac{w_{\rho}}{\rho}-\frac{\sin w\cos w}{\rho^2}=0$ for the last equality. Similarly,
	\begin{align*}
%	\begin{equation*}
%		\begin{aligned}
			&
			\Delta( \varphi_{2} E_2)
			=
			\Delta (\varphi_{2}) E_2
			-
			2 \frac{1}{\rho^2} \pp_{\theta} \varphi_{2} (\sin w W +\cos w E_1)
			-
			\varphi_{2} \frac{1}{\rho^2} E_2
			\\
			= \ &
			-\frac{2\cos w}{ \rho^2} \partial_{\theta} \varphi_{2} E_1
			+\Big(
			\partial_{\rho\rho} \varphi_{2} +\frac1\rho \partial_{\rho} \varphi_{2} +\frac{1}{\rho^2}\partial_{\theta\theta} \varphi_{2} -\frac{1}{\rho^2} \varphi_{2}\Big)E_2
			-\frac{2\sin w }{ \rho^2} \partial_{\theta} \varphi_{2} W.
%		\end{aligned}
%	\end{equation*}
	\end{align*}
	Thus
	\begin{equation}\label{ZLin-1}
		\begin{aligned}
			\Delta \Psi
			=&~\Big(\pp_{\rho\rho} \varphi_{1}+\frac1{\rho}\pp_{\rho} \varphi_{1}+\frac1{\rho^2}\pp_{\theta\theta} \varphi_{ 1}-\frac1{\rho^2} \varphi_{1}-\frac{2\cos w}{\rho^2}\pp_{\theta} \varphi_{2}\Big) E_1\\
			&~+\Big( \pp_{\rho\rho} \varphi_{2}+\frac1{\rho}\pp_{\rho} \varphi_{2}+\frac1{\rho^2}\pp_{\theta\theta} \varphi_{2}-\frac1{\rho^2} \varphi_{2}+\frac{2\cos w}{\rho^2}\pp_{\theta} \varphi_{1}\Big) E_2\\
			&~+\Big(-2w_{\rho} \pp_{\rho} \varphi_{1} -\frac{2\sin w}{\rho^2}\pp_{\theta} \varphi_{2} -\frac{2 \sin w \cos w }{\rho^2} \varphi_{1} \Big) W.
		\end{aligned}
	\end{equation}
	By \eqref{Frenet-deri},
	\begin{equation*}
		\begin{aligned}
			&\pp_{\rho} W=w_{\rho} E_1,\quad \pp_{\theta} W=\sin w E_2, \quad
			\pp_{\rho}\Psi =(\pp_{\rho} \varphi_{1} ) E_1+
			(\pp_{\rho} \varphi_{2}) E_2 -\varphi_{1} w_{\rho} W,\\
			&\pp_{\theta}\Psi = (\pp_{\theta} \varphi_{1}) E_1+
			(\pp_{\theta} \varphi_{2}) E_2 + \varphi_{1} \cos w E_2-\varphi_{2} (\sin w W+\cos w E_1) ,
		\end{aligned}
	\end{equation*}
	then we have
	\begin{equation}\label{ZLin-2}
		2(\nabla W \cdot \nabla \Psi ) W
		=
		2\Big(\pp_{\rho} W\cdot \pp_{\rho}\Psi +\frac1{\rho^2}\pp_{\theta}W\cdot\pp_{\theta}\Psi \Big)W
		= \Big(\frac{2\sin w \cos w }{\rho^2} \varphi_{1}
		+2w_{\rho} \pp_{\rho}\varphi_{1}
		+\frac{2\sin w}{\rho^2} \pp_{\theta}\varphi_{2} \Big) W.
	\end{equation}
	Then plugging \eqref{nablaW}, \eqref{ZLin-1}, \eqref{ZLin-2} into \eqref{L-in-def} and using \eqref{wedge-formula}, we have
	\begin{equation*}%\label{eqn-269269}
		\begin{aligned}
			\left( 	a -b W  \wedge \right)
			\left(
			L_{\rm{in}} \Psi
			\right)
			= \ &
			\Big\{
			\pp_{\rho\rho}\varphi_{1}+\frac1{\rho}\pp_{\rho}\varphi_{1}+\frac1{\rho^2}\pp_{\theta\theta}\varphi_{1}
			+
			\Big[
			\frac{8}{(\rho^2+1)^2}
			-\frac1{\rho^2}
			\Big]
			\varphi_{1}-\frac{2\cos w}{\rho^2}\pp_{\theta}\varphi_{2}\Big\}
			\left( 	aE_1 -b E_2 \right)
			\\
			&+\Big\{
			\pp_{\rho\rho}\varphi_{2}+\frac1{\rho}\pp_{\rho}\varphi_{2}+\frac1{\rho^2}\pp_{\theta\theta}\varphi_{2}
			+
			\Big[
			\frac{8}{(\rho^2+1)^2}
			-\frac1{\rho^2} \Big] \varphi_{2}+\frac{2\cos w}{\rho^2}\pp_{\theta}\varphi_{1}\Big\}
			\left( 	aE_2 +b E_1 \right),
		\end{aligned}
	\end{equation*}
	which implies \eqref{complex-1}, \eqref{complex-2} and the equivalence between \eqref{lin-inner} and \eqref{Complex-in-eq}.
	Finally, \eqref{com-3} and \eqref{inner-Fourier-form} are  derived directly.
\end{proof}

The linearly independent kernels $\mathcal{Z}_{k,1}, \mathcal{Z}_{k,2}$ of
$\mathcal{L}_k$ in \eqref{cal-L-ope} satisfying the Wronskian $W[\mathcal{Z}_{k,1}, \mathcal{Z}_{k,2}] = \rho^{-1}$  are given as follows:
\begin{equation}\label{scalr-Z}
	\begin{cases}
		\mathcal{Z}_{-1,1}(\rho) = \frac{\rho^2}{\rho^2 +1} ,
		\quad
		\mathcal{Z}_{-1,2}(\rho) =
		\frac{4\rho^2(\rho^2\ln(\rho) -1) -1}{4\rho^2(\rho^2 +1)},
		& k=-1,
		\\
		\mathcal{Z}_{0,1}(\rho)
		= \frac{\rho}{\rho^2 +1} ,
		\quad
		\mathcal{Z}_{0,2}(\rho)
		=
		\frac{\rho^4 +4 \rho^2\ln(\rho) -1}{2 \rho(\rho^2+1)} ,
		&
		k=0,
		\\
		\mathcal{Z}_{1,1}(\rho) = \frac{1}{\rho^2 +1} ,
		\quad
		\mathcal{Z}_{1,2}(\rho) =
		\frac{\rho^4 +4\rho^2 + 4\ln(\rho)}{4(\rho^2 +1)} ,
		&
		k=1,
		\\
		\mathcal{Z}_{k,1}(\rho) = \frac{\rho^{1-k}}{\rho^2 + 1}, \quad
		\mathcal{Z}_{k,2}(\rho) = \frac{\rho^{k-1}}{\rho^2 + 1}(\frac{\rho^4}{2k+2} + \frac{\rho^2}{k} + \frac{1}{2k-2}) ,
		& k\ne -1,0,1.
	\end{cases}
\end{equation}
It is straightforward to get
\begin{equation*}
	\mathcal{Z}_{k,1}(\rho)
	\sim
	\rho^{1-k} \1_{\{ 0<\rho\le 1 \}}
	+
	\rho^{-1-k} \1_{\{ \rho > 1 \}},
	\quad
	\mathcal{Z}_{k,2}(\rho)
	\sim
	k^{-1} \big( \rho^{k-1} \1_{\{ 0<\rho \le 1\}}
	+
	\rho^{k+1} \1_{\{ \rho>1 \}}
	\big)
	\mbox{ \ for \ }  k\ne -1,0,1.
\end{equation*}

\medskip

Recall \eqref{def-kernels} and \eqref{complex-Phi-H} and notice $\mathcal{Z}_{0,1}(\rho)=-\frac{1}{2} \rho w_{\rho}$. Then
for mode $0$,
\begin{equation}\label{ZZ0}
	\left( h_{0} (\rho,\tau) \right)_{\mathbb{C}^{-1}} \cdot Z_{0,1}
	+ i
	\left( h_{0} (\rho,\tau) \right)_{\mathbb{C}^{-1}} \cdot Z_{0,2} = \rho w_{\rho} h_{0} (\rho,\tau)
	=
	-2 \mathcal{Z}_{0,1}(\rho) h_{0} (\rho,\tau)
	.
\end{equation}
Notice $\mathcal{Z}_{1,1}(\rho)= -\frac 12 w_{\rho}$. For mode $1$,
\begin{equation*}
	\begin{aligned}
		( h_{1} (\rho,\tau) e^{i \theta} )_{\mathbb{C}^{-1}}\cdot Z_{1,1} =
		{\rm Re} ( h_{1} (\rho,\tau)  e^{i\theta}) w_{\rho }\cos \theta + {\rm Im}( h_{1} (\rho,\tau) e^{i\theta}) w_{\rho} \sin\theta,
		\\
		( h_{1} (\rho,\tau) e^{i \theta} )_{\mathbb{C}^{-1}} \cdot Z_{1,2} =
		{\rm Re} ( h_{1} (\rho,\tau) e^{i\theta}) w_{\rho }\sin \theta
		- {\rm Im}( h_{1} (\rho,\tau) e^{i\theta})
		w_{\rho }\cos \theta  ,
	\end{aligned}
\end{equation*}
whose equivalent complex form is given by
\begin{equation}\label{ZZ1}
	( h_{1} (\rho,\tau) e^{i \theta})_{\mathbb{C}^{-1}}\cdot Z_{1,1}
	-
	i ( h_{1} (\rho,\tau) e^{i \theta})_{\mathbb{C}^{-1}} \cdot Z_{1,2}
	=
	w_{\rho } h_1(\rho,\tau)
	=
	-2 \mathcal{Z}_{1,1}(\rho) h_1(\rho,\tau).
\end{equation}

For a radial complex-valued function $f(\rho)$, the quadratic form of $\mathcal{L}_k$ in $B_R$ is defined as
\begin{equation}\label{Qk-def}
	Q_{R,k}( f , f) =
	2\pi
	\int_{0}^R
	\Big[
	|\pp_\rho f|^2
	+
	\frac{(k+1)^2 \rho^4 + (2k^2-6)\rho^2 +(k-1)^2}{(\rho^2+1)^2} \frac{ |f| ^2}{\rho^2}
	\Big]
	\rho
	d \rho.
\end{equation}
By \cite[Lemma 4.2]{sun2021bubble}, $Q_{R,k}( f , f) \ge 0$ for all $f\in C^2(B_R) \cap C(\bar{B}_R)$ with $f=0$ on $\pp B_R$, and $Q_{R,k}( f , f) = 0$ implies $f\equiv 0$. Define the norms
\begin{equation*}
	\| f\|_{X(B_R) }
	=
	\Big[
	2\pi
	\int_0^R
	\Big(
	|\pp_\rho f|^2 + \frac{ |f|^2 }{\rho^2}
	\Big) \rho d \rho
	\Big]^{1/2},
\end{equation*}
\begin{equation*}
	\| f\|_{H^1_0(B_R) }
	=
	\Big(
	2\pi
	\int_0^R
	|\pp_\rho f|^2  \rho d \rho
	\Big)^{1/2},\
	\|f\|_{L^2(B_R)} =
	\Big(
	2\pi
	\int_0^R |f|^2 \rho d \rho
	\Big)^{1/2}.
\end{equation*}
Set $ X_0(B_R) = \{
f(\rho) \ \big| \
f(R) = 0, \quad
\| f\|_{X(B_R) }   <\infty
\}$.

\subsection{Energy estimates}
We use the method in \cite[Lemma 7.2]{Green16JEMS}
to analyze the first eigenvalue of $Q_{R,k}$.
\begin{lemma}\label{eigenvalue le}
	Let
	\begin{equation*}
		\lambda_{R,k} = \inf\limits_{f\in X_0(B_R) \backslash \{0\} }
		\frac{Q_{R,k}(f,f)}{\| f\|_{L^2 (B_R ) }^2 } \mbox{ \ for \ } k\ne 1,\quad
		\lambda_{R,1} = \inf\limits_{f\in H_0^1(B_R) \backslash \{0\} }
		\frac{Q_{R,1}(f,f)}{\| f\|_{L^2 (B_R ) }^2 } \mbox{ \ for \ } k = 1.
	\end{equation*}
	$\lambda_{R,k}$ is attained by a real-valued function in  $X_0(B_R)$ for $k\ne 1$ and $H_0^1(B_R)$ for $k= 1$.
	For $R$ large,
	\begin{equation*}
		\lambda_{R,0}\sim (R^2 \ln R)^{-1} ,\quad
		\lambda_{R,1}
		\sim
		R^{-4} ,\quad
		\lambda_{R,-1} \gtrsim
		(R^2 \ln R)^{-1}   ,
		\quad
		\lambda_{R,k} \gtrsim
		|k|^2 R^{-2}  \mbox{ \ for \ } |k|\ge2 .
	\end{equation*}
	
\end{lemma}

\begin{proof}
	
	For any complex-valued function  $f=f_1+ if_2$ with $f_1 =\mathrm{Re} f$, $f_2 =\mathrm{Im} f$,
	\begin{equation*}
		\frac{Q_{R,k}(f,f)}{\| f\|_{L^2 (B_R ) }^2 }
		=
		\frac{Q_{R,k}(f_1,f_1) + Q_{R,k}(f_2,f_2)}{\| f_1\|_{L^2 (B_R ) }^2 + \| f_2\|_{L^2 (B_R ) }^2 }
		\ge
		\min\bigg\{
		\frac{Q_{R,k}(f_1,f_1) }{\| f_1\|_{L^2 (B_R ) }^2  }
		,
		\frac{Q_{R,k}(f_2,f_2) }{\| f_2\|_{L^2 (B_R ) }^2  }
		\bigg\} .
	\end{equation*}
	Thus for $k\ne 1$,
	\begin{equation*}%\label{real}
		\lambda_{R,k} = \inf\bigg\{
		\frac{Q_{R,k}(f,f)}{\| f\|_{L^2 (B_R ) }^2 }
		\ | \
		f\in X_0(B_R) \backslash \{0\}, \mbox{ $f$ is real-valued}
		\bigg\}.
	\end{equation*}
	The same argument can be applied to $\lambda_{R,1}$.
	
	Hereafter, we focus on real-valued functions. We choose a sequence $f_n\in X_0(B_R)$ if $k\ne 1$ ($f_n\in H_0^1(B_R)$ if $k=1$) with $\|f_n\|_{L^2(B_R)} = 1$ and  $\lambda_{R,k} + 1 \ge  Q_{R,k}(f_n,f_n) \rightarrow \lambda_{R,k}$. By the form of $Q_{R,k}$ given in \eqref{Qk-def}, we have
	$ \int_0^R (\pp_\rho f_n)^2 \rho d \rho \lesssim \lambda_{R,k} + 1 $.
	The Sobolev compact embedding theorem implies $f_n \rightarrow f_\infty$ in $L^2(B_R)$ up to a subsequence.
	
	For $k\ne 1$,
	\begin{equation*}
		Q_{R,k}( f_n , f_n)
		=
		2\pi
		\int_{0}^R
		\bigg[
		(\pp_\rho f_n)^2
		+
		\frac{(k-1)^2}{(\rho^2+1)^2} \frac{ f_n^2}{\rho^2}
		+
		\frac{(k+1)^2 \rho^2 + (2k^2-6) }{(\rho^2+1)^2}  f_n^2
		\bigg]
		\rho
		d \rho.
	\end{equation*}
	Up to a subsequence, we have
	\begin{equation*}
		\int_{0}^R
		\bigg[
		(\pp_\rho f_\infty)^2
		+
		\frac{(k-1)^2}{(\rho^2+1)^2} \frac{ f_\infty^2}{\rho^2}
		\bigg]
		\rho
		d \rho
		\le
		\liminf\limits_{n\rightarrow \infty}
		\int_{0}^R
		\bigg[
		(\pp_\rho f_n)^2
		+
		\frac{(k-1)^2}{(\rho^2+1)^2} \frac{ f_n^2}{\rho^2}
		\bigg]
		\rho
		d \rho,
	\end{equation*}
	\begin{equation*}
		\int_{0}^R
		\frac{(k+1)^2 \rho^2 + (2k^2-6) }{(\rho^2+1)^2}  f_\infty^2
		\rho
		d \rho
		=
		\lim\limits_{n\rightarrow\infty}
		\int_{0}^R
		\frac{(k+1)^2 \rho^2 + (2k^2-6) }{(\rho^2+1)^2}  f_n^2
		\rho
		d \rho.
	\end{equation*}
	Moreover,
	\begin{equation*}
		\int_{0}^R
		\bigg[
		(\pp_\rho f_\infty)^2
		+
		\frac{(k-1)^2}{(\rho^2+1)^2} \frac{ f_\infty^2}{\rho^2}
		\bigg]
		\rho
		d \rho
		\sim
		C(R,k)
		\int_{0}^R
		\bigg[
		(\pp_\rho f_\infty)^2
		+
		\frac{ f_\infty^2}{\rho^2}
		\bigg]
		\rho
		d \rho.
	\end{equation*}
	Thus
	\begin{equation*}
		Q_{R,k}(f_\infty,f_\infty) \le \lambda_{R,k}, \quad \| f_\infty\|_{L^2(B_R)} = 1, \quad
		f_\infty \in X_0(B_R),
	\end{equation*}
	which implies that the minimum $\lambda_{R,k}$ is attained by $f_\infty$.
	
	For $k=1$,
	similarly, we choose a subsequence such that $f_n \rightharpoonup f_\infty$ in $H^1_0(B_R)$, $f_n \rightarrow f_{\infty}$ in $L^2(B_R)$.
	\begin{equation*}
		\int_{0}^R
		(\pp_\rho f_\infty)^2
		\rho
		d \rho
		\le
		\liminf\limits_{n\rightarrow\infty}
		\int_{0}^R
		(\pp_\rho f_n)^2
		\rho
		d \rho,
		\quad
		\int_{0}^R
		\frac{ 4 ( \rho^2 -1 )  }{(\rho^2+1)^2}  f_\infty ^2
		\rho
		d \rho
		=
		\lim\limits_{n\rightarrow \infty}
		\int_{0}^R
		\frac{ 4 ( \rho^2 -1 )  }{(\rho^2+1)^2}  f_n ^2
		\rho
		d \rho.
	\end{equation*}
	Then
	\begin{equation*}
		Q_{R,1}(f_\infty,f_\infty) \le \lambda_{R,1}, \quad \| f_\infty\|_{L^2(B_R)} = 1, \quad
		f_\infty \in H_0^1(B_R),
	\end{equation*}
	and thus $\lambda_{R,1}$ is attained by $f_\infty$.

	Next, we will use the Lagrange multiplier for the real-valued minimum function $f_\infty$ to estimate $\lambda_{R,k}$, $k=-1, 0, 1$. To avoid confusion, we denote $w_k$ as the eigenfunction corresponding to the eigenvalue $\lambda_{R,k}$ for every mode $k$ with the normalization $\| w_k \|_{L^2(B_R)} =1$.
	
	For $k=0$,
	\begin{equation*}
		\mathcal{L}_0 w_0 =
		- \lambda_{R,0} w_0
		\mbox{ \ in \ } B_{R},\quad
		w_0 = 0
		\mbox{ \ on \ } \pp B_{R}.
	\end{equation*}
	$w_0$ is given by
	\[
	w_0(\rho) =
	\mathcal{Z}_{0,2}(\rho) \int_0^\rho (- \lambda_{R,0} w_0(s) ) \mathcal{Z}_{0,1}(s) s d s
	+
	\mathcal{Z}_{0,1}(\rho) \int_\rho^{R} (- \lambda_{R,0} w_0(s) ) \mathcal{Z}_{0,2}(s) s d s
	-
	A_{R,0} \mathcal{Z}_{0,1}(\rho),
	\]
	where $
	A_{R,0} = ( \mathcal{Z}_{0,1}(R) )^{-1} \mathcal{Z}_{0,2}(R) \int_0^{R} (- \lambda_{R,0} w_0(s) )
	\mathcal{Z}_{0,1}(s)
	s d s $.
	For $0\le \rho \le 1$,
	\begin{equation*}
		|\mathcal{Z}_{0,2}(\rho) \int_0^\rho   w_0(s)  \mathcal{Z}_{0,1}(s) s d s|
		\lesssim
		\rho^{-1}
		\| w_0 \|_{L^2 (B_\rho)}
		\| \mathcal{Z}_{0,1}\|_{L^2 (B_\rho)}
		\lesssim
		1,
	\end{equation*}
	\begin{equation*}
		|\mathcal{Z}_{0,1}(\rho) \int_\rho^{R}   w_0(s)  \mathcal{Z}_{0,2}(s) s d s|
		\le
		|\mathcal{Z}_{0,1}(\rho) \int_1^{R}   w_0(s)  \mathcal{Z}_{0,2}(s) s d s|
		+
		|\mathcal{Z}_{0,1}(\rho) \int_\rho^{1}  w_0(s)  \mathcal{Z}_{0,2}(s) s d s|
		\lesssim
		R^2 ,
	\end{equation*}
	\begin{equation*}
		| A_{R,0} \mathcal{Z}_{0,1}(\rho) |
		\lesssim |A_{R,0}|
		\lesssim
		\lambda_{R,0}
		R^2 (\ln R )^{\frac 12} .
	\end{equation*}
	Thus $\| w_0 \|_{L^2 (B_1 )} \lesssim \lambda_{R,0}
	R^2 (\ln R )^{\frac 12}  $.
	For $\rho\ge 1$,
	\begin{equation*}
		\| \mathcal{Z}_{0,2}(\rho) \int_0^\rho w_0(s)  \mathcal{Z}_{0,1}(s) s d s \|_{L^2(B_R\backslash B_1)}
		\le
		\| \mathcal{Z}_{0,2} \|_{L^2(B_R\backslash B_1)}
		\| w_0 \|_{L^2(B_R)}
		\| \mathcal{Z}_{0,1} \|_{L^2(B_R)}
		\lesssim
		R^2 (\ln R)^{\frac 12} ,
	\end{equation*}
	\begin{equation*}
		\|\mathcal{Z}_{0,1}(\rho) \int_\rho^{R} w_0(s)  \mathcal{Z}_{0,2}(s) s d s \|_{L^2(B_R\backslash B_1)}
		\lesssim
		\|\mathcal{Z}_{0,1}  \|_{L^2(B_R\backslash B_1)}
		\| w_0 \|_{L^2(B_R\backslash B_1)}
		\| \mathcal{Z}_{0,2} \|_{L^2(B_R\backslash B_1)}
		\lesssim
		R^2 (\ln R)^{\frac 12} ,
	\end{equation*}
	\begin{equation*}
		\| A_{R,0} \mathcal{Z}_{0,1}(\rho) \|_{L^2(B_R\backslash B_1)}
		\lesssim
		\lambda_{R,0}
		R^2 \ln R.
	\end{equation*}
	In sum, when $R$ is large, we have $
	1 = \| w_0 \|_{L^2(B_R)} \lesssim
	\lambda_{R,0}
	R^2 \ln R$.
	
	On the other hand, when $R$ is large,
	$\| \eta_{\frac R2} \mathcal{Z}_{0,1} \|_{L^2(B_R)}^2\sim \ln R$,
	\begin{equation*}
		\begin{aligned}
			&
			Q_{R,0}(\eta_{\frac R2} \mathcal{Z}_{0,1} ,\eta_{\frac R2} \mathcal{Z}_{0,1} )
			\sim
			\bigg( \int_{0}^{\frac R2} +\int_{\frac R 2}^R
			\bigg)
			\bigg[
			(\pp_\rho (\eta_{\frac R2} \mathcal{Z}_{0,1}) )^2
			+
			\frac{  \rho^4 -6 \rho^2 + 1 }{(\rho^2+1)^2} \frac{(\eta_{\frac R2} \mathcal{Z}_{0,1})^2}{\rho^2}
			\bigg]
			\rho
			d \rho
			\\
			= \ &
			-
			\int_{\frac R2}^\infty
			\left[
			(\pp_\rho  \mathcal{Z}_{0,1} )^2
			+
			\frac{  \rho^4 -6 \rho^2 + 1 }{(\rho^2+1)^2} \frac{( \mathcal{Z}_{0,1})^2}{\rho^2}
			\right]
			\rho
			d \rho
			+
			\int_{\frac R 2}^R
			\bigg[
			(\pp_\rho (\eta_{\frac R2} \mathcal{Z}_{0,1}) )^2
			+
			\frac{  \rho^4 -6 \rho^2 + 1 }{(\rho^2+1)^2} \frac{(\eta_{\frac R2} \mathcal{Z}_{0,1})^2}{\rho^2}
			\bigg]
			\rho
			d \rho
			\sim
			R^{-2},
		\end{aligned}
	\end{equation*}
	where we used $\mathcal{L}_0 \mathcal{Z}_{0,1}=0$. Then we have
	$ \lambda_{R,0}
	\lesssim
	( R^2 \ln R )^{-1}  $.
	
	For $k=1$, similarly,
	\begin{equation*}
		\mathcal{L}_1 w_1 =
		- \lambda_{R,1} w_1
		\mbox{ \ in \ } B_R,
		\quad
		w_1 = 0 \mbox{ \ on \ } \pp B_R,
	\end{equation*}
	\begin{equation*}
		w_1(\rho) =
		\mathcal{Z}_{1,2}(\rho) \int_0^\rho (- \lambda_{R,1} w_1(s) ) \mathcal{Z}_{1,1}(s) s d s
		+
		\mathcal{Z}_{1,1}(\rho) \int_\rho^{R} (- \lambda_{R,1} w_1(s) ) \mathcal{Z}_{1,2}(s) s d s
		-
		A_{R,1} \mathcal{Z}_{1,1}(\rho),
	\end{equation*}
	where $
	A_{R,1} = ( \mathcal{Z}_{1,1}(R) )^{-1} \mathcal{Z}_{1,2}(R) \int_0^{R} (- \lambda_{R,1} w_1(s) )
	\mathcal{Z}_{1,1}(s)
	s d s $. For $R$ large, we have  $ 1 = \| w_1 \|_{L^2(B_R)} \lesssim
	\lambda_{R,1}
	R^4 $, $ \| \eta_{\frac R2} \mathcal{Z}_{1,1} \|_{L^2(B_R)}^2\sim 1 $, $ Q_{R,1}(\eta_{\frac R2} \mathcal{Z}_{1,1} ,\eta_{\frac R2} \mathcal{Z}_{1,1} )
	\sim
	R^{-4} $ by $\mathcal{L}_1 \mathcal{Z}_{1,1}=0$, and then $\lambda_{R,1}
	\lesssim
	R^{-4}  $.

	For $k=-1$, similarly,
	\begin{equation*}
		\mathcal{L}_{-1} w_{-1} =
		- \lambda_{R,-1} w_{-1}
		\mbox{ \ in \ } B_{R},
		\quad
		w_{-1} = 0
		\mbox{ \ on \ } \pp B_{R},
	\end{equation*}
	\begin{equation*}
		w_{-1}(\rho) =
		\mathcal{Z}_{-1,2}(\rho) \int_0^\rho (- \lambda_{R,-1} w_{-1}(s) ) \mathcal{Z}_{-1,1}(s) s d s
		+
		\mathcal{Z}_{-1,1}(\rho) \int_\rho^{R} (- \lambda_{R,-1} w_{-1}(s) ) \mathcal{Z}_{-1,2}(s) s d s
		-
		A_{R,-1} \mathcal{Z}_{-1,1}(\rho),
	\end{equation*}
	where $
	A_{R,-1} = ( \mathcal{Z}_{-1,1}(R) )^{-1} \mathcal{Z}_{-1,2}(R) \int_0^{R} (- \lambda_{R,-1} w_{-1}(s) )
	\mathcal{Z}_{-1,1}(s)
	s d s $. For $R$ large, we have $1 = \| w_{-1} \|_{L^2(B_R)} \lesssim
	\lambda_{R,-1}
	R^2 \ln R $.

	For $|k|\ge 2$, $
	Q_{R,k}( f , f)
	\gtrsim
	2\pi |k|^2
	\int_{0}^R
	\frac{ |f| ^2}{\rho^2}
	\rho
	d \rho
	\ge
	|k|^2 R^{-2} \|f\|_{L^2(B_R)}^2 $.
\end{proof}

\begin{lemma}\label{energy est}
	Consider
	\begin{equation}\label{phi-k eq}
		\begin{cases}
			\pp_{\tau} \phi_k =	(a-ib)\mathcal{L}_k  \phi_k + h(\rho,\tau)
			\mbox{ \ in \ } \mathcal{D}_R,
			\\
			\phi_k = 0
			\mbox{ \ on \ } \pp\mathcal{D}_R,
			\quad
			\phi_k(\cdot,\tau_0) = 0
			\mbox{ \ in \ } B_{R(\tau_0)}
		\end{cases}
	\end{equation}
	with $R$ given in \eqref{nu-assump}. Denote $\Phi_k(y,\tau) = (\phi_k(\rho, \tau) e^{ik\theta})_{\mathbb{C}^{-1}}$. Then \eqref{phi-k eq} is equivalent to
	\begin{equation}\label{Phi-k-eq}
		\begin{cases}
			\pp_{\tau} \Phi_k =	(a-bW\wedge)( L_{\rm{in}} \Phi_k) + (h(\rho,\tau) e^{ik\theta} )_{\mathbb{C}^{-1}}
			\mbox{ \ in \ } \mathcal{D}_R,
			\\
			\Phi_k = 0
			\mbox{ \ on \ } \pp\mathcal{D}_R,
			\quad
			\Phi_k(\cdot,\tau_0) = 0
			\mbox{ \ in \ } B_{R(\tau_0)}.
		\end{cases}
	\end{equation}
	
	Suppose that $\| h(\cdot,\tau) \|_{L^2(B_R)}^2 \lesssim g(\tau)$,
	\begin{equation}\label{gassump}
		\int_{\tau_0}^\tau
		e^{c\int^s \tilde{\lambda}_{R,k}(z) dz }
		(\tilde{\lambda}_{R,k}(s) )^{-1}
		g(s) d s
		\lesssim C(c)
		e^{c\int^\tau \tilde{\lambda}_{R,k}(z) dz }
		\min\left\{\tau ,(\tilde{\lambda}_{R,k} )^{-1}
		\right\}
		(\tilde{\lambda}_{R,k} )^{-1}
		g(\tau)
	\end{equation}
	for any fixed constant $c>0$, a constant $C(c)>0$ depending on $c$ and
	\begin{equation*}
		\tilde{\lambda}_{R,0}= (R^2 \ln R)^{-1} ,\quad
		\tilde{\lambda}_{R,1}
		=
		R^{-4} ,
		\quad
		\tilde{\lambda}_{R,-1} =
		( R^2 \ln R)^{-1}   ,
		\quad
		\tilde{\lambda}_{R,k} = R^{-2}
		\mbox{ \ for \ } |k|\ge2,
	\end{equation*}
	then, we have the estimates
	\begin{equation}\label{k ne 1}
		\|\phi_k(\cdot,\tau) \|_{L^2(B_R)}
		\lesssim
		\big[
		\min\big\{\tau ,(\tilde{\lambda}_{R,k} )^{-1}
		\big\}
		(\tilde{\lambda}_{R,k})^{-1}
		g(\tau)
		\big]^{1/2}.
	\end{equation}
	In particular, when $\|h \|_{v,\ell}^{R}<\infty$, then $\| h(\cdot,\tau) \|_{L^2(B_R)}^2 \lesssim \big( \theta_{R,\ell}
	v(\tau) \|h \|_{v,\ell}^{R} \big)^2$, and
	\begin{equation}\label{qd240807-1}
		\begin{aligned}
			&	\|\phi_0(\cdot, \tau) \|_{L^\infty(B_{R})}
			\lesssim
			R^2 \ln R \theta_{R,\ell}
			v(\tau)
			\|h \|_{v,\ell}^{R} ,
			\quad
			\|\phi_1(\cdot,\tau) \|_{L^{\infty}(B_R)}
			\lesssim
			\min\{\tau^{\frac 12}, R^2\} R^2 \theta_{R,\ell}  v(\tau) \|h \|_{v,\ell}^{R},
			\\
			&
			\|\phi_{-1}(\cdot, \tau) \|_{L^\infty(B_{R})}
			\lesssim
			R^2 \ln R
			\theta_{R,\ell}
			v(\tau)  \|h \|_{v,\ell}^{R},
			\quad
			\|\phi_k(\cdot, \tau) \|_{L^\infty(B_{R })}
			\lesssim
			R^2 \theta_{R,\ell}
			v(\tau)
			\|h \|_{v,\ell}^{R}
			\mbox{ \ for \ } |k|\ge 2,
		\end{aligned}
	\end{equation}
	where
	\begin{equation}\label{theta-Rl-def}
		\theta_{R,\ell} :=
		\begin{cases}
			1
			&\mbox{ \ if \ } \ell>1
			\\
			(\ln R)^{\frac 12}
			& \mbox{ \ if \ } \ell=1
			\\
			R^{1-\ell}
			& \mbox{ \ if \ } \ell<1.
		\end{cases}
	\end{equation}
	
	For $k\in \mathbb{Z}$, if $\tilde{\lambda}_{R,k} ,
	g(\tau) \in {\mathbf{AP}}$, and
	either
	Case 1: ${\mathbf{P}}_1[\tilde{\lambda}_{R,k}] >-1$,
	or
	Case 2: ${\mathbf{P} }_1[\tilde{\lambda}_{R,k}] <-1$,  ${\mathbf{P} }_1 [(\tilde{\lambda}_{R,k}(s) )^{-1}
	g(s)] >-1$ holds, then \eqref{gassump} is true.
	
\end{lemma}

\begin{proof}
	
	Lemma \ref{complex-lem} connects \eqref{phi-k eq} and \eqref{Phi-k-eq}. The theory of parabolic systems guarantees the existence and uniqueness of the solution. By continuity argument, it suffices to assume that $h$ is smooth. Multiplying $\bar{\phi}_k$ to \eqref{phi-k eq} and integrating by parts, we have
	\begin{equation*}
		\int_{B_R}\pp_{\tau} \phi_k \bar{\phi}_k
		+
		(a-ib) Q_{R,k}(\phi_k,\phi_k)
		=
		\int_{B_R} h \bar{\phi}_k.
	\end{equation*}
	We take the real part for both parts and use $\phi_k = 0$ on $\pp\mathcal{D}_R$, then
	\begin{equation*}%\label{L1}
		\frac 12
		\pp_{\tau}  \int_{B_R}|\phi_k|^2
		+
		a Q_{R,k}(\phi_k,\phi_k)
		=
		\int_{B_R} \mathrm{Re} \left( h \bar{\phi}_k \right) .
	\end{equation*}
	By Lemma \ref{eigenvalue le}, we have
	\begin{equation*}
		\pp_{\tau}  \int_{B_R}|\phi_k|^2
		+
		c\tilde{\lambda}_{R,k}
		\int_{B_R}|\phi_k|^2
		\le
		2 \int_{B_R} |h| |\phi_k|
	\end{equation*}
	for a fixed constant $c>0$. By Young's inequality, we have
	\begin{equation*}
		\pp_{\tau}  \int_{B_R}|\phi_k|^2
		+
		c\tilde{\lambda}_{R,k}
		\int_{B_R}|\phi_k|^2
		\lesssim
		(\tilde{\lambda}_{R,k} )^{-1}
		\int_{B_R} |h|^2
		\lesssim
		(\tilde{\lambda}_{R,k} )^{-1}
		g(\tau) .
	\end{equation*}
	Since $\phi_k(\cdot,\tau_0)=0$ in $B_{R(\tau_0)}$, by \eqref{gassump}, we have \eqref{k ne 1}.

	When $\|h \|_{v,\ell}^{R}<\infty$, then $\| h(\cdot,\tau) \|_{L^2(B_R)}^2 \lesssim \big( \theta_{R,\ell}
	v(\tau) \|h \|_{v,\ell}^{R} \big)^2$. By \eqref{k ne 1}, $\|h \|_{v,\ell}^{R}<\infty$, the application of $W_p^{2,1}$ estimate given in \cite{dong11-higherLp} twice to \eqref{Phi-k-eq}, and the Sobolev embedding theorem, we have
	\begin{equation*}
		\|\phi_k(\cdot,\tau) \|_{L^{\infty}(B_R)}
		\lesssim
		\big\{
		\big[
		\min\big\{\tau ,(\tilde{\lambda}_{R,k} )^{-1}
		\big\}
		(\tilde{\lambda}_{R,k})^{-1}
		\big]^{1/2} \theta_{R,\ell} + 1 \big\} v(\tau) \|h \|_{v,\ell}^{R}.
	\end{equation*}
	Using $R^2 \ln R \le C_1 \tau$ in \eqref{nu-assump}, we have \eqref{qd240807-1}.
\end{proof}

\begin{lemma}\label{convert dim}
	Given an integer $k\ge 0$, consider
	\begin{equation*}
		\partial_{ \tau }\phi =
		(a-ib) \Big(
		\partial_{\rho\rho}\phi +\frac{1}{\rho}\partial_{\rho}\phi -\frac{k^2}{\rho^2}\phi
		\Big) +h(\rho,\tau) ,
		\quad
		\phi(\rho,\tau_0)=g(\rho),
	\end{equation*}
	where $h(\rho,\tau)$, $g(\rho)$ are some functions with sufficient space-time decay. Then using $\Gamma_{d}^{\natural}$ defined in \eqref{Gammad-def} gives a solution $\phi$ of the form
	\begin{equation}\label{A4}
		\phi(\rho, \tau ) =
		\rho^{ k}
		\Gamma_{2k+2}^{\natural}**\big(|y|^{-k }h(|y|,s) \big)(\rho,\tau,\tau_0) + \rho^{ k}
		\Gamma_{2k+2}^{\natural}*\big(|y|^{-k }g(|y|) \big)(\rho,\tau,\tau_0).
	\end{equation}
	
\end{lemma}

\begin{proof}
	Set $
	\phi(\rho,\tau) =\rho^{ k } \psi(\rho,\tau) $.
	Then
	\begin{equation}\label{z7}
		\pp_{\tau}\psi=
		(a-ib)
		\Big(\pp_{\rho\rho}\psi+\frac{2k+1}{\rho}\pp_{ \rho}\psi
		\Big)
		+
		\rho^{-k } h (\rho,\tau),
		\quad \psi(\rho,\tau_0) =  \rho^{-k } g(\rho),
	\end{equation}
	which can be regarded as the heat equation in $\RR^{2k+2}$. Then $\psi$ is given by
	\begin{equation*}
		\psi(\rho,\tau) =
		\Gamma_{2k+2}^{\natural}**\big(|y|^{-k}h(|y|,s)\big)(\rho,\tau,\tau_0) + \Gamma_{2k+2}^{\natural}*\big(|y|^{-k}g(|y|)\big)(\rho,\tau,\tau_0),
	\end{equation*}
	which satisfies \eqref{z7} in weak sense and pointwise sense except at $\rho=0$. \eqref{A4} follows.
\end{proof}

\subsection{Mode $k$, $|k|\ge 2$}\label{hmode-subsec}

In order to analyze the case that the right-hand side of the equation has singularity at $y =0$, given $\mathcal{R} = \mathcal{R}(\tau)$, we introduce the norm
\begin{equation}
	\| h\|_{v,\ell_{1},\ell}^{ \mathcal{R}} := \sup\limits_{(y,\tau)\in \DD_{ \mathcal{R} }} v(\tau)^{-1}
	\big(
	\1_{\{ |y|\le 1 \}} |y|^{\ell_{1}}
	+
	\1_{\{ |y| > 1 \}} |y|^{\ell }
	\big) |h(y,\tau)|.
\end{equation}
We use the notation $\| h\|_{v,\ell_{1},\ell}^{\infty}$  if $\mathcal{R}(\tau)=\infty$. Obviously,
\begin{equation}\label{qd24July09-1}
	\| h\|_{v,0,\ell}^{ \mathcal{R} } \sim \| h\|_{v,\ell}^{ \mathcal{R} };
	\quad
	\| h\|_{v,\ell_{1},\ell}^{ \mathcal{R} } \lesssim \| h\|_{v,\ell}^{ \mathcal{R} }
	\mbox{ \ if \ }  \ell_{1} > 0;
	\quad
	\| h\|_{v,\ell_{1},\ell}^{ \mathcal{R} } \gtrsim \| h\|_{v,\ell}^{ \mathcal{R} }
	\mbox{ \ if \ }
	\ell_{1} < 0.
\end{equation}

\begin{lemma}\label{modek-rough}
	Consider
	\begin{equation*}
		\pp_\tau \Psi_k =\left(a-b W\wedge\right) \left( L_{\rm{in}}  \Psi_k \right)  + H_k
		\mbox{ \ in \ } \DD_{R},
		\quad
		\Psi_k(\cdot,\tau_0) = 0
		\mbox{ \ in \ } B_{R(\tau_0)},
	\end{equation*}
	where $H_k = \big( h_{k} (\rho,\tau) e^{ik\theta} \big)_{\mathbb{C}^{-1}}$, $\| H_k \|_{v,\ell_{1},\ell}^{R}<\infty$. Suppose \eqref{nu-assump}, $\ell_{1} \in [0, 1.9]$, $\ell \in (1,3)$, $\frac{3}{2} + \mathbf{P}_1[v(\tau) R^{4-\ell}]>0$, then
	there exists a solution $\Psi_k = \TT_{kr}^{R}[H_k]$ as a mapping linear in $H_k$ with the estimate
	\begin{equation}\label{Phik-rough-est}  |\Psi_k|
		\lesssim |k|^{-1-(0.05)^2} v(\tau)
		R^{5-\ell} \langle y \rangle^{-3} \ln(|y|+2) \| H_k \|_{v,\ell_{1},\ell}^{R}  \mbox{ \ in \ } \DD_{R},
	\end{equation}
	where ``$\lesssim$'' is independent of $k$.
	Moreover, $\Psi_k\cdot W=0$ and $ e^{-ik\theta}\left(\Psi_k\right)_{\mathbb{C}}$ is radial in space.
	
\end{lemma}

\begin{proof} For brevity, denote $\| h_k \| = \| h_k \|_{v,\ell_{1},\ell}^{R}$ in this proof.
	Assume $h_k(\rho,\tau)=0$ in $\DD_{R}^{c}$.
	Consider
	\begin{equation*}
		\left(a-b W\wedge\right) \left( L_{\rm{in}} G_k \right) = H_k, \mbox{ \ where \ } G_k = \big(g_k(\rho,\tau) e^{ik\theta} \big)_{\mathbb{C}^{-1}}.
	\end{equation*}
	By Lemma \ref{complex-lem}, it is equivalent to considering 	
	\begin{equation*}
		\left(a-ib\right) \mathcal{L}_k g_k = h_k ,
	\end{equation*}
	where $g_k$ is given by
	\begin{equation}\label{gk-def}
		g_k(\rho,\tau) =
		(a+ib)
		\begin{cases}
			\mathcal{Z}_{k,2}(\rho)\int_0^{\rho}
			\mathcal{Z}_{k,1}(r)  h_k(r,\tau) r d r
			+
			\mathcal{Z}_{k,1}(\rho) \int_{\rho}^{\infty}
			\mathcal{Z}_{k,2}(r)  h_k(r,\tau) r d r
			&
			\mbox{ \ if \ } k\le -2
			\\
			-\mathcal{Z}_{k,2}(\rho) \int_{\rho}^{\infty}
			\mathcal{Z}_{k,1}(r)  h_k(r,\tau) r d r
			-
			\mathcal{Z}_{k,1}(\rho)\int_0^{\rho}
			\mathcal{Z}_{k,2}(r)  h_k(r,\tau) r d r
			&
			\mbox{ \ if \ } k\ge 2.
		\end{cases}
	\end{equation}
	We will estimate the upper bound of $g_k$.
	For $k\le -2$, $\rho\le 1$,
	\begin{equation}\label{qd240811-1}
		\Big|\mathcal{Z}_{k,2}(\rho)\int_0^{\rho}
		\mathcal{Z}_{k,1}(r)  h_k(r,\tau) r d r \Big|
		\lesssim
		|k|^{-1} \| h_k\|v(\tau) \rho^{k-1}
		\int_0^{\rho}
		r^{2-k-\ell_{1}}   d r
		\sim
		|k|^{-2}
		\| h_k \|v(\tau)
		\rho^{2-\ell_{1}}
	\end{equation}
	for $0\le \ell_{1} \le 4$.
	\begin{equation*}
		\begin{aligned}
			&
			\Big|\mathcal{Z}_{k,1}(\rho) \int_{\rho}^{\infty}
			\mathcal{Z}_{k,2}(r)  h_k(r,\tau) r d r \Big|
			\lesssim
			\rho^{1-k} \|h_k\| v(\tau)
			|k|^{-1} \Big( \int_{\rho}^{1}
			r^{k -\ell_{1} }   d r +  \int_{1}^{\infty}
			r^{k+2-\ell}    d r \Big)
			\\
			\lesssim \ &
			|k|^{-1} \rho^{1-k} \|h_k\| v(\tau)
			\Big[
			\frac{\rho^{k-\ell_{1} +1}}{\ell_{1} -(k+1) }
			+  \frac{1}{\ell-(k+3)} \Big]
			\lesssim  C_1(\ell)
			|k|^{-2}
			\|h_k \| v(\tau)
			\rho^{2-\ell_{1}}
		\end{aligned}
	\end{equation*}
	for $0\le \ell_{1} \le  4$ and $1<\ell\le 5$, where $C_1(\ell) \rightarrow \infty$ as $\ell \rightarrow 1$ due to $k=-2$.
	
	For $k\le -2$, $\rho\ge 1$,
	\begin{equation}\label{qd240811-2}
		\begin{aligned}
			&
			\Big|\mathcal{Z}_{k,2}(\rho)\int_0^{\rho}
			\mathcal{Z}_{k,1}(r)  h_k(r,\tau) r d r \Big|
			\lesssim
			|k|^{-1} \| h_k\|v(\tau) \rho^{k+1}
			\Big(
			\int_0^{1}
			r^{2-k -\ell_{1} }    d r
			+ \int_1^{\rho}   r^{-\ell-k}  d r \Big)
			\\
			\lesssim \ &
			|k|^{-1} \| h_k\|v(\tau) \rho^{k+1}
			\Big(
			\frac{1}{3-k-\ell_{1}}
			+
			\frac{\rho^{1-k-\ell}}{1-k-\ell}
			\Big)
			\lesssim
			C_2(\ell) |k|^{-2} \| h_k \|v(\tau) \rho^{2-\ell}
		\end{aligned}
	\end{equation}
	for $0\le \ell_{1}\le 4$ and $0 \le \ell < 3$, where $C_{2}(\ell) \rightarrow \infty$ as $\ell \rightarrow 3$ due to $k=-2$.
	\begin{equation*}
		\begin{aligned}
			&
			\Big|\mathcal{Z}_{k,1}(\rho) \int_{\rho}^{\infty}
			\mathcal{Z}_{k,2}(r)  h_k(r,\tau) r d r \Big|
			\lesssim
			\rho^{-1-k} \|h_k\| v(\tau) \int_{\rho}^{\infty}
			|k|^{-1} r^{k+2 -\ell }  d r
			\lesssim
			C_3(\ell)	|k|^{-2}
			\|h_k\| v(\tau) \rho^{2-\ell}
		\end{aligned}
	\end{equation*}
	for $1 < \ell \le 4$, where $C_3(\ell) \rightarrow \infty$ as $\ell\rightarrow 1$ due to $k=-2$.
	
	In sum, for $0\le \ell_{1} \le 4$, $1 < l < 3$, $k\le -2$,
	\begin{equation}\label{Zg-1}
		\| g_k \|_{v,\ell_{1}-2,\ell-2}^{\infty} \lesssim C_4(\ell) |k|^{-2} \| h_k \|,
	\end{equation}
	where $C_4(\ell)  \rightarrow \infty$ as $\ell \rightarrow 1 $ or $3$.
	
	For $k\ge 2$, $\rho\le 1$, $0\le \ell_{1}\le 3$, $0\le \ell \le 4$,
	\begin{align*}
			&
			\Big|\mathcal{Z}_{k,2}(\rho) \int_{\rho}^{\infty}
			\mathcal{Z}_{k,1}(r)  h_k(r,\tau) r d r \Big|
			\lesssim
			k^{-1} \| h_k\| v(\tau) \rho^{k-1}
			\Big(
			\int_{1}^{\infty}
			r^{-k-\ell}  d r
			+
			\int_{\rho}^{1}
			r^{2-k-\ell_{1}}  d r
			\Big)
			\\
			= \ &
			k^{-1} \| h_k \| v(\tau) \rho^{k-1}
			\Big(
			\frac{1}{k+\ell-1}
			+
			\frac{1-\rho^{3-k-\ell_{1} }}{3-k-\ell_{1}} \1_{\{ \ell_{1} \ne 3-k \}}
			+
			(-\ln \rho)\1_{\{ \ell_{1} =3-k \}}
			\Big)
			\\
			\lesssim  \  &
			\| h_k\| v(\tau)
			\begin{cases}
				k^{-1}  \rho^{k-1}
				\big(
				k^{-1}
				+
				\frac{ \rho^{3-k-\ell_{1}} }{k+\ell_{1}-3}
				\big)
				\sim
				k^{-2} \rho^{2-\ell_{1}},
				&
				k\ge 4
				\\
				\rho^{2}
				\langle \ln \rho \rangle ,
				&
				k=3, ~\ell_{1} =0
				\\
				\rho^{2}
				\big(
				1
				+
				\frac{\rho^{-\ell_{1}}-1}{\ell_{1}}
				\big)
				\lesssim
				\rho^{2-\ell_{1}}
				\langle \ln \rho \rangle ,
				&
				k=3, ~0<\ell_{1} \le 3
				\\
				\rho
				\big(1
				+
				\frac{1-\rho^{1-\ell_{1}} }{1-\ell_{1}} \big)
				\lesssim
				\rho \langle \ln \rho \rangle,
				&
				k=2, ~0\le \ell_{1} <1
				\\
				\rho
				\langle \ln \rho \rangle,
				&
				k=2,~ \ell_{1} =1
				\\
				\rho
				\big(
				1
				+
				\frac{\rho^{1-\ell_{1}} - 1}{\ell_{1} -1}
				\big)
				\lesssim
				\rho^{2-\ell_{1}} \langle \ln \rho \rangle,
				&
				k=2,~ 1< \ell_{1} \le 3,
			\end{cases}
		\end{align*}
	where we used $\rho^{t} - 1 = t \rho^{c t} \ln \rho $ for some $c\in [0,1]$.

	For the other part,
	\begin{equation*}
		\Big|\mathcal{Z}_{k,1}(\rho)\int_0^{\rho}
		\mathcal{Z}_{k,2}(r)  h_k(r,\tau) r d r \Big|
		\lesssim
		\| h_k \| v(\tau)
		\rho^{1-k} \int_{0}^{\rho}
		k^{-1} r^{k-\ell_{1}}  d r
		\sim
		k^{-2} \| h_k \| v(\tau)
		\rho^{2 -\ell_{1}}
	\end{equation*}
	for $0\le \ell_{1} \le 2.9$.
	For $k\ge 2$, $\rho\ge 1$,
	\begin{equation*}
		\Big|\mathcal{Z}_{k,2}(\rho) \int_{\rho}^{\infty}
		\mathcal{Z}_{k,1}(r)  h_k(r,\tau) r d r \Big|
		\lesssim
		k^{-1} \| h_k \| v(\tau) \rho^{k+1}
		\int_{\rho}^{\infty}
		r^{-k-\ell} d r
		\sim
		k^{-2} \| h_k \| v(\tau) \rho^{2-\ell},
	\end{equation*}
	when $0\le \ell \le 4$.
	\begin{equation*}
		\Big|
		\mathcal{Z}_{k,1}(\rho)\int_0^{\rho}
		\mathcal{Z}_{k,2}(r)  h_k(r,\tau) r d r \Big|
		\lesssim
		k^{-1}
		\| h_k \| v(\tau) \rho^{-1-k}
		\Big(\int_0^{1}
		r^{k-\ell_{1}}  d r + \int_1^{\rho}  r^{k+2 -\ell }  d r \Big)
		\lesssim
		k^{-2} \| h_k \| v(\tau) \rho^{2-\ell},
	\end{equation*}
	when $0\le \ell_{1}\le 2.9$, $0\le \ell \le 4$.
	
	In sum, for $0\le \ell_{1}\le 2.9$, $0\le \ell \le 4$,
	\begin{equation}\label{Zg-2}
		\begin{cases}
			\| g_k \|_{v,\ell_{1}-2,\ell-2}^{\infty} \lesssim  k^{-2} \| h_k \|,
			&
			k\ge 4
			\\
			\| g_3 \|_{v,\epsilon+\ell_{1}-2,\ell-2}^{\infty} \lesssim C(\epsilon) \| h_k \| ,
			&
			k=3
			\\
			\| g_2  \|_{v,\epsilon+ (\ell_{1} -1)_{+} -1,\ell-2}^{\infty} \lesssim C(\epsilon) \| h_k \|,
			&
			k=2,
		\end{cases}
	\end{equation}
	where $\epsilon>0$ could be an arbitrarily small constant, and $C(\epsilon)$ is a constant depending on $\epsilon$.
	
	Combining \eqref{Zg-1}, \eqref{Zg-2}, \eqref{qd24July09-1}, for $0\le \ell_{1} \le 1.9$, $1 < \ell < 3$, we have
	\begin{equation}\label{H-est}
		\| G_k \|_{v,\ell-2}^{\infty}=	\| g_k \|_{v,\ell-2}^{\infty} \lesssim C(\ell ) |k|^{-2}  \| h_k \| \mbox{ \ for \ } |k|\ge 2,
	\end{equation}
	where $C(\ell)  \rightarrow \infty$ as $\ell \rightarrow 1 $ or $3$ and $C(\ell)$ will vary from line to line.
	Consider
	\begin{equation}\label{Phik-eq}
		\begin{cases}
			\pp_{\tau} \Phi_k =
			\left(a-bW\wedge \right) \left( L_{\rm{in}} \Phi_k \right) + G_k
			\mbox{ \ in \ } \DD_{2R},
			\\
			\Phi_k = 0
			\mbox{ \ on \ } \pp\DD_{2R}, \quad
			\Phi_k(\cdot,\tau_0) = 0
			\mbox{ \ in \ } B_{2R(\tau_0)} .
		\end{cases}
	\end{equation}
	To find a solution $\Phi_k$ with the form $\Phi_{ k} = \left(  \phi_{k} (\rho,\tau) e^{ik\theta}  \right)_{\mathbb{C}^{-1}}$, by Lemma \ref{complex-lem}, it suffices to consider
	\begin{equation}\label{phik-eq}
		\begin{cases}
			\pp_{\tau} \phi_k =
			\left(a-ib  \right)  \mathcal{L}_k \phi_k  + g_k
			\mbox{ \ in \ } \DD_{2R},
			\\
			\phi_k = 0
			\mbox{ \ on \ } \pp\DD_{2R}, \quad
			\phi_k(\cdot,\tau_0) = 0
			\mbox{ \ in \ } B_{2R(\tau_0)}.
		\end{cases}
	\end{equation}
	The existence and uniqueness follow by the theory of parabolic systems.
	
	Using ${\mathbf{P}}_1[R] < 1/2$ in  \eqref{nu-assump}, we have ${\mathbf{P}}_1[ \tilde{\lambda}_{R,k} ] >-1$ with $\tilde{\lambda}_{R,k} = R^{-2}$, $|k|\ge 2$. Then \eqref{gassump} is true.
	Thus, applying Lemma \ref{energy est} to \eqref{phik-eq}, we have
	\begin{equation}\label{phik-Linfty}
		\|\phi_k(\cdot,\tau)\|_{L^{\infty}(B_{2 R(\tau)})}
		\lesssim
		v(\tau) R^{5-\ell}
		\|g_k \|_{v,\ell -2}^{\infty}.
	\end{equation}
	
	To improve the spatial decay of $\phi_k$, we reformulate the equation \eqref{phik-eq} into the following form
	\begin{equation}\label{HZ-4}
		\begin{cases}
			\pp_{\tau} \phi_k =
			(a-ib)
			\big[
			\pp_{\rho\rho} \phi_k
			+
			\frac{\pp_{\rho} \phi_k }{\rho}
			-
			\frac{(k+1)^2}{\rho^2} \phi_k
			\big]
			+ \tilde{g}_k
			\mbox{ \ in \ } \mathcal{D}_{2R},
			\\
			\phi_k = 0
			\mbox{ \ on \ } \pp\mathcal{D}_{2R},
			\quad
			\phi_k(\cdot,\tau_0) = 0
			\mbox{ \ in \ } B_{2R(\tau_0)} ,
		\end{cases}
	\end{equation}
	where $\tilde{g}_k(\rho,\tau) :=
	(a-ib)\frac{(4k+8)\rho^2 +4k}{(\rho^2+1)^2}
	\frac{1}{\rho^2}
	\phi_k + g_k$.
	Set $\phi_{*k}(y,\tau) = e^{i(k+1)\theta} \phi_k(\rho,\tau)$. Then \eqref{HZ-4} is equivalent to
	\begin{equation}\label{HZ-5}
		\begin{cases}
			\pp_{\tau} \phi_{*k} =
			(a-ib)
			\Delta_{\mathbb{R}^2} \phi_{*k}
			+ e^{i(k+1)\theta} \tilde{g}_k
			\mbox{ \ in \ } \mathcal{D}_{2R},
			\\
			\phi_{*k} = 0
			\mbox{ \ on \ } \pp\mathcal{D}_{2R},
			\quad
			\phi_{*k}(\cdot,\tau_0) = 0
			\mbox{ \ in \ } B_{2R(\tau_0)} .
		\end{cases}
	\end{equation}
	Complex-valued equation \eqref{HZ-5} can be regarded as a real-valued parabolic system in a varying-time domain in $\mathbb{R}^{2+1}$. Combining \cite[Theorem 3.2]{14-DongKim-varyingtime} and \cite[Lemma 2.26 and Remark 2.27]{Kaj17},
	there exists a fundamental solution  $\Gamma_2(x,y,\tau,s)$ for the homogeneous part of \eqref{HZ-5} with the estimate
	\begin{equation*}
		| \Gamma_2(x,y,\tau,s)| \le C \left( \tau-s  \right)^{-1} e^{-\frac{\kappa|x-y|^2}{\tau-s}}
	\end{equation*}
	for some constants $C, \kappa>0$. Then, by scaling argument, we have
	\begin{equation}\label{H-Z7}
		\left| \nabla_{y} \Gamma_2(x,y,\tau,s)  \right| \lesssim
		\left( \tau-s  \right)^{-\frac{3}{2}} e^{-\frac{\kappa_1 |x-y|^2}{\tau-s}}
	\end{equation}
	for a constant $\kappa_1>0$.
	Then $\phi_{*k}$ can be written as
	\begin{equation}\label{H-Z5}
		\phi_{*k}(y,\tau)=
		\int_{\tau_0}^{\tau} \int_{B_{2R(s)}}
		\Gamma_2(y,z,\tau,s)
		e^{i(k+1)\theta(z)} \tilde{g}_k(|z|,s) dz ds,
	\end{equation}
	where $\theta(z)= \arctan(z_2/z_1)$.
	
	For utilizing the special form of $e^{i(k+1)\theta} \tilde{g}_k$, we set $\tilde{g}_k=0$ in $\mathcal{D}_{2R}^{c}$ and want to find $\tilde{P}_{k}(y,\tau)$ satisfying
	\begin{equation}\label{H-Z6}
		\Delta_{\mathbb{R}^2 } \tilde{P}_{k}(y,\tau) =  e^{i(k+1)\theta} \tilde{g}_k \mbox{ \ in \ } \mathbb{R}^2.
	\end{equation}
	Set $\tilde{P}_{k}(y,\tau) = e^{i(k+1)\theta} \tilde{p}_{k}(\rho,\tau)$. Then
	\begin{equation*}
		\pp_{\rho\rho} \tilde{p}_{k}
		+
		\rho^{-1} \pp_{\rho } \tilde{p}_{k}
		-
		(k+1)^2 \rho^{-2}  \tilde{p}_{k}
		= \tilde{g}_k.
	\end{equation*}
	Set $\tilde{p}_{k} =\rho^{|k+1|} \tilde{p}_{k,1}(\rho,\tau)$. It is equivalent to
	\begin{equation*}
		\pp_{\rho\rho}  \tilde{p}_{k,1} + \left(2|k+1| +1\right) \rho^{-1} \pp_{\rho }  \tilde{p}_{k,1}  = \rho^{-|k+1|} \tilde{g}_k.
	\end{equation*}
	We take $\tilde{p}_{k,1}$ as
	\begin{equation*}
		\tilde{p}_{k,1}(\rho,\tau)
		=
		- \rho^{-2|k+1|}\int_{0}^\rho
		u^{2|k+1|-1} \int_{u}^\infty
		r r^{-|k+1|} \tilde{g}_k(r,\tau) dr du .
	\end{equation*}
	Notice
	\begin{equation*}
		|\tilde{g}_k | \lesssim  \1_{\{ r\le 2R(\tau)\}}
		\big[
		|k| \big( \rho^{-2} \1_{\{ \rho\le 1\}} + \rho^{-4} \1_{\{ \rho > 1\}} \big)
		|\phi_k| +  v(\tau) \langle \rho \rangle^{2-\ell}\|g_k \|_{v,\ell -2}^{\infty} \big].
	\end{equation*}
	Then
	\begin{equation*}
		\begin{aligned}
			&
			\left|	\tilde{p}_{k,1} \right|
			\lesssim
			\rho^{-2|k+1|}\int_{0}^\rho
			u^{2|k+1|-1} \int_{u}^\infty
			r  \1_{\{ r\le 2R(\tau)\}}
			\\
			& \times
			\Big[
			|k| \Big( r^{-2 -|k+1| } \1_{\{ r\le 1\}} + r^{-4 -|k+1| } \1_{\{ r > 1\}} \Big)
			|\phi_k| +  v(\tau)
			\Big(
			r^{ -|k+1| }\1_{\{ r\le 1 \}}
			+
			r^{2-\ell -|k+1| }
			\1_{\{ r > 1 \}}
			\Big)
			\|g_k \|_{v,\ell -2}^{\infty}
			\Big]
			dr du.
		\end{aligned}
	\end{equation*}
	Here, by Lemma \ref{ellptic-compa}, one has
	\begin{align*}
			&
			\rho^{-2|k+1|}\int_{0}^\rho
			u^{2|k+1|-1} \int_{u}^\infty
			\1_{\{ r\le 2R(\tau)\}}
			r    v(\tau)
			\Big(
			r^{ -|k+1| }\1_{\{ r\le 1 \}}
			+
			r^{2-\ell -|k+1| }
			\1_{\{ r > 1 \}}
			\Big)
			\|g_k \|_{v,\ell -2}^{\infty}
			dr du
			\\
			\lesssim \ &
			C(\ell)
			v(\tau) \|g_k \|_{v,\ell -2}^{\infty}
			\begin{cases}
				|k|^{-2}
				\left( \rho^{2-|k+1|} \1_{\{ \rho\le 1 \}} + \rho^{4-\ell-|k+1|} \1_{\{ \rho > 1 \}} \right)
				&
				\mbox{ \ for \ } k\le -4 \mbox{ or } k\ge 2
				\\
				R(\tau) \left( \langle \ln \rho \rangle \1_{\{ \rho\le 1 \}} + \rho^{1-\ell } \1_{\{ \rho > 1 \}}
				\right)
				&
				\mbox{ \ for \ } k = -3
				\\
				R^2(\tau)
				\left(
				\1_{\{ \rho\le 1 \}} + \rho^{1-\ell } \1_{\{ \rho > 1 \}}
				\right)
				&
				\mbox{ \ for \ } k = -2,
			\end{cases}
		\end{align*}
	where we used
	$
	\1_{\{ r\le 2R(\tau)\}}
	r^{2-\ell -|k+1| }
	\1_{\{ r > 1 \}}
	\lesssim
	\begin{cases}
		R(\tau)  r^{-1-\ell}
		&
		\mbox{ \ for \ } k=-3
		\\
		R^2(\tau)  r^{-1-\ell}
		&
		\mbox{ \ for \ } k=-2
	\end{cases}$.
	By \eqref{phik-Linfty} and Lemma \ref{ellptic-compa},
	\begin{align*}
			&
			\rho^{-2|k+1|}\int_{0}^\rho
			u^{2|k+1|-1} \int_{u}^\infty
			r  \1_{\{ r\le 2R(\tau)\}}
			|k| \left( r^{-2 -|k+1| } \1_{\{ r\le 1\}} + r^{-4 -|k+1| } \1_{\{ r > 1\}} \right)
			|\phi_k|
			dr du
			\\
			\lesssim \ & C(\ell)
			v(\tau) R^{5-\ell}(\tau)
			\|g_k \|_{v,\ell -2}^{\infty}
			\begin{cases}
				|k|^{-1}
				\left(
				\rho^{-|k+1|}\1_{\{ \rho\le 1 \}} + \rho^{ -2-|k+1| } \1_{\{ \rho > 1 \}}
				\right)
				&
				\mbox{ \ for \ } k\le -4 \mbox{ or } k\ge 2
				\\
				\rho^{-2}\1_{\{ \rho\le 1 \}} + \rho^{ -4 } \langle \ln \rho \rangle \1_{\{ \rho > 1 \}}
				&
				\mbox{ \ for \ } k = -3
				\\
				\rho^{-1}\1_{\{ \rho\le 1 \}} + \rho^{ -2} \1_{\{ \rho > 1 \}}
				&
				\mbox{ \ for \ } k = -2.
			\end{cases}
		\end{align*}
	$|\pp_{\rho} \tilde{p}_{k,1} |$ can be bounded by \eqref{pp|x|-u} in Lemma \ref{ellptic-compa}  similarly. As a result, for $\rho \le 2R(\tau)$,
	\begin{align*}
			&
			|k|^{-1} \rho |\pp_{\rho} \tilde{p}_{k,1} |
			+  |\tilde{p}_{k,1} |
			\\
			\lesssim \ &
			C(\ell)
			v(\tau) R^{5-\ell}(\tau)
			\|g_k \|_{v,\ell -2}^{\infty}
			\begin{cases}
				|k|^{-1}
				\big(
				\rho^{-|k+1|}\1_{\{ \rho\le 1 \}} + \rho^{ -1-|k+1| } \1_{\{ \rho > 1 \}}
				\big),
				& k\le -4 \mbox{ or } k\ge 2
				\\
				\rho^{-2}\1_{\{ \rho\le 1 \}} + \rho^{ -3 } \1_{\{ \rho > 1 \}} ,
				&
				k = -3
				\\
				\rho^{-1}\1_{\{ \rho\le 1 \}} + \rho^{ -2} \1_{\{ \rho > 1 \}} ,
				&
				k = -2.
			\end{cases}
		\end{align*}
	Notice $\tilde{P}_{k}(y,\tau) = e^{i(k+1)\theta} \rho^{|k+1|} \tilde{p}_{k,1}(\rho,\tau)$. Then
	\begin{equation}\label{qd240808-3}
		\begin{aligned}
			&
			\big|\nabla \tilde{P}_{k} \big| =
			\Big(
			\big|\pp_{\rho } \tilde{P}_{k}  \big|^2
			+ \rho^{-2}  \big|\pp_{\theta} \tilde{P}_{k}  \big|^2 \Big)^{1/2}
			=
			\Big(
			\big| |k+1| \rho^{|k+1|-1} \tilde{p}_{k,1} + \rho^{|k+1|} \pp_{\rho } \tilde{p}_{k,1} \big|^2
			+ \rho^{-2}  |k+1|^2 \big|  \rho^{|k+1|} \tilde{p}_{k,1} \big|^2 \Big)^{1/2}
			\\
			\lesssim \ &
			|k+1| \rho^{|k+1|-1}
			\big( |\tilde{p}_{k,1}| + |k+1|^{-1} \rho  | \pp_{\rho }  \tilde{p}_{k,1} |
			\big)
			\lesssim
			C(\ell)
			v(\tau) R^{5-\ell}(\tau)
			\big(
			\rho^{-1}\1_{\{ \rho\le 1 \}} + \rho^{ -2} \1_{\{ \rho > 1 \}}
			\big)
			\|g_k \|_{v,\ell -2}^{\infty}.
		\end{aligned}
	\end{equation}
	
	By $\phi_{*k}(y,\tau) = e^{i(k+1)\theta} \phi_k(\rho,\tau)$, \eqref{H-Z5} and \eqref{H-Z6}, we have
	\begin{equation*}
		\phi_{k}(y,\tau)=
		- e^{ -i(k+1)\theta}
		\int_{\tau_0}^{\tau} \int_{B_{2R(s)}}
		\nabla_{z}
		\Gamma_2(y,z,\tau,s)
		\cdot
		\nabla \tilde{P}_{k}(z,s) dz ds.
	\end{equation*}
	By \eqref{H-Z7} and \eqref{qd240808-3}, then
	\begin{equation*}
		\left|\phi_{k} \right| \lesssim
		C(\ell) \|g_k \|_{v,\ell -2}^{\infty}
		\int_{\tau_0}^{\tau} \int_{B_{2R(s)}}
		\left( \tau-s  \right)^{-\frac{3}{2}} e^{-\frac{\kappa_1 |y-z|^2}{\tau-s}}
		v(s) R^{5-\ell}(s)
		\big(
		|z|^{-2}\1_{\{ |z|\le 1 \}} + |z|^{ -3} \1_{\{ |z| > 1 \}}
		\big)
		|z|
		dz ds.
	\end{equation*}
	By similar estimates as \cite[Lemma A.1]{infi4d} and Lemma \ref{qd24July10-2-lem} in $\mathbb{R}^3$,
	provided $\frac{3}{2} + \mathbf{P}_1[v(t) R^{4-\ell}]>0$, we have
	\begin{equation}\label{H-Z8}
		\left|\phi_{k} \right| \lesssim
		C(\ell)
		v(\tau) R^{5-\ell} |y|^{-1} \ln(|y|+2) \|g_k \|_{v,\ell -2}^{\infty}
		\mbox{ \ for \ } 1\le |y| < 2 R.
	\end{equation}
	Combining \eqref{phik-Linfty}, \eqref{H-Z8}, and \eqref{H-est}, we have
	\begin{equation}\label{qd240811-6}
		|\Phi_k| = |\phi_k| \lesssim C(\ell) |k|^{-2} v(\tau) R^{5-\ell} \langle y \rangle^{-1} \ln(|y|+2) \| h_k \|
		\mbox{ \ for \ } |y|< 2R.
	\end{equation}

	To get the pointwise estimate of $|D^2 \Phi_k|$, we need to calculate the $\mathsf{DMO_x}$ semi-norm of $G_k = \big(g_k(\rho,\tau) e^{ik\theta} \big)_{\mathbb{C}^{-1}}$. Recall \eqref{def-E1E2} and \eqref{gk-def}. We consider the following typical term in $G_k$.
	\begin{equation}\label{qd240811-4}
		f_1(y,\tau) = f_{11}(\rho,\tau)
		e^{ik\theta} \cos(\theta),
		\  f_{11}(\rho,\tau):=
		\mathcal{Z}_{k,2}(\rho)\int_0^{\rho}
		\mathcal{Z}_{k,1}(r)  h_k(r,\tau) r d r \frac{\rho^2-1}{\rho^2+1}
		\mbox{ \ for \ } k\le -2.
	\end{equation}

	We assume $(x,t) \in Q_2^{-}(0,\tau)$, $r\in (0,1)$, and arbitrary points $(y,s), (z,s) \in Q_r^{-}(x,t)$. Obviously, $s\sim\tau$. If $r\ge |x|/2$, then $|y|, |z| \le |x|+r \le 3r$. By \eqref{qd240811-1}, \eqref{qd240811-2}, and $2-\ell_1>0$, we have
	\begin{equation*}
		| f_1(y,s)- f_1(z,s) | \lesssim |k|^{-2} \|h_k\| v(\tau) r^{2-\ell_1}.
	\end{equation*}
	It is direct to see that
	\begin{equation*}
		\begin{aligned}
			&
			\partial_{\rho} ( f_{11}(\rho,\tau) )
			=
			\mathcal{Z}_{k,2}(\rho)\int_0^{\rho}
			\mathcal{Z}_{k,1}(r)  h_k(r,\tau) r d r
			\frac{4\rho}{(\rho^2+1)^2}
			\\
			& \quad+
			\mathcal{Z}'_{k,2}(\rho)\int_0^{\rho}
			\mathcal{Z}_{k,1}(r)  h_k(r,\tau) r d r \frac{\rho^2-1}{\rho^2+1}
			+
			\mathcal{Z}_{k,2}(\rho)
			\mathcal{Z}_{k,1}(\rho)  h_k(\rho,\tau) \rho \frac{\rho^2-1}{\rho^2+1}.
		\end{aligned}
	\end{equation*}
	By \eqref{qd240811-1}, \eqref{qd240811-2}, and \eqref{scalr-Z}, we have
	\begin{equation}\label{qd240811-9}
		\big| \partial_{\rho} f_{11}(\rho,\tau) \big|
		\lesssim
		|k|^{-1} \|h_k\| v(\tau) \big( \rho^{1-\ell_1} \1_{ \{ 0<\rho\le 1 \} }
		+
		\rho^{1-\ell} \1_{ \{ \rho> 1 \} }
		\big).
	\end{equation}
	Note that
	\begin{equation*}
		\begin{aligned}
			&
			e^{ik\theta(y)} f( \theta(y) ) - e^{ik\theta(z)} f( \theta(z) )
			\\
			= \ &
			e^{ik \theta(s_1 y + (1-s_1)z )}  (\nabla \theta)(s_1 y+(1-s_1)z ) \cdot (y-z)
			\big[ik f\big( \theta(s_1 y+(1-s_1)z) \big) +
			f'\big( \theta(s_1 y+(1-s_1)z) \big)
			\big]
		\end{aligned}
	\end{equation*}
	for some $s_1 \in [0,1]$.
	When $|f(\theta_1)|\le C_1, |f'(\theta_1)| \le C_1$ with a constant $C_1$ for all $\theta_1 \in [-\pi/2,\pi/2]$, then
	\begin{equation}\label{qd240808-9}
		\big| e^{ik\theta(y)} f( \theta(y) ) - e^{ik\theta(z)} f( \theta(z) ) \big|
		\le
		\frac{C_1 (|k|+1)}{| s_1 y+(1-s_1 )z |} |y-z|.
	\end{equation}
	
	If $r < |x|/2$, then $|x|/2 \le |y|, ~|z|\le 3|x|/2$. It follows that
	\begin{align*}
			&
			|f_1(y,s) - f_1(z,s) |
			\\
			= \ &
			\big| \big( f_{11}(|y|,s) - f_{11}(|z|,s) \big) e^{ik\theta(y)} \cos(\theta(y)) +
			f_{11}(|z|,s)
			\big( e^{ik\theta(y)} \cos(\theta(y)) - e^{ik\theta(z)} \cos(\theta(z))
			\big)
			\big|
			\\
			\lesssim \ &
			|k|^{-1} \|h_k\| v(\tau) |x|^{1-\ell_1} \big||y| - |z| \big|
			+
			|k|^{-2} \|h_k\| v(\tau) |x|^{2-\ell_1}
			(|k|+1) |x|^{-1} r
			\\
			\lesssim \ &
			|k|^{-1} \|h_k\| v(\tau) |x|^{1-\ell_1} r
			=
			|k|^{-1} \|h_k\| v(\tau) |x|^{1.9-\ell_1} |x|^{-0.9} r
			\lesssim
			|k|^{-1} \|h_k\| v(\tau) r^{0.1},
		\end{align*}
	where we used $1.9-\ell_1\ge 0$, $|x|\lesssim 1$, and $r < |x|/2$ for the last step.
	
	In sum, $|\omega|_{f_1}^{\mathsf x}(r,Q_2^{-}(0,\tau) ) \lesssim
	|k|^{-1} \|h_k\| v(\tau) r^{0.1} $. The other terms in $G_k(y,\tau)$ could be handled similarly and we deduce
	\begin{equation}\label{qd240811-3}
		|\omega|_{G_k}^{\mathsf x}(r,Q_2^{-}(0,\tau) ) \lesssim
		|k|^{-1} \|h_k\| v(\tau) r^{0.05}.
	\end{equation}
	We will not use $[ G_k]_{\mathsf{|DMO|_x}(Q_2^{-}(0,\tau))} \lesssim
	|k|^{-1} \|h_k\| v(\tau)$ to deduce the pointwise estimate of $|D^2 \Phi_k|$ directly since $\sum_{k\in\mathbb{Z}} |k|^{-1}$ is divergent.
	
	For any $(z,s) \in Q_r^{-}(x,t)$, we have $(x_* + \rho_* z, t_* +\rho_*^2 s) \in Q_{r\rho_*}^{-}(x_* +\rho_* x, t_* +\rho_*^2 t) $. Thus,
	\begin{align*}
%	\begin{equation*}
%		\begin{aligned}
			&
			\fint_{Q_r^{-}(x,t)}   \fint_{B_r(x) }
			\big| f(x_* + \rho_* y, t_* +\rho_*^2 s) -
			f(x_* + \rho_* z, t_* +\rho_*^2 s)  \big| dz dy ds
			\\
			= \ &
			\fint_{ Q_{r\rho_*}^{-}(x_* +\rho_* x, t_* +\rho_*^2 t) }  \fint_{ B_{r\rho_*}(x_* +\rho_* x) } \big| f( y_1, s_1) - f(z_1, s_1) \big| dz_1 dy_1 ds_1.
	%	\end{aligned}
%	\end{equation*}
	\end{align*}
	Then, for $t_*>\tau_0$, $|x_*|\le 1$,  $\rho_* \in (0, 1/100]$, by \eqref{qd240811-3}, we have
	\begin{equation*}
		|\omega|_{G_k(x_* + \rho_* y, t_* +\rho_*^2 s)}^{\mathsf x}(r,Q_2^{-}(0) ) \lesssim
		|k|^{-1} \|h_k\| v(t_*) \rho_*^{0.05} r^{0.05}.
	\end{equation*}
	Applying Proposition \ref{scaling-prop-0722} to \eqref{Phik-eq} with $\rho_* = 1/(100 |k|^{0.05})$, we get
	\begin{equation}\label{qd240811-7}
		|D^2 \Phi_k(x_*,t_*)| \lesssim C(\ell) |k|^{-1- (0.05)^2} v(t_*) R^{5-\ell}(t_*) \|h_k\| \mbox{ \ for \ }  t_*>\tau_0, \ |x_*|\le 1.
	\end{equation}
	
	Given $t_*>\tau_0$, $1 \le |x_*| \le 3R(t_*)/2$, $\rho_* \le |x_*|/100$, for any $(x,t) \in Q_2^{-}(0)$, $r\in (0,1)$, and $(y,s), (z,s) \in Q_r^{-}(x,t)$, we have $|x_*+ \rho_*  y| \sim |x_*+ \rho_*  z| \sim |x_*|$, $t_* + \rho_*^2 s\sim t_*$. By \eqref{qd240811-6}, \eqref{H-est}, then
	\begin{equation*}
		\begin{aligned}
			&
			\rho_*^{-2}
			\| \Phi_k(x_*+ \rho_*  z , t_* + \rho_*^2 s) \|_{L^{\infty}(Q_2^{-}(0) ) }
			\lesssim C(\ell) \rho_*^{-2} |k|^{-2} v(t_*) R^{5-\ell}(t_*) \langle x_* \rangle^{-1} \ln(|x_*|+2)  \| h_k \|,
			\\
			&
			\| G_k(x_*+ \rho_*  z , t_* + \rho_*^2 s) \|_{L^{\infty}(Q_2^{-}(0) ) }
			\lesssim C(\ell)
			|k|^{-2} v(t_*) \langle x_* \rangle^{2-\ell}   \| h_k \|.
		\end{aligned}
	\end{equation*}
	To estimate $ [ G_k(x_*+ \rho_*  z , t_* + \rho_*^2 s) ]_{\mathsf{|DMO|_x}(Q_2^{-}(0))}$, we still use the representative term $f_1$ in \eqref{qd240811-4} to show the general process of analysis. By \eqref{qd240811-9}, and \eqref{qd240811-2}, \eqref{qd240808-9}, we have
	\begin{equation*}
		\begin{aligned}
			&
			\big| f_1(x_*+ \rho_*  y , t_* + \rho_*^2 s) - f_1(x_*+ \rho_*  z , t_* + \rho_*^2 s) \big|
			\\
			= \ & \big| \big[ f_{11}(|x_*+ \rho_*  y|,t_* + \rho_*^2 s) -
			f_{11}(|x_*+ \rho_*  z|,t_* + \rho_*^2 s )
			\big]
			e^{ik\theta( x_*+ \rho_*  y )}
			\cos( \theta( x_*+ \rho_*  y ) )
			\\
			&
			+
			f_{11}(|x_*+ \rho_*  z|,t_* + \rho_*^2 s )
			\big[
			e^{ik\theta( x_*+ \rho_*  y )}
			\cos( \theta( x_*+ \rho_*  y ) )
			- e^{ik\theta(x_*+ \rho_*  z)}
			\cos( \theta(x_*+ \rho_*  z) ) \big] \big|
			\\
			\lesssim \ &
			|k|^{-1} \| h_k\| v(t_*) |x_*|^{1-\ell}
			\rho_* r.
		\end{aligned}
	\end{equation*}
	It follows that $[ f_1(x_*+ \rho_*  z , t_* + \rho_*^2 s) ]_{\mathsf{|DMO|_x}(Q_2^{-}(0))}
	\lesssim
	|k|^{-1} \| h_k\| v(t_*) |x_*|^{1-\ell}
	\rho_*$. The other terms in $G_k$ can be handled similarly. Then
	\begin{equation*}
		[ G_k(x_*+ \rho_*  z , t_* + \rho_*^2 s) ]_{\mathsf{|DMO|_x}(Q_2^{-}(0))}
		\lesssim
		|k|^{-1} \| h_k\| v(t_*) |x_*|^{1-\ell}
		\rho_*.
	\end{equation*}
	Applying Proposition \ref{scaling-prop-0722} to \eqref{Phik-eq} with $\rho_* = |x_*|/(100 |k|^{\frac{1}{3}})$, we have
	\begin{equation}\label{qd240811-8}
		| (D^2 \Phi_k)(x_*, t_*)|
		\lesssim C(\ell)
		|k|^{-\frac{4}{3}}
		v(t_*) R^{5-\ell}(t_*) \langle x_* \rangle^{-3} \ln(|x_*|+2)  \| h_k\| \mbox{ \ for \ } t_*>\tau_0, \ 1 \le |x_*| \le 3R(t_*)/2.
	\end{equation}

	Combining \eqref{qd240811-6}, \eqref{qd240811-7}, \eqref{qd240811-8}, and the interpolation inequality, we have
	\begin{equation}\label{Phik-est}
		\langle y\rangle^2 |D^2 \Phi_k| + \langle y\rangle |D \Phi_k| + 	|\Phi_k| \lesssim   C(\ell) |k|^{-1-(0.05)^2} v(\tau) R^{5-\ell} \langle y \rangle^{-1} \ln(|y|+2)  \| h_k\|  \mbox{ \ in \ } \mathcal{D}_{3R/2}.
	\end{equation}
	We take $\Psi_k = \left(a-bW\wedge \right) \left( L_{\rm{in}} \Phi_k \right)$, which is the exact strong solution we look for. \eqref{Phik-est} deduces \eqref{Phik-rough-est}. Recalling $\Phi_{ k} = \left(  \phi_{k} (\rho,\tau) e^{ik\theta}  \right)_{\mathbb{C}^{-1}}$ and applying Lemma \ref{complex-lem}, we have $\left(\Psi_k\right)_{\mathbb{C}} = e^{ik\theta} (a-ib)\mathcal{L}_k \phi_k (\rho,\tau)$.
\end{proof}

Since the $\Psi_k$ given in Lemma \ref{modek-rough} loses some power of $R$ when $|y|$ is small, we will construct $\Psi_k$ with a better estimate by another gluing procedure.

\begin{prop}\label{modek-prop}
	Consider
	\begin{equation*}
		\pp_\tau \Psi_k =\left(a-b W\wedge\right) \left( L_{\rm{in}}  \Psi_k \right)  + H_k
		\mbox{ \ in \ } \DD_{R},
		\quad
		\Psi_k(\cdot,\tau_0) =0 \mbox{ \ in \ } B_{R(\tau_0)},
	\end{equation*}
	where $H_k = \big( h_{k} (\rho,\tau) e^{ik\theta} \big)_{\mathbb{C}^{-1}}$, $\| H_k \|_{v,\ell}^{R} <\infty$. Suppose \eqref{nu-assump}, $ \ell \in (1,3)$, $\mathbf{P}_1[v(\tau)] >-1$, then there exists a solution $\Psi_k = \TT_k^{R}[H_k]$ as a mapping linear in $H_k$ with the estimate
	\begin{equation*}
		|\Psi_{k} |
		\lesssim |k|^{-1-(0.05)^2} v(\tau)
		\| H_k \|_{v,\ell}^{R}
		\begin{cases}
			\langle y \rangle^{2-\ell}
			& \mbox{ \ if \ } \ell \in (1,2) \cup (2,3)
			\\
			\ln(|y|+2) &  \mbox{ \ if \ }\ell = 2
		\end{cases}
		\mbox{ \ in \ } \DD_{R},
	\end{equation*}
	where ``$\lesssim$'' is independent of $k$. Moreover, $\Psi_k\cdot W=0$ and $e^{-ik\theta} \left( \Psi_k \right)_{\mathbb{C}} $ is radial in space.
	
\end{prop}

\begin{remark}
	The restriction $\mathbf{P}_1[v(\tau)] >-1$ is not optimal. To improve the lower bound of $\mathbf{P}_1[v(\tau)]$, we need to modify Lemma \ref{ellptic-compa} to catch the property of $\1_{ \{ |y|\le 4R \} }$.
	
\end{remark}

\begin{proof}
	Denote $\|h_k\| = \| h_k \|_{v,\ell}^{R} $ and
	take $h_k =0$ in $\DD_R^c$. In order to find a solution $\Psi_k$ with the form $\Psi_k = \big( \psi_k(\rho,\tau) e^{ik \theta}\big)_{\mathbb{C}^{-1}}$,
	by Lemma \ref{complex-lem}, it is equivalent to considering
	\begin{equation}\label{psik-eq}
		\pp_{\tau} \psi_k =	(a-ib)\mathcal{L}_k \psi_k + h_k \mbox{ \ in \ } \mathcal{D}_{R},
		\quad
		\psi_k(\cdot,\tau_0) =0 \mbox{ \ in \ } B_{R(\tau_0)}.
	\end{equation}
	Set $\psi_k = \eta_{R_0}(\rho) \psi_{i,k}(\rho,\tau) + \psi_{o,k}(\rho,\tau) $, where $\eta_{R_0}(\rho) = \eta(\rho/R_0)$ and $R_0$ is a large constant independent of $\tau_0$, $k$. To find a solution for \eqref{psik-eq}, it suffices to  consider the following inner-outer system
	\begin{equation}\label{phio-eq}
		\left\{
		\begin{aligned}
			& \pp_{\tau} \psi_{o,k}
			=
			(a-ib)
			\Big[
			\pp_{\rho\rho} \psi_{o,k}
			+
			\frac {\pp_{\rho} \psi_{o,k}}{\rho}
			-
			\frac{(k+1)^2 }{\rho^2} \psi_{o,k}
			\Big]
			+
			J[\psi_{o,k},\psi_{i,k}]
			\1_{\{ \rho \le 4R \} }
			\mbox{ \ in \ }
			\RR^2 \times (\tau_0,\infty),
			\\
			&
			\psi_{o,k}(\cdot,\tau_0) =0 \mbox{ \ in \ } \mathbb{R}^2,
		\end{aligned}
		\right.
	\end{equation}
	\begin{equation}\label{phii-eq}
		\pp_{\tau} \psi_{i,k}
		=
		(a-ib) \mathcal{L}_k \psi_{i,k}
		+
		K[\psi_{o,k}]
		\mbox{ \ in \ } \DD_{2R_0} ,
		\quad
		\psi_{i,k}(\cdot, \tau_0) =0 \mbox{ \ in \ } B_{2 R(\tau_0)},
	\end{equation}
	where we denote
	\begin{equation*}
		\begin{aligned}
			&	J[\psi_{o,k},\psi_{i,k}] =
			(a-ib)	\left(
			1-\eta_{R_0}
			\right) \tilde{V}_k(\rho)
			\psi_{o,k}
			+
			A_0[\psi_{i,k}] + (1-\eta_{R_0}) h_k,
			\quad
			\tilde{V}_k(\rho) = \frac{(4k+8)\rho^2 +4k}{(\rho^2+1)^2 \rho^2},
			\\
			&	K[\psi_{o,k}] =
			(a-ib) \tilde{V}_k(\rho)
			\psi_{o,k}
			+ h_k
			,
			\quad
			A_{0}[\psi_{i,k}] =
			(a-ib)
			\Big[
			\Big( \pp_{\rho\rho}\eta_{R_0}
			+ \frac 1{\rho}
			\pp_{\rho}\eta_{R_0}
			\Big) \psi_{i,k}
			+2 \pp_{\rho} \eta_{R_0} \pp_{\rho} \psi_{i,k}  \Big].
		\end{aligned}
	\end{equation*}
	Set $\Psi_{i,k}(y,\tau) = ( \psi_{i,k} e^{ik\theta} )_{\mathbb{C}^{-1} }$, that is, $\psi_{i,k} = e^{-ik\theta} ( \Psi_{i,k}\cdot E_1 + i \Psi_{i,k}\cdot E_2 )$. By Lemma \ref{complex-lem}, \eqref{phii-eq} is equivalent to
	\begin{equation}\label{Phii-eq}
		\pp_{\tau} \Psi_{i,k}
		=
		\left(a-bW\wedge\right) L_{\rm{in}} \Psi_{i,k}
		+
		\big(	K[\psi_{o,k}] e^{ik\theta} \big)_{\mathbb{C}^{-1} }
		\mbox{ \ in \ } \DD_{2R_0} ,
		\quad
		\Psi_{i,k}(\cdot, \tau_0) =0 \mbox{ \ in \ } B_{2 R(\tau_0)} .
	\end{equation}
	
	The linear theories of \eqref{phio-eq} and \eqref{Phii-eq} are given by Lemma \ref{convert dim} and Lemma \ref{modek-rough}, respectively. We reformulate \eqref{phio-eq} and \eqref{Phii-eq} into the following form
	\begin{equation}\label{H-z3}
		\begin{aligned}
			\psi_{o,k}(\rho,\tau) = \ &
			\rho^{|k+1|}
			\Big[
			\Gamma_{2|k+1|+2}^{\natural}**\Big(
			|z|^{-|k+1|}
			J[\psi_{o,k},\psi_{i,k}] \1_{\{ |z| \le 4R(s) \} }   \Big)
			\Big](\rho,\tau,\tau_0),
			\\
			\Psi_{i,k}(y,\tau) = \ &
			\TT_{kr}^{2R_0} \big[ \big(	K[\psi_{o,k}] e^{ik\theta} \big)_{\mathbb{C}^{-1} } \big].
		\end{aligned}
	\end{equation}
	We will solve $(\psi_{o,k}, \Psi_{i,k})$ for \eqref{H-z3} by the contraction mapping theorem.
	Since $\big|  \big( h_k e^{ik\theta}\big)_{\mathbb{C}^{-1} } \big| \le v(\tau) \langle y \rangle^{-\ell} \|h_k\|$, provided $ \ell \in (1,3)$, $\frac{3}{2} + \mathbf{P}_1[v(\tau) ]>0$, by Lemma \ref{modek-rough} and the scaling argument,
	\begin{equation*}
		\big|\TT_{kr}^{2R_{0}} \big[ \big( h_k e^{ik\theta}\big)_{\mathbb{C}^{-1} }\big]  \big| \le
		D_i   w_{i,k,1}(\rho,\tau) \| h_k \|,
		\quad
		\big| \nabla \TT_{kr}^{2R_{0}} \big[ \big( h_k e^{ik\theta}\big)_{\mathbb{C}^{-1} }\big]  \big| \le
		D_i w_{i,k,2}(\rho,\tau) \| h_k \|,
	\end{equation*}
	where $D_{i}\ge 1$ is a large constant  independent of $k$,
	\begin{equation*}
		w_{i,k,1}(\rho,\tau): = |k|^{-1-(0.05)^2} v(\tau) R_0^{5-\ell} \ln R_0 \langle \rho \rangle^{-3},
		\quad
		w_{i,k,2}(\rho,\tau): = v(\tau) R_0^{5-\ell} \ln R_0\big( \rho^{-1} \1_{\{ \rho \le 1\}} + \rho^{-4} \1_{\{ \rho >1 \}}\big).
	\end{equation*}
	Here the weight $\rho^{-1} \1_{\{ \rho \le 1\}}$ is due to the forthcoming estimate \eqref{qd24Aug12-1}. Denote
	\begin{align}
%	\begin{equation}
%		\begin{aligned}
			\B_{i,k } := \ &
			\Big\{
			F(y,\tau) \in C^{1}\big(B_{2R_0} \backslash \{0\} , \mathbb{R}^3 \big) \ \big| \ F(y,\tau) = \big( f(\rho,\tau) e^{ik \theta} \big)_{\mathbb{C}^{-1}}
			\mbox{ \ for some radial  scalar function}
			\nonumber
			\\
			&~ f(\rho,\tau) \mbox{ \ and \ }
			|F(y,\tau)| \le
			2D_i  w_{i,k}(\rho,\tau) \| h_k \|, \
			|\nabla F(y,\tau)| \le 2 D_i w_{i,k,2}(\rho,\tau) \| h_k \|
			\Big\}.
			\label{qd240817-1}
	%	\end{aligned}
	%\end{equation}
	\end{align}
	For any $\tilde{\Psi}_{i,k}\in \B_{i,k}$, denote $\tilde{\psi}_{i,k} = e^{-ik\theta} ( \tilde{\Psi}_{i,k} \cdot E_1 + i \tilde{\Psi}_{i,k} \cdot E_2 ) $. We will find a solution $\psi_{o,k} = \psi_{o,k}[\tilde{\psi}_{i,k}]$ of \eqref{phio-eq} by the contraction mapping theorem. Let us estimate $J[\psi_{o,k},\tilde{\psi}_{i,k}]$ term by term. By \eqref{Frenet-deri},
	\begin{equation}\label{Sep10-04}
		| \pp_{\rho} \tilde{\psi}_{i,k} | =
		\big| e^{-ik\theta} \big( \tilde{\Psi}_{i,k} \cdot \pp_{\rho} E_1
		+
		E_1 \cdot \pp_{\rho} \tilde{\Psi}_{i,k}
		+
		i \tilde{\Psi}_{i,k} \cdot \pp_{\rho} E_2
		+
		i  E_2
		\cdot \pp_{\rho} \tilde{\Psi}_{i,k}
		\big)
		\big|
		\lesssim
		| \tilde{\Psi}_{i,k} |
		\langle \rho \rangle^{-2}
		+
		|\pp_{\rho} \tilde{\Psi}_{i,k} |.
	\end{equation}
	Since $\tilde{\Psi}_{i,k}\in \B_{i,k}$, for $ \tilde{\ell}< \ell$, we have
	\begin{align}\notag
			&
			|A_{0}[ \tilde{\psi}_{i,k} ]| + |(1-\eta_{R_0}) h_k |
			\lesssim
			D_i \1_{\{ R_0 \le \rho \le 2R_0 \} } v(\tau) R_0^{-\ell} \ln R_0  \| h_k \| +
			\1_{\{ \rho\ge R_0 \}} v(\tau) \rho^{-\ell } \| h_k\|
			\\ \notag
			\lesssim \ &
			D_i \ln R_0 \| h_k\| v(\tau) \rho^{-\ell } \1_{\{ \rho\ge R_0 \}}
			\lesssim
			D_i  R_0^{(\tilde{\ell}-\ell)/2} \| h_k \|   v(\tau)  \rho^{-\tilde{\ell} } \1_{ \{ \rho\ge R_0 \} },
			\\ \label{qd24July13-1}
			&
			\rho^{-|k+1|}
			\big( \big|A_0[\tilde{\psi}_{i,k}] \big| + \big| (1-\eta_{R_0}) h_k \big| \big) \1_{\{ \rho \le 4R \} }
			\lesssim
			D_i R_0^{(\tilde{\ell}-\ell)/2} \| h_k \|
			v(\tau)
			\rho^{-|k+1| -\tilde{\ell} }
			\1_{\{ R_0 \le \rho \le 4R \} }.
		\end{align}
	
	For $k=-3, -2$, provided $\frac{3-\tilde{\ell}}{2} +1 +\mathbf{P}_1[v(\tau)] >0$, $1 < \tilde{\ell}< \min\{\ell,3\}$, by Lemma \ref{qd24July11-1-lem}, we have
	\begin{align}
		%\begin{equation*}
		%	\begin{aligned}
			&
			\rho^{|k+1|}
			\big| \Gamma_{2|k+1|+2}^{\natural} \big|**\Big\{
			|z |^{-|k+1|}
			\Big[ \big|A_0[\tilde{\psi}_{i,k}] \big| + \big| (1-\eta_{R_0}) h_k \big| \Big] \1_{\{ |z| \le 4R(s) \} } \Big\}
			\nonumber
			\\
			\lesssim \ &  D_i R_0^{(\tilde{\ell}-\ell)/2} \| h_k \| v(\tau) \big(
			\rho^{|k+1|}
			R_0^{2-|k+1|-\tilde{\ell}}\1_{\{\rho \le R_0\}}
			+
			\rho^{2-\tilde{\ell}} \1_{ \{ \rho >R_0 \} }
			\big)
			\nonumber
			\\
			\lesssim \ &  D_i R_0^{(\tilde{\ell}-\ell)/2} \| h_k \| v(\tau) \big(
			\rho \1_{\{\rho \le 1\}}
			+
			\rho^{2-\tilde{\ell}} \1_{ \{ \rho >1 \} }
			\big).
			\label{qd240819-1}
			%	\end{aligned}
		%\end{equation*}
	\end{align}
	The spatial decay rate near $\rho=0$ is restricted by the case $k=-2$.
	
	For $|k+1|$ large, since the direct pointwise estimate of $\big| \Gamma_{2|k+1|+2}^{\natural} \big|**[\cdot]$ will lead to an upper bound with a multiplicity of a constant with exponential growth in $|k+1|$ (with the form $a^{-|k+1|}$), which is a disaster for the convergence of summation. Instead, we search for another strategy. Set $\psi_{o,k}^{*}(y,\tau) = e^{i(k+1)\theta} \psi_{o,k}(\rho,\tau)$. Then
	\begin{equation*}
		\pp_{\tau} \psi_{o,k}^{*}
		=
		(a-ib)
		\Delta_{\mathbb{R}^2} \psi_{o,k}^{*}
		+
		e^{i(k+1)\theta}
		J[\psi_{o,k},\psi_{i,k}]
		\1_{\{ \rho \le 4R \} }
		\mbox{ \ in \ }
		\RR^2 \times (\tau_0,\infty),
		\quad
		\psi_{o,k}^{*}(\cdot,\tau_0) =0 \mbox{ \ in \ } \mathbb{R}^2,
	\end{equation*}
and $\psi_{o,k}^{*}$ is given by
	\begin{equation*}
		\psi_{o,k}^{*}(y,\tau) = \int_{\tau_0}^{\tau} \int_{\mathbb{R}^2} \Gamma_2^{\natural}(y-z,\tau-s)
		e^{i(k+1)\theta(z)}
		J[\psi_{o,k},\psi_{i,k}](|z|,s)
		\1_{\{ |z| \le 4R(s) \} } dz ds.
	\end{equation*}
	Similar to \eqref{H-Z6}, we will find $\tilde{P}_{k}(y,\tau)$ satisfying
	\begin{equation*}
		\Delta_{\mathbb{R}^2 } \tilde{P}_{k}(y,\tau) =  e^{i(k+1)\theta} J[\psi_{o,k},\psi_{i,k}](|y|,\tau)
		\1_{\{ |y| \le 4R(\tau) \} }
		\mbox{ \ in \ } \mathbb{R}^2.
	\end{equation*}
	Set $\tilde{P}_{k}(y,\tau) = e^{i(k+1)\theta} \rho^{|k+1|} \tilde{p}_{k,1}(\rho,\tau)$, where we take $\tilde{p}_{k,1}$ as
	\begin{equation*}
		\tilde{p}_{k,1}(\rho,\tau)
		=
		- \rho^{-2|k+1|}\int_{0}^\rho
		u^{2|k+1|-1} \int_{u}^\infty
		r r^{-|k+1|} J[\psi_{o,k},\psi_{i,k}](r,\tau)
		\1_{\{ r\le 4R(\tau) \} } dr du.
	\end{equation*}
	Then
	$\psi_{o,k}(\rho,\tau)$ in \eqref{H-z3} can be rewritten as $\psi_{o,k}(\rho,\tau) = e^{-i(k+1)\theta} \psi_{o,k}^{*}(y,\tau)$ with
	\begin{align*}
			&
			\psi_{o,k}^{*}(y,\tau) = \int_{\tau_0}^{\tau} \int_{\mathbb{R}^2} \Gamma_2^{\natural}(y-z,\tau-s) \Delta_{\mathbb{R}^2 } \tilde{P}_{k}(z,s) dz ds
			=
			\int_{\tau_0}^{\tau} \int_{\mathbb{R}^2} \Delta_{\mathbb{R}^2,z }\Gamma_2^{\natural}(y-z,\tau-s)  \tilde{P}_{k}(z,s) dz ds
			\\
			= \ &
			\int_{\tau_0}^{\tau}
			(a-ib)^{-1}	[4\pi (\tau-s)]^{-1}
			\int_{\mathbb{R}^2}
			e^{-\frac{|y-z|^2}{4(a-ib)(\tau-s)}}
			\frac{|y-z|^2 - 4(a-ib)(\tau-s)}{4(a-ib)^2(\tau-s)^2}
			e^{i(k+1)\theta(z) } |z|^{|k+1|}
			\tilde{p}_{k,1}(|z|,s) dz ds
			\\
			= \ &
			\int_{\tau_0}^{\tau}
			(a-ib)^{-1}	[4\pi (\tau-s)]^{-1}
			\int_{\mathbb{R}^2}
			\Big[
			e^{-\frac{|y-z|^2}{4(a-ib)(\tau-s)}}
			\frac{|y-z|^2 - 4(a-ib)(\tau-s)}{4(a-ib)^2(\tau-s)^2}
			\\
			&
			-
			e^{-\frac{|z|^2}{4(a-ib)(\tau-s)}}
			\frac{|z|^2 - 4(a-ib)(\tau-s)}{4(a-ib)^2(\tau-s)^2}
			\Big]
			e^{i(k+1)\theta(z) } |z|^{|k+1|}
			\tilde{p}_{k,1}(|z|,s) dz ds,
		\end{align*}
	where we used $\int_0^{2\pi} e^{i(k+1)\theta } d\theta =0$ for the last equality. By the last two equalities above, we have
	\begin{equation}\label{qd24Aug16-1}
		|\psi_{o,k} | \lesssim
		\min\big\{
		F_1[ |\tilde{p}_{k,1}(|y|,\tau) | ],~
		F_2[ |\tilde{p}_{k,1}(|y|,\tau) | ]
		\big\},
	\end{equation}
	where for any $f(y,\tau)$, we denote
	\begin{align*}
%	\begin{equation*}
%		\begin{aligned}
			&
			F_1[f] =
			F_1[f](y,\tau) :=
			\int_{\tau_0}^{\tau} (\tau-s)^{-2}
			\int_{\mathbb{R}^2}
			e^{-\frac{a |y-z|^2}{8 (\tau-s)}} |z|^{|k+1|}
			f(z,s) dz ds,
			\\
			&
			F_2[f] =
			F_2[f](y,\tau) := |y|
			\int_0^1
			\int_{\tau_0}^{\tau} (\tau-s)^{-5/2}
			\int_{\mathbb{R}^2}
			e^{-\frac{a |u y-z|^2}{8 (\tau-s)}} |z|^{|k+1|}
			f(z,s) dz ds du.
%		\end{aligned}
%	\end{equation*}
	\end{align*}
	By \eqref{qd24July13-1}, we have
	\begin{align*}
			& |\tilde{p}_{k,1} |
			\lesssim M_1 + M_2, \mbox{ \ where \ }
			M_1:= \rho^{-2|k+1|}\int_{0}^\rho
			u^{2|k+1|-1} \int_{u}^\infty
			r r^{-|k+1|} \1_{\{ R_0 \le r\le 4R(\tau) \} }  |\tilde{V}_k(r) \psi_{o,k}(r,\tau)|   dr du,
			\\
			& M_2:= \rho^{-2|k+1|}\int_{0}^\rho
			u^{2|k+1|-1} \int_{u}^\infty
			r \1_{\{ R_0 \le r\le 4R(\tau) \} }
			D_i R_0^{(\tilde{\ell}-\ell)/2} \| h_k \|    v(\tau)  r^{-|k+1|-\tilde{\ell} } dr du.
		\end{align*}
	For $k\le -4$ or $k\ge 2$, by Lemma \ref{ellptic-compa} (with $f(x) = |x|^{-|k+1|} \1_{\{ |x|\le 1\}} + |x|^{-|k+1|-\tilde{\ell}} \1_{\{ |x|> 1\}}$), and $\tilde{\ell} \in (1,3)$, we have
	\begin{equation*}
		\rho^{|k+1|}
		M_2
		\lesssim
		D_i R_0^{(\tilde{\ell}-\ell)/2}    \| h_k \|
		|k|^{-2}
		v(\tau)
		\big(
		\1_{\{\rho \le 1\}}
		\rho^2
		+
		\1_{\{\rho > 1\}}
		\rho^{2 -\tilde{\ell}}
		\big).
	\end{equation*}

	For $\tilde{\ell}\in (2,3)$, using $\1_{\{\rho \le 1\}}
	\rho^2
	+
	\1_{\{\rho > 1\}}
	\rho^{2 -\tilde{\ell}} = \rho^2 \big( \1_{\{\rho \le 1\}}
	+
	\1_{\{\rho > 1\}}
	\rho^{-\tilde{\ell}} \big) $, provided $1 + \mathbf{P}_1[v(\tau) ]>0$, by similar convolution estimates in $\mathbb{R}^4$ as Lemma \ref{qd24July11-1-lem} and \cite[Lemma A.2]{infi4d}, we have
	\begin{equation}\label{qd24Aug16-2}
		F_1[M_2] \lesssim
		D_i R_0^{(\tilde{\ell}-\ell)/2}    \| h_k \|
		|k|^{-2} v(\tau) \langle \rho \rangle^{2-\tilde{\ell}}.
	\end{equation}
	
	For $\tilde{\ell}\in (1,3)$, using $\1_{\{\rho \le 1\}}
	\rho^2
	+
	\1_{\{\rho > 1\}}
	\rho^{2 -\tilde{\ell}} = \rho^3 \big( \1_{\{\rho \le 1\}} \rho^{-1}
	+
	\1_{\{\rho > 1\}}
	\rho^{-1-\tilde{\ell}} \big) $, provided $1 + \mathbf{P}_1[v(\tau) ]>0$, by similar convolution estimates in $\mathbb{R}^5$ as Lemma \ref{qd24July11-1-lem} and \cite[Lemmas A.1, A.2]{infi4d}, we have
	\begin{equation}\label{qd24Aug16-3}
		\begin{aligned}
			F_2[M_2] \lesssim \ &
			D_i R_0^{(\tilde{\ell}-\ell)/2}    \| h_k \|
			|k|^{-2} v(\tau) |y|
			\int_0^1
			\big( \1_{\{u|y| \le 1\}} + |u y|^{1-\tilde{\ell}} \1_{\{u |y| > 1\}} \big) du
			\\
			\lesssim \ &
			D_i R_0^{(\tilde{\ell}-\ell)/2}    \| h_k \|
			|k|^{-2} v(\tau)
			\bigg(
			\1_{\{\rho \le 1\}} \rho
			+
			\1_{\{\rho > 1\}}
			\begin{cases}
				\rho^{2-\tilde{\ell}}, & \tilde{\ell} \in (1,2)
				\\
				\langle \ln \rho \rangle,  & \tilde{\ell} =2
				\\
				1, & \tilde{\ell} \in (2,3)
			\end{cases}
			\bigg).
		\end{aligned}
	\end{equation}
	
	By \eqref{qd240819-1}, \eqref{qd24Aug16-1}, \eqref{qd24Aug16-2} and \eqref{qd24Aug16-3}, we solve $\psi_{o,k}$ for $|k|\ge 2$ in the space
	\begin{equation*}%\label{qd24July09-7}
		\begin{aligned}
			&
			\B_{o,k} :=
			\big\{
			f(\rho,\tau) \ | \
			|f(\rho,\tau)| \le D_o D_i
			R_0^{(\tilde{\ell}-\ell)/2} \| h_k \|	w_{o,k}(\rho,\tau)
			\big\},
			\\
			\mbox{ \ where \ }	&
			w_{o,k}(\rho,\tau) := |k|^{-2}
			v(\tau)
			\Big(
			\1_{\{\rho \le 1\}} \rho
			+
			\1_{\{\rho > 1\}}
			\begin{cases}
				\rho^{2-\tilde{\ell}}, & \tilde{\ell} \in (1,2) \cup (2,3)
				\\
				\langle \ln \rho \rangle,  & \tilde{\ell} =2
			\end{cases}
			\Big)
		\end{aligned}
	\end{equation*}
	and $D_o$ is a sufficiently large constant.
	For any $\tilde{\psi}_{o,k} \in \B_{o,k}$,
	\begin{align}
		%	\begin{equation}
			%		\begin{aligned}
				&
				\big| \1_{\{R_0 \le  \rho \le 4R \} } \tilde{V}_k(\rho)
				\tilde{\psi}_{o,k} \big|
				\lesssim
				|k|^{-1} D_o D_i
				R_0^{(\tilde{\ell}-\ell)/2} \| h_k \| v(\tau) \rho^{-2-\tilde{\ell}} \langle \ln \rho \rangle
				\1_{\{R_0 \le  \rho \le 4R \} }
				\nonumber
				\\
				\lesssim \ &
				(R_{0}^{-2} \ln R_0) D_o D_i
				R_0^{(\tilde{\ell}-\ell)/2}
				\| h_k \| v(\tau) \rho^{-\tilde{\ell}}  \1_{\{R_0 \le  \rho \le 4R \} }.
				\label{qd24July13-2}
				%		\end{aligned}
			%	\end{equation}
	\end{align}
	Compared with \eqref{qd24July13-1}, since $R_{0}^{-2} \ln R_0$ is small, it follows that
	\begin{equation*}
		\rho^{|k+1|}
		\Gamma_{2|k+1|+2}^{\natural}**\big(
		|z|^{-|k+1|}
		J[\tilde{\psi}_{o,k},\tilde{\psi}_{i,k}]  \1_{\{ |z| \le 4R(s) \} } \big)  \in \B_{o,k}.
	\end{equation*}
	We can deduce the contraction mapping property in the same way.
	Now we have found a solution $\psi_{o,k} = \psi_{o,k}[\tilde{\psi}_{i,k}]\in \B_{o,k}$. Let us estimate the following term in $\DD_{2R_0}$.
	\begin{equation}\label{qd24Aug12-1}
		\big|  \tilde{V}_k(\rho)
		\psi_{o,k} \big|
		\lesssim
		|k|^{-1}
		D_o D_i
		R_0^{(\tilde{\ell}-\ell)/2}  \ln R_0 \| h_k \|
		v(\tau) \big(
		\1_{\{\rho \le 1\}} \rho^{-1}
		+
		\1_{ \{ \rho >1\} }
		\rho^{-2-\tilde{\ell} }
		\big)
		\mbox{ \ in \ }  \DD_{2R_0}.
	\end{equation}
	Provided $2+\tilde{\ell} > \ell$, $\mathbf{P}_1[v(\tau) ]> -3/2$, by Lemma \ref{modek-rough} and the scaling argument,
	\begin{equation*}
		\begin{aligned}
			&
			\Big|\TT_{kr}^{2R_0} \Big\{\Big[ e^{ik\theta} (a-ib) \tilde{V}_k(\rho)
			\psi_{o,k} \Big]_{\mathbb{C}^{-1} } \Big\}  \Big|
			\lesssim
			|k|^{-1}
			D_o D_i R_0^{(\tilde{\ell}-\ell)/2} \ln R_0
			\| h_k \| w_{i,k}(\rho,\tau),
			\\
			&
			\Big| \nabla \TT_{kr}^{2R_0} \Big\{\Big[ e^{ik\theta} (a-ib)\tilde{V}_k(\rho)
			\psi_{o,k} \Big]_{\mathbb{C}^{-1} } \Big\} \Big| \lesssim D_o D_i R_0^{(\tilde{\ell}-\ell)/2} \ln R_0 \| h_k \| w_{i,k,2}(\rho,\tau).
		\end{aligned}
	\end{equation*}
	In sum, we take $1<\tilde{\ell}<\min\{\ell,3\}$, $\ell\in (1,3)$. By the small quantity $R_0^{(\tilde{\ell}-\ell)/2} \ln R_0$, we have
	\begin{equation*}
		\TT_{kr}^{2R_0} \Big[ \Big(K[ \psi_{o,k}[\tilde{\psi}_{i,k}] ] e^{ik\theta}
		\Big)_{\mathbb{C}^{-1} } \Big] \in \B_{i,k}.
	\end{equation*}
	The contraction property can be deduced in the same way, and thus, we find a solution $\Psi_{i,k}=\Psi_{i,k}[h_k] \in \B_{i,k}$. Finally we find a solution $(\psi_{o,k},\Psi_{i,k})$ for \eqref{phio-eq} and \eqref{Phii-eq}.
	
	Substituting the right-hand side  $h_{k}$ by $c_{1} h_{k}^{(1)}$, $c_{2} h_{k}^{(2)}$, $c_{1} h_{k}^{(1)} + c_{2} h_{k}^{(2)}$ respectively, where $c_1,c_2$ are arbitrary constants and $h_{k}^{(1)}$, $h_{k}^{(2)}$ are in the same topology as $h_k$, then making subtraction and repeating process above,
	we deduce that $\Psi_{i,k}[h_k]$ and $\psi_{o,k}[h_k]$ are linear mappings about $h_k$. So does $\psi_k$.
	
	We will regard $D_o$, $D_i$, and $R_0$ as general constants hereafter. Recalling \eqref{qd24July13-1}, \eqref{qd24July13-2}, $2+\tilde{\ell}>\ell$, we have
	$ |J[\psi_{o,k},\psi_{i,k}]| \1_{\{ \rho \le 4R \} } \lesssim \| h_k\|
	v(\tau) \rho^{- \ell } \1_{\{R_0 \le  \rho \le 4R \} }  $.
	Similar to \eqref{qd240819-1}, \eqref{qd24Aug16-2}, \eqref{qd24Aug16-3}, for $\rho \le \tau^{\frac 12}$, the spatial decay of $\psi_{o,k}$ can be improved to
	\begin{equation}\label{qd240817-2}
		|\psi_{o,k}| \lesssim
		|k|^{-2} \|h_k\| v(\tau) \Big(
		\1_{\{\rho \le 1\}} \rho
		+
		\1_{\{\rho > 1\}}
		\begin{cases}
			\rho^{2-\ell}, & \ell \in (1,2) \cup (2,3)
			\\
			\langle \ln \rho \rangle,  & \ell =2
		\end{cases}
		\Big).
	\end{equation}
	By the upper bounds of $\psi_{o,k}$ in \eqref{qd240817-2} and $\Psi_{i,k}$ in \eqref{qd240817-1}, we get the upper bound of $\Psi_{k}$ in $\DD_{R}$.
\end{proof}

\subsection{Mode $0$}
\begin{prop}\label{m0-nonortho}
	Consider
	\begin{equation*}
		\begin{cases}
			\pp_\tau \Psi_0 =\left(a-b W\wedge\right) \left( L_{\rm{in}}  \Psi_0 \right)  + H_0
			\mbox{ \ in \ } \DD_{R},
			\\
			\Psi_0 = 0 	
			\mbox{ \ on \ } \pp\DD_{R},	
			\quad
			\Psi_0(\cdot,\tau_0) = 0
			\mbox{ \ in \ } B_{R(\tau_0)} ,
		\end{cases}
	\end{equation*}
	where $H_0 = \left( h_{0} (\rho,\tau) \right)_{\mathbb{C}^{-1}}$, $\| H_0 \|_{v,\ell}^{R}<\infty$. Suppose $\ln R \in \mathbf{AP}$, \eqref{nu-assump}, $\mathbf{P}_1[\tau^2 \theta_{R,\ell}
	v(\tau) ]>0$ with $\theta_{R,\ell}$ given in \eqref{theta-Rl-def}, then there exists a linear mapping $\Psi_0 = \TT_{00}^{R}[H_0]$ with the  estimate
	\begin{equation*}
		|\Psi_0 |
		\lesssim
		\|H_0 \|_{v,\ell}^{R} v(\tau)
		\big( \1_{\{|y|\le 1\}} |y| + \1_{\{|y| > 1\}} |y|^{-1}  \big)
		\begin{cases}
			R^2 \ln R
			&\mbox{ \ if \ } \ell>1
			\\
			R^2 (\ln R)^{\frac 32}
			& \mbox{ \ if \ } \ell=1
			\\
			R^{3-\ell} \ln R
			& \mbox{ \ if \ } \ell<1 .
		\end{cases}
	\end{equation*}
	Moreover, $\Psi_0\cdot W=0$ and $\left( \Psi_0 \right)_{\mathbb{C}}$ is radial in space.
	
\end{prop}

\begin{proof}
	Denote $\|h_0\| = \|h_0\|_{v,\ell}^{R}$.
	In order to find a solution with the form $\Psi_{0} = \left(  \psi_{0} (\rho,\tau) \right)_{\mathbb{C}^{-1}}$, by Lemma \ref{complex-lem}, it is equivalent to considering
	\begin{equation}\label{psi0-eq}
		\begin{cases}
			\pp_\tau \psi_0 =
			(a-ib)\mathcal{L}_0\psi_0 + h_0
			\mbox{ \ in \ } \DD_{R},
			\\
			\psi_0 = 0 	
			\mbox{ \ on \ } \pp\DD_{R},	
			\quad
			\psi_0(\cdot,\tau_0) = 0
			\mbox{ \ in \ } B_{R(\tau_0)}.
		\end{cases}
	\end{equation}
	By \eqref{nu-assump}, \eqref{gassump} is true. Then Lemma \ref{energy est} gives $
	\|\psi_0(\cdot, \tau) \|_{L^\infty(B_{R})}
	\lesssim
	R^2 \ln R \theta_{R,\ell}
	v(\tau) \|h_0 \| $. To improve the spatial decay, we  reformulate \eqref{psi0-eq} into the following form
	\begin{equation}\label{0Z-1}
		\begin{cases}
			\pp_{\tau} \psi_0 =	
			(a-ib)
			\big(
			\pp_{\rho\rho} \psi_0
			+
			\frac{1}{\rho} \pp_{\rho} \psi_0
			-
			\frac{1}{\rho^2}  \psi_0
			\big)
			+ \tilde{h}_0
			\mbox{ \ in \ } \mathcal{D}_{R},
			\\
			\psi_0 = 0
			\mbox{ \ on \ } \pp\mathcal{D}_{R},
			\quad
			\psi_0(\cdot,\tau_0) = 0
			\mbox{ \ in \ } B_{R(\tau_0)} ,
		\end{cases}
	\end{equation}
	where $\tilde{h}_0 :=
	(a-ib) \frac{8}{(\rho^2+1)^2}
	\psi_0 + h_0$.
	Set $\psi_{0}=\rho \psi_{*0}$. Then \eqref{0Z-1} is equivalent to
	\begin{equation}\label{0Z-2}
		\begin{cases}
			\pp_{\tau} \psi_{*0} =	
			(a-ib)
			\Delta_{\mathbb{R}^4 } \psi_{*0}
			+ |y|^{-1} \tilde{h}_0
			\mbox{ \ in \ } \mathcal{D}_{R},
			\\
			\psi_{*0} = 0
			\mbox{ \ on \ } \pp\mathcal{D}_{R},
			\quad
			\psi_{*0}(\cdot,\tau_0) = 0
			\mbox{ \ in \ } B_{R(\tau_0)},
		\end{cases}
	\end{equation}
	where we abuse the symbol $\mathcal{D}_{R}=\left\{(y,\tau) \ | \ y\in \mathbb{R}^4, \ |y|\le R(\tau) \right\}$ as the corresponding time-varying domain in $\mathbb{R}^4$ and similarly $\pp\mathcal{D}_{R}$, $B_{R(\tau_0)} $.
	By the same argument for deducing \eqref{H-Z5}, the fundamental solution for \eqref{0Z-2} is given by $\Gamma_4(x,y,t,s)$ with the bound
	\begin{equation*}
		| \Gamma_4(x,y,\tau,s)| \lesssim ( \tau-s )^{-2} e^{-\frac{\kappa|x-y|^2}{\tau-s}} \mbox{ \ for a constant \ } \kappa>0.
	\end{equation*}
	Provided $\mathbf{P}_1[\tau^2 \ln R \theta_{R,\ell}
	v(\tau) ]>0$, for $\rho \le R(\tau)$, we have
	\begin{align}
%	\begin{equation}
%		\begin{aligned}
			&
			|\psi_0| = \rho |\psi_{*0} |
			\lesssim
			\rho
			\big|\Gamma_{4}  **\big( |z|^{-1}
			|\tilde{h}_0 |
			\1_{\{ |z|\le R(s) \} } \big)  \big|
			\lesssim
			\|h_0\|
			\rho
			\int_{\tau_0}^{\tau} \int_{\mathbb{R}^4}
			(\tau-s )^{-2} e^{-\frac{\kappa|y-z|^2}{\tau-s}}
			\nonumber
			\\
			& \qquad\times
			\Big[
			(R^2 \ln R \theta_{R,\ell} v)(s)
			\Big( \1_{\{|z|\le 1\}} |z|^{-1} + \1_{\{ 1<|z| \le R(s) \}} |z|^{-5} \Big)
			+
			v(s) \1_{\{ 1<|z| \le R(s) \}} |z|^{-1-\ell}
			\Big] dz ds
			\nonumber
			\\
			\lesssim \ &
			R^2 \ln R \theta_{R,\ell}
			v(\tau)
			\big( \1_{\{\rho\le 1\}} \rho + \1_{\{\rho > 1\}} \rho^{-1}  \big) \|h_0\|,
			\label{typical-m0}
%		\end{aligned}
%	\end{equation}
	\end{align}
	where for the last ``$\lesssim$'', we used the following calculation.
	By Lemma \ref{qd24July10-2-lem},
	\begin{equation}\label{qd240818-1}
		\begin{aligned}
			&
			\int_{\tau_0}^{\tau} \int_{\mathbb{R}^4}
			(\tau-s )^{-2} e^{-\frac{\kappa|y-z|^2}{\tau-s}}
			(R^2 \ln R \theta_{R,\ell} v)(s) \1_{\{ 1<|z| \le R(s) \}} |z|^{-5} dz ds
			\\
			\lesssim \ &
			R^2 \ln R \theta_{R,\ell} v(\tau)
			\big( \1_{\{|y|\le 1\}} + \1_{\{ 1<|y|\le R \} } |y|^{-2} \big)
		\end{aligned}
	\end{equation}
	provided $\mathbf{P}_1[\tau^2 \ln R \theta_{R,\ell}
	v(\tau) ]>0$.
	By \cite[Lemma A.1]{infi4d},
	\begin{align*}
		%\begin{equation*}
		%\begin{aligned}
		&
		\int_{\tau_0}^{\tau} \int_{\mathbb{R}^4}
		(\tau-s )^{-2} e^{-\frac{\kappa|y-z|^2}{\tau-s}}
		(R^2 \ln R \theta_{R,\ell} v)(s) \1_{\{|z|\le 1\}} |z|^{-1} dz ds
		\\
		\lesssim \ &
		\tau^{-2} e^{-\frac{|y|^2}{16\tau}} \int_{\tau_0}^{\tau} (R^2 \ln R \theta_{R,\ell} v)(s) ds
		+
		R^2 \ln R \theta_{R,\ell} v(\tau)
		\big( \1_{ \{ |y|\le 1 \} } + \1_{\{ 1<|y|\le R \} } |y|^{-2} e^{-\frac{|y|^2}{16 \tau}} \big).
		%\end{aligned}
		%\end{equation*}
	\end{align*}
	To get the upper bound in \eqref{qd240818-1}, we require $
	\int_{\tau_0}^{\tau} (R^2 \ln R \theta_{R,\ell} v)(s) ds \lesssim
	\tau^{2}  \ln R \theta_{R,\ell} v(\tau) $.
	
	If $\mathbf{P}_1[R^2 \ln R \theta_{R,\ell} v(\tau)] >-1$, it holds since $ \tau R^2 \ln R \theta_{R,\ell} v(\tau) \lesssim  \tau^{2} \ln R \theta_{R,\ell} v(\tau) \Leftrightarrow R^2\lesssim \tau$, which is true. If $\mathbf{P}_1[R^2 \ln R \theta_{R,\ell} v(\tau)] = -1$, it suffices to make $ \tau (\ln \tau)^m  R^2 \ln R \theta_{R,\ell} v(\tau) \lesssim
	\tau^2  \ln R \theta_{R,\ell} v(\tau)
	\Leftrightarrow
	R^2 (\ln \tau)^m \lesssim \tau$
	for some large constant $m \ge 0$, which is true by \eqref{nu-assump}.
If $\mathbf{P}_1[R^2 \ln R \theta_{R,\ell} v(\tau)] < -1$, since $\mathbf{P}_1[\tau^2 \ln R \theta_{R,\ell}
	v(\tau) ]>0$, we only need to ensure $
	\tau_0 (R^2 \ln R \theta_{R,\ell} v)(\tau_0) \lesssim
	\tau_0^{2} (\ln R \theta_{R,\ell} v)(\tau_0)
	\Leftrightarrow
	R^2(\tau_0) \lesssim
	\tau_0$, which is true.
	
	The estimate including $\1_{\{ 1<|z| \le R(s) \}} |z|^{-1-\ell}$ is deduced by \cite[Lemma A.1]{infi4d}, due to the property of $\theta_{R,\ell}$, we only need to consider $\ell <3$. We omit the details.
\end{proof}

In contrast to \eqref{scalr-Z}$_4$ for mode $k$, $|k|\ge 2$, the elliptic operator in mode $0$
admits a {\it bounded} kernel function with {\it decay}, and as a consequence, the decay information of the right-hand side might get lost when deriving estimates. In fact, decay of the solution can be recovered if an orthogonality condition is imposed. The linear theory of mode $0$ with the orthogonality condition is given below.

\begin{lemma}\label{m0-rough}
	Consider
	\begin{equation*}
		\pp_\tau \Psi_0 =
		\left(a-b W\wedge\right) \left( L_{\rm{in}} \Psi_0 \right)  + H_0
		\mbox{ \ in \ } \DD_{R}, \quad \Psi_0(\cdot, \tau_0) =0 \mbox{ \ in \ } B_{R(\tau_0)},
	\end{equation*}
	where $H_0$ is defined in $\DD_{R_*}$ with $ R \le R_* \le \infty$, $H_0 = \left( h_{0} (\rho,\tau) \right)_{\mathbb{C}^{-1}}$, $\| H_0 \|_{v,\ell}^{R_*} <\infty$ with $\ell \in (1,3)$ and the orthogonality condition
	\begin{equation}\label{m0-orth-cond}
		\int_{0}^{R_*(\tau)} h_0(r,\tau) \mathcal{Z}_{0,1}(r) r dr =0
		\mbox{ \ for \ } \tau\in(\tau_0,\infty)
	\end{equation}
	holds. Suppose $\ln R \in \mathbf{AP}$, \eqref{nu-assump}, $\mathbf{P}_1[\tau^2 R^{3-\ell} v(\tau)]>0$, then there exists a solution $\Psi_0 = \TT_{0r}^{R}[H_0]$ as a linear mapping in $H_0$ with the estimate
	\begin{equation}\label{Phi0-rough-est}
		\langle y\rangle
		|\nabla \Psi_0 |
		+
		|\Psi_0|
		\lesssim v(\tau) R^{5-\ell} \ln R
		\langle y \rangle^{-3} \| H_0 \|_{v,\ell}^{R_*} \mbox{ \ in \ } \DD_{R}.
	\end{equation}
	Moreover, $\Psi_0\cdot W=0$ and $\left(\Psi_0\right)_{\mathbb{C}}$ is radial in space.
	
\end{lemma}

\begin{proof}
	
	Denote $\| h_0 \| = \| h_0 \|_{v,\ell}^{R_*}$ and assume $h_0=0$ in $\DD_{R_*}^{c}$. We consider
	\begin{equation*}
		\left(a-b W\wedge\right) \left( L_{\rm{in}} G_0 \right) = H_0 \mbox{ \ where \ } G_0 = \left(g_0(\rho,\tau)  \right)_{\mathbb{C}^{-1}} .
	\end{equation*}
	By Lemma \ref{complex-lem}, it is equivalent to $ \left(a-ib\right) \mathcal{L}_0 g_0 = h_0 $,
	where $g_0$ is given by
	\begin{equation*}
		g_0(\rho,\tau) =
		\left(a+ib\right)
		\Big(
		\mathcal{Z}_{0,2}(\rho) \int_{0}^{\rho} h_0(r,\tau) \mathcal{Z}_{0,1}(r) r dr
		-
		\mathcal{Z}_{0,1}(\rho) \int_{0}^{\rho} h_0(r,\tau) \mathcal{Z}_{0,2}(r) r dr
		\Big).
	\end{equation*}
	It follows that
	\begin{equation*}
		\partial_{\rho} g_0(\rho,\tau) =
		\left(a+ib\right)
		\Big(
		\mathcal{Z}'_{0,2}(\rho) \int_{0}^{\rho} h_0(r,\tau) \mathcal{Z}_{0,1}(r) r dr
		-
		\mathcal{Z}'_{0,1}(\rho) \int_{0}^{\rho} h_0(r,\tau) \mathcal{Z}_{0,2}(r) r dr
		\Big).
	\end{equation*}

	By \eqref{scalr-Z}, the orthogonality condition \eqref{m0-orth-cond}, if $1<\ell<3$, for all $\rho \in (0,\infty)$, we have
	\begin{equation}\label{G0-est}
		\rho |\partial_{\rho} g_0 | +
		|g_0| \lesssim
		\| h_0\| v(\tau) \big( \1_{\{ \rho\le 1\}} \rho^2 + \1_{\{ \rho > 1\}} \rho^{2-\ell}  \big),
		\quad
		\| G_0 \|_{v,\ell-2}^{\infty} = \| g_0 \|_{v,\ell-2}^{\infty} \lesssim \| h_0 \|.
	\end{equation}
	Similar to the computations on the $\mathsf{DMO_x}$ semi-norm of $G_k$ in Lemma \ref{modek-rough} (but simpler since we do not need its dependence on $k$ now), we have
	\begin{equation*}
		[ G_0]_{\mathsf{|DMO|_x}(Q_2^{-}(0,\tau))} \lesssim \|h_0\| v(\tau),
		\quad
		[ G_0(x_*+ \rho_*  z , t_* + \rho_*^2 s) ]_{\mathsf{|DMO|_x}(Q_2^{-}(0))}
		\lesssim \| h_0\| v(t_*) |x_*|^{1-\ell}
		\rho_*
	\end{equation*}
	for $t_*>\tau_0$, $1 \le |x_*| \le 3R(t_*)/2$, $\rho_* = |x_*|/100$.
	Next, let us consider
	\begin{equation*}
		\begin{cases}
			\pp_{\tau} \Phi_0 =
			\left(a-bW\wedge \right) \left( L_{\rm{in}} \Phi_0 \right) + G_0
			\mbox{ \ in \ } \DD_{2R},
			\\
			\Phi_0 = 0
			\mbox{ \ on \ } \pp\DD_{2R}, \quad
			\Phi_0(\cdot,\tau_0) = 0
			\mbox{ \ in \ } B_{2R(\tau_0)} .
		\end{cases}
	\end{equation*}
	Provided $\ln R \in \mathbf{AP}$, $\mathbf{P}_1[\tau^2 R^{3-\ell} v(\tau)]>0$, by \eqref{G0-est} and Proposition \ref{m0-nonortho}, there exists a solution $\Phi_0 = \Phi_0[ G_0 ]$ with the form $\Phi_{ 0} = \left(  \phi_{0} (\rho,\tau)   \right)_{\mathbb{C}^{-1}}$ for some scalar function $\phi_{0}$ and the estimate
	\begin{equation*}
		|\Phi_0( y ,\tau)|
		\lesssim
		v(\tau) R^{5-\ell} \ln R
		\langle y \rangle^{-1} 	 \| G_0 \|_{v,\ell-2}^{\infty}.
	\end{equation*}
	By Proposition \ref{qd240725-6-prop} (for $|y|\le 1$), Proposition \ref{scaling-prop-0722} (for $|y|> 1$), and the interpolation inequality, then
	\begin{equation}\label{Phi0-est}
		\langle y\rangle^2 |D^2 \Phi_0| + \langle y\rangle |D \Phi_0| + 	|\Phi_0| \lesssim  v(\tau) R^{5-\ell} \ln R \langle y \rangle^{-1} \| h_0 \|  \mbox{ \ in \ } \mathcal{D}_{3R/2}.
	\end{equation}
	We take the desired strong solution as $\Psi_0 = \left(a-bW\wedge \right) \left( L_{\rm{in}} \Phi_0 \right)$. Combining  \eqref{Phi0-est}  and the scaling argument, we conclude \eqref{Phi0-rough-est}. Applying $\Phi_{0} = ( \phi_{0} (\rho,\tau) )_{\mathbb{C}^{-1}}$ and Lemma \ref{complex-lem}, we have $\left(\Psi_0\right)_{\mathbb{C}} =  (a-ib)\mathcal{L}_0 \phi_0$.
\end{proof}

\begin{prop}\label{Re-m0-prop}
	Consider
	\begin{equation*}
		\pp_\tau \Psi_0 =
		\left(a-b W\wedge \right) \left(L_{\rm{in}}\Psi_0 \right) + H_0 +
		\left(
		c_0(\tau) \eta(\rho) \mathcal{Z}_{0,1}(\rho)  \right)_{\mathbb{C}^{-1}}
		\mbox{ \ in \ } \DD_{R},
		\quad
		\Psi_0(\cdot, \tau_0) =0  \mbox{ \ in \ } B_{R(\tau_0)},
	\end{equation*}
	where $H_0$ is defined in $\DD_{R_*}$ with $ R \le R_* \le \infty$,
	$H_0 = \left( h_{0} (\rho,\tau) \right)_{\mathbb{C}^{-1}}$, $\| H_0 \|_{v,\ell}^{R_*} <\infty$. Suppose \eqref{nu-assump}, $\ln R_0 \in \mathbf{AP}$,
	$\ell\in (1, 3) $,
	$\mathbf{P}_1[v(\tau)] > (\ell-5)/2$,
	then there exists a solution $(\Psi_0,c_{0} ) = ( \TT_0^{R}[H_0], c_{0}[H_0] )$ as a linear mapping in $H_0$ with the estimates
	\begin{equation*}
		\langle y\rangle |\nabla \Psi_{0}| + 	|\Psi_{0}|
		\lesssim \ln R_0 v(\tau)
		\left(
		R_0^{5-\ell }  \langle y \rangle^{-3}
		\1_{\{ |y|\le 2R_0\}}
		+
		\langle y \rangle^{2-\ell} \1_{\{ |y|> 2R_0\}} \right) \| H_0 \|_{v,\ell}^{R_*},
	\end{equation*}
	\begin{equation*}
		c_{0}(\tau)
		=
		-\Big(\int_{0}^2 \eta(r) \mathcal{Z}_{0,1}^2(r) r dr\Big)^{-1}
		\Big( \int_{0}^{R_1}
		h_0(r,\tau)  \mathcal{Z}_{0,1}(r) r dr
		+
		c_{*0}[H_0](\tau) \Big),
	\end{equation*}
	where $2R_0 \le R_1 \le R_*$, $c_{*0}[H_0] $ is a scalar function linearly depending on $H_0$ and satisfies
	$|c_{*0}[H_0](\tau) | \lesssim  R_0^{1 -\ell}\ln R_0  v(\tau) \| H_0 \|_{v,\ell}^{R_*}$.
	Moreover, $\Psi_0\cdot W=0$ and $ \left(\Psi_0\right)_{\mathbb{C}}$ is radial in space.
\end{prop}

\begin{proof}
	Denote $\|h_0\| = \| h_0 \|_{v,\ell}^{R_*}$ and
	set $h_0=0$ in $\DD_{R_*}^c$. By Lemma \ref{complex-lem}, to find a solution $\Psi_0$ with the form $\Psi_0 = \left( \psi_{0} (\rho,\tau) \right)_{\mathbb{C}^{-1}}$, it is equivalent to
	\begin{equation*}
		\pp_\tau \psi_0 =
		(a-ib)\mathcal{L}_0\psi_0 + h_0 +
		c_0(\tau) \eta(\rho) \mathcal{Z}_{0,1}(\rho)
		\mbox{ \ in \ } \DD_{R},
		\quad
		\psi_0(\cdot, \tau_0) =0  \mbox{ \ in \ } B_{R(\tau_0)}.
	\end{equation*}
	Set $\psi_0 = \eta_{R_0}(\rho) \psi_{i,0}(\rho,\tau) + \psi_{o,0}(\rho,\tau) $, where $\eta_{R_0}(\rho) = \eta(\rho / R_0)$. In order to find a solution $\psi_0$, it suffices to  consider the following inner-outer system
	\begin{equation}\label{phio-0-eq}
		\left\{
		\begin{aligned}
			& \pp_{\tau} \psi_{o,0}
			=
			(a-ib)\Big(	\pp_{\rho\rho} \psi_{o,0}
			+
			\frac 1{\rho}
			\pp_{\rho} \psi_{o,0}
			-
			\frac{1}{\rho^2} \psi_{o,0}
			\Big)
			+
			J[\psi_{o,0},\psi_{i,0}]
			\1_{\{ \rho \le 4R \} }
			\mbox{ \ in \ }
			\RR^2 \times (\tau_0,\infty) ,
			\\
			&
			\psi_{o,0}(\cdot,\tau_0) =0  \mbox{ \ in \ }
			\RR^2 ,
		\end{aligned}
		\right.
	\end{equation}
	\begin{equation}\label{phii-0-eq}
		\pp_{\tau} \psi_{i,0}
		=
		(a-ib)	\mathcal{L}_0 \psi_{i,0}
		+
		K[\psi_{o,0}] + c_0(\tau) \eta(\rho) \mathcal{Z}_{0,1}(\rho)
		\mbox{ \ in \ } \DD_{2R_0} ,
		\quad 	
		\psi_{i,0}(\cdot,\tau_0) =0	
		\mbox{ \ in \ } B_{2R_0(\tau_0 )},
	\end{equation}
	where
	\begin{equation*}
		\begin{aligned}
			&	J[\psi_{o,0},\psi_{i,0}]=
			(a-ib) (1-\eta_{R_0})
			\tilde{V}_0(\rho) \psi_{o,0}
			+
			A_0[\psi_{i,0}] + (1-\eta_{R_0}) h_0,
			\quad
			\tilde{V}_0(\rho) = \frac{8}{(\rho^2+1)^2},
			\\
			&
			K[\psi_{o,0}] =
			(a-ib) \tilde{V}_0(\rho)
			\psi_{o,0}
			+ h_0 ,
			\quad	A_{0}[\psi_{i,0}] =
			(a-ib)
			\Big[
			\Big( \pp_{\rho\rho}\eta_{R_0}
			+ \frac{\pp_{\rho}\eta_{R_0} }{\rho}
			\Big) \psi_{i,0}
			+2 \pp_{\rho} \eta_{R_0} \pp_{\rho} \psi_{i,0}
			\Big] - \psi_{i,0}  \pp_{\tau} \eta_{R_0}.
		\end{aligned}
	\end{equation*}
	Set $\Psi_{i,0}(y,\tau) = (\psi_{i,0} )_{\mathbb{C}^{-1}}$,  that is, $\psi_{i,0}= \Psi_{i,0} \cdot E_1 + i \Psi_{i,0} \cdot E_2$. By Lemma \ref{complex-lem}, \eqref{phii-0-eq} is equivalent to
	\begin{equation}\label{Phii-0-eq}
		\pp_{\tau} \Psi_{i,0}
		=
		\left(a-bW\wedge\right)	L_{\rm{in}} \Psi_{i,0}
		+
		\left(	K[\psi_{o,0}] + c_0(\tau) \eta(\rho) \mathcal{Z}_{0,1}(\rho)
		\right)_{\mathbb{C}^{-1} }
		\mbox{ \ in \ } \DD_{2R_0} ,
		\quad 	
		\Psi_{i,0}(\cdot,\tau_0) =0	
		\mbox{ \ in \ } B_{2R_0(\tau_0 )}.
	\end{equation}
	
	 For $2R_0 \le R_1 \le R_*$, in order for the orthogonality condition \eqref{m0-orth-cond} to hold in \eqref{Phii-0-eq}, we take
	\begin{equation*}
		c_{0}(\tau) =  c_{0}[\psi_{o,0}](\tau) :=
		C_{0,1} \int_{0}^{R_1}
		\Big[ (a-ib) \tilde{V}_0(r) \psi_{o,0}(r,\tau)
		+ h_0(r,\tau) \Big] \mathcal{Z}_{0,1}(r) r dr,
		\
		 C_{0,1} :=  - \Big(\int_{0}^2 \eta(r) \mathcal{Z}_{0,1}^2(r) r dr \Big)^{-1}.
	\end{equation*}

	By Lemma \ref{convert dim} and Lemma \ref{m0-rough}, we reformulate \eqref{phio-0-eq} and \eqref{Phii-0-eq} into the following form formally
	\begin{equation*}%\label{z6}
		\begin{aligned}
			\psi_{o,0}(\rho,\tau) = \ &
			\rho
			\big[
			\Gamma_{4}^{\natural}**\big(
			|z|^{-1}
			J[\psi_{o,0},\psi_{i,0}] \1_{\{ |z| \le 4R(s) \} } \big)  \big] (\rho,\tau,\tau_0)   ,
			\\
			\Psi_{i,0}(y,\tau) = \ &
			\TT_{0r}^{2R_0} \left[ \left(K[\psi_{o,0} ] + c_0[\psi_{o,0}](\tau) \eta(\rho) \mathcal{Z}_{0,1}(\rho)
			\right)_{\mathbb{C}^{-1}} \right].
		\end{aligned}
	\end{equation*}
	We will solve this system using the contraction mapping theorem.
	
	Denote $H_{I}:= \big[h_0+ C_{0,1}
	\big( \int_{0}^{R_1}
	h_0(r,\tau) \mathcal{Z}_{0,1}(r) r dr
	\big) \eta(\rho) \mathcal{Z}_{0,1}(\rho) \big]_{\mathbb{C}^{-1} } $. It is easy to have $\| H_{I} \|_{v,\ell}^{R_1} \lesssim \|h_0\|$. In view of Lemma \ref{m0-rough}, if $(H_I)_{\mathbb{C}}$ satisfies the orthogonality condition \eqref{m0-orth-cond} (with $R_*=R_1$) and $\ln R_0 \in \mathbf{AP}$, $2 + \mathbf{P}_1[v(\tau)] + (3-\ell) \mathbf{P}_1[R_0]>0$, we have the estimate
	\begin{equation}\label{qd24July09-4}
		\langle y \rangle
		|\nabla \TT_{0r}^{2R_{0}} [H_{I}] (y,\tau) |
		+	|\TT_{0r}^{2R_{0}} [H_I] (y,\tau) | \le
		D_i   w_{i,0}(\rho,\tau) \| h_0 \|
		,
	\end{equation}
	where  $D_{i}\ge 1$ is a constant and
	$
	w_{i,0}(\rho,\tau): =
	v(\tau)
	R_0^{5-\ell } \ln R_0 \langle \rho \rangle^{-3} $.
	Denote
	\begin{equation*}
		\begin{aligned}
			\B_{i,0} := \ &
			\Big\{
			F(y,\tau) \in C^{1}\left(B_{2R_0} , \mathbb{R}^3 \right) \ | \ F(y,\tau) = \left( f(\rho,\tau) \right)_{\mathbb{C}^{-1}}
			\mbox{ \ for some radial scalar function}
			\\
			& f(\rho,\tau) \mbox{ \ and \ }
			\langle y \rangle |\nabla F(y,\tau)|  + |F(y,\tau)| \le
			2D_i  w_{i,0}(\rho,\tau) \| h_0 \|
			\Big\} .
		\end{aligned}
	\end{equation*}
	For any $\tilde{\Psi}_{i,0}\in \B_{i,0}$, denote $\tilde{\psi}_{i,0} =  \tilde{\Psi}_{i,0}\cdot E_1 + i \tilde{\Psi}_{i,0}\cdot E_2 $.
	We will find a solution $\psi_{o,0} = \psi_{o,0}[\tilde{\psi}_{i,0}]$ of \eqref{phio-0-eq} by the contraction mapping theorem. Let us estimate $J[\psi_{o,0},\tilde{\psi}_{i,0}]$ term by term.
	By \eqref{Frenet-deri},
	\begin{align*}
			&
			\big| \pp_{\rho} \tilde{\psi}_{i,0} \big|=
			\big|  \tilde{\Psi}_{i,0}\cdot \pp_{\rho} E_1
			+
			E_1 \cdot \pp_{\rho} \tilde{\Psi}_{i,0}
			+ i \tilde{\Psi}_{i,0}\cdot \pp_{\rho} E_2
			+ i  E_2\cdot \pp_{\rho}  \tilde{\Psi}_{i,0}
			\big|
			\\
			\lesssim \ &
			|\tilde{\Psi}_{i,0}| \langle \rho\rangle^{-2}
			+
			| \nabla \tilde{\Psi}_{i,0}  |
			\lesssim
			D_i v(\tau)
			R_0^{5-\ell } \ln R_0 \langle \rho \rangle^{-4} \| h_0 \|.
		\end{align*}
	Since $|R_0'| =O( R_0^{-1} )$ in \eqref{nu-assump}, we have
	\begin{equation}\label{qd24July12-1}
		\begin{aligned}
			|A_{0}[ \tilde{\psi}_{i,0} ]|
			+
			\left|(1-\eta_{R_0}) h_0  \right|
			\lesssim \ &
			D_i \1_{\{ R_0 \le \rho \le 2R_0 \} }  v(\tau) R_0^{-\ell} \ln R_0  \| h_0 \| +
			\1_{\{ \rho\ge R_0 \}} v(\tau)  \rho^{-\ell } \| h_0\|
			\\
			\lesssim  \ &
			D_i \1_{\{ \rho \ge R_0 \} }  v(\tau) \ln R_0 \rho^{-\ell} \| h_0 \|.
		\end{aligned}
	\end{equation}
	Provided \eqref{nu-assump}, $\ln R_0 \in \mathbf{AP}$, $1<\ell <3$, and
	$\mathbf{P}_1[v(\tau)] > (\ell-5)/2 $, by Lemma \ref{qd24July11-1-lem}, we have
	\begin{equation*}%\label{qd24July10-3}
		\Big|\rho
		\Gamma_{4}^{\natural}**\Big\{
		|z|^{-1}
		\Big[A_0[\tilde{\psi}_{i,0}] + (1-\eta_{R_0}) h_0\Big] \1_{\{ |z| \le 4R(s) \} }  \Big\}
		\Big|
		\le
		D_{o} D_i w_{o,0}(\rho,\tau) \| h_0 \|,
	\end{equation*}
	where $D_o\ge 1$ is a  large constant,
	$w_{o,0}(\rho,\tau) :=
	v(\tau) \ln R_0
	\big( \rho R_0^{1-\ell} \1_{\{ \rho \le R_0 \}} + \rho^{2-\ell} \1_{\{ \rho > R_0 \}}\big) $.
	Denote
	\begin{equation}\label{qd24July12-3}
		\B_{o,0} :=
		\{
		f(\rho,\tau) \ | \
		|f(\rho,\tau)| \le
		2 D_o D_i w_{o,0}(\rho,\tau)  \| h_0 \|
		\mbox{ \ for \ } \rho\ge 0, \tau\ge  \tau_0 \}.
	\end{equation}
	For any $\tilde{\psi}_{o,0} \in \B_{o,0}$, we estimate
	\begin{equation*}%\label{qd24July12-2}
		\begin{aligned}
			&
			|
			(1-\eta_{R_0}) \tilde{V}_0(\rho)
			\tilde{\psi}_{o,0}  \1_{\{ \rho \le 4R \} } |
			\lesssim
			D_o D_i   v(\tau) \ln R_0 \rho^{-2-\ell }
			\1_{\{R_0 <  \rho \le 4R \} } \| h_0 \|
			\\
			\lesssim \ &
			D_o D_i v(\tau)  \ln R_0  \rho^{-\ell }
			\1_{\{R_0 < \rho \le 4R \} }
			\| h_0 \| \big(\inf_{s\ge \tau_0 } R_{0}(s) \big)^{-2}.
		\end{aligned}
	\end{equation*}
	Compared with \eqref{qd24July12-1}, due to the small quantity $\big(\inf_{s\ge \tau_0 } R_{0}(s) \big)^{-2}$ by $\inf_{s\ge \tau_0} R_0(s) \gg 1$ in \eqref{nu-assump}, we have
	\begin{equation*}
		\rho
		\Gamma_{4}^{\natural}**\left(
		|z|^{-1}
		J[\tilde{\psi}_{o,0},\tilde{\psi}_{i,0}] \1_{\{ |z| \le 4R(s) \} }  \right)  \in \B_{o,0} .
	\end{equation*}
	We can deduce the contraction mapping property in the same way.
	
	Now we have found a solution $\psi_{o,0} = \psi_{o,0}[\tilde{\psi}_{i,0}]\in \B_{o,0}$. It follows that for $\ell \in (1,3)$,
	\begin{equation*}
		\Big\|  \big| \tilde{V}_0(\rho) \psi_{o,0}[\tilde{\psi}_{i,0}] \big|
		+
		C_{0,1} \Big|\int_{0}^{R_1}
		\tilde{V}_0(r) \psi_{o,0}[\tilde{\psi}_{i,0}](r,\tau)  \mathcal{Z}_{0,1}(r) dr \Big|
		\eta(\rho) \mathcal{Z}_{0,1}(\rho) \Big\|_{v,\ell}^{2 R_0}
		\lesssim
		D_o D_i
		\sup_{s\ge \tau_0}
		\big( R_0^{1-\ell} \ln R_0 \big)(s) \| h\|.
	\end{equation*}
	Due to the choice of $c_0(\tau)$, $h_{II}:=K[\psi_{o,0}[\tilde{\psi}_{i,0}]] + c_0[\psi_{o,0}[\tilde{\psi}_{i,0}]](\tau) \eta(\rho) \mathcal{Z}_{0,1}(\rho)$ satisfies the orthogonality condition \eqref{m0-orth-cond} with $R_*=R_1$. Since $\sup_{s\ge \tau_0}
	\big( R_0^{1-\ell} \ln R_0 \big)(s)$ provides small quantity by \eqref{nu-assump},
	similar to \eqref{qd24July09-4}, we have
	\begin{equation*}
		\TT_{0r}^{2R_0}[ (h_{II})_{\mathbb{C}^{-1} }] \in \B_{i,0}.
	\end{equation*}

	The contraction property can be deduced in the same way. Therefore, we find a solution $\Psi_{i,0}=\Psi_{i,0}[h_0] \in \B_{i,0}$. Finally we find a solution $(\psi_{o,0},\Psi_{i,0})$ for \eqref{phio-0-eq} and \eqref{Phii-0-eq}.
	From the construction process,  $\psi_{i,0}[h_0]$, $\psi_{o,0}[h_0]$ and $c_0[h_0]$ are linear mappings in $h_0$. So does $\psi_0$.
	
	We will regard $D_o$, $D_i$  as general constants hereafter. Since $\psi_{o,0}[h_0] \in \B_{o,0}$, then
	\begin{equation*}
		c_{0}[h_0](\tau)
		=
		C_{0,1} \Big(\int_{0}^{R_1}
		h_0(r,\tau)  \mathcal{Z}_{0,1}(r) r dr
		+
		c_{*0}[h_0](\tau) \Big),
	\end{equation*}
	where $c_{*0}[h_0](\tau)$ is a linear mapping in $h_0$, and $|c_{*0}[h_0](\tau)|\lesssim R_0^{ 1-\ell}\ln R_0  v(\tau) \|h_0\|$.

	Combining the upper bound of $\psi_{o,0}$ and $\Psi_{i,0}$, we have
	\begin{equation*}
		| \Psi_{0}(y,\tau)|
		\lesssim \ln R_0 v(\tau)
		\big(
		R_0^{5-\ell }  \langle y \rangle^{-3}
		\1_{\{|y| \le 2R_0\}}
		+
		\langle y \rangle^{2-\ell} \1_{\{ |y|> 2R_0\}}\big) \|h_0\|
		\mbox{ \ in \ } \DD_{R}.
	\end{equation*}
	By the scaling argument, we conclude the validity of the proposition.
\end{proof}

\begin{remark}
	The reason that we solve $\psi_{o,0}$ in \eqref{qd24July12-3} with elaborated pointwise bound in $\rho > 4 R$ is to give a uniform estimate of $c_{*0}[H_0]$ when $R_1= R_*=\infty$. It will be more convenient when solving the reduced equations. The reasoning is the same for the refined estimates of $\psi_{o,1}$ in Proposition \ref{qd24July12-8-prop} below.
	
\end{remark}

\subsection{Mode $1$}

\begin{prop}\label{m1-nonortho}
	Consider
	\begin{equation*}
		\begin{cases}
			\pp_\tau \Psi_1 =\left(a-b W\wedge\right) \left( L_{\rm{in}}  \Psi_1 \right)  + H_1
			\mbox{ \ in \ } \DD_{R},
			\\
			\Psi_1 = 0 	
			\mbox{ \ on \ } \pp\DD_{R},	
			\quad
			\Psi_1(\cdot,\tau_0) = 0
			\mbox{ \ in \ } B_{R(\tau_0)} ,
		\end{cases}
	\end{equation*}
	where $H_1 = \left( h_{1} (\rho,\tau) e^{i\theta} \right)_{\mathbb{C}^{-1}}$, $\| H_1 \|_{v,\ell}^{R}<\infty$. Suppose $\ln R \in \mathbf{AP}$ when $\ell =1$, \eqref{nu-assump},
	\begin{equation}\label{qd24July09-8}
		\mbox{either Case 1: \ } \mathbf{P}_1[R]<1/4
		\mbox{ \ or Case 2: \ }
		\mathbf{P}_1[R] > 1/4, \
		4 \mathbf{P}_1[R] + 2 \mathbf{P}_1[\theta_{R,\ell}] + 2 \mathbf{P}_1[v(\tau)] > -1
		\mbox{ holds},
	\end{equation}
	and $3+ \min\{ 1/2 , \mathbf{P}_1[R^2] \} + \mathbf{P}_1[ R^{-2} \theta_{R,\ell} v(\tau) ] >0$  with $\theta_{R,\ell}$ given in \eqref{theta-Rl-def},
	then there exists a linear mapping $\Psi_1 = \TT_{10}^{R}[H_1]$ with the estimate
	\begin{equation*}
		|\Psi_1(y,\tau)|
		\lesssim
		\min\{\tau^{\frac 12}, R^2\} R^2 \theta_{R,\ell}  v(\tau) 	\langle y\rangle^{-2}  \| H_1 \|_{v,\ell}^{R} .
	\end{equation*}
	Moreover, $\Psi_1\cdot W=0$ and $ e^{-i\theta}\left(\Psi_1\right)_{\mathbb{C}}$ is radial in space.
	
\end{prop}

\begin{proof}
	
	To find a solution $\Psi_1 = \left( \psi_1(\rho,\tau) e^{i\theta}\right)_{\mathbb{C}^{-1} }$,	by Lemma \ref{complex-lem}, we consider
	\begin{equation*}
		\begin{cases}
			\pp_\tau \psi_1 =
			(a-ib)\mathcal{L}_1 \psi_1 + h_1
			=	(a-ib)
			\big(
			\pp_{\rho\rho} \psi_1
			+
			\frac{1}{\rho} \pp_{\rho} \psi_1
			-
			\frac{4}{\rho^2}  \psi_1
			\big)
			+ \tilde{h}_1
			\mbox{ \ in \ } \DD_{R},
			\\
			\psi_1 = 0 	
			\mbox{ \ on \ } \pp\DD_{R},	
			\quad
			\psi_1(\cdot,\tau_0) = 0
			\mbox{ \ in \ } B_{R(\tau_0)},
		\end{cases}
	\end{equation*}
	where $\tilde{h}_1 =
	(a-ib) \frac{12 \rho^2 +4}{(\rho^2+1)^2}
	\frac{1}{\rho^2}
	\psi_1 + h_1$. Denote $\|h_1\| = \|h_1\|_{v,\ell}^{R}$. By \eqref{qd24July09-8} and Lemma \ref{energy est}, we get $
	\|\psi_1(\cdot, \tau)\|_{L^{\infty}(B_R)}
	\lesssim
	\min\{\tau^{\frac 12}, R^2\} R^2 \theta_{R,\ell}  v(\tau) \|h_1 \| $.

	To get spatial decay, by similar argument for \eqref{typical-m0}, for $\rho \le R$,
	\begin{align}
%	\begin{equation}
%		\begin{aligned}
			& |\psi_1|
			\lesssim
			\rho^{2}
			\Big|\Gamma_{6} **\Big( |y|^{-2}
			|\tilde{h}_1 |
			\1_{ \{|y|\le R \} }
			\Big) \Big|
			\lesssim
			\|h_1 \| \rho^{2} \int_{\tau_0}^{\tau} \int_{\mathbb{R}^6}
			(\tau-s)^{-3} e^{-\frac{\kappa|y-z|^2}{\tau-s}}
			\nonumber
			\\
			&\quad \times
			\Big[
			\big( \min\{\tau^{\frac 12}, R^2\} R^2 \theta_{R,\ell}  v \big)(s)
			\Big( \1_{\{ |z|\le 1 \}} |z|^{-4} + \1_{\{ 1< |z|\le R(s) \}} |z|^{-6}  \Big)
			+
			v(s) \1_{\{ 1< |z|\le R(s) \}} |z|^{-2-\ell}
			\Big] dz ds
			\nonumber
			\\
			\lesssim \ &
			\min\{\tau^{\frac 12}, R^2\} R^2 \theta_{R,\ell}  v(\tau) \|h_1 \| \big( \1_{\{ \rho\le 1\}} +
			\1_{\{ \rho > 1\}} \rho^{-2} \langle \ln \rho \rangle \big),
			\label{qd240820-1}
%		\end{aligned}
%	\end{equation}
	\end{align}
	where $ | \Gamma_6(x,y,t,s)| \lesssim (t-s)^{-3} e^{-\frac{\kappa|x-y|^2}{t-s}} $ with a constant $\kappa>0$, and we used $\ln R \in \mathbf{AP}$ when $\ell =1$, \eqref{nu-assump}, $3+ \min\{ 1/2 , \mathbf{P}_1[R^2] \} + \mathbf{P}_1[ R^{-2} \theta_{R,\ell} v(\tau) ] >0$, Lemma \ref{qd24July10-2-lem} for the term $\1_{\{ 1< |z|\le R(s) \}} |z|^{-6}$ and \cite[Lemma A.1]{infi4d} for the other two terms. Since $\min\{\tau^{\frac 12}, R^2\} $ may not stay in $\mathbf{AP}$, we have used $\tau^{\frac 12}$ and $R^2$ in the calculations separately instead. In the last step, $ \1_{\{ \rho > 1\}} \rho^{-2} \langle \ln \rho \rangle$ is from the convolution related to $\1_{\{ 1< |z|\le R(s) \}} |z|^{-6}$, and the other two terms give the bound $ \1_{\{ \rho > 1\}} \rho^{-2}$.
	
	Plugging the new upper bound of $|\psi|$ into \eqref{qd240820-1}, the bound $ \1_{\{ 1< |z|\le R(s) \}} |z|^{-6}$ can be improved to $ \1_{\{ 1< |z|\le R(s) \}} |z|^{-7}$, for $\rho \le R$, we obtain $ |\psi_1|
	\lesssim
	\min\{\tau^{\frac 12}, R^2\} R^2 \theta_{R,\ell}  v(\tau) 	\langle \rho\rangle^{-2}  \|h_1 \| $.
\end{proof}

\begin{lemma}\label{m1-rough}
	Consider
	\begin{equation*}
		\pp_\tau \Psi_1 =
		\left(a-b W\wedge\right) \left( L_{\rm{in}} \Psi_1 \right)  + H_1
		\mbox{ \ in \ } \DD_{R}, \quad \Psi_1(\cdot, \tau_0) =0 \mbox{ \ in \ } B_{R(\tau_0)},
	\end{equation*}
	where $H_1$ is defined in $\DD_{R_*}$ with $ R \le R_* \le \infty$,
	$H_1 = \left( h_{1} (\rho,\tau) e^{i\theta} \right)_{\mathbb{C}^{-1}}$, $\|H_1\|_{v,\ell}^{R_*}<\infty$ with $\ell \in (0,3)$ and the orthogonality condition
	\begin{equation}\label{h1-orth-N1}
		\int_{0}^{R_*(\tau)} h_1(r,\tau) \mathcal{Z}_{1,1}(r) r d r = 0
		\mbox{ \ for \ } \tau \in (\tau_{0},\infty).
	\end{equation}
	Suppose \eqref{nu-assump} and one of the following cases:
	\begin{align}
%	\begin{equation}
%		\begin{aligned}
			&
			\mbox{either Case 1: \ } \mathbf{P}_1[R]<1/4
			\mbox{ \ or Case 2: \ }
			\mathbf{P}_1[R] > 1/4, \
			(10-2\ell) \mathbf{P}_1[R] + 2 \mathbf{P}_1[v(\tau)] > -1
			\mbox{ holds},
			\nonumber
			\\
			&
			3+ \min\{ 1/2 , \mathbf{P}_1[R^2] \} + \mathbf{P}_1[ R^{1-\ell}  v(\tau) ] >0,
			\label{qd24July09-9}
%		\end{aligned}
%	\end{equation}
	\end{align}
	then there exists a solution $\Psi_1 = \TT_{1r}^{R}[H_1]$ as a linear mapping in $H_1$ with the estimate
	\begin{equation}\label{qd240821-10}
		\langle y \rangle |\nabla \Psi_1 |  + 	|\Psi_1 |
		\lesssim
		\min\{\tau^{\frac 12}, R^2\} R^{5-\ell}  v(\tau) 	\langle y\rangle^{-4}  \|H_1\|_{v,\ell}^{R_*}
		\mbox{ \ in \ } \DD_{R}.
	\end{equation}
	Moreover, $\Psi_1\cdot W=0$ and $ e^{-i\theta}\left(\Psi_1\right)_{\mathbb{C}}$ is radial in space.
	
\end{lemma}

\begin{proof}
	
	The proof is the same as Lemma \ref{m0-rough}. Denote $\|h_1\| = \|h_1\|_{v,\ell}^{R_*}$ and set $h_1=0$ in $\DD_{R_*}^c$. Consider
	\begin{equation*}
		(a-b W\wedge) ( L_{\rm{in}} G_1 ) = H_1, \mbox{ \ where \ } G_1 = (g_1(\rho,\tau) e^{i\theta} )_{\mathbb{C}^{-1}} .
	\end{equation*}
	By Lemma \ref{complex-lem}, it is equivalent to $(a-ib ) \mathcal{L}_1 g_1 = h_1 $,	where $g_1$ is given by
	\begin{equation*}
		g_1(\rho,\tau ) =
		(a+ib)
		\Big(
		\mathcal{Z}_{1,2}(\rho)\int^{\rho}_{0}
		h_1(r,\tau) \mathcal{Z}_{1,1}(r) r d r
		-
		\mathcal{Z}_{1,1}(\rho)
		\int_{0}^{\rho}
		h_1(r,\tau) \mathcal{Z}_{1,2}(r) r d r \Big).
	\end{equation*}
	
	Similar to Lemma \ref{m0-rough}, we derive the following by \eqref{scalr-Z} and the orthogonality condition \eqref{h1-orth-N1}, for $0<\ell <4$,
	\begin{equation}\label{m1-z1}
		\begin{aligned}
			&
			|g_1| \lesssim
			\| h_1\| v(\tau)
			\big(
			\1_{\{\rho \le 1\}} \rho^2 \langle  \ln \rho \rangle  + \1_{\{\rho > 1\}} \rho^{2-\ell} \big),
			\quad
			|\partial_{\rho} g_1| \lesssim
			\| h_1\| v(\tau)
			\big(
			\1_{\{\rho \le 1\}} \rho+ \1_{\{\rho > 1\}}  \rho^{1-\ell} \big),
			\\
			&
			\| G_1 \|_{v,\ell-2}^{\infty} =	\| g_1 \|_{v,\ell-2}^{\infty}
			\lesssim  \|h_1 \|,
			\\
			&
			[G_1]_{\mathsf{|DMO|_x}(Q_2^{-}(0,\tau))} \lesssim \|h_0\| v(\tau),
			\quad
			[ G_1(x_*+ \rho_*  z , t_* + \rho_*^2 s) ]_{\mathsf{|DMO|_x}(Q_2^{-}(0))}
			\lesssim \| h_1\| v(t_*) |x_*|^{1-\ell}
			\rho_*
		\end{aligned}
	\end{equation}
	for $t_*>\tau_0$, $1 \le |x_*| \le 3R(t_*)/2$, $\rho_* = |x_*|/100$. Next, let us consider
	\begin{equation*}%\label{Phi1-eq}
		\begin{cases}
			\pp_{\tau} \Phi_1 =
			\left(a-bW\wedge \right) \left( L_{\rm{in}} \Phi_1 \right) + G_1
			\mbox{ \ in \ } \DD_{2R},
			\\
			\Phi_1 = 0
			\mbox{ \ on \ } \pp\DD_{2R}, \quad
			\Phi_1(\cdot,\tau_0) = 0
			\mbox{ \ in \ } B_{2R(\tau_0)} .
		\end{cases}
	\end{equation*}
	Suppose \eqref{nu-assump}, $\ell<3$,
	\eqref{qd24July09-9},
	then $\Phi_1$ is given by Proposition \ref{m1-nonortho} satisfying $\Phi_1\cdot W=0$ and $ \Phi_1 = (\phi_1(\rho,\tau ) e^{i\theta} )_{\mathbb{C}^{-1}}$ for some radial function $\phi_1$. Using \eqref{m1-z1}, we have the estimate
	\begin{equation*}
		|\Phi_1|
		\lesssim
		\min\{\tau^{\frac 12}, R^2\} R^{5-\ell}  v(\tau) 	\langle y \rangle^{-2}  \| h_1 \|  \mbox{ \ in \ } \DD_{2R}.
	\end{equation*}
	By Propositions \ref{qd240725-6-prop}, \ref{scaling-prop-0722}, and the interpolation inequality, we have
	\begin{equation}\label{qd240821-6}
		\langle y\rangle^2 |D^2 \Phi_1| + \langle y\rangle |D \Phi_1| + 	|\Phi_1| \lesssim   \min\{\tau^{\frac 12}, R^2\} R^{5-\ell}  v(\tau) 	\langle y \rangle^{-2}  \| h_1 \| \mbox{ \ in \ } \mathcal{D}_{3R/2}.
	\end{equation}

	The desired strong solution is given by $\Psi_1 = \left(a-bW\wedge \right) \left(L_{\rm{in}} \Phi_1 \right)$. By \eqref{qd240821-6} and the scaling argument, we conclude \eqref{qd240821-10}. By Lemma \ref{complex-lem}, $(\Psi_1)_{\mathbb{C}} = (a-ib) e^{i\theta} \mathcal{L}_1 \phi_1$.
\end{proof}

\begin{prop}\label{qd24July12-8-prop}
	Consider
	\begin{equation*}
		\pp_\tau \Psi_1 =
		\left(a-b W\wedge \right) \left(L_{\rm{in}}\Psi_1 \right) + H_1 +
		\big(
		c_1(\tau) \eta(\rho) \mathcal{Z}_{1,1}(\rho) e^{i\theta} \big)_{\mathbb{C}^{-1}}
		\mbox{ \ in \ } \DD_{R},
		\quad
		\Psi_1(\cdot, \tau_0) =0  \mbox{ \ in \ } B_{R(\tau_0)},
	\end{equation*}
	where $H_1$ is defined in $\DD_{R_*}$ with $ R \le R_* \le \infty$, $H_1 = \left( h_{1} (\rho,\tau) e^{i\theta} \right)_{\mathbb{C}^{-1}}$, $\| H_1 \|_{v,\ell}^{R_*} <\infty$. Suppose \eqref{nu-assump},
	$\ell\in (1,3)$,
	$\mathbf{P}_1[R_0] <1/4$,
	$\mathbf{P}_1[R_0 v(\tau)] > ( \ell -6)/2$,
	then there exists a solution $(\Psi_1,c_{1} ) = ( \TT_1^{R}[H_1], c_{1}[H_1] )$ as a linear mapping in $H_1$ with the estimates
	\begin{equation*}
		\langle y \rangle |\nabla \Psi_{1}| +	|\Psi_{1}|
		\lesssim R_0 v(\tau)
		\big(
		R_0^{6-\ell }  \langle y \rangle^{-4}
		\1_{\{|y|\le 2R_0\}}
		+
		\langle y \rangle^{2-\ell} \1_{\{ |y|> 2R_0\}} \big) \| H_1 \|_{v,\ell}^{R_*},
	\end{equation*}
	\begin{equation*}
		c_{1}(\tau)
		=
		-\Big(\int_{0}^{2} \eta(r) \mathcal{Z}_{1,1}^2(r) r dr\Big)^{-1} \Big( \int_{0}^{R_1}
		h_1(r,\tau)  \mathcal{Z}_{1,1}(r) r dr
		+
		c_{*1}[H_1](\tau) \Big),
	\end{equation*}
	where $2R_0 \le R_1 \le R_*$, $c_{*1}[H_1] $ is a scalar function linearly depending on $H_1$ and satisfies
	$| c_{*1}[H_1](\tau) |\lesssim R_0^{1 -\ell} v(\tau) \| H_1 \|_{v,\ell}^{R_*}$.
	Moreover, $\Psi_1\cdot W=0$ and $ e^{-i\theta}\left(\Psi_1\right)_{\mathbb{C}}$ is radial in space.
\end{prop}

\begin{proof}

	The proof is the verbatim repetition of Proposition \ref{Re-m0-prop}.
	In order to find a solution $\Psi_1$ with the form $\Psi_1 = \left( \psi_1(\rho,\tau) e^{i\theta}\right)_{\mathbb{C}^{-1} }$, by Lemma \ref{complex-lem}, it is equivalent to
	\begin{equation*}
		\pp_\tau \psi_1 =
		(a-ib)\mathcal{L}_1\psi_1 + h_1 +
		c_1(\tau) \eta(\rho) \mathcal{Z}_{1,1}(\rho)
		\mbox{ \ in \ } \DD_{R},
		\quad
		\psi_1(\cdot, \tau_0) =0  \mbox{ \ in \ } B_{R(\tau_0)} .
	\end{equation*}
	Denote $\|h_1\| = \| h_1 \|_{v,\ell}^{R_*}$ and
	take $h_1 =0$ in $\DD_{R_*}^c$. Set $\psi_1 = \eta_{R_0}(\rho)  \psi_{i,1}(\rho,\tau) + \psi_{o,1}(\rho,\tau) $, where $\eta_{R_0}(\rho)  = \eta(\rho/R_0)$. In order to find a solution $\psi_1$, it suffices to  consider the following inner-outer system
	\begin{equation}\label{phio-1-eq}
		\left\{
		\begin{aligned}
			& \pp_{\tau} \psi_{o,1}
			=
			(a-ib) \Big(\pp_{\rho\rho} \psi_{o,1}
			+
			\frac 1{\rho}
			\pp_{\rho} \psi_{o,1}
			-
			\frac{4}{\rho^2} \psi_{o,1}
			\Big)
			+
			J[\psi_{o,1},\psi_{i,1}]
			\1_{\{ \rho \le 4R \} }
			\mbox{ \ in \ }
			\RR^2 \times (\tau_0,\infty),
			\\
			& \psi_{o,1}(\cdot,\tau_0)=0  \mbox{ \ in \ }
			\RR^2,
		\end{aligned}
		\right.
	\end{equation}
	\begin{equation}\label{phii-1-eq}
		\pp_{\tau} \psi_{i,1}
		=
		(a-ib)
		\mathcal{L}_1 \psi_{i,1}
		+
		K[\psi_{o,1}]
		+
		c_1(\tau) \eta(\rho) \mathcal{Z}_{1,1}(\rho)
		\mbox{ \ in \ } \DD_{2R_0} ,
		\quad
		\psi_{i,1}(\cdot, \tau_0) =0 \mbox{ \ in \ } B_{2R_0(\tau_0)},
	\end{equation}
	where we denote
	\begin{small}
		\begin{align*}
			&	J[\psi_{o,1},\psi_{i,1}]
			=
			(a-ib)(1-\eta_{R_0}) \tilde{V}_1(\rho)
			\psi_{o,1}
			+
			A_0[\psi_{i,1}] + (1-\eta_{R_0}) h_1,
			\quad
			\tilde{V}_1(\rho) = \frac{12 \rho^2 +4}{(\rho^2+1)^2 \rho^2},
			\\
			&
			K[\psi_{o,1}] =   (a-ib) \tilde{V}_1(\rho)  \psi_{o,1}
			+ h_1,
			\quad
			A_{0}[\psi_{i,1}] =
			(a-ib)
			\Big[
			\Big( \pp_{\rho\rho}\eta_{R_0}
			+ \frac {\pp_{\rho}\eta_{R_0}}{\rho}
			\Big) \psi_{i,1}
			+2 \pp_{\rho} \eta_{R_0} \pp_{\rho} \psi_{i,1} \Big] - \psi_{i,1} \pp_{\tau}\eta_{R_0}.
		\end{align*}
	\end{small}
	Set $\Psi_{i,1}(y,\tau) = ( \psi_{i,1} e^{i\theta} )_{\mathbb{C}^{-1} }$, that is, $\psi_{i,1} = e^{-i\theta} \left( \Psi_{i,1}\cdot E_1 + i \Psi_{i,1}\cdot E_2 \right)$. Then \eqref{phii-1-eq} is equivalent to
	\begin{equation}\label{phii-1-eq-1}
		\pp_{\tau} \Psi_{i,1}
		=
		\left(a-bW\wedge \right)
		L_{\rm{in}} \Psi_{i,1}
		+
		\big[
		\big(
		K[\psi_{o,1}]
		+
		c_1(\tau) \eta(\rho) \mathcal{Z}_{1,1}(\rho)
		\big) e^{i\theta}
		\big]_{\mathbb{C}^{-1} }
		\mbox{ \ in \ } \DD_{2R_0},
		\quad
		\Psi_{i,1}(\cdot, \tau_0) =0 \mbox{ \ in \ } B_{2R_0(\tau_0)}.
	\end{equation}
	To meet the orthogonality condition \eqref{h1-orth-N1} for $2R_0 \le R_1 \le R_*$ to solve \eqref{phii-1-eq-1}, we take
	\begin{small}
	\begin{equation}\label{qd240725-1}
		c_{1}(\tau) = c_{1}[\psi_{o,1}](\tau) := C_{1,1} \int_{0}^{R_1} \Big[ (a-ib) \tilde{V}_1(r) \psi_{o,1}(r,\tau)
		+ h_1(r,\tau) \Big] \mathcal{Z}_{1,1}(r) r dr,
		\quad
		C_{1,1} :=  -\Big(\int_{0}^2 \eta(r) \mathcal{Z}_{1,1}^2(r) r dr \Big)^{-1}.
	\end{equation}
	\end{small}
	
	By Lemma \ref{convert dim} and Lemma \ref{m1-rough}, we reformulate \eqref{phio-1-eq} and \eqref{phii-1-eq-1} into the following form formally
	\begin{equation}\label{m1-z3}
		\begin{aligned}
			\psi_{o,1}(\rho,\tau) = \ &
			\rho^{2}
			\big[
			\Gamma_{6}^{\natural}**\big(
			|z|^{-2}
			J[\psi_{o,1},\psi_{i,1}]   \1_{\{ |z| \le 4R(s) \} }\big) \big](\rho,\tau,\tau_0),
			\\
			\Psi_{i,1}(y,\tau) = \ &
			\TT_{1r}^{2R_0} \big[ \big[
			\big(
			K[\psi_{o,1}]
			+
			c_1(\tau) \eta(\rho) \mathcal{Z}_{1,1}(\rho)
			\big) e^{i\theta}
			\big]_{\mathbb{C}^{-1} } \big].
		\end{aligned}
	\end{equation}
	
	Denote $H_{I}:=\big\{ \big[ h_1+ C_{1,1} \big(\int_{0}^{R_1}
	h(r,\tau) \mathcal{Z}_{1,1}(r) r dr \big) \eta(\rho) \mathcal{Z}_{1,1}(\rho) \big] e^{i\theta} \big\}_{\mathbb{C}^{-1} }$. Obviously, $\| H_{I} \|_{v,\ell}^{R_1} \lesssim \|h_1 \|$. Inspired Lemma \ref{m1-rough}, if $\ell \in (0,3)$, $\mathbf{P}_1[R_0] <1/4$, $3+ \mathbf{P}_1[R_0^{3-\ell} v(\tau)] >0$, and
	$e^{-i\theta} (H_{I})_{\mathbb{C}}$ satisfies the orthogonality condition \eqref{h1-orth-N1} (with $R_*=R_1$),
	then we have the estimate
	\begin{equation}\label{qd24July12-4}
		\langle y \rangle
		\big|\nabla \TT_{1r}^{2R_{0}}[H_{I}] \big|
		+
		\big|\TT_{1r}^{2R_{0}} [H_{I}] \big| \le
		D_i  w_{i,1}(\rho,\tau) \| h_1 \|
		\mbox{ \ in \ } \DD_{2R_0},
	\end{equation}
	where the constant $D_{i}\ge 1$ is large and  $w_{i,1}(\rho,\tau): =
	R_0^{7-\ell}  v(\tau) 	\langle \rho\rangle^{-4}  $.
	Denote
	\begin{equation*}
		\begin{aligned}
			\B_{i,1} := \ &
			\Big\{
			F(y,\tau) \in C^{1}\left(B_{2R_0} , \mathbb{R}^3 \right) \ | \ F(y,\tau) = \big( f(\rho,\tau) e^{i\theta} \big)_{\mathbb{C}^{-1}}
			\mbox{ \ for some radial scalar function}
			\\
			&~ f(\rho,\tau) \mbox{ \ and \ }
			\langle y \rangle |\nabla F(y,\tau)|  + |F(y,\tau)| \le
			2D_i  w_{i,1}(\rho,\tau) \| h_1 \|
			\Big\}.
		\end{aligned}
	\end{equation*}
	Given $\tilde{\Psi}_{i,1}\in \B_{i,1}$, denote $\tilde{\psi}_{i,1} = e^{-i\theta} \big( \tilde{\Psi}_{i,1}\cdot E_1 + i \tilde{\Psi}_{i,1}\cdot E_2 \big)$. We will find a solution $\psi_{o,1} = \psi_{o,1}[\tilde{\Psi}_{i,1}]$ of \eqref{phio-1-eq} by the contraction mapping theorem. Let us estimate $J[\psi_{o,1},\tilde{\psi}_{i,1}]$. Similar to \eqref{Sep10-04}, we have
	\begin{equation*}
		| \pp_{\rho} \tilde{\psi}_{i,1} |
		\lesssim
		|\tilde{\Psi}_{i,1}| \langle \rho\rangle^{-2}
		+
		| \nabla \tilde{\Psi}_{i,1}  |
		\lesssim
		D_i R_0^{7-\ell}  v(\tau) 	\langle \rho\rangle^{-5} \| h_1 \|.
	\end{equation*}
	By $|R_0'| =O( R_0^{-1} )$ in \eqref{nu-assump}, we have
	\begin{equation*}
		|A_{0}[ \tilde{\psi}_{i,1} ]|
		+
		|(1-\eta_{R_0}) h_1 |
		\lesssim
		D_i \1_{\{ R_0 \le \rho \le 2R_0 \} }  v(\tau) R_0^{1-\ell} \| h_1 \| +
		\1_{\{ \rho\ge R_0 \}} v(\tau)  \rho^{-\ell } \| h_1\|
		\lesssim D_i
		\1_{\{ \rho\ge R_0 \}} R_0 v(\tau)  \rho^{-\ell } \| h_1\|.
	\end{equation*}
	Provided \eqref{nu-assump}, $\ell\in (0,4)$, $\mathbf{P}_1[R_0 v(\tau)] > ( \ell -6)/2$,  by Lemma \ref{qd24July11-1-lem},
	it follows that
	\begin{equation}\label{qd24July11-1}
		\begin{aligned}
			&
			\Big|\rho^{2}
			\Gamma_{6}^{\natural}**\Big\{
			|z|^{-2}
			\big[ A_0[\tilde{\psi}_{i,1}] + (1-\eta_{R_0}) h_1 \big] \1_{\{ |z| \le 4R(s) \} }  \Big\}
			\Big|
			\\
			\le \ &
			C D_i
			\rho^{2}
			|\Gamma_{6}^{\natural}|**\Big[
			\1_{\{ R_0(s) \le |z| \le 4R(s) \} } R_0(s) v(s) |z|^{-2-\ell} \| h_1\|   \Big]
			\le D_{o} D_i w_{o,1}(\rho,\tau) \| h_1 \|
		\end{aligned}
	\end{equation}
	with a large constant $D_o\ge 1$,
	$w_{o,1}(\rho,\tau) := R_0 v(\tau)
	\big( \rho^2 R_0^{-\ell} \1_{\{ \rho\le R_0 \}}
	+   \rho^{2-\ell}
	\1_{\{ \rho>R_0 \}} \big)$.
	Denote
	\begin{equation*}
		\B_{o,1} :=
		\left\{
		f(\rho,\tau) \ \big| \
		|f(\rho,\tau)| \le
		2 D_o D_i w_{o,1}(\rho,\tau)  \| h_1 \|
		\mbox{ \ for \ } \rho\ge 0, \tau\ge \tau_0
		\right\} .
	\end{equation*}
	For any $\tilde{\psi}_{o,1} \in \B_{o,1}$,
	\begin{equation*}
		\big|(1-\eta_{R_0}) \tilde{V}_1(\rho)
		\tilde{\psi}_{o,1}  \1_{\{ \rho \le 4R \} } \big|
		\lesssim
		(\inf_{s\ge \tau_0} R_{0}(s) )^{-2}
		D_o D_i
		R_0 v(\tau) \rho^{-\ell}
		\1_{\{R_0 <  \rho \le 4R \} } \| h_1\|.
	\end{equation*}
	Similar to the estimate of \eqref{qd24July11-1}, due to the small quantity $(\inf_{s\ge \tau_0} R_{0}(s) )^{-2}$, we have
	\begin{equation*}
		\rho^{2}
		\Gamma_{6}^{\natural}**\big(
		|z|^{-2}
		J[\tilde{\psi}_{o,1},\tilde{\psi}_{i,1}] \1_{\{ |z| \le 4R(s) \} }   \big)  \in \B_{o,1}.
	\end{equation*}
The contraction mapping property can be deduced similarly. Thus we find a solution $\psi_{o,1} = \psi_{o,1}[\tilde{\psi}_{i,1}]\in \B_{o,1}$.  Then for $\rho \le 2R_0$, we have
	\begin{equation*}
		\begin{aligned}
			&
			\big| \tilde{V}_1(\rho) \psi_{o,1}[\tilde{\psi}_{i,1}]
			\big|
			\lesssim
			D_o D_i
			R_0^{ \max\{-1, 1-\ell \} }
			v(\tau)
			\langle \rho\rangle^{-\ell } \| h_1 \|,
			\\
			&
			\Big|C_{1,1} \Big( \int_{0}^{R_1} \tilde{V}_1(r) \psi_{o,1}[\tilde{\psi}_{i,1}](r,\tau)
			\mathcal{Z}_{1,1}(r) r dr \Big)\eta(\rho) \mathcal{Z}_{1,1}(\rho) \Big|
			\lesssim
			D_o D_i
			R_0^{1-\ell}
			v(\tau)
			\langle \rho\rangle^{-\ell } \| h_1\| .
		\end{aligned}
	\end{equation*}
	
	Due to the choice of $c_1(\tau)$, $h_{II}:=K[\psi_{o,1}[\tilde{\psi}_{i,1}]] + c_1 [\psi_{o,1}[\tilde{\psi}_{i,1}]](\tau) \eta(\rho) \mathcal{Z}_{1,1}(\rho)$ satisfies the orthogonality condition \eqref{h1-orth-N1} (with $R_*=R_1$).
	For $\ell>1$, $\big( \inf_{s\ge \tau_0}R_0(s) \big)^{ \max\{-1, 1-\ell \} }$ provides a small quantity.
	Under the same parameters restriction for deriving \eqref{qd24July12-4}, by Lemma \ref{m1-rough},
	we have
	\begin{equation*}
		\TT_{1r}^{2R_0}[(h_{II} e^{i\theta})_{\mathbb{C}^{-1} }] \in \B_{i,1}
	\end{equation*}
	The contraction property can be deduced similarly.
	Thus we find a solution $\Psi_{i,1}=\Psi_{i,1}[h_1] \in \B_{i,1}$, and then a solution $(\psi_{o,1},\Psi_{i,1})$ for \eqref{phio-1-eq} and \eqref{phii-1-eq-1}. $\psi_{o,1}$, $\Psi_{i,1}$ $c_1(\tau)$ depends on $h_1$ linearly.
	
	We will regard $D_o$, $D_i$  as general constants hereafter. Since $\psi_{o,1}[h_1] \in \B_{o,1}$, then
	\begin{equation*}
		c_{1}[h_1](\tau)
		=
		C_{1,1} \Big( \int_{0}^{R_1}
		h_1(r,\tau)  \mathcal{Z}_{1,1}(r) r dr
		+
		c_{*1}[h_1](\tau) \Big),
	\end{equation*}
	where $c_{*1}[h_1]$ depends on $h_1$ linearly and $| c_{*1}[h_1](\tau) |\lesssim R_0^{1-\ell} v(\tau) \|h_1\|$. By estimates of $\psi_{o,1}$, $\Psi_{i,1}$, then
	\begin{equation*}
		| \Psi_{1}(y,\tau)|
		\lesssim R_0 v(\tau)
		\big(
		R_0^{6-\ell }  \langle \rho \rangle^{-4}
		\1_{\{\rho\le 2R_0\}}
		+
		\langle \rho \rangle^{2-\ell} \1_{\{\rho> 2R_0\}} \big) \|h_1 \|
		\mbox{ \ in \ } \DD_{R}.
	\end{equation*}
	By the scaling argument, the proposition is concluded.
\end{proof}

\subsection{Mode $-1$}\label{sec-linearinner-1}

To prepare for the linear theory, we first summarize the properties of the spectrum, Fourier basis $\Phi^{-1}(\rho,\xi)$, spectrum measure $\rho_{-1}(d \xi)$ of $-\tilde{\LLL}_{-1}$ in the next proposition, where $$\tilde{\LLL}_{-1}:= \pp_{\rho\rho}
+
\frac 14 \rho^{-2} + V_{-1}(\rho) = \pp_{\rho\rho}
+
\frac 14 \rho^{-2}
-\frac{4}{\rho^2}
-\frac{-4}{\rho^2+1}
+\frac{8}{(\rho^2+1)^2}.
$$

\begin{prop}[\cite{KMS20WM}]\label{Phi-1 estimate}
	
	The spectrum of $-\tilde{\LLL}_{-1}$ is $[0,\infty)$, which is also the essential spectrum. Denote the Fourier basis of $-\tilde{\LLL}_{-1}$ as $\Phi^{-1}(\rho,\xi)$, where $\Phi^{-1}(\rho,\xi)$ satisfies
	\begin{equation*}
		-\tilde{\LLL}_{-1}
		\Phi^{-1}(\rho,\xi)
		= \xi \Phi^{-1}(\rho,\xi) \mbox{ \ for all \ } \xi \ge 0.
	\end{equation*}
	
	For $\rho\ge 0$, $\xi\ge 0$, we have
	\begin{equation}\label{qd24Sep21-6}
		\begin{aligned}
			&
			|\Phi^{-1}(\rho,\xi) |
			\lesssim
			\rho^{\frac 52} \langle \rho \rangle^{-2}
			\1_{\{ \rho^2\xi \le 1 \}}
			+
			\xi^{-\frac 14}
			\langle
			\xi \rangle^{-1}
			\1_{\{ \rho^2\xi > 1 \}},
			\\
			&
			|\partial_{\rho}\Phi^{-1}(\rho,\xi)|\lesssim \rho^{\frac 32} \1_{\{ \rho^2\xi \le 1 \}}
			+
			\xi^{\frac 14}
			\langle
			\xi \rangle^{-1}
			\1_{\{ \rho^2\xi > 1 \}}.
		\end{aligned}
	\end{equation}
	$\Phi^{-1}(\rho,\xi)$ has the expansion
	\begin{equation}\label{qd24Sep21-5}
		\Phi^{-1}(\rho,\xi)
		=
		\Phi^{-1}_{0}(\rho)
		+
		\rho^{\frac 12}
		\sum\limits_{j=1}^{\infty}
		(-\rho^2 \xi)^{j} \Phi_{j}(\rho^2),
	\end{equation}
	which converges absolutely, where $\Phi_0^{-1}(\rho) = \rho^{\frac 52} (1+\rho^2)^{-1}$. It converges uniformly if $\rho \xi^{\frac 12} $ remains bounded. Here $\Phi_{j}(u) \ge 0$ are smooth functions of $u\ge 0$ satisfying
	$ \Phi_{j}(u) \le \frac{1}{j!} \frac{u}{1+u}  $, $|\Phi_j'(u)| \le \frac{21}{j!}$  for $u\ge 0$, $j\ge 1$,
	and $\Phi_{1}(u) \ge c_1 \frac{u}{1+u}$ for $u\ge 0$ with a constant $c_1 >0$.
	
	The spectrum measure $\rho_{-1}(d \xi)$ of $-\tilde{\LLL}_{-1}$ is supported in $\xi \in [0,\infty)$ and absolutely continuous on $\xi\ge 0$ with density
	$ \frac{d \rho_{-1}(\xi)}{d \xi}
	\sim \langle \xi \rangle^2$.
	
\end{prop}

\begin{proof}

	Most of the estimates can be found in \cite[Propositions 5.1, 5.3, 5.4, 5.5]{KMS20WM}. Derivative estimate \eqref{qd24Sep21-6}$_2$ can be derived similarly as in \cite{KMS20WM}. Indeed, by similar induction in \cite[pp. 32-33]{KMS20WM}, we have $|\Phi_j'(u)| \le \frac{21}{j!}$ for $u\ge 0$, $j\ge 1$. Acting $\partial_{\rho}$ on \eqref{qd24Sep21-5}, for $\rho^2 \xi\le 1$, we have $|\partial_{\rho}\Phi^{-1}(\rho,\xi)|\lesssim \rho^{\frac 32}$. For the remote region $\rho^2\xi > 1$, the use of \cite[Propositions 5.4, 5.5]{KMS20WM} and the relation between the Weyl-Titchmarsh function and $\Phi^{-1}(\rho,\xi)$ gives the estimate $|\partial_{\rho} \Phi^{-1}(\rho,\xi) |
	\lesssim
	\xi^{\frac 14}
	\langle
	\xi \rangle^{-1}$.
\end{proof}

We emphasize that the assumption \eqref{nu-assump} is not required in the next proposition.
\begin{prop}\label{qd24July13-3-prop}
	
	Consider
	\begin{equation*}
		\pp_{\tau} \Phi_{-1} =	(a-bW\wedge)( L_{\rm{in}} \Phi_{-1}) + H_{-1}
		\mbox{ \ in \ } \mathbb{R}^2 \times (\tau_0,\infty),
		\quad
		\Phi_{-1}(\cdot,\tau_0) = 0
		\mbox{ \ in \ } \mathbb{R}^2,
	\end{equation*}
	where $\tau_{0}\ge 2$, $H_{-1} = (h(\rho,\tau) e^{-i\theta} )_{\mathbb{C}^{-1}}$, $
	\| h \|_{v,\ell}^{\infty} <\infty$, $0 \le v(\tau) \in L_{\rm{loc}}^{\infty}( [\tau_0,\infty) )$, $\ell > 3/2$. $\Phi_{-1} = \TT_{-1}[H_{-1}]$ is given as a linear mapping in $H_{-1}$ by the convolution via the fundamental solution of the parabolic system. Moreover, $\Phi_{-1}(y,\tau) = (\phi_{-1}(\rho, \tau) e^{-i\theta})_{\mathbb{C}^{-1}}$ and $\phi_{-1}$ satisfies
	\begin{equation}\label{qd240729-1}
		\pp_{\tau} \phi_{-1} =	(a-ib)\mathcal{L}_{-1}  \phi_{-1} + h
		\mbox{ \ for \ } (\rho,\tau) \in (0,\infty)\times (\tau_0,\infty),
		\quad
		\phi_{-1}(\rho,\tau_0) = 0 \mbox{ \ for \ } \rho \in (0,\infty)
	\end{equation}
	with the estimate
	\begin{align}
		%\begin{equation}
		%\begin{aligned}
		|\phi_{-1}(\rho,\tau)| \lesssim \ &
		\| h \|_{v,\ell}^{\infty} \1_{ \{ \rho\le \tau^{\frac 12} \} }
		\begin{cases}
			\tau^{1-\frac{\ell}{2}} \sup\limits_{s\in [\tau/2,\tau]} v(s)
			+
			\tau^{-\frac{\ell}{2}}
			\int_{\frac{\tau_0}{2}}^{\frac{\tau}{2}}
			v(s) ds
			&
			\mbox{ \ if  \ } \ell <2
			\\
			(\ln \tau)^2  \sup\limits_{s\in [\tau/2,\tau]} v(s)
			+
			\tau^{-1} \ln \tau
			\int_{\frac{\tau_0}{2}}^{\frac{\tau}{2}}
			v(s) ds
			&
			\mbox{ \ if  \ } \ell =2
			\\
			\ln \tau  \sup\limits_{s\in [\tau/2,\tau]} v(s)
			+
			\tau^{-1}
			\int_{\frac{\tau_0}{2}}^{\frac{\tau}{2}}
			v(s) ds
			&
			\mbox{ \ if  \ } \ell >2
		\end{cases}
		\nonumber
		\\
		&
		+
		\| h \|_{v,\ell}^{\infty} \1_{ \{ \rho > \tau^{\frac 12} \} }
		\rho^{-\frac 12}
		\begin{cases}
			\tau^{\frac 54-\frac{\ell}{2}} \sup\limits_{s\in [\tau/2,\tau]} v(s)
			+
			\tau^{\frac 14 - \frac{\ell}{2}}
			\int_{\frac{\tau_0}{2}}^{ \frac{\tau}{2} }
			v(s)  ds
			&
			\mbox{ \ if  \ }\ell<2
			\\
			\tau^{\frac 14}  \ln \tau  \sup\limits_{s\in [\tau/2,\tau]} v(s)
			+
			\tau^{-\frac 34} \ln \tau
			\int_{\frac{\tau_0}{2}}^{ \frac{\tau}{2} }
			v(s) ds
			&
			\mbox{ \ if  \ }\ell=2
			\\
			\tau^{\frac 14 } \sup\limits_{s\in [\tau/2,\tau]} v(s)
			+
			\tau^{-\frac 34 }
			\int_{\frac{\tau_0}{2}}^{ \frac{\tau}{2} }
			v(s)  ds
			&
			\mbox{ \ if  \ } \ell>2,
		\end{cases}
		\label{phi-1-nonorth}
		%\end{aligned}
		%\end{equation}
	\end{align}
	where we assume $v(s)=0$ for $s\le \tau_0$.
	Moreover,  if we suppose in addition $2<\ell < 5/2$ and the orthogonality condition
	\begin{equation}\label{h-1-ortho}
		\int_{0}^{\infty} h(r,\tau) \mathcal{Z}_{-1,1}(r) r d r =0
		\mbox{ \ for \ }\tau > \tau_{0}
	\end{equation}
	holds, then we have the estimate
	\begin{equation}\label{phi-1-orth}
		|\phi_{-1}(\rho,\tau)|
		\lesssim
		\| h \|_{v,\ell}^{\infty}
		\begin{cases}
			\langle \rho\rangle^{2 - \ell } \sup\limits_{s\in [\tau/2,\tau]} v(s)
			+
			\tau^{-\frac{\ell}{2}}
			\int_{\frac{\tau_0 }{2}}^{\frac{\tau }{2}}
			v(s) ds
			&
			\mbox{ \ if \ } \rho\le \tau^{\frac 12}
			\\
			\rho^{-\frac 12}
			\Big(
			\tau^{\frac 54 -\frac{\ell}{2}} \sup\limits_{s\in [\tau/2,\tau]} v(s)
			+
			\tau^{\frac 14-\frac{\ell}{2}} \int_{\frac{\tau_0}{2}}^{\frac{\tau}{2}}
			v(s)  ds
			\Big)
			&
			\mbox{ \ if \ } \rho > \tau^{\frac 12}.
		\end{cases}
	\end{equation}

\end{prop}

\begin{proof}
	
	Similar to the argument in Lemma \ref{energy est}, the theory of the parabolic system guarantee the existence of $\Phi_{-1}$ and $\phi_{-1}$. And $\Phi_{-1}, \nabla \Phi_{-1} \in L^{\infty} ( \Omega \times (\tau_0,\tau_1) )$ for any bounded domain $\Omega \subset \mathbb{R}^2$ and $\tau_1>\tau_0$. Using the argument in \eqref{Sep10-04}, we have that $\phi_{-1}, \partial_{\rho} \phi_{-1}$ is bounded in $(0,\rho_1)\times (\tau_0,\tau_1)$ for any $\rho_1>0$ and $\tau_1 > \tau_0$.
	
	We will give a representation formula for $\phi_{-1}$. First, besides $
	\| h \|_{v,\ell}^{\infty} <\infty$, we assume
	\begin{equation}\label{qd24Sep22-1}
		\begin{aligned}
			&
			\mbox{$h(\rho,\tau)$ is smooth in spatial variable $\rho\in \mathbb{R}$ and for  $\tau_0 < \tau <\tau_1<\infty$,}
			\\
			&
			\mbox{$h(\rho,\tau) = 0$ for $|\rho|\ge M_{\tau_1}$ with a constant $M_{\tau_1}>0$ depending on $\tau_1$.}
		\end{aligned}
	\end{equation}

	Take $\phi_{-1} = \rho^{-\frac{1}{2}} f(\rho,\tau)$. Since $\mathcal{L}_{-1} (\rho^{-\frac{1}{2}} f ) = \big( \rho^{-\frac{1}{2}} \partial_{\rho \rho} + \frac{1}{4} \rho^{-\frac 52} + V_{-1}(\rho) \rho^{-\frac{1}{2}} \big) f $, then
	\begin{equation}\label{qd24Sep20-3}
		\pp_{\tau} f
		= (a-ib) \tilde{\LLL}_{-1} f  +    \rho^{\frac{1}{2}} h \mbox{ \ for \ } (\rho,\tau) \in (0,\infty)\times (\tau_0,\infty),
		\quad
		f(\rho,\tau_0) = 0
		\mbox{ \ for \ } \rho \in (0,\infty).
	\end{equation}
	Due to the assumption \eqref{qd24Sep22-1}, given a fixed $\tau>\tau_0$, $\Phi_{-1}$, $\nabla \Phi_{-1}$ have fast spatial decay as $|y|\to \infty$. So do $\phi_{-1}$, $\partial_{\rho} \phi_{-1}$, $f$, and $\partial_{\rho} f$. And $|f(\rho,\tau)| \lesssim \rho^{\frac{1}{2}}$ and $|\partial_{\rho}f(\rho,\tau)| \lesssim \rho^{-\frac{1}{2}}$ as $\rho \downarrow 0$. Combining these with \eqref{qd24Sep21-6}, we are able to
	multiply \eqref{qd24Sep20-3} by $\Phi^{-1}(\rho,\xi)$ and integrate by parts in $\rho \in (0,\infty)$ to deduce
	\begin{equation*}
		\partial_{\tau} \hat{f}(\xi,\tau) + (a-ib) \xi \hat{f}(\xi,\tau) =
		\int_{0}^{\infty} \rho^{\frac{1}{2}} h(\rho,\tau) \Phi^{-1}(\rho,\xi)
		d \rho,
		\quad
		\hat{f}(\xi,\tau_0) = 0,
	\end{equation*}
	where we denote $	\hat{f}(\xi,\tau) :=\int_{0}^{\infty}
	f(\rho,\tau)
	\Phi^{-1}(\rho,\xi)
	d \rho$. It follows that
	\begin{equation*}
		\hat{f}(\xi,\tau)
		=
		\int_{\tau_0}^{\tau} e^{-(a-ib)\xi(\tau-s)}
		\int_{0}^{\infty} x^{\frac{1}{2}} h(x,s) \Phi^{-1}(x,\xi)
		d x ds.
	\end{equation*}
	Using the distorted Fourier transform, we get the representation formula
	\begin{align}
%	\begin{equation}
%		\begin{aligned}
			& \phi_{-1} =
			\rho^{-\frac{1}{2}} f(\rho,\tau) = \rho^{-\frac{1}{2}} \int_0^{\infty} \Phi^{-1}(\rho,\xi) \hat{f}(\xi,\tau) \rho_{-1}(d\xi)
			\nonumber
			\\
			= \ &
			\rho^{-\frac{1}{2}}
			\int_0^{\infty} \Phi^{-1}(\rho,\xi) \int_{\tau_0}^{\tau} e^{-(a-ib)\xi(\tau-s)}
			\int_{0}^{\infty} x^{\frac{1}{2}} h(x,s) \Phi^{-1}(x,\xi)
			d x ds \rho_{-1}(d\xi)
			\nonumber
			\\
			= \ & \rho^{-\frac{1}{2}}
			\int_{\tau_0}^{\tau}
			\int_{0}^{\infty}
			\int_0^{\infty}   e^{-(a-ib)\xi(\tau-s)}
			\Phi^{-1}(\rho,\xi)
			\Phi^{-1}(x,\xi)
			x^{\frac{1}{2}} h(x,s) \rho_{-1}(d\xi) d x ds.
			\label{phi-n-Duhamel}
%		\end{aligned}
%	\end{equation}
	\end{align}
	
	For general $h$ satisfying $
	\| h \|_{v,\ell}^{\infty} <\infty$, if the last integral in \eqref{phi-n-Duhamel} is absolutely integrable, then \eqref{phi-n-Duhamel} gives the representation formula of $\phi_{-1}$. Hereafter, without loss of generality, we assume $\| h \|_{v,\ell}^{\infty}=1$ and will prove that \eqref{phi-n-Duhamel} is absolutely integrable and give pointwise estimate of $\phi_{-1}$.
	
\noindent	\textbf{Estimate without orthogonality.} Using the property of $\rho_{-1}(d \xi)$ in Proposition \ref{Phi-1 estimate}, we have
	\begin{equation*}
		|\phi_{-1}|
		\lesssim
		\rho^{-\frac 12}
		\int_{\tau_{0}}^{\tau}
		v(s)
		\int_{0}^{\infty}
		\int_{0}^{\infty}
		e^{- a \xi (\tau -s)}
		|\Phi^{-1} (\rho,\xi) |
		|\Phi^{-1}(x,\xi)|
		x^{\frac 12}
		\langle x\rangle^{-\ell}
		\langle \xi \rangle^{2}
		d x d\xi d s.
	\end{equation*}
	
	We will use the estimate of $|\Phi^{-1}|$ in \eqref{qd24Sep21-6} repetitively. First, we consider
	\begin{equation*}
		F(\xi) :=
		\int_{0}^{\infty}
		|\Phi^{-1}(x,\xi)|
		x^{\frac 12} \langle x \rangle^{-\ell}
		d x
		= \int_{0}^{\xi^{-\frac 12}} + \int_{\xi^{-\frac 12}}^{\infty} \cdots : = F_1 + F_2 .
	\end{equation*}
	For $F_1$, one has
	\begin{equation*}
		F_1\lesssim \int_{0}^{\xi^{-\frac 12}}
		x^{\frac 52} \langle x \rangle^{-2}
		x^{\frac 12} \langle x \rangle^{-\ell}
		d x
		\lesssim
		\1_{\{ \xi\le 1 \}}
		\begin{cases}
			\xi^{\frac{\ell}{2} -1}
			&
			\mbox{ \ if  \ } \ell<2
			\\
			\langle \ln \xi \rangle
			&
			\mbox{ \ if  \ } \ell=2
			\\
			1
			&
			\mbox{ \ if  \ } \ell>2
		\end{cases}
		\quad
		+
		\1_{\{ \xi > 1 \}}
		\xi^{-2}.
	\end{equation*}
	For $F_2$, since $\ell >\frac 32$, $
	F_2 \lesssim \xi^{-\frac 14} \langle \xi \rangle^{-1} \int_{\xi^{-\frac 12}}^{\infty}	
	x^{\frac 12} \langle x \rangle^{-\ell}
	d x
	\lesssim
	\1_{\{ \xi \le 1 \}}
	\xi^{\frac{\ell}{2} -1}
	+
	\1_{\{ \xi > 1 \}}
	\xi^{-\frac 54} $,
	and thus
	\begin{equation}\label{F-est-DTF}
		F(\xi) \lesssim
		\1_{\{\xi\le 1\}}
		\begin{cases}
			\xi^{\frac{\ell}{2} -1}
			&
			\mbox{ \ if  \ } \ell<2
			\\
			\langle \ln \xi \rangle
			&
			\mbox{ \ if  \ } \ell=2
			\\
			1
			&
			\mbox{ \ if  \ } \ell>2
		\end{cases}
		\quad
		+ \1_{\{ \xi > 1 \}}
		\xi^{-\frac 54}.
	\end{equation}
	
	Next, let us estimate
	\begin{equation*}
		P(\rho,\tau,s): = \int_{0}^{\infty}
		e^{- a \xi (\tau -s)}
		|\Phi^{-1} (\rho,\xi) |
		F(\xi)
		\langle \xi \rangle^{2}
		d\xi
		=
		\int_{0}^{\frac{1}{\rho^2}}
		+
		\int_{\frac{1}{\rho^2}}^{\infty}
		\cdots :=
		P_1 + P_2.
	\end{equation*}
	First, let us estimate $P_1$.
	Note that $ P_{1} \lesssim
	\rho^{\frac 52} \langle \rho \rangle^{-2}
	\int_{0}^{\rho^{-2}}
	e^{- a \xi (\tau -s)}
	F(\xi)
	\langle \xi \rangle^{2}
	d\xi $.
	For $\rho \ge 1$, since $\xi \le \rho^{-2} \le 1$, by Lemma \ref{asy-lem-1},
		\begin{align*}
			P_{1} \lesssim \ &
			\rho^{\frac 12}
			\int_{0}^{\rho^{-2}}
			e^{- a \xi (\tau -s)}
			\begin{cases}
				\xi^{\frac{\ell}{2} -1}
				&
				\mbox{ \ if  \ } \ell<2
				\\
				\langle \ln \xi \rangle
				&
				\mbox{ \ if  \ } \ell=2
				\\
				1
				&
				\mbox{ \ if  \ } \ell>2
			\end{cases}
			d\xi
			\\
			\lesssim \ &
			\begin{cases}
				\begin{cases}
					\rho^{\frac 12-\ell }
					&
					\mbox{ \ if \ } \tau -s \le \rho^2
					\\
					\rho^{\frac 12}
					(\tau-s)^{-\frac{\ell}{2}}
					&
					\mbox{ \ if \ } \tau -s > \rho^2
				\end{cases}
				&
				\mbox{ \ if  \ } \ell<2
				\\
				\begin{cases}
					\rho^{-\frac 32} \langle \ln \rho \rangle
					&
					\mbox{ \ if \ } \tau -s \le \rho^2
					\\
					\rho^{\frac 12}  (\tau-s)^{-1} \langle \ln(a(\tau-s) ) \rangle
					&
					\mbox{ \ if \ } \tau -s > \rho^2
				\end{cases}
				&
				\mbox{ \ if  \ } \ell=2
				\\
				\begin{cases}
					\rho^{-\frac 32}
					&
					\mbox{ \ if \ } \tau -s \le \rho^2
					\\
					\rho^{\frac 12} (\tau-s)^{-1}
					&
					\mbox{ \ if \ } \tau -s > \rho^2
				\end{cases}
				&
				\mbox{ \ if  \ } \ell>2.
			\end{cases}
		\end{align*}
	For $\rho < 1$,
	$ P_{1} \lesssim
	\rho^{\frac 52}
	\big(
	\int_{0}^{1}
	+
	\int_{1}^{\rho^{-2}}
	\big)
	e^{- a \xi (\tau -s)}
	F(\xi)
	\langle \xi \rangle^{2}
	d\xi $. By the same estimate above,
\begin{align*}
		\int_{0}^{1}
		e^{- a \xi (\tau -s)}
		F(\xi)
		\langle \xi \rangle^{2}
		d\xi
		\lesssim
		\begin{cases}
			\begin{cases}
				1
				&
				\mbox{ \ if \ } \tau -s \le 1
				\\
				(\tau-s)^{-\frac{\ell}{2}}
				&
				\mbox{ \ if \ } \tau -s > 1
			\end{cases}
			&
			\mbox{ \ for  \ } \ell<2
			\\
			\begin{cases}
				1
				&
				\mbox{ \ if \ } \tau -s \le 1
				\\
				(\tau-s)^{-1} \langle \ln(a(\tau-s) ) \rangle
				&
				\mbox{ \ if \ } \tau -s > 1
			\end{cases}
			&
			\mbox{ \ for  \ } \ell=2
			\\
			\begin{cases}
				1
				&
				\mbox{ \ if \ } \tau -s \le 1
				\\
				(\tau-s)^{-1}
				&
				\mbox{ \ if \ } \tau -s > 1
			\end{cases}
			&
			\mbox{ \ for  \ } \ell>2.
		\end{cases}
\end{align*}
	Using \eqref{F-est-DTF} and direct calculation, we have
\begin{align*}
			&
			\int_{1}^{\frac{1}{\rho^2}}
			e^{- a \xi (\tau -s)}
			F(\xi)
			\langle \xi \rangle^{2}
			d\xi
			\lesssim
			\int_{1}^{\frac{1}{\rho^2}}
			e^{- a \xi (\tau -s)}
			\xi^{\frac 34}
			d\xi
			\sim
			(\tau-s)^{-\frac 74}
			\int_{a(\tau-s)}^{\frac{a(\tau-s)}{\rho^2}} e^{-z} z^{\frac 34} dz
			\\
			\lesssim \ &
			\begin{cases}
				\rho^{-\frac 72}
				&
				\mbox{ \ if \ } \tau-s \le \rho^2
				\\
				(\tau-s)^{-\frac 74}
				&
				\mbox{ \ if \ } \rho^2 < \tau-s \le 1
				\\
				(\tau-s)^{-\frac 74}
				e^{-\frac{a(\tau-s)}{2}}
				&
				\mbox{ \ if \ }  \tau-s > 1.
			\end{cases}
		\end{align*}
	Thus for $\rho<1$,
	\begin{equation*}
		P_{1} \lesssim
		\begin{cases}
			\begin{cases}
				\rho^{-1}
				&
				\mbox{ \ if \ } \tau-s \le \rho^2
				\\
				\rho^{\frac 52}  (\tau-s)^{-\frac 74}
				&
				\mbox{ \ if \ } \rho^2 < \tau-s \le 1
				\\
				\rho^{\frac 52} (\tau-s)^{-\frac{\ell}{2}}
				&
				\mbox{ \ if \ } \tau -s > 1
			\end{cases}
			&
			\mbox{ \ for  \ } \ell<2
			\\
			\begin{cases}
				\rho^{-1}
				&
				\mbox{ \ if \ } \tau-s \le \rho^2
				\\
				\rho^{\frac 52} (\tau-s)^{-\frac 74}
				&
				\mbox{ \ if \ } \rho^2 < \tau-s \le 1
				\\
				\rho^{\frac 52} (\tau-s)^{-1} \langle \ln(a(\tau-s) ) \rangle
				&
				\mbox{ \ if \ } \tau -s > 1
			\end{cases}
			&
			\mbox{ \ for  \ } \ell=2
			\\
			\begin{cases}
				\rho^{-1}
				&
				\mbox{ \ if \ } \tau-s \le \rho^2
				\\
				\rho^{\frac 52} 	(\tau-s)^{-\frac 74}
				&
				\mbox{ \ if \ } \rho^2 < \tau-s \le 1
				\\
				\rho^{\frac 52} (\tau-s)^{-1}
				&
				\mbox{ \ if \ } \tau -s > 1
			\end{cases}
			&
			\mbox{ \ for  \ } \ell>2.
		\end{cases}
	\end{equation*}
	
	Next, let us estimate $P_2$. By \eqref{qd24Sep21-6}, $ P_{2} \lesssim
	\int_{\frac{1}{\rho^2}}^{\infty}
	e^{- a \xi (\tau -s)}
	\xi^{-\frac 14} \langle \xi \rangle^{-1}
	F(\xi)
	\langle \xi \rangle^{2}
	d\xi $.
	For $\rho\le 1$, by \eqref{F-est-DTF},
	\begin{equation*}
		P_{2} \lesssim
		\int_{\frac{1}{\rho^2}}^{\infty}
		e^{- a \xi (\tau -s)}
		\xi^{-\frac 12}
		d\xi
		\sim
		(\tau-s)^{-\frac 12}
		\int_{\frac{a(\tau-s)}{\rho^2}}^\infty
		e^{-z} z^{-\frac 12} dz
		\lesssim
		\begin{cases}
			(\tau-s)^{-\frac 12}
			&
			\mbox{ \ if \ } \tau-s \le \rho^2
			\\
			(\tau-s)^{-\frac 12}
			e^{- \frac{a(\tau-s)}{2\rho^2}}
			&
			\mbox{ \ if \ } \tau-s > \rho^2.
		\end{cases}
	\end{equation*}
	For $\rho> 1$, $ P_{2} \lesssim
	\big(
	\int_{1}^{\infty}
	+
	\int_{\frac{1}{\rho^2}}^{1}
	\big)
	e^{- a \xi (\tau -s)}
	\xi^{-\frac 14} \langle \xi \rangle^{-1}
	F(\xi)
	\langle \xi \rangle^{2}
	d\xi $.
	For the same reason as above, we have
	\begin{equation*}
		\int_{1}^{\infty}
		e^{- a \xi (\tau -s)}
		\xi^{-\frac 14} \langle \xi \rangle^{-1}
		F(\xi)
		\langle \xi \rangle^{2}
		d\xi
		\lesssim
		\begin{cases}
			(\tau-s)^{-\frac 12}
			&
			\mbox{ \ if \ } \tau-s \le 1
			\\
			(\tau-s)^{-\frac 12}
			e^{- \frac{a(\tau-s)}{2}}
			&
			\mbox{ \ if \ } \tau-s > 1.
		\end{cases}
	\end{equation*}
	By \eqref{F-est-DTF}, $\ell > 1/2$, and Lemma \ref{asy-lem-1},
\begin{align*}
			&
			\int_{\frac{1}{\rho^2}}^{1}
			e^{- a \xi (\tau -s)}
			\xi^{-\frac 14} \langle \xi \rangle^{-1}
			F(\xi)
			\langle \xi \rangle^{2}
			d\xi
			\lesssim
			\int_{\frac{1}{\rho^2}}^{1}
			e^{- a \xi (\tau -s)}
			\begin{cases}
				\xi^{\frac{\ell}{2} -\frac 54}
				&
				\mbox{ \ if  \ } \ell<2
				\\
				\xi^{-\frac 14} \langle \ln \xi \rangle
				&
				\mbox{ \ if  \ } \ell=2
				\\
				\xi^{-\frac 14}
				&
				\mbox{ \ if  \ } \ell>2
			\end{cases}
			d\xi
			\\
			\lesssim \ &
			\begin{cases}
				\begin{cases}
					1
					&
					\mbox{ \ if \ } \tau-s\le 1
					\\
					(\tau-s)^{\frac 14 - \frac{\ell}{2}}
					&
					\mbox{ \ if \ } 1<\tau-s\le \rho^2
					\\
					(\tau-s)^{\frac 14 - \frac{\ell}{2}}
					e^{-\frac{a(\tau-s)}{2\rho^2}}
					&
					\mbox{ \ if \ } \tau-s >\rho^2
				\end{cases}
				&
				\mbox{ \ for  \ } \ell<2
				\\
				\begin{cases}
					1
					&
					\mbox{ \ if \ } \tau-s\le 1
					\\
					(\tau-s)^{-\frac 34}
					\langle \ln (a(\tau-s)) \rangle
					&
					\mbox{ \ if \ } 1<\tau-s\le \rho^2
					\\
					(\tau-s)^{-\frac 34}
					\langle \ln (a(\tau-s)) \rangle
					e^{-\frac{a(\tau-s)}{2\rho^2}}
					&
					\mbox{ \ if \ } \tau-s >\rho^2
				\end{cases}
				&
				\mbox{ \ for  \ } \ell=2
				\\
				\begin{cases}
					1
					&
					\mbox{ \ if \ } \tau-s\le 1
					\\
					(\tau-s)^{-\frac 34}
					&
					\mbox{ \ if \ } 1<\tau-s\le \rho^2
					\\
					(\tau-s)^{-\frac 34}
					e^{-\frac{a(\tau-s)}{2\rho^2}}
					&
					\mbox{ \ if \ } \tau-s >\rho^2
				\end{cases}
				&
				\mbox{ \ for  \ } \ell>2.
			\end{cases}
		\end{align*}
	Thus, for $\rho>1$,
	\begin{equation*}
		P_2 \lesssim
		\begin{cases}
			\begin{cases}
				(\tau-s)^{-\frac 12}
				&
				\mbox{ \ if \ } \tau-s\le 1
				\\
				(\tau-s)^{\frac 14 - \frac{\ell}{2}}
				&
				\mbox{ \ if \ } 1<\tau-s\le \rho^2
				\\
				(\tau-s)^{\frac 14 - \frac{\ell}{2}}
				e^{-\frac{a(\tau-s)}{4\rho^2}}
				&
				\mbox{ \ if \ } \tau-s >\rho^2
			\end{cases}
			&
			\mbox{ \ for  \ } \ell<2
			\\
			\begin{cases}
				(\tau-s)^{-\frac 12}
				&
				\mbox{ \ if \ } \tau-s\le 1
				\\
				(\tau-s)^{-\frac 34}
				\langle \ln (a(\tau-s)) \rangle
				&
				\mbox{ \ if \ } 1<\tau-s\le \rho^2
				\\
				(\tau-s)^{-\frac 34}
				\langle \ln (a(\tau-s)) \rangle
				e^{-\frac{a(\tau-s)}{4 \rho^2}}
				&
				\mbox{ \ if \ } \tau-s >\rho^2
			\end{cases}
			&
			\mbox{ \ for  \ } \ell=2
			\\
			\begin{cases}
				(\tau-s)^{-\frac 12}
				&
				\mbox{ \ if \ } \tau-s\le 1
				\\
				(\tau-s)^{-\frac 34}
				&
				\mbox{ \ if \ } 1<\tau-s\le \rho^2
				\\
				(\tau-s)^{-\frac 34}
				e^{-\frac{a(\tau-s)}{4 \rho^2}}
				&
				\mbox{ \ if \ } \tau-s >\rho^2
			\end{cases}
			&
			\mbox{ \ for  \ } \ell>2.
		\end{cases}
	\end{equation*}

	Combining the estimates of $P_1$ and $P_2$, we have the following estimates of $P$. For $\rho\le 1$,
	\begin{equation*}
		P \lesssim
		\begin{cases}
			\begin{cases}
				(\tau-s)^{-\frac 12}
				&
				\mbox{ \ if \ } \tau-s \le \rho^2
				\\
				\rho^{\frac 52}  (\tau-s)^{-\frac 74}
				&
				\mbox{ \ if \ } \rho^2 < \tau-s \le 1
				\\
				\rho^{\frac 52} (\tau-s)^{-\frac{\ell}{2}}
				&
				\mbox{ \ if \ } \tau -s > 1
			\end{cases}
			&
			\mbox{ \ for  \ } \ell<2
			\\
			\begin{cases}
				(\tau-s)^{-\frac 12}
				&
				\mbox{ \ if \ } \tau-s \le \rho^2
				\\
				\rho^{\frac 52} (\tau-s)^{-\frac 74}
				&
				\mbox{ \ if \ } \rho^2 < \tau-s \le 1
				\\
				\rho^{\frac 52} (\tau-s)^{-1} \langle \ln(a(\tau-s) ) \rangle
				&
				\mbox{ \ if \ } \tau -s > 1
			\end{cases}
			&
			\mbox{ \ for  \ } \ell=2
			\\
			\begin{cases}
				(\tau-s)^{-\frac 12}
				&
				\mbox{ \ if \ } \tau-s \le \rho^2
				\\
				\rho^{\frac 52} 	(\tau-s)^{-\frac 74}
				&
				\mbox{ \ if \ } \rho^2 < \tau-s \le 1
				\\
				\rho^{\frac 52} (\tau-s)^{-1}
				&
				\mbox{ \ if \ } \tau -s > 1
			\end{cases}
			&
			\mbox{ \ for  \ } \ell>2.
		\end{cases}
	\end{equation*}
	For $\rho>1$ with $\ell \ge 1/2$,
	\begin{equation*}
		P \lesssim
		\begin{cases}
			\begin{cases}
				(\tau-s)^{-\frac 12}
				&
				\mbox{ \ if \ } \tau-s\le 1
				\\
				(\tau-s)^{\frac 14 - \frac{\ell}{2}}
				&
				\mbox{ \ if \ } 1<\tau-s\le \rho^2
				\\
				\rho^{\frac 12}
				(\tau-s)^{-\frac{\ell}{2}}
				&
				\mbox{ \ if \ } \tau-s >\rho^2
			\end{cases}
			&
			\mbox{ \ for  \ } \ell<2
			\\
			\begin{cases}
				(\tau-s)^{-\frac 12}
				&
				\mbox{ \ if \ } \tau-s\le 1
				\\
				(\tau-s)^{-\frac 34}
				\langle \ln (a(\tau-s)) \rangle
				&
				\mbox{ \ if \ } 1<\tau-s\le \rho^2
				\\
				\rho^{\frac 12}  (\tau-s)^{-1} \langle \ln(a(\tau-s) ) \rangle
				&
				\mbox{ \ if \ } \tau-s >\rho^2
			\end{cases}
			&
			\mbox{ \ for  \ } \ell=2
			\\
			\begin{cases}
				(\tau-s)^{-\frac 12}
				&
				\mbox{ \ if \ } \tau-s\le 1
				\\
				(\tau-s)^{-\frac 34}
				&
				\mbox{ \ if \ } 1<\tau-s\le \rho^2
				\\
				\rho^{\frac 12} (\tau-s)^{-1}
				&
				\mbox{ \ if \ } \tau-s >\rho^2
			\end{cases}
			&
			\mbox{ \ for  \ } \ell>2.
		\end{cases}
	\end{equation*}
	
	Now we will use the upper bound of $P$ to estimate $\phi_{-1}$. For $\rho\le 1$, since we assume $\tau_0\ge 2$, then
	\begin{align*}
		%\begin{equation*}
		%\begin{aligned}
		&
		|\phi_{-1}|
		\lesssim
		\rho^{-\frac 12}
		\Big(
		\int_{\tau-\rho^2}^{\tau}
		+
		\int_{\tau-1}^{\tau-\rho^2}
		+
		\int_{\frac{\tau_{0}}{2}}^{\tau-1}
		\Big)
		v(s) P(\rho,\tau,s) ds
		\\
		\lesssim \ &
		\rho^{-\frac 12} \bigg[ \tilde{v}(\tau) \int_{\tau-\rho^2}^{\tau}
		(\tau-s)^{-\frac 12} ds
		+
		\tilde{v}(\tau) \rho^{\frac 52}\int_{\tau-1}^{\tau-\rho^2}
		(\tau-s)^{-\frac 74} ds
		\\
		&
		+
		\rho^{\frac 52}
		\begin{cases}
			\int_{\frac{\tau_{0}}{2}}^{\tau-1}
			v(s) (\tau-s)^{-\frac{\ell}{2}} ds
			&
			\mbox{ \ if \ } \ell <2
			\\
			\int_{\frac{\tau_{0}}{2}}^{\tau-1}
			v(s) (\tau-s)^{-1} \langle \ln(a(\tau-s) ) \rangle  ds
			&
			\mbox{ \ if \ } \ell =2
			\\
			\int_{\frac{\tau_{0}}{2}}^{\tau-1}
			v(s) (\tau-s)^{-1} ds
			&
			\mbox{ \ if \ } \ell >2
		\end{cases}
		\bigg]
		\\
		\lesssim \ &
		\tilde{v}(\tau) \rho^{\frac 12}
		+
		\rho^{2}
		\begin{cases}
			\tilde{v}(\tau) \tau^{1-\frac{\ell}{2}}
			+
			\tau^{-\frac{\ell}{2}}
			\int_{\frac{\tau_{0}}{2}}^{ \frac{\tau}{2} }
			v(s)  ds
			&
			\mbox{ \ if \ } \ell <2
			\\
			\tilde{v}(\tau) (\ln \tau)^2
			+
			\tau^{-1} \ln \tau
			\int_{\frac{\tau_{0}}{2}}^{\frac{\tau}{2}}
			v(s)   ds
			&
			\mbox{ \ if \ } \ell =2
			\\
			\tilde{v}(\tau) \ln \tau
			+
			\tau^{-1}
			\int_{\frac{\tau_{0}}{2}}^{ \frac{\tau}{2} } v(s) ds
			&
			\mbox{ \ if \ } \ell >2,
		\end{cases}
		%\end{aligned}
		%\end{equation*}
	\end{align*}
	where we denote $\tilde{v}(\tau):=\sup\limits_{s\in [\tau/2,\tau]} v(s)$.
	For $1<\rho\le (\tau/2)^{1/2} $,
	\begin{align*}
		%\begin{equation*}
		%	\begin{aligned}
			&
			|\phi_{-1}|
			\lesssim
			\rho^{-\frac 12}
			\Big(
			\int_{\tau -1}^{\tau}
			+	
			\int_{\tau -\rho^2}^{\tau-1}
			+
			\int_{\frac{\tau_0}{2}}^{\tau -\rho^2}
			\Big)
			v(s) P(\rho, \tau, s) ds
			\\
			\lesssim \ &
			\rho^{-\frac 12}
			\tilde{v}(\tau)
			+	
			\tilde{v}(\tau)
			\begin{cases}
				\rho^{2 -\ell}
				&
				\mbox{ \ if \ }\ell <2
				\\
				\langle \ln \rho \rangle
				&
				\mbox{ \ if \ }\ell =2
				\\
				1
				&
				\mbox{ \ if \ }\ell >2
			\end{cases}
			+
			\begin{cases}
				\tilde{v}(\tau) \tau^{1-\frac{\ell}{2}}
				+
				\tau^{-\frac{\ell}{2}}
				\int_{\frac{\tau_0}{2}}^{\frac{\tau}{2}}
				v(s) ds
				&
				\mbox{ \ if  \ } \ell <2
				\\
				\tilde{v}(\tau)
				\int_{\rho^2}^{\frac{\tau}{2}} \langle \ln z \rangle z^{-1} dz
				+
				\tau^{-1} \ln \tau
				\int_{\frac{\tau_0}{2}}^{\frac{\tau}{2}}
				v(s) ds
				&
				\mbox{ \ if  \ } \ell =2
				\\
				\tilde{v}(\tau) \ln (\frac{\tau}{2\rho^2} )
				+
				\tau^{-1}
				\int_{\frac{\tau_0}{2}}^{\frac{\tau}{2}}
				v(s) ds
				&
				\mbox{ \ if  \ } \ell >2
			\end{cases}
			\\
			\lesssim \ &
			\begin{cases}
				\tilde{v}(\tau) \tau^{1-\frac{\ell}{2}}
				+
				\tau^{-\frac{\ell}{2}}
				\int_{\frac{\tau_0}{2}}^{\frac{\tau}{2}}
				v(s) ds
				&
				\mbox{ \ if  \ } \ell <2
				\\
				\tilde{v}(\tau) (\langle \ln \rho \rangle +
				\int_{\rho^2}^{\frac{\tau}{2}} \langle \ln z \rangle z^{-1} dz)
				+
				\tau^{-1} \ln \tau
				\int_{\frac{\tau_0}{2}}^{\frac{\tau}{2}}
				v(s) ds
				&
				\mbox{ \ if  \ } \ell =2
				\\
				\tilde{v}(\tau) \langle \ln (\frac{\tau}{2\rho^2} ) \rangle
				+
				\tau^{-1}
				\int_{\frac{\tau_0}{2}}^{\frac{\tau}{2}}
				v(s) ds
				&
				\mbox{ \ if  \ } \ell >2.
			\end{cases}
			%	\end{aligned}
		%\end{equation*}
	\end{align*}
	For $(\tau/2)^{\frac 12} \le \rho \le \tau^{\frac 12}$,
	\begin{align*}
		%\begin{equation*}
		%	\begin{aligned}
			&
			|\phi_{-1}|
			\lesssim
			\rho^{-\frac 12}
			\Big[
			\int_{\tau -1}^{\tau}
			+	
			\Big(
			\int_{\frac{\tau}{2}}^{\tau-1}
			+
			\int_{\tau -\rho^2}^{\frac{\tau}{2}}
			\Big)
			+
			\int_{\frac{\tau_0}{2}}^{\tau -\rho^2}
			\Big]
			v(s) P(\rho,\tau,s) ds
			\\
			\lesssim \ &
			\rho^{-\frac 12}
			\Big[
			\tilde{v}(\tau)
			+	
			\tilde{v}(\tau)
			\begin{cases}
				\tau^{\frac 54 - \frac{\ell}{2}}
				&
				\mbox{ \ if \ }\ell <2
				\\
				\tau^{\frac 14}
				\langle \ln\tau \rangle
				&
				\mbox{ \ if \ }\ell =2
				\\
				\tau^{\frac 14}
				&
				\mbox{ \ if \ }\ell >2
			\end{cases}
			+
			\int_{\tau -\rho^2}^{\frac{\tau}{2}}
			v(s) ds
			\begin{cases}
				\tau^{\frac 14 - \frac{\ell}{2}}
				&
				\mbox{ \ if \ }\ell <2
				\\
				\tau^{-\frac 34}
				\langle \ln \tau \rangle
				&
				\mbox{ \ if \ }\ell =2
				\\
				\tau^{-\frac 34}
				&
				\mbox{ \ if \ }\ell >2
			\end{cases}
			\\
			&
			+
			\rho^{\frac 12}
			\int_{\frac{\tau_0}{2}}^{\tau -\rho^2}
			v(s) ds
			\begin{cases}
				\tau^{-\frac{\ell}{2}}
				&
				\mbox{ \ if  \ } \ell <2
				\\
				\tau^{-1} \langle \ln\tau  \rangle
				&
				\mbox{ \ if  \ } \ell =2
				\\
				\tau^{-1}
				&
				\mbox{ \ if  \ } \ell >2
			\end{cases}
			\Big]
			\lesssim
			\begin{cases}
				\tau^{1 - \frac{\ell}{2}} \tilde{v}(\tau) +
				\tau^{ - \frac{\ell}{2}} \int_{ \frac{\tau_0}{2} }^{\frac{\tau}{2}}
				v(s) ds
				&
				\mbox{ \ if \ }\ell <2
				\\
				\langle \ln\tau \rangle \tilde{v}(\tau)+
				\tau^{-1}
				\langle \ln \tau \rangle \int_{ \frac{\tau_0}{2} }^{\frac{\tau}{2}}
				v(s) ds
				&
				\mbox{ \ if \ }\ell =2
				\\
				\tilde{v}(\tau)
				+
				\tau^{-1} \int_{ \frac{\tau_0}{2} }^{\frac{\tau}{2}}
				v(s) ds
				&
				\mbox{ \ if \ }\ell >2.
			\end{cases}
			%	\end{aligned}
		%\end{equation*}
	\end{align*}
	For $\rho\ge \tau^{\frac 12}$,
	\begin{small}
		\begin{equation*}
			|\phi_{-1}|
			\lesssim
			\rho^{-\frac 12}
			\Big(
			\int_{\tau -1}^{\tau}
			+	
			\int_{\frac{\tau_0}{2}}^{\tau-1}
			\Big)
			v(s) P(\rho,\tau,s) ds
			\lesssim
			\tilde{v}(\tau) \rho^{-\frac 12}
			+
			\rho^{-\frac 12}
			\begin{cases}
				\tilde{v}(\tau) \tau^{\frac 54-\frac{\ell}{2}}
				+
				\tau^{\frac 14 - \frac{\ell}{2}}
				\int_{\frac{\tau_0}{2}}^{ \frac{\tau}{2} }
				v(s)  ds
				&
				\mbox{if  \ }\ell<2
				\\
				\tilde{v}(\tau) \tau^{\frac 14} \langle \ln \tau \rangle
				+
				\tau^{-\frac 34}
				\langle \ln \tau \rangle
				\int_{\frac{\tau_0}{2}}^{ \frac{\tau}{2} }
				v(s) ds
				&
				\mbox{if  \ }\ell=2
				\\
				\tilde{v}(\tau) \tau^{\frac 14 }
				+
				\tau^{-\frac 34 }
				\int_{\frac{\tau_0}{2}}^{ \frac{\tau}{2} }
				v(s)  ds
				&
				\mbox{if  \ }\ell>2.
			\end{cases}
		\end{equation*}
	\end{small}
	
	In sum, we have proved that \eqref{phi-n-Duhamel} is absolutely integrable and \eqref{phi-1-nonorth} holds.
	
	\medskip
	
\noindent	\textbf{Estimate with orthogonality.} Recalling the estimates of $|\phi_{-1}|$ in four cases above, we have
	\begin{equation}\label{phi-1-1}
		\rho^{-\frac 12}
		\Big| \int_{\tau- 1}^{\tau}
		\int_{0}^{\infty}
		\int_{0}^{\infty}
		e^{-(a-ib)\xi (\tau -s)}
		\Phi^{-1} (\rho,\xi)
		\Phi^{-1}(x,\xi)
		x^{\frac 12}
		h(x,s)
		\rho_{-1}(d \xi)
		d x d s \Big| \lesssim \tilde{v}(\tau)
		\Big(\rho^{\frac 12} \1_{\{\rho\le 1\}} + \rho^{-\frac 12} \1_{\{\rho > 1\}} \Big).
	\end{equation}
	For the other part, we denote
	\begin{equation*}
		\tilde{\phi}_{-1} :=
		\rho^{-\frac 12}
		\Big| \int_{\tau_{0}}^{\tau- 1}
		\int_{0}^{\infty}
		\int_{0}^{\infty}
		e^{-(a-ib)\xi (\tau -s)}
		\Phi^{-1} (\rho,\xi)
		\Phi^{-1}(x,\xi)
		x^{\frac 12}
		h(x,s)
		\rho_{-1}(d \xi)
		d x d s	\Big|.
	\end{equation*}

	By the orthogonality condition \eqref{h-1-ortho}, we have
	\begin{equation*}
		\tilde{F}(\xi,s):=   \Big|\int_{0}^{\infty}
		\Phi^{-1}(x,\xi)
		x^{\frac 12}
		h(x,s)
		d x \Big|
		=
		\Big| \Big(\int_{0}^{\xi^{-\frac 12}}
		+
		\int_{\xi^{-\frac 12}}^{\infty} \Big)
		\Big( \Phi^{-1}(x,\xi)
		-
		\frac{x^{\frac 52}}{1+x^2} \Big)
		x^{\frac 12}
		h(x,s)
		d x \Big|.
	\end{equation*}
	
	We will use Proposition \ref{Phi-1 estimate} repetitively. Firstly, for $\ell <4$, we have
	\begin{equation*}
		\Big| \int_{0}^{\xi^{-\frac 12}}
		\Big( \Phi^{-1}(x,\xi)
		-
		\frac{x^{\frac 52}}{1+x^2} \Big)
		x^{\frac 12}
		h(x,s)
		d x \Big|
		\lesssim
		v(s)\int_{0}^{\xi^{-\frac 12}}
		\frac{x^{\frac 52}}{1+x^2}
		x^2\xi
		x^{\frac 12}  \langle x \rangle^{-\ell}
		d x
		\lesssim
		v(s) \big( \xi^{\frac{\ell}{2} -1} \1_{\{ \xi \le 1 \}} + \xi^{-2} \1_{\{ \xi > 1 \}} \big).
	\end{equation*}
	
	Secondly,
	\begin{equation*}
		\Big| \int_{\xi^{-\frac 12}}^{\infty}
		\Phi^{-1}(x,\xi)
		x^{\frac 12}
		h(x,s)
		d x \Big|
		\lesssim
		v(s)
		\xi^{-\frac 14} \langle \xi \rangle^{-1}
		\int_{\xi^{-\frac 12}}^{\infty}
		x^{\frac 12}
		\langle x\rangle^{-\ell}
		d x
		\sim
		v(s) \big( \xi^{\frac {\ell}2 -1}  \1_{\{ \xi \le 1 \}} + \xi^{ -\frac 54} \1_{\{ \xi > 1 \}} \big),
	\end{equation*}
	where we require $\ell> 3/2$ to guarantee the integrability.
	
	Thirdly, by \eqref{h-1-ortho} and $\ell>2$, we have $\int_{\xi^{-\frac 12}}^{\infty}
	\frac{x^{\frac 52}}{1+x^2}
	x^{\frac 12}
	h(x,s)
	d x
	=
	-
	\int_{0}^{\xi^{-\frac 12}}
	\frac{x^{\frac 52}}{1+x^2}
	x^{\frac 12}
	h(x,s)
	d x$. Then
	\begin{equation*}
		\Big|
		\int_{\xi^{-\frac 12}}^{\infty}
		\frac{x^{\frac 52}}{1+x^2}
		x^{\frac 12}
		h(x,s)
		d x
		\Big|
		\lesssim
		v(s) \big( \xi^{\frac{\ell}{2} -1} \1_{\{\xi \le 1\}}  + \xi^{ -2} \1_{\{ \xi > 1 \}} \big),
	\end{equation*}
	and thus
	\begin{equation}\label{tildeF-est-DFT}
		\tilde{F}(\xi,s) \lesssim
		v(s)
		(
		\xi^{\frac {\ell}2 -1} \1_{\{\xi \le 1 \}}
		+
		\xi^{ -\frac 54} \1_{\{\xi > 1 \}}
		).
	\end{equation}
	
	Next, using $ \frac{d \rho_{-1}(\xi)}{d \xi}
	\sim \langle \xi \rangle^2$ on $\xi\ge 0$ in Proposition \ref{Phi-1 estimate}, we will estimate
	\begin{equation*}
		\tilde{P}(\rho,\tau,s)
		:=
		\int_{0}^{\infty}
		e^{- a \xi (\tau -s)}
		|\Phi^{-1} (\rho,\xi) |
		\tilde{F}(\xi,s)
		\langle \xi \rangle^2
		d \xi
		=
		\int_{0}^{\frac{1}{\rho^2}} + 	\int_{\frac{1}{\rho^2}}^{\infty}
		\cdots
		:= \tilde{P}_1 +\tilde{P}_2.
	\end{equation*}
	
	Let us estimate $\tilde{P}_{1}$. For $\rho\ge 1$, by Proposition \ref{Phi-1 estimate}, \eqref{tildeF-est-DFT}, $\ell>0$, and Lemma \ref{asy-lem-1},
	\begin{equation*}
		\begin{aligned}
			\tilde{P}_1 \lesssim \ &
			v(s)
			\int_{0}^{\frac{1}{\rho^2}}
			e^{- a \xi (\tau -s)}
			\rho^{\frac 52} \langle \rho \rangle^{-2}
			\xi^{\frac {\ell}2 -1}
			\langle \xi \rangle^2
			d \xi
			\sim
			v(s) \rho^{\frac 12}
			\int_{0}^{\frac{1}{\rho^2}}
			e^{- a \xi (\tau -s)}
			\xi^{\frac {\ell}2 -1}
			d \xi
			\\
			\lesssim \ &
			v(s)
			\big[
			\rho^{\frac 12-\ell} \1_{\{ \tau-s\le \rho^2 \}}
			+
			\rho^{\frac 12}(\tau-s)^{-\frac{\ell}{2}}
			\1_{\{ \tau-s > \rho^2 \}}
			\big].
		\end{aligned}
	\end{equation*}
	For $\rho<1$,
	\begin{equation*}
		\begin{aligned}
			& \tilde{P}_{1} \lesssim
			\int_{0}^{\frac{1}{\rho^2}}
			e^{- a \xi (\tau -s)}
			\rho^{\frac 52} \langle \rho \rangle^{-2}
			\tilde{F}(\xi,s)
			\langle \xi \rangle^2
			d \xi
			\sim
			\rho^{\frac 52}
			\Big(\int_{0}^{1}  + \int_{1}^{\frac{1}{\rho^2}} \Big)
			e^{- a \xi (\tau -s)}
			\tilde{F}(\xi,s)
			\langle \xi \rangle^2
			d \xi
			\\
			\lesssim  \ &
			v(s) \rho^{\frac 52}
			\Big( \int_{0}^{1} e^{- a \xi (\tau -s)}
			\xi^{\frac{\ell}{2} -1 }
			d \xi  + \int_{1}^{\frac{1}{\rho^2}} e^{- a \xi (\tau -s)}
			\xi^{\frac{3}{4}}
			d \xi \Big)
			\lesssim
			v(s) \rho^{\frac 52}
			\begin{cases}
				\rho^{-\frac 72}
				&
				\mbox{ \ if \ } \tau-s \le \rho^2
				\\
				(\tau-s)^{-\frac 74}
				&
				\mbox{ \ if \ } \rho^2 <\tau-s \le 1
				\\
				(\tau-s)^{-\frac{\ell}{2}}
				&
				\mbox{ \ if \ } \tau-s > 1
			\end{cases}
		\end{aligned}
	\end{equation*}
	since
	$
	\int_{0}^{1} e^{- a \xi (\tau -s)}
	\xi^{\frac{\ell}{2} -1 }
	d \xi
	\lesssim
	\1_{\{ \tau-s\le 1 \}}
	+
	(\tau-s)^{-\frac{\ell}{2}} \1_{\{ \tau-s > 1 \}}
	$
	by $\ell>0$ and Lemma \ref{asy-lem-1}, and
	\begin{equation*}
		\int_{1}^{\frac{1}{\rho^2}} e^{- a \xi (\tau -s)}
		\xi^{\frac{3}{4}}
		d \xi
		\sim
		(\tau-s)^{-\frac 74} \int_{a(\tau-s)}^{\frac{a(\tau-s)}{\rho^2}} e^{-z} z^{\frac 34} dz
		\lesssim
		\begin{cases}
			\rho^{-\frac 72}
			&
			\mbox{ \ if \ } \tau-s \le \rho^2
			\\
			(\tau-s)^{-\frac 74}
			&
			\mbox{ \ if \ } \rho^2 <\tau-s \le 1
			\\
			(\tau-s)^{-\frac 74} e^{-\frac{a(\tau-s)}{2}}
			&
			\mbox{ \ if \ } \tau-s > 1.
		\end{cases}
	\end{equation*}
	
	Next, we will use Proposition \ref{Phi-1 estimate} and \eqref{tildeF-est-DFT} to estimate $\tilde{P}_2$. For $\rho\le 1$,
	\begin{equation*}
		\begin{aligned}
			\tilde{P}_{2} \lesssim \ &
			v(s)
			\int_{\frac{1}{\rho^2}}^{\infty}
			e^{- a \xi (\tau -s)}
			\xi^{-\frac 14 }\langle \xi\rangle^{-1}
			\xi^{ -\frac 54}
			\langle \xi \rangle^2
			d \xi
			\sim
			v(s)
			\int_{\frac{1}{\rho^2}}^{\infty}
			e^{- a \xi (\tau -s)}  \xi^{-\frac 12} d \xi
			\\
			\sim \ &
			v(s) (\tau-s)^{-\frac 12}
			\int_{\frac{a(\tau-s)}{\rho^2}}^{\infty}
			e^{-z} z^{-\frac 12} dz
			\lesssim
			v(s) (\tau-s)^{-\frac 12}
			\Big(
			\1_{\{ \tau-s\le \rho^2 \}}
			+
			e^{-\frac{a(\tau-s)}{2\rho^2}}
			\1_{\{ \tau-s > \rho^2 \}}
			\Big).
		\end{aligned}
	\end{equation*}
	For $\rho > 1$,
\begin{align*}
			\tilde{P}_{2} \lesssim \ &
			\Big( \int_{\frac{1}{\rho^2}}^{1}
			+	\int_{1}^{\infty} \Big)
			e^{- a \xi (\tau -s)}
			\xi^{-\frac 14 }\langle \xi\rangle^{-1}
			\tilde{F}(\xi,s)
			\langle \xi \rangle^2
			d \xi
			\lesssim
			v(s)\Big(
			\int_{\frac{1}{\rho^2}}^{1}
			e^{- a \xi (\tau -s)}
			\xi^{\frac{\ell}{2}-\frac 54 }
			d \xi
			+
			\int_{1}^{\infty}
			e^{- a \xi (\tau -s)}
			\xi^{-\frac 12}
			d \xi \Big)
			\\
			\lesssim \ &
			v(s)\Big[
			\1_{\{ \tau-s \le 1 \}}
			+
			(\tau-s)^{\frac 14-\frac{\ell}{2}}
			\1_{\{ 1 < \tau-s \le \rho^2 \}}
			+
			(\tau-s)^{\frac 14-\frac{\ell}{2}}
			e^{-\frac{a(\tau-s)}{2\rho^2}}
			\1_{\{ \tau-s > \rho^2 \}}
			\\
			&
			+
			(\tau-s)^{-\frac 12} \1_{\{ \tau-s\le 1 \}}
			+
			(\tau-s)^{-\frac 12}
			e^{-\frac{a(\tau-s)}{2 }} \1_{\{ \tau-s > 1 \}} \Big]
			\\
			\lesssim \ &
			v(s)
			\Big[
			(\tau-s)^{-\frac 12} \1_{\{ \tau-s \le 1 \}}
			+
			(\tau-s)^{\frac 14-\frac{\ell}{2}}
			\1_{\{ 1 < \tau-s \le \rho^2 \}}
			+
			(\tau-s)^{\frac 14-\frac{\ell}{2}}
			e^{-\frac{a(\tau-s)}{4\rho^2}}
			\1_{\{ \tau-s > \rho^2 \}}
			\Big],
		\end{align*}
	where we used $\ell>1/2$ and Lemma \ref{asy-lem-1} to estimate $\int_{1/\rho^2}^{1}
	e^{- a \xi (\tau -s)}
	\xi^{\frac{\ell}{2}-\frac 54 }
	d \xi$. Combining the estimates of $\tilde{P}_1$, $\tilde{P}_2$, since $\ell\ge 1/2$, we have
	\begin{equation*}
		\tilde{P}(\rho,\tau,s)\lesssim
		\1_{\{ \rho\le 1 \}}
		v(s)
		\begin{cases}
			(\tau-s)^{-\frac 12}
			&
			\mbox{if \ } \tau-s \le \rho^2
			\\
			\rho^{\frac 52} (\tau-s)^{-\frac 74}
			&
			\mbox{if \ } \rho^2 <\tau-s \le 1
			\\
			\rho^{\frac 52} (\tau-s)^{-\frac{\ell}{2}}
			&
			\mbox{if \ } \tau-s > 1
		\end{cases}
		\
		+
		\
		\1_{\{ \rho>1 \}}
		v(s)
		\begin{cases}
			(\tau-s)^{-\frac 12}
			&
			\mbox{if \ } \tau-s \le 1
			\\
			(\tau-s)^{\frac 14-\frac{\ell}{2}}
			&
			\mbox{if \ } 1 < \tau-s \le \rho^2
			\\
			\rho^{\frac 12}(\tau-s)^{-\frac{\ell}{2}}
			&
			\mbox{if \ }  \tau-s > \rho^2.
		\end{cases}
	\end{equation*}
	
	Finally, we will estimate $\tilde{\phi}_{-1}$. Obviously,
	$
	\tilde{\phi}_{-1} \lesssim
	\rho^{-\frac 12} \int_{\frac{\tau_0}{2}}^{\tau-1}
	\tilde{P}(\rho,\tau,s) ds$.
	For $\rho\le 1$, since $\ell>2$,
	\begin{equation*}
		\tilde{\phi}_{-1} \lesssim
		\rho^{-\frac 12} \int_{\frac{\tau_0}{2}}^{\tau-1}
		v(s) \rho^{\frac 52} (\tau-s)^{-\frac{\ell}{2}} ds
		\lesssim
		\rho^2 \Big( \tilde{v}(\tau) + \tau^{-\frac{\ell}{2}}\int_{\frac{\tau_0}{2}}^{\frac{\tau}{2}} v(s) ds
		\Big).
	\end{equation*}
	For $1<\rho \le (\tau/2)^{\frac 12}$,
	$2< \ell< 5/2$,
	\begin{align*}
%	\begin{equation*}
%		\begin{aligned}
			&	\tilde{\phi}_{-1}
			\lesssim
			\rho^{-\frac 12} \Big(
			\int_{\tau-\rho^2 }^{\tau-1}
			+
			\int_{\frac{\tau_0}{2}}^{ \tau-\rho^2 }
			\Big)
			\tilde{P}(\rho,\tau,s) ds
			\\
			\lesssim \ &
			\rho^{-\frac 12} \Big[
			\tilde{v}(\tau)
			\int_{ \tau-\rho^2  }^{\tau-1}
			(\tau-s)^{\frac 14-\frac{\ell}{2}}
			ds
			+
			\Big(
			\int_{\frac{\tau }{2}}^{ \tau-\rho^2 }
			+
			\int_{\frac{\tau_0 }{2}}^{\frac{\tau }{2}} \Big)
			v(s)
			\rho^{\frac 12}(\tau-s)^{-\frac{\ell}{2}}
			ds \Big]
			\lesssim
			\tilde{v}(\tau)
			\rho^{2 - \ell }
			+
			\tau^{-\frac{\ell}{2}}
			\int_{\frac{\tau_0 }{2}}^{\frac{\tau }{2}}
			v(s) ds.
%		\end{aligned}
%	\end{equation*}
	\end{align*}
	For $ (\tau/2)^{\frac 12} < \rho \le \tau^{\frac 12}$, $\ell<5/2$,
	\begin{equation*}
		\tilde{\phi}_{-1}
		\lesssim
		\rho^{-\frac 12} \Big[
		\Big(
		\int_{\frac{\tau}{2}}^{\tau-1}
		+
		\int_{\tau-\rho^2 }^{\frac{\tau}{2}}
		\Big)
		+
		\int_{\frac{\tau_0}{2}}^{ \tau-\rho^2 }
		\Big]
		\tilde{P}(\rho,\tau,s) ds
		\lesssim
		\tilde{v}(\tau)
		\rho^{2 - \ell }
		+
		\tau^{-\frac{\ell}{2}}
		\int_{\frac{\tau_0 }{2}}^{\frac{\tau }{2}}
		v(s) ds.
	\end{equation*}
	For $\rho>\tau^{\frac 12}$, $\ell< 5/2$,
	\begin{equation*}
		\tilde{\phi}_{-1} \lesssim
		\rho^{-\frac 12} \int_{\frac{\tau_0}{2}}^{\tau-1}
		v(s) (\tau-s)^{\frac 14-\frac{\ell}{2}} ds
		\lesssim
		\rho^{-\frac 12}
		\Big(
		\tilde{v}(\tau) \tau^{\frac 54 -\frac{\ell}{2}}
		+
		\tau^{\frac 14-\frac{\ell}{2}} \int_{\frac{\tau_0}{2}}^{\frac{\tau}{2}}
		v(s)  ds
		\Big).
	\end{equation*}
	
	Combining \eqref{phi-1-1} with all the estimates of $\tilde{\phi}_{-1}$ above, we conclude the estimate \eqref{phi-1-orth}.
\end{proof}

\appendix

\section{Pointwise estimates for heat and Laplace equations}

Recall {\it algebraic power type} ($\mathbf{AP}$) defined at the very beginning of Section \ref{sec-linearinner}.

\begin{lemma}\label{qd24July10-2-lem}
	
	Suppose $n>2$, $v(t), l_1(t), l_2(t) \in \mathbf{AP}$, $b\in \RR$, $l_1(t) \le l_2(t) \le C_{*} t^{\frac{1}{2}}$ with a constant $C_*>0$,
	\begin{equation*}
		\begin{cases}
			\mathbf{P}_1[t^{\frac n2}
			v(t)
			l_2^{2-b}(t)] >0 & \mbox{ \ if \ } b\le n
			\\
			\mathbf{P}_1[t^{\frac n2} v(t)
			l_2^{2-n}(t)	l_1^{n-b}(t)] >0
			& \mbox{ \ if \ }
			b>n,
		\end{cases}
	\end{equation*}
	given constants $c_1, c_2>0$ and $f(x,t) =  v(t) |x|^{-b} \1_{\{ l_1(t) \le |x| \le l_2(t) \}}$ or $v(t) (|x| + l_1(t))^{-b} \1_{\{ |x| \le l_2(t) \}}$, then for all $|x|\le l_2(t)$, we have
\begin{align*}
			&
			\int_{t_0}^t \int_{\mathbb{R}^n} (t-s)^{-\frac{n}{2}} e^{-c_1\big(\frac{|x-y|^2}{t-s}\big)^{c_2}} f(y,s) dy ds
			\\
			\lesssim \ & v(t)
			\begin{cases}
				\begin{cases}
					l_2^{2-b}(t)
					& \mbox{ \ if \ }
					b<2
					\\
					\langle	\ln (\frac{l_2(t)}{l_1(t)}) \rangle
					& \mbox{ \ if \ }
					b=2
					\\
					l_1^{2-b}(t)
					& \mbox{ \ if \ }
					b>2
				\end{cases}
				&
				\mbox{ \ for \ }  |x| \le l_1(t)
				\\
				\begin{cases}
					l_2^{2-b}(t)
					&
					\mbox{ \ if \ } b<2
					\\
					\langle\ln(\frac{l_2(t)}{|x|} )  \rangle
					&
					\mbox{ \ if \ } b=2
					\\
					|x|^{2-b}
					&
					\mbox{ \ if \ } 2<b<n
					\\
					|x|^{2-n} \langle	\ln (\frac{|x|}{l_1(t)}) \rangle
					&
					\mbox{ \ if \ } b=n
					\\
					|x|^{2-n}	l_1^{n-b}(t)
					&
					\mbox{ \ if \ } b>n
				\end{cases}
				&
				\mbox{ \ for \ }
				l_1(t) < |x| \le l_2(t),
			\end{cases}
		\end{align*}
	where ``$\lesssim$'' is independent of $t_0$.
\end{lemma}

\begin{lemma}\label{qd24July11-1-lem}
	
	Suppose $2<b<n$, $v(t), l_1(t), l_2(t) \in \mathbf{AP}$, $l_1(t) \le l_2(t) \le C_{*} t^{\frac{1}{2}}$ with a constant $C_*>0$, $\frac n2 -\frac{b}{2} +1+ \mathbf{P}_1[v(t)] >0$,
	given constants $c_1, c_2>0$ and $f(x,t) =  v(t) |x|^{-b} \1_{\{ l_1(t) \le |x| \le l_2(t) \}}$ or $v(t) (|x| + l_1(t))^{-b} \1_{\{ |x| \le l_2(t) \}}$, then for all $(x,t) \in \mathbb{R}^n \times (t_0,\infty)$, we have
	\begin{equation*}
		\int_{t_0}^t \int_{\mathbb{R}^n} (t-s)^{-\frac{n}{2}} e^{-c_1\big(\frac{|x-y|^2}{t-s}\big)^{c_2}} f(y,s) dy ds
		\lesssim v(t)
		\big(
		l_1^{2-b}(t) \1_{\{ |x| \le l_1(t) \}}
		+
		|x|^{2-b} \1_{\{ |x|>l_1(t) \}}
		\big),
	\end{equation*}
	where ``$\lesssim$'' is independent of $t_0$.
\end{lemma}

\begin{proof}[Proof of Lemmas \ref{qd24July10-2-lem} and \ref{qd24July11-1-lem}]
	The proof is a direct application of the analog of \cite[Lemma A.1]{infi4d}. We omit details.
\end{proof}

The estimates in Lemmas \ref{qd24July10-2-lem} and \ref{qd24July11-1-lem} do not show the dependence on parameters clearly.
To get estimates with precise dependence on $k$ in the linear theory of mode $k$, $|k|\ge 2$ in Subsection \ref{hmode-subsec}, we need the following lemma.
\begin{lemma}\label{ellptic-compa}
	Consider $ -\Delta u = f(x) $ in $\mathbb{R}^n\backslash \{0\}$,
	where $n\ge 3$, $f(x) = f(|x|)$ is radial with the upper bound $|f(x)|\lesssim |x|^{-l_{1}} \1_{\{|x|\le 1\}} + |x|^{-l} \1_{\{|x| > 1\}} $, $l_1<n$, $l>2$. $u$ is given by
	\begin{equation}\label{qd240808-2}
		u(x ) = \frac{1}{(n-2)|S^{n-1}| } \int_{\RR^n} |x-y|^{2-n} f(y) dy,
	\end{equation}	
	where $|S^{n-1}|$ is the volume of the unit sphere $S^{n-1}$. Then
	\begin{equation}\label{qd240808-1}
		u(x) = u(|x|) =	|x|^{2-n} \int_{0}^{|x|} a^{n-3}
		\int_{a}^{\infty} b f(b) db da ,
		\quad
		\pp_{|x|}u =
		-|x|^{1-n} \int_0^{|x|} f(a) a^{n-1} da.
	\end{equation}
	When $f(x) = |x|^{-l_{1}} \1_{\{|x|\le 1\}} + |x|^{-l} \1_{\{|x|> 1\}} $, we have
	\begin{equation}\label{pp|x|-u}
		\pp_{|x|}u = \frac{-|x|^{1-l_1}}{n-l_1} \1_{\{ |x|\le 1 \}}
		-
		\bigg(
		\frac{|x|^{1-n}}{n-l_1} +
		\begin{cases}
			\frac{|x|^{1-l} - |x|^{1-n}}{n-l}
			& \mbox{ \ if \ } 2<l<n
			\\
			|x|^{1-n} \ln|x|
			& \mbox{ \ if \ } l=n
			\\
			\frac{ |x|^{1-n} - |x|^{1-l} }{l-n}
			& \mbox{ \ if \ } l>n
		\end{cases}
		\bigg)
		\1_{\{ |x|> 1 \}};
	\end{equation}
	for $|x|\le 1$,
	\begin{equation*}
		u(x)	=
		\frac{1}{(l-2)(n-2)}  +
		\begin{cases}
			\frac{1}{(2-l_{1})(n-2)}
			-
			\frac{ |x|^{2-l_1} }{(2-l_1)(n-l_1)}
			&
			\mbox{ \ if \ } l_{1} <2
			\\
			\frac{-\ln |x|}{n-2}
			+
			\frac{1}{(n-2)^2}
			&
			\mbox{ \ if \ } l_{1} =2
			\\
			\frac{ |x|^{2-l_{1} } }{(l_{1} -2)(n-l_{1})}
			- \frac{1}{(l_1-2)(n-2)}
			&
			\mbox{ \ if \ } 2<l_{1} <n ,
		\end{cases}
	\end{equation*}
	for $|x| \ge 1$,
\begin{align*}
			& 	u(x)
			=
			\frac{	|x|^{2-n} }{n-2}\Big(\frac{1}{l-2} + \frac{1}{n-l_1} \Big)
			+
			\begin{cases}
				\frac{|x|^{2-l}-	|x|^{2-n} }{(l-2)(n-l)}
				&
				\mbox{ \ if \ } 2<l<n
				\\
				\frac{	|x|^{2-n}  \ln |x|}{n-2}
				&
				\mbox{ \ if \ } l=n
				\\
				\frac{	|x|^{2-n} -|x|^{2-l}}{(l-2)(l-n)}
				&
				\mbox{ \ if \ } l>n
			\end{cases}
			\\
			= \ &
			\begin{cases}
				\frac{|x|^{2-l} }{(l-2)(n-l)}
				+
				\frac{|x|^{2-n}}{(n-2)(n-l_1)}-\frac{|x|^{2-n}}{(n-2)(n-l)}
				&
				\mbox{if \ } 2<l<n
				\\
				\frac{	|x|^{2-n}  \ln |x|}{n-2}
				+
				\frac{	|x|^{2-n} }{n-2}(\frac{1}{n-2} + \frac{1}{n-l_1})
				&
				\mbox{if \ } l=n
				\\
				\frac{|x|^{2-n}}{(n-2)(l-n)} + \frac{|x|^{2-n}}{(n-2)(n-l_1)}
				-\frac{|x|^{2-l}}{(l-2)(l-n)}
				&
				\mbox{if \ } l>n
			\end{cases}
			=
			\begin{cases}
				\frac{|x|^{2-l} }{(l-2)(n-l)}
				+
				\frac{(l_1 -l)|x|^{2-n}}{(n-2)(n-l_1)(n-l)}
				&
				\mbox{if \ } 2<l<n
				\\
				\frac{	|x|^{2-n}  \ln |x|}{n-2}
				+
				\frac{	|x|^{2-n} }{n-2}(\frac{1}{n-2} + \frac{1}{n-l_1})
				&
				\mbox{if \ } l=n
				\\
				\frac{(l-l_1)|x|^{2-n}}{(n-2)(n-l_1)(l-n)}
				-\frac{|x|^{2-l}}{(l-2)(l-n)}
				&
				\mbox{if \ } l>n .
			\end{cases}
		\end{align*}
	In particular,
	for $|x|\ge 1$,
	\begin{equation*}
		u(x)
		\ge
		\begin{cases}
			\frac{ |x|^{2-l} }{(l-2)(n-2) }
			+
			\frac{|x|^{2-n}}{(n-2)(n-l_1)}	&
			\mbox{ \ if \ } 2<l<n
			\\
			\frac{	|x|^{2-n}  \ln |x|}{n-2}
			+
			\frac{	|x|^{2-n} }{n-2}(\frac{1}{n-2} + \frac{1}{n-l_1})
			&
			\mbox{ \ if \ } l=n
			\\
			\frac{|x|^{2-n}}{(l-2)(n-2)} + \frac{|x|^{2-n}}{(n-2)(n-l_1)}	&
			\mbox{ \ if \ } l>n .
		\end{cases}
	\end{equation*}
\end{lemma}

\begin{proof}
	
	$l_1<n$ and $l>2$ ensure the integrability of \eqref{qd240808-2}. Since $f(x)=f(|x|)$, it is easy to see that $u$ is radial. Due to the upper bound of $f(|x|)$, by the removable singularity theorem for harmonic functions (It is used for the case $2\le l_1 <n$) and maximum principle, we have the formula of $u(x)$ in \eqref{qd240808-1}. And the deduction of $\pp_{|x|}u$ in \eqref{qd240808-1} is straightforward.

	When $f(r) = r^{-l_{1}} \1_{\{r\le 1\}} + r^{-l} \1_{\{ r > 1\}} $ with $r=|x|$, we only present the calculation of $u(x)$. $\pp_{|x|}u$ is similar. For $a\ge 1$, $ l>2$, then $ \int_{a}^{\infty} b f(b) db
	=
	\frac{a^{2-l}}{l-2}$.
	For $0<a\le 1$,
	\begin{equation*}
		\int_{a}^{\infty} b f(b) db
		=
		\int_{1}^{\infty} b^{1-l} db
		+
		\int_{a}^{1} b^{1-l_{1}} db
		=
		\frac{1}{l-2} +
		\begin{cases}
			\frac{1}{2-l_{1}}
			(1-a^{2-l_1} )
			&
			\mbox{ \ if \ } l_{1} <2
			\\
			-\ln a
			&
			\mbox{ \ if \ } l_{1} =2
			\\
			\frac{1 }{l_{1} -2} (a^{2-l_{1}} -1 )
			&
			\mbox{ \ if \ } l_{1} >2.
		\end{cases}
	\end{equation*}
	For $0<r\le 1$,
	\begin{equation}\label{ZO1}
		\int_{0}^{r} a^{n-3}
		\int_{a}^{\infty} b f(b) db da
		=
		\frac{r^{n-2}}{(l-2)(n-2)}  +
		\begin{cases}
			\frac{r^{n-2}}{(2-l_{1})(n-2)}
			-
			\frac{ r^{n-l_1} }{(2-l_1)(n-l_1)}
			&
			\mbox{ \ if \ } l_{1} <2
			\\
			\frac{r^{n-2}(-\ln r)}{n-2}
			+
			\frac{r^{n-2}}{(n-2)^2}
			&
			\mbox{ \ if \ } l_{1} =2
			\\
			\frac{ r^{n-l_{1} } }{(l_{1} -2)(n-l_{1})}
			- \frac{r^{n-2}}{(l_1-2)(n-2)}
			&
			\mbox{ \ if \ } 2<l_{1} <n,
		\end{cases}
	\end{equation}
	where $l_{1}<n$ guarantees the integrability around $0$.
	For $r\ge 1$, since
	\begin{equation*}
		\int_{1}^{r} a^{n-3}
		\int_{a}^{\infty} b f(b) db da
		=
		\int_{1}^{r} \frac{a^{n-1-l}}{l-2} da
		=
		\begin{cases}
			\frac{r^{n-l}-1}{(l-2)(n-l)}
			&
			\mbox{ \ if \ } 2<l<n
			\\
			\frac{\ln r}{n-2}
			&
			\mbox{ \ if \ } l=n
			\\
			\frac{1-r^{n-l}}{(l-2)(l-n)}
			&
			\mbox{ \ if \ } l>n ,
		\end{cases}
	\end{equation*}
	then combining \eqref{ZO1}, we have
	\begin{equation*}
		\int_{0}^{r} a^{n-3}
		\int_{a}^{\infty} b f(b) db da
		=
		\frac{1}{n-2}(\frac{1}{l-2} + \frac{1}{n-l_1})
		+
		\begin{cases}
			\frac{r^{n-l}-1}{(l-2)(n-l)}
			&
			\mbox{ \ if \ } 2<l<n
			\\
			\frac{\ln r}{n-2}
			&
			\mbox{ \ if \ } l=n
			\\
			\frac{1-r^{n-l}}{(l-2)(l-n)}
			&
			\mbox{ \ if \ } l>n.
		\end{cases}
	\end{equation*}
	The left calculations are direct, where for the last lower bound, we used that for $|x|\ge 1$, $|x|^{2-n} \le |x|^{2-l}$ when $l<n$, and $|x|^{2-l} \le |x|^{2-n}$ when $l>n$.
\end{proof}

\section{Integral estimates for the distorted Fourier transform}

\begin{lemma}\label{asy-lem-1}
	
	Suppose $0\le \lambda <\infty$, $a,b\in \mathbb{R}$, $
	\int_{0}^{1}  x^{a} \langle \ln x \rangle^{b} dx < \infty $,
	that is, either $(a,b) \in (-1,\infty) \times \mathbb{R}$ or
	$(a,b) \in \{-1\} \times (-\infty,-1)$ holds, then for $0\le x_{0} \le x_{1} \le 1/2$, we have	
	\begin{equation}\label{general case}
		\int_{x_{0}}^{x_{1}}  e^{-\lambda x} x^{a}
		(-\ln x)^{b} dx
		\le C_1
		\begin{cases}
			\begin{cases}
				x_{1}^{a+1}  (-\ln x_{1})^{b} & \mbox{if \ } a>-1
				\\
				( -\ln x_{1})^{b+1}
				-
				( -\ln x_{0})^{b+1}
				& \mbox{if \ } a=-1, b<-1
			\end{cases}
			&
			\mbox{for \ }
			0\le \lambda \le x_{1}^{-1}
			\\
			\frac{ (\ln \lambda)^b }{\lambda^{a+1}}
			+
			\begin{cases}
				0
				&
				\mbox{if \ } a>-1
				\\
				(\ln \lambda)^{b+1} - (-\ln x_{0})^{b+1}
				&
				\mbox{if \ } a=-1, b<-1
			\end{cases}
			&
			\mbox{for \ }
			x_{1}^{-1} \le \lambda  \le x_{0}^{-1}
			\\
			\frac{ (\ln \lambda)^b }{\lambda^{a+1}} e^{-\frac{x_{0 } \lambda}{2}}
			&
			\mbox{for \ }  \lambda \ge x_{0}^{-1},
		\end{cases}
	\end{equation}	
	where the constant $C_1>0$ only depends on $a, b$, and the case $\lambda \ge x_{0}^{-1}$ is vacuum when $x_0=0$.
	
	For $0\le x_{0} \le 1/2 < x_{1}$, we have
	\begin{equation}\label{qd24Sep22-3}
		\int_{x_{0}}^{x_1}  e^{-\lambda x} x^{a}
		\langle \ln x \rangle^{b} dx
		\le C_2
		\begin{cases}
			1
			&
			\mbox{for \ }
			0\le \lambda \le 2
			\\
			\frac{ (\ln \lambda)^b }{\lambda^{a+1}}
			+
			\begin{cases}
				0
				&
				\mbox{if \ } a>-1
				\\
				(\ln \lambda)^{b+1} - (-\ln x_{0})^{b+1}
				&
				\mbox{if \ } a=-1, b<-1
			\end{cases}
			&
			\mbox{for \ }
			2 \le \lambda  \le x_{0}^{-1}
			\\
			\frac{ (\ln \lambda)^b }{\lambda^{a+1}} e^{-\frac{x_{0 } \lambda}{2}}
			&
			\mbox{for \ } \lambda \ge x_{0}^{-1},
		\end{cases}
	\end{equation}
	where the constant $C_2>0$ only depends on $a, b, ~\int_{1/2}^{x_{1}} x^{a}
	\langle \ln x\rangle^{b} dx$.
	
\end{lemma}

\begin{proof}
	
	We first consider the case $0\le x_{0} \le x_{1} \le 1/2$. For $0\le \lambda\le x_{1}^{-1}$,
	\begin{equation}\label{z2}
			\int_{x_{0}}^{x_{1}}  e^{-\lambda x} x^{a} ( -\ln x)^{b} dx
			\sim
			\int_{x_{0}}^{x_{1}}   x^{a} (-\ln x)^{b} dx
			\lesssim
			\begin{cases}
				x_{1}^{a+1}  (-\ln x_{1})^{b} & \mbox{if \ } a>-1
				\\
				( -\ln x_{1})^{b+1}
				-
				( -\ln x_{0})^{b+1}
				& \mbox{if \ } a=-1, b<-1,
			\end{cases}
	\end{equation}	
	where for the last step for the case $a>-1$, we used the following calculation. If $a>-1$,
	\begin{equation*}
		\int_{x_{0}}^{x_{1}}   x^{a} (-\ln x)^{b} dx
		=
		\frac{1}{a+1} \big[ x_1^{a+1} (-\ln x_1)^b - x_0^{a+1} (-\ln x_0)^b \big]
		+
		\frac{b}{a+1} \int_{x_0}^{x_1} x^a (-\ln x)^{b-1} dx.
	\end{equation*}
	When $x_1 \le C_3$ with a constant $0<C_3<1/2$ sufficiently small depending on $a, b$, we have $\int_{x_{0}}^{x_{1}}   x^{a} (-\ln x)^{b} dx \lesssim x_{1}^{a+1}  (-\ln x_{1})^{b}$. When $C_3 < x_1 \le 1/2$, the estimate holds due to the assumption $
	\int_{0}^{1}  x^{a} \langle \ln x \rangle^{b} dx < \infty $.

	For $\lambda \ge x_{0}^{-1}$,
	\begin{equation}\label{z5}
		\begin{aligned}
			&
			\int_{x_{0}}^{x_{1}}  e^{-\lambda x} x^{a}
			( - \ln x )^{b} dx
			=
			\frac{ 1 }{\lambda^{a+1}}
			\int_{x_{0 } \lambda}^{x_{1} \lambda}
			e^{-z} z^{a} (\ln \lambda -\ln z)^b dz
			\\
			\lesssim \ &
			\frac{ 1 }{\lambda^{a+1}}
			\begin{cases}
				(\ln \lambda)^b e^{-\frac{x_{0 } \lambda}{2}}
				&
				\mbox{ \ if \ } x_{0}^{-1} \le \lambda \le x_{0}^{-2}
				\\
				e^{-\frac{3 x_{0} \lambda}{4 }}
				&
				\mbox{ \ if \ } \lambda \ge x_{0}^{-2}
			\end{cases}
			\
			\lesssim
			\frac{ (\ln \lambda)^b }{\lambda^{a+1}} e^{-\frac{x_{0 } \lambda}{2}}.
		\end{aligned}
	\end{equation}
	In order to get the first ``$\lesssim$'' above, we need the following estimates.
	
	If $ x_{1} \lambda \le \lambda^{\frac 12}$, that is, $\lambda \le x_{1}^{-2}$, then
	\begin{equation}\label{z3}
		\int_{x_{0 } \lambda}^{x_{1} \lambda}
		e^{-z} z^{a} (\ln \lambda -\ln z)^b dz
		\sim
		(\ln \lambda)^b
		\int_{x_{0 } \lambda}^{x_{1} \lambda}
		e^{-z} z^{a}  dz
		\lesssim
		(\ln \lambda)^b e^{-\frac{x_{0 } \lambda}{2}}.
	\end{equation}

	If $x_{0} \lambda \ge \lambda^{\frac 12}$, that is, $\lambda \ge x_{0}^{-2}$, then
	\begin{equation}\label{z4}
		\int_{x_{0 } \lambda}^{x_{1} \lambda}
		e^{-z} z^{a} (\ln \lambda -\ln z)^b dz
		\lesssim   e^{-\frac{3 x_{0} \lambda}{4 }},
	\end{equation}
since $-\ln x_{1} \le \ln \lambda - \ln z \le \ln (\frac{\lambda}{x_{0} \lambda})  \le \frac{\ln \lambda}{2}$,
$z^{a} \lesssim
\begin{cases}
	1 & \mbox{if \ } a\le 0
	\\
	\lambda^a
	& \mbox{if \ } a > 0
\end{cases}$,
$(\ln \lambda -\ln z)^b
\lesssim
\begin{cases}
	1 & \mbox{if \ } b\le 0
	\\
	(\ln  \lambda)^b
	& \mbox{if \ } b > 0
\end{cases}$.

	If $x_{0} \lambda \le \lambda^{\frac 12} \le x_{1} \lambda $, that is, $x_{1}^{-2} \le \lambda \le x_{0}^{-2}$, by \eqref{z3}, \eqref{z4}, then
	\begin{equation*}
		\int_{x_{0 } \lambda}^{x_{1} \lambda}
		e^{-z} z^{a} (\ln \lambda -\ln z)^b dz
		=
		\int_{x_{0 } \lambda}^{ \lambda^{\frac 12}}
		+
		\int_{ \lambda^{\frac 12} }^{x_{1} \lambda}
		\cdots
		\lesssim
		(\ln \lambda)^b e^{-\frac{x_{0 } \lambda}{2}}
		+
		e^{-\frac{3 \lambda^{\frac 12} }{4}}
		\lesssim
		(\ln \lambda)^b e^{-\frac{x_{0 } \lambda}{2}},
	\end{equation*}
	where we used $x_{0}^{-1} \le \lambda \le x_{0}^{-2}$ for the last step.

	For $x_{1}^{-1} \le \lambda \le x_{0}^{-1}$, by \eqref{z2}, \eqref{z5}, we have
	\begin{equation*}
		\int_{x_{0}}^{x_{1}}  e^{-\lambda x} x^{a}
		(-\ln x )^{b} dx
		= \int_{x_{0}}^{1/\lambda} +
		\int_{ 1/\lambda }^{x_{1}}
		\cdots
		\lesssim
		\frac{ (\ln \lambda)^b }{\lambda^{a+1}}
		+
		\begin{cases}
			0
			&
			\mbox{if \ } a>-1
			\\
			(\ln \lambda)^{b+1} - (-\ln x_{0})^{b+1}
			&
			\mbox{if \ } a=-1, b<-1.
		\end{cases}
	\end{equation*}
	In sum, we complete the proof of \eqref{general case}.
	
	For $0\le x_{0} \le 1/2 < x_{1}$, $\int_0^{1/2}  e^{-\lambda x} x^{a}
	\langle \ln x\rangle^{b} dx$ can be handled by \eqref{general case}. $\int_{1/2}^{x_{1}}  e^{-\lambda x} x^{a}
	\langle \ln x\rangle^{b} dx \le  e^{-\frac{\lambda}{2}} \int_{1/2}^{x_{1}} x^{a}
	\langle \ln x\rangle^{b} dx$. Thus we have \eqref{qd24Sep22-3}.
\end{proof}

\section{Convolution estimates in finite time}\label{convo-finite-tim-sec}
\subsection{Preliminaries}

We need the following relationship repetitively: for $s\le t$ and $t\le t_* \le T$,
\begin{equation}\label{tsplit}
	\begin{aligned}
	&	(T-s)/2 \le t-s \le T-s
		\mbox{ \ for \ }
		s \le  t-(T-t);
		& &
		T-t \le T-s \le 2(T-t)
		\mbox{ \ for \ }
		s\ge t-(T-t);
		\\
	&	(t_*-s)/2 \le t-s \le t_*-s
		\mbox{ \ for \ } s\le t-(t_*-t);
		& &
		t_*-t\le t_* -s \le 2(t_*-t)
		\mbox{ \ for \ } s\ge t-(t_*-t).
	\end{aligned}
\end{equation}

\begin{lemma}
	
	Given $x,q \in\RR^d$, $p>0$, $b\ge 0$, and $L>0$, $0\le L_1 \le  L_2 \le \infty$, we have
	\begin{equation}\label{basic cal in x}
			\int_{\RR^d}
			e^{ - c (\frac{|x-y|}{\sqrt{L}} )^{p} }
			|y-q|^{-b} \1_{ \{ L_1 \le |y-q| \le L_2 \} } dy
			\lesssim
			\begin{cases}
				0 & \mbox{if \ } L_1 =L_2
				\\
				\begin{cases}
					L^{\frac d2 } L_1^{-b}
					&
					\mbox{if \ } L \le L_1^2
					\\
					\begin{cases}
						L^{\frac d2 -\frac{b}{2}}
						&
						\mbox{if \ } b<d
						\\
						\langle \ln(\frac{L}{L_1^2 } ) \rangle
						&
						\mbox{if \ } b=d
						\\
						L_1^{d-b}
						&
						\mbox{if \ } b>d
					\end{cases}
					&
					\mbox{if \ }  L_1^2 < L \le  L_2^2
					\\
					\begin{cases}
						L_2^{d-b}
						&
						\mbox{if \ } b<d
						\\
						\langle \ln(\frac{ L_2 }{L_1 } ) \rangle
						&
						\mbox{if \ } b=d
						\\
						L_1^{d-b}
						&
						\mbox{if \ } b>d
					\end{cases}
					&
					\mbox{if \ }  L > L_2^2
				\end{cases}
				&
				\mbox{if \ }
				L_1 <L_2.
			\end{cases}
	\end{equation}
	In particular, for $L_3\ge CL>0$ with a constant $C>0$, we have
	\begin{equation}\label{basic-x-far}
		\int_{\RR^d}
		e^{ - c (\frac{|x-y|}{\sqrt{L}} )^{p} }
		|y-q|^{-b} \1_{ \{ |y-q| \ge \sqrt{L_3} \} }
		dy
		\lesssim
		L^{\frac{d}{2}} L_3^{-\frac{b}{2}}.
	\end{equation}
	
\end{lemma}

\begin{remark}
	The estimate for the case $b<0$ is different, and we do not analyze it here.
\end{remark}

\begin{proof}[Proof of \eqref{basic cal in x}]
	
	For $L_1=L_2$, the conclusion is trivial.
	For $L_1<L_2$,
	\begin{equation*}
		\int_{\RR^d}
		e^{ - c (\frac{|x-y|}{\sqrt{L}} )^{p} }
		|y-q|^{-b} \1_{ \{ L_1 \le |y-q| \le L_2 \} } dy
		=
		L^{\frac d2 -\frac{b}{2}}
		\int_{\RR^d} e^{-c|\tilde{x} - z|^{p}} |z|^{-b}
		\1_{\{ L_1 L^{-\frac 12} \le |z| \le L_2 L^{-\frac 12} \}} dz,
	\end{equation*}
	where $\tilde{x} = (x-q)L^{-\frac 12}$.
	If $L_1=0$,
	\begin{align*}
		%\begin{equation*}%\label{basic cal in x}
		%	\begin{aligned}
			&
			L^{\frac d2 -\frac{b}{2}}
			\int_{\RR^d} e^{-c|\tilde{x} - z|^{p}} |z|^{-b}
			\1_{\{ |z| \le L_2 L^{-\frac 12} \}} dz
			\le
			L^{\frac d2 -\frac{b}{2}}
			\int_{\RR^d} e^{-c| z|^{p}}
			|z|^{-b}
			\1_{\{  |z| \le L_2 L^{-\frac 12} \}} dz
			\\
			\lesssim \ &
			\begin{cases}
				\infty
				&
				\mbox{ \ if \ }
				b\ge d
				\\
				\begin{cases}
					L^{\frac d2 -\frac{b}{2}}
					&
					\mbox{ \ if \ }   L \le L_2^2
					\\
					L_2^{d-b}
					&
					\mbox{ \ if \ }  L > L_2^2
				\end{cases}
				&	\mbox{ \ if \ } b<d.
			\end{cases}
			%	\end{aligned}
		%\end{equation*}
	\end{align*}
	If $L_1>0$,
	\begin{align*}
		%\begin{equation*}
		%	\begin{aligned}
			&
			L^{\frac d2 -\frac{b}{2}}
			\int_{\RR^d} e^{-c|\tilde{x} - z|^{p}} |z|^{-b}
			\1_{\{ L_1 L^{-\frac 12} \le |z| \le L_2 L^{-\frac 12} \}} dz
			\le
			L^{\frac d2 -\frac{b}{2}}
			\int_{\RR^d} e^{-c|\tilde{x} - z|^{p}}
			\min\{
			|z|^{-b}, (L_1 L^{-\frac 12} )^{-b}
			\}
			\1_{\{  |z| \le L_2 L^{-\frac 12} \}} dz
			\\
			\le \ &
			L^{\frac d2 -\frac{b}{2}}
			\int_{\RR^d}
			e^{ -c |z|^{p} }
			\left[
			( L_1 L^{-\frac 12})^{-b}
			\1_{ \{ |z| \le  L_1 L^{-\frac 12}  \} } +
			|z|^{-b}
			\1_{ \{ L_1 L^{-\frac 12}  < |z| \le  L_2 L^{-\frac 12}  \} }
			\right]
			dz,
			%	\end{aligned}
		%\end{equation*}
	\end{align*}
whose upper bound is presented in \eqref{basic cal in x} for the case $L_1 < L_2$.
\end{proof}

Next, we want to establish the basic calculation of the time variable.
Given $0\le L_1 \le  L_2 \le \infty$, $t>0$, for $s<t$,  we set
\begin{equation}
	g(s) :=
	\begin{cases}
		(t-s)^{\frac d2 -d_* } L_1^{-b}
		&
		\mbox{ \ if \ } t-s \le L_1^2
		\\
		\begin{cases}
			(t-s)^{\frac d2 -\frac{b}{2}-d_*}
			&
			\mbox{ \ if \ } b<d
			\\
			(t-s)^{-d_*}
			\langle \ln(\frac{t-s}{L_1^2 } ) \rangle
			&
			\mbox{ \ if \ } b=d
			\\
			(t-s)^{-d_*} L_1^{d-b}
			&
			\mbox{ \ if \ } b>d
		\end{cases}
		&
		\mbox{ \ if \ }  L_1^2 < t-s \le  L_2^2
		\\
		\begin{cases}
			(t-s)^{-d_*} L_2^{d-b}
			&
			\mbox{ \ if \ } b<d
			\\
			(t-s)^{-d_*}	\langle \ln(\frac{ L_2 }{L_1 } ) \rangle
			&
			\mbox{ \ if \ } b=d
			\\
			(t-s)^{-d_*}	L_1^{d-b}
			&
			\mbox{ \ if \ } b>d
		\end{cases}
		&
		\mbox{ \ if \ }  t-s > L_2^2,
	\end{cases}
\end{equation}
where we use the convention $L_1^{-1} L_2=1$ if $L_1=L_2=0$.

Claim:
for $\delta >  0$, if $\frac{d}{2} +1 > d_* > \frac{d}{2} +1 -\delta$,
$ x^{-\delta}
\int_{t-x}^{t} g(s) ds
\le  \infty$,
and for $\delta = 0$, $d_* < \frac{d}{2} +1 $,
\begin{small}
	\begin{equation}\label{t-int-del=0}
		\int_{t-x}^{t} g(s) ds
		\lesssim
		\begin{cases}
			\begin{cases}
				(\max\{x,L_2^2 \} )^{1-d_*} L_2^{d-b}
				&
				\mbox{ \ if \ } d_*<1
				\\
				\langle \ln(\frac{\max\{x,L_2^2 \}}{L_2^2}) \rangle  L_2^{d-b}
				&
				\mbox{ \ if \ } d_*=1
				\\
				L_2^{ d+2-b-2d_*}
				&
				\mbox{ \ if \ } 1< d_* <1+\frac{d-b}{2}
				\\
				\langle \ln(\frac{L_2}{L_1}) \rangle
				&
				\mbox{ \ if \ } d_* = 1+\frac{d-b}{2}
				\\
				L_1^{ d +2 -b-2d_*}
				&
				\mbox{ \ if \ } d_* > 1+\frac{d-b}{2}
			\end{cases}
			&
			\mbox{ \ if \ } b<d
			\\
			\begin{cases}
				(\max\{x,L_2^2 \} )^{1-d_* } \langle \ln(\frac{ L_2 }{L_1 } ) \rangle
				&
				\mbox{ \ if \ } d_*<1
				\\
				\langle \ln(\frac{\max\{x,L_2^2 \}}{L_1^2}) \rangle \langle \ln(\frac{ L_2 }{L_1 } ) \rangle
				&
				\mbox{ \ if \ } d_*=1
				\\
				L_1^{2-2d_*}
				&
				\mbox{ \ if \ } d_*>1
			\end{cases}
			&
			\mbox{ \ if \ } b=d
			\\
			\begin{cases}
				( \max\{x,L_2^2 \} )^{1-d_*  } L_1^{d-b}
				&
				\mbox{ \ if \ } d_*<1
				\\
				\langle\ln(\frac{\max\{x,L_2^2 \}}{L_1^2})  \rangle
				L_1^{d-b}
				&
				\mbox{ \ if \ } d_*=1
				\\
				L_1^{d+2-b-2d_*}
				&
				\mbox{ \ if \ } d_*>1
			\end{cases}
			&
			\mbox{ \ if \ } b>d;
		\end{cases}
	\end{equation}
\end{small}
and if $\delta > 0$, $d_* \le \frac{d}{2} +1 -\delta$,
\begin{equation}\label{t-int-del>0}
	x^{-\delta}
	\int_{t-x}^{t} g(s) ds
	\lesssim
	\begin{cases}
		\begin{cases}
			(\max\{x,L_2^2 \})^{1-d_*-\delta} L_2^{d-b}
			&
			\mbox{if \ } d_* \le 1- \delta
			\\
			L_2^{d+2-b -2d_*-2\delta}
			&
			\mbox{if \ }  1- \delta  < d_* \le 1+\frac{d-b}{2} - \delta
			\\
			L_1^{d+2-b -2d_*-2\delta}
			&
			\mbox{if \ }
			d_* > 1+\frac{d-b}{2} - \delta
		\end{cases}
		&
		\mbox{ \ if \ } b<d
		\\
		\begin{cases}
			(\max\{x,L_2^2 \})^{1-d_*-\delta} \langle \ln(\frac{ L_2 }{L_1 } ) \rangle
			&
			\mbox{if \ } d_*\le 1-\delta
			\\
			L_1^{2-2d_*-2\delta}
			&
			\mbox{if \ } d_*> 1-\delta
		\end{cases}
		&
		\mbox{ \ if \ } b=d
		\\
		\begin{cases}
			(\max\{x,L_2^2 \} )^{1-d_*-\delta} L_1^{d-b}
			&
			\mbox{if \ } d_*\le 1- \delta
			\\
			L_1^{d+2-b-2d_*-2\delta}
			&
			\mbox{if \ } d_*> 1-\delta
		\end{cases}
		&
		\mbox{ \ if \ } b>d.
	\end{cases}
\end{equation}

\begin{proof}
	For $x\le L_1^2 $, when $d_*<\frac d2 +1$, then $
		\int_{t-x}^{t} g(s) ds \lesssim x^{\frac d2 +1 -d_*} L_1^{-b} $;  For $L_1^2  < x \le L_2^2 $,
	\begin{small}
		\begin{align*}
%	\begin{equation*}
%		\begin{aligned}
			&
			\int_{t-x}^{t} g(s) ds = \Big( \int_{t-L_1^2 }^{t}
			+ \int_{t-x}^{t -L_1^2  } \Big) g(s) ds
			\\
			\lesssim \ &
			L_1^{d+2-b-2d_*}
			+
			\begin{cases}
				\begin{cases}
					x^{\frac d2 +1 -\frac{b}{2}-d_*} ,
					&
					b<d+2-2d_*
					\\
					\ln(\frac{x}{L_1^2}),
					&
					b=d+2-2d_*
					\\
					L_1^{ d +2 -b-2d_*} ,
					&
					b>d+2-2d_*
				\end{cases}
				&
				\mbox{if \ } b<d
				\\
				\begin{cases}
					x^{1-d_*} \langle \ln (\frac{x}{L_1^2} ) \rangle ,
					&  d_*<1
					\\
					\langle \ln (\frac{x}{L_1^2} ) \rangle^2 ,
					&  d_*=1
					\\
					L_1^{2-2d_*} , &
					d_*>1
				\end{cases}
				&
				\mbox{if \ } b=d
				\\
				\begin{cases}
					x^{1-d_*} L_1^{d-b} ,
					& d_*<1
					\\
					L_1^{d-b} \ln(\frac{x}{L_1^2}) ,
					&  d_*=1
					\\
					L_1^{d+2-b-2d_*} ,
					&  d_*>1
				\end{cases}
				&
				\mbox{if \ } b>d
			\end{cases}
			\sim
			\begin{cases}
				\begin{cases}
					x^{\frac d2 +1 -\frac{b}{2}-d_*} ,
					&
					b<d+2-2d_*
					\\
					\langle \ln(\frac{x}{L_1^2}) \rangle ,
					&
					b=d+2-2d_*
					\\
					L_1^{ d +2 -b-2d_*} ,
					&
					b>d+2-2d_*
				\end{cases}
				&
				\mbox{if \ } b<d
				\\
				\begin{cases}
					x^{1-d_*} \langle \ln (\frac{x}{L_1^2} ) \rangle ,
					&  d_*<1
					\\
					\langle \ln (\frac{x}{L_1^2} ) \rangle^2 ,
					&  d_*=1
					\\
					L_1^{2-2d_*} , &  d_*>1
				\end{cases}
				&
				\mbox{if \ } b=d
				\\
				\begin{cases}
					x^{1-d_*} L_1^{d-b} ,
					&  d_*<1
					\\
					L_1^{d-b} \langle  \ln(\frac{x}{L_1^2})\rangle ,
					&  d_*=1
					\\
					L_1^{d+2-b-2d_*} ,
					&  d_*>1
				\end{cases}
				&
				\mbox{if \ } b>d;
			\end{cases}
%		\end{aligned}
%	\end{equation*}
	\end{align*}
	\end{small}
	For  $x > L_2^2  $,
	\begin{small}
	\begin{align*}
		%	\begin{equation*}
			%		\begin{aligned}
				&
				\int_{t-x}^{t} g(s) ds = \Big( \int_{t-L_2^2 }^{t}
				+ \int_{t-x}^{t -L_2^2 } \Big) g(s) ds
				\\
				\lesssim \ &
				\begin{cases}
					\begin{cases}
						L_2^{d +2 -b-2d_*} ,
						&
						b<d+2-2d_*
						\\
						\langle \ln(\frac{L_2}{L_1}) \rangle ,
						&
						b=d+2-2d_*
						\\
						L_1^{ d +2 -b-2d_*} ,
						&
						b>d+2-2d_*
					\end{cases}
					&
					\mbox{if \ } b<d
					\\
					\begin{cases}
						L_2^{2-2d_*} \langle \ln (\frac{L_2}{L_1} ) \rangle ,
						&  d_*<1
						\\
						\langle \ln (\frac{L_2}{L_1} ) \rangle^2 ,
						&  d_*=1
						\\
						L_1^{2-2d_*} , &
						d_*>1
					\end{cases}
					&
					\mbox{if \ } b=d
					\\
					\begin{cases}
						L_2^{2-2d_*} L_1^{d-b} ,
						&  d_*<1
						\\
						L_1^{d-b} \langle  \ln(\frac{L_2}{L_1})\rangle ,
						&  d_*=1
						\\
						L_1^{d+2-b-2d_*} ,
						&  d_*>1
					\end{cases}
					&
					\mbox{if \ } b>d
				\end{cases}
				+
				\begin{cases}
					\begin{cases}
						x^{1-d_*} L_2^{d-b} ,
						&  d_*<1
						\\
						\ln(\frac{x}{L_2^2}) L_2^{d-b} ,
						&  d_*=1
						\\
						L_2^{ d+2-b-2d_*} ,
						&  d_*>1
					\end{cases}
					&
					\mbox{if \ } b<d
					\\
					\begin{cases}
						x^{1-d_*} \langle \ln(\frac{ L_2 }{L_1 } ) \rangle ,
						&  d_*<1
						\\
						\ln(\frac{x}{L_2^2}) \langle \ln(\frac{ L_2 }{L_1 } ) \rangle ,
						&  d_*=1
						\\
						L_2^{2-2d_*} \langle \ln(\frac{ L_2 }{L_1 } ) \rangle ,
						&  d_*>1
					\end{cases}
					&
					\mbox{if \ } b=d
					\\
					\begin{cases}
						x^{1-d_*} L_1^{d-b} ,
						&  d_*<1
						\\
						\ln(\frac{x}{L_2^2}) L_1^{d-b} ,
						&  d_*=1
						\\
						L_2^{2-2d_*} L_1^{d-b} ,
						&  d_*>1
					\end{cases}
					&
					\mbox{if \ } b>d
				\end{cases}
				\\
				\sim \ &
				\begin{cases}
					\begin{cases}
						x^{1-d_*} L_2^{d-b} ,
						&  d_*<1
						\\
						L_2^{d-b} \langle \ln(\frac{x}{L_2^2}) \rangle ,
						&  d_*=1
						\\
						L_2^{ d+2-b-2d_*} ,
						&  1< d_* <1+\frac{d-b}{2}
						\\
						\langle \ln(\frac{L_2}{L_1}) \rangle ,
						&  d_* = 1+\frac{d-b}{2}
						\\
						L_1^{ d +2 -b-2d_*} ,
						&  d_* > 1+\frac{d-b}{2}
					\end{cases}
					&
					\mbox{if \ } b<d
					\\
					\begin{cases}
						x^{1-d_*} \langle \ln(\frac{ L_2 }{L_1 } ) \rangle ,
						&  d_*<1
						\\
						\langle \ln(\frac{x}{L_1^2}) \rangle \langle \ln(\frac{ L_2 }{L_1 } ) \rangle ,
						&  d_*=1
						\\
						L_1^{2-2d_*} ,
						&  d_*>1
					\end{cases}
					&
					\mbox{if \ } b=d
					\\
					\begin{cases}
						x^{1-d_*} L_1^{d-b} ,
						&  d_*<1
						\\
						L_1^{d-b}	\langle\ln(\frac{x}{L_1^2})  \rangle ,
						&  d_*=1
						\\
						L_1^{d+2-b-2d_*} ,
						&  d_*>1
					\end{cases}
					&
					\mbox{if \ } b>d.
				\end{cases}
				%		\end{aligned}
			%	\end{equation*}
	\end{align*}
	\end{small}
	
	Thus,
	for $d_*<\frac d2 +1$, if $x\le L_1^2$,
$\int_{t-x}^{t} g(s) ds \lesssim x^{\frac d2 +1 -d_*} L_1^{-b}$;
if $L_1^2 < x \le L_2^2$,
\begin{small}
\begin{equation*}
\int_{t-x}^{t} g(s) ds \lesssim \begin{cases}
	\begin{cases}
		x^{\frac d2 +1 -\frac{b}{2}-d_*}
		&
		\mbox{ \ if \ }
		b<d+2-2d_*
		\\
		\langle \ln(\frac{x}{L_1^2}) \rangle
		&
		\mbox{ \ if \ }
		b=d+2-2d_*
		\\
		L_1^{ d +2 -b-2d_*}
		&
		\mbox{ \ if \ }
		b>d+2-2d_*
	\end{cases}
	&
	\mbox{ \ if \ } b<d
	\\
	\begin{cases}
		x^{1-d_*} \langle \ln (\frac{x}{L_1^2} ) \rangle
		&
		\mbox{ \ if \ } d_*<1
		\\
		\langle \ln (\frac{x}{L_1^2} ) \rangle^2
		&
		\mbox{ \ if \ } d_*=1
		\\
		L_1^{2-2d_*} &
		\mbox{ \ if \ } d_*>1
	\end{cases}
	&
	\mbox{ \ if \ } b=d
	\\
	\begin{cases}
		x^{1-d_*} L_1^{d-b}
		&
		\mbox{ \ if \ } d_*<1
		\\
		L_1^{d-b} \langle  \ln(\frac{x}{L_1^2})\rangle
		&
		\mbox{ \ if \ } d_*=1
		\\
		L_1^{d+2-b-2d_*}
		&
		\mbox{ \ if \ } d_*>1
	\end{cases}
	&
	\mbox{ \ if \ } b>d;
\end{cases}
\end{equation*}
\end{small}
if $x > L_2^2$,
\begin{equation*}
\int_{t-x}^{t} g(s) ds \lesssim
\begin{cases}
	\begin{cases}
		x^{1-d_*} L_2^{d-b}
		&
		\mbox{ \ if \ } d_*<1
		\\
		L_2^{d-b} \langle \ln(\frac{x}{L_2^2}) \rangle
		&
		\mbox{ \ if \ } d_*=1
		\\
		L_2^{ d+2-b-2d_*}
		&
		\mbox{ \ if \ } 1< d_* <1+\frac{d-b}{2}
		\\
		\langle \ln(\frac{L_2}{L_1}) \rangle
		&
		\mbox{ \ if \ } d_* = 1+\frac{d-b}{2}
		\\
		L_1^{ d +2 -b-2d_*}
		&
		\mbox{ \ if \ } d_* > 1+\frac{d-b}{2}
	\end{cases}
	&
	\mbox{ \ if \ } b<d
	\\
	\begin{cases}
		x^{1-d_*} \langle \ln(\frac{ L_2 }{L_1 } ) \rangle
		&
		\mbox{ \ if \ } d_*<1
		\\
		\langle \ln(\frac{x}{L_1^2}) \rangle \langle \ln(\frac{ L_2 }{L_1 } ) \rangle
		&
		\mbox{ \ if \ } d_*=1
		\\
		L_1^{2-2d_*}
		&
		\mbox{ \ if \ } d_*>1
	\end{cases}
	&
	\mbox{ \ if \ } b=d
	\\
	\begin{cases}
		x^{1-d_*} L_1^{d-b}
		&
		\mbox{ \ if \ } d_*<1
		\\
		L_1^{d-b}	\langle\ln(\frac{x}{L_1^2})  \rangle
		&
		\mbox{ \ if \ } d_*=1
		\\
		L_1^{d+2-b-2d_*}
		&
		\mbox{ \ if \ } d_*>1
	\end{cases}
	&
	\mbox{ \ if \ } b>d.
\end{cases}
\end{equation*}
	
	Then for $\delta \in \mathbb{R}$, if
	$x\le L_1^2$,
\begin{equation*}
x^{-\delta}
\int_{t-x}^{t} g(s) ds
\lesssim
\begin{cases}
	L_1^{d +2 -b -2d_* -2\delta}
	&
	\mbox{ \ if \ }
	\delta\le \frac{d}{2} +1 -d_*
	\\
	\infty
	&
	\mbox{ \ if \ }
	\delta > \frac{d}{2} +1 -d_*;
\end{cases}
\end{equation*}
if $L_1^2 < x \le L_2^2$,
\begin{small}
\begin{equation*}
x^{-\delta}
\int_{t-x}^{t} g(s) ds
\lesssim
\begin{cases}
	\begin{cases}
		\begin{cases}
			L_2^{d +2 -b -2d_* -2\delta}
			&
			\mbox{ \ if \ }
			b\le 	d+2-2d_*-2\delta
			\\
			L_1^{d +2 -b -2d_* -2\delta}
			&
			\mbox{ \ if \ }
			b> d+2-2d_*-2\delta
		\end{cases}
		&
		\mbox{if \ }
		b<d+2-2d_*
		\\
		\begin{cases}
			L_2^{-2\delta} \langle \ln(\frac{L_2}{L_1}) \rangle
			&
			\mbox{ \ if \ } \delta\le 0
			\\
			L_1^{-2\delta}
			&
			\mbox{ \ if \ } \delta>0
		\end{cases}
		&
		\mbox{if \ }
		b=d+2-2d_*
		\\
		\begin{cases}
			L_2^{ -2\delta } L_1^{ d +2 -b-2d_*}
			&
			\mbox{ \ if \ } \delta \le 0
			\\
			L_1^{ d +2 -b-2d_*-2\delta}
			&
			\mbox{ \ if \ } \delta > 0
		\end{cases}
		&
		\mbox{if \ }
		b>d+2-2d_*
	\end{cases}
	&
	\mbox{if \ } b<d
	\\
	\begin{cases}
		\begin{cases}
			L_2^{2-2d_*-2\delta} \langle \ln (\frac{L_2}{L_1} ) \rangle
			&
			\mbox{ \ if \ } \delta\le 1-d_*
			\\
			L_1^{2-2d_*-2\delta}
			&
			\mbox{ \ if \ } \delta > 1-d_*
		\end{cases}
		&
		\mbox{ \ if \ } d_*<1
		\\
		\begin{cases}
			L_2^{-2\delta}\langle \ln (\frac{L_2}{L_1} ) \rangle^2
			&
			\mbox{ \ if \ } \delta\le 0
			\\
			L_1^{-2\delta}
			&
			\mbox{ \ if \ } \delta > 0
		\end{cases}
		&
		\mbox{ \ if \ } d_*=1
		\\
		\begin{cases}
			L_2^{-2\delta}	L_1^{2-2d_*}
			&
			\mbox{ \ if \ } \delta\le 0
			\\
			L_1^{2-2d_* -2\delta}
			&
			\mbox{ \ if \ } \delta > 0
		\end{cases}
		&
		\mbox{ \ if \ } d_*>1
	\end{cases}
	&
	\mbox{if \ } b=d
	\\
	\begin{cases}
		\begin{cases}
			L_2^{2-2d_*-2\delta} L_1^{d-b}
			&
			\mbox{ \ if \ } \delta \le 1-d_*
			\\
			L_1^{d+2-b-2d_*-2\delta}
			&
			\mbox{ \ if \ } \delta > 1-d_*
		\end{cases}
		&
		\mbox{ \ if \ } d_*<1
		\\
		\begin{cases}
			L_2^{-2\delta} 	L_1^{d-b} \langle  \ln(\frac{L_2}{L_1 })\rangle
			&
			\mbox{ \ if \ } \delta \le 0
			\\
			L_1^{d-b-2\delta}
			&
			\mbox{ \ if \ } \delta > 0
		\end{cases}
		&
		\mbox{ \ if \ } d_*=1
		\\
		\begin{cases}
			L_2^{-2\delta} L_1^{d+2-b-2d_*}
			&
			\mbox{ \ if \ } \delta\le 0
			\\
			L_1^{d+2-b-2d_* -2\delta}
			&
			\mbox{ \ if \ } \delta > 0
		\end{cases}
		&
		\mbox{ \ if \ } d_*>1
	\end{cases}
	&
	\mbox{if \ } b>d;
\end{cases}
\end{equation*}
\end{small}
if $x>L_2^2$,
\begin{small}
\begin{align*}
		x^{-\delta}
		\int_{t-x}^{t} g(s) ds
		\lesssim
		\begin{cases}
			\begin{cases}
				\begin{cases}
					x^{1-d_*-\delta} L_2^{d-b}
					&
					\mbox{if \ } \delta\le 1-d_*
					\\
					L_2^{d+2-b -2d_*-2\delta}
					&
					\mbox{if \ } \delta > 1-d_*
				\end{cases}
				&
				\mbox{ \ if \ } d_*<1
				\\
				\begin{cases}
					x^{-\delta} \langle \ln(\frac{x}{L_2^2}) \rangle  L_2^{d-b}
					&
					\mbox{if \ }
					\delta\le 0
					\\
					L_2^{d-b-2\delta}
					&
					\mbox{if \ }
					\delta > 0
				\end{cases}
				&
				\mbox{ \ if \ } d_*=1
				\\
				\begin{cases}
					x^{-\delta} L_2^{ d+2-b-2d_*}
					&
					\mbox{if \ } \delta\le 0
					\\
					L_2^{ d+2-b-2d_*-2\delta}
					&
					\mbox{if \ } \delta > 0
				\end{cases}
				&
				\mbox{ \ if \ } 1< d_* <1+\frac{d-b}{2}
				\\
				\begin{cases}
					x^{-\delta}	\langle \ln(\frac{L_2}{L_1}) \rangle
					&
					\mbox{if \ } \delta\le 0
					\\
					L_2^{-2\delta}	\langle \ln(\frac{L_2}{L_1}) \rangle
					&
					\mbox{if \ } \delta > 0
				\end{cases}
				&
				\mbox{ \ if \ } d_* = 1+\frac{d-b}{2}
				\\
				\begin{cases}
					x^{-\delta} L_1^{ d +2 -b-2d_*}
					&
					\mbox{if \ } \delta\le 0
					\\
					L_2^{-2\delta} L_1^{ d +2 -b-2d_*}
					&
					\mbox{if \ } \delta> 0
				\end{cases}
				&
				\mbox{ \ if \ } d_* > 1+\frac{d-b}{2}
			\end{cases}
			&
			\mbox{ \ if \ } b<d
			\\
			\begin{cases}
				\begin{cases}
					x^{1-d_*-\delta} \langle \ln(\frac{ L_2 }{L_1 } ) \rangle
					&
					\mbox{if \ } \delta\le 1-d_*
					\\
					L_2^{2-2d_*-2\delta} \langle \ln(\frac{ L_2 }{L_1 } ) \rangle
					&
					\mbox{if \ } \delta> 1-d_*
				\end{cases}
				&
				\mbox{ \ if \ } d_*<1
				\\
				\begin{cases}
					x^{-\delta}	\langle \ln(\frac{x}{L_1^2}) \rangle \langle \ln(\frac{ L_2 }{L_1 } ) \rangle
					&
					\mbox{if \ } \delta \le 0
					\\
					L_2^{-2\delta}	 \langle \ln(\frac{ L_2 }{L_1 } ) \rangle^2
					&
					\mbox{if \ } \delta > 0
				\end{cases}
				&
				\mbox{ \ if \ } d_*=1
				\\
				\begin{cases}
					x^{-\delta}L_1^{2-2d_*}
					&
					\mbox{if \ } \delta\le 0
					\\
					L_2^{-2\delta}L_1^{2-2d_*}
					&
					\mbox{if \ } \delta> 0
				\end{cases}
				&
				\mbox{ \ if \ } d_*>1
			\end{cases}
			&
			\mbox{ \ if \ } b=d
			\\
			\begin{cases}
				\begin{cases}
					x^{1-d_*-\delta} L_1^{d-b}
					&
					\mbox{if \ } \delta\le 1-d_*
					\\
					L_2^{2-2d_*-2\delta} L_1^{d-b}
					&
					\mbox{if \ } \delta> 1-d_*
				\end{cases}
				&
				\mbox{ \ if \ } d_*<1
				\\
				\begin{cases}
					x^{-\delta}		\langle\ln(\frac{x}{L_1^2})  \rangle
					L_1^{d-b}
					&
					\mbox{if \ } \delta\le 0
					\\
					L_2^{-2\delta}	L_1^{d-b}	\langle\ln(\frac{L_2}{L_1})  \rangle
					&
					\mbox{if \ } \delta > 0
				\end{cases}
				&
				\mbox{ \ if \ } d_*=1
				\\
				\begin{cases}
					x^{-\delta}	L_1^{d+2-b-2d_*}
					&
					\mbox{if \ } \delta\le 0
					\\
					L_2^{-2\delta}	L_1^{d+2-b-2d_*}
					&
					\mbox{if \ } \delta> 0
				\end{cases}
				&
				\mbox{ \ if \ } d_*>1
			\end{cases}
			&
			\mbox{ \ if \ } b>d.
		\end{cases}
	\end{align*}
\end{small}
	
	In particular,
	for $\delta = 0$,
	\begin{equation*}
		\int_{t-x}^{t} g(s) ds
		\lesssim
		\begin{cases}
			\begin{cases}
				L_1^{d +2 -b -2d_*}
				&
				\mbox{ \ if \ }
				d_* \le \frac{d}{2} +1
				\\
				\infty
				&
				\mbox{ \ if \ }
				d_* > \frac{d}{2} +1
			\end{cases}
			&
			\mbox{if \ }
			x\le L_1^2
			\\
			\begin{cases}
				\begin{cases}
					L_2^{d +2 -b -2d_* }
					&
					\mbox{ \ if \ }
					d_* <1+\frac{d-b}{2}
					\\
					\langle \ln(\frac{L_2}{L_1}) \rangle
					&
					\mbox{ \ if \ }
					d_* = 1+\frac{d-b}{2}
					\\
					L_1^{d +2 -b -2d_* }
					&
					\mbox{ \ if \ }
					d_* > 1+\frac{d-b}{2}
				\end{cases}
				&
				\mbox{if \ } b<d
				\\
				\begin{cases}
					L_2^{2-2d_*} \langle \ln (\frac{L_2}{L_1} ) \rangle
					&
					\mbox{ \ if \ } d_*< 1
					\\
					\langle \ln (\frac{L_2}{L_1} ) \rangle^2
					&
					\mbox{ \ if \ } d_*= 1
					\\
					L_1^{2-2d_*}
					&
					\mbox{ \ if \ } d_* > 1
				\end{cases}
				&
				\mbox{if \ } b=d
				\\
				\begin{cases}
					L_2^{2-2d_*} L_1^{d-b}
					&
					\mbox{ \ if \ } d_*<1
					\\
					L_1^{d-b} \langle  \ln(\frac{L_2}{L_1 })\rangle
					&
					\mbox{ \ if \ } d_*=1
					\\
					L_1^{d+2-b-2d_*}
					&
					\mbox{ \ if \ } d_*>1
				\end{cases}
				&
				\mbox{if \ } b>d
			\end{cases}
			&
			\mbox{if \ }
			L_1^2 < x \le L_2^2;
		\end{cases}
	\end{equation*}
	if $x>L_2^2$,
	\begin{equation*}
		\int_{t-x}^{t} g(s) ds
		\lesssim
		\begin{cases}
			\begin{cases}
				x^{1-d_*} L_2^{d-b}
				&
				\mbox{ \ if \ } d_*<1
				\\
				\langle \ln(\frac{x}{L_2^2}) \rangle  L_2^{d-b}
				&
				\mbox{ \ if \ } d_*=1
				\\
				L_2^{ d+2-b-2d_*}
				&
				\mbox{ \ if \ } 1< d_* <1+\frac{d-b}{2}
				\\
				\langle \ln(\frac{L_2}{L_1}) \rangle
				&
				\mbox{ \ if \ } d_* = 1+\frac{d-b}{2}
				\\
				L_1^{ d +2 -b-2d_*}
				&
				\mbox{ \ if \ } d_* > 1+\frac{d-b}{2}
			\end{cases}
			&
			\mbox{ \ if \ } b<d
			\\
			\begin{cases}
				x^{1-d_* } \langle \ln(\frac{ L_2 }{L_1 } ) \rangle
				&
				\mbox{ \ if \ } d_*<1
				\\
				\langle \ln(\frac{x}{L_1^2}) \rangle \langle \ln(\frac{ L_2 }{L_1 } ) \rangle
				&
				\mbox{ \ if \ } d_*=1
				\\
				L_1^{2-2d_*}
				&
				\mbox{ \ if \ } d_*>1
			\end{cases}
			&
			\mbox{ \ if \ } b=d
			\\
			\begin{cases}
				x^{1-d_*  } L_1^{d-b}
				&
				\mbox{ \ if \ } d_*<1
				\\
				\langle\ln(\frac{x}{L_1^2})  \rangle
				L_1^{d-b}
				&
				\mbox{ \ if \ } d_*=1
				\\
				L_1^{d+2-b-2d_*}
				&
				\mbox{ \ if \ } d_*>1
			\end{cases}
			&
			\mbox{ \ if \ } b>d .
		\end{cases}
	\end{equation*}

	For $\delta > 0$,
	\begin{equation*}
		x^{-\delta}
		\int_{t-x}^{t} g(s) ds
		\lesssim
		\begin{cases}
			\begin{cases}
				L_1^{d +2 -b -2d_* -2\delta}
				&
				\mbox{ \ if \ }
				\delta\le \frac{d}{2} +1 -d_*
				\\
				\infty
				&
				\mbox{ \ if \ }
				\delta > \frac{d}{2} +1 -d_*
			\end{cases}
			&
			\mbox{if \ }
			x\le L_1^2
			\\
			\begin{cases}
				\begin{cases}
					L_2^{d +2 -b -2d_* -2\delta}
					&
					\mbox{ \ if \ }
					d_*\le 1+\frac{d-b}{2} -\delta
					\\
					L_1^{d +2 -b -2d_* -2\delta}
					&
					\mbox{ \ if \ }
					d_* > 1+\frac{d-b}{2} -\delta
				\end{cases}
				&
				\mbox{if \ } b<d
				\\
				\begin{cases}
					L_2^{2-2d_*-2\delta} \langle \ln (\frac{L_2}{L_1} ) \rangle
					&
					\mbox{ \ if \ } \delta\le 1-d_*
					\\
					L_1^{2-2d_*-2\delta}
					&
					\mbox{ \ if \ } \delta > 1-d_*
				\end{cases}
				&
				\mbox{if \ } b=d
				\\
				\begin{cases}
					L_2^{2-2d_*-2\delta} L_1^{d-b}
					&
					\mbox{ \ if \ } \delta \le 1-d_*
					\\
					L_1^{d+2-b-2d_*-2\delta}
					&
					\mbox{ \ if \ } \delta > 1-d_*
				\end{cases}
				&
				\mbox{if \ } b>d
			\end{cases}
			&
			\mbox{if \ }
			L_1^2 < x \le L_2^2;
		\end{cases}
	\end{equation*}
	if $x>L_2^2$,
	\begin{equation*}
		x^{-\delta}
		\int_{t-x}^{t} g(s) ds
		\lesssim
		\begin{cases}
			\begin{cases}
				\begin{cases}
					x^{1-d_*-\delta} L_2^{d-b}
					&
					\mbox{if \ } \delta\le 1-d_*
					\\
					L_2^{d+2-b -2d_*-2\delta}
					&
					\mbox{if \ } \delta > 1-d_*
				\end{cases}
				&
				\mbox{ \ if \ } d_* <1+\frac{d-b}{2}
				\\
				L_2^{-2\delta}	\langle \ln(\frac{L_2}{L_1}) \rangle
				&
				\mbox{ \ if \ } d_* = 1+\frac{d-b}{2}
				\\
				L_2^{-2\delta} L_1^{ d +2 -b-2d_*}
				&
				\mbox{ \ if \ } d_* > 1+\frac{d-b}{2}
			\end{cases}
			&
			\mbox{ \ if \ } b<d
			\\
			\begin{cases}
				\begin{cases}
					x^{1-d_*-\delta} \langle \ln(\frac{ L_2 }{L_1 } ) \rangle
					&
					\mbox{if \ } \delta\le 1-d_*
					\\
					L_2^{2-2d_*-2\delta} \langle \ln(\frac{ L_2 }{L_1 } ) \rangle
					&
					\mbox{if \ } \delta> 1-d_*
				\end{cases}
				&
				\mbox{ \ if \ } d_*<1
				\\
				L_2^{-2\delta}	 \langle \ln(\frac{ L_2 }{L_1 } ) \rangle^2
				&
				\mbox{ \ if \ } d_*=1
				\\
				L_2^{-2\delta}L_1^{2-2d_*}
				&
				\mbox{ \ if \ } d_*>1
			\end{cases}
			&
			\mbox{ \ if \ } b=d
			\\
			\begin{cases}
				\begin{cases}
					x^{1-d_*-\delta} L_1^{d-b}
					&
					\mbox{if \ } \delta\le 1-d_*
					\\
					L_2^{2-2d_*-2\delta} L_1^{d-b}
					&
					\mbox{if \ } \delta> 1-d_*
				\end{cases}
				&
				\mbox{ \ if \ } d_*<1
				\\
				L_2^{-2\delta}	L_1^{d-b}	\langle\ln(\frac{L_2}{L_1})  \rangle
				&
				\mbox{ \ if \ } d_*=1
				\\
				L_2^{-2\delta}	L_1^{d+2-b-2d_*}
				&
				\mbox{ \ if \ } d_*>1
			\end{cases}
			&
			\mbox{ \ if \ } b>d .
		\end{cases}
	\end{equation*}

\end{proof}

\medskip

\begin{lemma}
	Given $c>0$, $p>0$, $b\ge 0$, $d_*<\frac d2 +1$, $c_1>0$, $t\in [0,T]$, suppose
	\begin{equation}\label{23Nov02-1}
		\begin{aligned}
			&
			v(s)\ge 0 \mbox{ \ for \ } s\in [0,T],
			\quad
			C_l^{-1} l_i(t) \le l_i(s) \le C_l l_i(t)
			\mbox{ \ for \ } i=1,2, \ s\in \left[ \left[t-(T-t) \right]_+,t \right] ,
			\\
			&
			0\le l_1(s) \le l_2(s) \le C (T-s)^{\frac 12}
			\mbox{ \ for \ } s\in [0,T]
		\end{aligned}
	\end{equation}
	with some constants $C_l\ge 1$, $C>0$ independent of $T$, then for any $x\in\mathbb{R}^d$, we have
	\begin{align*}
	%\begin{equation*}%\label{far-region-general}
		%\begin{aligned}
			&
			\int_0^t v(s)  (t-s)^{-d_*} \int_{\RR^d}
			e^{ - c (\frac{|x-y|}{\sqrt{t-s}} )^{p} }
			|y-q|^{-b} \1_{ \{ |y-q| \ge c_1 (T-s)^{\frac 12} \} }
			dy ds
			\\
			\lesssim \ &
			\int_0^{[t-(T-t)]_+}  v(s) (T-s)^{\frac d2 -d_* -\frac b2}  ds
			+
			\sup\limits_{t_1\in  [ [t-(T-t)]_+, t  ] }v(t_1) (T-t)^{1-d_* +\frac d2 -\frac b2 }  ,
		\end{align*}
	\begin{equation}\label{annular-lem}
		\int_0^t  v(s) (t-s)^{-d_*} \int_{\RR^d}
		e^{ -c  (\frac{|x-y|}{\sqrt{t-s}} )^{p} }
		|y-q|^{-b} \1_{ \{ l_1(s) \le |y-q| \le l_2(s) \} } dy ds
		\lesssim
		\tilde{P}_1 + \tilde{P}_2   ,
	\end{equation}
	where $c_+ :=\max\left\{c,0\right\}$ for any $c\in \mathbb{R}$,
	\begin{equation*}
		\begin{aligned}
			&
			\tilde{P}_1:=
			\int_0^{[t-(T-t)]_+} v(s) (T-s)^{ -d_*  }
			\begin{cases}
				l_2^{d-b}(s)
				&
				\mbox{ \ if \ } b<d
				\\
				\langle \ln(\frac{l_2(s)}{l_1(s)}) \rangle
				&
				\mbox{ \ if \ } b=d
				\\
				l_1^{d-b}(s)
				&
				\mbox{ \ if \ } b>d
			\end{cases}
			ds,
			\\
			&
			\tilde{P}_2:=
			\sup\limits_{t_1\in  [ [t-(T-t)]_+, t  ] }v(t_1)
			\begin{cases}
				\begin{cases}
					(T-t )^{1-d_*} l_2^{d-b}(t)
					&
					\mbox{ \ if \ } d_*<1
					\\
					\langle \ln(\frac{T-t}{l_2^2(t)}) \rangle  l_2^{d-b}(t)
					&
					\mbox{ \ if \ } d_*=1
					\\
					l_2^{ d+2-b-2d_*}(t)
					&
					\mbox{ \ if \ } 1< d_* <1+\frac{d-b}{2}
					\\
					\langle \ln(\frac{l_2(t)}{l_1(t)}) \rangle
					&
					\mbox{ \ if \ } d_* = 1+\frac{d-b}{2}
					\\
					l_1^{ d +2 -b-2d_*}(t)
					&
					\mbox{ \ if \ } d_* > 1+\frac{d-b}{2}
				\end{cases}
				&
				\mbox{ \ if \ } b<d
				\\
				\begin{cases}
					(T-t)^{1-d_* } \langle \ln(\frac{ l_2(t) }{l_1(t) } ) \rangle
					&
					\mbox{ \ if \ } d_*<1
					\\
					\langle \ln(\frac{T-t}{l_1^2(t) }) \rangle \langle \ln(\frac{ l_2(t) }{l_1(t) } ) \rangle
					&
					\mbox{ \ if \ } d_*=1
					\\
					l_1^{2-2d_*}(t)
					&
					\mbox{ \ if \ } d_*>1
				\end{cases}
				&
				\mbox{ \ if \ } b=d
				\\
				\begin{cases}
					(T-t )^{1-d_*  } l_1^{d-b}(t)
					&
					\mbox{ \ if \ } d_*<1
					\\
					\langle\ln(\frac{T-t}{l_1^2(t)})  \rangle
					l_1^{d-b}(t)
					&
					\mbox{ \ if \ } d_*=1
					\\
					l_1^{d+2-b-2d_*}(t)
					&
					\mbox{ \ if \ } d_*>1
				\end{cases}
				&
				\mbox{ \ if \ } b>d.
			\end{cases}
		\end{aligned}
	\end{equation*}

\end{lemma}

\begin{remark}
	When $b=0<d$, the cases $d_* =1+ \frac {d-b}2$ and
	$d_* > 1+ \frac {d-b}2$ are vacuum.
\end{remark}

\begin{proof}
	
	For the first part,
	by \eqref{basic-x-far}, $d_*<\frac d2 +1$,
	\begin{align*}%\label{far-region-general}
			&
			\int_0^t v(s)  (t-s)^{-d_*} \int_{\RR^d}
			e^{ - c (\frac{|x-y|}{\sqrt{t-s}} )^{p} }
			|y-q|^{-b} \1_{ \{ |y-q| \ge c_1 (T-s)^{\frac 12} \} }
			dy ds
			\\
			\lesssim \ &
			\Big( \int_0^{[t-(T-t)]_+}  + \int_{[t-(T-t)]_+}^t \Big)  v(s) (t-s)^{\frac d2 -d_* } (T-s)^{ -\frac b2 }  ds
			\\
			\lesssim \ &
			\int_0^{[t-(T-t)]_+}  v(s) (T-s)^{\frac d2 -d_* -\frac b2}  ds
			+ \sup\limits_{t_1\in  [ [t-(T-t)]_+, t  ] }v(t_1) (T-t)^{1-d_* +\frac d2 -\frac b2 }.
		\end{align*}
	
	For the second part,
	by \eqref{basic cal in x},
	\begin{align*}
		%	\begin{equation*}
			%		\begin{aligned}
				&
				\int_0^t  v(s) (t-s)^{-d_*} \int_{\RR^d}
				e^{ -c  (\frac{|x-y|}{\sqrt{t-s}} )^{p} }
				|y-q|^{-b} \1_{ \{ l_1(s) \le |y-q| \le l_2(s) \} } dy ds
				\\
				\lesssim \ &
				\int_0^t  v(s)
				\begin{cases}
					(t-s)^{\frac d2 -d_* } l_1^{-b}(s)
					&
					\mbox{ \ if \ } t-s \le l_1^2(s)
					\\
					\begin{cases}
						(t-s)^{\frac d2 -\frac{b}{2} -d_* }
						&
						\mbox{ \ if \ } b<d
						\\
						(t-s)^{ -d_*} 	\langle \ln(\frac{t-s}{l_1^2(s) } ) \rangle
						&
						\mbox{ \ if \ } b=d
						\\
						(t-s)^{ -d_*}  l_1^{d-b}(s)
						&
						\mbox{ \ if \ } b>d
					\end{cases}
					&
					\mbox{ \ if \ }  l_1^2(s) < t-s \le  l_2^2(s)
					\\
					\begin{cases}
						(t-s)^{ -d_*}  l_2^{d-b}(s)
						&
						\mbox{ \ if \ } b<d
						\\
						(t-s)^{ -d_*}  \langle \ln(\frac{ l_2(s) }{l_1(s) } ) \rangle
						&
						\mbox{ \ if \ } b=d
						\\
						(t-s)^{ -d_*}  l_1^{d-b}(s)
						&
						\mbox{ \ if \ } b>d
					\end{cases}
					&
					\mbox{ \ if \ }  t-s > l_2^2(s)
				\end{cases}
				ds
				\\
			= \ &  \int_0^{[t-(T-t)]_+} + \int_{[t-(T-t)]_+}^{t} \cdots	
			:=
				P_1 + P_2.
				%	\end{aligned}
			%	\end{equation*}
	\end{align*}
	
	For $P_{1}$, since $(T-s)/2 \le t-s \le T-s$, $l_2^2(s) \le C(T-s) $, we have $P_1
	\lesssim
	\tilde{P}_1 $.
	
	For $P_2$, since $T-t\le T-s \le 2(T-t)$, $d_*<\frac{d}{2}+1$, we have
	\begin{equation*}
		P_2 \lesssim
		\sup\limits_{t_1\in  [ [t-(T-t)]_+, t  ] }v(t_1)
		\int_{[t-(T-t)]_+}^t
		\begin{cases}
			(t-s)^{\frac d2 -d_* } l_1^{-b}(t)
			&
			\mbox{if \ } t-s \le l_1^2(t)
			\\
			\begin{cases}
				(t-s)^{\frac d2 -\frac{b}{2} -d_* }
				&
				\mbox{if \ } b<d
				\\
				(t-s)^{ -d_*} 	\langle \ln(\frac{t-s}{l_1^2(t) } ) \rangle
				&
				\mbox{if \ } b=d
				\\
				(t-s)^{ -d_*}  l_1^{d-b}(t)
				&
				\mbox{if \ } b>d
			\end{cases}
			&
			\mbox{if \ }  l_1^2(t) < t-s \le  l_2^2(t)
			\\
			\begin{cases}
				(t-s)^{ -d_*}  l_2^{d-b}(t)
				&
				\mbox{if \ } b<d
				\\
				(t-s)^{ -d_*}  \langle \ln(\frac{ l_2(t) }{l_1(t) } ) \rangle
				&
				\mbox{if \ } b=d
				\\
				(t-s)^{ -d_*}  l_1^{d-b}(t)
				&
				\mbox{if \ } b>d
			\end{cases}
			&
			\mbox{if \ }  t-s > l_2^2(t)
		\end{cases}
		ds
		\lesssim
		\tilde{P}_2
	\end{equation*}
	by \eqref{t-int-del=0} (used for the second ``$\lesssim$'') and $l_2^2(t)\le C(T-t)$.
\end{proof}

\subsection{Convolution involving $v(t) |x-q|^{-b} \1_{ \{ l_1(t) \le |x-q| \le l_2(t) \} } $ }

\begin{prop}\label{qd24Apr11-02-prop}
	
	Let $d\ge 1$ be an integer, $b\ge 0$, $0\le t<T$, $q\in \mathbb{R}^d$. Given $\Gamma(x,t,y,s)$ in Proposition \ref{funda-prop} and $
	|f(y,s)| \le v(s) |y-q|^{-b} \1_{ \{ l_1(s) \le |y-q| \le l_2(s) \} } $ for $(y,s)\in \mathbb{R}^d\times (0,T)$
	with functions $v, l_{1}, l_2$ satisfying \eqref{23Nov02-1}, denote
	$  \mathcal{T}_{d}^{\rm{out} }[f] (x,t) :=
	\int_{0}^{t} \int_{\RR^d}
	\Gamma(x,t,y,s) f(y,s) dyds $.
	Then, using the convention $C_1 / C_2=1$ if $C_1=C_2=0$, we have
	\begin{equation}\label{bound-annular}
		\begin{aligned}
			|\mathcal{T}_{d}^{\rm{out} }[f] (x,t)|
			\lesssim \ &
			\int_0^{[t-(T-t)]_+} v(s) (T-s)^{ -\frac{d}{2}}
			\begin{cases}
				l_2^{d-b}(s)
				&
				\mbox{ \ if \ } b<d
				\\
				\langle \ln(\frac{l_2(s)}{l_1(s)}) \rangle
				&
				\mbox{ \ if \ } b=d
				\\
				l_1^{d-b}(s)
				&
				\mbox{ \ if \ } b>d
			\end{cases}
			ds
			\\
			& +
			\sup\limits_{t_1\in  [ [t-(T-t)]_+, t  ] }v(t_1)
			\begin{cases}
				\begin{cases}
					(T-t)^{1-\frac{d}{2}} l_2^{d-b}(t)
					&
					\mbox{ \ if \ } d<2
					\\
					l_2^{2-b}(t) \langle \ln (\frac{T-t}{l_2^2(t)} ) \rangle
					&
					\mbox{ \ if \ } d=2
					\\
					l_2^{2-b }(t)
					&
					\mbox{ \ if \ }
					d>2,b<2
					\\
					\langle \ln(\frac{l_2(t)}{l_1(t)} ) \rangle
					&
					\mbox{ \ if \ } b=2
					\\
					l_1^{ 2 -b }(t)
					&
					\mbox{ \ if \ } b>2
				\end{cases}
				&
				\mbox{ \ if \ } b<d
				\\
				\begin{cases}
					(T-t)^{1-\frac{d}{2}} \langle \ln( \frac{l_2(t)}{l_1(t)}) \rangle
					&
					\mbox{ \ if \ } d<2
					\\
					\langle \ln (\frac{T-t}{l_1^2(t)} ) \rangle  \langle  \ln( \frac{l_2(t)}{l_1(t)}) \rangle
					&
					\mbox{ \ if \ } d=2
					\\
					l_1^{2-d }(t)
					&
					\mbox{ \ if \ } d>2
				\end{cases}
				&
				\mbox{ \ if \ } b=d
				\\
				\begin{cases}
					(T-t)^{1-\frac{d}{2}} l_1^{d-b}(t)
					&
					\mbox{ \ if \ } d<2
					\\
					l_1^{2-b}(t)
					\langle \ln (\frac{T-t}{l_1^2(t)} )  \rangle
					&
					\mbox{ \ if \ } d=2
					\\
					l_1^{2-b}(t)
					&
					\mbox{ \ if \ } d>2
				\end{cases}
				&
				\mbox{ \ if \ } b>d .
			\end{cases}
		\end{aligned}
	\end{equation}
	\begin{equation}\label{nabla-annular}
		\begin{aligned}
			|\nabla \mathcal{T}_{d}^{\rm{out} }[f](x,t)|
			\lesssim \ &
			\int_0^{[t-(T-t)]_+} v(s) (T-s)^{ -\frac{d+1}{2}  }
			\begin{cases}
				l_2^{d-b}(s)
				&
				\mbox{ \ if \ } b<d
				\\
				\langle \ln(\frac{l_2(s)}{l_1(s)}) \rangle
				&
				\mbox{ \ if \ } b=d
				\\
				l_1^{d-b}(s)
				&
				\mbox{ \ if \ } b>d
			\end{cases}
			ds
			\\
			& +
			\sup\limits_{t_1\in  [ [t-(T-t)]_+, t  ] }v(t_1)
			\begin{cases}
				\begin{cases}
					l_2^{1-b}(t) \langle \ln (\frac{T-t}{l_2^2(t)} ) \rangle
					&
					\mbox{ \ if \ } d=1
					\\
					l_2^{1-b }(t)
					&
					\mbox{ \ if \ } d>1, b<1
					\\
					\langle \ln(\frac{l_2(t)}{l_1(t)} ) \rangle
					&
					\mbox{ \ if \ } b=1
					\\
					l_1^{1-b }(t)
					&
					\mbox{ \ if \ } b>1
				\end{cases}
				&
				\mbox{ \ if \ } b<d
				\\
				\begin{cases}
					\langle \ln (\frac{T-t}{l_1^2(t)} ) \rangle  \langle  \ln( \frac{l_2(t)}{l_1(t)}) \rangle
					&
					\mbox{ \ if \ } d=1
					\\
					l_1^{1-d}(t)
					&
					\mbox{ \ if \ } d>1
				\end{cases}
				&
				\mbox{ \ if \ } b=d
				\\
				\begin{cases}
					l_1^{1-b}(t)
					\langle \ln (\frac{T-t}{l_1^2(t)} )  \rangle
					&
					\mbox{ \ if \ } d=1
					\\
					l_1^{1-b}(t)
					&
					\mbox{ \ if \ } d>1
				\end{cases}
				&
				\mbox{ \ if \ } b>d .
			\end{cases}
		\end{aligned}
	\end{equation}
	\begin{equation}\label{T-t-annular}
		\left|\mathcal{T}_{d}^{\rm{out} }[f] (x,t) - \mathcal{T}_{d}^{\rm{out} }[f] (x,T)\right|
		\lesssim
		\tilde{T}_{31}
		+
		\tilde{T}_{32}
		+
		\tilde{T}_{33} ,
	\end{equation}
	where
	\begin{equation*}
		\begin{aligned}
			& \tilde{T}_{31} :=
			(T-t)
			\int_{0}^{[t-(T-t)]_+ } v(s) (T-s)^{-1-\frac d2}
			\begin{cases}
				l_2^{d-b}(s)
				&
				\mbox{ \ if \ } b<d
				\\
				\langle \ln(\frac{ l_2(s) }{l_1(s) } ) \rangle
				&
				\mbox{ \ if \ } b=d
				\\
				l_1^{d-b}(s)
				&
				\mbox{ \ if \ } b>d
			\end{cases}
			ds,
			\\
			& \tilde{T}_{32} :=
			\sup\limits_{t_1\in  [ [t-(T-t)]_+, t  ] }v(t_1)
			\int_{[t-(T-t)]_+}^{t}
			\begin{cases}
				l_1^{-b}(t)
				&
				\mbox{ \ if \ } t-s \le l_1^2(t)
				\\
				\begin{cases}
					(t-s)^{ -\frac{b}{2}}
					&
					\mbox{ \ if \ } b<d
					\\
					(t-s)^{-\frac d2}  	\langle \ln(\frac{t-s}{l_1^2(t) } ) \rangle
					&
					\mbox{ \ if \ } b=d
					\\
					l_1^{d-b}(t) (t-s)^{-\frac d2}
					&
					\mbox{ \ if \ } b>d
				\end{cases}
				&
				\mbox{ \ if \ }  l_1^2(t) < t-s \le  l_2^2(t)
				\\
				\begin{cases}
					l_2^{d-b}(t) 	(t-s)^{-\frac d2}
					&
					\mbox{ \ if \ } b<d
					\\
					\langle \ln(\frac{ l_2(t) }{l_1(t) } ) \rangle
					(t-s)^{-\frac d2}
					&
					\mbox{ \ if \ } b=d
					\\
					l_1^{d-b}(t) (t-s)^{-\frac d2}
					&
					\mbox{ \ if \ } b>d
				\end{cases}
				&
				\mbox{ \ if \ }  t-s > l_2^2(t)
			\end{cases}
			ds
			\\
			& \qquad\qquad +
			\sup\limits_{t_1\in  [ [t-(T-t)]_+, t  ] }v(t_1) (T-t)^{1-\frac d2}
			\begin{cases}
				l_2^{d-b}(t)
				&
				\mbox{ \ if \ } b<d
				\\
				\langle \ln(\frac{ l_2(t) }{l_1(t) } ) \rangle
				&
				\mbox{ \ if \ } b=d
				\\
				l_1^{d-b}(t)
				&
				\mbox{ \ if \ } b>d ,
			\end{cases}
			\\
			&
			\tilde{T}_{33} :=
			\int_{t}^{T}  (T-s)^{-\frac d2} v(s)
			\begin{cases}
				l_2^{d-b}(s)
				&
				\mbox{ \ if \ } b<d
				\\
				\langle \ln(\frac{ l_2(s) }{l_1(s) } ) \rangle
				&
				\mbox{ \ if \ } b=d
				\\
				l_1^{d-b}(s)
				&
				\mbox{ \ if \ } b>d
			\end{cases}
			ds.
		\end{aligned}
	\end{equation*}
	
	For $0<\alpha<1$,
	\begin{equation}\label{nablaT-t-annular}
		\left|\nabla \mathcal{T}_{d}^{\rm{out} }[f](x,t) - \nabla  \mathcal{T}_{d}^{\rm{out} }[f](x,T) \right|
		\lesssim  C(\alpha)
		\big(
		\tilde{T}_{41}  +  \tilde{T}_{42}
		\big) + \tilde{T}_{43},
	\end{equation}
	where
	\begin{align*}
		%\begin{equation*}
		%\begin{aligned}
		&
		\tilde{T}_{41}:= \left(T-t\right)^{\frac{\alpha}{2}}
		\int_0^{[t-(T-t)]_+} v(s) (T-s)^{ - \frac {d+1+\alpha}2  }
		\begin{cases}
			l_2^{d-b}(s)
			&
			\mbox{ \ if \ } b<d
			\\
			\langle  \ln(\frac{l_2(s)}{l_1(s)}) \rangle
			&
			\mbox{ \ if \ } b=d
			\\
			l_1^{d-b}(s)
			&
			\mbox{ \ if \ } b>d
		\end{cases}
		ds,
		\\
		&
		\tilde{T}_{42} :=\sup\limits_{t_1\in  [ [t-(T-t)]_+, t  ] }v(t_1)
		\begin{cases}
			\begin{cases}
				\langle \ln(\frac{T-t}{l_2^2(t)}) \rangle  l_2^{1-b}(t)
				&
				\mbox{ \ if \ } d=1
				\\
				l_2^{ 1-b }(t)
				&
				\mbox{ \ if \ }
				d>1, b<1
				\\
				\langle \ln(\frac{l_2(t)}{l_1(t)}) \rangle
				&
				\mbox{ \ if \ } b=1
				\\
				l_1^{ 1 -b }(t)
				&
				\mbox{ \ if \ } b>1
			\end{cases}
			&
			\mbox{ \ if \ } b<d
			\\
			\begin{cases}
				\langle \ln(\frac{T-t}{l_1^2(t)}) \rangle \langle \ln(\frac{ l_2(t) }{l_1(t) } ) \rangle
				&
				\mbox{ \ if \ } d=1
				\\
				l_1^{1-d}(t)
				&
				\mbox{ \ if \ } d>1
			\end{cases}
			&
			\mbox{ \ if \ } b=d
			\\
			\begin{cases}
				\langle\ln(\frac{T-t}{l_1^2(t)})  \rangle
				l_1^{1-b}(t)
				&
				\mbox{ \ if \ } d=1
				\\
				l_1^{1-b }(t)
				&
				\mbox{ \ if \ } d>1
			\end{cases}
			&
			\mbox{ \ if \ } b>d ,
		\end{cases}
		\\
		&
		\tilde{T}_{43}:=
		\int_{t}^{T} v(s)  (T-s)^{-\frac {d+1}2}
		\begin{cases}
			l_2^{d-b}(s)
			&
			\mbox{ \ if \ } b<d
			\\
			\langle \ln(\frac{ l_2(s) }{l_1(s) } ) \rangle
			&
			\mbox{ \ if \ } b=d
			\\
			l_1^{d-b}(s)
			&
			\mbox{ \ if \ } b>d
		\end{cases}
		ds.
		%\end{aligned}
		%\end{equation*}
	\end{align*}
	
	For $0<\alpha<1$ and $t<t_* \le (T+t)/2$, suppose additional  assumption
	\begin{equation}\label{qd24Apr11-1}
			C_l^{-1} l_i(t) \le l_i(s) \le C_l l_i(t) \mbox{ \ for \ } i=1,2, \quad s\in [t,(T+t)/2 ],
	\end{equation}
	we have
	\begin{equation}\label{nablat*-t-annular}
			|\nabla \mathcal{T}_{d}^{\rm{out} }[f](x,t) - \nabla \mathcal{T}_{d}^{\rm{out} }[f](x_*,t_*) |
			\lesssim
			C(\alpha)
			(| x - x_*| + \sqrt{| t- t_*|} )^{\alpha}
			( \tilde{T}_{1}^{1} + \tilde{T}_{2}^{1} + \tilde{T}_{3}^{1}),
	\end{equation}
	where for $\gamma \in \mathbb{R}$, we define
	\begin{align*}
	&	\tilde{T}_{1}^{\gamma} := \int_0^{[t-(T-t)]_+} v(s) (T-s)^{ - \frac {d+\gamma+\alpha}2  }
		\begin{cases}
			l_2^{d-b}(s)
			&
			\mbox{ \ if \ } b<d
			\\
			\langle  \ln(\frac{l_2(s)}{l_1(s)}) \rangle
			&
			\mbox{ \ if \ } b=d
			\\
			l_1^{d-b}(s)
			&
			\mbox{ \ if \ } b>d
		\end{cases}
		ds,
\\
&
\tilde{T}_{2}^{\gamma}:= \sup\limits_{t_1\in [ [t-(T-t)]_+, [t -(t_*-t)]_+  ] }v(t_1)
\begin{cases}
	l_2^{2-\gamma-b-\alpha}(t)
	&
	\mbox{ \ if \ } b< 2-\gamma-\alpha
	\\
	\langle \ln (\frac{l_2(t)}{l_1(t)}) \rangle
	&
	\mbox{ \ if \ } b= 2-\gamma-\alpha
	\\
	l_1^{2-\gamma-b-\alpha}(t)
	&
	\mbox{ \ if \ } b>2-\gamma-\alpha,
\end{cases}
\\
&
\tilde{T}_{3}^{\gamma}:=  \sup\limits_{t_1\in [[t-(t_*-t) ]_+,t_*]}v(t_1)
\begin{cases}
	l_2^{2-\gamma-b-\alpha}(t)
	& \mbox{ \ if \ }
	b\le 2-\gamma-\alpha
	\\
	l_1^{2-\gamma-b-\alpha}(t)
	& \mbox{ \ if \ }
	b>2-\gamma-\alpha.
\end{cases}
	\end{align*}

	For $0 <\alpha <1$ and $t<t_* \le (T+t)/2$, suppose $d>2-\alpha$, \eqref{qd24Apr11-1} additionally, we have
	\begin{equation}\label{qd24July17-1}
		|\mathcal{T}_{d}^{\rm{out} }[f](x,t) -  \mathcal{T}_{d}^{\rm{out} }[f](x_*,t_*)|
		\lesssim
		C(\alpha)
		(| x - x_*| + \sqrt{| t- t_*|} )^{\alpha}
		( \tilde{T}_{1}^{0} + \tilde{T}_{2}^{0} + \tilde{T}_{3}^{0}).
	\end{equation}
	All ``$\lesssim$'' above are independent of $T$.
\end{prop}

\begin{proof}[Proof of \eqref{bound-annular}, \eqref{nabla-annular}]
	\eqref{bound-annular} and \eqref{nabla-annular} are derived by \eqref{Ga-est} and \eqref{annular-lem}.
\end{proof}

\begin{proof}[Proof of \eqref{T-t-annular}]
	\begin{align*}
		%	\begin{equation*}
			%		\begin{aligned}
				& \mathcal{T}_{d}^{\rm{out} }[f](x,t) - \mathcal{T}_{d}^{\rm{out} }[f](x,T) = \int_{0}^{[t-(T-t)]_+} \int_{\RR^d}
				(\Gamma(x,t,y,s) - \Gamma(x,T,y,s) )
				f(y,s) dyds
				\\
				& +
				\int_{[t-(T-t)]_+}^{t} \int_{\RR^d}
				(\Gamma(x,t,y,s) - \Gamma(x,T,y,s) )
				f(y,s) dyds
				-
				\int_{t}^{T} \int_{\RR^d}
				\Gamma(x,T,y,s)
				f(y,s) dyds:= I_1 + I_2 + I_3.
				%		\end{aligned}
			%	\end{equation*}
	\end{align*}
	
	By \eqref{Ga-est}, one has
	\begin{align*}
%	\begin{equation*}
%		\begin{aligned}
			&	|I_1| \lesssim
			(T-t)
			\int_{0}^{[t-(T-t)]_+} \int_{\RR^d}
			\int_0^1  \left| (\pp_{t} \Gamma)(x,\theta t + (1-\theta)T,y,s)  \right|
			|f(y,s)| dy ds
			\\
			\lesssim \ &
			(T-t)
			\int_{0}^{[t-(T-t)]_+ } \int_{\RR^d}
			\int_0^1
			\left[\theta t + (1-\theta)T-s\right]^{-1-\frac d2}
			e^{ -c  \big(\frac{|x-y|}{\sqrt{\theta t + (1-\theta)T-s}} \big)^{2-\delta} }
			v(s) |y-q|^{-b} \1_{ \{ l_1(s) \le |y-q| \le l_2(s) \} }
			d\theta
			dy ds
			\\
			\lesssim \ &
			(T-t)
			\int_{0}^{[t-(T-t)]_+ } v(s) (T-s)^{-1-\frac d2}  \int_{\RR^d}
			e^{ -c  \big(\frac{|x-y|}{\sqrt{T-s}} \big)^{2-\delta} }
			|y-q|^{-b} \1_{ \{ l_1(s) \le |y-q| \le l_2(s) \} }
			dy ds
			\lesssim
			\tilde{T}_{31},
%		\end{aligned}
%	\end{equation*}
	\end{align*}
	where we used \eqref{tsplit} for the third ``$\lesssim$''; \eqref{basic cal in x}, $l_2^2(s) \le C(T-s)$ for the fourth ``$\lesssim$''.
	
	By \eqref{Ga-est}, we estimate
\begin{align*}
			|I_2|
			\lesssim \ &
			\int_{[t-(T-t)]_+}^{t} v(s) (t-s)^{-\frac d2}  \int_{\RR^d}
			e^{ -c  \big(\frac{|x-y|}{\sqrt{t-s}} \big)^{2-\delta} }
			|y-q|^{-b} \1_{ \{ l_1(s) \le |y-q| \le l_2(s) \} }
			dy ds
			\\
			& +
			\int_{[t-(T-t)]_+}^{t} v(s) (T-s)^{-\frac d2}  \int_{\RR^d}
			e^{ -c  \big(\frac{|x-y|}{\sqrt{T-s}} \big)^{2-\delta} }
			|y-q|^{-b} \1_{ \{ l_1(s) \le |y-q| \le l_2(s) \} }
			dy ds
			\\
			\lesssim \ &
			\int_{[t-(T-t)]_+}^{t} v(s) (t-s)^{-\frac d2}
			\begin{cases}
				(t-s)^{\frac d2 } l_1^{-b}(s)
				&
				\mbox{ \ if \ } t-s \le l_1^2(s)
				\\
				\begin{cases}
					(t-s)^{\frac d2 -\frac{b}{2}}
					&
					\mbox{ \ if \ } b<d
					\\
					\langle \ln(\frac{t-s}{l_1^2(s) } ) \rangle
					&
					\mbox{ \ if \ } b=d
					\\
					l_1^{d-b}(s)
					&
					\mbox{ \ if \ } b>d
				\end{cases}
				&
				\mbox{ \ if \ }  l_1^2(s) < t-s \le  l_2^2(s)
				\\
				\begin{cases}
					l_2^{d-b}(s)
					&
					\mbox{ \ if \ } b<d
					\\
					\langle \ln(\frac{ l_2(s) }{l_1(s) } ) \rangle
					&
					\mbox{ \ if \ } b=d
					\\
					l_1^{d-b}(s)
					&
					\mbox{ \ if \ } b>d
				\end{cases}
				&
				\mbox{ \ if \ }  t-s > l_2^2(s)
			\end{cases}
			ds
			\\
			& +
			\int_{[t-(T-t)]_+}^{t} v(s) (T-s)^{-\frac d2}
			\begin{cases}
				l_2^{d-b}(s)
				&
				\mbox{ \ if \ } b<d
				\\
				\langle \ln(\frac{ l_2(s) }{l_1(s) } ) \rangle
				&
				\mbox{ \ if \ } b=d
				\\
				l_1^{d-b}(s)
				&
				\mbox{ \ if \ } b>d
			\end{cases}
			ds
			\
			\lesssim
			\tilde{T}_{32},
		\end{align*}
	where we used \eqref{basic cal in x}, $l_2^2 (s)\le C(T-s)$ in the second ``$\lesssim$''; and \eqref{tsplit} in the third ``$\lesssim$''.
	
	By \eqref{Ga-est}, \eqref{basic cal in x}, and $l_2^2(s)\le C(T-s)$, we have $ |I_3|
	\lesssim
	\tilde{T}_{33}$.
\end{proof}

\begin{proof}[Proof of \eqref{nablaT-t-annular}]
	\begin{align*}
			&
			\pp_{x_i}\mathcal{T}_{d}^{\rm{out} }[f](x,t) - \pp_{x_i}\mathcal{T}_{d}^{\rm{out} }[f](x,T)
			\\
			= \ &
			\int_{0}^{t} \int_{\RR^d}
			( \pp_{x_i} \Gamma(x,t,y,s) - \pp_{x_i} \Gamma(x,T,y,s)  ) f(y,s) dyds
			-
			\int_{t}^{T} \int_{\RR^d}
			\pp_{x_i} \Gamma(x,T,y,s) f(y,s) dyds
			:= I_1 + I_2.
		\end{align*}
	
	For $I_1$, by \eqref{nabla-Ga-Holder} and \eqref{basic cal in x},
	\begin{align*}
		%	\begin{equation*}
			%		\begin{aligned}
				| I_1 |
				\lesssim \ &
				C(\alpha) (T-t)^{\frac{\alpha}{2}}
				\int_{0}^{t} \int_{\RR^d}
				(T-s)^{-\frac{\alpha}{2} }
				\big[
				(t-s)^{-\frac {d+1}2}
				e^{ -c  (\frac{|x-y|}{\sqrt{t-s}} )^{2-\delta} }
				+
				(T-s)^{-\frac {d+1}2}
				e^{ -c (\frac{|x-y|}{\sqrt{T-s}} )^{2-\delta} }
				\big]
				\\
				& \times
				v(s) |y-q|^{-b} \1_{ \{ l_1(s) \le |y-q| \le l_2(s) \} }
				dyds
				\lesssim  C(\alpha)
				(T-t)^{\frac{\alpha}{2}}
				\\
				& \times
				\Bigg[ \int_{0}^{t} v(s) (T-s)^{-\frac{\alpha}{2} }
				\begin{cases}
					(t-s)^{-\frac {1}2}  l_1^{-b}(s)
					&
					\mbox{if \ } t-s \le l_1^2(s)
					\\
					\begin{cases}
						(t-s)^{ -\frac{b+1}{2}}
						&
						\mbox{ \ if \ } b<d
						\\
						(t-s)^{-\frac {d+1}2}  \langle \ln(\frac{t-s}{l_1^2(s) } ) \rangle
						&
						\mbox{ \ if \ } b=d
						\\
						(t-s)^{-\frac {d+1}2}  l_1^{d-b}(s)
						&
						\mbox{ \ if \ } b>d
					\end{cases}
					&
					\mbox{if \ }  l_1^2(s) < t-s \le  l_2^2(s)
					\\
					\begin{cases}
						(t-s)^{-\frac {d+1}2}  l_2^{d-b}(s)
						&
						\mbox{ \ if \ } b<d
						\\
						(t-s)^{-\frac {d+1}2}  \langle \ln(\frac{ l_2(s) }{l_1(s) } ) \rangle
						&
						\mbox{ \ if \ } b=d
						\\
						(t-s)^{-\frac {d+1}2}  l_1^{d-b}(s)
						&
						\mbox{ \ if \ } b>d
					\end{cases}
					&
					\mbox{if \ }  t-s > l_2^2(s)
				\end{cases}
				ds
				\\
				&\qquad +
				\int_{0}^{t} v(s)
				\begin{cases}
					(T-s)^{-\frac{d+1+\alpha}{2} }  l_2^{d-b}(s)
					&
					\mbox{ \ if \ } b<d
					\\
					(T-s)^{-\frac{d+1+\alpha}{2} }  \langle \ln(\frac{ l_2(s) }{l_1(s) } ) \rangle
					&
					\mbox{ \ if \ } b=d
					\\
					(T-s)^{-\frac{d+1+\alpha}{2} }  l_1^{d-b}(s)
					&
					\mbox{ \ if \ } b>d
				\end{cases}
				ds
				\Bigg]
				\\
				= \ & C(\alpha)(T-t)^{\frac{\alpha}{2}}
				\Big(
				\int_{0}^{[t-(T-t)]_+}   + \int_{[t-(T-t)]_+}^t \cdots
				\Big)
				:= C(\alpha)
				(T-t)^{\frac{\alpha}{2}} ( I_{11} + I_{12}  ).
				%		\end{aligned}
			%	\end{equation*}
	\end{align*}
	
	For $I_{11}$, by $t-s\sim T-s\gtrsim l_2^2(s)$, we have
	$I_{11}
	\lesssim  \left(T-t\right)^{-\frac{\alpha}{2}} \tilde{T}_{41} $.

	For $I_{12}$, by \eqref{23Nov02-1}, \eqref{t-int-del=0},
\begin{small}
\begin{align*}
			I_{12}
			\lesssim \ &
			\sup\limits_{t_1\in  [ [t-(T-t)]_+, t  ] }v(t_1) (T-t)^{-\frac{\alpha}{2} }
			\int_{[t-(T-t)]_+}^{t}
			\begin{cases}
				(t-s)^{-\frac {1}{2} }	 l_1^{-b}(t)
				&
				\mbox{if \ } t-s\le l_1^2(t)
				\\
				\begin{cases}
					(t-s)^{-\frac {1}{2} -\frac b2}
					&
					\mbox{ \ if \ } b<d
					\\
					(t-s)^{-\frac {1}{2} -\frac d2}	\langle \ln(\frac{t-s}{l_1^2(t)} ) \rangle
					&
					\mbox{ \ if \ } b=d
					\\
					(t-s)^{-\frac {1}{2} -\frac d2}	l_1^{d-b}(t)
					&
					\mbox{ \ if \ } b>d
				\end{cases}
				&
				\mbox{if \ }  l_1^2(t) < t-s\le  l_2^2(t)
				\\
				\begin{cases}
					(t-s)^{-\frac {1}{2} -\frac d2}	 l_2^{d-b}(t)
					&
					\mbox{ \ if \ } b<d
					\\
					(t-s)^{-\frac {1}{2} -\frac d2}	\langle \ln(\frac{ l_2(t) }{l_1(t)} ) \rangle
					&
					\mbox{ \ if \ } b=d
					\\
					(t-s)^{-\frac {1}{2} -\frac d2}	l_1^{d-b}(t)
					&
					\mbox{ \ if \ } b>d
				\end{cases}
				&
				\mbox{if \ }  t-s > l_2^2(t)
			\end{cases}
			ds
			\\
			& +
			\sup\limits_{t_1\in  [ [t-(T-t)]_+, t  ] }v(t_1)  (T-t)^{ -\frac {\alpha}2 }
			\begin{cases}
				(T-t)^{\frac 12 -\frac d2} 	 l_2^{d-b}(t)
				&
				\mbox{ \ if \ } b<d
				\\
				(T-t)^{\frac 12 -\frac d2} 	\langle \ln(\frac{ l_2(t) }{l_1(t)} ) \rangle
				&
				\mbox{ \ if \ } b=d
				\\
				(T-t)^{\frac 12  -\frac d2}  l_1^{d-b}(t)
				&
				\mbox{ \ if \ } b>d
			\end{cases}
			\
			\lesssim   (T-t)^{-\frac{\alpha}{2} } \tilde{T}_{42}.
		\end{align*}
	\end{small}

	For $I_2$, by \eqref{Ga-est}, \eqref{basic cal in x}, we have $|I_2| \lesssim \tilde{T}_{43}$.
\end{proof}

\begin{proof}[Proof of \eqref{nablat*-t-annular} and \eqref{qd24July17-1}]
	
	For brevity, denote $|X-X_*|=| x - x_*| + \sqrt{| t- t_*|}$.
	By \eqref{nabla-Ga-Holder}, \eqref{Ga-est}, \eqref{Ga-Holder},
	\begin{equation*}
		\begin{aligned}
			&
			\big| \pp_{x_i}\mathcal{T}_{d}^{\rm{out} }[f](x,t) - \pp_{x_i}\mathcal{T}_{d}^{\rm{out} }[f](x_*,t_*)  \big|
			\\
			\le \ &
			\int_{0}^{t} \int_{\RR^d}
			| \pp_{x_i} \Gamma(x,t,y,s) - \pp_{x_i} \Gamma(x_*,t_*,y,s)  | |f(y,s)| dyds
			+
			\int_{t}^{t_*} \int_{\RR^d}
			|\pp_{x_i} \Gamma(x_*,t_*,y,s)| |f(y,s)| dyds
			\\
			\lesssim \ &
			C(\alpha)
			|X-X_*|^{\alpha}
			\int_{0}^{t} v(s) \int_{\RR^d}
			(t_*-s)^{-\frac{\alpha}{2} }
			\Big[
			(t-s)^{-\frac {d+1}2}
			e^{ -c (\frac{|x-y|}{\sqrt{t-s}} )^{2-\delta} }
			+
			(t_*-s)^{-\frac {d+1}2}
			e^{ -c (\frac{|x_*-y|}{\sqrt{t_*-s}} )^{2-\delta} }
			\Big] |y-q|^{-b}
			\\
			& \times
			  \1_{ \{ l_1(s) \le |y-q| \le l_2(s) \} }
			dy ds
			 +
			\int_{t}^{t_*} v(s) (t_*-s)^{-\frac {d+1}2}  \int_{\RR^d}
			e^{ -c (\frac{|x_*-y|}{\sqrt{t_*-s}} )^{2-\delta} }
			|y-q|^{-b} \1_{ \{ l_1(s) \le |y-q| \le l_2(s) \} } dy ds,
		\end{aligned}
	\end{equation*}
\begin{equation*}
	\begin{aligned}
		&
		\big| \mathcal{T}_{d}^{\rm{out} }[f](x,t) -  \mathcal{T}_{d}^{\rm{out} }[f](x_*,t_*) \big|
		\\
		\le \ &
		\int_{0}^{t} \int_{\RR^d}
		| \Gamma(x,t,y,s) - \Gamma(x_*,t_*,y,s) | |f(y,s)| dyds
		+
		\int_{t}^{t_*} \int_{\RR^d}
		|\Gamma(x_*,t_*,y,s) | |f(y,s)| dyds
		\\
		\lesssim \ &
		C(\alpha)
		|X-X_*|^{\alpha}
		\int_{0}^{t} v(s) \int_{\RR^d}
		(t_*-s)^{-\frac{\alpha}{2} }
		\Big[
		(t-s)^{-\frac {d}2}
		e^{ -c (\frac{|x-y|}{\sqrt{t-s}} )^{2-\delta} }
		+
		(t_*-s)^{-\frac {d}2}
		e^{ -c (\frac{|x_* -y|}{\sqrt{t_*-s}} )^{2-\delta} }
		\Big] |y-q|^{-b}
		\\
		& \times
		  \1_{ \{ l_1(s) \le |y-q| \le l_2(s) \} }
		dy ds
		 +
		\int_{t}^{t_*} v(s) (t_*-s)^{-\frac {d}2}  \int_{\RR^d}
		e^{ -c (\frac{|x_* -y|}{\sqrt{t_*-s}} )^{2-\delta} }
		|y-q|^{-b} \1_{ \{ l_1(s) \le |y-q| \le l_2(s) \} } dy ds.
	\end{aligned}
\end{equation*}
These two cases can be solved uniformly by the following Lemma.

\begin{lemma}

Given a constant $\gamma$ satisfying $d>2-\gamma-\alpha$, $\gamma+\alpha<2$, denote
\begin{equation*}
	\begin{aligned}
		&
		I_1 := \int_{0}^{t} v(s) \int_{\RR^d}
		(t_*-s)^{-\frac{\alpha}{2} }
		\Big[
		(t-s)^{-\frac {d+\gamma}2}
		e^{ -c (\frac{|x-y|}{\sqrt{t-s}} )^{2-\delta} }
		+
		(t_*-s)^{-\frac {d+\gamma}2}
		e^{ -c (\frac{|x_*-y|}{\sqrt{t_*-s}} )^{2-\delta} }
		\Big] |y-q|^{-b}
		\1_{ \{ l_1(s) \le |y-q| \le l_2(s) \} }
		dy ds,
		\\
		&
		I_2:= \int_{t}^{t_*} v(s) (t_*-s)^{-\frac {d+\gamma}2}  \int_{\RR^d}
		e^{ -c (\frac{|x_*-y|}{\sqrt{t_*-s}} )^{2-\delta} }
		|y-q|^{-b} \1_{ \{ l_1(s) \le |y-q| \le l_2(s) \} } dy ds.
	\end{aligned}
\end{equation*}
Then we have $
I_1 \lesssim \tilde{T}_{1}^{\gamma} + \tilde{T}_{2}^{\gamma} + \tilde{T}_{3}^{\gamma}$. Suppose $t_*-t\le (T-t)/2$, \eqref{qd24Apr11-1} additionally, then
$I_2 \lesssim C(\alpha) (t_* -t)^{\alpha/2} \tilde{T}_{3}^{\gamma}$.
\end{lemma}

The remaining text in this part will be dedicated to proving the above Lemma. For $I_1$, by \eqref{basic cal in x},
	\begin{align*}
%			\begin{equation*}
%					\begin{aligned}
				 I_1
				\lesssim \ &
				\int_{0}^{t} v(s) (t_*-s)^{-\frac{\alpha}{2} }
				\begin{cases}
					(t-s)^{-\frac {\gamma}2}  l_1^{-b}(s)
					&
					\mbox{ \ if \ } t-s \le l_1^2(s)
					\\
					\begin{cases}
						(t-s)^{ -\frac{b+\gamma}{2}}
						&
						\mbox{ \ if \ } b<d
						\\
						(t-s)^{-\frac {d+\gamma}2}  \langle \ln(\frac{t-s}{l_1^2(s) } ) \rangle
						&
						\mbox{ \ if \ } b=d
						\\
						(t-s)^{-\frac {d+\gamma}2}  l_1^{d-b}(s)
						&
						\mbox{ \ if \ } b>d
					\end{cases}
					&
					\mbox{ \ if \ }  l_1^2(s) < t-s \le  l_2^2(s)
					\\
					\begin{cases}
						(t-s)^{-\frac {d+\gamma}2}  l_2^{d-b}(s)
						&
						\mbox{ \ if \ } b<d
						\\
						(t-s)^{-\frac {d+\gamma}2}  \langle \ln(\frac{ l_2(s) }{l_1(s) } ) \rangle
						&
						\mbox{ \ if \ } b=d
						\\
						(t-s)^{-\frac {d+\gamma}2}  l_1^{d-b}(s)
						&
						\mbox{ \ if \ } b>d
					\end{cases}
					&
					\mbox{ \ if \ }  t-s > l_2^2(s)
				\end{cases}
				ds
				\\
				& \qquad+
				\int_{0}^{t} v(s)
				\begin{cases}
					(t_*-s)^{-\frac{\gamma+\alpha}{2} }
					l_1^{-b}(s)
					&
					\mbox{ \ if \ } t_*-s \le l_1^2(s)
					\\
					\begin{cases}
						(t_*-s)^{-\frac{b+\gamma+\alpha}{2} }
						&
						\mbox{ \ if \ } b<d
						\\
						(t_*-s)^{-\frac{d+\gamma+\alpha}{2} }  \langle \ln(\frac{t_*-s}{l_1^2(s) } ) \rangle
						&
						\mbox{ \ if \ } b=d
						\\
						(t_*-s)^{-\frac{d+\gamma+\alpha}{2} }  l_1^{d-b}(s)
						&
						\mbox{ \ if \ } b>d
					\end{cases}
					&
					\mbox{ \ if \ }  l_1^2(s) < t_*-s \le  l_2^2(s)
					\\
					\begin{cases}
						(t_*-s)^{-\frac{d+\gamma+\alpha}{2} }  l_2^{d-b}(s)
						&
						\mbox{ \ if \ } b<d
						\\
						(t_*-s)^{-\frac{d+\gamma+\alpha}{2} }  \langle \ln(\frac{ l_2(s) }{l_1(s) } ) \rangle
						&
						\mbox{ \ if \ } b=d
						\\
						(t_*-s)^{-\frac{d+\gamma+\alpha}{2} }  l_1^{d-b}(s)
						&
						\mbox{ \ if \ } b>d
					\end{cases}
					&
					\mbox{ \ if \ }  t_*-s > l_2^2(s)
				\end{cases}
				ds
				\\
				= \ &
				\int_{0}^{[t-(T-t)]_+}  +
				\int_{[t-(T-t)]_+ }^{[t-(t_*-t)]_+} + \int_{[t-(t_*-t)]_+}^t \cdots
				:=  I_{11} + I_{12} + I_{13}.
%						\end{aligned}
%				\end{equation*}
	\end{align*}

	For $I_{11}$, by \eqref{tsplit}, \eqref{23Nov02-1}, we have $t-s\sim t_*-s\sim T-s\gtrsim l_2^2(s)$. Then $I_{11} \lesssim
\tilde{T}_{1}^{\gamma}$.

	For $I_{12}$, denote ${\rm Int}_{12} := [ [t-(T-t)]_+, [t -(t_*-t)]_+ ]$. By \eqref{tsplit}, \eqref{23Nov02-1}, \eqref{basic cal in x}, $d>2-\gamma-\alpha$, $\gamma+\alpha<2$,
	\begin{align*}
%		\begin{equation*}
%		\begin{aligned}
		I_{12}
		\lesssim \ &
		\sup\limits_{t_1\in {\rm Int}_{12} }v(t_1)
		\int_{[t-(T-t)]_+}^{[t -(t_*-t)]_+ }
		\begin{cases}
			(t-s)^{-\frac {\gamma+\alpha}2 }
			l_1^{-b}(t)
			&
			\mbox{if \ } t-s \le l_1^2(t)
			\\
			\begin{cases}
				(t-s)^{-\frac {b+\gamma+\alpha}2 }
				&
				\mbox{ \ if \ } b<d
				\\
				(t-s)^{-\frac {d+\gamma+\alpha}2 }
				\langle \ln(\frac{t-s}{l_1^2(t) } ) \rangle
				&
				\mbox{ \ if \ } b=d
				\\
				(t-s)^{-\frac {d+\gamma+\alpha}2 }
				l_1^{d-b}(t)
				&
				\mbox{ \ if \ } b>d
			\end{cases}
			&
			\mbox{if \ }  l_1^2(t) < t-s \le  l_2^2(t)
			\\
			\begin{cases}
				(t-s)^{-\frac {d+\gamma+\alpha}2 }
				l_2^{d-b}(t)
				&
				\mbox{ \ if \ } b<d
				\\
				(t-s)^{-\frac {d+\gamma+\alpha}2 }
				\langle \ln(\frac{ l_2(t) }{l_1(t) } ) \rangle
				&
				\mbox{ \ if \ } b=d
				\\
				(t-s)^{-\frac {d+\gamma+\alpha}2 }
				l_1^{d-b}(t)
				&
				\mbox{ \ if \ } b>d
			\end{cases}
			&
			\mbox{if \ }  t-s > l_2^2(t)
		\end{cases}
		ds
		\\
		\lesssim \ &
		\sup\limits_{t_1\in {\rm Int}_{12} }v(t_1)
		\begin{cases}
			\begin{cases}
				l_2^{2-\gamma-b-\alpha}(t)
				&
				\mbox{ \ if \ } b< 2-\gamma-\alpha
				\\
				\langle \ln (\frac{l_2(t)}{l_1(t)}) \rangle
				&
				\mbox{ \ if \ } b= 2-\gamma-\alpha
				\\
				l_1^{2-\gamma-b-\alpha}(t)
				&
				\mbox{ \ if \ } b>2-\gamma-\alpha
			\end{cases}
			&
			\mbox{if \ }
			t_*-t \le l_1^2(t)
			\\
			\begin{cases}
				\begin{cases}
					l_2^{2-\gamma-b-\alpha}(t)
					&
					\mbox{ \ if \ } b< 2-\gamma-\alpha
					\\
					\langle \ln(\frac{l_2^2(t)}{t_*-t}) \rangle
					&
					\mbox{ \ if \ } b= 2-\gamma-\alpha
					\\
					(t_*-t)^{\frac {2-\gamma-b-\alpha}2 }
					&
					\mbox{ \ if \ } b> 2-\gamma-\alpha
				\end{cases}
				&
				\mbox{if \ } b<d
				\\
				(t_*-t)^{\frac{2-\gamma-d-\alpha}{2}}
				\langle \ln (\frac{t_*-t}{l_1^2(t) } )\rangle
				&
				\mbox{if \ } b=d
				\\
				(t_*-t)^{\frac {2-\gamma-d-\alpha}2 }
				l_1^{d-b}(t)
				&
				\mbox{if \ } b>d
			\end{cases}
			&
			\mbox{if \ }
			l_1^2(t) < t_*-t \le  l_2^2(t)
			\\
			(t_*-t)^{\frac{2-\gamma-d-\alpha}{2}}
			\begin{cases}
				l_2^{d-b}(t)
				&
				\mbox{ \ if \ } b<d
				\\
				\langle \ln(\frac{ l_2(t) }{l_1(t) } ) \rangle
				&
				\mbox{ \ if \ } b=d
				\\
				l_1^{d-b}(t)
				&
				\mbox{ \ if \ } b>d
			\end{cases}
			&
			\mbox{if \ } t_*-t > l_2^2(t)
		\end{cases}
	\
	\lesssim \tilde{T}_{2}^{\gamma},
%			\end{aligned}
%	\end{equation*}
	\end{align*}
	where for the second ``$\lesssim$'', we used $d>2-\gamma-\alpha$ and the following calculations. First, we check the case $t_*-t > l_2^2(t)$ directly; Next, for $l_1^2(t) < t_*-t \le  l_2^2(t)$, we split $\int_{[t-(T-t)]_+}^{[t -(t_*-t)]_+ } = \int_{[t-(T-t)]_+}^{ [t-l_2^2(t)]_+ } + \int_{[t-l_2^2(t)]_+}^{[t -(t_*-t)]_+ }$, and then apply the estimate for $t_*-t > l_2^2(t)$ (with $t_*-t = l_2^2(t)$) to $\int_{[t-(T-t)]_+}^{ [t-l_2^2(t)]_+ }$ and direct calculation to $\int_{[t-l_2^2(t)]_+}^{[t -(t_*-t)]_+ }$. It follows the desired upper bound in the second ``$\lesssim$'' for the case $l_1^2(t) < t_*-t \le  l_2^2(t)$;
Finally, the case $ t_*-t \le l_1^2(t)$ is deduced by $\int_{[t-(T-t)]_+}^{[t -(t_*-t)]_+ } =
	\int_{[t-(T-t)]_+}^{[t -l_1^2(t)]_+ } + \int_{[t-l_1^2(t)]_+}^{[t -(t_*-t)]_+ }$, where we used the estimate for $l_1^2(t) < t_*-t \le  l_2^2(t)$ (with $t_*-t = l_1^2(t)$) to $\int_{[t-(T-t)]_+}^{[t -l_1^2(t)]_+ }$, and $\gamma+\alpha<2$ to $\int_{[t-l_1^2(t)]_+}^{[t -(t_*-t)]_+ }$.

	For $I_{13}$, denote ${\rm Int}_{13} := [[t-(t_*-t) ]_+,t ] $. By \eqref{tsplit}, \eqref{23Nov02-1},
	\begin{small}
		\begin{align*}
	%		\begin{equation*}
	 %       \begin{aligned}
			I_{13}
			\lesssim \ &
			\sup\limits_{t_1\in {\rm Int}_{13} }v(t_1) (t_*-t)^{-\frac{\alpha}{2} }
			\int_{[t-(t_*-t) ]_+}^{t}
			\begin{cases}
				(t-s)^{-\frac {\gamma}{2} }	 l_1^{-b}(t)
				&
				\mbox{if \ } t-s\le l_1^2(t)
				\\
				\begin{cases}
					(t-s)^{-\frac {\gamma}{2} -\frac b2}
					&
					\mbox{ \ if \ } b<d
					\\
					(t-s)^{-\frac {\gamma}{2} -\frac d2}	\langle \ln(\frac{t-s}{l_1^2(t)} ) \rangle
					&
					\mbox{ \ if \ } b=d
					\\
					(t-s)^{-\frac {\gamma}{2} -\frac d2}	l_1^{d-b}(t)
					&
					\mbox{ \ if \ } b>d
				\end{cases}
				&
				\mbox{if \ }  l_1^2(t) < t-s\le  l_2^2(t)
				\\
				\begin{cases}
					(t-s)^{-\frac {\gamma}{2} -\frac d2}	 l_2^{d-b}(t)
					&
					\mbox{ \ if \ } b<d
					\\
					(t-s)^{-\frac {\gamma}{2} -\frac d2}	\langle \ln(\frac{ l_2(t) }{l_1(t)} ) \rangle
					&
					\mbox{ \ if \ } b=d
					\\
					(t-s)^{-\frac {\gamma}{2} -\frac d2}	l_1^{d-b}(t)
					&
					\mbox{ \ if \ } b>d
				\end{cases}
				&
				\mbox{if \ }  t-s > l_2^2(t)
			\end{cases}
			ds
			\\
			& +
			\sup\limits_{t_1\in {\rm Int}_{13} }v(t_1)  (t_*-t)^{1-\frac{\gamma + \alpha}2 }
			\begin{cases}
				l_1^{-b}(t)
				&
				\mbox{if \ } t_*-t\le l_1^2(t)
				\\
				\begin{cases}
					(t_*-t)^{-\frac b2}
					&
					\mbox{ \ if \ } b<d
					\\
					(t_*-t)^{-\frac d2} 	\langle \ln(\frac{t_*-t}{l_1^2(t)} ) \rangle
					&
					\mbox{ \ if \ } b=d
					\\
					(t_*-t)^{-\frac d2} 	 l_1^{d-b}(t)
					&
					\mbox{ \ if \ } b>d
				\end{cases}
				&
				\mbox{if \ }  l_1^2(t) < t_*-t\le  l_2^2(t)
				\\
				\begin{cases}
					(t_*-t)^{-\frac d2} 	 l_2^{d-b}(t)
					&
					\mbox{ \ if \ } b<d
					\\
					(t_*-t)^{-\frac d2} 	\langle \ln(\frac{ l_2(t) }{l_1(t)} ) \rangle
					&
					\mbox{ \ if \ } b=d
					\\
					(t_*-t)^{-\frac d2}  l_1^{d-b}(t)
					&
					\mbox{ \ if \ } b>d
				\end{cases}
				&
				\mbox{if \ }  t_*-t > l_2^2(t)
			\end{cases}
			\
			\lesssim
			  \tilde{T}_{3}^{\gamma},
%		\end{aligned}
%	\end{equation*}
		\end{align*}
	\end{small}
	where we used \eqref{t-int-del>0}, $d>2-\gamma-\alpha$, $\gamma+\alpha \le 2$ for the last ``$\lesssim$''. For $I_2$, by \eqref{basic cal in x},
	\begin{equation*}
		I_2
		\lesssim
		\int_{t}^{t_*} v(s)  \begin{cases}
			(t_*-s)^{-\frac {\gamma}2}   l_1^{-b}(s)
			&
			\mbox{ \ if \ } t_*-s \le l_1^2(s)
			\\
			\begin{cases}
				(t_*-s)^{-\frac{\gamma}{2} -\frac{b}{2}}
				&
				\mbox{ \ if \ } b<d
				\\
				(t_*-s)^{-\frac {d+\gamma}2} 	\langle \ln(\frac{t_*-s}{l_1^2(s) } ) \rangle
				&
				\mbox{ \ if \ } b=d
				\\
				(t_*-s)^{-\frac {d+\gamma}2} 	l_1^{d-b}(s)
				&
				\mbox{ \ if \ } b>d
			\end{cases}
			&
			\mbox{ \ if \ }  l_1^2(s) < t_*-s \le  l_2^2(s)
			\\
			\begin{cases}
				(t_*-s)^{-\frac {d+\gamma}2} 	l_2^{d-b}(s)
				&
				\mbox{ \ if \ } b<d
				\\
				(t_*-s)^{-\frac {d+\gamma}2} 	\langle \ln(\frac{ l_2(s) }{l_1(s) } ) \rangle
				&
				\mbox{ \ if \ } b=d
				\\
				(t_*-s)^{-\frac {d+\gamma}2} 	l_1^{d-b}(s)
				&
				\mbox{ \ if \ } b>d
			\end{cases}
			&
			\mbox{ \ if \ }  t_*-s > l_2^2(s)
		\end{cases}  	
	ds.
	\end{equation*}
Suppose $t_*-t\le (T-t)/2$, \eqref{qd24Apr11-1} additionally, similar to $I_{13}$,
	by \eqref{t-int-del>0}, $d>2-\gamma-\alpha$, $\gamma+\alpha \le 2$, then
	\begin{align*}
%	\begin{equation*}
%		\begin{aligned}
			I_2
			\lesssim \ &
			\sup\limits_{t_1\in [t,t_*]} v(t_1) \int_{t_*-(t_*-t)}^{t_*}
			\begin{cases}
				(t_*-s)^{-\frac {\gamma}2}  l_1^{-b}(t)
				&
				\mbox{ \ if \ } t_*-s \le l_1^2(t)
				\\
				\begin{cases}
					(t_*-s)^{-\frac {\gamma}{2} -\frac{b}{2}}
					&
					\mbox{ \ if \ } b<d
					\\
					(t_*-s)^{-\frac {d+\gamma}2}  	\langle \ln(\frac{t_*-s}{l_1^2(t) } ) \rangle
					&
					\mbox{ \ if \ } b=d
					\\
					(t_*-s)^{-\frac {d+\gamma}2}  	l_1^{d-b}(t)
					&
					\mbox{ \ if \ } b>d
				\end{cases}
				&
				\mbox{ \ if \ }  l_1^2(t) < t_*-s \le  l_2^2(t)
				\\
				\begin{cases}
					(t_*-s)^{-\frac {d+\gamma}2}  	l_2^{d-b}(t)
					&
					\mbox{ \ if \ } b<d
					\\
					(t_*-s)^{-\frac {d+\gamma}2}  	\langle \ln(\frac{ l_2(t) }{l_1(t) } ) \rangle
					&
					\mbox{ \ if \ } b=d
					\\
					(t_*-s)^{-\frac {d+\gamma}2}  	l_1^{d-b}(t)
					&
					\mbox{ \ if \ } b>d
				\end{cases}
				&
				\mbox{ \ if \ }  t_*-s > l_2^2(t)
			\end{cases}
			ds
			\\
			\lesssim \ & C(\alpha)
			(t_*-t)^{\alpha/2}  	\tilde{T}_{3}^{\gamma}.
%		\end{aligned}
%	\end{equation*}
	\end{align*}
\end{proof}

\subsection{Convolution involving $v(t) |x-q|^{-b} \1_{ \{ |x-q| \ge (T-t)^{1/2} \} }$}

\begin{prop}\label{qd24Apr11-03-prop}
	Let $d\ge 1$ be an integer, $b\ge 0$, $0\le t<T$, $q\in \mathbb{R}^d$. Given $\Gamma(x,t,y,s)$ in Proposition \ref{funda-prop},
	$v(s)\ge 0$ for $s\in [0,T]$, $
	|f(y,s)| \le
	v(s) |y-q|^{-b} \1_{ \{ |y-q| \ge (T-s)^{1/2} \} } $ for $(y,s)\in \mathbb{R}^d \times (0,T)$, denote $
	\mathcal{T}_{d}^{\rm{out} }[f] (x,t)  := \int_{0}^{t} \int_{\RR^d}
	\Gamma(x,t,y,s) f(y,s) dyds $.
	Then,
	\begin{equation}\label{qd2023Nov24-1}
		\left|\mathcal{T}_{d}^{\rm{out} }[f](x,t) \right|
		\lesssim
		\int_0^{t}  v(s) (T-s)^{-\frac b2}  ds ,
		\quad
		\left|\nabla \mathcal{T}_{d}^{\rm{out} }[f](x,t) \right|
		\lesssim
		\int_0^{t}  v(s) (t-s)^{-\frac{1}{2}} (T-s)^{-\frac b2}  ds,
	\end{equation}
	\begin{equation}\label{t-Holder-far}
		\begin{aligned}
			&
			\left|\mathcal{T}_{d}^{\rm{out} }[f](x,t) - \mathcal{T}_{d}^{\rm{out} }[f](x,T) \right|
			\\
			\lesssim \ &
			(T-t)
			\int_{0}^{[t-(T-t)]_+}  v(s) (T-s)^{-1-\frac b2}   ds
			+
			\sup\limits_{t_1\in [[t-(T-t)]_+,t]} v(t_1)  (T-t)^{1-\frac b2}
			+
			\int_{t}^{T} v(s) (T-s)^{-\frac{b}{2}} ds  ,
		\end{aligned}
	\end{equation}
	\begin{equation}\label{nabla-T-t-far}
		\left|\nabla \mathcal{T}_{d}^{\rm{out} }[f](x,t) - \nabla \mathcal{T}_{d}^{\rm{out} }[f] (x,T) \right|
		\lesssim   C(\alpha)	\left( T-t \right)^{\frac{\alpha}{2} }  \tilde{S}_{41}
		+
		\int_{t}^{T }  v(s)
		(T-s)^{-\frac {1+b}2}	ds ,
	\end{equation}	
	where
	\begin{equation*}
		\tilde{S}_{41}:=
		\int_0^{[t-(T-t)]_+}  v(s) (T-s)^{-\frac {1+b+\alpha}2}  ds   +
		\sup\limits_{t_1\in[[t-(T-t) ]_+ , t]} v(t_1) \left(T-t\right)^{\frac{1-b-\alpha}{2}}.
	\end{equation*}
	
	For $0<\alpha<1$, $0\le t < t_* \le T$,
	\begin{equation}\label{nabla-Holder-far}
		|\nabla \mathcal{T}_{d}^{\rm{out} }[f](x,t) - \nabla  \mathcal{T}_{d}^{\rm{out} }[f] (x_*,t_*) |
		\lesssim   C(\alpha)
	(| x - x_*| + \sqrt{| t- t_*|} )^{\alpha}
		\tilde{S}_{51}
		+
		\tilde{S}_{52},
	\end{equation}
	where
	\begin{equation*}
		\begin{aligned}
			&
			\tilde{S}_{51}:=
			\int_0^{[t-(T-t)]_+}  v(s) (T-s)^{-\frac {1+b+\alpha}2}  ds
			+
			\sup\limits_{t_1\in [[t-(T-t)]_+, t]} v(t_1)
			(T-t)^{ \frac{1-b-\alpha}{2} } ,
			\\
			&
			\tilde{S}_{52}:=
			\1_{\{ t_*\le \frac{T+t}{2} \}}
			\sup\limits_{t_1\in [t,t_*]}
			v(t_1)  (T-t)^{-\frac b2}
			(t_*-t)^{\frac 12}
			+
			\1_{\{ t_* > \frac{T+t}{2} \}}
			\int_{t}^{t_* }  v(s)
			(t_*-s)^{-\frac 12} (T-s)^{-\frac b2} ds.
		\end{aligned}
	\end{equation*}

For $0<\alpha<1$, $0\le t < t_* \le T$,
	\begin{equation}\label{qd24July17-2}
		| \mathcal{T}_{d}^{\rm{out} }[f](x,t) -   \mathcal{T}_{d}^{\rm{out} }[f](x_*,t_*)  |
		\lesssim
		C(\alpha) (| x - x_*| + \sqrt{| t- t_*|} )^{\alpha} \tilde{S}_{61}
		+
		\int_{t}^{t_*} v(s) (T-s)^{-\frac b2}	ds,
	\end{equation}
	where
	\begin{equation*}
		\tilde{S}_{61}:=
		\int_0^{[t-(T-t)]_+}  v(s) (T-s)^{-\frac {b+\alpha}2}  ds
		+
		\sup\limits_{t_1\in [[t-(T-t)]_+, t]} v(t_1)
		(T-t)^{1- \frac{b+\alpha}{2} }.
	\end{equation*}

\end{prop}

\begin{proof}[Proof of \eqref{qd2023Nov24-1}]
	\eqref{qd2023Nov24-1} is deduced by \eqref{Ga-est} and \eqref{basic-x-far} directly.
\end{proof}

\begin{proof}[Proof of \eqref{t-Holder-far}]
	\begin{equation*}
		\begin{aligned}
			&
			\mathcal{T}_{d}^{\rm{out} }[f](x,t) - \mathcal{T}_{d}^{\rm{out} }[f](x,T)
			= \int_{0}^{[t-(T-t)]_+} \int_{\RR^d}
			(\Gamma(x,t,y,s) - \Gamma(x,T,y,s) )
			f(y,s) dyds
			\\
			&
			+
			\int_{[t-(T-t)]_+}^{t} \int_{\RR^d}
			(\Gamma(x,t,y,s) - \Gamma(x,T,y,s) )
			f(y,s) dyds
			-
			\int_{t}^{T} \int_{\RR^d}
			\Gamma(x,T,y,s)
			f(y,s) dyds:= I_1 + I_2 + I_3.
		\end{aligned}
	\end{equation*}
	By \eqref{Ga-est} and \eqref{basic-x-far},
	\begin{align*}
		%	\begin{equation*}
			%		\begin{aligned}
				&	|I_1| \le
				(T-t)
				\int_{0}^{[t-(T-t)]_+} \int_{\RR^d} \int_0^1
				\left| (\pp_{t} \Gamma)(x,\theta t + (1-\theta)T,y,s) \right|
				|f(y,s)| d\theta dy ds
				\\
				\lesssim \ &
				(T-t)
				\int_{0}^{[t-(T-t)]_+ } v(s) (T-s)^{-1-\frac d2}  \int_{\RR^d}
				e^{ -c  (\frac{|x-y|}{\sqrt{T-s}} )^{2-\delta} }
				|y-q|^{-b} \1_{ \{ |y-q| \ge (T-s)^{\frac 12} \} } dyds
				\\
				\lesssim \ &
				(T-t)
				\int_{0}^{[t-(T-t)]_+}  v(s) (T-s)^{-1-\frac b2}   ds,
				%		\end{aligned}
			%	\end{equation*}
	\end{align*}
\begin{align*}
			&	|I_2| \lesssim
			\int_{[t-(T-t)]_+}^{t} v(s) (t-s)^{-\frac d2}  \int_{\RR^d}
			e^{ -c  (\frac{|x-y|}{\sqrt{t-s}} )^{2-\delta} }
			|y-q|^{-b} \1_{ \{ |y-q| \ge (T-s)^{\frac 12} \} } dy ds
			\\
			& +
			\int_{[t-(T-t)]_+}^{t} v(s) (T-s)^{-\frac d2}  \int_{\RR^d}
			e^{ -c  (\frac{|x-y|}{\sqrt{T-s}} )^{2-\delta} }
			|y-q|^{-b} \1_{ \{ |y-q| \ge (T-s)^{\frac 12} \} } dy ds
			\lesssim
			\sup\limits_{t_1\in [[t-(T-t)]_+,t]} v(t_1)  (T-t)^{1-\frac b2},
		\end{align*}
	\begin{equation*}
		|I_3| \lesssim
		\int_{t}^{T} \int_{\RR^d}
		(T-s)^{-\frac d2}
		e^{ -c  (\frac{|x-y|}{\sqrt{T-s}} )^{2-\delta} }
		v(s) |y-q|^{-b} \1_{ \{ |y-q| \ge (T-s)^{\frac 12} \} }
		dy ds
		\lesssim
		\int_{t}^{T} v(s) (T-s)^{-\frac{b}{2}} ds.
	\end{equation*}
\end{proof}	

\begin{proof}[Proof of \eqref{nabla-T-t-far}]
	\begin{equation*}
		\begin{aligned}
			&
			\pp_{x_i} \mathcal{T}_{d}^{\rm{out} }[f](x,t) - \pp_{x_i} \mathcal{T}_{d}^{\rm{out} }[f](x,T)
			\\
			= \ &
			\int_{0}^{t} \int_{\RR^d}
			\left( \pp_{x_i} \Gamma(x,t,y,s) - \pp_{x_i} \Gamma(x,T,y,s)  \right) f(y,s) dyds
			-
			\int_{t}^{T} \int_{\RR^d}
			\pp_{x_i} \Gamma(x,T,y,s) f(y,s) dyds
			:= I_1 + I_2.
		\end{aligned}
	\end{equation*}

	For $I_1$, by \eqref{nabla-Ga-Holder}, \eqref{basic-x-far},
	\begin{equation*}
			| I_1 |
			\lesssim C(\alpha)
			(T-t)^{\frac{\alpha}{2} }
			\int_{0}^{t} v(s)
			\left[
			\left(t-s\right)^{-\frac{1}{2}} \left(T-s\right)^{-\frac{b+\alpha}{2}}
			+
			\left(T-s\right)^{-\frac{1+b+\alpha}{2}}
			\right]
			ds
			\lesssim
			C(\alpha)	\left( T-t \right)^{\frac{\alpha}{2} }  \tilde{S}_{41}.
	\end{equation*}

	For $I_2$, by \eqref{Ga-est}, \eqref{basic-x-far}, we have $
	|I_2| \lesssim
	\int_{t}^{T }  v(s)
	(T-s)^{-\frac {1+b}2}	ds $.
\end{proof}

\begin{proof}[Proof of \eqref{nabla-Holder-far}]
	\begin{equation*}
		\begin{aligned}
			&
			\pp_{x_i} \mathcal{T}_{d}^{\rm{out} }[f](x,t) - \pp_{x_i} \mathcal{T}_{d}^{\rm{out} }[f](x_*,t_*)
			\\
			= \ &
			\int_{0}^{t} \int_{\RR^d}
			\left( \pp_{x_i} \Gamma(x,t,y,s) - \pp_{x_i} \Gamma(x_*,t_*,y,s)  \right) f(y,s) dyds
			-
			\int_{t}^{t_*} \int_{\RR^d}
			\pp_{x_i} \Gamma(x_*,t_*,y,s) f(y,s) dyds
			:= I_1 + I_2
			.
		\end{aligned}
	\end{equation*}
	
	For $I_1$, by \eqref{nabla-Ga-Holder}, \eqref{basic-x-far},
	\begin{equation*}
			| I_1 | \big(| x - x_*| + \sqrt{| t- t_*|} \big)^{-\alpha}
			\lesssim
			C(\alpha)
			\int_{0}^{t}
			v(s) \left(t_*-s\right)^{-\frac{\alpha}{2}} \left[
			\left(t-s\right)^{-\frac{1}{2}}
			+
			\left(t_*-s\right)^{-\frac{1}{2}}
			\right] \left(T-s\right)^{-\frac{b}{2}} ds
			\lesssim
			C(\alpha) \tilde{S}_{51},
	\end{equation*}	
	where in the last step, we split the integral into $\int_0^t= \int_{0}^{[t-(T-t)]_+}  +
	\int_{ [t-(T-t)]_+ }^{[t-(t_*-t)]_+} + \int_{[t-(t_*-t)]_+}^t$ to estimate.

	For $I_2$, by \eqref{Ga-est}, \eqref{basic-x-far}, then $ |I_2|
	\lesssim
	\int_{t}^{t_*} v(s)
	(t_*-s)^{-\frac 12} (T-s)^{-\frac b2}	ds
	\lesssim \tilde{S}_{52} $.
\end{proof}

\begin{proof}[Proof of \eqref{qd24July17-2}]
	\begin{equation*}
		\begin{aligned}
			&  \mathcal{T}_{d}^{\rm{out} }[f](x,t) -   \mathcal{T}_{d}^{\rm{out} }[f](x_*,t_*)
			\\
			= \ &
			\int_{0}^{t} \int_{\RR^d}
			\left(  \Gamma(x,t,y,s) -  \Gamma(x_*,t_*,y,s)  \right) f(y,s) dyds
			-
			\int_{t}^{t_*} \int_{\RR^d}
			\Gamma(x_*,t_*,y,s) f(y,s) dyds
			:= I_1 + I_2.
		\end{aligned}
	\end{equation*}
	
	For $I_1$, by \eqref{Ga-Holder}, \eqref{basic-x-far},
	\begin{equation*}
		| I_1 | \big(| x - x_*| + \sqrt{| t- t_*|} \big)^{-\alpha}
			\lesssim
			C(\alpha)
			\int_{0}^{t}
			v(s) \left(t_*-s\right)^{-\frac{\alpha}{2}}  \left(T-s\right)^{-\frac{b}{2}} ds
			\lesssim
			C(\alpha) \tilde{S}_{61},
	\end{equation*}	
	where for the last step, we split the integral into $\int_0^t= \int_{0}^{[t-(T-t)]_+}  +
	\int_{ [t-(T-t)]_+ }^{[t-(t_*-t)]_+} + \int_{[t-(t_*-t)]_+}^t$ to estimate and used \eqref{tsplit}.
	For $I_2$, by \eqref{Ga-est}, \eqref{basic-x-far}, then $ |I_2|
	\lesssim
	\int_{t}^{t_*} v(s) (T-s)^{-\frac b2}	ds$.
\end{proof}

\section{Derivation of the weighted topology for the outer problem}\label{out-top-Deri-sec}

\begin{prop}\label{convolu-prop}
	Given $|f(x,t)| \lesssim \sum_{j=1}^N \big(\varrho_1^{\J}+\varrho_2^{\J}\big)+\varrho_3 $ with $\varrho_1^{\J}, \varrho_2^{\J}, \varrho_3$ given in \eqref{rho-weights}, suppose
	\begin{align}
%	\begin{equation}
%		\begin{aligned}
			&
			0< \Theta<\beta<1/2, \ 0<\alpha<1, \ \Theta < \alpha/2,
			\nonumber
			\\
			&
			0< \sigma_0<\beta, \ \beta-\sigma_0 <\alpha/2, \
			1-\sigma_0-(1+\alpha)(1-\beta)<0,
			\
			\Theta +2\sigma_0-\beta <0,
			\nonumber
			\\
			&
			0<A_{\rm{o,h}} <\min\{ \Theta+ (1-\beta)(1-\alpha), 1-2\sigma_0-\alpha(1-\beta), 1-\sigma_0-\frac{\alpha}{2} \},
			\label{para-rho-con}
%		\end{aligned}
%	\end{equation}
	\end{align}
	then for
	$  \mathcal{T}_{2}^{\rm{out} }[f] (x,t) :=
	\int_{0}^{t} \int_{\RR^2}
	\Gamma(x,t,y,s) f(y,s) dy ds $ with $\Gamma(x,t,y,s)$ given in Proposition \ref{funda-prop} with dimension $d=2$, we have
	\begin{equation*}
		| \mathcal{T}_{2}^{\rm{out} }[f] |
		\lesssim |\ln T |
		\lambda_*^{\Theta+1}(0) R(0) ,
		\quad
		|\nabla \mathcal{T}_{2}^{\rm{out} }[f] |
		\lesssim  \lambda_*^{\Theta} (0) ,
	\end{equation*}
	\begin{equation*}
			|\mathcal{T}_{2}^{\rm{out} }[f](x,t) -  \mathcal{T}_{2}^{\rm{out} }[f] (x,T)|
			\lesssim
			|\ln(T-t)|	\lambda_*^{\Theta+1} R,
	\quad
		|\nabla \mathcal{T}_{2}^{\rm{out} }[f](x,t) - \nabla \mathcal{T}_{2}^{\rm{out} }[f](x,T) |
		\lesssim \lambda_*^{\Theta},
	\end{equation*}
	and for $0<t<t_* \le (T+t)/2$,
	\begin{equation*}
		|\nabla \mathcal{T}_{2}^{\rm{out} }[f](x,t) - \nabla \mathcal{T}_{2}^{\rm{out} }[f](x_*,t_*) |
		\lesssim
	(| x - x_*| + \sqrt{| t- t_*|})^{\alpha}
		\lambda_*^{\Theta}(t) ( \lambda_*  R)^{-\alpha }(t),
	\end{equation*}
	\begin{equation*}
		\left| \mathcal{T}_{2}^{\rm{out} }[f](x,t) - \mathcal{T}_{2}^{\rm{out} }[f](x,t_*) \right|
		\lesssim T^{A_{\rm{o,h}}} (t_* -t)^{\alpha/2}.
	\end{equation*}
	
\end{prop}

\begin{proof}
	
	\textbf{Convolution estimates about $\varrho_1^{\J}=  \lambda_*^{\Theta} (\lambda_*R)^{-1} \1_{\{ |x-q^{\J}|\le 3 \lambda_*  R \} } $}.
	For $ |f|\lesssim	\varrho_1^{\J}$ with $\lambda_*  R \le (T-t)^{1/2}$ provided $\beta<1/2$,
	by \eqref{bound-annular},
	\begin{equation}\label{c-1-1-upp}
		\begin{aligned}
			\mathcal{T}_{2}^{\rm{out} }[f]
			\lesssim \ &
			\int_0^{[t-(T-t)]_{+}}  \lambda_*^{\Theta}(s) (\lambda_*R)^{-1}(s) (T-s)^{ -1}
			(\lambda_*  R )^2(s)
			ds
			+
			\lambda_*^{\Theta} (\lambda_*R)^{-1}
			(\lambda_*  R)^2 |\ln (T-t) |
			\\
			= \ &
			\int_0^{[t-(T-t)]_{+}}  \lambda_*^{\Theta}(s) (\lambda_*R)(s) (T-s)^{ -1}
			ds
			+
			\lambda_*^{\Theta} \lambda_*R |\ln (T-t) |
			\lesssim
			\lambda_*^{\Theta}(0) (\lambda_*R)(0) |\ln T |
		\end{aligned}
	\end{equation}
	provided
	\begin{equation}\label{Con-1-1}
		\beta<1/2, \quad
		1+\Theta -\beta >0.
	\end{equation}
	
	By \eqref{nabla-annular},
	\begin{equation}\label{c-1-2-nab}
			|\nabla \mathcal{T}_{2}^{\rm{out} }[f] |
			\lesssim
			\int_0^{[t-(T-t)]_+} \lambda_*^{\Theta}(s) (\lambda_*R)(s) (T-s)^{ -\frac{3}{2}  }
			ds
			+
			\lambda_*^{\Theta}
			\lesssim
			\lambda_*^{\Theta} (0)
	\end{equation}
	provided
	\begin{equation}\label{Con-1-2}
		\beta<1/2,
		\quad  \beta-\Theta<1/2,
		\quad  \Theta\ge 0.
	\end{equation}
	
	By \eqref{T-t-annular},
	\begin{equation}\label{c-1-3-T-t}
		\begin{aligned}
			&
			|\mathcal{T}_{2}^{\rm{out} }[f](x,t) -  \mathcal{T}_{2}^{\rm{out} }[f] (x,T)|
			\lesssim
			(T-t)
			\int_{0}^{[t-(T-t)]_+ }  \lambda_*^{\Theta}(s) (\lambda_*R)^{-1}(s) (T-s)^{-2}
			(\lambda_*  R)^2(s)
			ds
			\\
			& ~+
			\lambda_*^{\Theta} (\lambda_*R)^{-1}
			\int_{[t-(T-t)]_+}^{t}
			\big[ \1_{\{ t-s \le  (\lambda_*  R)^2(t) \}}
			+
			(\lambda_*  R)^{2}(t) 	(t-s)^{-1}  \1_{\{ t-s > (\lambda_*  R)^2(t) \}}
		\big]
			ds
			+
			\lambda_*^{\Theta +1} R
			\\
			& ~
			+\int_{t}^{T}  (T-s)^{-1} \lambda_*^{\Theta}(s) (\lambda_*R)^{-1}(s)
			(\lambda_*  R)^2(s)
			ds
			\
			\lesssim
			(T-t)
			\int_{0}^{[t-(T-t)]_+ }  \lambda_*^{\Theta+1}(s) R(s) (T-s)^{-2}
			ds
			\\
			& ~
			+
			\lambda_*^{\Theta+1} R |\ln(T-t)|
			+\int_{t}^{T}  (T-s)^{-1} \lambda_*^{\Theta+1}(s) R(s)
			ds
			\lesssim
			\lambda_*^{\Theta+1} R |\ln(T-t)|
		\end{aligned}
	\end{equation}
	provided
	\begin{equation}\label{Con-1-3}
		0<\beta - \Theta <1.
	\end{equation}
	
	By \eqref{nablaT-t-annular},
	\begin{align}
%	\begin{equation}
%		\begin{aligned}
			&
			|\nabla \mathcal{T}_{2}^{\rm{out} }[f](x,t) - \nabla  \mathcal{T}_{2}^{\rm{out} }[f](x,T) |
			\lesssim
			(T-t)^{\frac{\alpha}{2}}
			\int_0^{[t-(T-t)]_+} \lambda_*^{\Theta}(s) (\lambda_*R)(s) (T-s)^{ - \frac {3+\alpha}2  }   ds
			+
			\lambda_*^{\Theta}
			\nonumber
			\\
			&
			\qquad +
			\int_{t}^{T} \lambda_*^{\Theta}(s)  (\lambda_*R)(s)  (T-s)^{-\frac {3}2}
			ds
			\lesssim
			\lambda_*^{\Theta}
			\label{c-1-4-nabT-t}
%		\end{aligned}
%	\end{equation}
	\end{align}
	provided
	\begin{equation}\label{Con-1-4}
		0<\alpha<1, \quad
		\beta<1/2, \quad
		\beta -\Theta <1/2, \quad
		\Theta < \alpha/2.
	\end{equation}
	
	By \eqref{nablat*-t-annular}, for $0 <\alpha <1$ and $0<t<t_* \le (T+t)/2$,
	\begin{equation}\label{c-1-5-nabt*-t}
		\begin{aligned}
			&
			\left| \nabla \mathcal{T}_{2}^{\rm{out} }[f](x,t) - \nabla \mathcal{T}_{2}^{\rm{out} }[f](x_*,t_*) \right| \big(| x - x_*| + \sqrt{| t- t_*|} \big)^{-\alpha}
			\\
			\lesssim \ &
			\int_0^{[t-(T-t)]_+}  \lambda_*^{\Theta}(s) (\lambda_*R)(s) (T-s)^{ - \frac {3+\alpha}2  }
			ds
			+
			\lambda_*^{\Theta}(t) ( \lambda_*  R)^{-\alpha }(t)
			\lesssim
			\lambda_*^{\Theta}(t) ( \lambda_*  R)^{-\alpha }(t)
		\end{aligned}
	\end{equation}
	provided
	\begin{equation}\label{Con-1-5}
		\Theta -\alpha (1-\beta) < 0,
		\ \beta< 1/2.
	\end{equation}

	By \eqref{qd24July17-1}, for $0 <\alpha <1$, $\beta<1/2$, and $0<t<t_* \le (T+t)/2$,  and
	\begin{equation}\label{qd24July18-1}
		\begin{aligned}
			&
			\left| \mathcal{T}_{2}^{\rm{out} }[f](x,t) - \mathcal{T}_{2}^{\rm{out} }[f](x,t_*) \right| (t_* -t)^{-\alpha/2}
			\\
			\lesssim \ &
			\int_0^{[t-(T-t)]_+}
			\lambda_*^{\Theta+1}(s) R(s) (T-s)^{-1-\frac{\alpha}{2}} ds
			+
			\lambda_*^{\Theta} ( \lambda_* R )^{1-\alpha}
			\lesssim
			\big( \lambda_*^{\Theta} ( \lambda_* R )^{1-\alpha} \big)(0).
		\end{aligned}
	\end{equation}

	\textbf{Convolution estimates about $\varrho_2^{\J}=  T^{-\sigma_0}
		\lambda_*^{1-\sigma_0} |x-q^{\J}|^{-2}
		\1_{ \{  \lambda_* R/2 \le |x-q^{\J}|\le d_q \} } $}.
	Consider
	\begin{equation*}
		|f| \le  \lambda_*^{1-\sigma_0} |x-q^{\J}|^{-2}
		\big(
		\1_{ \{  \lambda_* R/2 \le |x-q^{\J}|\le (T-t)^{\frac{1}{2}} \} }
		+
		\1_{ \{  (T-t)^{\frac{1}{2}}<  |x-q^{\J}|\le d_q \} }  \big).
	\end{equation*}
	We will use Propositions \ref{qd24Apr11-02-prop} and \ref{qd24Apr11-03-prop} repetitively hereafter. Provided
	$\sigma_0<1$,
	\begin{equation}\label{c-2-1-upp}
		| \mathcal{T}_{2}^{\rm{out} }[f] |
			\lesssim
			\int_0^{[t-(T-t)]_+}
			\frac{\lambda_*^{1-\sigma_0}(s)
				|\ln(T-s)|}{T-s} ds
			+
			\lambda_*^{1-\sigma_0}
			|\ln (T-t)|^2
			+
			\int_0^{t}
			\frac{\lambda_*^{1-\sigma_0}(s)}{T-s}    ds
			\lesssim
			\lambda_*^{1-\sigma_0}(0)
			(\ln T)^2.
	\end{equation}
	\begin{equation}\label{c-2-2-nab}
		\begin{aligned}
			| \nabla \mathcal{T}_{2}^{\rm{out} }[f] |
			\lesssim \ &
			\int_0^{[t-(T-t)]_+} \lambda_*^{1-\sigma_0}(s) (T-s)^{ -\frac{3}{2}  }
			|\ln(T-s)|
			ds
			+
			\lambda_*^{1-\sigma_0} (\lambda_* R)^{-1}
			\\
			&
			+
			\int_0^{t}  \lambda_*^{1-\sigma_0}(s)
			(t-s)^{-\frac{1}{2}} (T-s)^{-1}  ds
			\lesssim
			\lambda_*^{1-\sigma_0}(0) (\lambda_* R)^{-1}(0),
		\end{aligned}
	\end{equation}
	where for the second ``$\lesssim $'', we used integration by part for $\int_0^{t}  \lambda_*^{1-\sigma_0}(s)
	(t-s)^{-\frac{1}{2}} (T-s)^{-1}  ds $, and
	\begin{equation}\label{Con-2-1}
		\sigma_0<\beta<1/2.
	\end{equation}
	\begin{align}\notag
			& |\mathcal{T}_{2}^{\rm{out} }[f] (x,t) - \mathcal{T}_{2}^{\rm{out} }[f] (x,T)|
			\lesssim
			(T-t)
			\int_{0}^{[t-(T-t)]_+ } \lambda_*^{1-\sigma_0}(s) (T-s)^{-2}
			|\ln (T-s)|
			ds
			\\ \notag
			& ~+
			\lambda_*^{1-\sigma_0}
			\int_{[t-(T-t)]_+}^{t}
			\begin{cases}
				(\lambda_* R)^{-2}(t)
				&
				\mbox{ \ if \ } t-s \le (\lambda_* R)^2(t)
				\\
				(t-s)^{-1}  	\langle \ln(\frac{t-s}{(\lambda_* R )^2(t) } ) \rangle
				&
				\mbox{ \ if \ }  (\lambda_* R )^2(t) < t-s \le  T-t
			\end{cases}
			ds
			\\ \notag
			& ~+
			\lambda_*^{1-\sigma_0}
			|\ln (T-t)|
			+\int_{t}^{T}  (T-s)^{-1} \lambda_*^{1-\sigma_0} (s)
			|\ln(T-s)|
			ds
			+
			(T-t)
			\int_{0}^{[t-(T-t)]_+}  \lambda_*^{1-\sigma_0}(s) (T-s)^{-2}   ds
			\\
			& ~
			+
			\lambda_*^{1-\sigma_0}
			+
			\int_{t}^{T} \lambda_*^{1-\sigma_0}(s) (T-s)^{-1} ds
			\lesssim
			\lambda_*^{1-\sigma_0}  \ln^2 (T-t), \label{c-2-3-T-t}
		\end{align}
	where we used $\int_{[t-(T-t)]_+}^{t- (\lambda_* R)^2(t) }
	(t-s)^{-1}  	\langle \ln(\frac{t-s}{(\lambda_* R )^2(t) } ) \rangle ds
	\le
	\int_{1}^{(T-t) / (\lambda_*R)^2(t)} z^{-1} \langle \ln z\rangle dz \lesssim \ln^2 (T-t)$, and
	\begin{equation}\label{Con-2-2}
		0<\sigma_0<1, \quad \beta<1/2.
	\end{equation}
	
	For $0<\alpha<1$,
	\begin{align}
%	\begin{equation}
%		\begin{aligned}
			&
			|\nabla \mathcal{T}_{2}^{\rm{out} }[f](x,t) - \nabla  \mathcal{T}_{2}^{\rm{out} }[f](x,T) |
			\lesssim
			(T-t)^{\frac{\alpha}{2}}
			\int_0^{[t-(T-t)]_+} \lambda_*^{1-\sigma_0}(s) (T-s)^{ - \frac {3+\alpha}2  }
			|\ln(T-s)|
			ds
			\nonumber
			\\
			&
			+
			\lambda_*^{1-\sigma_0}  (\lambda_* R)^{-1}
			 +
			\int_{t}^{T} \lambda_*^{1-\sigma_0}(s)  (T-s)^{-\frac {3}2}
			|\ln(T-s)|
			ds
			\nonumber
			\\
			& +
			\left( T-t \right)^{\frac{\alpha}{2} }
			\Big[ \int_0^{[t-(T-t)]_+}  \lambda_*^{1-\sigma_0}(s) (T-s)^{-\frac {3+\alpha}2}  ds + 	\lambda_*^{1-\sigma_0}  (T-t)^{\frac{-1-\alpha}{2} } \Big]
			+
			\int_{t}^{T }  \lambda_*^{1-\sigma_0}(s)
			(T-s)^{-\frac {3}2}	ds
			\nonumber
			\\
			\lesssim \ &
			(T-t)^{\frac{\alpha}{2}}
			\int_0^{[t-(T-t)]_+} \lambda_*^{1-\sigma_0}(s) (T-s)^{ - \frac {3+\alpha}2  }
			|\ln(T-s)|
			ds
			+
			\lambda_*^{1-\sigma_0}  (\lambda_* R)^{-1}
			\lesssim \lambda_*^{1-\sigma_0}  (\lambda_* R)^{-1},
			\label{c-2-4-nabT-t}
%		\end{aligned}
%	\end{equation}
	\end{align}
	where for the last ``$\lesssim$'', we discussed three cases $1-\sigma_0 - \frac {1+\alpha}2 < 0, =0, >0 $, and used
	\begin{equation}\label{Con-2-3}
		\beta<1/2, \quad \sigma_0<1/2, \quad \beta-\sigma_0<\alpha/2.
	\end{equation}
	
	For $0<\alpha<1$ and $0< t < t_* \le (T+t)/2$,
	\begin{equation}\label{c-2-5-nabt*-t}
		\begin{aligned}
			&
			|\nabla \mathcal{T}_{2}^{\rm{out} }[f](x,t) - \nabla \mathcal{T}_{2}^{\rm{out} }[f](x_*,t_*) | (| x - x_*| + \sqrt{| t- t_*|} )^{-\alpha}
			\\
			\lesssim \ &
			\int_0^{[t-(T-t)]_+} \lambda_*^{1-\sigma_0}(s) (T-s)^{ - \frac {3+\alpha}2  }
			|\ln(T-s)|
			ds
			+
			\lambda_*^{1-\sigma_0}(t)	(\lambda_* R)^{ -1-\alpha }(t)
			\\
			& +
			\int_0^{[t-(T-t)]_+}  \lambda_*^{1-\sigma_0}(s) (T-s)^{-\frac {3+\alpha}2}  ds
			+
			\lambda_*^{1-\sigma_0}(t) (T-t)^{ \frac{-1-\alpha}{2}}
			\lesssim
			\lambda_*^{1-\sigma_0}(t)	(\lambda_* R)^{ -1-\alpha }(t),
		\end{aligned}
	\end{equation}
	where for the last ``$\lesssim$'', we discussed three cases $1-\sigma_0 -\frac {1+\alpha}2 < 0, = 0, >0$ and required
	\begin{equation}\label{Con-2-4}
		\beta<1/2, \quad 1-\sigma_0 - (1+\alpha)(1-\beta)<0.
	\end{equation}

	For $0 <\alpha <1$ and $0<t<t_* \le (T+t)/2$,
	by \eqref{qd24July17-1},  \eqref{qd24July17-2}, and $\beta<1/2$,
	\begin{equation}\label{qd24July18-2}
		\begin{aligned}
			&
			| \mathcal{T}_{2}^{\rm{out} }[f](x,t) -  \mathcal{T}_{2}^{\rm{out} }[f](x,t_*) | (t_* -t)^{-\alpha/2}
			\\
			\lesssim \ &
			\int_0^{[t-(T-t)]_+} \lambda_*^{1-\sigma_0}(s) (T-s)^{-1-\frac{\alpha}{2}} |\ln(T-s)| ds
			+\lambda_*^{1-\sigma_0} (\lambda_* R)^{-\alpha}
			\lesssim
			\lambda_*^{1-\sigma_0}(0) (\lambda_* R)^{-\alpha}(0),
		\end{aligned}
	\end{equation}
	where the last step is guaranteed by the restriction \eqref{Con-2-1}.

	\textbf{Convolution estimates about $\varrho_3=  T^{-\sigma_0}$}. Consider $
	|f|\le \1_{ \{ |x|\le \sqrt{T-t} \} } + \1_{ \{ |x| > \sqrt{T-t} \} } $. Then
	\begin{equation}\label{c-3-1-upp}
		| \mathcal{T}_{2}^{\rm{out} }[f] |
		\lesssim T,
		\quad
		| \nabla \mathcal{T}_{2}^{\rm{out} }[f] |
		\lesssim  T^{1/2},
		\quad
		| \mathcal{T}_{2}^{\rm{out} }[f] (x,t) -  \mathcal{T}_{2}^{\rm{out} }[f] (x,T)|
		\lesssim
		(T-t) |\ln(T-t)|.
	\end{equation}
	
	For $0<\alpha<1$,
	\begin{equation}\label{c-3-4-nabT-t}
		|\nabla \mathcal{T}_{2}^{\rm{out} }[f] (x,t) - \nabla  \mathcal{T}_{2}^{\rm{out} }[f] (x,T) |
		\lesssim
		(T-t)^{\frac{\alpha}{2}}
		\int_0^{[t-(T-t)]_+}  (T-s)^{ - \frac {1+\alpha}2  }
		ds + (T-t)^{\frac{1}{2}}
		\lesssim T^{\frac{1-\alpha}{2}}
		(T-t)^{\frac{\alpha}{2}}.
	\end{equation}
	
	For $0<\alpha<1$ and $0< t < t_* \le (T+t)/2$,
	\begin{equation}\label{c-3-5-nabt*-t}
		\begin{aligned}
			&
			|\nabla \mathcal{T}_{2}^{\rm{out} }[f](x,t) - \nabla \mathcal{T}_{2}^{\rm{out} }[f](x_*,t_*) |  (| x - x_*| + \sqrt{| t- t_*|} )^{-\alpha}
			\\
			\lesssim \ &
			\int_0^{[t-(T-t)]_+} (T-s)^{ - \frac {1+\alpha}2  }
			ds + (T-t)^{\frac{1-\alpha}{2}}
			\lesssim
			T^{\frac{1-\alpha}{2}}.
		\end{aligned}
	\end{equation}

	For $0 <\alpha <1$ and $0<t<t_* \le (T+t)/2$,
	by \eqref{qd24July17-1},  \eqref{qd24July17-2},
	\begin{equation}\label{qd24July18-3}
		|\mathcal{T}_{2}^{\rm{out} }[f](x,t) - \mathcal{T}_{2}^{\rm{out} }[f](x,t_*) |  (t_*-t)^{-\alpha/2}
		\lesssim
		T^{1- \frac{\alpha}{2}}.
	\end{equation}

	In sum,	for $
	|f| \lesssim \sum_{j=1}^N \big(\varrho_1^{\J}+\varrho_2^{\J}\big)+\varrho_3 $,
	combining \eqref{c-1-1-upp}, \eqref{c-2-1-upp} and \eqref{c-3-1-upp}, we have
	\begin{equation*}
		| \mathcal{T}_{2}^{\rm{out} }[f] |
		\lesssim
		\lambda_*^{\Theta}(0) (\lambda_*R)(0) |\ln T |
		+
		T^{-\sigma_0} \lambda_*^{1-\sigma_0}(0)
		(\ln T)^2 + T^{1-\sigma_0}
		\lesssim
		\lambda_*^{\Theta}(0) (\lambda_*R)(0) |\ln T |,
	\end{equation*}
	where for the last ``$\lesssim$'', we require
	\begin{equation}\label{Con-4-1}
		\Theta +2\sigma_0 - \beta < 0, \ \sigma_0>0.
	\end{equation}
	
	Combining \eqref{c-1-2-nab}, \eqref{c-2-2-nab} and \eqref{c-3-1-upp}, we have
	\begin{equation*}
		|\nabla \mathcal{T}_{2}^{\rm{out} }[f] |
		\lesssim  \lambda_*^{\Theta} (0)
		+
		T^{-\sigma_0} \lambda_*^{1-\sigma_0}(0) (\lambda_* R)^{-1}(0)
		+
		T^{\frac{1}{2}-\sigma_0}
		\lesssim  \lambda_*^{\Theta} (0),
	\end{equation*}
	where for the last ``$\lesssim$'', we used
	\begin{equation}\label{Con-4-2}
		\Theta +2\sigma_0 -\beta<0, \ \sigma_0>0, \ \beta<1/2.
	\end{equation}

	Combining \eqref{c-1-3-T-t}, \eqref{c-2-3-T-t} and \eqref{c-3-1-upp}, then
	\begin{equation*}
		\begin{aligned}
			|\mathcal{T}_{2}^{\rm{out} }[f](x,t) -  \mathcal{T}_{2}^{\rm{out} }[f] (x,T)|
			\lesssim \ &
			\lambda_*^{\Theta} (\lambda_*R) |\ln(T-t)|
			+
			T^{-\sigma_0}
			\lambda_*^{1-\sigma_0}  \ln^2 (T-t)
			+
			T^{-\sigma_0} (T-t) |\ln(T-t)|
			\\
			\lesssim \ &
			\lambda_*^{\Theta} (\lambda_*R) |\ln(T-t)|,
		\end{aligned}
	\end{equation*}
	where for the last ``$\lesssim$'', we used
	\begin{equation}\label{Con-4-3}
		\Theta +2\sigma_0 -\beta<0, \ \sigma_0>0.
	\end{equation}
	
	Combining \eqref{c-1-4-nabT-t}, \eqref{c-2-4-nabT-t} and \eqref{c-3-4-nabT-t}, then
	\begin{equation*}
		|\nabla \mathcal{T}_{2}^{\rm{out} }[f](x,t) - \nabla  \mathcal{T}_{2}^{\rm{out} }[f](x,T) |
		\lesssim
		\lambda_*^{\Theta}
		+
		T^{-\sigma_0} 	\lambda_*^{1-\sigma_0}  (\lambda_* R)^{-1}
		+
		T^{-\sigma_0}
		T^{\frac{1-\alpha}{2}}
		(T-t)^{\frac{\alpha}{2}}
		\lesssim \lambda_*^{\Theta},
	\end{equation*}
	where for the last ``$\lesssim$'', we used
	\begin{equation}\label{Con-4-4}
		\Theta +2\sigma_0-\beta <0,
		\
		\sigma_0>0,
		\
		\Theta <\alpha/2, \
		\Theta + \sigma_0 <1/2.
	\end{equation}
	
	Combining \eqref{c-1-5-nabt*-t}, \eqref{c-2-5-nabt*-t}, \eqref{c-3-5-nabt*-t}, for $0 <\alpha <1$ and $0< t < t_* \le (T+t)/2$, we have
	\begin{equation*}
		\begin{aligned}
			&
			|\nabla \mathcal{T}_{2}^{\rm{out} }[f](x,t) - \nabla \mathcal{T}_{2}^{\rm{out} }[f](x_*,t_*) |
			\big(| x - x_*| + \sqrt{| t- t_*|} \big)^{-\alpha}
			\\
			\lesssim \ &
			\lambda_*^{\Theta}(t) ( \lambda_*  R)^{-\alpha }(t) +
			T^{-\sigma_0}
			\lambda_*^{1-\sigma_0}(t)	(\lambda_* R)^{ -1-\alpha }(t)
			+
			T^{-\sigma_0} T^{\frac{1-\alpha}{2}}
			\lesssim
			\lambda_*^{\Theta}(t) ( \lambda_*  R)^{-\alpha }(t),
		\end{aligned}
	\end{equation*}
	where for the last ``$\lesssim $'', we used
	\begin{equation}\label{Con-4-5}
		\Theta +2\sigma_0-\beta <0, \
		\sigma_0>0,
		\
		\Theta -\alpha(1-\beta) < 0, \
		\Theta +\sigma_0 - \frac{1}{2} -\alpha(\frac{1}{2} -\beta) <0.
	\end{equation}

	Combining \eqref{qd24July18-1}, \eqref{qd24July18-2}, \eqref{qd24July18-3},
	for $0 <\alpha <1$ and $0<t<t_* \le (T+t)/2$, we have
	\begin{equation*}
		\begin{aligned}
			&
			\left| \mathcal{T}_{2}^{\rm{out} }[f](x,t) - \mathcal{T}_{2}^{\rm{out} }[f](x,t_*) \right| (t_* -t)^{-\alpha/2}
			\\
			\lesssim \ &
			\big( \lambda_*^{\Theta} ( \lambda_* R )^{1-\alpha} \big)(0) +
			T^{-\sigma_0} \lambda_*^{1-\sigma_0}(0) (\lambda_* R)^{-\alpha}(0) +
			T^{-\sigma_0} T^{1- \frac{\alpha}{2}}
			\lesssim T^{A_{\rm{o,h}}},
		\end{aligned}
	\end{equation*}
	provided
	\begin{equation}\label{qd24July18-4}
		0<A_{\rm{o,h}} <\min\{ \Theta+ (1-\beta)(1-\alpha), 1-2\sigma_0-\alpha(1-\beta), 1-\sigma_0-\frac{\alpha}{2} \}.
	\end{equation}

	Collecting \eqref{Con-1-1}, \eqref{Con-1-2}, \eqref{Con-1-3}, \eqref{Con-1-4}, \eqref{Con-1-5}, \eqref{Con-2-1}, \eqref{Con-2-2}, \eqref{Con-2-3}, \eqref{Con-2-4}, \eqref{Con-4-1}, \eqref{Con-4-2}, \eqref{Con-4-3},
	\eqref{Con-4-4}, \eqref{Con-4-5},  and  \eqref{qd24July18-4}, we conclude  the restrictions  \eqref{para-rho-con} on the parameters.
\end{proof}

\section{Estimates of $\mathcal{G}$ in $(\ref{G-def})$}\label{G-est-sec}

\subsection{Estimates for terms involving $U_*$, $\Phi_{\rm {out}}$, $\Phi_{\rm {in}}^{\J}$, $\Phi^{*{\J}}_0$}

First, we prepare some useful formulas.  By \eqref{nablaW},
\begin{equation}\label{nabU*-est}
	\left| \nabla_x U^{\K} \right| \lesssim \lambda_k^{-1} \langle \rho_k \rangle^{-2},
	\quad 	\left|\nabla_x U_* \right|
	\lesssim
	\sum\limits_{j=1}^N \1_{\{ |x-q^{\J}| < 3d_q\}}  \lambda_*^{-1}\langle \rho_j \rangle^{-2}
	+   \1_{\{ \cap_{j=1}^N \{|x-q^{\J}| \ge 3 d_q  \} \}}  \lambda_*.
\end{equation}
\begin{equation}\label{Del-U*-est}
	\left| \Delta_x U_*  \right|
	\lesssim \sum\limits_{j=1}^N \lambda_*^{-2} \langle \rho_j \rangle^{-4}
	\lesssim
	\sum\limits_{j=1}^N \1_{\{ |x-q^{\J}| < 3d_q\}}  \lambda_*^{-2}\langle \rho_j \rangle^{-4}
	+   \1_{\{ \cap_{j=1}^N \{|x-q^{\J}| \ge 3 d_q  \} \}}  \lambda_*^2.
\end{equation}
\begin{align}
%\begin{equation}
%	\begin{aligned}
		| U_* \cdot \nabla_x U_* |
		= \ &
		\Big|  \sum\limits_{m=1}^N \sum\limits_{k\ne m} \left(U^{\K} -U_{\infty} \right) \cdot  \nabla_x U^{\M}\Big|
		\lesssim
		\sum\limits_{m=1}^N \sum\limits_{k\ne m}
		\langle \rho_k\rangle^{-1} \lambda_m^{-1} \langle \rho_m \rangle^{-2}
		\nonumber
		\\
		\lesssim \ &
		\sum\limits_{j=1}^N
		\1_{\{|x-q^{\J}| < 3d_q \}}
		\langle \rho_j\rangle^{-2}
		+
		\1_{\{\cap_{j=1}^N |x-q^{\J}| \ge 3d_q \}} \lambda_*^2,
		\label{U*cdot-nabU*-split}
%	\end{aligned}
%\end{equation}
\end{align}
where for the last step, we used that for any fixed $j=1,2,\dots, N$,
\begin{align*}
%\begin{equation*}
%	\begin{aligned}
		&
		\1_{\{|x-q^{\J}|< 3d_q\}} \Big[
		\lambda_j^{-1} \langle \rho_j\rangle^{-2}
		\sum\limits_{k\ne j}
		\langle \rho_k\rangle^{-1}
		+
		\sum\limits_{m\ne j} \sum\limits_{k\ne m}
		\langle \rho_k\rangle^{-1} \lambda_m^{-1} \langle \rho_m \rangle^{-2}
		\Big]
		\\
		\lesssim \ &
		\1_{\{|x-q^{\J}| < 3d_q\}}
		\Big(
		\langle \rho_j\rangle^{-2}
		+ \lambda_*
		\sum\limits_{m\ne j} \sum\limits_{k\ne m}
		\langle \rho_k\rangle^{-1}
		\Big)
		\sim
		\1_{\{|x-q^{\J}| < 3d_q \}}
		\langle \rho_j\rangle^{-2}
		;
		\\
		&
		\1_{ \{  \cap_{j=1}^N \{ |x-q^{\J}| \ge 3d_q \} \} } \sum\limits_{m=1}^N \sum\limits_{k\ne m}
		\langle \rho_k\rangle^{-1} \lambda_m^{-1} \langle \rho_m \rangle^{-2}
		\lesssim  \lambda_*^2 .
%	\end{aligned}
%\end{equation*}
\end{align*}

By the same argument for \eqref{U*cdot-nabU*-split}, then
\begin{align}\label{DeltaU*-type}
\bigg|  \sum\limits_{j=1}^N |\nabla_x U^{\J}|^2
		\left(U^{\J} -U_* \right)  \bigg|
		\lesssim
		\sum\limits_{j=1}^{N}
		\1_{\{ |x-q^{\J}| <3d_q \}}
		\left(
		\lambda_*^{-1} \langle \rho_j\rangle^{-4}
		+
		\lambda_*^{2}  \langle \rho_j \rangle^{-1}
		\right)
		+
		\1_{\{ \cap_{j=1}^N \{|x-q^{\J}| \ge 3d_q \} \}}
		\lambda_*^3.
\end{align}

Notice
\begin{equation*}
	\begin{aligned}
		& |\nabla_x U_*|^2 + U_* \cdot \Delta_x U_*
		=
		\bigg| \sum\limits_{j=1}^N  \nabla_x U^{\J} \bigg|^2 -
		\sum\limits_{j=1}^N |\nabla_x U^{\J}|^2
		-
		\sum\limits_{j=1}^N |\nabla_x U^{\J}|^2
		\left( U_* - U^{\J} \right) \cdot  U^{\J}
		\\
		= \ &
		\sum\limits_{j=1}^N  \sum\limits_{k\ne j}  \nabla_x U^{\J} \cdot \nabla_x U^{\K}  -
		\sum\limits_{j=1}^N |\nabla_x U^{\J}|^2
		\left( U_* - U^{\J} \right) \cdot  U^{\J}
		.
	\end{aligned}
\end{equation*}
Then
\begin{equation}\label{Delta|U*|^2}
	\begin{aligned}
		\left| |\nabla_x U_*|^2 + U_* \cdot \Delta_x U_* \right|
		\lesssim \ &
		\sum\limits_{m=1}^N  \sum\limits_{k\ne m}  \lambda_m^{-1}\lambda_k^{-1} \langle \rho_m\rangle^{-2}   \langle \rho_k\rangle^{-2} +
		\sum\limits_{m=1}^N \sum\limits_{k\ne m}  \lambda_m^{-2} \langle \rho_m\rangle^{-4}
		\langle \rho_k\rangle^{-1}
		\\
		\lesssim \ &
		\sum\limits_{j=1}^N
		\1_{\{ |x-q^{\J}| <3d_q \}}
		\left(  \langle \rho_j\rangle^{-2}
		+   \lambda_*^{-1} \langle \rho_j\rangle^{-4}
		\right)
		+
		\1_{\{\cap_{j=1}^N \{ |x-q^{\J}| \ge 3d_q \} \}}
		\lambda_*^2,
	\end{aligned}
\end{equation}
where for the last step, we used that for any fixed $j=1,2,\dots, N$,
\begin{align*}
		&
		\1_{\{ |x-q^{\J}| <3d_q \}}
		\Big[ \lambda_j^{-1} \langle \rho_j\rangle^{-2}  \sum\limits_{k\ne j}  \lambda_k^{-1}   \langle \rho_k\rangle^{-2}
		+  \lambda_j^{-2} \langle \rho_j\rangle^{-4} \sum\limits_{k\ne j}
		\langle \rho_k\rangle^{-1}
		\\
		&
		\quad
		+
		\sum\limits_{m\ne j}  \sum\limits_{k\ne m}
		\left( \lambda_m^{-1}\lambda_k^{-1} \langle \rho_m\rangle^{-2}   \langle \rho_k\rangle^{-2}
		+
		\lambda_m^{-2} \langle \rho_m \rangle^{-4}
		\langle \rho_k\rangle^{-1}
		\right)   \Big]
		\\
		\lesssim \ &
		\1_{\{ |x-q^{\J}| <3d_q \}}
		\Big[  \langle \rho_j\rangle^{-2}
		+   \lambda_*^{-1} \langle \rho_j\rangle^{-4}
		+
		\sum\limits_{m\ne j}  \sum\limits_{k\ne m}
		\left( \langle \rho_k\rangle^{-2}
		+
		\lambda_*^{2}
		\langle \rho_k\rangle^{-1}
		\right)  \Big]
		\sim
		\1_{\{ |x-q^{\J}| <3d_q \}}
		\left(  \langle \rho_j\rangle^{-2}
		+   \lambda_*^{-1} \langle \rho_j\rangle^{-4}
		\right) ;
	\end{align*}
\begin{equation*}
	\1_{\{\cap_{j=1}^N \{ |x-q^{\J}| \ge 3d_q \} \}} \Big( \sum\limits_{m=1}^N  \sum\limits_{k\ne m}  \lambda_m^{-1}\lambda_k^{-1} \langle \rho_m\rangle^{-2}   \langle \rho_k\rangle^{-2} +
	\sum\limits_{m=1}^N \sum\limits_{k\ne m}  \lambda_m^{-2} \langle \rho_m\rangle^{-4}
	\langle \rho_k\rangle^{-1}  \Big)
	\lesssim
	\1_{\{\cap_{j=1}^N \{ |x-q^{\J}| \ge 3d_q \} \}}
	\lambda_*^2.
\end{equation*}
\begin{align}
%\begin{equation}
%	\begin{aligned}
		&
		\left|U_*\wedge \Delta_x U_*  \right|
		=
		\bigg|U_*\wedge \sum\limits_{j=1}^N \left| \nabla_x U^{\J}\right|^2 U^{\J} \bigg|
		\lesssim
		\sum\limits_{j=1}^N \lambda_j^{-2} \langle \rho_j \rangle^{-4}  \sum\limits_{k\ne j} \langle \rho_j \rangle^{-1}
		\nonumber
		\\
		\lesssim \ &
		\sum\limits_{j=1}^{N}
		\1_{\{ |x-q^{\J}| <3d_q \}}
		\left(
		\lambda_*^{-1} \langle \rho_j\rangle^{-4}
		+
		\lambda_*^{2}  \langle \rho_j \rangle^{-1}
		\right)
		+
		\1_{\{ \cap_{j=1}^N \{|x-q^{\J}| \ge 3d_q \} \}}
		\lambda_*^3,
		\label{U*wed-DeltaU*}
%	\end{aligned}
%\end{equation}
\end{align}
where we used \eqref{DeltaU*-type} for the last step.

\medskip

Next, we derive some estimates about $\Phi_{\rm {out}}$, $\Phi_{\rm {in}}$ that will be used frequently in the estimate of $\mathcal{G}$.

For $\Phi_{\rm {out} } \in B_{\rm{out}}$ defined in \eqref{out-space}, since $\Phi_{\rm {out}}(q^{\J},T)=0$ for all $j=1,2,\dots,N$, then
\begin{equation*}
	\begin{aligned}
		&
		\left| \Phi_{\rm {out} }(x,t) \right|
		=
		\left| \Phi_{\rm {out}}(x,t) - \Phi_{\rm {out}}(x,T)  + \Phi_{\rm {out}}(x,T)
		- \Phi_{\rm {out}}(q^{\J},T) \right|
		\\
		\lesssim \ &
		\| \Phi_{\rm {out} }\|_{\sharp, \Theta,\alpha}
		\left[
		|\ln(T-t)|
		\lambda^{\Theta+1}_*  R
		+
		(T-t) \| Z_* \|_{C^3(\R^2)}
		+
		|x-q^{\J}| \left(\lambda^{\Theta}_*(0)
		+
		\| Z_* \|_{C^3(\R^2)}  \right)
		\right].
	\end{aligned}
\end{equation*}
Thus
\begin{equation}\label{out-upp}
	\begin{aligned}
		&
		\left| \Phi_{\rm {out} } \right|
		\lesssim
		\| \Phi_{\rm {out} }\|_{\sharp, \Theta,\alpha}
		\min\Big\{  |\ln T| \lambda^{\Theta+1}_*(0) R(0)
		+
		\| Z_* \|_{C^3(\R^2)}
		,
		\\
		&  |\ln(T-t)|
		\lambda^{\Theta+1}_*  R
		+
		(T-t) \| Z_* \|_{C^3(\R^2)}
		+
		\inf_{j=1,\dots,N}
		|x-q^{\J}| \left(\lambda^{\Theta}_*(0)
		+
		\| Z_* \|_{C^3(\R^2)}  \right)
		\Big\}.
	\end{aligned}
\end{equation}
Combining \eqref{lam-ansatz} and the parameter restriction $\Theta<\beta$, we have
\begin{equation}\label{out-upp-split}
	\left| \Phi_{\rm {out} } \right|
	\lesssim
	\sum\limits_{j=1}^N \1_{\{ |x-q^{\J}| < 3d_q \}}  \| \Phi_{\rm {out} }\|_{\sharp, \Theta,\alpha}
	\left(
	|\ln(T-t)|
	\lambda^{\Theta+1}_*  R
	+
	\lambda_j \rho_j
	\right)
	+
	\1_{\{ \cap_{j=1}^N \{|x-q^{\J}| \ge  3d_q \} \}}
	\| \Phi_{\rm {out} }\|_{\sharp, \Theta,\alpha}.
\end{equation}

By \eqref{out-topo}, we have
\begin{equation}\label{out-nabla-upp}
	|\nabla_x \Phi_{\rm {out} } | \le
	\| \Phi_{\rm {out} }\|_{\sharp, \Theta,\alpha} \left(\lambda^{\Theta}_*(0)
	+
	\| Z_* \|_{C^3(\R^2)}  \right).
\end{equation}

For $|x-\xi^{\J}(t)|\le 2\lambda_* R$, by \eqref{lam-ansatz}, then for $T \ll 1$, we have $|x-q^{\J}|\le 3\lambda_* R$. Recall $\Phi$ given in \eqref{u-def}, then
\begin{equation*}
	\Phi - \Phi_{\rm{out}}
	=
	\sum_{j=1}^{N}
	\left( \eta_R^{\J}Q_{\gamma_j}\Phi_{\rm in}^{\J}(y^{\J},t)
	+
	\eta_{d_q}^{\J} \Phi^{*{\J}}_0(r_j,t)\right) .
\end{equation*}
By \eqref{inn-topo} and \eqref{Phi*-0-j-upp}, we have
\begin{align}
%\begin{equation}
%	\begin{aligned}
		&
		\left| \Phi - \Phi_{\rm{out}}  \right|
		\lesssim
		\sum_{j=1}^{N}
		\left[ \eta_R^{\J} \| \Phi_{\rm{in} }^{\J} \|_{{\rm in},\nu-\delta_0,l} \lambda_*^{\nu-\delta_0} \langle \rho_j\rangle^{-l}
		+
		\eta_{d_q}^{\J}
		\left(
		z_j \1_{\{ z_j^2< t+T \}}  +  T |\ln T|^{-1} z_j^{-1}  \1_{\{ z_j^2\ge t+T \}}
		\right)
		\right]
		\nonumber
		\\
		\lesssim \ &
		\sum_{j=1}^{N}
		\left[
		\1_{\{ |x-q^{\J}|\le 3\lambda_* R \}}
		\left(
		\| \Phi_{\rm{in} }^{\J} \|_{{\rm in},\nu-\delta_0,l} \lambda_*^{\nu-\delta_0} \langle \rho_j\rangle^{-l}
		+
		\lambda_* \langle \rho_j\rangle
		\right)
		+
		\1_{\{  3\lambda_* R < |x-q^{\J}| <3d_q \}}
		\lambda_* \langle \rho_j\rangle
		\right].
		\label{P-P_o-est}
%	\end{aligned}
%\end{equation}
\end{align}

By \eqref{inn-topo}, we get
\begin{align}
%\begin{equation}
%	\begin{aligned}
		&
		\big|\nabla_x
		\big( \eta_R^{\J}Q_{\gamma_j}\Phi_{\rm in}^{\J}(y^{\J},t)
		\big) \big|
		=
		\big| \eta_R^{\J} \nabla_x  \big(Q_{\gamma_j}\Phi_{\rm in}^{\J}(y^{\J},t) \big)
		+
		Q_{\gamma_j}\Phi_{\rm in}^{\J}(y^{\J},t)  \nabla_x  \eta_R^{\J}
		\big|
		\nonumber
		\\
		\lesssim \ &
		\eta_R^{\J} \lambda_j^{-1}
		\| \Phi_{\rm{in} }^{\J} \|_{{\rm in},\nu-\delta_0,l} \lambda_*^{\nu-\delta_0} \langle y^{\J}\rangle^{-l-1}
		+
		(\lambda_* R)^{-1} \1_{\{ \lambda_* R \le |x-\xi^{\J}| \le 2  \lambda_* R \}}
		\| \Phi_{\rm{in} }^{\J} \|_{{\rm in},\nu-\delta_0,l} \lambda_*^{\nu-\delta_0} \langle y^{\J}\rangle^{-l}
		\nonumber
		\\
		\lesssim \ &
		\1_{\{ |x-q^{\J}| \le 3\lambda_* R \}}
		\| \Phi_{\rm{in} }^{\J} \|_{{\rm in},\nu-\delta_0,l} \lambda_*^{\nu-\delta_0-1} \langle \rho_j\rangle^{-l-1}.
		\label{nabla-eta*Phiin-upp}
%	\end{aligned}
%\end{equation}
\end{align}

By \eqref{Phi*-0-j-upp}, we have
\begin{equation}\label{nabla-eta*Phi*-upp}
	\big|
	\nabla_x
	\big(
	\eta_{d_q}^{\J} \Phi^{*{\J}}_0(r_j,t)\big) \big|
	=
	\big|
	\eta_{d_q}^{\J} \nabla_x  \Phi^{*{\J}}_0(r_j,t)
	+
	\Phi^{*{\J}}_0(r_j,t) \nabla_x \eta_{d_q}^{\J}  \big|
	\lesssim
	\1_{\{  |x-q^{\J}| < 3d_q\}}.
\end{equation}

Combining \eqref{nabla-eta*Phiin-upp} and \eqref{nabla-eta*Phi*-upp}, we have
\begin{equation}\label{nabla-P-P_o-est}
| \nabla_x ( \Phi - \Phi_{\rm{out}} ) |
	\lesssim
	\sum_{j=1}^{N}
	\Big[
	\1_{\{ |x-q^{\J}| \le 3\lambda_* R \}}
	\Big(
	\| \Phi_{\rm{in} }^{\J} \|_{{\rm in},\nu-\delta_0,l} \lambda_*^{\nu-\delta_0-1} \langle \rho_j\rangle^{-l-1}
	+1
	\Big)
	+ \1_{\{ 3\lambda_* R  < |x-q^{\J}| < 3d_q\}}
	\Big].
\end{equation}

By \eqref{Phi*-0-j-upp} and \eqref{inn-topo}, it holds that
\begin{align}
%\begin{equation}
%	\begin{aligned}
		&
		|\Delta_x ( \Phi - \Phi_{\rm{out}} )|
		=
		\Big| \sum_{j=1}^{N}
		\Big[
		\eta_R^{\J}
		\Delta_x
		\Big(
		Q_{\gamma_j}\Phi_{\rm in}^{\J}(y^{\J},t)
		\Big)
		+
		2 \nabla_x
		\eta_R^{\J}
		\nabla_x  \Big( Q_{\gamma_j}\Phi_{\rm in}^{\J}(y^{\J},t) \Big)
		+
		Q_{\gamma_j}\Phi_{\rm in}^{\J}(y^{\J},t)
		\Delta_x
		\eta_R^{\J}
		\nonumber
		\\
		&
		+
		\eta_{d_q}^{\J}  \Delta_x  \Phi^{*{\J}}_0(r_j,t)
		+
		2\nabla_x\eta_{d_q}^{\J} \nabla_x \Phi^{*{\J}}_0(r_j,t)
		+
		\Phi^{*{\J}}_0(r_j,t)
		\Delta_x  \eta_{d_q}^{\J}
		\Big] \Big|
		\nonumber
		\\
		\lesssim \ &
		\sum_{j=1}^{N}
		\Big[
		\1_{\{ |x-q^{\J}| \le 3 \lambda_* R \}}
		\Big(
		\| \Phi_{\rm{in} }^{\J} \|_{{\rm in},\nu-\delta_0,l} \lambda_*^{\nu-\delta_0-2} \langle \rho_j\rangle^{-l-2}
		+
		\lambda_*^{-1} \langle \rho_j\rangle^{-1}
		\Big)
		+
		\1_{\{ 3 \lambda_* R  < |x-q^{\J}| < 3d_q \}}
		\lambda_*^{-1} \langle \rho_j\rangle^{-1}
		\Big].
		\label{Delta-P-P_o-est}
%	\end{aligned}
%\end{equation}
\end{align}

Combining \eqref{out-upp-split} and \eqref{P-P_o-est}, we have
\begin{equation}\label{Phi-upp}
	\begin{aligned}
		&	|\Phi|
		\lesssim
		\sum_{j=1}^{N}
		\Big[
		\1_{\{ |x-q^{\J}|\le 3\lambda_* R \}}
		\left(
		1+ \|\Phi_{\rm {out} }\|_{\sharp, \Theta,\alpha}
		+
		\| \Phi_{\rm{in} }^{\J} \|_{{\rm in},\nu-\delta_0,l}
		\right)
		\left(
		\lambda_*^{\nu-\delta_0} \langle \rho_j\rangle^{-l}
		+
		\lambda_* \langle \rho_j\rangle
		+
		|\ln(T-t)|
		\lambda^{\Theta+1}_*  R
		\right)
		\\
		&~
		+
		\1_{\{  3\lambda_* R < |x-q^{\J}| <3d_q \}}
		\left(1+
		\| \Phi_{\rm {out} }\|_{\sharp, \Theta,\alpha} \right)
		\left(
		\lambda_* \langle  \rho_j  \rangle
		+
		|\ln(T-t)|
		\lambda^{\Theta+1}_*  R
		\right)
		\Big]
		+
		\1_{\{ \cap_{j=1}^N \{|x-q^{\J}| \ge  3d_q \} \}}
		\| \Phi_{\rm {out} }\|_{\sharp, \Theta,\alpha}.
	\end{aligned}
\end{equation}

Integrating \eqref{out-nabla-upp}, \eqref{nabla-P-P_o-est}, we have
\begin{align}
	%\begin{equation}
	%	\begin{aligned}
		|\nabla_x \Phi |
		\lesssim \ &
		\sum_{j=1}^{N}
		\Big[
		\1_{\{ |x-q^{\J}| \le 3\lambda_* R \}}
		\left(
		1
		+
		\| \Phi_{\rm {out} }\|_{\sharp, \Theta,\alpha}
		+
		\| \Phi_{\rm{in} }^{\J} \|_{{\rm in},\nu-\delta_0,l}
		\right)
		\left(
		\lambda_*^{\nu-\delta_0-1} \langle \rho_j\rangle^{-l-1}
		+1
		\right)
		\nonumber
		\\
		&
		+ \1_{\{ 3\lambda_* R  < |x-q^{\J}| < 3d_q\}}
		\left(1+ \| \Phi_{\rm {out} }\|_{\sharp, \Theta,\alpha}   \right)
		\Big]
		+
		\1_{\{ \cap_{j=1}^N \{|x-q^{\J}| \ge 3 d_q  \} \}}
		\| \Phi_{\rm {out} }\|_{\sharp, \Theta,\alpha}.
		\label{nabla-Phi-upp}
		%	\end{aligned}
	%\end{equation}
\end{align}

By \eqref{Phi-upp}, \eqref{nabla-Phi-upp}, then
\begin{equation}\label{Phi-nablaPhi}
	\begin{aligned}
		&	|\Phi|  |\nabla_x \Phi |
		\lesssim
		\sum_{j=1}^{N}
		\Big[
		\1_{\{ |x-q^{\J}|\le 3\lambda_* R \}}
		\left(
		1+ \|\Phi_{\rm {out} }\|_{\sharp, \Theta,\alpha}
		+
		\| \Phi_{\rm{in} }^{\J} \|_{{\rm in},\nu-\delta_0,l}
		\right)^2
		\\
		&~\times
		\left(
		\lambda_*^{2\nu-2\delta_0-1}\langle \rho_j \rangle^{-2l -1}
		+
		\lambda_*^{\nu-\delta_0} \langle \rho_j \rangle^{-l }
		+
		|\ln(T-t)| \lambda_*^{\nu-\delta_0+\Theta} R
		\langle \rho_j \rangle^{-l-1}
		+
		\lambda_* \langle \rho_j\rangle
		+
		|\ln(T-t)|
		\lambda^{\Theta+1}_*  R
		\right)
		\\
		&~
		+
		\1_{\{  3\lambda_* R < |x-q^{\J}| <3d_q \}}
		\left(1+
		\| \Phi_{\rm {out} }\|_{\sharp, \Theta,\alpha} \right)^2
		\left(
		\lambda_* \langle  \rho_j  \rangle
		+
		|\ln(T-t)|
		\lambda^{\Theta+1}_*  R
		\right)
		\Big]
		+
		\1_{\{ \cap_{j=1}^N \{|x-q^{\J}| \ge  3d_q \} \}}
		\| \Phi_{\rm {out} }\|_{\sharp, \Theta,\alpha}^2.
	\end{aligned}
\end{equation}

Recalling $\Phi_{\rm in}^{\J} \cdot W^{\J}=0$ in \eqref{u-def} yields
\begin{equation*}
	\bigg( 	\sum_{j=1}^{N} \eta_R^{\J}  Q_{\gamma_j}\Phi_{\rm in}^{\J}
	\bigg)  \cdot U_* =
	\sum_{j=1}^{N}  \eta_R^{\J}  Q_{\gamma_j}\Phi_{\rm in}^{\J}
	\cdot \left( U_* -U^{\J} \right),
\end{equation*}
which implies
\begin{equation}\label{Phiin-cdot-U*-up}
	\bigg|	\bigg( 	\sum_{j=1}^{N} \eta_R^{\J}  Q_{\gamma_j}\Phi_{\rm in}^{\J}
	\bigg)  \cdot U_*  \bigg| \lesssim
	\sum_{j=1}^{N}  \eta_R^{\J}  \big| \Phi_{\rm in}^{\J}  \big|
	\lambda_*
	\lesssim
	\sum_{j=1}^{N}
	\1_{\{ |x-q^{\J}|\le 3\lambda_* R \}}
	\| \Phi_{\rm{in} }^{\J} \|_{{\rm in},\nu-\delta_0,l} \lambda_*^{\nu-\delta_0 +1} \langle \rho_j\rangle^{-l}.
\end{equation}

By \eqref{Phiin-cdot-U*-up}, \eqref{Phi*-0-j-upp} and \eqref{out-upp-split}, we obtain
\begin{equation}\label{Phi-cdot-U*-upp}
	\begin{aligned}
		&
		|\Phi\cdot U_* |
		\lesssim
		\sum_{j=1}^{N}
		\Big[
		\1_{\{ |x-q^{\J}|\le 3\lambda_* R \}}
		\Big(1+ \| \Phi_{\rm {out} }\|_{\sharp, \Theta,\alpha} + \| \Phi_{\rm{in} }^{\J} \|_{{\rm in},\nu-\delta_0,l}  \Big)
		\Big(
		\lambda_*^{\nu-\delta_0+1} \langle \rho_j\rangle^{-l}
		+
		\lambda_* \langle \rho_j\rangle
		+
		|\ln(T-t)|
		\lambda^{\Theta+1}_* R \Big)
		\\
		& ~
		+
		\1_{\{  3\lambda_* R < |x-q^{\J}| <3d_q \}}
		( 1 + \| \Phi_{\rm {out} }\|_{\sharp, \Theta,\alpha}  )
		\big(
		|\ln(T-t)|
		\lambda^{\Theta+1}_*  R
		+
		\lambda_* \langle \rho_j \rangle
		\big)
		\Big]
		+
		\1_{\{ \cap_{j=1}^N \{|x-q^{\J}| \ge 3 d_q  \} \}}
		\| \Phi_{\rm {out} }\|_{\sharp, \Theta,\alpha}.
	\end{aligned}
\end{equation}

Using \eqref{inn-topo}, \eqref{nabla-eta*Phiin-upp}, \eqref{nabU*-est}, we have
\begin{align} \notag
		&
		\Big|
		\nabla_x\Big[ \eta_R^{\J}  Q_{\gamma_j}\Phi_{\rm in}^{\J}   \cdot ( U_* -U^{\J} )  \Big] \Big|
		=
		\Big|
		( U_* -U^{\J} )  \cdot \nabla_x\Big(\eta_R^{\J}  Q_{\gamma_j}\Phi_{\rm in}^{\J} \Big)
		+
		\Big( \sum\limits_{k\ne j } \nabla_x U^{\K} \Big)  \cdot
		\eta_R^{\J}  Q_{\gamma_j}\Phi_{\rm in}^{\J}
		\Big|
		\\ \notag
		\lesssim \ &
		\1_{\{ |x-q^{\J}| \le 3\lambda_* R \}}
		\big(
		\lambda_*
		\| \Phi_{\rm{in} }^{\J} \|_{{\rm in},\nu-\delta_0,l} \lambda_*^{\nu-\delta_0-1} \langle \rho_j\rangle^{-l-1}
		+
		\lambda_* \| \Phi_{\rm{in} }^{\J} \|_{{\rm in},\nu-\delta_0,l} \lambda_*^{\nu-\delta_0} \langle \rho_j\rangle^{-l}  \big)
		\\ \label{nab-eatPhiin(U*-U)-up}
		\lesssim \ &
		\1_{\{ |x-q^{\J}| \le 3\lambda_* R \}}
		\| \Phi_{\rm{in} }^{\J} \|_{{\rm in},\nu-\delta_0,l} \lambda_*^{\nu-\delta_0} \langle \rho_j\rangle^{-l-1} .
	\end{align}

\eqref{Phi*-0-j-upp},  \eqref{nabla-eta*Phi*-upp}, and \eqref{nabU*-est} imply
\begin{align}
	%\begin{equation}
	%	\begin{aligned}
		&
		\big|
		\nabla_x\big( \eta_{d_q}^{\J} \Phi^{*{\J}}_0 \cdot U_*  \big) \big|
		=
		\big|
		U_* \cdot 	\nabla_x\big( \eta_{d_q}^{\J} \Phi^{*{\J}}_0 \big)
		+
		\eta_{d_q}^{\J} \Phi^{*{\J}}_0  \cdot \nabla_x U_*
		\big|
		\nonumber
		\\
		\lesssim \ &
		\1_{\{ |x-q^{\J}| < 3d_q\}}
		\Big( 1
		+
		\lambda_j \langle \rho_j \rangle
		\sum\limits_{k=1}^N \lambda_k^{-1} \langle \rho_k \rangle^{-2}
		\Big)
		\lesssim
		\1_{\{ |x-q^{\J}| < 3d_q\}}.
		\label{nab-eatPhi0*U*-up}
		%	\end{aligned}
	%\end{equation}
\end{align}

By \eqref{out-upp-split}, \eqref{out-nabla-upp} and \eqref{nabU*-est}, we have
\begin{equation}\label{nab-out*U*-up}
	\left|
	\nabla_x\left( \Phi_{\rm out} \cdot U_*  \right) \right|
	\lesssim
	\sum\limits_{j=1}^N \1_{\{ |x-q^{\J}| < 3d_q \}}
	\| \Phi_{\rm {out} }\|_{\sharp, \Theta,\alpha}
	\left(
	|\ln(T-t)|
	\lambda^{\Theta}_*  R \langle \rho_j \rangle^{-2}
	+
	1\right)
	+
	\1_{\{ \cap_{j=1}^N \{|x-q^{\J}| \ge 3d_q  \} \}}
	\| \Phi_{\rm {out} }\|_{\sharp, \Theta,\alpha}.
\end{equation}

Combining \eqref{nab-eatPhiin(U*-U)-up}, \eqref{nab-eatPhi0*U*-up} and \eqref{nab-out*U*-up}, we have
\begin{align}
%\begin{equation}
%	\begin{aligned}
		| \nabla_x (\Phi\cdot U_* ) |
		\lesssim \ &
		\sum\limits_{j=1}^N
		\Big[
		\1_{\{ |x-q^{\J}| \le 3\lambda_* R \}}
		\left(1+
		\| \Phi_{\rm {out} }\|_{\sharp, \Theta,\alpha} + \| \Phi_{\rm{in} }^{\J} \|_{{\rm in},\nu-\delta_0,l}  \right)
		\left(
		|\ln(T-t)|
		\lambda^{\Theta}_*  R \langle \rho_j \rangle^{-2}
		+
		1\right)
		\nonumber
		\\
		& +
		\1_{\{ 3\lambda_*R < |x-q^{\J}| < 3d_q \}}
		\left(1+
		\| \Phi_{\rm {out} }\|_{\sharp, \Theta,\alpha}  \right)
		\Big]
		+
		\1_{\{ \cap_{j=1}^N \{|x-q^{\J}| \ge 3d_q  \} \}}
		\| \Phi_{\rm {out} }\|_{\sharp, \Theta,\alpha}.
		\label{nab-Phi-cdot-U*}
%	\end{aligned}
%\end{equation}
\end{align}

\subsection{Estimates of $\nabla_x A$}

Claim:
Suppose that $l>0,
0<\delta_0<\nu <1 $ given in  \eqref{inn-top0-para},
\begin{equation}\label{nab-A-para}
	\Theta>0, \quad
	\Theta+\beta+ 2\delta_0 -2\nu<0, \quad
	\Theta+\beta-1<0,\quad \beta< 1/2,
	\quad
	\Theta+\beta+4\delta_0-4\nu+1<0,\quad
	3\beta< \Theta + 1,
\end{equation}
then for $\nabla_x A$ given in \eqref{nab-A-def} and  $\epsilon>0$ sufficiently small, we have
\begin{align}
\notag
		&	\nabla_x A =  - U_* \cdot \nabla_x U_* - \Phi\cdot \nabla_x \Phi
		+ O\bigg( \sum_{j=1}^{N}
		\Big[ 	\1_{\{ |x-q^{\J}|\le 3\lambda_* R \}}   \left(1+ \| \Phi_{\rm {out} }\|_{\sharp, \Theta,\alpha}   + \| \Phi_{\rm{in} }^{\J} \|_{{\rm in},\nu-\delta_0,l} \right)^4
		\\ \notag
		&~
		\times
		\left(
		\lambda_*^{\epsilon +1} \lambda_*^{\Theta}
		(\lambda_* R)^{-1}
		+
		\lambda_* \langle \rho_j\rangle \right)
		+
		\1_{\{  3\lambda_* R < |x-q^{\J}| <3d_q \}}
		\left(
		1
		+
		\| \Phi_{\rm {out} }\|_{\sharp, \Theta,\alpha}
		\right)^4 \left(
		|\ln(T-t)|
		\lambda^{\Theta+1}_*  R
		+
		\lambda_* \langle \rho_j \rangle
		\right)   \Big]
		\\ \label{nab-A-est}
		& ~
		+
		\1_{\{ \cap_{j=1}^N \{|x-q^{\J}| \ge  3d_q \} \}}
		\left(
		1
		+
		\| \Phi_{\rm {out} }\|_{\sharp, \Theta,\alpha}
		\right)^4 \bigg).
	\end{align}
Using \eqref{U*cdot-nabU*-split}, \eqref{Phi-nablaPhi}, then
\begin{align}
%\begin{equation}
%	\begin{aligned}
		&	\nabla_x A =   O\bigg( \sum_{j=1}^{N}
		\Big[ 	\1_{\{ |x-q^{\J}|\le 3\lambda_* R \}}   \left(1+ \| \Phi_{\rm {out} }\|_{\sharp, \Theta,\alpha}   + \| \Phi_{\rm{in} }^{\J} \|_{{\rm in},\nu-\delta_0,l} \right)^4
		\nonumber
		\\
		&~
		\times
		\left( \langle \rho_j\rangle^{-2} +
		\lambda_*^{2\nu-2\delta_0-1}\langle \rho_j \rangle^{-2l -1}
		+
		\lambda_*^{\nu-\delta_0} \langle \rho_j \rangle^{-l }
		+
		|\ln(T-t)| \lambda_*^{\nu-\delta_0+\Theta} R
		\langle \rho_j \rangle^{-l-1}
		+
		\lambda_*^{\epsilon +1} \lambda_*^{\Theta}
		(\lambda_* R)^{-1}
		+
		\lambda_* \langle \rho_j\rangle \right)
		\nonumber
		\\
		&~
		+
		\1_{\{  3\lambda_* R < |x-q^{\J}| <3d_q \}}
		\left(
		1
		+
		\| \Phi_{\rm {out} }\|_{\sharp, \Theta,\alpha}
		\right)^4 \left(
		\langle \rho_j\rangle^{-2}
		+
		|\ln(T-t)|
		\lambda^{\Theta+1}_*  R
		+
		\lambda_* \langle \rho_j \rangle
		\right)   \Big]
		\nonumber
		\\
		& ~
		+
		\1_{\{ \cap_{j=1}^N \{|x-q^{\J}| \ge  3d_q \} \}}
		\left(
		1
		+
		\| \Phi_{\rm {out} }\|_{\sharp, \Theta,\alpha}
		\right)^4 \bigg).
		\label{nab-A-rough}
%	\end{aligned}
%\end{equation}
\end{align}

\begin{proof}[Proof of Claim]
	
	First, let us simplify \eqref{nab-A-def}.
	\begin{align*}
%	\begin{equation*}
%		\begin{aligned}
			&
			|\Pi_{U_*^{\perp}}\Phi|^2
			=
			|\Phi|^2 +
			( | U_* |^2 -2 )
			( \Phi\cdot U_* )^2 ,\quad
			U_* \cdot  \Pi_{U_*^{\perp}}\Phi
			=
			( 1-|U_*|^2 ) ( \Phi\cdot U_* ) ,
			\\
			&
			\nabla_x (	|\Pi_{U_*^{\perp}}\Phi|^2 )
			=
			2\Phi\cdot \nabla_x \Phi +
			2 ( \Phi\cdot U_* )^2  U_*\cdot \nabla_x U_*
			+
			2 (  | U_*|^2 -2)
			( \Phi\cdot U_* )\nabla_x ( \Phi\cdot U_* ) ,
			\\
			&
			\nabla_x ( U_* \cdot  \Pi_{U_*^{\perp}}\Phi )
			=
			( 1-|U_*|^2 ) \nabla_x ( \Phi\cdot U_* )
			-2 ( \Phi\cdot U_* ) U_*\cdot \nabla_x U_*.
%		\end{aligned}
%	\end{equation*}
	\end{align*}
	
	By \eqref{A-est}, \eqref{U*-norm}, \eqref{lam-ansatz} and \eqref{RPhi-ansatz}, we have
	\begin{equation}\label{one in nabla A}
		(1+A)|U_*|^2+( U_* \cdot
		\Pi_{U_*^{\perp}}\Phi )
		=
		1+ O\big(\lambda_*+|\Phi|^2\big),
		\mbox{ \ then \ }
		\big[ (1+A)|U_*|^2+( U_* \cdot
		\Pi_{U_*^{\perp}}\Phi ) \big]^{-1}
		= 1+ O\big(\lambda_*+|\Phi|^2\big).
	\end{equation}
	Thus we obtain
	\begin{align*}
		%	\begin{equation}
			%		\begin{aligned}
				&
				\nabla_x A =  -
				\left( 1+ O\left(\lambda_*+|\Phi|^2\right) \right)
				\Big\{ (1+A)^2  U_* \cdot \nabla_x U_*
				+ \Phi\cdot \nabla_x \Phi +
				\left( \Phi\cdot U_*\right)^2  U_*\cdot \nabla_x U_*
				\\
				&
				+
				\left( \left| U_* \right|^2 -2 \right)
				\left( \Phi\cdot U_*\right)\nabla_x\left( \Phi\cdot U_*\right)
				+(1+A) \left[ \left( 1-|U_*|^2\right) \nabla_x\left( \Phi\cdot U_*\right)
				-2 \left( \Phi\cdot U_*\right) U_*\cdot \nabla_x U_* \right] \Big\}
				\\
				= \ &
				-
				\left( 1+ O\left(\lambda_*+|\Phi|^2\right) \right)
				\Big\{ \left[ 1+ A(2+A)  -2(1+A) \left( \Phi\cdot U_*\right) + \left( \Phi\cdot U_*\right)^2   \right] U_* \cdot \nabla_x U_*
				+ \Phi\cdot \nabla_x \Phi
				\\
				&
				+
				\left( \left| U_* \right|^2 -2 \right)
				\left( \Phi\cdot U_*\right)\nabla_x\left( \Phi\cdot U_*\right)
				+(1+A) \left( 1-|U_*|^2\right) \nabla_x\left( \Phi\cdot U_*\right)  \Big\}
				\\
				= \ &
				- U_* \cdot \nabla_x U_*
				-
				\Phi\cdot \nabla_x \Phi
				+ \left(2  \Phi\cdot U_* +
				O\left(\lambda_*+|\Phi|^2\right)  \right)  U_* \cdot \nabla_x U_*
				-  O\left(\lambda_*+|\Phi|^2\right)     \Phi\cdot \nabla_x \Phi
				\\
				&
				-
				\left( 1+ O\left(\lambda_*+|\Phi|^2\right) \right)
				\Big[
				\left( \left| U_* \right|^2 -2 \right)
				\left( \Phi\cdot U_*\right)\nabla_x\left( \Phi\cdot U_*\right)
				+(1+A) \left( 1-|U_*|^2\right) \nabla_x\left( \Phi\cdot U_*\right)  \Big],
				%		\end{aligned}
			%	\end{equation}
	\end{align*}
	where we used $ A =O( \lambda_* + |\Phi|^2 ) \ll 1$ by \eqref{A-est} and \eqref{RPhi-ansatz}.

	By \eqref{Phi-upp}, \eqref{Phi-cdot-U*-upp}, and parameter assumption \eqref{inn-top0-para}, we have
	\begin{align}
	\notag
			&  |\Phi\cdot U_*|+ \lambda_* + |\Phi|^2
			\\ \notag
			\lesssim \ &
			\sum_{j=1}^{N}
			\Big[
			\1_{\{ |x-q^{\J}|\le 3\lambda_* R \}}
			\left(
			1+ \|\Phi_{\rm {out} }\|_{\sharp, \Theta,\alpha}
			+
			\| \Phi_{\rm{in} }^{\J} \|_{{\rm in},\nu-\delta_0,l}
			\right)^2
			\left(
			\lambda_*^{2\nu-2\delta_0} \langle \rho_j\rangle^{-l}
			+
			\lambda_* \langle \rho_j\rangle
			+
			|\ln(T-t)|
			\lambda^{\Theta+1}_*  R
			\right)
			\\ \notag
			&
			+
			\1_{\{  3\lambda_* R < |x-q^{\J}| <3d_q \}}
			\left(1+
			\| \Phi_{\rm {out} }\|_{\sharp, \Theta,\alpha} \right)^2
			\left(
			\lambda_* \langle  \rho_j  \rangle
			+
			|\ln(T-t)|
			\lambda^{\Theta+1}_*  R
			\right)
			\Big]
			\\ \label{toget-est-1}
			&
			+
			\1_{\{ \cap_{j=1}^N \{|x-q^{\J}| \ge  3d_q \} \}}
			\left(1+
			\| \Phi_{\rm {out} }\|_{\sharp, \Theta,\alpha}
			\right)^2.
		\end{align}
	
	Notice $\lambda_*^{\epsilon +1} \lambda_*^{\Theta}
	(\lambda_* R)^{-1}  =  \lambda_*^{\epsilon +\Theta+\beta}  $. Then  using \eqref{U*cdot-nabU*-split} and \eqref{toget-est-1}, we get
	\begin{align}
	\notag
			&
			\left|
			\left(2  \Phi\cdot U_* +
			O\left(\lambda_*+|\Phi|^2\right)  \right)  U_* \cdot \nabla_x U_* \right|
			\lesssim
			\sum_{j=1}^{N}
			\bigg\{ 	\1_{\{ |x-q^{\J}|\le 3\lambda_* R \}} \left(1+ \| \Phi_{\rm {out} }\|_{\sharp, \Theta,\alpha}   + \| \Phi_{\rm{in} }^{\J} \|_{{\rm in},\nu-\delta_0,l} \right)^2
			\\ \notag
			& \times \Big(
			\lambda_*^{2\nu-2\delta_0} \langle \rho_j\rangle^{-l-2}
			+
			\lambda_* \langle \rho_j\rangle^{-1} +
			|\ln(T-t)|
			\lambda^{\Theta+1}_* R  \langle \rho_j\rangle^{-2}
			\Big)
			\\ \notag
			&
			+
			\1_{\{  3\lambda_* R < |x-q^{\J}| <3d_q \}}
			\left(
			1
			+
			\| \Phi_{\rm {out} }\|_{\sharp, \Theta,\alpha}
			\right)^2
			\left(
			\lambda_* \langle  \rho_j  \rangle^{-1}
			+
			|\ln(T-t)|
			\lambda^{\Theta+1}_*  R
			\langle \rho_j\rangle^{-2}
			\right)  \bigg\}
			\\ \notag
			&
			+
			\1_{\{ \cap_{j=1}^N \{|x-q^{\J}| \ge  3d_q \} \}}
			\left(
			1
			+
			\| \Phi_{\rm {out} }\|_{\sharp, \Theta,\alpha}
			\right)^2  \lambda_*^2
			\\ \notag
			\lesssim \ &
			\sum_{j=1}^{N}
			\Big[ 	\1_{\{ |x-q^{\J}|\le 3\lambda_* R \}}
			\left(1+ \| \Phi_{\rm {out} }\|_{\sharp, \Theta,\alpha}   + \| \Phi_{\rm{in} }^{\J} \|_{{\rm in},\nu-\delta_0,l} \right)^2
			\lambda_*^{\epsilon +1} \lambda_*^{\Theta}
			(\lambda_* R)^{-1}
			\\ \label{Z-nabA-1}
			&
			+
			\1_{\{  3\lambda_* R < |x-q^{\J}| <3d_q \}}
			\left(
			1
			+
			\| \Phi_{\rm {out} }\|_{\sharp, \Theta,\alpha}
			\right)^2
			\lambda_* \langle  \rho_j  \rangle^{-1}    \Big]
			+
			\1_{\{ \cap_{j=1}^N \{|x-q^{\J}| \ge  3d_q \} \}}
			\left(
			1
			+
			\| \Phi_{\rm {out} }\|_{\sharp, \Theta,\alpha}
			\right)^2  \lambda_*^2
		\end{align}
	for some $\epsilon>0$, where for the last ``$\lesssim$'', we require
	\begin{equation}\label{N-nabA-para-1}
		\Theta>0 ,\quad
		\Theta+\beta+ 2\delta_0 -2\nu<0,\quad
		\Theta+\beta-1<0,\quad \beta< 1/2.
	\end{equation}

	By \eqref{Phi-upp} and \eqref{nabla-Phi-upp}, it follows that
	%\begin{equation}\label{Z-nabA-2}
	%\begin{aligned}
	\begingroup
	\allowdisplaybreaks
	\begin{align}
		&
		\left|   O\left(\lambda_*+|\Phi|^2\right)  \Phi\cdot \nabla_x \Phi  \right|
		\lesssim
		\big( |\Phi|^3 + \lambda_* |\Phi| \big) \left|\nabla_x \Phi \right|  \nonumber
		\\
		\lesssim \ &
		\sum_{j=1}^{N}
		\Big[
		\1_{\{ |x-q^{\J}|\le 3\lambda_* R \}}  \left(
		1+ \|\Phi_{\rm {out} }\|_{\sharp, \Theta,\alpha}
		+
		\| \Phi_{\rm{in} }^{\J} \|_{{\rm in},\nu-\delta_0,l}
		\right)^4   \nonumber
		\\
		&\times
		\Big(
		\lambda_*^{4\nu-4\delta_0-1}
		+
		\lambda_*^3 \langle \rho_j\rangle^3
		+
		\lambda_*^{\nu-\delta_0+2} \langle \rho_j\rangle^{2-l}
		+
		|\ln(T-t)|^3
		\lambda^{\nu-\delta_0+3\Theta+2}_*  R^3   \nonumber
		\\
		&
		+
		\lambda_*^{2\nu-2\delta_0}
		+
		|\ln(T-t)| \lambda_*^{\nu-\delta_0+\Theta+1} R
		+
		\lambda_*^2 \langle \rho_j\rangle
		\Big)   \nonumber
		\\
		&
		+
		\1_{\{  3\lambda_* R < |x-q^{\J}| <3d_q \}}
		\left(1+
		\| \Phi_{\rm {out} }\|_{\sharp, \Theta,\alpha} \right)^4
		\left(
		\lambda_*^3 \langle  \rho_j  \rangle^3
		+
		|\ln(T-t)|^3
		\lambda^{3\Theta+3}_*  R^3
		+
		\lambda_*^2 \langle  \rho_j  \rangle
		+
		|\ln(T-t)|
		\lambda^{\Theta+2}_*  R
		\right)
		\Big]    \nonumber
		\\
		&
		+
		\1_{\{ \cap_{j=1}^N \{|x-q^{\J}| \ge  3d_q \} \}}
		\left(1+ \| \Phi_{\rm {out} }\|_{\sharp, \Theta,\alpha} \right)^4    \nonumber
		\\
		\lesssim \ &
		\sum_{j=1}^{N}
		\Big[
		\1_{\{ |x-q^{\J}|\le 3\lambda_* R \}}
		\left(
		1+ \|\Phi_{\rm {out} }\|_{\sharp, \Theta,\alpha}
		+
		\| \Phi_{\rm{in} }^{\J} \|_{{\rm in},\nu-\delta_0,l}
		\right)^4
		\lambda_*^{\epsilon +1} \lambda_*^{\Theta}
		(\lambda_* R)^{-1}   \nonumber
		\\
		&
		+
		\1_{\{  3\lambda_* R < |x-q^{\J}| <3d_q \}}
		\left(1+
		\| \Phi_{\rm {out} }\|_{\sharp, \Theta,\alpha} \right)^4
		\left(
		\lambda_*^3 \langle  \rho_j  \rangle^3
		+
		|\ln(T-t)|^3
		\lambda^{3\Theta+3}_*  R^3
		+
		\lambda_*^2 \langle  \rho_j  \rangle
		+
		|\ln(T-t)|
		\lambda^{\Theta+2}_*  R
		\right)
		\Big]   \nonumber
		\\
		&
		+
		\1_{\{ \cap_{j=1}^N \{|x-q^{\J}| \ge  3d_q \} \}}
		\left(1+ \| \Phi_{\rm {out} }\|_{\sharp, \Theta,\alpha} \right)^4 \label{Z-nabA-2},
		%	\end{aligned}
	%\end{equation}
\end{align}
\endgroup
where  we require \eqref{inn-top0-para} and additional parameters restriction for the last ``$\lesssim$''
\begin{equation}\label{N-nabA-para-2}
	\beta< 1/2,
	\quad
	\Theta+\beta+4\delta_0-4\nu+1<0,\quad
	\Theta +\beta +2\delta_0-2\nu<0,
	\quad \Theta+ 2\beta <2 .
\end{equation}

Combining \eqref{Phi-cdot-U*-upp} and \eqref{nab-Phi-cdot-U*}, we have
\begin{equation}\label{Z-nabA-3}
	\begin{aligned}
		&
		\left| \left( \Phi\cdot U_*\right)\nabla_x\left( \Phi\cdot U_*\right) \right|
		\lesssim
		\sum_{j=1}^{N}
		\Big[
		\1_{\{ |x-q^{\J}|\le 3\lambda_* R \}}
		\left(1+ \| \Phi_{\rm {out} }\|_{\sharp, \Theta,\alpha} + \| \Phi_{\rm{in} }^{\J} \|_{{\rm in},\nu-\delta_0,l}  \right)^2
		\\
		&\times
		\left( |\ln(T-t)|
		\lambda_*^{\nu-\delta_0+1+\Theta} R
		+
		|\ln(T-t)|
		\lambda^{\Theta+1}_*  R
		+
		|\ln(T-t)|^2
		\lambda^{2\Theta+1}_* R^2
		+
		\lambda_*^{\nu-\delta_0+1}
		+
		\lambda_* \langle \rho_j\rangle
		\right)
		\\
		&
		+
		\1_{\{  3\lambda_* R < |x-q^{\J}| <3d_q \}}
		\left( 1 + \| \Phi_{\rm {out} }\|_{\sharp, \Theta,\alpha}  \right)^2
		\left(
		|\ln(T-t)|
		\lambda^{\Theta+1}_*  R
		+
		\lambda_* \langle \rho_j \rangle
		\right)
		\Big]
		+
		\1_{\{ \cap_{j=1}^N \{|x-q^{\J}| \ge 3 d_q  \} \}}
		\| \Phi_{\rm {out} }\|_{\sharp, \Theta,\alpha}^2
		\\
		\lesssim \ &
		\sum_{j=1}^{N}
		\Big[
		\1_{\{ |x-q^{\J}|\le 3\lambda_* R \}}
		\left(1
		+  \| \Phi_{\rm {out} }\|_{\sharp, \Theta,\alpha} + \| \Phi_{\rm{in} }^{\J} \|_{{\rm in},\nu-\delta_0,l} \right)^2
		\left(
		\lambda_*^{\epsilon +1} \lambda_*^{\Theta}
		(\lambda_* R)^{-1}
		+
		\lambda_* \langle\rho_j\rangle
		\right)
		\\
		&
		+
		\1_{\{  3\lambda_* R < |x-q^{\J}| <3d_q \}}
		\left(1+
		\| \Phi_{\rm {out} }\|_{\sharp, \Theta,\alpha} \right)^2
		\left(
		|\ln(T-t)|
		\lambda^{\Theta+1}_*  R
		+
		\lambda_* \langle \rho_j \rangle
		\right)
		\Big]
		+
		\1_{\{ \cap_{j=1}^N \{|x-q^{\J}| \ge 3 d_q  \} \}} 	\| \Phi_{\rm {out} }\|_{\sharp, \Theta,\alpha}^2,
	\end{aligned}
\end{equation}
where for the last step, we require
\begin{equation}\label{N-nabA-para-3}
	\beta<1/2,\quad 3\beta<\Theta+1,\quad \Theta+\beta<
	\nu-\delta_0 +1.
\end{equation}

By \eqref{U*-norm}, \eqref{nab-Phi-cdot-U*}, we have
\begin{align}
%\begin{equation}
%	\begin{aligned}
		&
		\left| \left( 1-|U_*|^2\right) \nabla_x\left( \Phi\cdot U_*\right)  \right|
		\lesssim
		\sum\limits_{j=1}^N
		\Big[
		\1_{\{ |x-q^{\J}|\le 3\lambda_* R \}}
		\left(
		1+ \| \Phi_{\rm {out} }  \|_{\sharp, \Theta,\alpha}
		+ \| \Phi_{\rm{in} }^{\J} \|_{{\rm in},\nu-\delta_0,l}
		\right)
		\lambda_*^{\epsilon +1} \lambda_*^{\Theta}
		(\lambda_* R)^{-1}
		\nonumber
		\\
		&
		+ \1_{\{ 3\lambda_* R  < |x-q^{\J}| < 3d_q\}}
		\lambda_* \left(
		1+ \| \Phi_{\rm {out} }\|_{\sharp, \Theta,\alpha}
		\right)
		\Big]
		+
		\1_{\{ \cap_{j=1}^N \{|x-q^{\J}| \ge  3 d_q  \} \}}
		\lambda_*
		\| \Phi_{\rm {out} }\|_{\sharp, \Theta,\alpha}
		\label{Z-nabA-4}
%	\end{aligned}
%\end{equation}
\end{align}
provided
\begin{equation}\label{N-nabA-para-4}
	\Theta+\beta<1,\quad \beta<1/2 .
\end{equation}

Under the parameters restriction \eqref{nab-A-para}, which is the combination of \eqref{N-nabA-para-1}, \eqref{N-nabA-para-2}, \eqref{N-nabA-para-3}, and \eqref{N-nabA-para-4},
we conclude the validity of \eqref{nab-A-est} from \eqref{Z-nabA-1}, \eqref{Z-nabA-2}, \eqref{Z-nabA-3}, and \eqref{Z-nabA-4}.
\end{proof}

\subsection{Complete estimates of $\mathcal{G}$}

\begin{lemma}\label{G-est-lem}
For $\mathcal{G}$ given in \eqref{G-def},
suppose that the ansatz \eqref{lam-ansatz} holds,
$\Phi_{\rm {out} } \in B_{\rm{out}}$ defined in \eqref{out-space},
$\| \Phi_{\rm{in} }^{\J} \|_{{\rm in},\nu-\delta_0,l} \le \Lambda_{\rm{in}}$, $ \Phi_{\rm{in} }^{\J} \cdot W^{\J}=0$ for $j=1,2,\dots,N$, under the parameter assumptions
\begin{align}
%\begin{equation}
%	\begin{aligned}
		& l>0,\quad
		0<\Theta<\beta<1/2,
		\quad
		\Theta  +\beta + \delta_0 	-\nu <0,
		\quad 3\beta< 1+\Theta , \quad
		\beta(l+1)-1+\nu-\delta_0 -\Theta>0 ,
		\nonumber
		\\
		&
		\Theta+	2\beta -1<0, \quad 0<\delta_0  < \nu<1 , \quad 	2\beta+\delta_0-\nu<0,
		\quad \Theta+\beta+1+3\delta_0-3\nu<0,
		\label{G-para}
%	\end{aligned}
%\end{equation}
\end{align}
then there exists $0<\epsilon \ll 1$ such that
$\| \mathcal{G}\|_{**} \lesssim T^\epsilon $ with the norm  $\|\cdot\|_{**}$ defined in \eqref{G-topo-def}.

\end{lemma}

\begin{remark}\label{qd240727-1-rem}
In the process of the analysis, there is a delicate cancellation for $\Delta_x U_* -2\left(U_* \cdot \nabla_x U_*\right)\cdot \nabla_x U_*$. See \eqref{cancel-1} and \eqref{special-1}, \eqref{special-2}.
\end{remark}

\begin{proof}

\noindent $\bullet$
Recall $\eta_R^{\J}$ given in \eqref{qd24Apr12-3}.
By \eqref{lam-ansatz}, $R(t) = \lambda_*^{-\beta}(t)$ with $\beta>0$ given \eqref{RPhi-ansatz}, $T\ll 1$, then $1-\eta_R^{\J} \le \1_{\{ |x-q^{\J}| \ge \lambda_* R/2 \}}$ and $\1_{\{ |x-q^{\J}| \ge \lambda_* R/2 \}} |x-\xi^{\J}| \sim \1_{\{ |x-q^{\J}| \ge \lambda_* R/2 \}} |x-q^{\J}|$.
By \eqref{nabU*-est}, \eqref{out-upp}, and \eqref{out-nabla-upp}, we have
\begin{align*}
		&
		\big|
		(1-\eta_R^{\J} )
		( a-bU^{\J}  \wedge )
		\big[	|\nabla_x U^{\J} |^2 \Phi_{\rm {out} }
		- 2 \nabla_x ( U^{\J} \cdot \Phi_{\rm out} ) \cdot \nabla_x U^{\J}
		\big]
		\big|
		\\
		\lesssim \ &
		\1_{\{ |x-q^{\J}| \ge \lambda_* R/2 \}}
		\big(	 | \nabla_x \Phi_{\rm {out} }  | \lambda_j^{-1} \langle \rho_j\rangle^{-2}
		+
		\lambda_j^{-2}\langle \rho_j\rangle^{-4} |\Phi_{\rm {out} } |
		\big)
		\\
		\lesssim \ &
		\1_{\{ |x-q^{\J}| \ge \lambda_* R/2 \}}
		\| \Phi_{\rm {out} }\|_{\sharp, \Theta,\alpha}
		\big\{	
		(\lambda^{\Theta}_*(0)
		+
		\| Z_* \|_{C^3(\R^2)} )
		\lambda_* |x-q^{\J}|^{-2}
		 + \lambda_*^{2} |x-q^{\J}|^{-4}
		\big[
		|\ln(T-t)|
		\lambda^{\Theta+1}_* R
		\\
		&
		+
		(T-t) \| Z_* \|_{C^3(\R^2)}
		+
		|x-q^{\J}| (\lambda^{\Theta}_*(0)
		+
		\| Z_* \|_{C^3(\R^2)} )
		\big]
		\big\}
		\lesssim
		T^{\epsilon} (\varrho_2^{\J} + \varrho_3 ).
	\end{align*}

\noindent $\bullet$
\begin{align*}
	%\begin{equation*}
	%	\begin{aligned}
		&
		\Big|
		\left(1-\eta_R^{\J}\right)
		\Big\{
		-\pp_t (\eta_{d_q}^{\J} \Phi_0^{*\J} )
		+
		\left( a-bU^{\J}  \wedge \right)
		\Big[
		\Delta_x (\eta_{d_q}^{\J}\Phi_0^{*\J} ) +
		|\nabla_x U^{\J} |^2 \eta_{d_q}^{\J} \Phi^{*{\J}}_0
		\\
		&
		- 2 \nabla_x \left(
		U^{\J} \cdot
		\eta_{d_q}^{\J} \Phi^{*{\J}}_0  \right)
		\cdot \nabla_x U^{\J}
		\Big]
		- \partial_t U^{\J}
		\Big\} \Big|
		\\
		= \ &
		\Big|
		\left(1-\eta_R^{\J}\right) \eta_{d_q}^{\J}
		\Big\{
		- \pp_t (\Phi_0^{*\J})
		+
		\left( a-bU^{\J}  \wedge \right)
		\Big[
		\Delta_x \Phi_0^{*\J}
		+
		|\nabla_x U^{\J} |^2 \Phi^{*{\J}}_0
		-
		2 \nabla_x \left(
		U^{\J} \cdot
		\Phi^{*{\J}}_0  \right) \cdot \nabla_x U^{\J}
		\Big]
		\\
		&
		- \partial_t U^{\J}
		\Big\}
		- (1-\eta_{d_q}^{\J}) \partial_t U^{\J}
		+
		\left(1-\eta_R^{\J}\right)
		\Big\{
		-  \Phi_0^{*\J}  \pp_t \eta_{d_q}^{\J}
		+
		\left( a-bU^{\J}  \wedge \right)
		\Big[
		2\nabla_x \eta_{d_q}^{\J} \cdot \nabla_x \Phi_0^{*\J}
		+
		\Phi_0^{*\J} \Delta_x  \eta_{d_q}^{\J}
		\\
		&
		-
		2  \left(
		U^{\J} \cdot
		\Phi^{*{\J}}_0  \right) \nabla_x \eta_{d_q}^{\J} \cdot \nabla_x U^{\J}
		\Big]
		\Big\} \Big|
		\\
		\lesssim \ &
		\1_{\{ \lambda_* R \le |x-\xi^{\J}| \le 2d_q \}}
		\big(
		\lambda_*^{-1} \langle \rho_j\rangle^{-2}
		+
		|\dot{\lambda}_*|
		\langle \rho_j\rangle^{-1}
		+
		|\dot{\xi}^{\J}|
		\big)
		\\
		&
		+ \1_{\{  |x-\xi^{\J}| \ge d_q \}} \big[
		\big(\la_j^{-1}|\dot\la_j| + |\dot\gamma_j |  \big) \langle \rho_j \rangle^{-1}
		+
		\la_j^{-1}|\dot{\xi}^{\J}| \langle \rho_j \rangle^{-2}
		\big]
		+
		\1_{\{ d_q  \le |x-\xi^{\J}| \le 2d_q \}}
		\\
		\lesssim \ &
		\1_{\{ \lambda_* R/2 \le |x-q^{\J}| \le 3d_q \}}
		\big(
		\lambda_* |x-q^{\J}|^{-2}
		+
		|\dot{\lambda}_*| R^{-1}
		+
		|\dot{\xi}^{\J}|
		\big)
		\\
		&
		+ \1_{\{  |x-\xi^{\J}| \ge d_q \}} \big[
		\big( |\dot\la_j| + \la_j |\dot\gamma_j |  \big)
		+
		\la_j |\dot{\xi}^{\J}|
		\big]
		+
		\1_{\{ d_q  \le |x-\xi^{\J}| \le 2d_q \}}
		\lesssim
		T^\epsilon \big(  \varrho_2^{\J} + \varrho_3
		\big),
		%	\end{aligned}
	%\end{equation*}
\end{align*}
where we used  \eqref{Sj-est}, \eqref{ppt-U-est}, \eqref{Phi*-0-j-upp}, \eqref{nabU*-est} for the first ``$\lesssim$'',
and \eqref{lam-ansatz} for the second and third ``$\lesssim$''.

\noindent $\bullet$
By \eqref{C-cal-inverse}, \eqref{M-est}, we have
\begin{equation*}
	\big| \eta_R^{\J}
	\big( \tilde{M}_0^{\J} + e^{i\theta_j} \tilde{M}_1^{\J} + e^{-i\theta_j} M_{-1}^{\J}\big)_{\mathcal{C}_j^{-1}} \big|
	\lesssim \eta_R^{\J}
	\big(
	|\dot{\lambda}_*|
	\langle \rho_j\rangle^{-1}
	+ |\dot{\xi}^{\J}| \big)
	\lesssim T^{\epsilon} \varrho_1.
\end{equation*}

\noindent $\bullet$
Using $| \dot{\gamma}_j(t) |\le C_{\gamma} (T-t)^{-1}$,
$| \dot{\xi}^{\J}(t) | \le C_{\xi} \lambda_*^{\epsilon_{\xi}}(t)$ in \eqref{lam-ansatz}, one has
\begin{equation*}
	\left| \eta_R^{\J}Q_{\gamma_j}\left[
	\left(\la_j^{-1}\dot\la_j y^{\J}+\la_j^{-1}\dot\xi^{\J} \right)
	\cdot\nabla_{y^{\J}} \Phi_{\rm in}^{\J}-\dot\gamma_j J\Phi_{\rm in}^{\J}\right] \right|
	\lesssim
	\eta_R^{\J}
	(T-t)^{-1}
	\|  \Phi_{\rm{in} }^{\J} \|_{{\rm in},\nu-\delta_0,l} \lambda_*^{\nu-\delta_0}
	\lesssim T^{\epsilon} \varrho_1^{\J}
\end{equation*}
provided
\begin{equation}\label{o-para-2}
	\Theta  + \delta_0 +\beta	-\nu <0.
\end{equation}

\noindent $\bullet$
Using $\| \Phi_{\rm{in} }^{\J} \|_{{\rm in},\nu-\delta_0,l} \le \Lambda_{\rm{in}}$, $ \Phi_{\rm{in} }^{\J} \cdot W^{\J}=0$, and \eqref{lam-ansatz}, we have
\begin{align}
\notag
		&
		\left|
		Q_{\gamma_j}
		\left\{ - \Phi_{\rm in}^{\J} \pp_t \eta_R^{\J}
		+
		\left( 	a -b W^{\J}  \wedge \right)
		\left[ \Phi_{\rm in}^{\J} \Delta_x \eta_R^{\J}
		+2 \nabla_x \eta_R^{\J}
		\cdot \nabla_x \Phi_{\rm in}^{\J}
		-
		\left(
		W^{\J} \cdot   \Phi_{\rm in}^{\J} \right) \left(
		2  \nabla_x \eta_R^{\J} \cdot \nabla_x W^{\J}
		\right)
		\right] \right\} \right|
		\\ \notag
		= \ &
		\Big| \Phi_{\rm in}^{\J} (\nabla \eta ) \Big(\frac{x-\xi^{\J} }{\la_* R }\Big)
		\cdot
		\Big( \frac{\dot{\xi}^{\J}}{\lambda_* R} +
		\frac{x-\xi^{\J}}{\lambda_* R } \frac{(\lambda_* R )'}{\lambda_* R }
		\Big)
		\\ \notag
		&
		+
		\left( 	a -b W^{\J}  \wedge \right)
		\Big[ \Phi_{\rm in}^{\J}  (\la_* R)^{-2} ( \Delta \eta ) \Big(\frac{x-\xi^{\J} }{\la_* R }\Big)
		+2 (\la_* R)^{-1} ( \nabla \eta )\Big(\frac{x-\xi^{\J} }{\la_* R }\Big)  \cdot \la_j^{-1} \nabla_{y^{\J}} \Phi_{\rm in}^{\J} \Big|
		\\ \notag
		\lesssim \ &
		\| \Phi_{\rm{in} }^{\J} \|_{{\rm in},\nu-\delta_0,l}
		\1_{\{ \la_* R \le |x-\xi^{\J}| \le 2\la_* R \}}
		\Big[
		(T-t)^{-1}
		\lambda_*^{\nu-\delta_0} \langle y^{\J}\rangle^{-l}
		+
		(\la_* R)^{-2}  \lambda_*^{\nu-\delta_0} \langle y^{\J}\rangle^{-l}
		\\ \label{coupling est}
		&
		+ (\la_* R)^{-1}    \la_j^{-1}
		\lambda_*^{\nu-\delta_0} \langle y^{\J}\rangle^{-l-1}
		\Big]
		\lesssim
		\1_{\{ \la_* R/2 \le |x-q^{\J}| \le 3\la_* R \}}
		(\la_* R)^{-2}  \lambda_*^{\nu-\delta_0} R^{-l}
		\lesssim
		T^{\epsilon} \varrho_1
	\end{align}
provided
\begin{equation}\label{o-para-3} 	\beta<1/2, \quad
	\nu-\delta_0 +\beta l -(1-\beta) > \Theta.
\end{equation}

\noindent $\bullet$
Using
$U^{\J} \cdot ( Q_{\gamma_j}\Phi_{\rm in}^{\J} )
=
W^{\J} \cdot \Phi_{\rm in}^{\J}
=0
$, we have
\begin{align*}
%\begin{equation*}
%	\begin{aligned}
		&
		( U_*-U^{\J} ) \wedge  \Big\{
		\Delta_x \Big(\eta_{d_q}^{\J}\Phi_0^{*\J} \Big)
		+ \eta_R^{\J}Q_{\gamma_j}  \Delta_x \Phi_{\rm in}^{\J}
		+ Q_{\gamma_j}\Big( \Phi_{\rm in}^{\J} \Delta_x \eta_R^{\J}
		+2 \nabla_x \eta_R^{\J}
		\cdot \nabla_x \Phi_{\rm in}^{\J} \Big)
		\\
		&
		- 2 \nabla_x ( U^{\J} \cdot \Phi_{\rm out}  )\cdot \nabla_x U^{\J}
		- 2 \nabla_x \Big[
		U^{\J} \cdot
		\Big( \eta_R^{\J}Q_{\gamma_j}\Phi_{\rm in}^{\J} + \eta_{d_q}^{\J} \Phi^{*{\J}}_0 \Big) \Big] \cdot \nabla_x U^{\J}
		\Big\}
		\\
		= \ &
		( U_*-U^{\J} ) \wedge  \Big\{
		\Delta_x \Big(\eta_{d_q}^{\J}\Phi_0^{*\J} \Big)
		- 2 \nabla_x \Big(
		U^{\J} \cdot \eta_{d_q}^{\J}\Phi_0^{*\J}   \Big) \cdot \nabla_x U^{\J}
		+ Q_{\gamma_j}\Big( \Phi_{\rm in}^{\J} \Delta_x \eta_R^{\J}
		+2 \nabla_x \eta_R^{\J}
		\cdot \nabla_x \Phi_{\rm in}^{\J} \Big)
		\\
		&
		+ \eta_R^{\J}Q_{\gamma_j}  \Delta_x \Phi_{\rm in}^{\J}
		- 2 \nabla_x ( U^{\J} \cdot \Phi_{\rm out}  ) \cdot \nabla_x U^{\J}
		\Big\}.
%	\end{aligned}
%\end{equation*}
\end{align*}
Here, by \eqref{U*-norm}, \eqref{Phi*-0-j-upp}, \eqref{nabU*-est}, it follows that
\begin{align*}
	%\begin{equation*}
	%	\begin{aligned}
		&
		\left|
		\left( U_*-U^{\J}
		\right) \wedge  \left[ \Delta_x \left(\eta_{d_q}^{\J}\Phi_0^{*\J} \right)
		- 2 \nabla_x \left(
		U^{\J} \cdot \eta_{d_q}^{\J}\Phi_0^{*\J}   \right) \cdot \nabla_x U^{\J}    \right] \right|
		\\
		= \ &
		\Big|
		\left( U_*-U^{\J}
		\right) \wedge  \Big[ \Phi_0^{*\J}  \Delta_x \eta_{d_q}^{\J}
		+
		2\nabla_x \eta_{d_q}^{\J}
		\cdot \nabla_x \Phi_0^{*\J}
		+
		\eta_{d_q}^{\J} \Delta_x \Phi_0^{*\J}
		\\
		&
		- 2 \left(
		U^{\J} \cdot \Phi_0^{*\J}  \right) \nabla_x \eta_{d_q}^{\J}
		\cdot \nabla_x U^{\J}
		- 2 \eta_{d_q}^{\J} \left(
		U^{\J} \cdot \nabla_x  \Phi_0^{*\J}   \right)
		\cdot \nabla_x U^{\J}
		- 2 \eta_{d_q}^{\J} \left(
		\Phi_0^{*\J}  \cdot \nabla_x   U^{\J}  \right) \cdot \nabla_x U^{\J}
		\Big]
		\Big|
		\\
		\lesssim \ &
		\lambda_*  \left(1+
		\lambda_j^{-1} \langle \rho_j\rangle^{-1}
		+
		\lambda_*^{-1} \langle \rho_j\rangle^{-2}
		+
		\lambda_j \langle \rho_j\rangle \lambda_*^{-2} \langle \rho_j\rangle^{-4}\right) \1_{\{|x-q^{\J}|\le 3d_q \}}
		\lesssim T^\epsilon \varrho_3.
		%	\end{aligned}
	%\end{equation*}
\end{align*}
\noindent $\bullet$
By \eqref{U*-norm}, similar to the estimate in \eqref{coupling est},
\begin{equation*}
	\left|
	\left( U_*-U^{\J}
	\right) \wedge
	\left[
	Q_{\gamma_j}\left( \Phi_{\rm in}^{\J} \Delta_x \eta_R^{\J}
	+2 \nabla_x \eta_R^{\J}
	\cdot \nabla_x \Phi_{\rm in}^{\J} \right)
	\right] \right|
	\lesssim
	\lambda_*
	\left|
	\Phi_{\rm in}^{\J} \Delta_x \eta_R^{\J}
	+2 \nabla_x \eta_R^{\J}
	\cdot \nabla_x \Phi_{\rm in}^{\J}
	\right|
	\lesssim
	T^{\epsilon} \varrho_1.
\end{equation*}

\noindent $\bullet$
By \eqref{U*-norm},
\begin{align*}
%\begin{equation*}
%	\begin{aligned}
		&
		\big|
		( U_*-U^{\J} ) \wedge   \big( \eta_R^{\J}
		Q_{\gamma_j}  \Delta_x \Phi_{\rm in}^{\J}   \big)
		\big|
		\lesssim
		\1_{\{  |x-\xi^{\J}| \le 2\la_* R \}}
		\lambda_*^{-1}
		\| \Phi_{\rm{in} }^{\J} \|_{{\rm in},\nu-\delta_0,l} \lambda_*^{\nu-\delta_0}
		\langle y^{\J}\rangle^{-l-2}
		\lesssim T^{\epsilon}\varrho_1
%	\end{aligned}
%\end{equation*}
\end{align*}
provided
\begin{equation}\label{o-para-4}
	\Theta+ \beta +\delta_0 -\nu<0.
\end{equation}
\noindent $\bullet$
For a fixed $j$, by \eqref{U*-norm}, \eqref{nabU*-est},
\begin{equation*}
	\begin{aligned}
		&
		\left|
		\left( U_*-U^{\J}
		\right) \wedge
		\left[
		\nabla_x \left( U^{\J} \cdot \Phi_{\rm out}  \right)\cdot \nabla_x U^{\J}
		\right]
		\right|
		\lesssim
		\sum\limits_{k\ne j}  \langle \rho_k \rangle^{-1}
		\big(
		|\nabla_x  \Phi_{\rm {out} } |
		\lambda_{j}^{-1} \langle \rho_j \rangle^{-2}
		+
		|\Phi_{\rm {out} } | \lambda_{j}^{-2} \langle \rho_j \rangle^{-4}
		\big).
	\end{aligned}
\end{equation*}

We claim that
\begin{equation}\label{rho_j*rho_k-est}
	\langle \rho_j \rangle \langle \rho_k \rangle
	\gtrsim
	\lambda_*^{-1} \min\left\{\langle \rho_j \rangle,\langle \rho_k \rangle \right\}
	\mbox{ \ for \ } j\ne k .
\end{equation}
Indeed, for $|x-\xi^{\J}|\le
	|\xi^{\J}-\xi^{\K}|/2$, then $\langle \rho_k \rangle \sim \lambda_*^{-1}$, which implies
	\begin{equation*}
		\langle \rho_j \rangle \langle \rho_k \rangle
		\sim \lambda_*^{-1} \langle \rho_j \rangle \sim \lambda_*^{-1} \min\left\{\langle \rho_j \rangle,\langle \rho_k \rangle \right\} .
	\end{equation*}
	For $|x-\xi^{\K}|\le
	|\xi^{\J}-\xi^{\K}|/2$, similarly, we have $
	\langle \rho_j \rangle \langle \rho_k \rangle
	\sim \lambda_*^{-1} \min\left\{\langle \rho_j \rangle,\langle \rho_k \rangle \right\} $.
	For $|x-\xi^{\J}| >
	|\xi^{\J}-\xi^{\K}|/2$ and  $|x-\xi^{\K}| >
	|\xi^{\J}-\xi^{\K}|/2$, then
	\begin{equation*}
		\langle \rho_j \rangle \langle \rho_k \rangle
		\sim \lambda_*^{-2} |x-\xi^{\J}| |x-\xi^{\K}|
		\gtrsim
		\lambda_*^{-1} \min\left\{\langle \rho_j \rangle,\langle \rho_k \rangle \right\}.
	\end{equation*}

Recall $\| \Phi_{\rm {out} }\|_{\sharp, \Theta,\alpha}$ given in \eqref{out-topo}. For $j\ne k$, by \eqref{rho_j*rho_k-est},
\begin{equation*}
	\langle \rho_k \rangle^{-1}
	|\nabla_x  \Phi_{\rm {out} } |
	\lambda_{j}^{-1} \langle \rho_j \rangle^{-2}
	\lesssim
	\| \Phi_{\rm {out} }\|_{\sharp, \Theta,\alpha} \left(\lambda^{\Theta}_*(0)
	+
	\| Z_* \|_{C^3(\R^2)}  \right)
	\lesssim
	T^{\epsilon} \varrho_3 .
\end{equation*}
Under the restriction $\Theta<\beta <1/2$, by \eqref{out-upp-split},
\begin{align*}
		&
		\langle \rho_k \rangle^{-1}
		|\Phi_{\rm {out} } | \lambda_{j}^{-2} \langle \rho_j \rangle^{-4}
		\lesssim
		\| \Phi_{\rm {out} }\|_{\sharp, \Theta,\alpha}
		\Big[
		\1_{\{ |x-q^{\K}| < 3d_q \}}
		\left(
		|\ln(T-t)|
		\lambda^{\Theta+1}_*  R
		\lambda_*^2 \langle \rho_k \rangle^{-1}
		+
		\lambda_*^3
		\right)
		\\
		& +
		\1_{\{ |x-q^{\J}| < 3d_q \}}
		\left(
		|\ln(T-t)|
		\lambda^{\Theta}_*  R
		\langle \rho_j\rangle^{-4}
		+
		\langle \rho_j\rangle^{-3}
		\right)
		+
		\sum\limits_{m=1, m\ne j,k }^N \1_{\{ |x-q^{\M}| < 3d_q \}}
		\left(
		|\ln(T-t)|
		\lambda^{\Theta+4}_*  R
		+
		\lambda_*^3
		\right)
		\\
		&
		+
		\1_{\{ \cap_{m=1}^N \{|x-q^{\M}| \ge  3d_q \} \}}
		\lambda_*^3
		\Big]
		\lesssim
		T^{\epsilon} \big( \varrho_1^{\J} + \varrho_3 \big).
	\end{align*}

\noindent $\bullet$
By \eqref{U*-norm}, \eqref{nabU*-est}, and \eqref{rho_j*rho_k-est},
\begin{align*}
	%\begin{equation*}
	%\begin{aligned}
	&
	\Big| (a -bU_* \wedge ) \Big\{
	 \nabla_x \big[
	\Phi\cdot (  U_*- U^{\J}  )
	\big] \cdot \nabla_x U^{\J}
	\Big\} \Big|
	\lesssim
	\Big(
	| \nabla_x \Phi | \sum\limits_{k\ne j} \langle \rho_k \rangle^{-1}
	+
	|\Phi |
	\sum\limits_{k\ne j}
	\lambda_{*}^{-1} \langle \rho_k \rangle^{-2}
	\Big) \lambda_{*}^{-1} \langle \rho_j \rangle^{-2}
	\\
	\lesssim \ &
	\left| \nabla_x \Phi \right|
	\langle \rho_j \rangle^{-1} \sum\limits_{k\ne j} \left(\min\left\{\langle \rho_j \rangle,\langle \rho_k \rangle \right\} \right)^{-1}
	+
	\left|	\Phi \right|
	\sum\limits_{k\ne j}
	\left(\min\left\{\langle \rho_j \rangle,\langle \rho_k \rangle \right\} \right)^{-2}.
	%\end{aligned}
	%\end{equation*}
\end{align*}

By \eqref{out-upp}, \eqref{P-P_o-est}, and the assumption $\delta_0<\nu$ in \eqref{inn-top0-para}, we have
\begin{equation*}
	\left|	\Phi \right|
	\left(\min\left\{\langle \rho_j \rangle,\langle \rho_k \rangle \right\} \right)^{-2}
	\lesssim T^{\epsilon} \varrho_3.
\end{equation*}

By \eqref{out-nabla-upp} and \eqref{nabla-P-P_o-est}, we have
\begin{align*}
	%\begin{equation*}
	%	\begin{aligned}
		&
		\left| \nabla_x \Phi \right|
		\langle \rho_j \rangle^{-1} \left(\min\left\{\langle \rho_j \rangle,\langle \rho_k \rangle \right\} \right)^{-1}
		\lesssim
		\| \Phi_{\rm {out} }\|_{\sharp, \Theta,\alpha} \left(\lambda^{\Theta}_*(0)
		+
		\| Z_* \|_{C^3(\R^2)}  \right)
		\\
		&
		+
		\sum_{m=1}^{N}
		\left[
		\1_{\{ |x-q^{\M}| \le 3\lambda_* R \}}
		\left(
		\| \Phi_{\rm{in} }^{\M} \|_{{\rm in},\nu-\delta_0,l} \lambda_*^{\nu-\delta_0-1} \langle \rho_m \rangle^{-l-1}
		+1
		\right)
		+ \1_{\{ 3\lambda_* R  < |x-q^{\M}| < 3d_q\}}
		\right]
		\lesssim
		T^\epsilon \Big(
		\sum\limits_{m=1}^{N} \varrho_1^{\M} + \varrho_3 \Big)
		%	\end{aligned}
	%\end{equation*}
\end{align*}
provided
\begin{equation}\label{o-para-6}
	\Theta +\beta +\delta_0 -\nu<0.
\end{equation}

\noindent $\bullet$ To estimate
$
\left(a -bU_* \wedge \right) \Big\{
- 2 \sum\limits_{j=1}^{N}\nabla_x \Big[
U^{\J} \cdot \sum_{k=1, k\ne j}^{N} \left( \eta_R^{\K}Q_{\gamma_k}\Phi_{\rm in}^{\K} + \eta_{d_q}^{\K} \Phi^{*{\K}}_0 \right) \Big]
\cdot \nabla_x U^{\J}
\Big\} $.
For $k\ne j$,
by \eqref{nabU*-est}, and $\| \cdot \|_{{\rm in},\nu-\delta_0,l }$ given in \eqref{inn-topo},
\begin{align*}
		&
		\Big|
		\nabla_x \left[
		U^{\J} \cdot \left( \eta_R^{\K}Q_{\gamma_k}\Phi_{\rm in}^{\K} \right) \right]\cdot \nabla_x U^{\J} \Big|
		\\
		\lesssim \ &
		\| \Phi_{\rm{in} }^{\K} \|_{{\rm in},\nu-\delta_0,l} \lambda_*^{\nu-\delta_0}
		\Big[
		\1_{\{ |x-\xi^{\K}| \le 2\la_* R \}}
		\left(   \lambda_k^{-1}
		\langle \rho_k \rangle^{-l-1}
		+
		\langle \rho_k \rangle^{-l} \lambda_*^{-1} \langle \rho_j\rangle^{-2}
		\right)
		\\
		&
		+
		\langle \rho_k \rangle^{-l}
		(\lambda_* R)^{-1}
		\1_{\{ \la_* R \le |x-\xi^{\K}| \le 2\la_* R \}}
		\Big] \lambda_*^{-1}
		\langle \rho_j\rangle^{-2}
		\\
		\lesssim \ &
		\1_{\{ |x-q^{\K}| \le 3 \lambda_* R \}}  \lambda_*^{\nu-\delta_0}
		\big(
		\langle \rho_k \rangle^{-l-1}
		+
		\langle \rho_k \rangle^{-l} \lambda_*^2
		\big)
		\lesssim
		T^{\epsilon} \varrho_1^{\K}
	\end{align*}
when
\begin{equation}\label{o-para-7}
	\nu - \delta_0 >\Theta+\beta-1;
\end{equation}
additionally, by \eqref{Phi*-0-j-upp}, \eqref{nabU*-est},
\begin{equation*}
	\begin{aligned}
		&
		\left|
		\nabla_x \left[
		U^{\J} \cdot \left(\eta_{d_q}^{\K} \Phi^{*{\K}}_0 \right) \right] \cdot \nabla_x U^{\J}  \right|
		\lesssim
		\lambda_* \1_{\{ |x-\xi^{\K}|\le 2d_{q}\}}
		\lesssim T^{\epsilon} \varrho_3 .
	\end{aligned}
\end{equation*}

\noindent $\bullet$
To estimate $	\sum_{j=1}^{N} |\nabla_x U^{\J} |^2
\left( a-bU^{\J}  \wedge \right)
\sum_{k=1, k\ne j}^{N} \left( \eta_R^{\K}Q_{\gamma_k}\Phi_{\rm in}^{\K} + \eta_{d_q}^{\K} \Phi^{*{\K}}_0 \right) $.
For $k\ne j$, by \eqref{nabU*-est}, \eqref{Phi*-0-j-upp},
\begin{equation*}
	\begin{aligned}
		\left|
		|\nabla_x U^{\J} |^2 \eta_R^{\K}Q_{\gamma_k}\Phi_{\rm in}^{\K}  \right|
		\lesssim \lambda_*^2 \eta_R^{\K}
		\| \Phi_{\rm{in} }^{\K} \|_{{\rm in},\nu-\delta,l} \lambda_*^{\nu-\delta_0} \langle \rho_k \rangle^{-l}
		\lesssim T^{\epsilon} \varrho_1^{\K} ,
\quad
		|\nabla_x U^{\J} |^2
		| \eta_{d_q}^{\K} \Phi^{*{\K}}_0   |
		\lesssim
		\lambda_*^{2} \lesssim T^\epsilon \varrho_3
	\end{aligned}
\end{equation*}
under the assumption $2+\nu - \delta_0 >\Theta+\beta-1$.

\noindent $\bullet$
To estimate $a \Phi \sum_{j,k=1, j\ne k}^{N} \nabla_x U^{\J}\cdot \nabla_x U^{\K}$. For $j\ne k$,
by \eqref{nabU*-est}, \eqref{rho_j*rho_k-est}, \eqref{out-upp}, and  \eqref{P-P_o-est},
\begin{equation*}
	\left| \Phi  \nabla_x U^{\J}\cdot \nabla_x U^{\K}
	\right|
	\lesssim
	| \Phi | \lambda_*^{-2} \langle \rho_j \rangle^{-2}  \langle \rho_k \rangle^{-2}
	\lesssim
	| \Phi | \left(
	\min\left\{\langle \rho_j \rangle,\langle \rho_k \rangle \right\}
	\right)^{-2}
	\lesssim
	T^\epsilon \varrho_3
\end{equation*}
under the assumption \eqref{o-para-7}.

\noindent $\bullet$
By \eqref{A-est}, \eqref{ppt-U-est}, \eqref{toget-est-1}, and the ansatz \eqref{lam-ansatz},
\begin{align*}
		&
		\left| [(\Phi\cdot U_*)-A]\pp_t U_* \right|
		\lesssim \left(|\Phi\cdot U_*|+ \lambda_* + |\Phi|^2 \right)
		\sum\limits_{j=1}^N
		\left[
		\left(\la_j^{-1}|\dot\la_j| + |\dot\gamma_j |  \right) \langle \rho_j \rangle^{-1}
		+
		\la_j^{-1}|\dot{\xi}^{\J}| \langle \rho_j \rangle^{-2}
		\right]
		\\
		\lesssim \ &
		\bigg\{
		\sum_{j=1}^{N}
		\Big[
		\1_{\{ |x-q^{\J}|\le 3\lambda_* R \}}
		\left(
		1+ \|\Phi_{\rm {out} }\|_{\sharp, \Theta,\alpha}
		+
		\| \Phi_{\rm{in} }^{\J} \|_{{\rm in},\nu-\delta_0,l}
		\right)^2
		\left(
		\lambda_*^{2\nu-2\delta_0} \langle \rho_j\rangle^{-l}
		+
		\lambda_* \langle \rho_j\rangle
		+
		|\ln(T-t)|
		\lambda^{\Theta+1}_*  R
		\right)
		\\
		&
		+
		\1_{\{  3\lambda_* R < |x-q^{\J}| <3d_q \}}
		\left(1+
		\| \Phi_{\rm {out} }\|_{\sharp, \Theta,\alpha} \right)^2
		\left(
		\lambda_* \langle  \rho_j  \rangle
		+
		|\ln(T-t)|
		\lambda^{\Theta+1}_*  R
		\right)
		\Big]
		\\
		&
		+
		\1_{\{ \cap_{j=1}^N \{|x-q^{\J}| \ge  3d_q \} \}}
		\left(1+
		\| \Phi_{\rm {out} }\|_{\sharp, \Theta,\alpha}
		\right)^2 \bigg\}
		\sum\limits_{j=1}^N
		\left(\la_j^{-1}|\dot\la_j| + |\dot\gamma_j | + \la_j^{-1}|\dot{\xi}^{\J}|  \right) \langle \rho_j \rangle^{-1}
		\lesssim
		T^{\epsilon} \bigg(\sum_{j=1}^N \varrho_1^{\J} +\varrho_3 \bigg)
	\end{align*}
provided
\begin{equation}\label{o-para-7.1}
	\Theta+\beta+2\delta_0-2\nu<0, \quad \beta<1/2.
\end{equation}

\noindent $\bullet$  To estimate
\begin{equation*}
	\begin{aligned}
		&
		\sum_{j=1}^{N}
		\eta_R^{\J} \left( U^{\J} -U_{*} \right)		
		\Big[
		-
		2a   \left( 	\nabla_x W^{\J} \cdot    \nabla_x \Phi_{\rm in}^{\J} \right)
		+
		a
		|\nabla_x U^{\J} |^2
		\left(U^{\J} \cdot \Phi_{\rm {out} } \right)
		\\
		& +
		\Big\{
		-\pp_t (\Phi_0^{*\J} )
		+
		\left( a-bU^{\J}  \wedge \right)
		\left[
		\Delta_x \Phi_0^{*\J} +
		|\nabla_x U^{\J} |^2 \Phi^{*{\J}}_0
		- 2 \nabla_x \left(
		U^{\J} \cdot
		\Phi^{*{\J}}_0  \right)\cdot \nabla_x U^{\J}
		\right]
		- \partial_t U^{\J}
		\Big\} \cdot U^{\J}
		\Big]
		.
	\end{aligned}
\end{equation*}
We estimate term by term. First, by \eqref{nablaW}, \eqref{U*-norm},
\begin{equation*}
	\begin{aligned}
		&
		\left|
		\eta_R^{\J} \left( U^{\J} -U_{*} \right)		
		\left( 	\nabla_x W^{\J} \cdot    \nabla_x \Phi_{\rm in}^{\J} \right)  \right|
		\lesssim
		\lambda_j^{-2}
		\1_{\{  |x-\xi^{\J}| \le 2\la_* R \}}
		\lambda_*		
		\langle \rho_j \rangle^{-2}    \| \Phi_{\rm{in} }^{\J} \|_{{\rm in},\nu-\delta_0,l} \lambda_*^{\nu-\delta_0} \langle \rho_j \rangle^{-l-1}
		\lesssim
		T^{\epsilon} \varrho_1^{\J}
	\end{aligned}
\end{equation*}
provided
\begin{equation}\label{o-para-8}
	\Theta+\delta_0+\beta-\nu<0 .
\end{equation}
Next, by \eqref{U*-norm}, \eqref{nabU*-est}, \eqref{out-upp-split},
\begin{equation*}
	\left|	\eta_R^{\J} \left( U^{\J} -U_{*} \right)		
	|\nabla_x U^{\J} |^2
	\left(U^{\J} \cdot \Phi_{\rm {out} } \right) \right|
	\lesssim
	\| \Phi_{\rm {out} }\|_{\sharp, \Theta,\alpha}
	\eta_R^{\J}  \lambda_{*}^{-1}
	\langle \rho_j\rangle^{-4}
	\left(
	|\ln(T-t)|
	\lambda^{\Theta+1}_*  R
	+
	\lambda_j \rho_j
	\right)
	\lesssim T^\epsilon \varrho_1^{\J}
\end{equation*}
provided $\Theta<\beta<1/2$. Finally, by \eqref{U*-norm} and \eqref{Sj-U-dire-est}, we obtain
\begin{align*}
		& \Big|
		\eta_R^{\J} \left( U^{\J} -U_{*} \right)		
		\Big[
		\Big\{
		-\pp_t (\Phi_0^{*\J} )
		+
		\left( a-bU^{\J}  \wedge \right)
		\left[
		\Delta_x \Phi_0^{*\J} +
		|\nabla_x U^{\J} |^2 \Phi^{*{\J}}_0
		- 2 \nabla_x \left(
		U^{\J} \cdot
		\Phi^{*{\J}}_0  \right)\cdot \nabla_x U^{\J}
		\right]
		\\
		&
		- \partial_t U^{\J}
		\Big\} \cdot U^{\J}
		\Big] \Big|
		\lesssim
		\eta_R^{\J} \lambda_* \left(| \dot{\xi}^{\J}| \langle  \rho_j \rangle^{-1}
		+ |\lambda_j|^{-1}  \langle  \rho_j \rangle^{-2} \right)
		\lesssim T^{\epsilon} \varrho_3 .
	\end{align*}

\medskip

\noindent$\bullet$
By \eqref{one in nabla A}, we have
\begin{align*}
		&
		(\Phi\wedge U_*)
		\left[(1+A)|U_*|^2+\left( U_*  \cdot \Pi_{U_*^{\perp}}\Phi \right)\right]^{-1}
		\left[ \Phi
		+ \left( 1+A-\Phi\cdot U_* \right) U_* \right] \cdot \Delta_x \left(\Phi-
		\Phi_{\rm {out}} \right)
		\\
		&
		- 	\left(
		AU_* +  \Pi_{U_*^{\perp}} \Phi \right) \wedge \Delta_x  \left(\Phi-
		\Phi_{\rm {out}} \right)
		\\
		= \ &
		(\Phi\wedge U_*)
		\left(  1+ O\left(\lambda_*+|\Phi|^2\right) \right)
		\left[ \Phi
		+ \left( 1+A-\Phi\cdot U_* \right) U_* \right] \cdot \Delta_x \left(\Phi-
		\Phi_{\rm {out}} \right)
		\\
		&
		- 	\left[
		AU_* +  \Phi - (\Phi\cdot U_*) U_* \right] \wedge \Delta_x  \left(\Phi-
		\Phi_{\rm {out}} \right)
		\\
		= \ &
		(\Phi\wedge U_*)
		\left[ U_* \cdot \Delta_x \left(\Phi-
		\Phi_{\rm {out}} \right)  \right]
		- \Phi \wedge \Delta_x  \left(\Phi-
		\Phi_{\rm {out}} \right)
		\\
		& +
		(\Phi\wedge U_*)
		\left[ \Phi
		+ \left( A-\Phi\cdot U_* \right) U_* \right] \cdot \Delta_x \left(\Phi-
		\Phi_{\rm {out}} \right)
		\\
		& +
		(\Phi\wedge U_*) O\left(\lambda_*+|\Phi|^2\right)
		\left[ \Phi
		+ \left( 1+A-\Phi\cdot U_* \right) U_* \right] \cdot \Delta_x \left(\Phi-
		\Phi_{\rm {out}} \right)
		\\
		&
		- 	\left[
		AU_*  - (\Phi\cdot U_*) U_* \right] \wedge \Delta_x  \left(\Phi-
		\Phi_{\rm {out}} \right)  .
	\end{align*}
For above terms, we estimate by \eqref{A-est} and the ansatz $|\Phi|\ll 1$ in \eqref{RPhi-ansatz},
\begin{equation*}
	\begin{aligned}
		&
		\Big|
		(\Phi\wedge U_*)
		\left[ \Phi
		+ \left( A-\Phi\cdot U_* \right) U_* \right] \cdot \Delta_x \left(\Phi-
		\Phi_{\rm {out}} \right)
		\\
		& +
		(\Phi\wedge U_*) O\left(\lambda_*+|\Phi|^2\right)
		\left[ \Phi
		+ \left( 1+A-\Phi\cdot U_* \right) U_* \right] \cdot \Delta_x \left(\Phi-
		\Phi_{\rm {out}} \right)
		\\
		&
		- 	\left[
		AU_*  - (\Phi\cdot U_*) U_* \right] \wedge \Delta_x  \left(\Phi-
		\Phi_{\rm {out}} \right)
		\Big|
		\lesssim
		\left(
		\lambda_* + |\Phi|^2 +  \left|\Phi\cdot U_*\right| \right)
		\left| \Delta_x  \left(\Phi-
		\Phi_{\rm {out}} \right) \right| .
	\end{aligned}
\end{equation*}

Using \eqref{toget-est-1} and \eqref{Delta-P-P_o-est}, we have
\begin{align*}
	%\begin{equation*}
	%	\begin{aligned}
		&
		(
		\lambda_* + |\Phi|^2 +  |\Phi\cdot U_* | )
		| \Delta_x  (\Phi-
		\Phi_{\rm {out}} ) |
		\\
		\lesssim \ &
		\sum_{j=1}^{N}
		\Big[
		\1_{\{ |x-q^{\J}|\le 3\lambda_* R \}}
		\Big(
		\lambda_*^{3\nu-3\delta_0-2} +
		\lambda_*^{\nu-\delta_0-1}
		+
		|\ln(T-t)|
		\lambda_*^{\Theta+\nu-\delta_0-1} R
		+
		\lambda_*^{2\nu-2\delta_0-1}
		+1
		+
		|\ln(T-t)|
		\lambda_*^{\Theta} R
		\Big)
		\\
		& 	
		+
		\1_{\{  3\lambda_* R < |x-q^{\J}| <3d_q \}}
		\Big]  \lesssim T^\epsilon \Big(
		\sum_{j=1}^{N}
		\varrho_1^{\J}
		+
		\varrho_3\Big)
		%	\end{aligned}
	%\end{equation*}
\end{align*}
provided
\begin{equation}\label{qd24May12-5}
	\Theta+\beta+1+3\delta_0 -3\nu<0,
	\quad
	\Theta+\beta+\delta_0-\nu<0,
	\quad
	2\beta+\delta_0-\nu<0,
	\quad
	\Theta+\beta+2\delta_0-2\nu<0,
	\quad
	\beta<1/2.
\end{equation}

We need more refined estimates for the other part. Recalling \eqref{u-def}, we have
\begin{align*}
		&
		\left(\Phi\wedge U_* \right)
		\left[ U_* \cdot \Delta_x \left(\Phi-
		\Phi_{\rm {out}} \right)  \right]
		- \Phi \wedge \Delta_x  \left(\Phi-
		\Phi_{\rm {out}} \right)
		=
		- \Phi \wedge \left\{\Delta_x  \left(\Phi-
		\Phi_{\rm {out}} \right)
		-
		\left[ U_* \cdot \Delta_x \left(\Phi-
		\Phi_{\rm {out}} \right)  \right] U_*
		\right\}
		\\
		= \ &
		-
		\Big(
		\sum_{j=1}^{N}
		\eta_{d_q}^{\J}  \Phi^{*{\J}}_0  +\Phi_{\rm out}  \Big)
		\wedge \left\{\Delta_x  \left(\Phi-
		\Phi_{\rm {out}} \right)
		-
		\left[ U_* \cdot \Delta_x \left(\Phi-
		\Phi_{\rm {out}} \right)  \right] U_*
		\right\}
		\\
		&
		-  \sum_{j=1}^{N}  \eta_R^{\J}  (Q_{\gamma_j}\Phi_{\rm in}^{\J} )
		\wedge \left\{
		\left[ U^{\J} \cdot \Delta_x \left(\Phi-
		\Phi_{\rm {out}} \right)  \right] U^{\J}
		-
		\left[ U_* \cdot \Delta_x \left(\Phi-
		\Phi_{\rm {out}} \right)  \right] U_*
		\right\}
		\\
		&
		- \sum_{j=1}^{N}  \eta_R^{\J}  (Q_{\gamma_j}\Phi_{\rm in}^{\J} )
		\wedge \left\{\Delta_x  \left(\Phi-
		\Phi_{\rm {out}} \right)
		-
		\left[ U^{\J} \cdot \Delta_x \left(\Phi-
		\Phi_{\rm {out}} \right)  \right] U^{\J}
		\right\}.
	\end{align*}

By \eqref{Phi*-0-j-upp} and \eqref{out-upp-split}, one has
\begin{equation*}%\label{nonloc+out-est}
	\Big| \sum_{j=1}^{N}
	\eta_{d_q}^{\J}  \Phi^{*{\J}}_0  +\Phi_{\rm out}  \Big|
	\lesssim
	\sum\limits_{j=1}^N \1_{\{ |x-q^{\J}| < 3d_q \}}
	\left(
	\lambda_j\langle \rho_j\rangle
	+
	|\ln(T-t)|
	\lambda^{\Theta+1}_*  R
	\right)
	+
	\1_{\{ \cap_{j=1}^N \{|x-q^{\J}| \ge  3d_q \} \}}.
\end{equation*}
Combining \eqref{Delta-P-P_o-est}, we get
\begin{align*}
	%\begin{equation*}
	%	\begin{aligned}
		&
		\Big| \Big(
		\sum_{j=1}^{N}
		\eta_{d_q}^{\J}  \Phi^{*{\J}}_0  +\Phi_{\rm out}  \Big)
		\wedge \left\{\Delta_x  \left(\Phi-
		\Phi_{\rm {out}} \right)
		-
		\left[ U_* \cdot \Delta_x \left(\Phi-
		\Phi_{\rm {out}} \right)  \right] U_*
		\right\} \Big|
		\\
		\lesssim \ &
		\sum_{j=1}^{N}
		\Big[
		\1_{\{ |x-q^{\J}| \le 3 \lambda_* R \}}
		\left(
		\lambda_*^{\nu-\delta_0-1} \langle \rho_j\rangle^{-l-1}
		+
		1+
		|\ln(T-t)| \lambda_*^{\nu-\delta_0+\Theta-1} R \langle \rho_j\rangle^{-l-2}
		+
		|\ln(T-t)| \lambda_*^{\Theta} R \langle \rho_j\rangle^{-1}
		\right)
		\\
		&
		+
		\1_{\{ 3 \lambda_* R  < |x-q^{\J}| < 3d_q \}}  \Big]
		\lesssim
		T^{\epsilon}\Big(\sum_{j=1}^N \varrho_1^{\J}+\varrho_3
		\Big)
		%	\end{aligned}
	%\end{equation*}
\end{align*}
provided
\begin{equation}\label{2deri-para-1}
	\Theta +\beta+\delta_0 -\nu<0,\quad
	2\beta+\delta_0-\nu<0,
	\quad \beta<1/2.
\end{equation}

By \eqref{U*-norm}, \eqref{inn-topo}, and \eqref{Delta-P-P_o-est}, we have
\begin{align}
%\begin{equation}
%	\begin{aligned}
		&
		\big|
		\eta_R^{\J}  (Q_{\gamma_j}\Phi_{\rm in}^{\J} )
		\wedge \big\{
		\big[ U^{\J} \cdot \Delta_x (\Phi-
		\Phi_{\rm {out}} )  \big] U^{\J}
		-
		\big[ U_* \cdot \Delta_x (\Phi-
		\Phi_{\rm {out}} )  \big] U_*
		\big\}
		\big|
		\nonumber
		\\
		= \ &
		\big|
		\eta_R^{\J}  (Q_{\gamma_j}\Phi_{\rm in}^{\J} )
		\wedge \big\{
		\big[ (U^{\J} - U_* ) \cdot \Delta_x (\Phi-
		\Phi_{\rm {out}} )  \big] U^{\J}
		+
		\big[ U_* \cdot \Delta_x (\Phi-
		\Phi_{\rm {out}} )  \big] (U^{\J}-U_* )
		\big\}
		\big|
		\nonumber
		\\
		\lesssim \ &
		\eta_R^{\J} \lambda_* | \Phi_{\rm in}^{\J} |
		| \Delta_x  (\Phi-
		\Phi_{\rm {out}} )  |
		\lesssim
		\1_{\{ |x-q^{\J}| \le 3\lambda_* R \}}
		\big(
		\lambda_*^{2\nu-2\delta_0-1} \langle \rho_j \rangle^{-2l-2}
		+
		\lambda_*^{\nu-\delta_0} \langle \rho_j \rangle^{-l-1}
		\big)
		\lesssim
		T^{\epsilon}  \varrho_1^{\J}
		\label{qd24May4-1}
%	\end{aligned}
%\end{equation}
\end{align}
provided
\begin{equation}\label{2deri-para-2}
	\Theta+\beta+2\delta_0-2\nu<0 , \quad
	\nu-\delta_0 >\Theta+\beta-1.
\end{equation}

Since $ \Phi_{\rm{in} }^{\J} \cdot W^{\J}=0$, we get
\begin{equation*}
	\eta_R^{\J}  (Q_{\gamma_j}\Phi_{\rm in}^{\J})
	\wedge \left\{\Delta_x  \left(\Phi-
	\Phi_{\rm {out}} \right)
	-
	\left[ U^{\J} \cdot \Delta_x \left(\Phi-
	\Phi_{\rm {out}} \right)  \right] U^{\J}
	\right\} = \eta_R^{\J} f_j(x,t) U^{\J}
\end{equation*}
with a scalar function
$ f_j(x,t) = \pm
\big| Q_{\gamma_j}\Phi_{\rm in}^{\J}
\wedge \big\{ \Delta_x  (\Phi-
\Phi_{\rm {out}} )
-
\big[ U^{\J} \cdot \Delta_x (\Phi-
\Phi_{\rm {out}} )  \big] U^{\J}
\big\} \big| $.
By $U_*$-operation, it suffices to estimate $ \eta_R^{\J} f_j(x,t) (U^{\J} -U_*)$. By \eqref{U*-norm}, and same estimate as \eqref{qd24May4-1}, then
\begin{equation*}
	\big| \eta_R^{\J} f_j(x,t) (U^{\J} -U_* ) \big|
	\lesssim
	\eta_R^{\J} \lambda_* | \Phi_{\rm in}^{\J} |
	\left| \Delta_x  (\Phi-
	\Phi_{\rm {out}} )  \right|
	\lesssim
	T^{\epsilon}  \varrho_1^{\J}.
\end{equation*}

$\bullet$ The remaining terms will not be strictly handled in order.
Under the parameter assumptions \eqref{nab-A-para}, by \eqref{nab-A-est} and \eqref{nabU*-est}, we have
\begin{equation}\label{nabA-toget-1}
	\begin{aligned}
		&
		\left| \left(\nabla_x A + U_* \cdot \nabla_x U_* + \Phi\cdot \nabla_x \Phi \right) \cdot \nabla_x U_*  \right|
		\\
		\lesssim \ &
		\bigg\{
		\sum_{j=1}^{N}
		\left[	\1_{\{ |x-q^{\J}|\le 3\lambda_* R \}}
		\left(
		\lambda_*^{\epsilon +1} \lambda_*^{\Theta}
		(\lambda_* R)^{-1}
		+
		\lambda_* \langle \rho_j\rangle \right)
		+
		\1_{\{  3\lambda_* R < |x-q^{\J}| <3d_q \}}
		\left(
		|\ln(T-t)|
		\lambda^{\Theta+1}_*  R
		+
		\lambda_* \langle \rho_j \rangle
		\right)  \right]
		\\
		&
		+
		\1_{\{ \cap_{j=1}^N \{|x-q^{\J}| \ge  3d_q \} \}}
		\bigg\}
		\bigg(
		\sum\limits_{j=1}^N \1_{\{ |x-q^{\J}| < 3d_q\}}  \lambda_*^{-1}\langle \rho_j \rangle^{-2}
		+   \1_{\{ \cap_{j=1}^N \{|x-q^{\J}| \ge 3 d_q  \} \}}  \lambda_*
		\bigg)
		\lesssim
		T^\epsilon \bigg(\sum\limits_{j=1}^N \varrho_1^{\J} + \varrho_3 \bigg) .
	\end{aligned}
\end{equation}

\noindent$\bullet$
To estimate $ \left(A -  \Phi\cdot U_* \right) \Delta_x U_*
= - \left(A -  \Phi\cdot U_* \right)
\sum\limits_{j=1}^N |\nabla_x U^{\J}|^2  U^{\J}  $.
By $U_*$-operation and \eqref{U*-norm}, \eqref{A-est}, \eqref{nabU*-est}, it suffices to estimate
\begin{equation*}
	\Big|
	\left(A -  \Phi\cdot U_* \right)
	\sum\limits_{j=1}^N |\nabla_x U^{\J}|^2
	\left( U^{\J} - U_* \right)
	\Big|
	\lesssim \left(
	\lambda_* + |\Phi|^2 + |\Phi\cdot U_*|
	\right)
	\sum\limits_{j=1}^N  \lambda_*^{-2}
	\langle \rho_j\rangle^{-4}
	\sum\limits_{k=1, k\ne j}^N \langle \rho_k\rangle^{-1},
\end{equation*}
which will be dealt with uniformly in \eqref{toget-2} later.

\noindent$\bullet$
By \eqref{one in nabla A}, one has
\begin{align*}
		&
		\big|
		(\Phi\wedge U_*)
		\big[(1+A)|U_*|^2+\big( U_*  \cdot \Pi_{U_*^{\perp}}\Phi \big)\big]^{-1}
		( 1+A-\Phi\cdot U_* )
		( 2\nabla_x \Phi\cdot\nabla_x U_*  )
		-
		(\Phi\wedge U_*) ( 2\nabla_x \Phi\cdot\nabla_x U_*  )
		\big|
		\\
		= \ &
		\big|
		(\Phi\wedge U_*)
		\left(
		1+ O\left(\lambda_*+|\Phi|^2+ |\Phi\cdot U_*| \right)  \right)^2
		\left( 2\nabla_x \Phi\cdot\nabla_x U_*  \right)
		-   (\Phi\wedge U_*) \left( 2\nabla_x \Phi\cdot\nabla_x U_*  \right)
		\big|
		\\
		\lesssim \ &
		|\Phi |
		\left( \lambda_*+|\Phi|^2+ |\Phi\cdot U_*| \right)
		\left|\nabla_x \Phi\cdot\nabla_x U_*  \right|
		\lesssim
		\left( \lambda_*+|\Phi|^2+ |\Phi\cdot U_*| \right)
		\big(
		|\Phi |
		\left|\nabla_x U_*  \right|^2
		+
		|\Phi |
		\left|\nabla_x \Phi \right|^2  \big),
	\end{align*}
which will be controlled by \eqref{toget-2} and \eqref{|nabPhi|^2*Phi} later.

\noindent$\bullet$ By \eqref{one in nabla A}, $|U_*|=1+O(\lambda_*)$ by \eqref{U*-norm}, \eqref{A-est},
\begin{align*}
%	\begin{equation*}
%		\begin{aligned}
		&
		\Big|
		\Big\{
		|\nabla_x A|^2 |U_*|^2
		+
		2(1+A)\nabla_x A \cdot \left( U_* \cdot \nabla_x U_*\right)
		+
		A(2+A) \left|\nabla_x U_*\right|^2
		\\
		&
		+
		2 \sum\limits_{k=1}^{2}
		\Big\{
		\left[  \left(\pp_{x_k} A  \right) U_{*}\cdot
		\pp_{x_k} \Phi   + A \pp_{x_k} U_{*} \cdot
		\pp_{x_k} \Phi  \right]
		-
		\pp_{x_k}\left( U_{*} \cdot  \Phi \right) \left[ |U_{*}|^2 \pp_{x_k} A + (1+A)
		U_{*} \cdot  \pp_{x_k} U_{*} \right]
		\\
		&
		- \left(U_* \cdot \Phi \right) \Big[ \left(\pp_{x_k} A \right) U_{*} \cdot   \pp_{x_k} U_*   + (1+A) \left|\pp_{x_k} U_{*}\right|^2 \Big]
		\Big\}
		\\
		&
		+
		\sum\limits_{k=1}^2
		\left| \pp_{x_k} \Phi
		-
		U_* \pp_{x_k}\left(
		\Phi\cdot  U_* \right)
		-
		(\Phi\cdot U_*) \pp_{x_k} U_*
		\right|^2
		\Big\}
		\Pi_{U_*^{\perp}} \Phi
		\Big|
		\\
		&
		+
		\Big|
		-b\Big\{
		-2^{-1} (\Phi\wedge U_*)
		\Big[(1+A)|U_*|^2+\Big( U_*  \cdot \Pi_{U_*^{\perp}}\Phi \Big)\Big]^{-1}
		\Big\{  2 ( 1+A-\Phi\cdot U_* )
		(  \Phi\cdot \Delta_x U_*  )
		\\
		& +2(|U_*|^2-2)| \nabla_x \left(\Phi\cdot  U_* \right)  |^2+2|\nabla_x \Phi|^2
		+8[(\Phi\cdot U_*)-(1+A)] (U_* \cdot \nabla_x U_*)\cdot  \nabla_x\left(\Phi\cdot  U_* \right)
		\\
		& +2|U_*|^2|\nabla_x A|^2+4
		\left[
		-2(\Phi\cdot U_*) U_*\cdot \nabla_x U_*
		+
		(1-|U_*|^2)
		\nabla_x
		\left(
		\Phi\cdot  U_*
		\right)
		\right]
		\cdot \nabla_x A
		\\
		&
		+8(1+A)\left(U_* \cdot  \nabla_x  U_*\right)\cdot \nabla_x A
		+ 2\left[(\Phi\cdot U_*)-(1+A)\right]^2  \left(|\nabla_x U_*|^2 + U_* \cdot \Delta_x U_*\right) \Big\}
		\\
		& -(\Pi_{U_*^{\perp}} \Phi +AU_*)\wedge \left[2 \nabla_x \left(\Phi\cdot  U_* \right)\cdot \nabla_x U_* \right]+[A-(\Phi\cdot U_*)]\Phi\wedge \Delta_x U_*\\
		& +\Pi_{U_*^{\perp}} \Phi \wedge (2\nabla_x A \cdot \nabla_x U_*)+
		\left[(\Phi\cdot U_*)^2-2A(\Phi\cdot U_*) -
		2(\Phi\cdot U_*)\right] U_*\wedge \Delta_x U_*
		\\
		& +
		(1+A)U_* \wedge  \left[ A \Delta_x U_* + 2\left( \nabla_x A + U_* \cdot \nabla_x U_* + \Phi\cdot \nabla_x \Phi \right) \cdot \nabla_x U_* + \Delta_x U_* -2\left( U_* \cdot \nabla_x U_*  \right)\cdot \nabla_x U_* \right] \Big\}
		\\
		& +2b
		A U_* \wedge
		\left[
		\left(   \Phi\cdot \nabla_x \Phi \right) \cdot \nabla_x U_* \right]
		\Big|
		\\
		\lesssim \ &
		|\nabla_x A|^2  |\Phi|
		+
		\big(\lambda_* + |\Phi|^2 + |U_* \cdot \Phi |  \big) |\Phi|
		\big( |\nabla_x U_* |^2 + |\Delta_x U_*|  \big)
		\\
		&
		+
		\left|U_{*} \cdot  \nabla_x U_{*} \right|^2  |\Phi|
		+
		| \nabla_x \Phi |^2  |\Phi|
		+
		\Big|  \Phi\cdot \sum\limits_{j=1}^N |\nabla_x U^{\J}|^2 U^{\J}  \Big| |\Phi|
		+   \left| |\nabla_x U_*|^2 + U_* \cdot \Delta_x U_*\right| |\Phi|
		\\
		&
		+ ( \lambda_* + |\Phi| )   \left| \nabla_x \left(\Phi\cdot  U_* \right) \right| \left| \nabla_x U_* \right|
		+  |\Phi|\left| \nabla_x A \cdot \nabla_x U_* \right| +
		\left(\lambda_* + |\Phi|^2 + \left|U_* \cdot \Phi \right|  \right) \left|U_*\wedge \Delta_x U_*  \right|
		\\
		&
		+
		\left| U_* \wedge  \left[   \Delta_x U_* -2\left( U_* \cdot \nabla_x U_*  \right)\cdot \nabla_x U_* \right] \right|
		+
		\left| \left( \nabla_x A + U_* \cdot \nabla_x U_* + \Phi\cdot \nabla_x \Phi \right) \cdot \nabla_x U_*\right|
		\\
		\lesssim \ &
		(\lambda_* + |\Phi|^2 + |U_* \cdot \Phi | )
		\big[
		|\Phi|
		\big( |\nabla_x U_* |^2 + |\Delta_x U_*|  \big)
		+
		|U_*\wedge \Delta_x U_*  |
		\big]
		\\
		&
		+
		\left|U_{*} \cdot  \nabla_x U_{*} \right|^2  |\Phi|
		+
		| \nabla_x \Phi |^2  |\Phi|
		+
		\Big|  \Phi\cdot \sum\limits_{j=1}^N |\nabla_x U^{\J}|^2 U^{\J}  \Big| |\Phi|
		+   \left| |\nabla_x U_*|^2 + U_* \cdot \Delta_x U_*\right| |\Phi|
		\\
		&
		+ ( \lambda_* + |\Phi| )   \left| \nabla_x \left(\Phi\cdot  U_* \right) \right| \left| \nabla_x U_* \right|
		+
		|\Phi|\left| \left(
		U_* \cdot \nabla_x U_*   \right)\cdot \nabla_x U_* \right|
		+
		|\Phi|\left| \left(   \Phi\cdot \nabla_x \Phi  \right)\cdot \nabla_x U_* \right|
		\\
		&
		+
		\left| U_* \wedge  \left[   \Delta_x U_* -2\left( U_* \cdot \nabla_x U_*  \right) \cdot \nabla_x U_* \right] \right|
		+
		T^{\epsilon}
		\Big(\sum\limits_{j=1}^N \varrho_1^{\J} + \varrho_3 \Big),
%			\end{aligned}
%	\end{equation*}
\end{align*}
where for the last ``$\lesssim$'', we require the assumption  \eqref{nab-A-para} and then by \eqref{nab-A-est},
\begin{equation*}
	|\nabla_x A|^2  |\Phi|
	\lesssim
	\left| U_* \cdot \nabla_x U_* \right|^2 |\Phi|
	+
	|\Phi|^3 |\nabla_x \Phi |^2
	+
	T^\epsilon \varrho_3,
\end{equation*}
\begin{equation*}
	\begin{aligned}
		|\Phi|\left| \nabla_x A \cdot \nabla_x U_* \right|
		\le \ &
		|\Phi|\left| \left(\nabla_x A
		+
		U_* \cdot \nabla_x U_* + \Phi\cdot \nabla_x \Phi  \right) \cdot \nabla_x U_* \right|
		\\
		&
		+
		|\Phi|\left| \left(
		U_* \cdot \nabla_x U_*   \right) \cdot \nabla_x U_* \right|
		+
		|\Phi|\left| \left(   \Phi\cdot \nabla_x \Phi  \right) \cdot \nabla_x U_* \right|,
	\end{aligned}
\end{equation*}
and $\left| \left( \nabla_x A + U_* \cdot \nabla_x U_* + \Phi\cdot \nabla_x \Phi \right) \cdot \nabla_x U_*\right|  $ has been controlled by \eqref{nabA-toget-1}.

\noindent$\bullet$
Combining \eqref{nabU*-est}, \eqref{Del-U*-est}, \eqref{U*wed-DeltaU*}, \eqref{toget-est-1}, and \eqref{Phi-upp}, we then obtain
\begin{align}
	%\begin{equation}
	%	\begin{aligned}
		&
		(\lambda_* + |\Phi|^2 + |U_* \cdot \Phi |  ) \big[ |\Phi|
		\big( |\nabla_x U_* |^2 + |\Delta_x U_*|  \big)
		+
		|U_*\wedge \Delta_x U_* | \big]
		\lesssim
		\left(\lambda_* + |\Phi|^2 + \left|U_* \cdot \Phi \right|  \right)
		\nonumber
		\\
		&\times
		\bigg[ \sum\limits_{j=1}^N \1_{\{ |x-q^{\J}| < 3d_q\}}
		\left(
		|\Phi| \lambda_*^{-2}\langle \rho_j \rangle^{-4}
		+
		\lambda_*^{-1} \langle \rho_j\rangle^{-4}
		+
		\lambda_*^{2}  \langle \rho_j \rangle^{-1}
		\right)
		+   \1_{\{ \cap_{j=1}^N \{|x-q^{\J}| \ge 3 d_q  \} \}}  \left(|\Phi| \lambda_*^2  + \lambda_*^3 \right)
		\bigg]
		\nonumber
		\\
		\lesssim \ &
		\sum_{j=1}^{N}
		\Big[
		\1_{\{ |x-q^{\J}|\le 3\lambda_* R \}}
		\left(
		\lambda_*^{3\nu-3\delta_0-2}
		+
		|\ln(T-t)|
		\lambda^{\nu-\delta_0+\Theta-1}_*  R
		+
		|\ln(T-t)|^2
		\lambda^{2\Theta}_*  R^2
		\right)
		\nonumber
		\\
		&
		+
		\1_{\{  3\lambda_* R < |x-q^{\J}| <3d_q \}}
		\langle  \rho_j  \rangle^{-2}
		\Big]
		+
		\1_{\{ \cap_{j=1}^N \{|x-q^{\J}| \ge  3d_q \} \}}
		\lambda_*^2
		\lesssim
		T^\epsilon \Big(\sum\limits_{j=1}^N \varrho_1^{\J} + \varrho_3 \Big)
		\label{toget-2}
		%	\end{aligned}
	%\end{equation}
\end{align}
provided $\delta_0<\nu$ in \eqref{inn-top0-para},
\begin{equation}\label{NN-par-1}
	\Theta<\beta,
	\quad
	\Theta+\beta+1+3\delta_0-3\nu<0,\quad
	2\beta+\delta_0-\nu<0,\quad 3\beta<1+\Theta.
\end{equation}

\noindent$\bullet$
By \eqref{U*cdot-nabU*-split}, \eqref{Phi-upp}, and $\delta_0<\nu$ in \eqref{inn-top0-para}, we get $ \left| U_* \cdot \nabla_x U_* \right|^2 |\Phi| \lesssim T^{\epsilon} \varrho_3 $.

\noindent$\bullet$ By \eqref{Phi-upp}, \eqref{nabla-Phi-upp},
\begin{align}
	%\begin{equation}
	%	\begin{aligned}
		&
		| \nabla_x \Phi |^2  |\Phi|
		\lesssim
		\sum_{j=1}^{N}
		\Big[
		\1_{\{ |x-q^{\J}|\le 3\lambda_* R \}}
		\left(
		\lambda_* \langle \rho_j\rangle
		+
		\lambda_*^{3\nu-3\delta_0-2}
		+
		|\ln(T-t)|
		\lambda_*^{2\nu-2\delta_0+\Theta-1} R
		\right)
		\nonumber
		\\
		&
		+
		\1_{\{  3\lambda_* R < |x-q^{\J}| <3d_q \}}
		\left(
		\lambda_* \langle  \rho_j  \rangle
		+
		|\ln(T-t)|
		\lambda^{\Theta+1}_*  R
		\right)
		\Big]
		+
		\1_{\{ \cap_{j=1}^N \{|x-q^{\J}| \ge  3d_q \} \}}
		\lesssim
		T^\epsilon \Big(\sum\limits_{j=1}^N \varrho_1^{\J} + \varrho_3 \Big)
		\label{|nabPhi|^2*Phi}
		%	\end{aligned}
	%\end{equation}
\end{align}
provided \eqref{inn-top0-para},
\begin{equation}\label{NN-par-2}
	\Theta<\beta,
	\quad
	\Theta + 2\beta<2,
	\quad
	\Theta+\beta+1+3\delta_0-3\nu<0,\quad \beta+\delta_0-\nu <0 .
\end{equation}

\noindent$\bullet$
By \eqref{DeltaU*-type}, \eqref{nabU*-est}, we have
\begin{align*}
		&
		\bigg|  \Phi\cdot \sum\limits_{j=1}^N |\nabla_x U^{\J}|^2 U^{\J}  \bigg| |\Phi|
		\le
		\bigg|   \sum\limits_{j=1}^N |\nabla_x U^{\J}|^2 \Phi\cdot \left(U^{\J} - U_* \right)  \bigg| |\Phi|
		+
		\bigg|   \sum\limits_{j=1}^N |\nabla_x U^{\J}|^2 \Phi\cdot U_*  \bigg| |\Phi|
		\\
		\lesssim \ &
		|\Phi|^2
		\bigg[
		\sum\limits_{j=1}^{N}
		\1_{\{ |x-q^{\J}| <3d_q \}}
		\left(
		\lambda_*^{-1} \langle \rho_j\rangle^{-4}
		+
		\lambda_*^{2}  \langle \rho_j \rangle^{-1}
		\right)
		+
		\1_{\{ \cap_{j=1}^N \{|x-q^{\J}| \ge 3d_q \} \}}
		\lambda_*^3
		\bigg]
		\\
		&
		+
		\left| \Phi\cdot U_*  \right| |\Phi|
		\bigg[
		\sum\limits_{j=1}^N \1_{\{ |x-q^{\J}| < 3d_q\}}  \lambda_*^{-2}\langle \rho_j \rangle^{-4}
		+   \1_{\{ \cap_{j=1}^N \{|x-q^{\J}| \ge 3 d_q  \} \}}  \lambda_*^2
		\bigg] \lesssim
		T^\epsilon \bigg(\sum\limits_{j=1}^N \varrho_1^{\J} + \varrho_3 \bigg),
	\end{align*}
where the last step is derived by the same way as \eqref{toget-2} under the parameter assumption \eqref{NN-par-1}.

\noindent$\bullet$
By \eqref{Delta|U*|^2} and \eqref{Phi-upp}, we get
\begin{align}
	%\begin{equation}
	%	\begin{aligned}
		&
		\big| |\nabla_x U_*|^2 + U_* \cdot \Delta_x U_*\big| |\Phi|
		\lesssim
		\sum_{j=1}^{N}
		\Big[
		\1_{\{ |x-q^{\J}|\le 3\lambda_* R \}}
		\left(
		\lambda_*^{\nu-\delta_0-1}
		+
		|\ln(T-t)|
		\lambda^{\Theta}_*  R
		\right)
		\nonumber
		\\
		&
		+
		\1_{\{  3\lambda_* R < |x-q^{\J}| <3d_q \}}
		\left(
		\lambda_* \langle  \rho_j  \rangle^{-1}
		+
		|\ln(T-t)|
		\lambda^{\Theta+1}_*  R
		\langle  \rho_j  \rangle^{-2}
		+
		\langle \rho_j\rangle^{-3}
		+
		|\ln(T-t)|
		\lambda^{\Theta}_*  R  \langle \rho_j\rangle^{-4}
		\right)
		\Big]
		\nonumber
		\\
		&
		+
		\1_{\{ \cap_{j=1}^N \{|x-q^{\J}| \ge  3d_q \} \}} \lambda_*^2
		\lesssim
		T^\epsilon \Big(\sum\limits_{j=1}^N \varrho_1^{\J} + \varrho_3 \Big)
		\label{toget-3}
		%	\end{aligned}
	%\end{equation}
\end{align}
provided
\begin{equation}\label{NN-par-3}
	\Theta<\beta<1/2,
	\quad	\Theta+\beta+\delta_0-\nu<0 .
\end{equation}

\noindent$\bullet$
Combining \eqref{Phi-upp}, \eqref{nab-Phi-cdot-U*}, and \eqref{nabU*-est}, one has
\begin{align*}
	%\begin{equation*}
	%	\begin{aligned}
		&
		( \lambda_* + |\Phi| )   \left| \nabla_x \left(\Phi\cdot  U_* \right) \right| \left| \nabla_x U_* \right|
		\lesssim
		\sum_{j=1}^{N}
		\Big[
		\1_{\{ |x-q^{\J}|\le 3\lambda_* R \}}
		\left(
		|\ln(T-t)|
		\lambda_*^{\nu-\delta_0+\Theta-1} R
		+
		|\ln(T-t)|^2
		\lambda^{2\Theta}_*  R^2
		\right)
		\\
		&
		+
		\1_{\{  3\lambda_* R < |x-q^{\J}| <3d_q \}}
		\left(
		\langle \rho_j \rangle^{-1}
		+
		|\ln(T-t)|
		\lambda^{\Theta}_*  R
		\langle \rho_j \rangle^{-2}
		\right)
		\Big]
		+
		\1_{\{ \cap_{j=1}^N \{|x-q^{\J}| \ge  3d_q \} \}}
		\lambda_*
		\lesssim
		T^\epsilon \bigg(\sum\limits_{j=1}^N \varrho_1^{\J} + \varrho_3 \bigg)
		%	\end{aligned}
	%\end{equation*}
\end{align*}
provided
\begin{equation}\label{NN-par-4}
	\Theta<\beta,
	\quad
	2\beta+\delta_0-\nu<0,\quad  3\beta<1+\Theta.
\end{equation}

\noindent$\bullet$
By \eqref{nabU*-est} and \eqref{U*cdot-nabU*-split}, we have
\begin{equation*}
	|\Phi|\left| \left(
	U_* \cdot \nabla_x U_*   \right) \cdot \nabla_x U_* \right|
	\lesssim
	|\Phi|
	\bigg(
	\sum\limits_{j=1}^N \1_{\{ |x-q^{\J}| < 3d_q\}}  \lambda_*^{-1}\langle \rho_j \rangle^{-4}
	+   \1_{\{ \cap_{j=1}^N \{|x-q^{\J}| \ge 3 d_q  \} \}}  \lambda_*^3
	\bigg)
	\lesssim
	T^\epsilon \bigg(\sum\limits_{j=1}^N \varrho_1^{\J} + \varrho_3 \bigg),
\end{equation*}
which is obtained by the same calculation as in \eqref{toget-3} under the parameter assumption \eqref{NN-par-3}.

\noindent$\bullet$
By \eqref{Phi-upp}, \eqref{nabla-Phi-upp}, and \eqref{nabU*-est}, we get
\begin{align*}
	%\begin{equation*}
	%	\begin{aligned}
		&
		|\Phi|\left| \left(   \Phi\cdot \nabla_x \Phi  \right) \cdot \nabla_x U_* \right|
		\lesssim
		\sum_{j=1}^{N}
		\Big[
		\1_{\{ |x-q^{\J}|\le 3\lambda_* R \}}
		\Big(
		\lambda_*^{3\nu-3\delta_0-2}
		+
		|\ln(T-t)|^2
		\lambda^{2\Theta+\nu-\delta_0}_*  R^2
		\Big)
		\\
		& \quad
		+
		\1_{\{  3\lambda_* R < |x-q^{\J}| <3d_q \}}
		\lambda_*
		\Big]
		+
		\1_{\{ \cap_{j=1}^N \{|x-q^{\J}| \ge  3d_q \} \}}
		\lambda_*
		\lesssim
		T^\epsilon \bigg(\sum\limits_{j=1}^N \varrho_1^{\J} + \varrho_3 \bigg)
		%	\end{aligned}
	%\end{equation*}
\end{align*}
provided $\delta_0<\nu <1$ in \eqref{inn-top0-para},
\begin{equation}\label{NN-par-5}
	\Theta+\beta+1+3\delta_0-3\nu<0,\quad
	3\beta<\Theta+1+\nu-\delta_0 .
\end{equation}

\noindent$\bullet$
$	\left| U_* \wedge  \left[   \Delta_x U_* -2\left( U_* \cdot \nabla_x U_*  \right) \cdot \nabla_x U_* \right] \right|
\lesssim
T^\epsilon \Big[
\sum\limits_{j=1}^{N} \left( \varrho_1^{\J} + \varrho_2^{\J} \right) + \varrho_3 \Big] $ will be deduced by \eqref{cancel-1} later.

\noindent$\bullet$  To estimate
\begin{small}
	\begin{align*}
		%\begin{equation*}
		%	\begin{aligned}
			&
			2\left(a  -b U_* \wedge  \right)
			\bigg[
			\left( \nabla_x U_* \cdot \nabla_x \Phi \right) \Phi - \left(   \Phi\cdot \nabla_x \Phi \right) \cdot \nabla_x U_*
			\\
			&
			-
			\sum\limits_{j=1}^N
			\left\{ \left[\nabla_x U_* \cdot \nabla_x \left( \eta_R^{\J}  Q_{\gamma_j}\Phi_{\rm in}^{\J}  \right) \right]  \left( \eta_R^{\J}  Q_{\gamma_j}\Phi_{\rm in}^{\J}  \right) -  \left[  \left( \eta_R^{\J}  Q_{\gamma_j}\Phi_{\rm in}^{\J}  \right) \cdot \nabla_x \left( \eta_R^{\J}  Q_{\gamma_j}\Phi_{\rm in}^{\J}  \right) \right] \cdot \nabla_x U_*  \right\}
			\\
			& + \sum\limits_{j=1}^N
			\left\{ \left[\nabla_x U_* \cdot \nabla_x \left( \eta_R^{\J}  Q_{\gamma_j}\Phi_{\rm in}^{\J}  \right) \right]  \left( \eta_R^{\J}  Q_{\gamma_j}\Phi_{\rm in}^{\J}  \right) -  \left[  \left( \eta_R^{\J}  Q_{\gamma_j}\Phi_{\rm in}^{\J}  \right) \cdot \nabla_x \left( \eta_R^{\J}  Q_{\gamma_j}\Phi_{\rm in}^{\J}  \right) \right] \cdot \nabla_x U_*  \right\}
			\\
			& -
			\sum\limits_{j=1}^N
			\left\{ \left[\nabla_x U^{\J} \cdot \nabla_x \left( \eta_R^{\J}  Q_{\gamma_j}\Phi_{\rm in}^{\J}  \right) \right]  \left( \eta_R^{\J}  Q_{\gamma_j}\Phi_{\rm in}^{\J}  \right) -  \left[  \left( \eta_R^{\J}  Q_{\gamma_j}\Phi_{\rm in}^{\J}  \right) \cdot \nabla_x \left( \eta_R^{\J}  Q_{\gamma_j}\Phi_{\rm in}^{\J}  \right) \right] \cdot \nabla_x U^{\J}  \right\}
			\bigg]
			\\
			&
			+
			2\left(a  -b U_* \wedge  \right)
			\sum\limits_{j=1}^N
			\left\{ \left[\nabla_x U^{\J} \cdot \nabla_x \left( \eta_R^{\J}  Q_{\gamma_j}\Phi_{\rm in}^{\J}  \right) \right]  \left( \eta_R^{\J}  Q_{\gamma_j}\Phi_{\rm in}^{\J}  \right) -  \left[  \left( \eta_R^{\J}  Q_{\gamma_j}\Phi_{\rm in}^{\J}  \right) \cdot \nabla_x \left( \eta_R^{\J}  Q_{\gamma_j}\Phi_{\rm in}^{\J}  \right) \right] \cdot \nabla_x U^{\J}  \right\}
			\\
			&
			-
			\sum\limits_{j=1}^N
			2\left(
			a -
			b U^{\J} \wedge \right)
			\Big\{ \left[\nabla_x U^{\J} \cdot \nabla_x \left( \eta_R^{\J}  Q_{\gamma_j}\Phi_{\rm in}^{\J}  \right) \right]  \left( \eta_R^{\J}  Q_{\gamma_j}\Phi_{\rm in}^{\J}  \right)
			-  \left[  \left( \eta_R^{\J}  Q_{\gamma_j}\Phi_{\rm in}^{\J}  \right) \cdot \nabla_x \left( \eta_R^{\J}  Q_{\gamma_j}\Phi_{\rm in}^{\J}  \right) \right] \cdot \nabla_x U^{\J}  \Big\}.
			%	\end{aligned}
		%\end{equation*}
	\end{align*}
\end{small}

Due to the cut-off functions, \eqref{nabU*-est}, \eqref{nabla-Phi-upp}, and \eqref{Phi*-0-j-upp}, \eqref{out-upp-split} with $\Phi_{\rm {out} } \in B_{\rm{out}}$ and $\Theta<\beta$, \eqref{out-nabla-upp}, it follows that
\begin{small}
	\begin{align*}
		%\begin{equation*}
		%	\begin{aligned}
			&
			\Big|
			(\nabla_x U_* \cdot \nabla_x \Phi )  \Phi
			-
			\sum\limits_{j=1}^N
			\left[\nabla_x U_* \cdot \nabla_x \left( \eta_R^{\J}  Q_{\gamma_j}\Phi_{\rm in}^{\J}  \right) \right]  \left( \eta_R^{\J}  Q_{\gamma_j}\Phi_{\rm in}^{\J}  \right)
			\Big|
			\\
			= \ &
			\bigg|
			\left(\nabla_x U_* \cdot \nabla_x \Phi \right)  \Phi
			-
			\bigg[ \nabla_x U_* \cdot \nabla_x \bigg(\sum_{j=1}^{N} \eta_R^{\J}  Q_{\gamma_j}\Phi_{\rm in}^{\J}  \bigg) \bigg]  \bigg(\sum_{j=1}^{N} \eta_R^{\J}  Q_{\gamma_j}\Phi_{\rm in}^{\J}  \bigg)
			\bigg|
			\\
			= \ &
			\bigg|
			\bigg(\nabla_x U_* \cdot \nabla_x \Phi \bigg)
			\bigg(  \Phi
			-
			\sum_{j=1}^{N} \eta_R^{\J}  Q_{\gamma_j}\Phi_{\rm in}^{\J}  \bigg)
			+
			\bigg[ \nabla_x U_* \cdot \nabla_x \bigg(\Phi- \sum_{j=1}^{N} \eta_R^{\J}  Q_{\gamma_j}\Phi_{\rm in}^{\J}  \bigg) \bigg]  \bigg(\sum_{j=1}^{N} \eta_R^{\J}  Q_{\gamma_j}\Phi_{\rm in}^{\J}  \bigg)
			\bigg|
			\\
			\lesssim \ &
			\bigg( \sum\limits_{j=1}^N \1_{\{ |x-q^{\J}| < 3d_q\}}  \lambda_*^{-1}\langle \rho_j \rangle^{-2}
			+   \1_{\{ \cap_{j=1}^N \{|x-q^{\J}| \ge 3 d_q  \} \}}  \lambda_* \bigg)
			\\
			&\times
			\bigg[
			\bigg\{
			\sum_{j=1}^{N}
			\Big[
			\1_{\{ |x-q^{\J}| \le 3\lambda_* R \}}
			\left(
			\lambda_*^{\nu-\delta_0-1} \langle \rho_j\rangle^{-l-1}
			+1
			\right)
			+ \1_{\{ 3\lambda_* R  < |x-q^{\J}| < 3d_q\}}
			\Big]
			+
			\1_{\{ \cap_{j=1}^N \{|x-q^{\J}| \ge 3 d_q  \} \}}
			\bigg\}
			\\
			&
			\times
			\bigg\{
			\sum_{j=1}^{N}
			\1_{\{ |x-q^{\J}| <3d_q \}}
			\left(
			\lambda_* \langle  \rho_j  \rangle
			+
			|\ln(T-t)|
			\lambda^{\Theta+1}_*  R
			\right)
			+
			\1_{\{ \cap_{j=1}^N \{|x-q^{\J}| \ge  3d_q \} \}}
			\bigg\}
			+  \sum_{j=1}^{N}
			\1_{\{ |x-q^{\J}| \le 3\lambda_* R \}}
			\lambda_*^{\nu-\delta_0} \langle \rho_j\rangle^{-l}   \bigg]
			\\
			\lesssim \ &
			\sum_{j=1}^{N}
			\Big[
			\1_{\{ |x-q^{\J}| \le 3\lambda_* R \}}
			\left(
			\lambda_*^{\nu-\delta_0-1}
			+
			|\ln (T-t)| \lambda_*^{\Theta +\nu-\delta_0-1} R
			\right)
			+ \1_{\{ 3\lambda_* R  < |x-q^{\J}| < 3d_q\}}
			\langle \rho_j\rangle^{-1}
			\Big]
			+
			\1_{\{ \cap_{j=1}^N \{|x-q^{\J}| \ge 3 d_q  \} \}} \lambda_*
			\\
			\lesssim \ &
			T^\epsilon \bigg(\sum\limits_{j=1}^N \varrho_1^{\J} + \varrho_3 \bigg)
			%	\end{aligned}
		%\end{equation*}
	\end{align*}
\end{small}
provided $\delta_0<\nu <1$ in \eqref{inn-top0-para},
\begin{equation}\label{NN-par-6}
	\Theta+\beta+\delta_0-\nu<0,
	\quad
	2\beta+\delta_0-\nu<0 .
\end{equation}

We estimate by \eqref{nabU*-est}, \eqref{nabla-eta*Phiin-upp} that
\begin{equation*}
	\begin{aligned}
		&
		\big|
		\big[\nabla_x U_* \cdot \nabla_x \big( \eta_R^{\J}  Q_{\gamma_j}\Phi_{\rm in}^{\J}  \big) \big]  \big( \eta_R^{\J}  Q_{\gamma_j}\Phi_{\rm in}^{\J}  \big)
		-
		\big[\nabla_x U^{\J} \cdot \nabla_x \big( \eta_R^{\J}  Q_{\gamma_j}\Phi_{\rm in}^{\J}  \big) \big]  \big( \eta_R^{\J}  Q_{\gamma_j}\Phi_{\rm in}^{\J}  \big)
		\big|
		\\
		\lesssim \ &
		\lambda_*
		\big|  \nabla_x \big( \eta_R^{\J}  Q_{\gamma_j}\Phi_{\rm in}^{\J}  \big) \big|  \big|  \eta_R^{\J}  Q_{\gamma_j}\Phi_{\rm in}^{\J}
		\big|
		\lesssim
		\1_{\{ |x-q^{\J}| \le 3\lambda_* R \}}
		\lambda_*^{2\nu-2\delta_0} \langle \rho_j\rangle^{-2l-1}
		\lesssim T^{\epsilon} \varrho_1^{\J}
	\end{aligned}
\end{equation*}
under the assumption $\delta_0<\nu$ in \eqref{inn-top0-para} and $\Theta+\beta-1<0$ in \eqref{out-topo-para}.
By \eqref{U*-norm}, \eqref{nabU*-est}, \eqref{nabla-eta*Phiin-upp},
\begin{equation*}
	\big|
	( U_*- U^{\J} ) \wedge
	\big[\nabla_x U^{\J} \cdot \nabla_x \big( \eta_R^{\J}  Q_{\gamma_j}\Phi_{\rm in}^{\J}  \big) \big]  \big( \eta_R^{\J}  Q_{\gamma_j}\Phi_{\rm in}^{\J}  \big)
	\big|
	\lesssim
	\1_{\{ |x-q^{\J}| \le 3\lambda_* R \}}    \lambda_*^{2\nu-2\delta_0-1} \langle \rho_j\rangle^{-2l-3}
	\lesssim T^{\epsilon} \varrho_1^{\J}
\end{equation*}
provided
\begin{equation}\label{NN-par-5.1}
	\Theta+\beta+2\delta_0-2\nu<0.
\end{equation}

The other terms in this collection can be handled in the same way.

\noindent$\bullet$
Before proceeding, we take a closer look at $\Delta_x U_*$ and $\nabla_x(|U_*|^2) \cdot \nabla_x U_{*} = 2 (U_* \cdot \nabla_x U_* )\cdot \nabla_x U_*$.
In the single bubble case $N=1$, $\Delta_x U_*$ can be neglected  by the $U_*$-operation and $\nabla_x(|U_*|^2) \cdot \nabla_x U_{*} $ vanishes automatically. However, in the case of multiple bubbles, the phenomenon is different. There exists delicate cancellation for $\Delta_x U_*-2 (U_* \cdot \nabla_x U_* )\cdot \nabla_x U_*$.

Recall the definition of $f_{\mathcal{C}_j^{-1}}$ given in \eqref{C-cal-inverse}. Claim:
\begin{align}
\notag
		&
		\Delta_x U_* =	- \sum\limits_{j=1}^N |\nabla_x U^{\J}|^2 U^{\J}
		\\ \notag
		= \ & \sum\limits_{j=1}^N \sum\limits_{k\ne j} \Big\{ 16 \eta_{R}^{\J}
		\lambda_j^{-2}
		\rho_j^2
		(\rho_j^2+1)^{-3}
		\lambda_k
		|q^{\J}-q^{\K}|^{-2}
		\big[  q^{\J}_1- q^{\K}_1 + i \big(  q^{\J}_2 - q^{\K}_2 \big) \big]
		e^{i ( \gamma_k -\gamma_j )}
		e^{-i \theta_j}
		\\ \notag
		&
		-
		16 \eta_{R}^{\J}
		\lambda_j^{-2}
		(\rho_j^2+1)^{-3}
		\lambda_k
		|q^{\J}-q^{\K}|^{-2}
		\big[  q^{\J}_1- q^{\K}_1 - i \big(  q^{\J}_2 - q^{\K}_2 \big) \big]
		e^{-i ( \gamma_k -\gamma_j )}
		e^{i \theta_j} \Big\}_{\mathcal{C}_j^{-1}}
		\\ \label{special-1}
		& + O(T^\epsilon) \bigg[
		\sum\limits_{j=1}^{N} \big( \varrho_1^{\J} + \varrho_2^{\J} \big) + \varrho_3 \bigg] - \Xi_1(x,t) U_*
	\end{align}
for some scalar function $\Xi_1(x,t)$ when
\begin{equation}\label{para-special1}
	\Theta + 2\beta  -1<0.
\end{equation}
Under the assumption \eqref{para-special1}, then
\begin{equation}\label{special-2}
	\begin{aligned}
		&
		\nabla_x(|U_*|^2) \cdot \nabla_x U_{*}
		=
		2 (U_* \cdot \nabla_x U_* ) \cdot \nabla_x U_{*}
		\\
		= \ &  \sum\limits_{j=1}^{N} \sum\limits_{m\ne j}
		\Big\{  16 \eta_{R}^{\J}
		\lambda_j^{-2}
		\rho_j^2
		(\rho_j^2+1)^{-3}
		\lambda_m
		|q^{\J}-q^{\M}|^{-2}
		\big[  q^{\J}_1- q^{\M}_1 + i \big(  q^{\J}_2 - q^{\M}_2\big) \big]
		e^{i ( \gamma_m -\gamma_j )}
		e^{-i \theta_j}
		\\
		&
		-
		16 \eta_{R}^{\J}
		\lambda_j^{-2}
		(\rho_j^2+1)^{-3}
		\lambda_m
		|q^{\J}-q^{\M}|^{-2}
		\big[  q^{\J}_1- q^{\M}_1 - i \big(  q^{\J}_2 - q^{\M}_2\big) \big]
		e^{-i ( \gamma_m -\gamma_j )}
		e^{i \theta_j} \Big\}_{\mathcal{C}_j^{-1}}
		\\
		& +  O(T^{\epsilon} ) \bigg[ \sum\limits_{j=1}^N \big(\varrho_1^{\J} +\varrho_2^{\J} \big) + \varrho_3 \bigg].
	\end{aligned}
\end{equation}
In particular,
\begin{equation}\label{cancel-1}
	\Delta_x U_* -2\left(U_* \cdot \nabla_x U_*\right)\cdot \nabla_x U_*
	=   O(T^\epsilon) \bigg[
	\sum\limits_{j=1}^{N} \left( \varrho_1^{\J} + \varrho_2^{\J} \right) + \varrho_3 \bigg] - \Xi_1(x,t) U_*.
\end{equation}

\begin{proof}[Proof of \eqref{special-1}]
	Given $j \in \{1,\dots,N\}$,
	$|\nabla_x U^{\J}|^2 U^{\J}
	=   |\nabla_x U^{\J}|^2 [( U^{\J} -U_{*} )
	+  U_{*} ]$.
	By \eqref{nabU*-est}, \eqref{U*-norm},
	\begin{equation*}
		\begin{aligned}
			&
			\big| |\nabla_x U^{\J}|^2 ( U^{\J} -U_{*} ) \big|
			\lesssim
			\1_{\{ |x-q^{\J}| < 3d_q\}}  \lambda_*^{-1}\langle \rho_j \rangle^{-4}
			+   \1_{\{ |x-q^{\J}| \ge 3 d_q \}}  \lambda_*^2,
			\\
			&
			\big|
			\big( 1-\eta_{R}^{\J} \big)
			|\nabla_x U^{\J}|^2 ( U^{\J} -U_{*} )
			\big|
			\lesssim
			\1_{\{ \lambda_* R /2 \le  |x-q^{\J}| < 3d_q\}}  \lambda_*^3 |x-q^{\J}|^{-4}
			+   \1_{\{ |x-q^{\J}| \ge 3 d_q \}}  \lambda_*^2
			\lesssim
			T^\epsilon \big(  \varrho_2^{\J} + \varrho_3 \big).
		\end{aligned}
	\end{equation*}
	
	The estimate $\big| \eta_{R}^{\J}
	|\nabla_x U^{\J}|^2 ( U^{\J} -U_{*} )
	\big|
	\lesssim \eta_{R}^{\J} \lambda_*^{-1}\langle \rho_j \rangle^{-4}$  is too rough and can not be controlled by the outer topology. More sophisticated analysis will be applied. Indeed, by \eqref{nablaW} and the representation \eqref{rota-vector},
	\begin{equation*}
		\begin{aligned}
			&
			|\nabla_x U^{\J}|^2 \left( U^{\J} -U_{*} \right)
			=
			-2  \lambda_j^{-2}|\nabla_{y^{\J}} U^{\J}|^2
			\sum\limits_{k\ne j}
			(|y^{\K}|^2+1)^{-1}
			Q_{\gamma_k}
			\begin{bmatrix}
				y^{\K}_1, y^{\K}_2, -1
			\end{bmatrix}^{\tr}
			\\
			= \ &
			-16  \lambda_j^{-2}
			(\rho_j^2+1)^{-2}
			\sum\limits_{k\ne j}
			( \rho_k^2+1)^{-1}
			\Big[
			e^{i\gamma_k}\left(y^{\K}_1 + i y^{\K}_2 \right), -1
			\Big]^{\tr}.
		\end{aligned}
	\end{equation*}
	For $k\ne j$,
	\begin{equation*}
		\begin{aligned}
			&
			\rho_k^2 = \lambda_k^{-2} |x-\xi^{\K}|^2
			=
			\lambda_k^{-2}
			|\xi^{\J}-\xi^{\K}|^2
			\left\{
			1+
			|\xi^{\J}-\xi^{\K}|^{-2}
			\left[
			|x-\xi^{\J}|^2
			+
			2(x-\xi^{\J})\cdot (\xi^{\J}-\xi^{\K})
			\right]
			\right\},
			\\
			&
			\left(\rho_k^2 + 1\right)^{-1}
			=
			\lambda_k^{2}
			|\xi^{\J}-\xi^{\K}|^{-2}
			\left\{
			1+
			|\xi^{\J}-\xi^{\K}|^{-2}
			\left[
			\lambda_k^{2}
			+
			|x-\xi^{\J}|^2
			+
			2(x-\xi^{\J})\cdot (\xi^{\J}-\xi^{\K})
			\right]
			\right\}^{-1}.
		\end{aligned}
	\end{equation*}
	In particular,
	\begin{equation}\label{rhok in eta-j}
		\eta_R^{\J}
		\left(\rho_k^2 + 1\right)^{-1}
		=
		\eta_R^{\J}
		\lambda_k^{2}
		|\xi^{\J}-\xi^{\K}|^{-2}
		\left(
		1+ O\left(\lambda_*^2 + \lambda_* R\right)
		\right),
	\end{equation}
	which implies
	\begin{equation}\label{qd24May11-1}
		\begin{aligned}
			&
			\eta_{R}^{\J}  |\nabla_x U^{\J}|^2 \left( U^{\J} -U_{*} \right)
			=
			-16 \eta_{R}^{\J}
			\lambda_j^{-2}
			(\rho_j^2+1)^{-2}
			\sum\limits_{k\ne j}
			( \rho_k^2+1)^{-1}
			\Big[
			e^{i\gamma_k}\left(y^{\K}_1 + i y^{\K}_2 \right), -1
			\Big]^{\tr}
			\\
			= \ &
			-16 \eta_{R}^{\J}
			\lambda_j^{-2}
			(\rho_j^2+1)^{-2}
			\sum\limits_{k\ne j}
			\lambda_k^{2}
			|\xi^{\J}-\xi^{\K}|^{-2}
			\left(
			1+ O\left(\lambda_*^2 + \lambda_* R\right)
			\right)
			\Big[
			e^{i\gamma_k}\left(y^{\K}_1 + i y^{\K}_2 \right), -1
			\Big]^{\tr}  .
		\end{aligned}
	\end{equation}
	Notice by \eqref{polar-coor},
	\begin{align*}
			y^{\K}_1 + i y^{\K}_2
			= \ & \big[ \lambda_j y_1^{\J}+ \xi^{\J}_1- \xi^{\K}_1 + i \big( \lambda_j y_2^{\J} + \xi^{\J}_2 - \xi^{\K}_2\big) \big]
			\lambda_k^{-1}
			\\
			= \ &
			\lambda_j \lambda_k^{-1} \rho_j e^{i\theta_j}
			+
			\lambda_k^{-1}
			\big[  \xi^{\J}_1- \xi^{\K}_1 + i \big( \xi^{\J}_2 - \xi^{\K}_2\big) \big].
		\end{align*}
	Recall $\mathbf{f}_{\mathcal{C}_j }$ defined in \eqref{qd24May12-1}. Combining \eqref{proj-cal-C}, we have
	\begin{align}
		%	\begin{equation}
			%		\begin{aligned}
				&
				\Big(\Pi_{U^{\J \perp}} \Big[
				e^{i\gamma_k}\left(y^{\K}_1 + i y^{\K}_2 \right), -1
				\Big]^{\tr}  \Big)_{\mathcal{C}_j}
				=
				\Big(1
				-
				\frac{2}{\rho_j^2+1}
				{\rm{Re}}
				\Big) \left[ \left(y^{\K}_1 + i y^{\K}_2 \right)
				e^{i\left(-\theta_j+ \gamma_k -\gamma_j\right)}\right]
				+
				\frac{2\rho_j}{\rho_j^2+1}
				\nonumber
				\\
				= \ &
				\lambda_j
				\lambda_k^{-1} \rho_j  \Big(1
				-
				\frac{2}{\rho_j^2+1}
				{\rm{Re}}
				\Big) \left[
				e^{i\left( \gamma_k -\gamma_j\right)}
				\right]
				+
				\frac{2\rho_j}{\rho_j^2+1}
				\nonumber
				\\
				& +
				\lambda_k^{-1}
				\Big(1
				-
				\frac{2}{\rho_j^2+1}
				{\rm{Re}}
				\Big) \left\{ \left[  \xi^{\J}_1- \xi^{\K}_1 + i \left(  \xi^{\J}_2 - \xi^{\K}_2\right) \right]
				e^{i\left(-\theta_j+ \gamma_k -\gamma_j\right)}
				\right\} ,
				\nonumber
				\\
				&
				\Big[
				e^{i\gamma_k}\left(y^{\K}_1 + i y^{\K}_2 \right), -1
				\Big]^{\tr} \cdot U^{\J}
				=
				\frac{2\rho_j}{\rho_j^2+1}
				{\rm{Re}} \left[
				\left(y^{\K}_1 + i y^{\K}_2 \right) e^{i\left(-\theta_j+ \gamma_k -\gamma_j\right)}
				\right]
				-  \frac{\rho_j^2-1}{\rho_j^2+1}
				\nonumber
				\\
				= \ &
				\lambda_j \lambda_k^{-1}  \frac{2\rho_j^2}{\rho_j^2+1}
				{\rm{Re}} \left[
				e^{i\left( \gamma_k -\gamma_j\right)}
				\right]
				-  \frac{\rho_j^2-1}{\rho_j^2+1}  +
				\lambda_k^{-1}
				\frac{2\rho_j}{\rho_j^2+1}
				{\rm{Re}} \left\{
				\left[  \xi^{\J}_1- \xi^{\K}_1 + i \left( \xi^{\J}_2 - \xi^{\K}_2\right) \right] e^{i\left(-\theta_j+ \gamma_k -\gamma_j\right)}
				\right\}.
				\label{Uk-U-infty-proj}
				%		\end{aligned}
			%	\end{equation}
	\end{align}
	Then
	\begin{equation}\label{qd24May12-3}
		\begin{aligned}
			&
			\eta_{R}^{\J}
			\left\{\Pi_{U^{\J \perp}}
			\left[
			|\nabla_x U^{\J}|^2 \left( U^{\J} -U_{*} \right)
			\right]
			\right\}_{\mathcal{C}_j}
			=
			-16 \eta_{R}^{\J}
			\lambda_j^{-2}
			(\rho_j^2+1)^{-2}
			\sum\limits_{k\ne j}
			\lambda_k^{2}
			|\xi^{\J}-\xi^{\K}|^{-2}
			\left(
			1+ O\left(\lambda_*^2 + \lambda_* R\right)
			\right)
			\\
			&\quad \times
			\bigg\{
			\lambda_j
			\lambda_k^{-1} \rho_j  \Big(1
			-
			\frac{2}{\rho_j^2+1}
			{\rm{Re}}
			\Big) \left[
			e^{i\left( \gamma_k -\gamma_j\right)}
			\right]
			+
			\frac{2\rho_j}{\rho_j^2+1}
			\bigg\}
			\\
			&\quad
			-16 \eta_{R}^{\J}
			\lambda_j^{-2}
			(\rho_j^2+1)^{-2}
			\sum\limits_{k\ne j}
			\lambda_k
			|\xi^{\J}-\xi^{\K}|^{-2}
			\left(
			1+ O\left(\lambda_*^2 + \lambda_* R\right)
			\right)
			\\
			&\quad \times
			\Big(1
			-
			\frac{2}{\rho_j^2+1}
			{\rm{Re}}
			\Big) \left\{ \left[  \xi^{\J}_1- \xi^{\K}_1 + i \left(  \xi^{\J}_2 - \xi^{\K}_2\right) \right]
			e^{i\left(-\theta_j+ \gamma_k -\gamma_j\right)}
			\right\}.
		\end{aligned}
	\end{equation}
	Here, due to $\Theta+\beta-1<0$ given in \eqref{out-topo-para},
	\begin{equation*}
		\begin{aligned}
			&
			\Big| 16 \eta_{R}^{\J}
			\lambda_j^{-2}
			(\rho_j^2+1)^{-2}
			\sum\limits_{k\ne j}
			\lambda_k^{2}
			|\xi^{\J}-\xi^{\K}|^{-2}
			\left(
			1+ O\left(\lambda_*^2 + \lambda_* R\right)
			\right)
			\\
			&\times
			\Big\{
			\lambda_j
			\lambda_k^{-1} \rho_j  \Big(1
			-
			\frac{2}{\rho_j^2+1}
			{\rm{Re}}
			\Big) \left[
			e^{i\left( \gamma_k -\gamma_j\right)}
			\right]
			+
			\frac{2\rho_j}{\rho_j^2+1}
			\Big\} \Big|
			\lesssim
			\eta_{R}^{\J}
			\langle \rho_j \rangle^{-3} \lesssim
			T^{\epsilon} \varrho_1^{\J} ,
		\end{aligned}
	\end{equation*}
	\begin{equation*}
		\begin{aligned}
			&
			\Big| 16 \eta_{R}^{\J}
			\lambda_j^{-2}
			(\rho_j^2+1)^{-2}
			\sum\limits_{k\ne j}
			\lambda_k
			|\xi^{\J}-\xi^{\K}|^{-2}
			\left(
			O\left(\lambda_*^2 + \lambda_* R\right)
			\right)
			\\
			&\times
			\Big(1
			-
			\frac{2}{\rho_j^2+1}
			{\rm{Re}}
			\Big) \left\{ \left[  \xi^{\J}_1- \xi^{\K}_1 + i \left(  \xi^{\J}_2 - \xi^{\K}_2\right) \right]
			e^{i\left(-\theta_j+ \gamma_k -\gamma_j\right)}
			\right\}   \Big|
			\lesssim
			\eta_{R}^{\J} R
			\langle \rho_j \rangle^{-4}
			\lesssim T^{\epsilon} \varrho_1^{\J} ,
		\end{aligned}
	\end{equation*}
	where for the second inequality, we require
	\begin{equation}\label{qd24May12-4}
		\Theta + 2\beta  -1<0.
	\end{equation}
	\begin{align*}
		%	\begin{equation*}
			%		\begin{aligned}
				& \Big|
				-16 \eta_{R}^{\J}
				\lambda_j^{-2}
				(\rho_j^2+1)^{-2}
				\sum\limits_{k\ne j}
				\lambda_k
				|\xi^{\J}-\xi^{\K}|^{-2}
				\Big(1
				-
				\frac{2}{\rho_j^2+1}
				{\rm{Re}}
				\Big) \left\{ \left[  \xi^{\J}_1- \xi^{\K}_1 + i \left(  \xi^{\J}_2 - \xi^{\K}_2\right) \right]
				e^{i\left(-\theta_j+ \gamma_k -\gamma_j\right)}
				\right\}
				\\
				& +
				16 \eta_{R}^{\J}
				\lambda_j^{-2}
				(\rho_j^2+1)^{-2}
				\sum\limits_{k\ne j}
				\lambda_k
				|q^{\J}-q^{\K}|^{-2}
				\Big(1
				-
				\frac{2}{\rho_j^2+1}
				{\rm{Re}}
				\Big) \left\{ \left[  q^{\J}_1- q^{\K}_1 + i \left(  q^{\J}_2 - q^{\K}_2\right) \right]
				e^{i\left(-\theta_j+ \gamma_k -\gamma_j\right)}
				\right\}   \Big|
				\\
				\lesssim \ &
				\eta_{R}^{\J}
				|\ln T|^{-1} \ln^2(T-t)
				\langle \rho_j \rangle^{-4}
				\lesssim T^{\epsilon} \varrho_1^{\J}.
				%		\end{aligned}
			%	\end{equation*}
	\end{align*}
	\begin{align*}
		%	\begin{equation}
			%		\begin{aligned}
				&
				-16 \eta_{R}^{\J}
				\lambda_j^{-2}
				(\rho_j^2+1)^{-2}
				\sum\limits_{k\ne j}
				\lambda_k
				|q^{\J}-q^{\K}|^{-2}
				\Big(1
				-
				\frac{2}{\rho_j^2+1}
				{\rm{Re}}
				\Big) \left\{ \left[  q^{\J}_1- q^{\K}_1 + i \left(  q^{\J}_2 - q^{\K}_2\right) \right]
				e^{i\left(-\theta_j+ \gamma_k -\gamma_j\right)}
				\right\}
				\\
				= \ &
				-16 \eta_{R}^{\J}
				\lambda_j^{-2}
				\rho_j^2
				(\rho_j^2+1)^{-3}
				\sum\limits_{k\ne j}
				\lambda_k
				|q^{\J}-q^{\K}|^{-2}
				\left[  q^{\J}_1- q^{\K}_1 + i \left(  q^{\J}_2 - q^{\K}_2\right) \right]
				e^{i\left( \gamma_k -\gamma_j\right)}
				e^{-i \theta_j}
				\\
				&
				+
				16 \eta_{R}^{\J}
				\lambda_j^{-2}
				(\rho_j^2+1)^{-3}
				\sum\limits_{k\ne j}
				\lambda_k
				|q^{\J}-q^{\K}|^{-2}
				\left[  q^{\J}_1- q^{\K}_1 - i \left(  q^{\J}_2 - q^{\K}_2\right) \right]
				e^{-i\left( \gamma_k -\gamma_j\right)}
				e^{i \theta_j},
				%		\end{aligned}
			%	\end{equation}
	\end{align*}
	which will be put into mode  $-1$ and mode $1$ respectively.

	For the projection in $U^{\J}$, by \eqref{qd24May11-1}, \eqref{Uk-U-infty-proj}, we calculate
	\begin{equation*}
		\begin{aligned}
			&
			\eta_{R}^{\J}  |\nabla_x U^{\J}|^2 \left( U^{\J} -U_{*} \right) \cdot  U^{\J}
			\\
			= \ &
			-16 \eta_{R}^{\J}
			\lambda_j^{-2}
			(\rho_j^2+1)^{-2}
			\sum\limits_{k\ne j}
			\lambda_k^{2}
			|\xi^{\J}-\xi^{\K}|^{-2}
			\left(
			1+ O\left(\lambda_*^2 + \lambda_* R\right)
			\right)
			\\
			& \times
			\Big[
			\lambda_j \lambda_k^{-1}  \frac{2\rho_j^2}{\rho_j^2+1}
			{\rm{Re}} \left[
			e^{i\left( \gamma_k -\gamma_j\right)}
			\right]
			-  \frac{\rho_j^2-1}{\rho_j^2+1}  +
			\lambda_k^{-1}
			\frac{2\rho_j}{\rho_j^2+1}
			{\rm{Re}} \left\{
			\left[  \xi^{\J}_1- \xi^{\K}_1 + i \left( \xi^{\J}_2 - \xi^{\K}_2\right) \right] e^{i\left(-\theta_j+ \gamma_k -\gamma_j\right)}
			\right\}
			\Big].
		\end{aligned}
	\end{equation*}
	Thus by $U_*$-operation and \eqref{U*-norm},
	\begin{equation*}
		\big|
		\eta_{R}^{\J}  |\nabla_x U^{\J}|^2
		[ ( U^{\J} -U_{*} ) \cdot  U^{\J} ] ( U^{\J} -U_{*} )
		\big|
		\lesssim
		\eta_{R}^{\J} \langle \rho_j\rangle^{-4}
		\lesssim T^{\epsilon }
		\varrho_1^{\J}.
	\end{equation*}
	In sum, we conclude the validity of \eqref{special-1}.	
\end{proof}

\begin{proof}[Proof of \eqref{special-2}]
	Note that
	$2 ( U_* \cdot \nabla_x U_* )
	\cdot \nabla_x U_{*}
	=
	2
	\Big(  \sum\limits_{j=1}^{N} U_* \cdot \nabla_x U^{\J} \Big) \cdot
	\sum\limits_{n=1}^{N} \nabla_x U^{\N}
	=
	2
	\Big[ \sum\limits_{j=1}^{N}  \sum\limits_{m\ne j}   (U^{\M} - U_{\infty} ) \cdot \nabla_x U^{\J} \Big] \cdot
	\sum\limits_{n=1}^{N} \nabla_x U^{\N}$.
	For any fixed $j \in \{1,\dots,N\}$ and any $m\ne j$, $n\ne j$, by \eqref{nabU*-est},
	\begin{align*}
%	\begin{equation*}
%		\begin{aligned}
			&
			\left|
			\left[    \left(U^{\M} - U_{\infty} \right) \cdot \nabla_x U^{\J} \right] \cdot \nabla_x U^{\N} \right|
			\lesssim
			\left| \nabla_x U^{\J} \right| \left| \nabla_x U^{\N}  \right|
			\\
			\lesssim \ &
			\1_{\{ |x-q^{\J}| < 3d_q\}}  \langle \rho_j \rangle^{-2}
			+
			\1_{\{ |x-q^{\N}| < 3d_q\}}  \langle \rho_n \rangle^{-2}
			+   \1_{\{ \{ |x-q^{\J}| \ge 3 d_q  \} \cap \{ |x-q^{\N}| \ge 3 d_q  \} \} }  \lambda_*^2
			\lesssim
			T^{\epsilon} \varrho_3,
			\\
			&
			\big|
			(1-\eta_R^{\J} )
			\big[   (U^{\M} - U_{\infty} ) \cdot \nabla_x U^{\J} \big] \cdot
			\nabla_x U^{\J}
			\big|
			\\
			\lesssim \ &
			\1_{ \{\lambda_* R/2 \le |x-q^{\J}| \le d_q \} } \lambda_*^3 |x-q^{\J}|^{-4}
			+
			\1_{ \{ |x-q^{\J}| > d_q \} }
			\lambda_*^2
			\lesssim
			T^\epsilon \big(  \varrho_2^{\J} + \varrho_3 \big).
%		\end{aligned}
%	\end{equation*}
	\end{align*}
	A more refined analysis is needed for $ 2 \eta_R^{\J}
	\big[  \sum\limits_{m \ne j}   (U^{\M} - U_{\infty} ) \cdot \nabla_x U^{\J} \big] \cdot
	\nabla_x U^{\J}  $.
	For any $m\ne j$,
		\begin{equation*}
			\Big[  \big(U^{\M} - U_{\infty} \big) \cdot \nabla_x U^{\J} \Big] \cdot
			\nabla_x U^{\J}
			=
			\sum\limits_{k=1}^2
			\Big[  \big(U^{\M} - U_{\infty} \big) \cdot \pp_{x_k} U^{\J} \Big]
			\partial_{x_k} U^{\J}
			=
			\lambda_j^{-2}
			\sum\limits_{k=1}^2
			\Big[  \big(U^{\M} - U_{\infty} \big) \cdot \pp_{y_k^{\J}} U^{\J} \Big] \pp_{y_k^{\J}} U^{\J}.
		\end{equation*}
Recall \eqref{polar-coor}. By \eqref{Frenet-deri} and \eqref{nablaW}, we write in polar coordinates
	\begin{equation*}
		\begin{aligned}
			\pp_{y_1^{\J}} U^{\J} = \ &  \cos \theta_j \pp_{\rho_j} U^{\J} -
			\rho_j^{-1} \sin \theta_j \pp_{\theta_j} U^{\J}
			=
			-2  (\rho_j^2+1)^{-1} \big(\cos \theta_j Q_{\gamma_j}E_1^{\J} + \sin \theta_j Q_{\gamma_j}  E_2^{\J} \big)
			,
			\\
			\pp_{y_2^{\J}} U^{\J} = \ &
			\sin \theta_j \pp_{\rho_j} U^{\J} +
			\rho_j^{-1} \cos \theta_j \pp_{\theta_j} U^{\J}
			=
			-2	(\rho_j^2+1)^{-1}  \big(\sin \theta_j  Q_{\gamma_j} E_1^{\J} -
			\cos \theta_j Q_{\gamma_j} E_2^{\J} \big),
		\end{aligned}
	\end{equation*}
	which implies
	\begin{align*}
		%	\begin{equation*}
			%		\begin{aligned}
				&
				\lambda_j^{-2}
				\sum\limits_{k=1}^2
				\Big[  (U^{\M} - U_{\infty} ) \cdot \pp_{y_k^{\J}} U^{\J} \Big] \pp_{y_k^{\J}} U^{\J}
				\\
				= \ &
				4
				\lambda_j^{-2}
				(\rho_j^2+1)^{-2}
				\Big\{
				\left[  \left(U^{\M} - U_{\infty} \right) \cdot \left(\cos \theta_j Q_{\gamma_j}E_1^{\J} + \sin \theta_j Q_{\gamma_j}  E_2^{\J} \right) \right] \left(\cos \theta_j Q_{\gamma_j}E_1^{\J} + \sin \theta_j Q_{\gamma_j}  E_2^{\J} \right)
				\\
				&
				+
				\left[  \left(U^{\M} - U_{\infty} \right) \cdot \left(\sin \theta_j  Q_{\gamma_j} E_1^{\J} -
				\cos \theta_j Q_{\gamma_j} E_2^{\J} \right)  \right] \left(\sin \theta_j  Q_{\gamma_j} E_1^{\J} -
				\cos \theta_j Q_{\gamma_j} E_2^{\J} \right)
				\Big\}
				\\
				= \ &
				4
				\lambda_j^{-2}
				(\rho_j^2+1)^{-2}
				\left\{
				\left[  \left(U^{\M} - U_{\infty} \right) \cdot  Q_{\gamma_j}E_1^{\J}  \right]  Q_{\gamma_j}E_1^{\J}
				+
				\left[  \left(U^{\M} - U_{\infty} \right) \cdot  Q_{\gamma_j} E_2^{\J}  \right]  Q_{\gamma_j} E_2^{\J}
				\right\}.
				%		\end{aligned}
			%	\end{equation*}
	\end{align*}
The representation \eqref{rota-vector} then reads
	\begin{equation*}
		U^{\M} -U_{\infty}
		=
		2
		(|y^{\M}|^2+1)^{-1}
		Q_{\gamma_m}
		\begin{bmatrix}
			y^{\M}_1, y^{\M}_2, -1
		\end{bmatrix}^{\tr}
		=
		2
		( \rho_m^2+1)^{-1}
		\Big[
		e^{i\gamma_m} (y^{\M}_1 + i y^{\M}_2 ), -1
		\Big]^{\tr}.
	\end{equation*}
From this and \eqref{Uk-U-infty-proj}, it follows that
	\begin{align*}
		%	\begin{equation*}
			%		\begin{aligned}
				&
				\Big( \lambda_j^{-2}
				\sum\limits_{k=1}^2
				\Big[  (U^{\M} - U_{\infty} ) \cdot \pp_{y_k^{\J}} U^{\J} \Big] \pp_{y_k^{\J}} U^{\J} \Big)_{\mathcal{C}_j}
				\\
				= \ &
				4
				\lambda_j^{-2}
				(\rho_j^2+1)^{-2}
				\Big[
				(U^{\M} - U_{\infty} ) \cdot  Q_{\gamma_j}E_1^{\J}
				+
				i
				(U^{\M} - U_{\infty} ) \cdot  Q_{\gamma_j} E_2^{\J}
				\Big]
				\\
				= \ &
				8
				\lambda_j^{-2}
				(\rho_j^2+1)^{-2}
				( \rho_m^2+1)^{-1}
				\Big( \Pi_{U^{\J \perp}}
				\Big[
				e^{i\gamma_m}\big(y^{\M}_1 + i y^{\M}_2 \big), -1
				\Big]^{\tr} \Big)_{\mathcal{C}_j}
				\\
				= \ &
				8
				\lambda_j^{-2}
				(\rho_j^2+1)^{-2}
				( \rho_m^2+1)^{-1}
				\Big[
				\lambda_j
				\lambda_m^{-1} \rho_j  \Big(1
				-
				\frac{2}{\rho_j^2+1}
				{\rm{Re}}
				\Big) \left[
				e^{i\left( \gamma_m -\gamma_j\right)}
				\right]
				+
				\frac{2\rho_j}{\rho_j^2+1}
				\\
				& +
				\lambda_m^{-1}
				\Big(1
				-
				\frac{2}{\rho_j^2+1}
				{\rm{Re}}
				\Big) \left\{ \left[  \xi^{\J}_1- \xi^{\M}_1 + i \left(  \xi^{\J}_2 - \xi^{\M}_2\right) \right]
				e^{i\left(-\theta_j+ \gamma_m -\gamma_j\right)}
				\right\}
				\Big].
				%		\end{aligned}
			%	\end{equation*}
	\end{align*}
	Combining \eqref{rhok in eta-j}, we have
	\begin{align*}
			&
			2 \eta_R^{\J}
			\Big( \lambda_j^{-2}
			\sum\limits_{k=1}^2
			\Big[  (U^{\M} - U_{\infty} ) \cdot \pp_{y_k^{\J}} U^{\J} \Big] \pp_{y_k^{\J}} U^{\J} \Big)_{\mathcal{C}_j}
			\\
			= \ &
			16 \eta_R^{\J}
			\lambda_j^{-2}
			(\rho_j^2+1)^{-2}
			\lambda_m^{2}
			|\xi^{\J}-\xi^{\M}|^{-2}
			\left(
			1+ O\left(\lambda_*^2 + \lambda_* R\right)
			\right)
			\Big[
			\lambda_j
			\lambda_m^{-1} \rho_j  \Big(1
			-
			\frac{2}{\rho_j^2+1}
			{\rm{Re}}
			\Big) \big[
			e^{i ( \gamma_m -\gamma_j )}
			\big]
			+
			\frac{2\rho_j}{\rho_j^2+1}
			\Big]
			\\
			& +
			16 \eta_R^{\J}
			\lambda_j^{-2}
			(\rho_j^2+1)^{-2}
			\lambda_m
			|\xi^{\J}-\xi^{\M}|^{-2}
			\left(
			1+ O\left(\lambda_*^2 + \lambda_* R\right)
			\right)
			\\
			& \times
			\Big(1
			-
			\frac{2}{\rho_j^2+1}
			{\rm{Re}}
			\Big) \left\{ \left[  \xi^{\J}_1- \xi^{\M}_1 + i \left(  \xi^{\J}_2 - \xi^{\M}_2\right) \right]
			e^{i\left(-\theta_j+ \gamma_m -\gamma_j\right)}
			\right\}
			\\
			= \ & 16 \eta_{R}^{\J}
			\lambda_j^{-2}
			\rho_j^2
			(\rho_j^2+1)^{-3}
			\lambda_m
			|q^{\J}-q^{\M}|^{-2}
			\big[  q^{\J}_1- q^{\M}_1 + i \big(  q^{\J}_2 - q^{\M}_2\big) \big]
			e^{i ( \gamma_m -\gamma_j )}
			e^{-i \theta_j}
			\\
			&
			-
			16 \eta_{R}^{\J}
			\lambda_j^{-2}
			(\rho_j^2+1)^{-3}
			\lambda_m
			|q^{\J}-q^{\M}|^{-2}
			\big[  q^{\J}_1- q^{\M}_1 - i \big(  q^{\J}_2 - q^{\M}_2\big) \big]
			e^{-i ( \gamma_m -\gamma_j )}
			e^{i \theta_j}  +
			O(T^\epsilon) \varrho_1^{\J},
		\end{align*}
	where for the last equality, we used the assumption \eqref{qd24May12-4} and the same estimates for \eqref{qd24May12-3}.	
\end{proof}

Combining $l>0,
0<\delta_0<\nu <1 $ given in  \eqref{inn-top0-para}, \eqref{nab-A-para} (These are used for the estimate of $\nabla_x A$ in \eqref{nab-A-est}),  \eqref{o-para-2}, \eqref{o-para-3}, \eqref{o-para-4},  \eqref{o-para-6}, \eqref{o-para-7}, \eqref{o-para-7.1}, \eqref{o-para-8},
\eqref{qd24May12-5},  \eqref{2deri-para-1}, \eqref{2deri-para-2},
\eqref{NN-par-1}, \eqref{NN-par-2}, \eqref{NN-par-3}, \eqref{NN-par-4}, \eqref{NN-par-5},  \eqref{NN-par-6}, \eqref{NN-par-5.1}, and \eqref{para-special1},
we get the parameter requirement \eqref{G-para}.
\end{proof}

\medskip

\section{Index}\label{Index}

\medskip

In this section, some frequently used terminologies and symbols are compiled.

\medskip

\noindent
\begin{tabular}{|c|c|c|c|}
	\hline
	{\it $U_*$-operation} &  {\it Algebraic power type} (${\mathbf{AP}}$)   & {\it Re-gluing process} & {\it Sub-Gaussian estimates} \\
	\hline
	\eqref{U-direc-trick} & the context near \eqref{P1-def}
	& Propositions \ref{modek-prop}, \ref{Re-m0-prop}, \ref{qd24July12-8-prop}
	& Subsection \ref{fund-out-sec}
	\\
	\hline
\end{tabular}

\noindent
	\begin{tabular}{|c|c|c|c|c|c|c|c|c|c|c|}
		\hline
		 $W$ & $Q_{\gamma}$ & {~} $y^{\J}$, $\rho_j$, $r_j$, $\theta_j$ {~} &  $W^{\J}$  &  $U^{\J}$ & $U_*$ &  $E_1^{\J}$, $E_2^{\J}$  & $Z_{k,l}^{\J}$ & $\mathcal{Z}_{k,l}$ & $\mathbf{f}_{\mathcal{C}_j }$ & $f_{\mathcal{C}_j^{-1}} $   \\
		\hline
		 \eqref{W-def} & \eqref{Q-gamma-def} & \eqref{polar-coor} &  \eqref{def-Wjjj} & \eqref{def-U1U2} & \eqref{def-U*} & \eqref{def-E1E2}
		& \eqref{def-kernels}
		& \eqref{scalr-Z} & \eqref{qd24May12-1} & \eqref{C-cal-inverse}
		\\
		\hline
	\end{tabular}

\noindent
\begin{tabular}{|c|c|c|c|c|c|c|c|c|c|c|}
	\hline
	 $\mathbf{g}_{\mathbb{C}_j }$ & $g_{\mathbb{C}_j^{-1}}$ & $\Pi_{\gbf^{\perp}} \fbf$ & $d_q$, $p_j(t)$ & $\la_*(t)$  & $S[\cdot]$  & $z_j$, $x^{\J}$
	& $\Phi^{*{\J}}_0$ & $\Phi^{\J}_0$, $K_0(\zeta_j)$ &  $\zeta_j$, $\iota_j$ & $\mathcal{S}^{\J}$ \\
	\hline
	 \eqref{Cbb-def} & \eqref{Cbb-inverse}   & \eqref{def-proj} & \eqref{dq-pj-def}
	&  \eqref{lam-ansatz} & \eqref{S-error fun}  & \eqref{zj-def}
	& \eqref{def-globalcorrection-J} &  \eqref{def-Phi0} &  \eqref{zetaj} & \eqref{Sj-def}
	\\
	\hline
\end{tabular}

\noindent
	\begin{tabular}{|c|c|c|c|c|c|c|c|c|c|c|c|}
		\hline
		 $M_0^{\J}$ & $\tilde{M}_0^{\J}$ & $M_1^{\J}$ & $\tilde{M}_1^{\J}$ & $M_{-1}^{\J}$
		& $\Phi$ & {~} $\eta_R^{\J}$,
		$\eta_{d_q}^{\J}$  & $R(t)$ & $A$ & $\mathcal N[\Phi]$ & $\Phi_{\rm out}$ &  $\Phi_{\rm in}^{\J}$  \\
		\hline
		 \eqref{M0-def-mu3}
		& \eqref{til-M0-def}
		& \eqref{M1-def} & \eqref{tilde-M1-def}
		& \eqref{M(-1)-def}
		& \eqref{u-def} & \eqref{qd24Apr12-3} & \eqref{RPhi-ansatz} & \eqref{A-def} & \eqref{def-N} & \eqref{outer-eq} & \eqref{inner-eq}
		\\
		\hline
	\end{tabular}

\noindent
	\begin{tabular}{|c|c|c|c|c|c|c|c|c|}
		\hline
		{\quad}  $\mathsf{D}_{2 C_{\lambda} R}$, \quad  $\mathcal{H}^{\J}$ {\quad} & {~} $\mathcal{H}_1^{\J}$ {~} & {~} $\mathcal{H}_{ \rm{in} }^{\J}$ {~} & {~} $\mathcal{G}$ {~} & {~} $\mathbf{B}_{\Phi,U_{*}}$ {~} & {~}  $\tilde{ \mathbf{B} }_{\Phi, U_{*}}$ {~}
		& {\quad} $Z_*$  {\quad} & {~} $\vartheta_{mn}$ {~} & {~} $\tau_j(t)$ {~} \\
		\hline
		 \eqref{Hj-def} & \eqref{Hj-1-def} & \eqref{def-HPhi2} & \eqref{G-def} & \eqref{B-matrix}  & \eqref{til-B-matrix}
		& \eqref{Z*-def}  & \eqref{vartheta-def} &  \eqref{tau-j-def}
		\\
		\hline
	\end{tabular}

\noindent
	\begin{tabular}{|c|c|c|c|c|c|}
		\hline
		 {~} $\| \cdot \|_{{\rm in},\nu-\delta_0,l }$, $[\cdot]_{{\rm in},\nu-\delta_0,l,\varsigma_{\rm{in} }}$, $\|\cdot\|_{{\rm in},\nu-\delta_0,l,\varsigma_{\rm{in} }}$ {~} & {~} $B_{\rm{in}}^{\J}$ {~} & {~} $\varrho_1^{\J}$, $\varrho_2^{\J}$, $\varrho_3$ {~} & {~} $\|\cdot\|_{**}$ {~} & $\| \cdot\|_{\sharp, \Theta,\alpha}$ & $B_{\rm{out}}$  \\
		\hline
		 \eqref{inn-topo} & \eqref{inner-space} & \eqref{rho-weights}
		& \eqref{G-topo-def} & \eqref{out-topo}
		&  \eqref{out-space}
		\\
		\hline
	\end{tabular}

\noindent
	\begin{tabular}{|c|c|c|c|c|c|c|}
		\hline
		 \  $\tilde{L}_{j}^{\#}[\Phi_{\rm {out} }]$, $\tilde{L}_{j,k}^{\#}[\Phi_{\rm {out} }]$ \  & {~}  $\tilde{l}_{j}^{\#}[\Phi_{\rm {out} }]$ {~} & {~} $\tilde{l}_{j,k}^{\#}[\Phi_{\rm {out} }]$ {~} & {~} $\Phi_{\rm in}^{\mbox{{\tiny{$[j1]$}}}}$ {~} & $\Phi_{\rm in}^{\mbox{{\tiny{$[j2]$}}}}$ & $\mathbf{R}_0$, ${\rm DC}_j[\cdot]$ & $\mathcal{R}_0[\cdot]$ \\
		\hline
		 \eqref{Llarge}
		& \eqref{Lsmall} & \eqref{l-jk0-def} & \eqref{inner-eq-1} & \eqref{inner-eq-2}
		& \eqref{R0-def} & Proposition \ref{keyprop}
		\\
		\hline
	\end{tabular}

\noindent
	\begin{tabular}{|c|c|c|c|c|c|c|c|}
		\hline
		 {~} $c_{0}^{\J}$, $c_{*0}^{\J}$, $c_{1}^{\J}$, $c_{*1}^{\J}$ {~} & {~} $\mathcal{B}_0[\cdot]$ {~} & {~} $\mathsf{DMO_x}$, $\mathsf{|DMO|_x}$ spaces {~} & $\Phi_{\rm per}$ & $u_*$ & ${\mathbf{P}}_1[\cdot]$ &  $\| \cdot \|_{v,\ell}^{\mathcal{R}}$ & $\mathcal{L}_k$, $V_k$
		  \\
		\hline
		  \eqref{qd24July14-1} & \eqref{qd240802-3}
		&  Subsection \ref{DMO-sec} & \eqref{Phi-per-def}
		& \eqref{u*-def-1018}
		& \eqref{P1-def} &   \eqref{qd24Dec07-1} & \eqref{cal-L-ope}
		\\
		\hline
	\end{tabular}

\bigskip

\section*{Acknowledgements}

We would like to thank Professor Hongjie Dong for the helpful discussion. J. Wei is partially supported by GRF of Hong Kong ``New frontiers in singularity formations of nonlinear partial differential equations''. Q. Zhang is partially supported by by Hebei Province Yanzhao Golden Stage Talent Gathering Program Key Talent Project (Study Abroad Returning Platform) [Grant No. B2024018]. Y. Zhou is supported in part by the Fundamental Research Funds for the Central Universities.

\bigskip

%\bibliography{localbib}
%\bibliography{mrabbrev,localbib}

%\bibliographystyle{amsalpha-lmp}

%\providecommand{\href}[2]{#2}

%\bibliography{RefDatabase}{}
%\bibliographystyle{plain}

\end{document}